\documentclass[11pt, a4paper]{amsart}
\pdfoutput=1

\makeatletter
\@namedef{subjclassname@2020}{\textup{2020} Mathematics Subject Classification}
\makeatother

\usepackage{amsmath,amssymb,amsthm,amsfonts,enumitem,amscd,amstext,amsbsy,upgreek,mathtools,xspace} 
\usepackage[alphabetic]{amsrefs}
\usepackage{citeref, breakcites}
\usepackage{lipsum}
\usepackage{xcolor}

\usepackage[geometry]{ifsym} 

\usepackage{imakeidx} 
\makeindex[options=-s myindexstyle]

\makeatletter
\newcommand{\idf}{\leavevmode \cleaders \hb@xt@ .67em{\hss \textcolor{gray} .\hss }\hfill \kern \z@}
\makeatother

\makeatletter
\newcommand{\mygobble}[1]{\@gobble{#1}\ignorespaces} 
\makeatother

\usepackage{tikz}
\usetikzlibrary{positioning}
\usetikzlibrary{quotes}
\usetikzlibrary{arrows}
\usetikzlibrary{arrows.meta} 

\usepackage[mathscr]{eucal}
\usepackage{verbatim} 

\usepackage[unicode,bookmarks=true,pdfpagemode=UseOutlines]{hyperref}

\usepackage{float}

\usepackage[a4paper, top=3.048 cm, bottom=2.948 cm, left=3.048 cm, right=3.048 cm, marginparwidth=70 pt]{geometry}

\numberwithin{equation}{section}

\usepackage[disable, textsize=footnotesize,color=yellow!50, bordercolor=white]{todonotes} 
\newcommand{\todol}[1]{\todo[color=vturkiz!80]{#1}} 

\newcommand{\todog}[1]{\todo[color=vzold!70]{#1}} 
\newcommand{\todoig}[1]{\todo[inline, color=vzold!70]{#1}} 
\newcommand{\todoir}[1]{\todo[inline, color=red!70!magenta]{#1}}
\newcommand{\todoo}[1]{\todo[color=orange!80!red]{#1}} 

\newcommand{\todon}[1]{\todo[color=yellow!80!orange]{#1}} 

\newcommand{\todoq}[1]{\todo[color=magenta!50]{#1}} 

\definecolor{BblueOrig}{RGB}{0,119,204}
\colorlet{Bblue}{BblueOrig!87!black} 

\definecolor{turkiz}{RGB}{6,121,147} 
\definecolor{vturkiz}{RGB}{46,184,184} 
\definecolor{skek}{RGB}{6,90,147} 
\definecolor{vzold}{cmyk}{.48, .05, .91, 0} 
\definecolor{zold}{RGB}{28,172,0} 
\definecolor{narancs}{cmyk}{0,0.61,0.87,0} 

\definecolor{skek2}{RGB}{6,6,147} 
\colorlet{VeryDarkBlueNode}{skek2!40!black}
\colorlet{GreyBlueNode}{VeryDarkBlueNode!40}
\colorlet{DarkBlueNode}{skek2!90!black}

\colorlet{DottedLine}{VeryDarkBlueNode}
\colorlet{BlueArrow}{Bblue} 
\definecolor{LightBlueArrow}{RGB}{122, 163, 214} 
\colorlet{LightGrey}{black!40}
\colorlet{VeryLightBlueArrow}{LightBlueArrow!70!LightGrey}

\definecolor{TealArrow}{RGB}{0,146,179}
\colorlet{LightGreenArrow}{vzold!80!black}
\colorlet{DarkGreenArrow}{zold!85!black}
\colorlet{RedArrow}{red!85!black}
\colorlet{OrangeArrow}{orange!90!black}
\colorlet{RedOrangeArrow}{RedArrow!40!OrangeArrow}
\colorlet{YellowArrow}{yellow!70!black}

\hypersetup{pdftitle = , pdfauthor = , pdfsubject= }
\hypersetup{pdfstartview=}

\hypersetup{breaklinks=true,
bookmarksnumbered,
colorlinks=true,
linktoc=all,
linkcolor=Bblue,
urlcolor=Bblue,
citecolor=Bblue
}

\makeatletter

\renewcommand\subsubsection{
  \@startsection{subsubsection}{3} 
  \z@{.5\linespacing\@plus.7\linespacing}{-.5em} 
  {\normalfont\bfseries\boldmath} 
}

\makeatother

\DeclareRobustCommand{\SkipTocEntry}[5]{}

\makeatletter
\renewcommand{\tocsection}[3]{%
  \indentlabel{\@ifnotempty{#2}{\bfseries\ignorespaces#1 #2\quad}}\bfseries#3}
\renewcommand{\tocsubsection}[3]{%
  \indentlabel{\@ifnotempty{#2}{\ignorespaces#1 #2\quad}}#3}
\renewcommand{\tocsubsubsection}[3]{%
  \indentlabel{\hspace{30pt}\@ifnotempty{#2}{\ignorespaces#1 #2\quad}}#3}

\newcommand\@dotsepp{4.5}
\def\@tocline#1#2#3#4#5#6#7{\relax
  \ifnum #1>\c@tocdepth 
  \else
    \par \addpenalty\@secpenalty\addvspace{#2}%
    \begingroup \hyphenpenalty\@M
    \@ifempty{#4}{%
      \@tempdima\csname r@tocindent\number#1\endcsname\relax
    }{%
      \@tempdima#4\relax
    }%
    \parindent\z@ \leftskip#3\relax \advance\leftskip\@tempdima\relax
    \rightskip\@pnumwidth plus1em \parfillskip-\@pnumwidth
    #5\leavevmode\hskip-\@tempdima{#6}\nobreak
    \leaders\hbox{$\m@th\mkern \@dotsepp mu\hbox{.}\mkern \@dotsepp mu$}\hfill
    \nobreak
    \hbox to\@pnumwidth{\@tocpagenum{\ifnum#1=1\bfseries\fi#7}}\par
    \nobreak
    \endgroup
  \fi}
\AtBeginDocument{%
\expandafter\renewcommand\csname r@tocindent0\endcsname{0pt}
}
\def\l@subsection{\@tocline{2}{0pt}{2.5pc}{5pc}{}}
\makeatother

\makeatletter
\newcommand{\makesmalleroperator}[1]{%
  \expandafter\let\csname saved\string#1\endcsname#1
  \def#1{\mathop{\small@mathchoice#1}}
}
\def\small@mathchoice#1{%
  \mathchoice{\textstyle\@nameuse{saved\string#1}}%
             {\@nameuse{saved\string#1}}%
             {\@nameuse{saved\string#1}}%
             {\@nameuse{saved\string#1}}%
}
\makeatother
\makesmalleroperator\bigotimes
\makesmalleroperator\bigoplus

\makesmalleroperator\prod
\makesmalleroperator\sum

\makesmalleroperator\bigcup
\makesmalleroperator\bigcap
\makesmalleroperator\bigvee
\makesmalleroperator\bigwedge

\makesmalleroperator\bigsqcup
\makesmalleroperator\bigsqcap
\makesmalleroperator\bigtriangledown
\makesmalleroperator\bigtriangleup
\makesmalleroperator\triangle

\newcommand{\RR}{\mathbb{R}}

\newcommand{\PP}{\mathbb{P}}
\newcommand{\QQ}{\mathbb{Q}}

\newcommand{\KK}{\mathbb{K}}
\newcommand{\CC}{\mathbb{C}}
\newcommand{\BB}{\mathbb{B}}

\newcommand{\HH}{\mathbb{H}}
\newcommand{\JJ}{\mathbb{J}}

\renewcommand{\SS}{\mathbb{S}}

\let\oldsetminus\setminus
\renewcommand{\setminus}{{\oldsetminus}} 

\newcommand{\id}{\operatorname{id}}

\newcommand{\Col}{\operatorname{Col}}

\newcommand{\Add}{\operatorname{Add}}
\newcommand{\Addd}{\mathrm{Add^{*}}}

\newcommand{\lh}{\operatorname{lh}}

\newcommand{\Lim}{{\mathrm{Lim}}}

\newcommand{\Sym}{\operatorname{\mathrm{Sym}}}

\newcommand{\GCH}{{\mathrm{GCH}}}
\newcommand{\CH}{{\mathsf{CH}}}
\newcommand{\restr}{{\upharpoonright}}

\newcommand{\AD}{{\sf AD}}
\newcommand{\ZFC}{{\sf ZFC}}

\newcommand{\DC}{{\sf DC}}

\newcommand{\megj}[1]{}

\newcommand{\downset}{\ensuremath{\subseteq_{\textnormal{\sf down}}}}

\newcommand{\lle}{\ensuremath{{<}}}

\newcommand{\lleq}{\ensuremath{{\leq}}}

\newcommand{\lexleq}{\ensuremath{\leq_{\mathsf{lex}}}}

\newcommand\comp{\circ}
\newcommand{\conc}{ {}^\frown}
\newcommand{\cwrestr}[1]{_{\upharpoonright #1}} 
\newcommand{\partialto}{\rightharpoonup}

\newcommand{\dom}{\operatorname{dom}}
\newcommand{\ran}{\operatorname{ran}}

\newcommand{\pwrset}{\ensuremath{\mathcal{P}}}
\newcommand{\powerset}{\ensuremath{\mathcal{P}}}
\renewcommand{\bar}{\overline}
\newcommand{\proj}{\operatorname{\mathbf{p}}}
\newcommand{\graph}{\operatorname{G}} 

\newcommand{\llbracket}{\ensuremath{[\hspace{-1.7pt}[}}
\newcommand{\rrbracket}{\ensuremath{]\hspace{-1.7pt}]}}

\newcommand{\boolval}[1]{\ensuremath{\llbracket{#1}\rrbracket}}
\newcommand{\boolvalpo}[2]{\ensuremath{\llbracket{#1}\rrbracket}_{#2}}

\newcommand{\Ord}{\ensuremath{\mathrm{Ord}}} 
\newcommand{\Succ}{\ensuremath{\mathrm{Succ}}} 

\newcommand{\cof}{\ensuremath{\mathrm{cof}}}
\newcommand{\stat}{\ensuremath{\mathrm{stat}}}

\newcommand{\Diamondi}[1]{\ensuremath{\Diamond^{i}_{#1}}}

\newcommand{\Dell}[1]{\ensuremath{(\mathrm{D}\ell)_{#1}}}

\newcommand{\compat}{\ensuremath{\parallel}}
\newcommand{\forces}{\ensuremath{\Vdash}}

\newcommand{\bigveepo}[1]{\ensuremath{\bigvee\nolimits^{#1}}}
\newcommand{\bigwedgepo}[1]{\ensuremath{\bigwedge\nolimits^{#1}}}

\DeclareMathAlphabet{\mymathbbold}{U}{bbold}{m}{n}
\newcommand{\one}{\ensuremath{\mymathbbold{1}}}
\newcommand{\zero}{\ensuremath{\mymathbbold{0}}}

\newcommand{\height}{\operatorname{ht}}

\newcommand{\kkppred}[1]{{{#1}{\downarrow}}} 

\renewcommand{\succ}[1]{{{#1}{\uparrow}}} 

\newcommand{\closedsets}{\ensuremath{\mathbf{\Pi}^0_1}} 
\newcommand{\closure}[1]{\ensuremath{\overline{#1}}}

\newcommand{\Fsigma}{\ensuremath{\mathbf{\Sigma}^0_2}}
\newcommand{\Gdelta}{\ensuremath{\mathbf{\Pi}^0_2}}
\newcommand{\DeltaZeroTwo}{\ensuremath{\mathbf{\Delta}^0_2}}

\newcommand{\analytic}{\ensuremath{\mathbf{\Sigma}^1_1}}
\newcommand{\coanalytic}{\ensuremath{\mathbf{\Pi}^1_1}}

\newcommand{\PSP}{\ensuremath{\mathsf{PSP}}}
\newcommand{\OCA}{\ensuremath{\mathsf{OCA}}}

\newcommand{\OGA}{\ensuremath{\mathsf{OGA}}}
\newcommand{\OGD}{\ensuremath{\mathsf{OGD}}}
\newcommand{\ODD}[2]{\ensuremath{\mathsf{ODD}^{#2}_{#1}}} 

\newcommand{\OHD}[2]{\ensuremath{\mathsf{OHD}^{#2}_{#1}}} 

\newcommand{\ODDI}[2]{\ensuremath{\mathsf{ODDI}^{#2}_{#1}}} 
\newcommand{\ODDInp}{\ensuremath{\mathsf{ODDI}}} 
\newcommand{\ODDH}[2]{\ensuremath{\mathsf{ODDH}^{#2}_{#1}}} 
\newcommand{\ODDC}[2]{\ensuremath{\mathsf{ODDC}^{#2}_{#1}}} 

\newcommand{\ddim}{d} 
\newcommand{\ddimq}{\delta} 

\newcommand{\domh}[1]{\ensuremath{\dom_{#1}}} 

\newcommand{\leqf}{\ensuremath{\leq_{\mathsf f}}} 
\newcommand{\leqif}{\ensuremath{\leq_{\mathsf {if}}}} 
\newcommand{\equivf}{\ensuremath{\equiv_{\mathsf f}}}
\newcommand{\equivif}{\ensuremath{\equiv_{\mathsf {if}}}}

\newcommand{\indtrees}[1]{\ensuremath{\mathcal T_{#1}^{\mathrm{ind}}}}

\newcommand{\dhH}[1]{\ensuremath{\mathbb{H}_{{}^\kappa\!{#1}}}} 

\newcommand{\sym}{{\mathsf{sym}}}
\newcommand{\hhH}[1]{\ensuremath{{\mathbb{H}}^{\sym}_{{}^\kappa\!{#1}}}} 
\newcommand{\hypg}[1]{{#1}^{\sym}} 

\newcommand{\dhHomega}[1]{\ensuremath{\mathbb{H}_{{}^\omega\!{#1}}}} 
\newcommand{\dhC}[2]{\ensuremath{{\Delta}^{#1}_{#2}}} 
\newcommand{\dhK}[2]{\ensuremath{\mathbb{K}^{#1}_{#2}}} 
\newcommand{\gK}[1]{\ensuremath{\mathbb{K}_{#1}}} 

\newcommand{\dhII}[2]{\ensuremath{\mathbb{I}^{#1}_{#2}}} 
\newcommand{\dhI}[1]{\ensuremath{\mathbb{I}^{#1}}} 
\newcommand{\dhIkappa}{\ensuremath{\mathbb{I}}} 

\newcommand{\dhN}[1]{\ensuremath{\mathbb{N}^{#1}}} 
\newcommand{\dhNkk}{\ensuremath{\mathbb{N}}} 
\newcommand{\Xform}[2]{\ensuremath{X_{#1,#2}}}  
\newcommand{\Xphia}{\Xform{\varphi}{a}}

\newcommand{\defsets}[1]{\ensuremath{\mathsf D_{#1}}} 
\newcommand{\defsetsk}{\ensuremath{\mathsf D_\kappa}} 

\newcommand{\mpo}{\bar{\QQ}}

\newcommand{\mleq}{\leq}

\newcommand{\mbigcap}{\bigwedge}

\newcommand{\na}{\sigma} 
\newcommand{\nna}{\tau}  
\newcommand{\mmap}{\pi^{\ast}} 

\newcommand{\THD}{\ensuremath{\mathsf{THD}}} 
\newcommand{\HD}{\ensuremath{\mathsf{HD}}}  

\newcommand{\ABP}{\ensuremath{\mathsf{ABP}}} 
\newcommand{\dhD}[1]{\ensuremath{\mathbb{D}_{#1}}} 

\newcommand{\KLW}{\ensuremath{\mathsf{KLW}}} 
\newcommand{\KLWdis}{\ensuremath{\mathsf{KLW}^{\mathsf{dis}}}} 
\newcommand{\sKLW}{\ensuremath{\mathsf{KLW}^{\ast}}} 

\newcommand{\JR}{\ensuremath{\mathsf{JR}}} 

\newcommand{\pone}{\ensuremath{\mathbf{I}}}
\newcommand{\ptwo}{\ensuremath{\mathbf{II}}}
\newcommand{\GV}{\ensuremath{\mathcal{V}}}
\newcommand{\Ker}{\ensuremath{\mathrm{Ker}}}
\newcommand{\Sc}{\ensuremath{\mathrm{Sc}}}
\newcommand{\Det}{\ensuremath{\mathsf{Det}}}
\newcommand{\cbo}{\ensuremath{\mathsf{CB_\kappa^-}}}
\newcommand{\cbi}{\ensuremath{\mathsf{CB_\kappa^1}}}
\newcommand{\cbii}{\ensuremath{\mathsf{CB_\kappa^2}}}
\newcommand{\cboomegaone}{\ensuremath{\mathsf{CB_{\omega_1}^-}}}
\newcommand{\cbiomegaone}{\ensuremath{\mathsf{CB_{\omega_1}^1}}}
\newcommand{\cbiiomegaone}{\ensuremath{\mathsf{CB_{\omega_1}^2}}}
\newcommand{\cbig}[1]{\ensuremath{{\mathsf{CB}_\kappa^{1,#1}}}}

\newcommand{\DK}{\ensuremath{\mathsf{DK}}} 

\newcommand{\ths}{\textrm{\upshape({\sf TH1})}\xspace} 
\newcommand{\thl}{\textrm{\upshape({\sf TH2})}\xspace} 
\newcommand{\hds}{\textrm{\upshape({\sf HD1})}\xspace} 
\newcommand{\hdl}{\textrm{\upshape({\sf HD2})}\xspace} 
\newcommand{\klws}{\textrm{\upshape({\sf KLW1})}\xspace} 
\newcommand{\klwl}{\textrm{\upshape({\sf KLW2})}\xspace} 

\newcommand{\dhth}[1]{\ensuremath{\mathbb{H}_{\mathsf{top}}^{#1}}}
\newcommand{\dhthkk}{\ensuremath{\mathbb{H}_{\mathsf{top}}}}

\newcommand{\dhhd}{\ensuremath{\mathbb{H}_{\mathsf{super}}}}
\newcommand{\dhhdom}{\ensuremath{\mathbb{H}_{\mathsf{super}}}}

\newcommand{\ka}{\kappa}

\newcommand{\kk}{{}^{\ka}\ka}

\newcommand{\widtharrow}{{.15mm}}
\newcommand{\widthline}{{.15mm}}
\newcommand{\widthlinethin}{{.08mm}}

\newtheorem{theorem}{Theorem}[section]
\newtheorem*{theorem*}{Theorem}
\newtheorem{lemma}[theorem]{Lemma}
\newtheorem*{lemma*}{Lemma}
\newtheorem{corollary}[theorem]{Corollary}
\newtheorem{proposition}[theorem]{Proposition}

\newtheorem{problem}[theorem]{Problem}

\newtheorem*{claim*}{Claim}

\newtheorem*{subclaim*}{Subclaim}
\newtheorem*{case*}{Case}

\newtheorem{case}{Case}[]

\theoremstyle{definition}
\newtheorem{definition}[theorem]{Definition}
\newtheorem{fact}[theorem]{Fact}
\newtheorem{remark}[theorem]{Remark}
\newtheorem{example}[theorem]{Example}

\newtheorem{assumptions}[theorem]{Assumptions}

\theoremstyle{plain} 
\newcommand{\thistheoremname}{}
\newtheorem{genericthm}[theorem]{\thistheoremname}

\newtheorem*{genericthm*}{\thistheoremname}
\newenvironment{namedtheorem*}[1]
  {\renewcommand{\thistheoremname}{#1}%
   \begin{genericthm*}}
  {\end{genericthm*}}

\newenvironment{enumerate-(a)}{\begin{enumerate}[label={\upshape (\alph*)}]}{\end{enumerate}}

\newenvironment{enumerate-(A)}{\begin{enumerate}[label={\upshape (\Alph*)}]}{\end{enumerate}}

\newenvironment{enumerate-(i)}{\begin{enumerate}[label={\upshape (\roman*)}]}{\end{enumerate}}

\newenvironment{enumerate-(I)}{\begin{enumerate}[label={\upshape (\Roman*)}]}{\end{enumerate}}

\newenvironment{enumerate-(1)}{\begin{enumerate}[label={\upshape (\arabic*)}]}{\end{enumerate}}

\AtBeginDocument{%
   \def\MR#1{}
}

\begin{document}

\title[The open dihypergraph dichotomy for generalized Baire spaces]%
{The open dihypergraph dichotomy for generalized \\ Baire spaces and its applications} 

\author{Philipp Schlicht}
\address{Dipartimento di Ingegneria dell'Informazione e Scienze Matematiche, Università di Siena, via Roma 56, 53100 Siena, Italy
} 
\email{philipp.schlicht@unisi.it}
  
\author{Dori Szir\'aki}
\address{\H{O}rhegy~u.~112, H--2085 Pilisv\"or\"osv\'ar, Hungary}
\email{dsziraki@renyi.hu}



\thanks{The authors are grateful to Stevo Todor\v{c}evi\'c for comments, to Rapha\"el Carroy and Ben Miller for answering a query regarding their work and to the referee for the thorough reading of the manuscript and  useful suggestions.} 

\thanks{The first-listed author acknowledges the kind support of INdAM\--GNSAGA (Istituto Nazionale di Alta Matematica, Gruppo Nazionale per le Strutture Algebriche, Geometriche e le loro Applicazioni). This research was funded in part by EPSRC grant number EP/V009001/1 of the first-listed author. 
The second-listed author was partially supported by NKFIH grant number K129211.
For the purpose of open access, the authors have applied a ‘Creative Commons Attribution' (CC BY) public copyright licence to any Author Accepted Manuscript (AAM) version arising from this submission. 
}

\maketitle 

\begin{abstract}

The open graph dichotomy 
for a
subset $X$ of the Baire space ${}^\omega\omega$ 
states that any open graph on $X$ either admits a coloring in countably many colors or contains a perfect complete subgraph. 
This strong version of the open graph axiom 
for $X$
was introduced by Feng and Todor\v{c}evi\'c to study definable sets of reals.  
We first show that its recent generalization to infinite dimensional directed hypergraphs by Carroy, Miller and Soukup 
holds for all subsets of the Baire space in Solovay's model, 
extending a theorem of Feng in dimension $2$. 

The main theorem 
lifts this result to generalized Baire spaces~${}^\kappa\kappa$ in two ways. 
\begin{enumerate-(1)}
\item 
For any regular infinite cardinal $\kappa$, the following 
holds after a L\'evy collapse of an inaccessible cardinal $\lambda>\kappa$ to $\kappa^+$. 
\begin{quote} 
Suppose that $H$ is a $\kappa$-dimensional box-open directed hypergraph on a subset of ${}^\kappa\kappa$ such that $H$ is definable from a $\kappa$-sequence of ordinals. 
Then either $H$ admits a 
coloring in $\kappa$ many colors
or there exists a continuous homomorphism from a 
canonical 
large directed hypergraph to $H$. 
\end{quote} 
\item 
If $\lambda$ is a Mahlo cardinal, then
the previous result extends to all  box-open directed hypergraphs on any subset of ${}^\kappa\kappa$ that is definable from a $\kappa$-sequence of ordinals.
\end{enumerate-(1)}

This approach is applied to solve several problems about definable subsets of generalized Baire spaces. 
For instance, we obtain 
variants of the Hurewicz dichotomy that characterizes subsets of $K_\sigma$ sets, strong variants of the Kechris--Louveau--Woodin dichotomy that characterizes when two disjoint sets can be separated by an $F_\sigma$ set, the determinacy of V\"a\"an\"anen's perfect set game for all subsets of ${}^\kappa\kappa$, an asymmetric version of the Baire property and an analogue of the Jayne--Rogers theorem that characterizes $G_\delta$-measurable functions. 
\end{abstract}

\newpage

\tableofcontents

\newpage 

\section{Introduction}

It is natural to wonder whether Ramsey's theorem for $n$-tuples of natural numbers \cite{Ramsey1930} can be extended to the set of real numbers. 
Sierpinski's counterexample is a partition of pairs of reals into two pieces such that no uncountable homogeneous set exists \cite{sierpinski1933probleme}. 
Since his example is not constructive, one can ask if a version of Ramsey's theorem holds for simple partitions. 
Galvin 
gave a partial answer to this
by proving that for any partition 
of pairs of reals
into two open pieces there exists an uncountable (in fact, perfect) homogeneous set
\cite{galvin1968partition}.
Blass generalized this to partitions of $n$-tuples into finitely many Borel pieces
by weakening the conclusion 
\cite{blass1981partition}: 
instead of homogeneity, one requires that
all pairs in the set belong to at most $(n-1)!$ pieces.

The next step was to generalize these results to spaces other than the reals. 
Abraham, Rubin and Shelah formulated a strengthening of Galvin's theorem for countably based metric spaces of size $\aleph_1$ \cite{AbrahamRubinShelah}.
Abstracting from many applications, Todor\v{c}evi\'c introduced the \emph{open graph axiom}
$\OGA$ \cite{TodorcevicPartitionProblems}.%
\footnote{It is also often called the \emph{open coloring axiom} $\OCA$. 
Note that a graph $G$ on a set $X$ can be identified with a partition or coloring of pairs in $X$, for example by coloring edges in $G$ red and edges in the complement of $G$ blue. 
Galvin's theorem is then equivalent to the statement that for any clopen graph $G$ on the reals, there exists either a perfect independent set or a perfect complete subgraph.} 
For a topological space $X$,
$\OGA(X)$ states that any open graph $G$ on $X$ either admits a coloring%
\footnote{Recall that a coloring of a graph $G$ assigns different colors to adjacent vertices.} 
in countably many colors or else contains an uncountable complete subgraph. 
$\OGA$ states that $\OGA(X)$ holds 
for all countably based metric spaces $X$. 
$\OGA$ follows from the proper forcing axiom $\mathsf{PFA}$ \cite{TodorcevicPartitionProblems}
 and has many applications \cites{bekkali1991topics, TodorcevicFarah, VelickovicOCA}. 
Feng and Todor\v{c}evi\'c studied a stronger version of $\OGA(X)$, the \emph{open graph dichotomy} $\OGD(X)$, that is useful for applications to definable\footnote{Subsection~\ref{preliminaries: notation} explains for
what we mean precisely by definable subsets of (generalized) Baire spaces.}
sets of reals \cites{FengOCA, TodorcevicFarah}. 
Here the uncountable 
complete subgraph
has to be perfect. 
$\OGD(X)$ is provable for all analytic sets $X$ of reals and it is in fact consistent that $\OGD(X)$ holds for all definable sets $X$ \cite{FengOCA}.\footnote{See \cite{FengOCA}*{Theorem~1.1 and Section~4}, as well as \cite{FengOCA}*{pages~676-677} for an alternative approach for analytic sets due to Todor\v{c}evi\'c.} 
The previous statements only mention open graphs because they cannot hold for all closed graphs \cite{TodorcevicFarah}.%
\footnote{See Example~\ref{OGD fails for closed graphs} in this paper.}

The verbatim analogues of the open graph axiom and dichotomy fail in dimension $3$ and higher:
an open dihypergraph%
\footnote{A dihypergraph, or directed hypergraph, of dimension $\ddim$ is a set of nonconstant functions with domain~$\ddim$.} 
on the Baire or Cantor space might neither admit a countable coloring nor have an uncountable complete subhypergraph \cites{TodorcevicFarah,HeHigherDimOCA}.${}$\footnote{See Examples~\ref{3 dim OGA fails} and~\ref{3 dim OGA fails He} in this paper.}
\todon{Keep this todo-note. We added the symbols \$\{\}\$ just before the next footnotes 6 and 7 to make the superscript number appear in the right size. It looks like a latex bug.}
However, one can replace the latter
by a continuous homomorphic image of a canonical large dihypergraph.
Carroy, Miller and Soukup 
proved a countably infinite dimensional version of $\OGD(X)$ for 
analytic sets $X$
and extended this to  
all sets of reals assuming the axiom of determinacy $\AD$
\cite{CarroyMillerSoukup}.%
\footnote{It follows that this dichotomy is consistent with $\ZFC$ for all box-open $\aleph_0$-dimensional dihypergraphs on the Baire space which are definable from a countable sequence of ordinals. To see this,  take a model of $\AD+\DC+{V=L(\RR)}$ and add a well-order of the reals by a $\sigma$-closed homogeneous forcing.}
\todog{Dori's note to self for the footnote: Since the forcing (in the 2nd sentence) is $\sigma$-closed, no new $\omega$-sequences of ordinals are added. Hence by homogeneity, every box-open definable $\omega$-dihypergraph $H$ on the Baire space in the generic extension is in already in $V$ and is definable in $V$. (The proof is similar to Lemma \ref{restricting sets to intermediate models} \ref{restricting sets to intermediate models 1}. ``Definable'' means definable from an $\omega$-sequence of ordinals, as usual.) Homogeneity is needed since otherwise the forcing can make any dihypergraph definable even without parameters.}
A similar idea 
is implicit in work of
Avil\'es and Todor\v{c}evi\'c \cite{avitod2010}.${}$\footnote{We would like to thank Stevo Todor\v{c}evi\'c for pointing out that the idea for a higher dimensional version of $\OGD(X)$ is implicit in the proof of \cite{avitod2010}*{Theorem~7}.}
The main motivation behind this dichotomy 
is to provide new proofs and strong versions of several seemingly unrelated results in descriptive set theory \cite{CarroyMillerSoukup}. 
The first such result is the Hurewicz dichotomy that 
characterizes the circumstances in which an analytic subset of a Polish space is contained in a $K_\sigma$ subset
\cites{Hurewicz1928, SaintRaymond1975, Kechris1977}. 
The Hurewicz dichotomy has also been studied by theoretical computer scientists \cite{de2018generalization}. 
A dichotomy due to Kechris, Louveau and Woodin describes when 
an analytic subset of a Polish space can be separated from a disjoint subset by an $F_\sigma$ set
\cite{KechrisLouveauWoodin1987}. 
A celebrated theorem of Jayne and Rogers characterizes $G_\delta$-measurable functions on analytic sets \cite{JayneRogers1982}. 

We study the above problems in the more general context where $\omega$ is replaced by a regular infinite cardinal $\kappa$ with $\kappa^{<\kappa}=\kappa$. 
The topological properties of the generalized Baire space ${}^\kappa\kappa$ of functions on $\kappa$ resemble that of the Baire space ${}^\omega\omega$. In fact, larger classes of spaces including ${}^\kappa\kappa$ have been identified as analogues to 
Polish spaces
at uncountable cardinals \cites{coskey2013generalized, agostini2021generalized}. 
The study of descriptive set theory for these spaces in the uncountable setting is motivated, 
inter alia, 
by connections 
with model theory 
\todoo{added Vaught's paper} 
\cites{vaught1974invariant, vaan,VaananenGamesTrees} 
and Shelah's classification theory in particular \cites{MR3235820, moreno2022isomorphism, moreno2022unsuperstable, moreno2023shelah}.
There is an ample amount of background literature   
on generalized Baire spaces and their descriptive set theory, 
beginning with the study of closed, $\kappa$-analytic and $\kappa$-coanalytic subsets \cites{VaananenCantorBendixson, vaan, lucke2012definability, MR3430247}. 
It is known that some properties of definable subsets of generalized Baire spaces are closely linked to statements in combinatorial set theory, for instance to the existence of Kurepa trees \cites{lucke2020descriptive,LambieStejPSP}. 
The extension of several classical dichotomies to the uncountable setting paves the way for a structure theory of definable sets. 
For instance, the Hurewicz dichotomy for $\kappa$-analytic sets is consistent 
\cite{LuckeMottoRosSchlichtHurewicz}.
Assuming the existence of an inaccessible cardinal,
it is consistent for all definable sets \cite{Hurewiczdef},
and so is the $\kappa$-perfect set property for definable sets \cite{SchlichtPSPgames}.
Moreover, the analogue of the open graph dichotomy for $\kappa$-analytic sets is consistent relative to an inaccessible cardinal \cite{SzThesis}. 
There is keen interest in analogues of the open graph axiom $\OGA$ 
and similar principles at cardinals beyond $\omega_1$ \cite{mohammadpour2021guessing}. 

The aim here is to extend the aforementioned results on the open graph dichotomy and
its infinite dimensional version 
to uncountable cardinals and utilize such results to derive applications to the descriptive set theory of generalized Baire spaces. 
We next introduce the relevant variants of these dichotomies.
The precise definitions of some of the following concepts will be given in Section \ref{section preliminaries}. 
We assume that $\kappa$ is a regular infinite cardinal with 
$\kappa^{<\kappa}=\kappa$
throughout this paper, unless otherwise stated.
The $\kappa$-Baire space
\index{Baire space@$\kappa$-Baire space\idf${}^\kappa\kappa$}
is the set ${}^\kappa\kappa$ of functions from $\kappa$ to $\kappa$ equipped 
with the \emph{bounded topology} 
(or \emph{$\lle\kappa$-box topology}),
i.e., the topology generated by the base $\{N_t:t\in{}^{<\kappa}\kappa\}$, where
$N_t=\{x\in{}^\kappa\kappa : t\subseteq x\}$ 
for each $t\in{}^{<\kappa}\kappa$.%
\footnote{The $\omega$-Baire space is the Baire space ${}^\omega\omega$, since the bounded topology equals the product topology for $\kappa=\omega$.} 
The $\kappa$-Cantor space 
\index{Cantor space@$\kappa$-Cantor space\idf${}^\kappa 2$}%
is ${}^\kappa 2$ with the subspace topology.
A \emph{graph} is a symmetric irreflexive binary relation. 
A graph $G$ on a topological space $X$ is \emph{open} 
\index{graph!open\idf}
if it is an open subset of the product space $X\times X$.
A subset of $X$ is \emph{$G$-independent} if it contains no edges, and
a \emph{$\kappa$-coloring} of $G$ is a partition of $X$ into at most $\kappa$ many disjoint $G$-independent sets. 

Todor\v{c}evi\'c's 
{open graph axiom} 
\cite{TodorcevicPartitionProblems}
can be generalized to the uncountable setting.
For any class $\mathcal C$, let
\index{open graph axiom for a!A@class \idf $\OGA_\kappa({\mathcal C})$} 
$\OGA_\kappa({\mathcal C})$ state that the following holds for all subsets 
$X\in \mathcal{C}$ of ${}^\kappa\kappa$: \begin{quotation}
$\OGA_\kappa(X)$:
\index{open graph axiom for a!B@set \idf $\OGA_\kappa(X)$} 
If $G$ is an open graph on $X$, then
either $G$ admits a $\kappa$-coloring or 
$G$ has a complete subgraph of size $\kappa^+$. 
\end{quotation}
The following generalization of 
the {open graph dichotomy} \cite{FengOCA} is a stronger version of this axiom. 

\begin{definition} 
\label{OGD def}
For any class $\mathcal C$, 
$\OGD_\kappa({\mathcal C})$ 
\index{open graph dichotomy for a!A@class\idf $\OGD_\kappa({\mathcal C})$} 
states that the following holds for all subsets $X\in \mathcal{C}$ of ${}^\kappa\kappa$: 
\begin{quotation}
$\OGD_\kappa(X)$:
\index{open graph dichotomy for a!B@set\idf $\OGD_\kappa(X)$} 
If $G$ is an open graph on $X$, then
either $G$ admits a $\kappa$-coloring or 
$G$ has a $\kappa$-perfect 
complete subgraph.%
\footnote{I.e., $G$ has a complete subgraph whose domain is a $\kappa$-perfect subset of ${}^\kappa\kappa$. 
See Definitions~\ref{def: domain of H} and~\ref{def: perfect set}.}
\end{quotation}
\end{definition}

Note that $G$ has a $\kappa$-perfect complete subgraph if and only if there exists 
a continuous homomorphism from the complete graph
$\gK{{}^\kappa 2}$ to $G$.\footnote{See Corollary \ref{homomorphisms and perfect homogeneous subgraphs}.} 
Thus, $\OGD_\kappa(X)$ can be formulated for arbitrary topological spaces as well. 
\todoo{added the sentence about naive generalizations} 
However, naive generalizations of $\OGD_\kappa(X)$ in higher dimensions fail.\footnote{See Subsection~\ref{subsection: examples}.} 
In the countable case,\footnote{I.e., for $\kappa=\omega$.} 
Carroy, Miller and Soukup introduced the \emph{box-open $\ddim$-dimensional dihypergraph dichotomy} as an analogue of $\OGD_\omega(X)$ in higher dimensions $\ddim$ \cite{CarroyMillerSoukup}. 
We now consider its analogue for regular infinite cardinals $\kappa$ with $\kappa^{<\kappa}=\kappa$. 

Let $X$ be any set and $\ddim$ a set of size at least $2$.
A \emph{$\ddim$-dihypergraph}\footnote{I.e., a \emph{$\ddim$-dimensional directed hypergraph}.}
on $X$ is a subset $H$ of ${}^\ddim X$ consisting of non-constant sequences.
A subset of $X$ is \emph{$H$-independent} 
if it contains no hyperedges,%
\footnote{See Subsection~\ref{subsection: dihypergraphs} for the precise definitions of the concepts discussed in this paragraph.}
and a \emph{$\kappa$-coloring} of $H$ is a partition of $X$ into at most $\kappa$ many disjoint $H$-independent sets. 
A \emph{homomorphism} from one $\ddim$-dihypergraph to another is a map that takes hyperedges to hyperedges.
We write elements of ${}^\ddim X$ as $\bar{x}=\langle x_i : i\in\ddim\rangle$. 
Let $\dhH\ddim$
\index{dihypergraph!of uniformly splitting sequences\idf$\dhH\ddim$}
denote the following $\ddim$-dihypergraph on~${}^\kappa\ddim$ (see Figure \ref{figure: dhgrs on kappa^kappa}). 
We call this the \emph{dihypergraph of uniformly splitting sequences.} 
\todon{In the next figure, ``abovecaptionskip'' (distance between figure and caption) and ``intextsep'' (distance between text and figure) were set LOCALLY. Remove before submitting to journal?}
\vspace{5pt} 
\begin{equation*}
\dhH\ddim
:=
\bigcup_{t\in{}^{<\kappa}\ddim}\prod_{i\in \ddim} N_{t\conc\langle i\rangle}
=
\{\bar x\in{}^\ddim({}^{\kappa}\ddim)
\::\:
{\exists t\in{}^{<\kappa}\ddim}\ \ {\forall i\in\ddim}\ \ 
t\conc\langle i\rangle\subseteq x_i
\}
.
\end{equation*}
{
\setlength{\abovecaptionskip}{5pt}
\setlength{\intextsep}{-5 pt}
\newcommand{\x}{0.5} 
\newcommand{\y}{1} 
\begin{figure}[ht] 
\centering
\tikz[scale=1.0,font=\small]{ 
\draw 
(0*\x-110pt,0); 

\draw 
(0*\x,0) edge[-] (-4*\x,4*\y) 
(0*\x,0) edge[-] (4*\x,4*\y); 

\draw
(-0.75*\x,0.75*\y) edge[-] (-0.4*\x,1.5*\y) 
(-0.4*\x,1.5*\y) edge[-] (-0.6*\x,2.25*\y); 

\draw
(-0.6*\x,2.25*\y) edge[thick, -, YellowArrow] (-1.4*\x,3*\y) 
(-0.6*\x,2.25*\y) edge[thick, -, DarkGreenArrow] (-0.6*\x,3*\y) 
(-0.6*\x,2.25*\y) edge[thick, -, TealArrow] (0.2*\x,3*\y); 

\draw
(-1.4*\x,3*\y) edge[thick, -, dashed, YellowArrow] (-1.8*\x,4*\y) 
(-0.6*\x,3*\y) edge[thick, -, dashed, DarkGreenArrow] (-0.6*\x,4*\y) 
(0.2*\x,3*\y) edge[thick, -, dashed, TealArrow] (0.8*\x,4*\y); 

\draw
    (-0.2*\x,3*\y) node {\tiny \color{DarkGreenArrow} ..} 
    (1.2*\x,3*\y) node {\tiny \color{TealArrow} ..}; 
       
\draw
    (-1.2*\x,3*\y) node[left] {\tiny \color{YellowArrow} $t_0$} 
    (-0.9*\x,3*\y) node {\tiny \color{DarkGreenArrow} $t_1$} 
    (0.1*\x,3*\y) node[right] {\tiny \color{TealArrow} $t_\alpha$}; 

\draw
    (1.8*\x,4.2*\y) node {\tiny \color{TealArrow} ..}; 
    
\draw
    (-1.8*\x,4*\y) node[above] {\tiny \color{YellowArrow} $x_0$} 
    (-0.6*\x,4*\y) node[above] {\tiny \color{DarkGreenArrow} $x_1$} 
    (1*\x,4*\y) node[above] {\tiny \color{TealArrow} $x_\alpha$}; 
    
    
\draw
    (6*\x+80pt,3*\y) node {\tiny $t_\alpha=t^\smallfrown \langle\alpha\rangle$}; 

\draw 
    (0*\x,0*\y) node[above] {} 
    (4*\x,4*\y) node[above] {}; 
} 
\caption{$\dhH\ddim$} 
\label{figure: dhgrs on kappa^kappa}
\end{figure}
}

We will understand $d$ as a discrete topological space and equip ${}^\kappa d$ with the ${<}\kappa$-box topology. 
Ordinals are equipped with the discrete topology. 
If $X$ is a subset of the $\kappa$-Baire space, 
then ${}^\ddim X$ is equipped with the box-topology unless otherwise stated.

\pagebreak
\begin{definition} 
\label{ODD def}
For any $\ddim$-dihypergraph $H$ 
on a topological space $X$,
let {$\ODD\kappa H$} 
\index{open dihypergraph dichotomy!A@for a dihypergraph\idf $\ODD\kappa H$} 
denote the following statement: 
\todog{We define $\ODD\kappa H$ for all dihypergraphs $H$ on $X$ because we sometimes talk about the case when $H$ is not box-open}
\begin{quotation}
{$\ODD\kappa H$:}
Either $H$ admits a $\kappa$-coloring, or 
there exists a continuous homomorphism 
$f:{}^\kappa\ddim\to X$
from~$\dhH{\ddim}$ to~$H$.
\end{quotation}
\end{definition}

\todoo{added this sentence as a motivation. check line breaks from here}  
This is a valid generalization of the open graph dichotomy to higher dimensions, since $\OGD_\kappa(X)$ is equivalent to $\ODD\kappa H$ for all open digraphs $H$ on $X$.\footnote{See Corollary \ref{homomorphisms and perfect homogeneous subgraphs}.} 

\begin{definition}
\label{def: ODD for classes}
For classes $\mathcal C$ and $\mathcal D$, let $\ODD\kappa\ddim(\mathcal C,\mathcal D)$
\index{open dihypergraph dichotomy!$d$-dimensional for!two classes\idf $\ODD\kappa\ddim(\mathcal C,\mathcal D)$}  
denote statement that 
$\ODD\kappa{H\restr X}$ holds%
\footnote{$H\restr X:=H\cap{}^\ddim X$ denotes the {restriction} of $H$ to $X$.}
 whenever 
\begin{enumerate-(a)} 
\item 
$X$ is a subset of ${}^\kappa\kappa$ in $\mathcal C$ and 
\item 
$H$ is a box-open%
\footnote{I.e, $H$ is an open subset of ${}^\ddim({}^\kappa\kappa)$ in the box topology.} 
$\ddim$-dihypergraph on ${}^\kappa\kappa$ in $\mathcal D$.
\end{enumerate-(a)} 
If $\mathcal D$ is the class of all sets or all $\ddim$-dihypergraphs on ${}^\kappa\kappa$, then we omit it from the notation
and write $\ODD\kappa\ddim({\mathcal C})$ 
\index{open dihypergraph dichotomy!$d$-dimensional for!a class\idf $\ODD\kappa\ddim(\mathcal C)$}  
for short.
\index{open dihypergraph dichotomy!$d$-dimensional for!a set and class\idf $\ODD\kappa\ddim(X,\mathcal C)$}  
If $\mathcal C$ has one single element $X$, then we write $X$ instead of $\mathcal C$, and similarly for $\mathcal D$. 
\end{definition} 

For example, $\ODD\kappa\ddim(X)$
\index{open dihypergraph dichotomy!$d$-dimensional for!a set aa@a set\idf $\ODD\kappa\ddim(X)$}
states that $\ODD\kappa {H{\restr}X}$ holds for all box-open $\ddim$-dihypergraphs $H$ on ${}^\kappa\kappa$. 
It is equivalent that $\ODD\kappa H$ holds for all $\ddim$-dihypergraphs $H$ on $X$ that are box-open on $X$.%
\footnote{In general, we have to work with the restrictions of dihypergraphs $H\in\mathcal D$ on ${}^\kappa\kappa$ to $X$ instead of dihypergraphs on $X$ which are in $\mathcal D$. 
For instance, $\ODD\kappa{H{\restr}X}$ might hold for all definable box-open dihypergraphs $H$ 
even though $H{\restr}X$ is not necessarily definable. 
This is relevant for several applications.}
\todog{this is exactly what it's called in \cite{CarroyMillerSoukup}}
Following \cite{CarroyMillerSoukup}, we call $\ODD\kappa\ddim(X)$ the \emph{box-open $\ddim$-dimensional dihypergraph dichotomy} for $X$,
or \emph{open dihypergraph dichotomy} for short. 
The open graph dichotomy $\OGD_\kappa(X)$ 
is equivalent to $\ODD\kappa 2(X)$, 
and therefore it
follows from
$\ODD\kappa \ddim(X)$ for any $d$.%
\footnote{See Subsection~\ref{subsection: OGD and ODD} and Lemma~\ref{comparing ODD for different dimensions}.}

Section \ref{section preliminaries} collects basic facts that we will use throughout the paper. 
Section \ref{section: ODD observations} provides tools for working with dihypergraphs and homomorphisms. 

Let $\defsetsk$
denote the class of all sets that are definable from a $\kappa$-sequence of ordinals.
The following main result of this paper 
is proved in 
Section~\ref{section: proof of the main result}. 

\pagebreak

\begin{theorem}
\label{main theorem}
Suppose $\kappa<\lambda$ are regular infinite cardinals and $2\leq\ddim\leq\kappa$.
For any $\Col(\kappa,\lle\lambda)$-generic filter $G$ over $V$, the following statements hold in $V[G]$:%
\begin{enumerate-(1)}
\item 
\label{mainthm ddim=kappa}
$\ODD\kappa\ddim(\defsets\kappa,\defsets\kappa)$ if $\lambda$ is inaccessible in $V$.%
\footnote{This is equivalent to (1) in the abstract by
Lemma~\ref{lemma: two versions of definable ODD}.}
\item 
\label{mainthm Mahlo} 
\index{Mahlo cardinal\idf} 
$\ODD\kappa\ddim(\defsets\kappa)$ if $\lambda$ is Mahlo in $V$.%
\footnote{Recall that a cardinal $\lambda$ is \emph{Mahlo} if the set of inaccessible cardinals $\nu<\lambda$ is stationary in $\lambda$.}  
\end{enumerate-(1)}
\end{theorem}

A crucial idea in the proof of the uncountable case is to construct a forcing which adds a perfect tree of generic filters over an intermediate model that induces a homomorphism of dihypergraphs. 
This technical part of the proof is done in Subsection \ref{section: proof in the uncountable case}. 

Note that $\ODD\kappa\ddim(\defsetsk,\defsetsk) \Longrightarrow \ODD\kappa\ddim(\defsetsk)$ holds for any $\ddim<\kappa$, since 
any box-open $\ddim$-dihypergraph on ${}^\kappa\kappa$ is in $\defsetsk$ by virtue of the topology having a base of size $\kappa$.%
\footnote{See Lemma~\ref{<kappa dim hypergraphs are definable}.}
It thus suffices to assume in \ref{mainthm Mahlo} that $\lambda$ is inaccessible in $V$ if $\ddim<\kappa$. 
It is open whether the Mahlo cardinal is neccessary for \ref{mainthm Mahlo} if $\ddim=\kappa$.
However, the inaccessible cardinal is necessary for \ref{mainthm ddim=kappa}, since the statement implies the $\kappa$-perfect set property $\PSP_\kappa(X)$%
\footnote{See Definition~\ref{def: perfect set}. Note that $\PSP_\kappa(X)$ from $\OGD_\kappa(X)$. To see this, consider the complete graph $\gK{X}$ on $X$.}
for all closed sets $X\subseteq{}^\kappa\kappa$ and 
thus the existence of an inaccessible cardinal above $\kappa$ in $L$.%
\footnote{See \cite{JechTrees}*{Sections~3\&4}
and \cite{lucke2012definability}*{Section~7}.
The inaccessible cardinal is also neccessary in the countable case
by \cite{MR1940513}*{Theorem 25.38}.}

We also obtain a global version of Theorem \ref{main theorem} \ref{mainthm ddim=kappa} for all regular infinite cardinals
in Subsection \ref{subsection global}. 
It remains open whether \ref{mainthm Mahlo} admits a global version.

In Section~\ref{section: stronger forms of ODD}, 
\todol{This paragraph changed slightly.}
we investigate alternatives to the open dihypergraph dichotomy 
and
show that its definition is optimal in various aspects.
It is nontrivial for any non-$\kappa$-colorable box-open dihypergraph $H$ on a subset of ${}^\kappa\kappa$.\footnote{By nontrivial, we mean that $\ODD\kappa {H\restr X}$ fails for some set $X$. 
This and the next result are shown in Subsection~\ref{subsection: examples}.} 
Moreover, one cannot replace open by closed dihypergraphs in $\ODD\kappa\ddim(X)$ in any dimension $\ddim\geq 2$. 
$\ODD\kappa\ddim(X)$ is equivalent to a dichotomy 
for $\ddim$-hypergraphs 
with the open graph dichotomy $\OGD_\kappa(X)$ as a special case for $\ddim=2$.%
\footnote{See Subsection~\ref{subsection: OGD and ODD}. 
A hypergraph is a dihypergraph that is closed under permutation of hyperedges as in Definition~\ref{def: dihypergraph}.}
While $\dhH\ddim$ witnesses that for any $\ddim\geq 3$, the existence of a continuous homomorphism cannot be replaced by the existence of a large complete subhypergraph,\footnote{See Subsection~\ref{subsection: examples} for stronger counterexamples.}  
it is natural to ask for continuous homomorphisms with additional properties such as injectivity; 
note that the homomorphism constructed in the proof of Theorem~\ref{main theorem} is injective when $\kappa$ is uncountable.\footnote{See Remark~\ref{the f constructed is injective remark 1}.}
We study this in Subsection \ref{subsection: ODD ODDI ODDH}.
While it cannot be done for $\ddim=\kappa=\omega$,\footnote{The dihypergraph in Definition~\ref{def: dhD} is a counterexample by Propositions~\ref{lemma: dhD} and~\ref{ODDI fails for D omega}.}
\todog{If $\ddim<\kappa$, then we can always choose it to be a homeomorphism onto a closed set (even when $\kappa=\omega$).}
it is possible in the uncountable setting 
assuming a weak version 
of $\Diamond_\kappa$
by the next result. 
This follows from Theorems~\ref{theorem: ODD ODDC when ddim<kappa} and~\ref{theorem: ODD ODDI}. 

\pagebreak
\begin{theorem}
\label{intro theorem: ODD ODDI}
Assume $\Diamondi\kappa$.%
\footnote{$\Diamondi\kappa$ and its relationship with some other versions of $\Diamond_\kappa$ is discussed in Subsection~\ref{subsection: diamondi}. 
Note that $\Diamondi\kappa$ implies that $\kappa$ is uncountable by Remark~\ref{diamondi omega}.}
If $H$ is a box-open $\ddim$-dihypergraph on a subset of ${}^\kappa\kappa$ where $2\leq\ddim\leq\kappa$, 
then $\ODD\kappa H$ is equivalent to the following statement:
\begin{quotation}
{$\ODDI\kappa H$:}
Either $H$ admits a $\kappa$-coloring or 
there exists an \emph{injective} continuous homomorphism 
from~$\dhH{\ddim}$ to~$H$.
\end{quotation}
\end{theorem}

For $\ddim<\kappa$, or alternatively with additional assumptions on $H$, the continuous homomorphism in $\ODD\kappa H$ can 
 be chosen to be a homeomorphism onto a closed subset even 
without assuming $\Diamondi\kappa$.%
\footnote{See Theorems~\ref{theorem: ODD ODDC when ddim<kappa},~\ref{theorem: strong variants for dhI dhth} and Definition~\ref{def: dhth}. This cannot be done in general by  Propositions~\ref{lemma: dhD} and~\ref{ODDH fails for D kappa}.}
It is open whether the assumption $\Diamondi\kappa$ in the previous theorem can be removed for $\ddim=\kappa$.
$\Diamondi\kappa$ holds for all inaccessible cardinals $\kappa$ and all successor cardinals $\kappa\geq\omega_2$ with $\kappa^{<\kappa}=\kappa$ by a result of Shelah \cite{ShelahDiamonds}.%
\footnote{Shelah \cite{ShelahDiamonds} proved that 
$\Diamond_\kappa$ is equivalent to $\kappa^{<\kappa}=\kappa$ 
for successor cardinals $\kappa\geq\omega_2$.
It is easy to show that $\Diamondi\kappa$ follows from $\Diamond_\kappa$ and that it holds at inaccessible cardinals $\kappa$; see Lemma~\ref{diamondi claim}.}
Moreover, if $\kappa$ is a regular uncountable cardinal, then $\Diamondi\kappa$ holds in $\Col(\kappa,\lle\lambda)$-generic extensions $V[G]$ where $\lambda>\kappa$ is inaccessible.
This is another way to see that 
\todol{This sentence changed due to changes in the previous paragraph.}
Theorem~\ref{main theorem} provides a model of $\ODDI\kappa\ddim$
 when $\kappa$ is uncountable.%
\footnote{See Corollary~\ref{main theorem strong version}.}

Section \ref{section: applications} studies a number of applications of the open dihypergraph dichotomy. 
We now describe the main results of the section.
The missing definitions will be introduced in the respective subsections and Section~\ref{section preliminaries}. 
The \emph{Hurewicz dichotomy} for a subset $X$ of a Polish space 
states that $X$ is contained in a $K_\sigma$ set
if and only if it does not contain a closed homeomorphic copy
of the Baire space. 
Hurewicz showed this for Polish subspaces $X$ 
\cite{Hurewicz1928}.%
\footnote{It follows immediately from \cite{Hurewicz1928}*{Section 6} by passing to a compactification that an analytic subset of a Polish space is $K_\sigma$ if and only if does not contain a relatively closed homeomorphic copy of the Baire space, but this statement does not seem to be mentioned directly in the paper. We call this the \emph{exact topological Hurewicz dichotomy} (see Section~\ref{subsection: Hurewicz}).} 
Saint Raymond \cite{SaintRaymond1975} and Kechris \cite{Kechris1977} proved it for analytic sets. 
Kechris extended 
it to more complex sets under appropriate determinacy assumptions 
 \cite{Kechris1977}. 
In the uncountable setting, the Hurewicz dichotomy has a topological version 
$\THD_\kappa(X)$
similar to the classical one \cite{LuckeMottoRosSchlichtHurewicz} as well as a combinatorial version 
$\HD_\kappa(X)$
using superperfect trees \cite{Hurewiczdef}. 
\todoo{added motivating sentence about different generalizations of $K_\sigma$} 
These correspond to different generalizations of $K_\sigma$ sets. 
The next result is proved in Theorems~\ref{theorem: THD and ODD dhth} and~\ref{theorem: dhhd from ODD} using variants of dihypergraphs from \cite{CarroyMillerSoukup}. 

\begin{theorem}
\label{theorem: Hurewicz intro}
For all subsets $X$ of ${}^\kappa\kappa$, 
$\ODD\kappa\kappa(X,\defsetsk)$ implies the following analogues of the Hurewicz dichotomy: 
\begin{enumerate-(1)}
\item
{$\THD_\kappa(X)$:}
Either 
$X$ can be covered by $\kappa$ many $\kappa$-compact%
\footnote{A topological space is \emph{$\kappa$-compact} (or \emph{$\kappa$-Lindel\"of}) if every open cover has a subcover of size strictly less than $\kappa$.} 
subsets of ${}^\kappa\kappa$, or
$X$ contains a closed homeomorphic copy of ${}^\kappa\kappa$.
\item
{$\HD_\kappa(X)$:}
Either
$X$ can be covered by the sets of branches of at most $\kappa$ many $\lle\kappa$-splitting subtrees of ${}^{<\kappa}\kappa$, or
$X$ contains a $\kappa$-superperfect subset.%
\footnote{See Definition~\ref{def: superperfect}.} 
\end{enumerate-(1)}
\end{theorem}

In fact, we obtain a generalization $\THD_\kappa(X,Y)$ of $\THD_\kappa(X)$ for all supersets $Y$ of $X$ from $\ODD\kappa\kappa(X)$.%
\footnote{See Definition~\ref{def: top Hwd} and Theorem~\ref{theorem: THD and ODD dhth}. If $Y$ is in $\defsetsk$, then $\ODD\kappa\kappa(X,\defsetsk)$ suffices.}
Here, the $K_\kappa$ set in the first option needs to be a subset of $Y$,
while it suffices that the homeomorphic copy of ${}^\kappa\kappa$ is a relatively closed subset of $Y$ in the second option.
We will use this principle to answer two questions of Zdomskyy in Section \ref{subsection: Hurewicz}. 

Kechris, Louveau and Woodin 
proved a dichotomy that characterizes when 
an analytic subset of a Polish space
can be separated from a disjoint subset by an $F_\sigma$ set \cite{KechrisLouveauWoodin1987}*{Theorem 4}. 
The special case of complements describes when an analytic set is $F_\sigma$. 
This was already shown by Hurewicz \cite{Hurewicz1928}*{Section~6}. 
The theorem of Kechris, Louveau and Woodin strengthens the Hurewicz dichotomy for analytic sets in the sense that one can derive the latter by a short argument using compactifications.%
\footnote{See \cite{KechrisBook}*{Theorem 21.22 \& Corollary 21.23}. 
Note that this argument has a version for weakly compact cardinals by
Proposition~\ref{KLW and THD for precompact X} and Corollary~\ref{cor: KLW implies THD weakly compact}, but it does not generalize to an uncountable regular cardinal that is not weakly compact.} 
Let $\RR_\kappa$ denote the set of elements of ${}^\kappa 2$ that take 
nonzero values
unboundedly often and $\QQ_\kappa$ the set of those that do not. 
The proof of the next theorem is based on ideas from \cite{CarroyMillerSoukup}.

\begin{theorem}
\label{intro thm KLW} 
Suppose that $X$ is a subset of ${}^\kappa\kappa$.
$\ODD\kappa\kappa(X)$ implies the following analogue of the Kechris--Louveau--Woodin dichotomy
for all subsets $Y$ of ${}^\kappa\kappa$ disjoint from $X$:%
\footnote{If $Y$ is in $\defsetsk$, then $\ODD\kappa\kappa(X,\defsetsk)$ suffices.}
\begin{quotation}
$\KLW_\kappa(X,Y)$: 
Either there is a $\Fsigma(\kappa)$ set $A$ separating $X$ from $Y$,%
\footnote{I.e., $X\subseteq A$ and $A\cap Y=\emptyset$.} 
or there is a homeomorphism 
$f$ from ${}^\kappa 2$ onto a closed subset of ${}^\kappa\kappa$
that reduces
$(\RR_\kappa,\QQ_\kappa)$ to $(X,Y)$.%
\footnote{I.e., $f(\RR_\kappa)\subseteq X$ and $f(\QQ_\kappa)\subseteq Y$.}
\end{quotation} 
\end{theorem}

We extend the previous theorem to sets $X$ and $Y$ that are not neccessarily disjoint
in Theorem~\ref{theorem: KLW from ODD}.
Here, the first option is replaced by the statement
that $X$ is the a union of $\kappa$-many sets $X_\alpha$ with no limit points in $Y$.\footnote{See Definition~\ref{def: KLW dichotomy}. 
Note that  by Lemma~\ref{lemma: CC coloring}, this statement is equivalent to the conjunction of the statements that $|X\cap Y|\leq\kappa$ and $X\setminus Y$ can be separated from $Y$ by a $\Fsigma(\kappa)$ set.} 
\todoo{added motivation to this sentence} 
The next, even more general, dichotomy $\KLW_\kappa^H(X,Y)$ arose when we studied the strength of $\ODD\kappa\kappa({}^\kappa\kappa)$. 
It is defined 
relative to an arbitrary box-open $\ddim$-dihypergraph $H$ where $2\leq\ddim\leq\kappa$
in Theorem~\ref{theorem: ODD and KLW^H}. 
In the first option of the dichotomy, the role of limit points is replaced by a notion of
limit points relative to $H$.\footnote{See Definition \ref{def: H-limit point}.} 
In the second option, $\RR_\kappa$ and $\QQ_\kappa$ are replaced by analogous subsets $\RR^\ddim_\kappa$ and $\QQ^\ddim_\kappa$ of ${}^\kappa\ddim$ and moreover, 
the map needs to be a continuous homomorphism from $\dhH \ddim$ to $H$ instead of a homeomorphism from ${}^\kappa 2$ onto a closed set.%
\footnote{See Definition~\ref{def: KLW for hypergraphs} and the preceding paragraph. Note that in the second option of $\KLW_\kappa(X,Y)$, it is equivalent to ask that $f$ is continuous and injective instead of a homeomorphism by Theorem~\ref{theorem: KLW from ODD}. This shows that $\KLW_\kappa(X,Y)$ is equivalent to the special case of $\KLW_\kappa^{\gK{{}^\kappa \kappa}}(X,Y)$ for the complete graph $\gK{{}^\kappa \kappa}$ on ${}^\kappa\kappa$ by Remark~\ref{remark: KLW^G for graphs G}.} 
$\KLW^\ddim_\kappa(X)$ then states that $\KLW_\kappa^H(X,Y)$ holds for all subsets $Y$ of ${}^\kappa\kappa$
and all box-open $\ddim$-dihypergraphs $H$ on ${}^\kappa\kappa$.
The next result follows from Theorem~\ref{theorem: ODD and KLW^H} and 
Corollaries~\ref{cor: from KLW to ODD} and \ref{cor: from KLW to ODD kappa}, 
as does its restriction to definable dihypergraphs $H$ 
and definable subsets $Y$.

\pagebreak

\begin{theorem}\ 
\label{thm: KLW ODD intro}
The following hold for all $2\leq\ddim\leq\kappa$ and all subsets $X$ of ${}^\kappa\kappa$.
\begin{enumerate-(1)}
\item
$\ODD\kappa\kappa(X) \Longrightarrow 
\KLW_\kappa^\ddim(X)$. 
\item
$\KLW_\kappa^\ddim(X) \Longrightarrow \ODD\kappa\ddim(X)$.
\end{enumerate-(1)}
In particular, 
$\ODD\kappa\kappa(X)$
is equivalent to 
$\KLW^\kappa_\kappa(X)$.
\end{theorem}

This general version of the Kechris--Louveau--Woodin dichotomy
is used 
to prove the next result. 
It follows from Lemma~\ref{ODD for continuous images} and 
Corollary~\ref{cor: ODD(X) implies ODD(closed subsets of X)}. 

\begin{theorem}\ 
\label{thm: ODD for kappa^kappa and analytic sets}
\begin{enumerate-(1)}
\item $\ODD\kappa\kappa({}^\kappa\kappa)\Longrightarrow
\ODD\kappa\kappa(\analytic(\kappa))$.
\item $\ODD\kappa\kappa({}^\kappa\kappa,\defsetsk)\Longrightarrow
\ODD\kappa\kappa(\analytic(\kappa),\defsetsk)$.
\end{enumerate-(1)}
\end{theorem}
Note that it is easy to see that $\ODD\kappa\kappa({}^\kappa\kappa) \Longrightarrow\ODD\kappa\kappa(X)$ for continuous images $X$ of~${}^\kappa\kappa$,\footnote{See Lemma \ref{ODD for continuous images}.}  
but not all closed subsets of ${}^\kappa\kappa$ are of this form \cite{MR3430247}*{Theorem~1.5}. 
It follows from the previous theorem
that $\ODD\kappa\kappa({}^\kappa\kappa,\defsetsk)$ has the same consistency strength as an inaccessible cardinal.

$\ODD\kappa\kappa({}^\kappa\kappa)$ has further applications to \emph{V\"a\"an\"anen's perfect set game} that is described next. 
Any set of reals can be decomposed as the disjoint union of a crowded set and a countable scattered set by Cantor--Bendixson analysis.%
\footnote{Recall that a set of reals
$X$ is \emph{crowded} if it has no isolated points, \emph{perfect} if it is closed and crowded, and \emph{scattered} if each nonempty subspace contains an isolated point.}
V\"a\"an\"anen showed that an analogue of this statement for all closed subsets of ${}^\kappa\kappa$ is consistent%
\footnote{V\"a\"an\"anen used a measurable cardinal
and Galgon reduced this to an inaccessible cardinal \cite{GalgonThesis}.
In fact, the analogue for closed sets is equivalent to the $\kappa$-perfect set property $\PSP_\kappa(X)$ for closed sets \cite{SzThesis}.}
but not provable \cite{VaananenCantorBendixson}. 
To this end, he generalized the notions of scattered and crowded sets  
to subsets $X$ of ${}^\kappa\kappa$
via a game $\GV_\kappa(X,x)$ of length $\kappa$.%
\footnote{See Definition~\ref{def: Vaananen's game}. 
Here, $x\in X$ denotes the first move of player $\ptwo$; this is fixed. We also consider a global version $\GV_\kappa(X)$ of the game, where player $\ptwo$ may play any element of $X$ as their first move.} 
We extend V\"a\"an\"anen's result to all subsets of ${}^\kappa\kappa$ in the next theorem, which follows from 
Theorems~\ref{theorem: ODD implies CB1} and~\ref{prop: cbi iff cbii}. 

\begin{theorem}
\label{intro theorem: CB}
$\ODD\kappa\kappa({}^\kappa\kappa)$ implies the following version of the Cantor--Bendixson decomposition 
for all subsets $X$ of ${}^\kappa\kappa$:%
\footnote{$\ODD\kappa\kappa({}^\kappa\kappa,\defsetsk)$ suffices if $X$ is in $\defsetsk$.}
\begin{quotation}
$\cbii(X)$:
$X$ equals the disjoint union of a $\kappa$-scattered set of size $\kappa$ and a $\kappa$-crowded set.%
\footnote{See Definition~\ref{def: kernel}.}
\end{quotation}
In particular, 
V\"a\"ananen's perfect set game $\GV_\kappa(X,x)$
is determined for all subsets $X$ of ${}^\kappa\kappa$ and all $x\in X$.%
\footnote{It is easy to see from the definitions that this statement follows from $\cbii(X)$. Note that this statement implies the determinacy of the global version $\GV_\kappa(X)$ of V\"a\"ananen's game as well.}
%
\end{theorem}

While there is a verbatim analogue of the Baire property at uncountable cardinals, it is not as useful as in the countable setting since even simple sets such as the club filter do not satisfy it \cite{MR1880900}.%
${}$\footnote{See the beginning of Subsection~\ref{subsection: almost Baire property}. Note that the club filter can be coded as a $\kappa$-analytic subset of the $\kappa$-Baire space.}
\todon{Footnote size problem fixed in the usual way}
The \emph{asymmetric $\kappa$-Baire property} $\ABP_\kappa(X)$ introduced in \cite{SchlichtPSPgames} is a more general principle
that is equivalent to the determinacy of the Banach-Mazur game of length $\kappa$ for $X$. 
In the countable setting, $\ABP_\omega(\defsets\omega)$%
\footnote{I.e., $\ABP_\omega(X)$ for all subsets $X\in\defsets\omega$ of ${}^\omega\omega$.}  
is thus equivalent to the Baire property for sets in $\defsets\omega$. 
\todog{We wrote $\defsets\omega$ instead of $X$ here since. The equivalence only holds for sufficiently closed classes of subsets of ${}^\omega\omega$.}
The next result is proved in Theorem \ref{theorem: ABP from ODD}. 

\begin{theorem}
\label{intro them Baire} 
For all subsets $X$ of ${}^\kappa\kappa$,
$\ODD\kappa\kappa(X,\defsetsk)$ implies the asymmetric $\kappa$-Baire property $\ABP_\kappa(X)$. 
\end{theorem}

A new consequence in the countable setting is that 
$\ODD\omega\omega(\defsets\omega,\defsets\omega)$ implies the Baire property for subsets of the Baire space in $\defsets\omega$.

A celebrated result of Jayne and Rogers characterizes ${\bf \Delta}^0_2$-measurable functions on analytic sets 
as those that are $\sigma$-continuous with closed pieces \cite{JayneRogers1982}*{Theorem 5}.%
\footnote{A function between countably based metric spaces is ${\bf \Delta}^0_2$-measurable if and only if it is ${\bf \Pi}^0_2$-measurable.
It is \emph{$\sigma$-continuous with closed pieces} if it can be obtained as a countable union of continuous functions on relatively closed subsets of its domain.}
A simpler proof of this theorem was subsequently discovered by Motto Ros and Semmes \cites{ros2010new,kacenamottorossemmes}.
Possible generalizations thereof are a subject of intense study \cites{ros2013structure, gregoriades2021turing}. 
Recently, Carroy, Miller and Soukup found a new proof that derives the theorem from the open dihypergraph dichotomy \cite{CarroyMillerSoukup}. 
Their proof allows us to generalize the Jayne--Rogers theorem to the uncountable setting.
The next result is proved in Corollary~\ref{cor: Jayne Rogers}.

\begin{theorem}
\label{intro theorem: JR}
$\ODD\kappa\kappa(\defsetsk,\defsetsk)$ implies the following analogue of the Jayne--Rogers theorem for all subsets $X\in\defsetsk$ of ${}^\kappa\kappa$: 
\begin{quotation}
$\JR_\kappa(X)$: 
Any function $f:X\to{}^\kappa\kappa$
is $\DeltaZeroTwo(\kappa)$-measurable if and only if it is a union of $\kappa$ many continuous functions on relatively closed subsets of $X$. 
\end{quotation}
\end{theorem}

At the end of Section \ref{section: applications}, 
\todol{This paragraph changed due to changes in Section 6.6}
we prove some implications and separations
between applications of the open dihypergraph dichotomy such as
the above versions of the Hurewicz dichotomy,
the Kechris--Louveau--Woodin dichotomy and its generalization to arbitrary sets, 
the $\kappa$-perfect set property, 
and the analogue of the Jayne--Rogers theorem. 
Several of these results rely on the comparison relation between dihypergraphs introduced in Subsection~\ref{subsection: full dihypergraphs}. 
For example, the next results follow as special cases from 
Propositions~\ref{connection btw KLW, KLWdis and PSP} 
\ref{prop: PSP KLW new},
and
\ref{prop: PSP THD new} 
and
Corollaries~\ref{cor: THD implies JR weakly compact} and~\ref{cor: KLW to THD new}.

\begin{theorem}\ 
\label{intro theorem on implications}
\begin{enumerate-(1)}
\item
\label{intro theorem on implications 1}
$\KLW_\kappa(\defsetsk,\defsetsk)
\Longleftrightarrow\KLWdis_\kappa(\defsetsk,\defsetsk)\land\PSP_\kappa(\defsetsk)
\Longrightarrow\THD_\kappa(\defsetsk,\defsetsk)$.%
\footnote{$\KLW_\kappa(\mathcal C,\mathcal D)$ states that $\KLW_\kappa(X,Y)$ holds for all subsets $X\in\mathcal C,\,Y\in\mathcal D$ of ${}^\kappa\kappa$, while
$\KLWdis_\kappa(\mathcal C,\mathcal D)$ states this for all such \emph{disjoint subsets}. 
Thus, $\KLWdis_\kappa(\mathcal C):=\KLWdis(\mathcal C,\pwrset({}^\kappa\kappa))$ 
is a direct analogue of the Kechris--Louveau--Woodin dichotomy
for classes $\mathcal C$ of subsets of the $\kappa$-Baire space.
$\THD_\kappa(\mathcal C,\mathcal D)$, $\PSP_\kappa(\mathcal C)$ and $\JR_\kappa(\mathcal C)$ are defined similarly in Definitions~\ref{def: perfect set}, \ref{def: top Hwd} and~\ref{def: JR}.}%
${}^{\text{,}}$%
\footnote{Note that if $\kappa$ is not weakly compact, then
$\PSP_\kappa(\defsetsk)\Longrightarrow\THD_\kappa(\defsetsk,\pwrset({}^\kappa\kappa))$ 
by \cite{LuckeMottoRosSchlichtHurewicz}*{Proposition~2.7};
see also Proposition~\ref{remark: PSP THD non weakly compact}.
This follows from the fact that ${}^\kappa 2$ is homeomorphic to ${}^\kappa\kappa$ by \cite{hung1973spaces}*{Theorem 1}.}
\item\label{intro theorem on implications new KLWdis}
$\KLW_\kappa(\defsetsk,\defsetsk)\Longleftrightarrow\KLWdis_\kappa(\defsetsk,\defsetsk)$
if the $\kappa$-perfect set property holds for closed sets.%
\item\label{intro theorem on implications new THD}
$\THD_\kappa(\defsetsk,\defsetsk)\Longrightarrow\PSP_\kappa(\defsetsk)$
if the $\kappa$-perfect set property holds for $\kappa$-compact sets.%
\item
\label{intro theorem on implications 2}
$\THD_\kappa(\defsetsk,\defsetsk)\Longrightarrow\JR_\kappa(\defsetsk)$ if $\kappa$ is weakly compact.
\end{enumerate-(1)}
\end{theorem}

The above versions of the Hurewicz dichotomy, 
the asymmetric $\kappa$-Baire property and the Jayne--Rogers theorem for definable subsets of ${}^\kappa\kappa$
are consistent relative to an inaccessible cardinal by 
Theorem~\ref{main theorem}. 
The first two statements provide new proofs of the main results of \cites{SchlichtPSPgames, Hurewiczdef} and answer two questions of Zdomskyy \cite{zdomskyy}.\footnote{See \cite{zdomskyy}*{pages 6-7}.} 
The variant $\KLW_\kappa(\defsets\kappa):=\KLW_\kappa(\defsetsk,\powerset({}^\kappa\kappa))$ of the Kechris--Louveau--Woodin dichotomy,
its extension $\KLW^\ddim_\kappa(\defsetsk)$
relative to box-open $\ddim$-dihypergraphs for $2\leq\ddim\leq\kappa$,
the determinacy of V\"a\"an\"anen's perfect set game and 
in fact $\cbii(X)$ 
for arbitrary subsets $X$ of ${}^\kappa\kappa$
are consistent relative to a Mahlo cardinal. 
The former answers a question of L\"ucke, Motto Ros and the first-listed author and solves a recent problem of Bergfalk, Chodounsk\'y, Guzm\'an and Hru\v s\'ak.\footnote{See Remark~\ref{remark: Sigma02 complete}.}
The latter extends a result of V\"a\"an\"anen \cite{VaananenCantorBendixson} from closed sets to arbitrary subsets of~${}^\kappa\kappa$. 
Moreover, it lowers the upper bound for the consistency strength of a dichotomy studied in \cite{SzVaananen} 
from a measurable to a Mahlo cardinal.

\newpage 

\section{Preliminaries} 
\label{section preliminaries} 

\numberwithin{theorem}{subsection} 
\subsection{Notation}
\label{preliminaries: notation}

$\alpha$, $\beta$, $\gamma$, $\delta$ denote ordinals and $\kappa$, $\lambda$, $\mu$, $\nu$ cardinals. 
$\kappa$ always denotes an infinite cardinal with $\kappa^{<\kappa}=\kappa$, unless otherwise stated.
$\Lim$ and $\Succ$ denote the classes of limit and successor ordinals, respectively. 
If $\alpha=\beta+1$, then $\alpha-1$ denotes $\beta$. 
For any set $\ddim$,
a \emph{$\ddim$-sequence} is a function $f: \ddim\to V$. 
We write $\ddim$-sequences as $\bar{x}=\langle x_i : i\in\ddim\rangle$. 
$\langle c\rangle^{\ddim}$ denotes the constant 
$\ddim$-sequence with value $c$. 
We write $\bar x^\pi:=\langle x_{\pi(i)}:i\in\ddim\rangle$ for any $D$-sequence $\bar x=\langle x_i:i\in D\rangle$ and $\pi: d\rightarrow D$, where $D$ is any set. 
A subsequence of $\bar x=\langle x_i:i\in \alpha\rangle$ is $\bar x^\pi=\langle x_{\pi(i)}:i\in \beta\rangle$ for some $\beta\leq\alpha$ and some strictly increasing map $\pi:\beta\to\alpha$. 
\index{sequences!coordinatewise restriction \idf $\bar x\cwrestr\alpha$}
For any set $\ddim$,
$\alpha<\kappa$ and any sequence 
$\bar x=\langle x_i:i\in\ddim \rangle\in{}^\ddim{({}^\kappa\kappa)}$, 
let $\bar x\cwrestr\alpha:=\langle x_i\restr\alpha:i\in\ddim\rangle$. 
\index{sequences!concatenation\idf $s\conc t$, $\bigoplus_{\alpha<\beta} s_\alpha$}
$s\conc t$ denotes the concatenation of two sequences $s$ and $t$,
and $\bigoplus_{\alpha<\beta} s_\alpha$ denotes the concatenation of the sequences
$s_\alpha$ for $\alpha<\beta$.
Given a function $f:X\to Y$, write $f(U)$ for the pointwise image of a subset $U$ of $X$ under $f$, and
write $f^{-1}(W)$ for the preimage of a subset $W$ of $Y$.
For any set $\ddim$, let $f^\ddim:{}^\ddim X\to{}^\ddim Y$ denote the function defined by letting
$f^\ddim(\langle x_i:i\in \ddim\rangle) := 
\langle f(x_i):i\in\ddim\rangle$
for all $\langle x_i:i\in\ddim\rangle\in{}^\ddim X$. 
$\id_X$ denotes the identity function on $X$.

A class $X$ is \emph{definable from a set $y$} if $X=\{x : \varphi(x,y)\}$ for some first order formula $\varphi(v_0,v_1)$ with two free variables. 
\index{definable sets\idf $\defsetsk$, $X^\kappa_{\varphi,y}$} 
Let $\defsets\kappa$ denote the class of those sets which are definable from a $\kappa$-sequence of ordinals. 
By a \emph{definable} set, we always mean an element of $\defsetsk$ when $\kappa$ is clear from the context.
We use the following notation for definable subsets of the $\kappa$-Baire space:
If $\varphi(v_0,v_1)$ is a first order formula with two free variables and $y$ is a set, write  
$X^\kappa_{\varphi,y}:=\{x\in{}^{\kappa}\kappa: \varphi(x,y)\}$.  
We further write $X_{\varphi,y}$ for $X^\kappa_{\varphi,y}$ if $\kappa$ is clear from the context. 

For any discrete topological space $\ddim$, ${}^\kappa\ddim$ 
denotes the space of functions from $\kappa$ to $\ddim$ equipped with 
\index{space of functions\idf${}^\kappa\ddim$}%
\index{topology!bounded\idf}
the \emph{bounded topology} 
\index{topology!box@$\lle\kappa$-box\idf} 
(or \emph{$\lle\kappa$-box topology}),
i.e., the topology generated by the base $\{N_t:t\in{}^{<\kappa}\ddim\}$, where
$N_t=\{x\in{}^\kappa\ddim : t\subseteq x\}$ 
for each $t\in{}^{<\kappa}\ddim$.%
\footnote{In particular, ${}^\kappa\kappa$ denotes the $\kappa$-Baire space and ${}^\kappa 2$ denotes the $\kappa$-Cantor space. Moreover, if $\kappa=\omega$, then the bounded topology equals the product topology.}
\todog{By definition, the \emph{$\lle\kappa$-box topology} is generated by basic open sets of the form $U:=\prod_{i\in \kappa} U_i$, where each $U_i$ is an open subset of $\ddim$ and $U_i\neq\ddim$ only $\lle\kappa$ many times. If $\ddim$ is a discrete space, then the two topologies are equal: since $N_t=\prod_{i<\lh(t)}\{t(i)\}$, the bounded topology is coarser that the $\lle\kappa$-box topology. Conversely, the set $U:=\prod_{i<\kappa} U_i$ is the union of all those $N_t$ with $t\in\prod_{i<\alpha} U_i$, where $\alpha=\sup\{i:U_i\neq\ddim\}$, and hence it is open in the bounded topology as well.}
Ordinals are always equipped with the discrete topology. 
Suppose that $X$ is a topological space. 
The closure of a subset $Y$ of $X$ is denoted by $\closure Y$.
$Y'$ denotes the set of \emph{limit points} of $Y$,
\index{limit pointsa@limit points\idf $X'$}
i.e., the set of those $x\in X$ 
with $x\in\closure{Y\setminus\{x\}}$.
The \emph{$\kappa$-Borel} subsets of $X$ are defined by closing the set of basic open sets under unions of length $\kappa$ and complements. 
The first level of the 
\index{Borel@$\kappa$-Borel levels!additive\idf${\bf\Sigma}^0_\gamma(\kappa)$} 
\index{Borel@$\kappa$-Borel levels!multiplicative\idf${\bf\Pi}^0_\gamma(\kappa)$} 
\index{Borel@$\kappa$-Borel levels!selfdual\idf${\bf\Delta}^0_\gamma(\kappa)$} 
$\kappa$-Borel hierarchy is the class ${\bf\Sigma}^0_1(\kappa)$ of open sets. 
For any $\gamma$ with $1<\gamma<\kappa^+$, ${\bf\Sigma}^0_\gamma(\kappa)$ sets are of the form $\bigcup_{\alpha<\kappa} A_\alpha$, where each $A_\alpha$ is in ${\bf\Pi}^0_\beta(\kappa)$ for some $\beta<\gamma$. 
For any $\gamma$ with $0<\gamma<\kappa^+$, ${\bf\Pi}^0_\gamma(\kappa)$ sets are complements of ${\bf\Sigma}^0_\gamma(\kappa)$ sets, and 
a set is ${\bf\Delta}^0_\gamma(\kappa)$ if it is both ${\bf\Sigma}^0_\gamma(\kappa)$ and ${\bf\Pi}^0_\gamma(\kappa)$. 
The \emph{$\kappa$-analytic} 
\index{analytic@$\kappa$-analytic\idf $\analytic(\kappa)$}
or $\analytic(\kappa)$ subsets of ${}^\kappa\kappa$ 
are subsets of the form $f(Y)$, where $f: {}^\kappa\kappa\to {}^\kappa\kappa$ is continuous and $Y$ is a closed subset of ${}^\kappa\kappa$. 
A $\coanalytic(\kappa)$ set is a complement of a $\analytic(\kappa)$ set. 
For collections $\mathcal C$ and $\mathcal D$ of subsets of $X$, let
\index{meet of!classes\idf$\mathcal C\wedge\mathcal D$}
$\mathcal C\land\mathcal D$ denote the class of all sets of the form $Y\cap Z$ where $Y\in\mathcal C$ and $Z\in\mathcal D$, 
and
\index{dual class of $\mathcal C$\idf$\check{\mathcal C}$}
let $\check{\mathcal C}$ denote the class of complements of elements of $\mathcal C$.
A function $f$ from  $X$ to a topological space $Y$ is called \emph{$\mathcal{C}$-measurable} 
\index{function!measurable@$\mathcal{C}$-measurable\idf}
if $f^{-1}(U) \in\mathcal{C}$ for every open subset $U$ of~$Y$.

A topological space is 
\index{compact@$\kappa$-compact\idf} 
\emph{$\kappa$-compact} (or 
\index{Lindelöf@$\kappa$-Lindelöf space\idf} 
\emph{$\kappa$-Lindel\"of}) if every open cover has a subcover of size strictly less than $\kappa$. 
It is \emph{$K_\kappa$} 
\index{Kkappa@$K_\kappa$ space\idf}
if it is a union of $\kappa$ many $\kappa$-compact subspaces. 
A subset $X$ of ${}^\kappa\kappa$ is 
\index{precompact@$\kappa$-precompact\idf} 
\emph{$\kappa$-precompact} if its closure is $\kappa$-compact. 
\index{bounded subset of ${}^\kappa\kappa$\idf} 
$X$ is \emph{bounded} by
$b\in{}^\kappa\kappa$ 
if $x(\alpha)<b(\alpha)$ for all $x\in X$ and $\alpha<\kappa$. 
Write $x\leq^* y$ if there is some $\alpha<\kappa$ such that $x(\beta)<y(\beta)$ for all $\alpha<\beta<\kappa$. 
A subset $X$ of ${}^\kappa\kappa$ is 
\index{eventually bounded subset of ${}^\kappa\kappa$\idf} 
\emph{eventually bounded} by $b\in{}^\kappa\kappa$ if $x\leq^* b$ for all $x\in X$. 
An uncountable cardinal $\kappa$ is 
\index{weakly compact cardinal\idf} 
weakly compact if and only if 
${}^\kappa 2$ is $\kappa$-compact, and if and only if
every closed and bounded subset of ${}^\kappa\kappa$ is $\kappa$-compact.%
\footnote{See \cite{hung1973spaces}*{Theorem 1} and 
\cite{LuckeMottoRosSchlichtHurewicz}*{Lemma~2.6}. Conversely, it is easy to see that $\kappa$-compact subsets of ${}^\kappa\kappa$ are always closed and bounded for all infinite cardinals $\kappa$ with $\kappa^{<\kappa}=\kappa$.}

\subsection{Dihypergraphs} 
\label{subsection: dihypergraphs}
In this subsection, $X$ denotes a set and $\ddim$ a set of size at least $2$.
$\dhC\ddim{X}$ \index{diagonal\idf$\dhC\ddim{X}$} denotes the \emph{diagonal}, i.e., the set of constant sequences
$\langle c\rangle^{\ddim}$ with $c\in X$. 
Let $\Sym(\ddim)$ \index{symmetric group\idf $\Sym(\ddim)$} denote the \emph{symmetric group} of all permutations of $\ddim$. 

\begin{definition}\ 
\label{def: dihypergraph}
\begin{enumerate-(a)} 
\item\label{def: dihypergraph ddim}
A \emph{$\ddim$-dihypergraph}\footnote{This is short for \emph{$\ddim$-dimensional directed hypergraph}.} 
\index{dihypergraph!$\ddim$-dihypergraph\idf} 
on $X$ is a subset $H$ of ${}^\ddim X\,\setminus\,\dhC\ddim{X}$. 
Elements of $H$ are called \emph{hyperedges}\index{hyperedge\idf} of $H$.
\item\label{def: dihypergraph hypergraph}
A \emph{$\ddim$-hypergraph}\index{hypergraph!$\ddim$-hypergraph\idf} 
on $X$ is a $\ddim$-dihypergraph $H$ on $X$ that is closed under permutation of hyperedges, i.e., $\bar x^\pi\in H$ for all $\bar x\in H$ and all $\pi\in \Sym(\ddim)$.
\end{enumerate-(a)} 
\end{definition}

A \emph{digraph}\index{digraph\idf} on $X$ is then a $2$-dihypergraph on $X$, while a \emph{graph}
\index{graph!A@on a set\idf}%
on $X$ is a $2$-hypergraph on $X$, i.e., a symmetric irreflexive binary relation on $X$.

\begin{definition}
\ 
\begin{enumerate-(a)}
\item
$\dhK\ddim X:= 
{}^\ddim X-\dhC\ddim{X}$ 
is the \emph{complete $\ddim$-hypergraph on $X$}.
\index{hypergraph!complete $\ddim$-hypergraph\idf $\dhK\ddim X$}
\item 
$\gK X:=\dhK 2 X$ is the \emph{complete graph} on $X$. 
\index{graph!complete graph on!$X$\idf$\gK X$}
\item
$\dhII\ddim X$ 
\index{hypergraph!of sequences!IIX@injective in $X$\idf$\dhII\ddim X$} 
denotes the $\ddim$-hypergraph consisting of all 
injective elements of ${}^\ddim X$.
\end{enumerate-(a)}
\end{definition}

Note that $\dhII 2 X=\gK X$, but $\dhII\ddim X\subsetneq \dhK\ddim X$ if $\ddim\geq 3$ and $|X|\geq 2$.

\begin{definition}
\label{def: independent}
\label{def: coloring}
\label{def: homomorphism}
\label{def: clique}
Let $H$ and $I$ be $\ddim$-dihypergraphs on $X$ and $Y$, respectively. Let $Z$ be a subset of $X$. 
\begin{enumerate-(a)}
\item The \emph{restriction} \index{dihypergraph!restriction to a set\idf $H\restr X$}
of $H$ to $Z$ is $H\restr Z:=H\cap{}^\ddim Z$.
\item
$Z$ is \emph{$H$-independent} \index{independent@$H$-independent set\idf}
if $H\restr Z=\emptyset$.
\item 
$Z$ is an \emph{$H$-clique} \index{clique@$H$-clique\idf}
if $\dhK \ddim Z \subseteq H$.
\item
Let $\lambda$ be a cardinal. A function $c:X\to\lambda$
is a \emph{$\lambda$-coloring} \index{coloring@$\lambda$-coloring\idf} 
of $H$ if $c^{-1}(\{\alpha\})$ is $H$-independent for all $\alpha<\lambda$.
\item
A \emph{homomorphism} \index{homomorphism!between dihypergraphs \idf}
from $H$ to $I$ is a function $f:X\to Y$
such that $f^\ddim(H)\subseteq I$.
\end{enumerate-(a)}
\end{definition}

Given a family of topological spaces $\langle X_i : i\in \ddim\rangle$,
the \emph{box-topology}\index{topology!box\idf}
on $\prod_{i\in \ddim} X_i$ is the topology generated by sets of the form
$\prod_{i\in \ddim} U_i$, where $U_i$ is an arbitrary open subset of $X_i$ for each $i\in \ddim$.
If $X$ is a subset of ${}^\kappa\kappa$, then
we always work with the {box topology} on ${}^\ddim X$ unless otherwise stated. 
\todog{We sometimes also consider the \emph{product topology}, but this is always explicitly stated, and ``product-open'' etc are defined in footnotes, so we don't need to mention it here.}%

In the following, let $H$ be a $\ddim$-dihypergraph on a topological space $X$.
We call $H$
\emph{box-open} 
\index{dihypergraph!box-open\idf}
on $X$ if it is open as a subset of 
${}^\ddim X$ with the box-topology.
{For finite $\ddim$, we will simply say that $H$ is \emph{open} on $X$.}
\todon{pagebreak inserted}

\pagebreak

\begin{definition}
\label{def: domain of H}
\label{def: relatively box-open}
\ 
\begin{enumerate-(a)}
\item
The \emph{domain} $\domh H$ \index{dihypergraph!domain\idf $\domh H$}
of $H$ is the set of all $x\in X$ such that $x$ is an element of some hyperedge of $H$.%
\footnote{More precisely, $\domh H=\bigcup_{i\in\ddim} \proj_i(H)$, where $\proj_i$ denotes projection onto the $i^\textrm{th}$ coordinate.}
\item 
$H$ is \emph{relatively box-open}
\index{relatively box-open\idf} 
\index{dihypergraph!relatively box-open\idf} 
if it is a box-open dihypergraph on its domain $\dom(H)$.
\end{enumerate-(a)}
\end{definition}

Note that $H$ is box-open on $X$ if and only if $H$ is relatively box-open and $\domh H$ is a relatively open subset of $X$.

\subsection{Trees and order preserving maps}
\label{preliminaries: trees}

Let $\ddim$ and $D$ denote discrete topological spaces of size at least $2$ throughout this subsection. 
Suppose $s,t,u\in{}^{\leq\kappa}\ddim$ and $A\subseteq{}^{\leq\kappa}\ddim$. 
\index{length\idf $\lh(t)$, $\lh(A)$}
Let $\lh(t):=\dom(t)$ and $\lh(A):=\sup\{\lh(u):u\in A\}$. 
\index{predecessors of $t$\idf$\kkppred t$}
Let $\kkppred t:=\{s\in{}^{<\kappa}\ddim: s\subsetneq t\}$ 
\index{successors of $t$\idf$\succ t$}
and $\succ{t}:=\{u\in {}^{<\kappa}\ddim: t\subsetneq u\}$. 
\todoo{new sentence: defined direct predecessor and direct successor} 
$s$ is called a \emph{direct predecessor} of $u$ and $u$ a \emph{direct successor} of $s$ if $s\subsetneq u$ and there is no $s\subsetneq t\subsetneq u$. 
$s$ and $t$ are called \emph{compatible},
denoted $s \compat t$, if $s\subseteq t$ or $t\subseteq s$.%
\footnote{For general posets $\PP$, this is called \emph{comparable}. Conditions $p,q\in \PP$ are called \emph{compatible} if there exists a common extension $r\leq p,q$.} 
$s$ and $t$ are called \emph{incompatible}, denoted $s \perp t$, if they are not compatible. 
\index{meet of!sequences\idf$s\wedge t$}
Let $s\wedge t$ denote the maximal $r\in {}^{\leq\kappa}\ddim$ with $r\subseteq s$ and $r\subseteq t$. 
Note that $s\wedge t$ is the node where $s$ and $t$ split if $s\perp t$, and $s\wedge t=s$ if $s\subseteq t$. 
For incompatible $s$ and $t$, let 
\index{distance\idf$\Delta(s,t)$}
$\Delta(s,t):=\min\{ \alpha<\kappa : s(\alpha)\neq t(\alpha)\}$. 
Then $\Delta(s,t)=\lh(s\wedge t)$ and $s\wedge t= s\restr\Delta(s,t)=t\restr\Delta(s,t)$.

A \emph{subtree} \index{tree!subtree\idf}
of ${}^{<\kappa}\ddim$ is a subset $T$ that is closed under forming initial segments (i.e., it is downward closed). 
\index{tree!body\idf$[T]$}%
\index{tree!branch\idf}%
$[T]:=\{ x\in {}^\kappa\ddim :\: {\forall \alpha<\kappa}\ \ {x{\restr}\alpha \in T}\}$
denotes its \emph{body}; we call elements of $[T]$ \emph{branches} of $T$.
Note that $[T]$ is a closed subset of ${}^\kappa\ddim$.
The \emph{height} \index{tree!height\idf$\height(T)$}
of $T$ is defined as $\height(T):=\sup\{\lh(t)+1:t\in T\}$. 
Thus, $\height(T)=\lh(T)$ if $\height(T)\in\Lim\cup \{0\}$ and $\height(T)=\lh(T)+1$ if $\height(T)\in\Succ$. 
The \emph{stem} \index{tree!stem\idf}
of $T$ is 
the unique maximal node $s\in T$ with $s\compat t$ for all $t\in T$. 

If $A$ is a subset of ${}^{\leq\kappa}\ddim$, let 
\index{tree!of initial segments\idf$T(A)$}
$T(A):=\{t\in{}^{<\kappa}\ddim :\: {\exists a\in A}\ \ {t\subseteq a} \}$
\todog{$T(A)$ is used for $A\subseteq{}^{<\kappa}\ddim$ in Remark~\ref{ran[e] not closed} and Def~\ref{T(e) def}}%
denote the tree of initial segments of its elements.
Note that if $X\subseteq{}^\kappa\ddim$, 
then
$[T(X)]$ is the closure of~$X$.

\begin{definition}
\label{def: perfect set} 
Suppose $T$ is a subtree of ${}^{<\kappa}\ddim$. 
\begin{enumerate-(a)} 
\item 
\index{tree!pruned\idf}
$T$ is \emph{pruned} if every node $t\in T$ extends to a branch $x\in [T]$. 
\item 
\index{node!splitting A@splitting\idf}
A node $t$ in $T$ is \emph{splitting} if it has at least two direct successors in $T$.
\item 
\index{tree!cofinally splitting\idf}
$T$ is \emph{cofinally splitting} if for each $t\in T$, there exists a splitting node $u$ in $T$ with $t\subseteq u$. 
\item 
\index{tree!closed@${<}\kappa$-closed\idf}%
\index{closed tree@${<}\kappa$-closed tree\idf}
$T$ is \emph{${<}\kappa$-closed} if for any strictly increasing sequence $\bar{t}=\langle t_i : i<\alpha \rangle$ in $T$ with $\alpha<\kappa$, there exists an upper bound $t\in T$ for $\bar{t}$. 
\item 
\index{tree!perfect@$\kappa$-perfect\idf}%
\index{perfect@$\kappa$-perfect!tree\idf}
$T$ is \emph{$\kappa$-perfect} if it is cofinally splitting and ${<}\kappa$-closed. 
\item 
\index{perfect@$\kappa$-perfect!set\idf}
A closed subset $X$ of ${}^\kappa\ddim$ is \emph{$\kappa$-perfect} if $T(X)$ is $\kappa$-perfect.
\end{enumerate-(a)} 
The 
\index{perfect set property@$\kappa$-perfect set property for a!set\idf $\PSP_\kappa(X)$}
\emph{$\kappa$-perfect set property} 
$\PSP_\kappa(X)$
holds for a subset $X$ of ${}^\kappa\ddim$ if either $|X|\leq\kappa$ or $X$ contains a $\kappa$-perfect subset. 
\index{perfect set property@$\kappa$-perfect set property for a!class\idf $\PSP_\kappa(\mathcal C)$}
$\PSP_\kappa(\mathcal C)$ states that 
$\PSP_\kappa(X)$
\todog{We need to keep ${}^\kappa\kappa$ instead of ${}^\kappa\ddim$ in the definition of $\PSP_\kappa(\mathcal C)$} 
holds for all subsets $X\in \mathcal C$ of~${}^\kappa\kappa$.
\end{definition}

Suppose that $\PP$ and $\QQ$ are posets. 
A partial function $\iota: \PP\partialto \QQ$ 
is called \emph{strict order preserving} 
\index{function!strict order preserving\idf}
if  $p < q$ implies $\iota(p)< \iota(q)$ for all 
$p,q\in\dom(\iota)$, 
and 
$\iota$ is called \emph{strict order reversing} 
\index{function!strict order reversing\idf}
if  $p < q$ implies $\iota(p) > \iota(q)$ for all 
$p,q\in\dom(\iota)$. 

\begin{definition}
\label{def: order preserving}
Let $\iota$ be a partial function from ${}^{<\kappa}\ddim$ to ${}^{<\kappa}D$.
\begin{enumerate-(a)}
\item \label{pp def}
$\iota$ is \emph{$\perp$-preserving} 
\index{function!$\perp$-preserving\idf}
if $s\perp t$ implies $\iota(s)\perp\iota(t)$ for all $s,t\in\dom(\iota)$.
\item \label{cont def}
Suppose that $\iota$ is strict order preserving. 
It is called \emph{continuous} 
\index{function!continuous strict order preserving\idf}
if $\iota(t)=\bigcup_{s\subsetneq t}\iota(s)$ 
for all $t\in\dom(\iota)$ with  $\lh(t)\in\Lim$ 
and $\kkppred t\subseteq\dom(\iota)$.
\end{enumerate-(a)}
\end{definition}

\begin{definition}
\label{[e] def}
\label{T(e) def}
Let $\iota$ be a strict order preserving partial function from ${}^{<\kappa}\ddim$ to~${}^{<\kappa}D$.
\begin{enumerate-(a)}
\item\label{item: [e] def}
\index{function!induced by a strict order preserving function \idf$[\iota]$} 
$[\iota]$ denotes the partial function from ${}^\kappa\ddim$ to ${}^\kappa D$ 
such that 
$\dom([\iota])$ consists of those
\todog{This more complicated def of $\dom([\iota])$ is needed since we won't always assume that $\dom(\iota)$ is a subtree}%
$x\in{}^{\kappa}\ddim$ with $x\restr\alpha\in\dom(\iota)$  for unboundedly many  $\alpha<\kappa$ and \text{for all $x\in\dom([\iota])$} 
\[
[\iota](x):= 
\bigcup_{\substack{\alpha<\kappa,\\ x\restr\alpha\in\dom(\iota)}} \iota(x\restr\alpha).
\]
\item\label{item: T(e) def}
\index{tree!induced by a strict order!preserving function \idf$T(\iota)$}
$T(\iota):=T(\ran(\iota))$ 
denotes the tree of initial segments of elements of $\ran(\iota)$. 
\end{enumerate-(a)}
\end{definition}

We will usually assume that $\dom(\iota)$ is a pruned subtree of ${}^{<\kappa}\ddim$, 
\todog{For $\dom([\iota])=[\dom(\iota)]$, it suffices that $\dom(\iota)$ is a subtree. For $T(\iota)=T(\ran([\iota]))$,
it suffices that every $s\in\dom(\iota)$ extends to some $x\in\dom([\iota])$ (but $\dom(\iota)$ need not be a subtree).}
in which case $\dom([\iota])=[\dom(\iota)]$ and
$\iota(x)=\bigcup_{\alpha<\kappa} \iota(x\restr\alpha)$ for all $x\in\dom([\iota])$.
Moreover,
$T(\iota)=T(\ran([\iota]))$ in this case,
and therefore
$[T(\iota)]$ is the closure of $\ran([\iota])$ in ${}^\kappa D$.

\begin{lemma}
\label{[e] continuous}
Let $\iota$ be a strict order preserving partial function from ${}^{<\kappa}\ddim$ to ${}^{<\kappa} D$.
\begin{enumerate-(1)}
\item\label{[e] cont}
$[\iota]$ is a continuous function from 
$\dom([\iota])$
to ${}^\kappa D$.
\item\label{[e] homeomorphism}
If $\iota$ is $\perp$-preserving, 
then
$[\iota]$ is a homeomorphism from $\dom([\iota])$ to $\ran([\iota])$.
\end{enumerate-(1)}
\end{lemma}
\begin{proof}
Since $[\iota]$ is strict order preserving, we have
$[\iota]\big(N_t\cap \dom([\iota])\big)\subseteq N_{\iota(t)}$
for all $t\in\dom(\iota)$.
\ref{[e] cont} follows easily from this observation. 
To show \ref{[e] homeomorphism},
suppose that $\iota$ is $\perp$-preserving.
Then $N_{\iota(s)}\cap N_{\iota(t)}=\emptyset$
for all $s,t\in\dom(\iota)$ such that $s\perp t$.
Therefore $[\iota]$ is injective and
$[\iota]\big(N_t\cap \dom([\iota])\big)= N_{\iota(t)}\cap\ran([\iota])$
for all $t\in\dom(\iota)$.
This implies that $[\iota]$ is a homeomorphism from $\dom([\iota])$ to $\ran([\iota])$.
\end{proof}

Note that the converse of~\ref{[e] homeomorphism} fails for the map
$\iota:{}^{<\kappa}\kappa\to{}^{<\kappa}\kappa$ defined by fixing a nonempty element $s$ of ${}^{<\kappa}\ddim$ and 
defining
$\iota(\emptyset):=\emptyset$,
$\iota(t):= s$ for all $t\in{}^{<\kappa}\ddim$ with $\lh(t)=1$ 
and
$\iota(t)=s\conc t$ for all $t\in{}^{<\kappa}\ddim$ with $\lh(t)>1$. 
The converse of~\ref{[e] homeomorphism} does hold for maps $\iota$ such that $\ran([\iota])\cap N_{\iota(t)}=[\iota](N_t)$ for all $t\in\dom(\iota)$. 
In fact, if $\iota$ has the latter property and 
$[\iota]$ is injective, then $\iota$ is $\perp$-preserving.

A \emph{$\wedge$-semilattice} is a 
\index{semilattice@$\wedge$-semilattice\idf}
poset $\PP$ such that for any two $p,q\in \PP$, there exists a greatest lower bound $p \wedge q$. 
We regard
${}^{<\kappa}\ddim$ and ${}^{<\kappa} D$ as a $\wedge$-semilattices with respect to $\subseteq$.%
\footnote{The greatest lower bound $s\wedge t$ of $s,\,t\in{}^{<\kappa}\ddim$ is defined at the beginning of this subsection.}

\begin{definition}
\label{def: wedge-homomorphism} 
Suppose $\PP$ and $\QQ$ are $\wedge$-semilattices and $\iota:\PP\to\QQ$.
\begin{enumerate-(i)}
\item
\index{homomorphism!meet homomorphism@$\wedge$-homomorphism\idf} 
$\iota$ is a \emph{$\wedge$-homomorphism} if $\iota(p\wedge q)=\iota(p)\wedge \iota(q)$ for all $p,q\in \PP$.
\item
$\iota$ is a \emph{strict $\wedge$-homomorphism} 
\index{homomorphism!strict meet homomorphism@strict $\wedge$-homomorphism\idf} 
if it is a strict order preserving $\wedge$-homomorphism.
\end{enumerate-(i)}
\end{definition}
Note that any $\wedge$-homomorphism is order preserving, 
but not necessarily strict order preserving.
Moreover, a strict order preserving map
\todog{``order preserving'' does not suffice for the right-to-left-direction. To see this, let $\iota(\emptyset)=\emptyset$,
and for all $s\in{}^{<\kappa}\kappa$, let
$\iota(\langle 0\rangle\conc s)=s$ and $\iota(\langle\alpha\rangle\conc s)=\langle\alpha\rangle\conc s$ if $\alpha\geq 1$.}
$\iota: {}^{<\kappa}\ddim\to {}^{<\kappa} D$ is a $\wedge$-homomorphism if and only if 
$\iota(t\conc\langle i\rangle) \wedge
\iota(t\conc\langle j\rangle)
= \iota(t)$
for all $t\in{}^{<\kappa}\ddim$ and all $i,j\in\ddim$ with $i\neq j$.\footnote{Note that if $\iota: {}^{<\kappa}\ddim\to {}^{<\kappa}D$ is strict order preserving, then $\iota$ is a strict $\wedge$-homomorphism with $\lh(\iota(t))=\lh(t)$ for all $t\in{}^\kappa\kappa$ if and only if $[\iota]$ is an isometry in the standard $\kappa$-ultrametric.}
This is a strong form of $\perp$-preservation.\footnote{$\perp$ refers to reverse inclusion on ${}^{\leq\kappa}\ddim$ throughout the paper.}%

\subsection{Forcing}
\label{subsection: forcing preliminaries}
A forcing $\PP=(P,\leq_\PP,\one_\PP)$ is a triple such that $\leq_\PP$ is a pre-order (i.e., a reflexive transitive relation) on $P$ and $\one_\PP$ is a largest element of $\PP$ with respect to $\leq_\PP$. We 
\todoo{changed} 
conflate $\PP$ with its domain $P$, and we usually write $\leq$, $\perp$, $\one$ 
and $\forces$ instead of $\leq_{\PP}$, $\perp_\PP$, $\one_\PP$ and $\forces_\PP$, respectively.
For all $p,q\in \PP$, we let $p\compat q$ (or sometimes $p\compat_\PP q$) denote the statement that $p$ and $q$ are compatible, i.e., that $p\not\perp q$.
\index{Boolean!completion\idf$\BB(\PP)$}  
Let $\BB(\PP)$ denote the Boolean completion of $\PP$.\footnote{Note that $\BB(\PP)$ is unique up to isomorphism.} 
If $\PP$ is separative, we assume that $\PP$ is a dense subset of~$\BB(\PP)$.

\begin{definition} \ 
\begin{enumerate-(a)} 
\item 
\index{forcing!non-atomic\idf}
An \emph{atom} in a forcing $\PP$ is a condition $p\in\PP$ with no incompatible extensions. A forcing $\PP$ is \emph{non-atomic} if it has no atoms. 
\item 
\index{forcing!homogeneous\idf}
A forcing $\PP$ is \emph{homogeneous} if for all $p,q\in \PP$, there is an automorphism $\pi: \PP\to \PP$ such that $\pi(p)$ and $q$ are compatible. 
\end{enumerate-(a)} 
\end{definition}

\begin{definition} 
\todog{For separative posets, $i:\PP\to\QQ$ is a dense embedding iff it's a \emph{sub-isomorphism} (as defined in Philipp's paper).}
\label{def: equivalent forcings} 
\label{def: dense embedding} 
\label{def: pulling back names}
Suppose $\PP,\QQ$ are forcings.
\begin{enumerate-(a)} 
\item A \emph{dense embedding} 
\index{embedding!dense\idf}
$\iota:\PP\to\QQ$ is a homomorphism with respect to $\leq,\,\perp$ and $\one$ 
such that $\iota(\PP)$ is a dense subset of $\QQ$. 
\item\label{eqf}
\index{forcing!equivalence\idf}
Two forcings $\PP$ and $\QQ$ are 
\emph{equivalent} 
($\PP\simeq \QQ$) if there exist 
dense embeddings 
$\iota: \PP\to \RR$ and $\nu: \QQ\to\RR$ into some forcing $\RR$. 
\item\label{pbn} 
Let $\iota: \PP\to \QQ$ be a dense embedding. Define a $\PP$-name 
\index{name!induced by a dense embedding \idf$\sigma^\iota$}
$\sigma^\iota$ for each $\QQ$-name $\sigma$ by recursion on the rank as 
\[
\sigma^{\iota}:=\left\{(\tau^{\iota},p) \::\: 
p\in \PP,\ \,
\exists q\in\QQ\ \big(\iota(p)\leq q\wedge (\sigma,q)\in \tau\big)
\right\}.
\]
\end{enumerate-(a)} 
\end{definition}

Recall the following standard facts.\footnote{See e.g. \cite{KunenBook2013}.}
Two forcings 
are equivalent if and only if they have isomorphic Boolean completions.
Suppose $\iota:\PP\to\QQ$ is a dense embedding between forcings $\PP$ and $\QQ$.
Given a $\PP$-generic filter $G$ over $V$,
 the upwards closure $H$ of $\iota(G)$ in $\QQ$ is a $\QQ$-generic filter over $V$, and $G=\iota^{-1}(H)$.
Moreover, 
$(\sigma^\iota)^G=\sigma^H$ for each $\QQ$-name $\sigma$.
Conversely, suppose $H$ is a $\QQ$-generic filter over $V$. Then $G=\iota^{-1}(H)$ is a $\PP$-generic filter over $V$, and $H$ is equal to the upwards closure of $G$ in $\QQ$. In both of the above cases, we have $V[G]=V[H]$.

The following definition of the forcing for adding Cohen subsets is non-standard but essential in several arguments below.\todog{It is essential throughout Section~\ref{section: proof of the main result} that we use this non-standard definition}

\begin{definition} 
\label{definition of Add(kappa,1)}
Suppose that $\kappa$ is a regular uncountable cardinal.  
\begin{enumerate-(a)} 
\item 
\index{forcing!Cohen subset!standard\idf$\Add(\kappa,1)$}
$\Add(\kappa,1)$ is defined as the forcing
$$\Add(\kappa,1):={}^{<\kappa}\kappa
=\{p:\alpha\to\kappa\,\mid\, \alpha<\kappa\},$$ 
ordered by reverse inclusion. 
\item 
\index{forcing!Cohen subset!multiple\idf$\Add(\kappa,\xi)$}
$\Add(\kappa,\xi)$ is defined as the ${<}\kappa$-support product $\prod_{i<\xi}\Add(\kappa,1)$ for any ordinal $\xi$. 
\end{enumerate-(a)} 
\end{definition}

We use the following convention. 
For a given $y\in{}^{\kappa}\kappa$, 
we sometimes conflate $y$ with the set 
$\kkppred y=\{t\in{}^{<\kappa}\kappa: t\subsetneq y\}$.
For example, we say that $y$ is $\Add(\kappa,1)$-generic if and only if $\kkppred y$ is an $\Add(\kappa,1)$-generic filter. 
Moreover, if $y$ is $\Add(\kappa,1)$-generic and $\sigma$ is an $\Add(\kappa,1)$-name, let $\sigma^y:=\sigma^{\kkppred y}$.
Also let $\prod_{i<\xi}y_i$ denote $\prod_{i<\xi}\kkppred{y_i}$ whenever $\langle y_i:i<\xi\rangle$ is a sequence of elements of ${}^{\kappa}\kappa$.
The next two lemmas contain
standard facts about adding Cohen subsets and collapse forcings.

\begin{lemma} 
\label{forcing equivalent to Add(kappa,1)} 
\todog{will only need this for uncountable $\kappa$}
Suppose that $\kappa$ is an uncountable cardinal
such that $\kappa^{<\kappa}=\kappa$. 
If $\PP$ is a non-atomic ${<}\kappa$-closed forcing of size $\kappa$, then $\PP$ has a dense subset which is isomorphic to the dense subforcing 
\index{forcing!Cohen subset!modified\idf$\Addd(\kappa,1)$}
$$\Addd(\kappa,1):=\{p\in \Add(\kappa,1):\: \dom(p)\in \Succ\}$$ 
of $\Add(\kappa,1)$.
In particular, $\PP$ is equivalent to $\Add(\kappa,1)$. 
\end{lemma}
We omit the proof of Lemma~\ref{forcing equivalent to Add(kappa,1)}, since it is straightforward and well known.

\begin{lemma}[\cite{MR2387944}*{Lemma 2.2}] 
\label{forcing equivalent to Col(kappa,mu)} 
Let $\kappa$ be a regular cardinal. 
Suppose that $\nu>\kappa$ is a cardinal with $\nu^{<\kappa}=\nu$. Let $\PP$ be a separative ${<}\kappa$-closed forcing of size $\nu$ which forces that $\nu$ has size $\kappa$. 
Then $\PP$ has a dense subset which is isomorphic to the dense subforcing\index{forcing!collapse!modified\idf$\Col^*(\kappa,\nu)$}
$$\mathrm{Col}^*(\kappa,\nu):=\{p\in \Col(\kappa,\nu):\: \dom(p)\in\Succ\}$$ 
\index{forcing!collapse!standard\idf$\Col(\kappa,\nu)$}
of $\Col(\kappa,\nu)$. 
In particular, $\PP$ is equivalent to $\Col(\kappa,\nu)$. 
\end{lemma}
Lemma~\ref{forcing equivalent to Col(kappa,mu)} can be proven 
using an adaptation of the proof of 
\cite{MR1940513}*{Lemma 26.7} 
which can be found in the proof of \cite{MR2387944}*{Lemma 2.2}. 

\label{notation for the Levy collapse and subforcings} 
We will use the following notation
for subforcings of the 
\index{forcing!L\'evy collapse\idf$\Col(\kappa,\lle\lambda),\,\PP_\lambda$}
L\'evy collapse $\Col(\kappa,\lle\lambda)$. 
Suppose that $\kappa<\lambda$ are cardinals, $\alpha\leq\lambda$ and $I\subseteq\lambda$ is not an ordinal (to avoid a conflict with the notation for the standard collapse). 
\index{forcing!subforcing of the L\'evy!collapse\idf$\PP_\alpha,\,\PP_I,\,\PP^\alpha$}
Then let $\PP_\alpha:=\Col(\kappa,\lle\alpha)$ 
and
\[
\PP_I:=\Col(\kappa,I):=\{p\in \Col(\kappa,\lle\lambda):\: \dom{p}\subseteq I\times \kappa\}.
\] 
The notation $\PP_I=\Col(\kappa,I)$ will be used for intervals $I$, for which we use the standard notation 
\begin{align*}
(\alpha,\gamma)&:=\{\beta\in \Ord: \alpha<\beta<\gamma\},
\\
[\alpha,\gamma)&:=\{\beta\in \Ord: \alpha\leq\beta<\gamma\}. 
\end{align*}
\megj{(}
Let $\PP^\alpha:=\PP_{[\alpha,\lambda)}$
for all $\alpha<\lambda$.
\index{forcing!generic for a subforcing of the L\'evy collapse\idf$G_\alpha,\,G_I,\,G^\alpha$}
If $G$ is a $\PP_\lambda$-generic filter over $V$, write
$G_\alpha:=G\cap\PP_\alpha$, $G_I:=G\cap\PP_I$  
and $G^\alpha:=G\cap \PP^\alpha$.
%


The next corollary follows easily from Lemma~\ref{forcing equivalent to Col(kappa,mu)} 
using a straightforward analogue of the proof of 
\cite{MR2387944}*{Corollary~2.4}.\footnote{See also \cite{MR2387944}*{Corollary~2.3}. 
\cite{MR2387944}*{Corollary~2.4} 
asserts that the conclusion of Corollary~\ref{forcing equivalent to the Levy collapse} holds when $\gamma=0$, under the weaker assumption that  
$\lambda>\kappa$ is a cardinal such that $\mu^{<\kappa}<\lambda$ for all cardinals $\mu<\lambda$.} 

\begin{corollary}
\label{forcing equivalent to the Levy collapse}
Let $\lambda>\kappa$ be an inaccessible cardinal and $\gamma<\lambda$.
If $\PP$ is a $\lle\kappa$-closed forcing of size $\lle\lambda$, 
then the forcings $\PP\times\PP_\lambda$ and 
$\PP^\gamma=\PP_{[\gamma,\lambda)}$ 
are equivalent.
\end{corollary}

\begin{definition} 
\label{def: complete embedding}
\label{def: quotient forcing} 
Suppose that $\PP$, $\QQ$ are forcings. 
\begin{enumerate-(a)} 
\item\label{qf1} 
\index{embedding!complete\idf}
A \emph{complete embedding} $i: \PP\to \QQ$ is a homomorphism with respect to $\leq,\,\perp$ and $\one$ with the property that for every $q\in\QQ$, there is a condition $p\in \PP$ 
\index{reduction of!a forcing condition\idf} 
(called a \emph{reduction} of $q$ to $\PP$) such that for every $r\leq p$ in $\PP$, $i(r)$ is compatible with $q$. 
\item \label{qf2}
\index{forcing!complete subforcing\idf}
$\PP$ is a \emph{complete subforcing} of $\QQ$  
if $\PP$ is a subforcing of $\QQ$ and the inclusion map $\id_{\PP}:\PP\to\QQ$ is a complete embedding.
\item \label{qf3}
Suppose that $i: \PP\to \QQ$ is a complete embedding and $G$ is $\PP$-generic over $V$.
\index{forcing!quotient\idf$\QQ/G$} 
The \emph{quotient forcing $\QQ/G$ for $G$ in $\QQ$} is defined as the subforcing 
\[\QQ/G:=\{q\in \QQ:\: \forall p\in G\ \,i(p)\compat q\}\] 
of $\QQ$. 
Moreover, we fix a $\PP$-name $\QQ/\PP$ for  for the quotient forcing for $\dot G$ in $\PP$, where $\dot{G}$ is the canonical $\PP$-name for the $\PP$-generic filter. We also refer to $\QQ/\PP$ as (a name for) the quotient forcing for $\PP$ in $\QQ$. 
\end{enumerate-(a)} 
\end{definition} 

We recall the following standard facts.\footnote{See e.g. \cite{KunenBook2013} and  
Exercises (C7), (C8), (D4) and (D5) in 
\cite{KunenBook1980}*{Chapter~VII}.}

\begin{fact} 
\label{fact}
Suppose $\PP$ and $\QQ$ are forcings and $i:\PP\to\QQ$.
\begin{enumerate-(1)}
\item\label{fact: complete embeddings equiv}
If $i$ is a homomorphism with respect to $\leq,\,\perp$ and $\one$, then 
$i$ is a complete embedding if and only if
for all maximal antichains $A$ of $\PP$,
$i(A)$ is a maximal antichain of $\QQ$. 
If $\PP$ and $\QQ$ are complete Boolean algebras, then $i$ is a complete embedding (in the sense of Definition~\ref{def: complete embedding}) if and only if $i$ is 
an injective complete homomorphism of Boolean algebras%
\footnote{A homomorphism between Boolean algebras is called \emph{complete} if it preserves arbitrary infima and suprema.}
\item\label{fact: complete embeddings and Boolean completions}
Any complete embedding $i:\PP\to\QQ$ defines a complete embedding 
$j:\BB(\PP)\to\BB(\QQ)$.
If $\PP$ and $\QQ$ are separative, we may assume that $i\subseteq j$.
In particular, if $\PP$ is a complete subforcing of $\QQ$ and $\QQ$ is separative,%
\footnote{Note that complete subforcings of separative forcings are also separative.} 
we may assume that $\BB(\PP)\subseteq\BB(\QQ)$.
Then $\BB(\PP)$ is a complete Boolean subalgebra of $\BB(\QQ)$ by~\ref{fact: complete embeddings equiv}.
\todog{If $\QQ$ is separative, then any complete subforcing $\PP$ of $\QQ$ is also separative. Proof: Let $p,q\in \PP$ with $p\not\leq q$. Since $\QQ$ is separative, take $s\in \QQ$ with $s\leq p$ and $s\perp q$. Let $r$ be a reduction of $s$ to $\PP$. Then $r\compat_{\PP} p$ since $\id$ is $\perp$-preserving and $r\perp q$ since $r$ is a reduction of $s$ (and $q\in\PP$).}
\item\label{fact: generic extensions wrt quotient forcings}
Suppose $i$ is a complete embedding. 
Let $G$ be a $\PP$-generic filter over $V$.
If $D$ is a dense subset of $\QQ$, then $D\cap\QQ/G$ is a dense subset of $\QQ/G$.
Therefore
all $\QQ/G$-generic filters $H$ over $V[G]$ are also $\QQ$-generic filters over $V$, and $G=i^{-1}(H)$ for all such $H$. 
Conversely, 
if $H$ is $\QQ$-generic over $V$ and $G=i^{-1}(H)$, then $G$ is a $\PP$-generic filter over $V$ and $H$ is $\QQ/G$-generic over $V[G]$.
Furthermore, under either of the above assumptions, we have  
$V[H]_{\QQ}=V[G][H]_{\QQ/G}$.
\end{enumerate-(1)}
\end{fact} 

\begin{definition}
\label{def: retraction}
\todog{(Philipp:) I think both `natural projection' and 'retraction' are nice. } 
Given complete Boolean algebras $\PP$ and $\QQ$ and a complete embedding $i:\PP\to\QQ$, 
the 
\index{retraction\idf$\pi_i$} 
\emph{retraction} associated to $i$ is the map
$\pi_i:\QQ\to\PP$ 
defined by letting
\[
\pi_i(q):=\bigwedgepo\PP\{p\in\PP: i(p)\geq q\}
\]
{for all } $q\in \QQ$.
\end{definition}

The following lemma is standard, but we include a proof for the reader.

\begin{lemma}
\label{retraction facts}
Suppose $i:\PP\to\QQ$ is a complete embedding between the complete Boolean algebras $\PP$ and $\QQ$.
\begin{enumerate-(1)}
\item\label{rf1} For all $q\in\QQ$, $\pi_i(q)$ is the largest reduction of $q$ to $\PP$.
\item\label{rf2} 
If $G$ is a $\PP$-generic filter over $V$, then $q\in\QQ/G$ if and only if $\pi_i(q)\in G$.
\end{enumerate-(1)}
\end{lemma}
\begin{proof}
Let $q\in\QQ$. 
First, observe that $i(p)\perp q$ if and only if $p\perp \pi_i(q)$ holds for all $p\in\PP$.
Indeed, given $p\in\PP$, we have
$i(p)\perp q$
$\Longleftrightarrow$ $i(p)\wedge q=\zero$
$\Longleftrightarrow$ $\neg i(p)\geq q$. 
By the definition of $\pi_i$ and since $i$ is a Boolean homomorphism, the last statement is equivalent to
$\neg p\geq \pi_i(q)$,
and is therefore also equivalent to
$p\wedge \pi_i(q)=\zero$ and to
$p\perp\pi_i(q)$.

The previous observation easily implies that $\pi_i(q)$ is a reduction of $q$ to $\PP$.
To see that $\pi_i(q)$ is the largest reduction of $q$ to $\PP$, suppose that
$p\in\PP$ and $p\not\leq\pi_i(q)$.
Let $r:=p\wedge\neg\pi_i(q)$.
Then $p\geq r\geq\zero_\PP$ and $r\perp \pi_i(q)$, and therefore $i(r)\perp q$.
Thus, $r$ witnesses that 
$p$ is not a reduction of $q$ to~$\PP$.
Moreover, if $G$ is a $\PP$-generic filter over $V$ and $q\in\QQ$, 
then
$q\in\QQ/G$ 
$\Longleftrightarrow$
$p\compat \pi_i(q)$  for all $p\in\PP$
$\Longleftrightarrow$
$\pi_i(q)\in G$.
\end{proof}

As usual, by a name $\sigma$ for an element of ${}^\kappa\kappa$, we mean a name $\sigma$ such that $\one\forces\sigma\in {}^\kappa\kappa$.
\todoo{added footnote with def. of $\check{V}$} 
By a name $\sigma$ for a \emph{new} object, 
we mean a name $\sigma$ such that $\one\forces\sigma\notin \check V$.\footnote{$\check{V}:=\{(\check{x},1) : x\in V \}$ is the canonical name for $V$.} 
In general, if $\varphi(v)$ is any formula with one variable $v$, then by a name $\sigma$ for an object with property $\varphi$, we mean a name $\sigma$ such that $\one\forces \varphi(\sigma)$. 

We will use the following notation when working with quotient forcings induced by names. 
\begin{definition}
\label{def: generated Boolean subalgebra} 
Suppose $\QQ$ is a complete Boolean algebra and $\sigma$ is a $\QQ$-name for a subset of a set $x$.
\index{Boolean!subalgebra for a name\idf$\BB(\sigma)$} 
Let $\BB(\sigma):=\BB^{\QQ}(\sigma)$ denote the complete Boolean subalgebra of $\QQ$ that is completely generated by 
the set of Boolean values $\{\boolvalpo{y\in \sigma}{\QQ}: y\in x\}$. 
\end{definition}

The following lemma is standard, but we include a proof for the reader. 

\begin{lemma}
\label{fact: generated Boolean subalgebra} 
Suppose that $\QQ$ is a complete Boolean algebra and $\sigma$ is a $\QQ$-name for a subset of a set $x$. 
If $H$ is a $\QQ$-generic filter over $V$, then
$V\big[\sigma^H\big]=V\big[\BB^{\QQ}(\sigma)\cap H\big]$. 
\end{lemma}
\begin{proof}
Let $A:=\{\boolval{y\in\sigma}_{\QQ}:y\in x\}$. 
We then have $V\big[\BB^{\QQ}(\sigma)\cap H\big]=V\big[A\cap H\big]$ 
by \cite{MR1940513}*{Lemma~15.40}.
Moreover, 
$V\big[A\cap H\big]=V\big[\sigma^H\big]$ 
since
$y\in\sigma^H$ 
$\Longleftrightarrow$ $\boolvalpo{y\in\sigma}\QQ\in H$
$\Longleftrightarrow$ $\boolvalpo{y\in\sigma}\QQ\in A\cap H$
for all $y\in x$.
\end{proof}

In the next definition and lemmas, suppose that $\QQ\in V$ is a forcing, $q\in\QQ$ and $\sigma\in V$ is a $\QQ$-name for an element of~${}^\kappa\kappa$.

\begin{definition}\ 
\label{T^sigma def}
\label{sigma_q def}
\label{sigma[q] def}
\begin{enumerate-(a)}
\item
\index{name!decided part of $\sigma$!initial segment\idf$\sigma_{[q]}$} 
$\sigma_{[q]}:=\bigcup\{t\in{}^{<\kappa}\kappa: q\forces t\subseteq\sigma\}$. 
\item
\index{name!decided part of $\sigma$!full\idf$\sigma_{(q)}$} 
$\sigma_{(q)}:=\{(\alpha,\beta):q\forces(\alpha,\beta)\in\sigma\}$.
\item
\index{tree!of possible values for $\sigma$\idf$T^{\sigma, q}$}%
The \emph{tree $T^{\sigma, q}$ of possible values for $\sigma$ below $q$} is defined as follows:
\todog{Equivalently, $t\in T^{\sigma, q}$ $\Leftrightarrow$ $r\forces t\subseteq \sigma$ for some $r\leq q$}
\[
T^{\sigma, q}:=
\{t\in{}^{<\kappa}\kappa\::\: \exists r\leq q\ \ t\subseteq \sigma_{[r]}\}.\] 
\end{enumerate-(a)}
For clarity, we at times write $\sigma_{[q,\QQ]},\ \sigma_{(q,\QQ)}$ and $T^{\sigma,q}_{\QQ}$ instead of $\sigma_{[q]},\ \sigma_{(q)}$ and $T^{\sigma,q}$, respectively.
\end{definition}

The next two lemmas list basic properties of $\sigma_{[q]}$ and $T^{\sigma,q}$. 
The first one follows immediately from the definitions.

\begin{lemma}\ 
\label{lemma: sigma_q and sigma[q]}
\begin{enumerate-(1)}
\item $\sigma_{[q]}\subseteq \sigma_{(q)}$, and
$\sigma_{[q]}=\sigma_{(q)}$ if and only if $\dom(\sigma_{(q)})$ is an ordinal. 
\item $\sigma_{[q]}\restr\alpha=\sigma_{(q)}\restr\alpha$ for all ordinals $\alpha$ with $\alpha\subseteq\dom(\sigma_{(q)})$. 
\end{enumerate-(1)}
\end{lemma}

\begin{lemma}\ 
\label{T^sigma facts}
\begin{enumerate-(1)}
\item \label{T^sigma nodes comp}
If $q\not\forces_{\QQ} \sigma \in\check{V}$, 
then $\sigma_{[q]}\in{}^{<\kappa}\kappa$ equals the stem of  $T^{\sigma,q}$.  
\item
\label{sigma is a branch of T^sigma}
$q\forces^V_{\QQ} \sigma\in[T^{\sigma,q}].$
\item 
\label{T^sigma subtree of S}
If $S$ is a subtree of ${}^{<\kappa}\kappa$ with $q\forces^V_{\QQ} \sigma\in[S]$, then $T^{\sigma,q}\subseteq S$.
\item
\label{T^sigma abs}
If $M\supseteq V$ is a transitive model of $\ZFC$ with 
$({}^{<\kappa}\kappa)^V=({}^{<\kappa}\kappa)^M$, then 
${(\sigma_{[q]})}^V={(\sigma_{[q]})}^M$
and
${(T^{\sigma, q})}^V={(T^{\sigma, q})}^M.$
\todog{If $\QQ$ is a $\lle\kappa$-closed forcing 
and $q\forces\sigma\notin\check{V}$,
then $T^{\sigma,q}$ is a $\kappa$-perfect tree.
\\
PROOF: $T:=T^{\sigma,q}$ is cofinally splitting because $\sigma$ is a name for a new element of ${}^\kappa\kappa$.
To obtain a $\kappa$-perfect (in the stronger sense) subtree containing a given node $s\in T$: use that $T$ is cofinally splitting at successor stages of defining the tree, and use the fact that 
 $\QQ$ is $\lle\kappa$-closed at limit stages.
It's probably enough to assume that $\QQ$ is $\lle\kappa$-strategically closed.}
\end{enumerate-(1)}
\end{lemma}
\begin{proof} 
\ref{T^sigma nodes comp}--\ref{T^sigma subtree of S} are immediate. 
\ref{T^sigma abs} holds since the formula ``$q\forces t\subseteq \sigma$'' is absolute between $M$ and $V$, as it  can be defined by a recursion which uses only absolute concepts \cite{KunenBook2013}*{Theorem II.4.15}. 
\end{proof}

\newpage 

\section{Dihypergraphs and homomorphisms}
\label{section: ODD observations}

We discuss preliminary observations about colorings, continuous homomorphisms and 
\todoo{modified intro: added first sentence. Moved the second sentence here (from later in the paragraph)}
the open dihypergraph dichotomy in this section. 
These are necessary for the rest of the paper. 
Subsection~\ref{subsection: basic observations} discusses some basic facts. 
We characterize $H$-independence for box-open dihypergraphs $H$ at the level of subtrees of ${}^{<\kappa}\kappa$ in Section~\ref{subsection: independent trees}. 
We provide a characterization of
the existence of continuous homomorphisms from $\dhH\ddim$ to $H$
via certain strict order preserving maps from ${}^{<\kappa}\ddim$ to ${}^{<\kappa}\kappa$ in Subsection~\ref{subsection: order homomorphisms}. 
The relation 
of \emph{fullness} between dihypergraphs that is introduced in Subsection~\ref{subsection: full dihypergraphs} 
compares the induced
$\kappa$-ideals of $\kappa$-colorable sets and, as we show, allows us 
to compare the existence of continuous homomorphisms
for relatively box-open dihypergraphs as well.%
\footnote{See Lemma~\ref{lemma: ODD subsequences}.}
In particular, 
equivalent relatively box-open dihypergraphs give rise to 
the same instances of the open dihypergraph dichotomy.


\subsection{Colorings and homomorphisms}
\label{subsection: basic observations}

Throughout this subsection, 
we assume that $\ddim\geq 2$ is an ordinal. 
We first observe 
that the two options in the definition of $\ODD\kappa H$ are mutually exclusive. 

\begin{lemma}
\label{two options in ODD are mutually exclusive}
Let $H$ be a $\ddim$-dihypergraph on a topological space $X$.
Suppose that there is a homomorphism $f:{}^\kappa\ddim\to X$ from $\dhH\ddim$ to $H$.
\begin{enumerate-(1)}
\item\label{me 1} $H$ does not have a $\kappa$-coloring.
\item\label{me 2} 
$H\restr f(N_t)$ does not have a $\kappa$-coloring for any $t\in{}^{<\kappa}\ddim$.
\todog{Stating~\ref{me 2} explicitly will be useful in some of the arguments below.}
\end{enumerate-(1)}
\end{lemma}
\begin{proof}
The proof of~\ref{me 1} is a straightforward analogue of the proof for the $\kappa=\omega$ case 
in \cite{CarroyMillerSoukup}*{Theorem~1}.
We give the details for completeness. 
Suppose $c':X\to\kappa$ is a $\kappa$-coloring of $H$, and let $c:=c'\comp f$. Since $f$ is a homomorphism from $\dhH\ddim$ to $H$, $c$ is a $\kappa$-coloring of $\dhH\ddim$. 
Construct by recursion a continuous increasing sequence
$\langle t_\alpha:\alpha<\kappa\rangle$ such that
$t_\alpha\in{}^{\alpha}\ddim$ and
$c^{-1}(\{\alpha\})\cap N_{t_{\alpha+1}}=\emptyset$
for each $\alpha<\kappa$.
Use the
$\dhH\ddim$-independence of $c^{-1}(\{\alpha\})$ at
stage $\alpha+1$ of the construction.
Then
$\bigcup_{\alpha<\kappa} t_\alpha$ is an element of ${}^\kappa\ddim$ 
which is not in $c^{-1}(\{\alpha\})$ for any $\alpha<\ddim$,
a contradiction.

\ref{me 2} follows from \ref{me 1}, since
the map
$f_t:{}^\kappa\ddim\to f(N_t);\; x\mapsto f(t\conc x)$ 
is a homomorphism from~$\dhH\ddim$ to~$H\restr f(N_t)$
for any given $t\in{}^\kappa\ddim$. 
\end{proof}

Note that
\ref{me 1} implies $|X|>\kappa$ and~\ref{me 2} implies 
$\left|f (N_t)\right|>\kappa$ 
for all $t\in{}^{<\kappa}\ddim$.
We did not need to assume that $f$ is continuous or that $H$ is box-open. 
Thus, if $\ODD\kappa H$ holds and there exists an arbitrary homomorphism from $\dhH \ddim$ to $H$, then there also exists a continuous one. 
This is analogous to the fact that for any subset $X$ of ${}^\kappa\kappa$ that satisfies the $\kappa$-perfect set property, the existence of an injective function from ${}^\kappa 2$ to $X$ already implies the existence of a continuous injective function.

\begin{lemma}
\label{<kappa dim hypergraphs are definable}
Suppose $\ddim<\kappa$ and $X\subseteq{}^\kappa\kappa$.
\begin{enumerate-(1)}
\item\label{dhad 1}
If $U$ is a box-open subset of ${}^\ddim ({}^\kappa\kappa)$
then $U\in\defsets\kappa$.
\item\label{dhad 2}
$\ODD\kappa\ddim(X)$ is equivalent to $\ODD\kappa\ddim(X,\defsets\kappa)$.
\end{enumerate-(1)}
\end{lemma}
\begin{proof}
For \ref{dhad 1}, note that $U$ is given by a sequence $x: \kappa\to {}^\ddim({}^{<\kappa}\kappa)$ since the base of the topology has size $|{}^\ddim({}^{<\kappa}\kappa)|=\kappa$. Since $x$ can be coded by an element of ${}^\kappa\kappa$, we have $U\in \defsets\kappa$. 
\ref{dhad 2} follows from \ref{dhad 1}.
\end{proof}

The next two lemmas give reformulations of $\ODD\kappa\ddim(\defsetsk)$ and $\ODD\kappa\ddim(\defsetsk,\defsetsk)$. 
Recall that $H$ is a \emph{relatively box-open dihypergraph} if it is box-open on its domain $\domh H$.

\begin{lemma}
\label{lemma: two versions of ODD for definable sets}
$\ODD\kappa\ddim(\defsetsk)$ is equivalent to the statement that $\ODD\kappa H$ holds for all
relatively box-open $\ddim$-dihypergraphs 
$H$  on ${}^\kappa\kappa$
with $\domh H\in\defsetsk$.
\end{lemma}
\begin{proof}
This follows from the observation that a dihypergraph
$H$ is relatively box-open with $\domh H\in\defsetsk$ 
if and only if
$H=H'\restr X$ for some 
subset $X\in\defsetsk$ of ${}^\kappa\kappa$ and some
box-open $\ddim$-dihypergraph $H'$ on ${}^\kappa\kappa$.
To see the direction from right to left, note that $\domh{H'\restr X}$ is a relatively open subset of $X$,
and hence it is in $\defsetsk$ if $X$ is in $\defsetsk$.
\end{proof}

\begin{lemma}\ 
\label{lemma: two versions of definable ODD}
Suppose that $d\leq\kappa$. Then
$\ODD\kappa\ddim(\defsetsk,\defsetsk)$ is equivalent to the statement that $\ODD\kappa H$ holds for all
relatively box-open $\ddim$-dihypergraphs 
$H\in \defsetsk$ on ${}^\kappa\kappa$.
\end{lemma}
\begin{proof}
This follows from the equivalence of the following statements for any $\ddim\leq\kappa$ and any $\ddim$-dihypergraph $H$ on ${}^\kappa\kappa$: 
\begin{enumerate-(a)} 
\item\label{def ODD 1a} 
$H\in \defsetsk$ is relatively box-open.
\item\label{def ODD 1b} 
$H=H'\restr X$ for some 
subset $X\in\defsetsk$ of ${}^\kappa\kappa$ and some
box-open $\ddim$-dihypergraph $H'\in\defsetsk$ on ${}^\kappa\kappa$.
\end{enumerate-(a)} 
For \ref{def ODD 1a} $\Rightarrow$ \ref{def ODD 1b}, 
note that 
$X:=\domh H\in \defsetsk$ and
$$H':=\bigcup\{\prod_{\alpha<\ddim}N_{t_\alpha}:\, \prod_{\alpha<\ddim}(N_{t_\alpha}\cap X)\subseteq H\}\cap\dhK\ddim{{}^\kappa\kappa} \in\defsetsk$$ 
is a box-open $\ddim$-dihypergraph on ${}^\kappa\kappa$ 
with $H=H'\restr X$. 
The converse is clear. 
\end{proof}

\begin{lemma}
\label{ODD for continuous images}
Let $f:X\to Y$ be a continuous surjection between subsets $X$ and $Y$ of~${}^\kappa\kappa$. 
\begin{enumerate-(1)}
\item \label{ci 1}
$\ODD\kappa\ddim(X)$ implies $\ODD\kappa\ddim(Y)$.%
\item \label{ci 2} 
If $f$ can be extended to a continuous function \todog{``continuous'' is needed  in the proof, and it implies $g\in\defsetsk$}
$g: {}^\kappa\kappa\to {}^\kappa\kappa$, 
then $\ODD\kappa\ddim(X,\defsets\kappa)$ implies $\ODD\kappa\ddim(Y,\defsets\kappa)$.
\end{enumerate-(1)}
\end{lemma}
\begin{proof}
$\ODD\kappa\ddim(X)$ is equivalent to the statement that $\ODD\kappa I$ holds for all box-open $\ddim$-dihypergraphs $I$ on $X$. 
For a given box-open $\ddim$-dihypergraph $H$ on $Y$, let
\[H_f={{(f^\ddim)}^{-1}(H)}
=\{\langle x_\alpha:\alpha<\ddim\rangle\in{}^\ddim X: 
\langle f(x_\alpha):\alpha<\ddim\rangle\in H\}.
\]
Since ${f^\ddim}$ is a continuous map between the spaces ${}^\ddim X$ and ${}^\ddim Y$,
where both spaces are 
equipped with the box-topology, 
$H_f$ is a box-open $\ddim$-dihypergraph on $X$. 

%
\begin{claim*}
$\ODD\kappa {H_f}$ implies $\ODD\kappa H$.
\end{claim*}
\begin{proof}
Note that $f$ is a continuous homomorphism from $H_f$ to $H$.
Thus, if $h:{}^\kappa\ddim\to X$ is a continuous homomorphism from $\dhH\ddim$ to $H_f$, then $h\comp f$ is a continuous homomorphism from $\dhH\ddim$ to $H$.
On the other hand, if $c:X\to\kappa$ is a $\kappa$-coloring of $H_f$, then the map 
$c':Y\to\kappa$ defined by letting
$c'(y):=
\min\big\{\alpha:\:f^{-1}(\{y\})\cap {(c)}^{-1}(\{\alpha\})\neq\emptyset\big\}$
for each $y\in Y$
is a $\kappa$-coloring of $H$.
\end{proof}

\ref{ci 1} follows immediately from the previous claim.
For \ref{ci 2}, suppose that $H'\in \defsets\kappa$ is a box-open $\ddim$-dihypergraph on ${}^\kappa\kappa$. 
Let $H:=H'{\restr}Y$. 
It suffices to show 
$\ODD\kappa H$.
Let $H_f:={(f^\ddim)}^{-1}(H)$ and $H'_g={(g^\ddim)}^{-1}(H')$. 
Since $H_f=H'_g\restr X$ and $g$ is continuous, it follows that 
$H'_g\in\defsetsk$ 
is a box-open \todog{needs $g$ to be continuous}
$\ddim$-dihypergraph on ${}^\kappa\kappa$.
Since we assume 
$\ODD\kappa\ddim(X,\defsetsk)$, we have $\ODD\kappa\ddim(X,H'_g)\Longleftrightarrow\ODD\kappa{H'_g{\restr}X}\Longleftrightarrow\ODD\kappa{H_f}$. 
Thus $\ODD\kappa H$ by the previous claim.
\todog{Short explicit proof of $H'_g\restr X=H_f$ (I think it can be omitted though): $\bar x\in H'_g\restr X$ if and only if ($\bar x\in{}^\ddim X$ and $H'\ni g(\bar x)=f(\bar x)$). In the last equality, we use that $f=g\restr X$. Since $H=H'\restr Y$ and $\ran(f)\subseteq H$, the latter condition holds if and only if ($\bar x\in{}^\ddim X$ and $f(\bar x)\in H$), i.e., if and only if $\bar x\in H_f$}
\end{proof}

We next show that $\ODD\kappa\ddim(X)$ implies $\ODD\kappa c(X)$ 
for all $c\subseteq d$ 
with $ |c|\geq 2$.
We will need the following lemma.
Let $\proj_{\ddim,c}$ denote 
\index{projection of $\ddim$-sequences onto!coordinates in $c$ \idf$\proj_{\ddim,c}$}%
the projection of $\ddim$-sequences onto 
coordinates in~$c$,
i.e.,
for any sequence 
$\bar x=\langle x_i:i\in\ddim\rangle$,
\[\proj_{\ddim,c}(\bar x)=\bar x\restr c=\langle x_i: i\in c\rangle.\]
Recall that $\dhK\ddim X$ is the complete $\ddim$-hypergraph on $X$ and
$\dhII\ddim X$ consists of 
the injective elements of $\dhK\ddim X$.

\begin{lemma}
\label{sublemma: ODD for different dimensions}
Suppose that $c\subseteq d$ and $|c|\geq 2$.
Let $H$ be a
$c$-dihypergraph 
on a topological space $X$,
and let 
$I:=\proj_{\ddim, c}^{-1}(H){\restr}X$. 
\todog{Equivalently, $I:=\proj_{\ddim, c}^{-1}(H)\cap\dhK\ddim X$, since 
$\proj_{\ddim, c}^{-1}(\dhK c X)\cap{}^\ddim X\subseteq\dhK\ddim X$.}
\begin{enumerate-(i)}
\item\label{dif dims dh} 
$I$ is a $\ddim$-dihypergraph on $X$.
\item\label{dif dims open}
$H$ is box-open on $X$ if and only if $I$ is box-open on $X$. 
In fact, if $c$ is finite and $H$ is open on $X$, 
then $I$ is product-open on $X$.%
\footnote{I.e., $H$ is an open subset of ${}^\ddim X$ in the product topology.} 
\index{dihypergraph!product-open\idf}
\todog{This is not true for closed dihypergraphs, since maybe $I\subsetneq \proj_{\ddim, c}^{-1}(H\cup\dhC c X)\cap\dhK\ddim X$.}
\item
\label{dif dims definability}
If $c,X\in\defsetsk$, then $H\in\defsetsk$ if and only if $I\in\defsetsk$. 
\item\label{dif dims complete subgraphs}
Suppose that $Y\subseteq X$.
\begin{enumerate-(a)}
\item
$\dhK\ddim Y\not\subseteq I$ if $c\neq d$.
\todog{For example, $\langle x\rangle ^\ddim\conc\langle y\rangle^{c\setminus d}$ is in $I\setminus\dhK\ddim Y$}
\item If $|Y|\geq|\ddim|$, then
$\dhII c Y\subseteq H$ if and only if $\dhII \ddim Y\subseteq I$.
\end{enumerate-(a)}
\item\label{dif dims independent}
Suppose that $Y\subseteq X$.
\begin{enumerate-(a)}
\item
$Y$ is $H$-independent if and only if $Y$ is $I$-independent. 
\item If $|Y|\geq|\ddim|$, then
$\dhII c Y\cap H=\emptyset$ if and only if $\dhII\ddim Y\cap I=\emptyset$. 
\end{enumerate-(a)}
\item\label{dif dims colorings}
A function $f:X\to \kappa$ is a $\kappa$-coloring of $H$ if and only if it is a $\kappa$-coloring of $I$.
\item\label{dif dims homomorphisms}
There exists a continuous homomorphism from $\dhH c$ to $H$ if and only if there exists a continuous homomorphism from $\dhH\ddim$ to $I$. 
\item\label{dif dims ODD}
$\ODD\kappa{H}$ $\Longleftrightarrow$ $\ODD\kappa{I}$.
\end{enumerate-(i)}
\end{lemma}
\begin{proof}
\ref{dif dims dh}--\ref{dif dims independent} 
are easy to show. 
\ref{dif dims colorings} follows from \ref{dif dims independent}. 
To show \ref{dif dims homomorphisms}, 
suppose first that $g:{}^\kappa\ddim\to X$ is a continuous homomorphism from $\dhH\ddim$ to $I$. 
Let $f:=g\restr{}^\kappa c$.
It is easy to see that $f$ is continuous.
To see that $f$ is a homomorphism from $\dhH c$ to $H$,
suppose that~$\bar y\in \dhH c$.
Then there exists $\bar x\in{}\dhH\ddim$ such that $\proj_{\ddim, c}(\bar x)=\bar y$.
Since $g$ is a homomorphism from $\dhH\ddim$ to $I$, 
we have $g^\ddim(\bar x)\in I$.
Therefore
$f^c(\bar y)= 
g^c(\bar y)= 
\proj_{\ddim,c}\big(g^\ddim(\bar x)\big)
\in H,$
as required.
Conversely, suppose that $f:{}^\kappa c\to X$ is a continuous homomorphism from $\dhH c$ to $H$. 
Fix a map $k: d\to c$ with $k{\upharpoonright}c = \id_c$ and define $g:{}^\kappa\ddim\to X$ by letting $g(x):=f(k\circ x)$. 
It is easy to see that $g$ is continuous. 
To see that $g$ is a homomorphism from $\dhH \ddim$ to $I$, 
suppose that~$\bar y= \langle y_i : i\in \ddim\rangle \in \dhH \ddim$. 
Then $g^\ddim(\bar{y})=\langle f(k\circ y_i) : i\in \ddim\rangle$ and its projection equals $\langle f(k\circ y_i) : i\in c\rangle$. 
Suppose that $y_i$ extends $t^\smallfrown \langle i\rangle$ for all $i\in d$. 
Then $k\circ y_i$ extends $(k\circ t)^\smallfrown \langle i\rangle$ for all $i\in c$ and hence $\langle f(k\circ y_i) : i\in c\rangle\in H$. 
Thus $g^\ddim(\bar{y})\in I$. 
\ref{dif dims ODD} follows from \ref{dif dims colorings} and~\ref{dif dims homomorphisms}.
\end{proof}

\begin{lemma}
\label{comparing ODD for different dimensions}
\hypertarget{ODD^kappa implies definable ODD^ddim}{}%
Suppose that $c\subseteq d$, $|c|\geq 2$
and $X$ is a subset of ${}^\kappa\kappa$.
\begin{enumerate-(1)}
\item
\label{codd 1}
$\ODD\kappa\ddim(X)$ implies $\ODD\kappa c(X)$.
\item
\label{codd 2}
$\ODD\kappa\ddim(X,\defsetsk)$ implies $\ODD\kappa c(X,\defsetsk)$
if $c\in\defsetsk$. 
\end{enumerate-(1)}
\end{lemma}
\begin{proof} 
For \ref{codd 1}, suppose $H$ is a box-open $c$-dihypergraph on ${}^\kappa\kappa$. 
It suffices to show that $\ODD\kappa {H{\restr}X}$ holds. 
$I=\proj_{\ddim, c}^{-1}(H){\restr}{}^\kappa\kappa$ is a box-open $d$-dihypergraph on ${}^\kappa\kappa$ by \ref{dif dims dh} and \ref{dif dims open} in the previous lemma. 
By $\ODD\kappa\ddim(X)$ we have  $\ODD\kappa {I{\restr}X}$ and then obtain $\ODD\kappa {H{\restr}X}$ from \ref{dif dims ODD}. 
The proof of \ref{codd 2} is similar, but also uses \ref{dif dims definability}. 
\end{proof} 

In particular, $\ODD\kappa\kappa(X,\defsetsk)$ implies $\ODD\kappa c(X)$ for all $2\leq c<\kappa$ and $X\subseteq{}^\kappa\kappa$ by 
Lemmas~\ref{<kappa dim hypergraphs are definable} and \ref{comparing ODD for different dimensions}. 


\subsection{Independent trees}
\label{subsection: independent trees}
We assume throughout this subsection that $d$ is a set
of size at least $2$.
We characterize 
independence for box-open dihypergraphs
at the level of subtrees of ${}^{<\kappa}\kappa$.

\begin{definition}
\label{def: independent tree}
Suppose $X$ is a subset of ${}^\kappa\kappa$ and $H$ is a $\ddim$-dihypergraph on ${}^\kappa\kappa$.
\index{tree!independent\idf}%
\index{independent tree!$(X,H)$-independent tree\idf}
A subtree $T$ of ${}^{<\kappa}\kappa$ is an \emph{$(X,H)$-independent tree} if 
$\prod_{i\in\ddim}(N_{t_i}\cap X)\not\subseteq H$ 
for all sequences $\langle t_i:i\in\ddim\rangle \in{}^\ddim T$.
\index{independent tree!$H$-independent tree\idf}
We omit $X$ if~$X={}^\kappa\kappa$, i.e., $T$ is called \emph{$H$-independent} if it is $({}^\kappa\kappa,H)$-independent.%
\end{definition}

We will most often work with $H$-independent trees.
Note that if $T$ is an $(X,H)$-independent tree, then $T\subseteq T(X)$ because otherwise, there exists some $t\in T$ with $N_t\cap X=\emptyset$.

\begin{lemma}
\label{lemma: independent trees}
Suppose $H$ is a $\ddim$-dihypergraph on a subset $X$ of ${}^\kappa\kappa$.
\begin{enumerate-(1)}
\item\label{ind tree ind set}
If $H$ is box-open on $X$ and $T$ is an $(X,H)$-independent tree, then $[T]$ is an $H$-independent set. 
\item\label{ind set ind tree}
If $Y\subseteq X$ is an $H$-independent set, then $T(Y)$ is an $(X,H)$-independent tree. 
\end{enumerate-(1)}
\end{lemma}
\begin{proof}
For~\ref{ind tree ind set}, suppose that $[T]$ is not $H$-independent, and
let $\langle y_i:i\in\ddim\rangle \in{}^\ddim[T]\cap H$.
Since $H$ is box-open on $X$, 
there exist nodes $t_i\subsetneq y_i$
such that
$\prod_{i\in\ddim}(N_{t_i}\cap X)\subseteq H$.
Then $t_i\in T$ for all $i\in\ddim$, so $T$ is not $(X,H)$-independent.

For~\ref{ind set ind tree}, suppose that $Y\subseteq X$ and $T(Y)$ is not $(X,H)$-independent. 
Take $\langle t_i:i\in\ddim\rangle$ in ${}^\ddim T(Y)$ such that
$\prod_{i\in\ddim}(N_{t_i}\cap X)\subseteq H$.
For all $i\in\ddim$, 
fix $y_i\in Y$ with $t_i\subseteq y_i$.
Then $\langle y_i:i\in\ddim\rangle$ is a hyperedge of $H$.
\end{proof}

\begin{corollary}
\label{independence and closure}
Suppose $H$ is a box-open $\ddim$-dihypergraph on a subset $X$ of ${}^\kappa\kappa$.
If $Y\subseteq X$ is $H$-independent, then $\closure Y$
is $H$-independent.%
\footnote{Recall that $\closure Y$ denotes the closure of $Y$ in ${}^\kappa\kappa$.}
\end{corollary}
\begin{proof}
The statement follows from the previous lemma, since $\closure Y=[T(Y)]$.
\end{proof}


\subsection{Order homomorphisms}
\label{subsection: order homomorphisms}

We assume throughout this subsection that $d$ is a discrete topological space of size at least $2$.
We first characterize the existence of a continuous homomorphism from $\dhH\ddim$ to 
box-open dihypergraphs
using the following types of strict order preserving maps.
We use terminology from Subsection \ref{preliminaries: trees}.

\begin{definition}
\label{def: order homomorphism}
Suppose $X\subseteq{}^\kappa\kappa$ and $H$ is a $\ddim$-dihypergraph on ${}^\kappa\kappa$.
A strict order preserving map $\iota:{}^{<\kappa}\ddim\to{}^{<\kappa}\kappa$ 
\index{order homomorphism for!dihypergraphs on ${}^\kappa\kappa$\idf}%
\index{homomorphism!Z@\mygobble|see {order homomorphism}} 
is an \emph{order homomorphism for $(X,H)$} if
\begin{enumerate-(a)}
\item \label{oh1} 
$\ran([\iota])\subseteq X$, and 
\item \label{oh2}
$\prod_{i\in \ddim}(N_{\iota(t\conc\langle i\rangle)}\cap X)\subseteq H$
 for all $t\in{}^{<\kappa}\ddim$.
\end{enumerate-(a)}
We omit $X$ if it equals $\domh H$, i.e., 
$\iota$ is an \emph{order homomorphism for $H$} if it is 
an {order homomorphism for $(\domh H,H)$}. 
\end{definition}

We will most often work with order homomorphisms for $H$. 
If $\domh H\subseteq X$, then any order homomorphism $\iota$ for $(X,H)$ is also an order homomorphism for $H$,
since \ref{oh1} and \ref{oh2} imply $\ran([\iota])\subseteq \domh H$. 
If $\domh{H}\subseteq X$ is relatively open,%
\footnote{For instance, if $H$ is box-open on $X$.} 
then the converse holds: 
if there exists an order homomorphism $\iota$ for $H$, then there exists one for $(X,H)$. 
\todog{The special case when $H$ is box-open on $X$ also follows from Lemma~\ref{homomorphisms and order preserving maps} below.} 
To see this, one constructs by recursion a strict order preserving map $e:{}^{<\kappa}\ddim\to{}^{<\kappa}\ddim$ such that $\theta:=\iota\circ e$ satisfies $N_{\theta(t\conc\langle i\rangle))}\cap X\subseteq 
N_{\iota(e(t)\conc\langle i\rangle)}\cap \domh{H}$ for all $t\in{}^{<\kappa}\kappa$; 
then $\theta$ is as required. 

\begin{remark} 
\label{remark: H^-_X}
\label{remark: HH_iota}
Suppose $X$ and $H$ are as in the previous definition
and $\iota:{}^{<\kappa}\ddim\to T(X)$.
We define the following $\ddim$-dihypergraphs on ${}^{<\kappa}\ddim$ and ${}^{<\kappa}\kappa$, respectively: 
\begin{enumerate-(1)} 
\item 
$\HH^-_{{}^{<\kappa}\ddim}$ 
\index{dihypergraph!on kk@on ${}^{<\kappa}\kappa$!of uniformly splitting\newline\hspace{32 pt}sequences\idf$\HH^-_{{}^{<\kappa}\ddim}$}
consists of all sequences in ${}^{<\kappa}\ddim$ of the form
$\langle t\conc\langle \alpha\rangle: \alpha<\ddim\rangle$. 
We call this the \emph{dihypergraph on ${}^{<\kappa}\kappa$ of uniformly splitting sequences}.
\item 
$H^-_X$ 
\index{dihypergraph!on kk@on ${}^{<\kappa}\kappa$!inducing $H\restr X$\idf $H^-_X$}
consists of all sequences $\langle u_\alpha:\alpha<\ddim\rangle$ in $T(X)$
with $\prod_{\alpha<\kappa} (N_{u_\alpha}\cap X)\subseteq H$.%
\footnote{Note that a subtree $T$ of $T(X)$ is $(X,H)$-independent in the sense of Definition~\ref{def: independent tree} if and only if it is $H^-_X$-independent in the sense that $H^-_X\restr T=\emptyset$.}
\end{enumerate-(1)} 
Using these definitions, \ref{oh2} in Definition \ref{def: order homomorphism} is equivalent to the statement that 
$\iota$
is a homomorphism from $\HH^-_{{}^{<\kappa}\ddim}$ to $H^-_X$. 
\todog{This motivates the name ``order homomorphism''.}

Let 
\index{dihypergraph!induced by a strict order preserving function\idf$\HH_{\iota,X},\,\HH_\iota$}
$\HH_{\iota,X} := 
\bigcup_{t\in{}^{<\kappa}\ddim}\prod_{i\in \ddim} (N_{\iota(t\conc\langle i\rangle)}\cap X) 
$. 
\ref{oh2} is further equivalent to the inclusion 
$\HH_{\iota,X}\subseteq H$. 
Note that $\HH_{\iota,X}$ is a box-open subset of ${}^\ddim X$. 
It is a dihypergraph if and only if 
$X\cap\bigcap_{i\in\ddim} N_{\iota(t\conc\langle i\rangle)}=\emptyset$ for all $t\in{}^{<\kappa}\ddim$ since otherwise it contains constant sequences. 
For $\HH_{\iota,X}$ to be a dihypergraph, it suffices that for every $t\in {}^{<\kappa}\ddim$, there exist $i,j\in \ddim$ with $\iota(t\conc\langle i \rangle)\perp \iota(t\conc\langle j \rangle)$. 
This condition is necessary if $\ran(\iota)\subseteq T(X)$,  
(a) 
$T(X)$ is ${\leq}|\ddim|$-closed if $|d|<\kappa$ 
and (b) 
$X$ is closed and $T(X)$ is ${<}\kappa$-closed if $|d|=\kappa$. 
\end{remark}

The next lemma shows that for relatively box-open dihypergraphs $H$, 
the existence of an order homomorphism for $H$ 
is equivalent to the existence of a continuous homomorphism from $\dhH\ddim$ to $H$.
We state a more general version for order homomorphisms for $(X,H)$ which implies 
the previous statement in the special case $X=\domh H$.

\begin{lemma} 
\label{homomorphisms and order preserving maps}
Suppose $X\subseteq{}^\kappa\kappa$ and $H$ is a $\ddim$-dihypergraph on ${}^\kappa\kappa$.
\begin{enumerate-(1)}
\item\label{hop 2} 
If $\iota$ is an order homomorphism for $(X,H)$, then 
$[\iota]$ is a continuous homomorphism from $\dhH\ddim$ to $H\restr X$. 
\item\label{hop 1} 
If $H\restr X$ is box-open on $X$ and $f$ is
\todog{$X$ is important here, e.g when $\domh H\supseteq X$, an order homomorphism for $(X,H)$ is a stronger notion than an order homomorphism for $H$ whose range is a subset of $X$.}
a continuous homomorphism 
from $\dhH\ddim$ to~$H\restr X$, then
there exists a continuous%
\footnote{%
See Definition~\ref{def: order preserving}~\ref{cont def}.
}
order homomorphism $\iota$ for $(X,H)$
and a continuous strict $\wedge$-homomorphism%
\footnote{See Definition~\ref{def: wedge-homomorphism}. Note that $e:{}^{<\kappa}\ddim\to{}^{<\kappa}\kappa$  is a strict $\wedge$-homomorphism if and only if it is strict order preserving and $e(t\conc\langle i\rangle) \wedge e(t\conc\langle j\rangle) =e(t)$ for all $t\in{}^{<\kappa}\ddim$ and $i,\,j\in\ddim$ with $i\neq j$. It is easy to see that such a map $e$ is $\perp$-preserving, and hence $[e]$ is a homeomorphism onto its image by Lemma~\ref{[e] continuous}~\ref{[e] homeomorphism}.} 
$e:{}^{<\kappa}\ddim\to{}^{<\kappa}\ddim$
with $[\iota]=f\comp [e]$.%
\footnote{In general, it is not possible to find a strict order preserving map $\iota$ that induces the original function $f$ (i.e., $[\iota]=f$). For instance, let $f$ map $N_{\langle 0\rangle}$ onto $N_{\langle 0\rangle} \cup N_{\langle 1\rangle}$. Note that it is always possible to find an order preserving map $\iota$ that induces $f$ \cite{andretta2022souslin}*{Lemma~3.7}. But order preserving maps $\iota$ as in Definition~\ref{def: order homomorphism}~\ref{oh2} for graphs $H$ are always $\perp$- and hence strict order preserving.}
\end{enumerate-(1)}
In particular, \ref{hop 2} and \ref{hop 1} apply to order homomorphisms for $H$ and continuous homomorphisms from $\dhH\ddim$ to $H$ by letting $X=\domh H$.
\end{lemma}
\begin{proof}
We first show \ref{hop 2}. 
Since $\iota$ is strict order preserving, 
$[\iota]$ is a continuous function from ${}^{\kappa}\ddim$ to $X$
by Lemma~\ref{[e] continuous}.
It thus suffices to show that $[\iota]$ is a homomorphism from $\dhH\ddim$ to $H$.
Suppose $\bar x=\langle x_i:i\in\ddim\rangle\in\dhH\ddim$.
Let $u\in{}^{<\kappa}\ddim$ be such that
$u\conc\langle i\rangle\subseteq x_i$ 
for all $i\in\ddim$.
Then
$\iota(u\conc\langle i\rangle)\subseteq [\iota](x_i)$ 
for all $i\in\ddim$, and therefore
\[
{[\iota]^\ddim}(\bar x)=
\langle[\iota](x_i):i\in\ddim\rangle\in
\prod_{i\in\ddim}(N_{\iota(u\conc\langle i\rangle)}\cap X)\subseteq H.
\]

To prove \ref{hop 1}, let $f$ be a continuous homomorphism from $\dhH\ddim$ to $H\restr X$. 
We construct 
continuous strict order preserving 
maps
$\iota:{}^{<\kappa}\ddim\to{}^{<\kappa}\kappa$ 
and 
$e:{}^{<\kappa}\ddim\to{}^{<\kappa}\ddim$
such that for all $t\in{}^{<\kappa}\ddim$ and $i\in\ddim$,
\begin{enumerate-(i)}
\item\label{hop proof e}
$e(t)\conc\langle i\rangle\subseteq e(t\conc\langle i\rangle)$, 
\item\label{hop proof 1}
$f(N_{e(t)})\subseteq N_{\iota(t)}$, and 
\item\label{hop proof 2} 
$\prod_{i\in\ddim}(N_{\iota(t\conc\langle i\rangle)}\cap X)\subseteq H$.
\end{enumerate-(i)}
We construct $\iota(t)$ and $e(t)$ by recursion on $\lh(t)$.
Let $\iota(\emptyset):=e(\emptyset)=\emptyset$.
Let $t\in{}^{<\kappa}\ddim\,\setminus\{\emptyset\}$, and suppose that $\iota(u)$ and $e(u)$ have been defined for all $u\subsetneq t$.
If $\lh(t)\in\Lim$,
let $\iota(t):=\bigcup_{u\subsetneq t}\iota(u)$ and
 $e(t):=\bigcup_{u\subsetneq t}e(u)$. 
Since 
$N_{e(t)}=\bigcap_{u\subsetneq t}N_{e(u)}$
and \ref{hop proof 1} holds for all $u\subsetneq t$, 
we have
\[
f(N_{e(t)})\subseteq
\bigcap_{u\subsetneq t}f(N_{e(u)})\subseteq
\bigcap_{u\subsetneq t}N_{\iota(u)}
=N_{\iota(t)}.
\] 
Suppose $\lh(t)\in\Succ$, and let $u$ be the direct predecessor of $t$.
We define 
$\iota(u\conc\langle i\rangle)$ 
and
$e(u\conc\langle i\rangle)$ 
for each $i\in\ddim$ simultaneously as follows. 
Let 
$x_i\in{}^{\kappa}\ddim$ extend $e(u)\conc\langle i\rangle$ for each $i\in\ddim$.
Since $f$ is a homomorphism from $\dhH\ddim$ to $H\restr X$, the sequence 
$\langle f(x_i):i\in\ddim\rangle$ is a hyperedge of $H\restr X$. 
Since $H\restr X$ is box-open on $X$, there exists a sequence
$\langle\iota(u\conc\langle i\rangle): i\in\ddim\rangle$ 
such that 
$\iota(u)\subsetneq\iota(u\conc\langle i\rangle)\subsetneq f(x_i)$ 
and \ref{hop proof 2} holds for $u$.
By continuity of $f$, there exists for each $i\in\ddim$
an initial segment $e(u\conc\langle i\rangle)$ of $x_i$ such that $e(u)\conc\langle i\rangle\subseteq e(u\conc\langle i\rangle)$ and 
\ref{hop proof 1} holds for $u\conc\langle i\rangle$.

Suppose the continuous strict order preserving 
maps $\iota$ and $e$ have been constructed.
$e$~is a strict $\wedge$-homomorphism by~\ref{hop proof e}.
\ref{hop proof 1} implies that
for all $x\in{}^\kappa\ddim$, 
\[
f\big([e](x)\big)\in
\bigcap_{t\subsetneq x}f(N_{e(t)})\subseteq
\bigcap_{t\subsetneq x}N_{\iota(t)}
=\big\{[\iota](x)\big\}.
\]
Thus $[\iota]=f\comp[e]$, and therefore
$\ran([\iota])\subseteq\ran(f)\subseteq X$.
This fact and~\ref{hop proof 2} imply that $\iota$ is an order homomorphism for $(X,H)$.
\end{proof}


\begin{remark}
\label{remark: order homomorphisms even length}
Fix any club $C$ in $\kappa$. 
The construction in the proof of Lemma \ref{homomorphisms and order preserving maps} \ref{hop 1} can be easily modified to obtain $\lh(\iota(t))\in C$ for all $t\in {}^{<\kappa}\ddim$ in addition to the above properties. 
\end{remark}

\subsection{Full dihypergraphs}
\label{subsection: full dihypergraphs}
In this subsection, 
$\ddim$ and $c$ are discrete topological spaces of cardinality at least~$2$.
We show that if $H$ and $I$ are relatively box-open 
dihypergraphs on ${}^\kappa\kappa$,
then 
$\ODD\kappa H\Longleftrightarrow\ODD\kappa I$ 
whenever $H\equivf I$ in the sense of the next definition.

\begin{definition} 
\label{def: H-full}
Suppose that $H$ and $I$ are $\ddim$- and $c$- dihypergraphs, respectively, on 
the same set $X$. 
\begin{enumerate-(a)}
\item\label{def: H-full 1}
\index{dihypergraph!fullness, equivalence\idf}%
\index{full!full@$H$-full\idf$H\leqf I$} 
$I$ is \emph{$H$-full} ($H\leqf I$)
if $H\restr A$ admits a $\kappa$-coloring for every $I$-independent subset $A$ of~$X$.%
\footnote{Equivalently, $H\restr A$ admits a $\kappa$-coloring for all $A\subseteq X$ such that $I\restr A$ admits a $\kappa$-coloring.}
\item\label{def: H-full 2}
\index{full!Z@equivalence\idf$H\equivf I$, $H\equivif I$} 
$H\equivf I$ 
if for all $A\subseteq X$,
$H\restr A$ admits a $\kappa$-coloring if and only if $I\restr A$ admits a $\kappa$-coloring.
\end{enumerate-(a)}
\end{definition}

Note that $H\equivf I$ if and only if $H\leqf I$ and $I\leqf H$.
For example, 
$H\leqf I$ if $H\subseteq I$ or if $c\leq\ddim$ are ordinals and
any hyperedge of $H$ has a subsequence in $I$, 
since  in these cases, every $I$-independent subset of ${}^\kappa\kappa$ is $H$-independent.
\todog{Previously, we defined $I$ to be $H$-full if every $I$-independent subset of ${}^\kappa\kappa$ was also $H$-independent}

\begin{lemma}
\label{lemma: ODD subsequences} 
Suppose $H$ and $I$ are relatively box-open $\ddim$- and $c$- dihypergraphs, respectively, on ${}^\kappa\kappa$.
If $H\leqf I$ and
there exists a 
continuous homomorphism 
from $\dhH \ddim$ to $H$, 
then there exists a continuous homomorphism 
from $\dhH  c$ to~$I$.
\end{lemma}
\begin{proof}
Suppose that $f$ is a continuous homomorphism from $\dhH \ddim$ to $H$.
\todoo{The next two sentences changed to make the proof (and the next remark) a bit clearer}
The next claim is needed 
to ensure that 
$\ran([\iota])\subseteq \domh{I}$ for the map $\iota$ constructed below.
Note that it is immediate in the case that
$\domh{H}\subseteq \domh{I}$ 
by letting $U_\alpha:={}^\kappa\kappa$.
\begin{claim*}
There is a sequence $\langle U_\alpha : \alpha<\kappa\rangle$ of open subsets of ${}^\kappa\kappa$ such that their intersection $U:=\bigcap_{\alpha<\kappa} U_\alpha$ satisfies
$U\cap\domh H\subseteq\domh I$ and
${I\restr(U\cap f(N_{t}))}\neq\emptyset$ 
\todoo{We corrected the typo: we wrote $f(N_{t})$ instead of $f(N_{e(t)})$} 
for all $t\in{}^{<\kappa}\ddim$.
\end{claim*}
\begin{proof}
Since $A:=\domh H\setminus\domh I$ is $I$-independent and $H\leqf I$, $A$ is the union of $H$-independent sets $A_\alpha$ for $\alpha<\kappa$.
We show that the sets $U_\alpha:={}^\kappa\kappa\oldsetminus \closure{A}_\alpha$ are as required. 
The first requirement is immediate. 
Seeking a contradiction, suppose that $I\restr (U\cap f(N_t))=\emptyset$ for some $t\in{}^\kappa\ddim$.
Then $H\restr (U\cap f(N_t))$ admits a $\kappa$-coloring
since $H\leqf I$.
Moreover, $H\restr (f(N_t)\setminus U)$ also admits a $\kappa$-coloring, since $\closure A_\alpha$ is $H$-independent for each $\alpha<\kappa$ by Corollary~\ref{independence and closure}.
\todog{This is the only place we need to use that $H$ is relatively box-open}
Hence $H\restr f(N_t)$ admits a $\kappa$-coloring, contradicting Lemma~\ref{two options in ODD are mutually exclusive}~\ref{me 2}.
\end{proof}

We construct 
continuous strict order preserving 
maps
$\iota:{}^{<\kappa} c\to{}^{<\kappa}\kappa$ 
and 
$e:{}^{<\kappa} c\to{}^{<\kappa} \ddim$
such that for all $t\in{}^{<\kappa} c$ and $i\in c$,
\begin{enumerate-(i)}
\item\label{subs proof 1}
$f(N_{e(t)})\subseteq N_{\iota(t)}$,
\item\label{subs proof 2} 
$\prod_{j\in  c}(N_{\iota(t\conc\langle j\rangle)}\cap\domh I)\subseteq I$, and 
\item\label{subs proof new} 
$N_{\iota(t\conc\langle i\rangle)}\subseteq U_{\lh(t)}$. 
\end{enumerate-(i)}
$\iota(t)$ and $e(t)$ are constructed by recursion on $\lh(t)$
by modifying the construction for Lemma \ref{homomorphisms and order preserving maps}~\ref{hop 1} at successor lengths. In detail, let 
$\iota(\emptyset):=e(\emptyset)=\emptyset$.
Let $t\in{}^{<\kappa} c\,\setminus\{\emptyset\}$, and suppose that $\iota(u)$ and $e(u)$ have been defined for all $u\subsetneq t$.
If $\lh(t)\in\Lim$,
let $\iota(t):=\bigcup_{u\subsetneq t}\iota(u)$ and
 $e(t):=\bigcup_{u\subsetneq t}e(u)$. 
Then \ref{subs proof 1} holds for 
$t$ similarly to the proof of Lemma~\ref{homomorphisms and order preserving maps}~\ref{hop 1}.
Now, suppose $\lh(t)\in\Succ$ and
let $u$ be the direct predecessor of $t$.
Define 
$\iota(u\conc\langle i\rangle)$ 
and
$e(u\conc\langle i\rangle)$ 
for each $i\in  c$ simultaneously as follows. 
Let $\bar x:=\langle x_i:i\in c\rangle$ be a sequence in $N_{e(u)}$
with $\langle f(x_i):i\in c\rangle\in I\restr U$.
Since $I$ is relatively box-open and $f(x_i)\in U\cap\domh I \subseteq U_{\lh(u)}\cap \domh I$ for each $i\in c$,
take a sequence
$\langle\iota(u\conc\langle i\rangle): i\in  c\rangle$ 
with
$\iota(u)\subsetneq\iota(u\conc\langle i\rangle)\subsetneq f(x_i)$ 
such that \ref{subs proof 2} and \ref{subs proof new} hold for $u$.
By the continuity of $f$, choose 
$e(u\conc\langle i\rangle)$ with $e(u)\subsetneq e(u\conc\langle i\rangle)\subsetneq x_{i}$ such that 
\ref{subs proof 1} holds for $u\conc\langle i\rangle$ for each $i\in c$.

Suppose that $\iota$ and $e$ have been constructed.
Then $[\iota]=f\comp[e]$ 
as in the proof of Lemma~\ref{homomorphisms and order preserving maps}~\ref{hop 1}, so 
$\ran([\iota])\subseteq\ran(f)\subseteq\domh H$. Moreover, $\ran([\iota])\subseteq U$ by~\ref{subs proof new}, and hence $\ran([\iota])\subseteq \domh I$.
Thus $\iota$ is an order homomorphism for $I$ by~\ref{subs proof 2},
so $[\iota]$ is a continuous homomorphism from $\dhH  c$ to $I$ by
Lemma~\ref{homomorphisms and order preserving maps}~\ref{hop 2}.
\end{proof}

\begin{corollary}
\label{cor: ODD subsequences}
Suppose $H$ and $I$ are relatively box-open $\ddim$- and $c$- dihypergraphs, respectively, on ${}^\kappa\kappa$.
If $H\equivf I$, then
$\ODD\kappa{H}\Longleftrightarrow\ODD\kappa{I}$.%
\footnote{Equivalently, there exists a continuous homomorphism from $\dhH\ddim$ to $H$ if and only if there exists a continuous homomorphism from $\dhH c$ to $I$.}
\end{corollary}

\begin{remark}
\label{remark: ODD subsequences}
Suppose that $H$ and $I$ are $\ddim$- and $c$- dihypergraphs, respectively, on ${}^\kappa\kappa$ and $I$ is box-open on some superset $X$ of $\domh H$. 
\begin{enumerate-(1)}
\item\label{remark: ODD subsequences 1} 
The statement of Lemma \ref{lemma: ODD subsequences} holds even if we remove the assumption that $H$ is relatively box-open.
\todoo{The first few lines of~\ref{remark: ODD subsequences 1} were modified to make the presentation clearer, per the referees request}
To see this, we modify its proof as follows. 
Note that the claim will not be needed. 
We construct $\iota$ and $e$ satisfying \ref{subs proof 1} and \ref{subs proof 2} with $\domh{I}$ replaced by $X$. 
The construction works as above, except at successor lengths when $\iota(u)$ and $e(u)$ have been constructed,
it suffices to find a hyperedge
$\langle f(x_i): i\in c\rangle$ of $I\restr f(N_{e(u)})$.
This can be done since 
$H\leqf I$ and
$H\restr f(N_{e(u)})$ does not admit a $\kappa$-coloring 
by Lemma~\ref{two options in ODD are mutually exclusive}~\ref{me 2}.
Since $I$ is box-open on $X$, choose $\iota(u\conc\langle i\rangle)\subsetneq f(x_i)$ extending $\iota(u)$ for each $i\in c$ so that 
$\prod_{i\in  c}(N_{\iota(u\conc\langle i\rangle)}\cap X)\subseteq I$, 
and choose each $e(u\conc\langle i\rangle)$
as above.
Then $[\iota]=f\comp[e]$
as in the proof of Lemma~\ref{homomorphisms and order preserving maps}~\ref{hop 1}, 
so $\ran([\iota])\subseteq \domh H\subseteq X$.
Therefore $\iota$ is an order homomorphism for $(X,I)$ and hence $[\iota]$ is a continuous homomorphism from $\dhH c$ to $I$.
\item\label{remark: ODD subsequences 2} 
If $c=\ddim$ and $I\subseteq H$,
then the assumption that $H$ is relatively box-open is not needed in 
Corollary~\ref{cor: ODD subsequences}.
This follows from~\ref{remark: ODD subsequences 1} and the fact that any homomorphism from $\dhH\ddim$ to $I$ is a homomorphism to $H$. 
\end{enumerate-(1)}
By~\ref{remark: ODD subsequences 2},
$\ODD\kappa\ddim(X)$ implies $\ODD\kappa{H}$ 
if $H$ contains a subdihypergraph $I$ which is box-open on $X$ and $H$-full, even if $H$ is not relatively box-open itself. 
\end{remark}

\begin{remark} 
\label{remark: ODD injective subsequences} 
Suppose that $H$ and $I$ are $\ddim$- and $c$- dihypergraphs, respectively, on ${}^\kappa\kappa$. 
We call $I$ \emph{injectively $H$-full} ($H\leqif I$) 
\index{full!injectively $H$-full\idf$H\leqif I$} 
if for every hyperedge $\bar x=\langle x_j:j\in \ddim\rangle$ of $H$,
there exists an injective map $\pi:c\to\ddim$
with $\langle x_{\pi(i)}:i\in c\rangle\in I$.%
\footnote{If $I$ is injectively $H$-full, then $I$ is $H$-full 
since every $I$-independent subset of ${}^\kappa\kappa$ is also $H$-independent.}
We write $H\equivif I$ 
\index{full!Z@equivalence\idf$H\equivf I$, $H\equivif I$} 
if $H\leqif I$ and $I \leqif H$. 
\todog{If $I$ is injectively $H$-full, then $|c|\leq\|\ddim|$.}
For instance, if $c\leq \ddim$ are ordinals and any hyperedge of $H$ has a subsequence in $I$, 
then $I\leqif H$. 
If $I\leqif H$ and $I$ is box-open on a set $X$ with $\domh H\subseteq X$, then we can assume that the continuous homomorphism from $\dhH c$ to $I$ in  Lemma~\ref{lemma: ODD subsequences}
is of the form $f\comp [e]$
for some continuous strict $\wedge$-homomorphism $e:{}^{<\kappa} c\to{}^{<\kappa}\ddim$ and some continuous homomorphism $f$ from $\dhH\ddim$ to $H$. 
Moreover, we do not need to assume that $H$ is relatively box-open.

To see this, we modify the construction in the previous remark at successor lengths as follows.
Assuming $\iota(u)$ and $e(u)$ have been constructed, 
let
$x_j\in{}^{\kappa} \ddim$ extend $e(u)\conc\langle j\rangle$ for each $j\in \ddim$.
Since $\langle f(x_j):j\in  \ddim\rangle\in H$ and $I\leqif H$, 
$\langle f(x_{\pi_u(i)}):i\in  c\rangle$ is in $I$ for some injective map $\pi_u:c\to\ddim$. 
Since $I$ is box-open on $X$, take a sequence
$\langle\iota(u\conc\langle i\rangle): i\in  c\rangle$ 
such that 
$\iota(u)\subsetneq\iota(u\conc\langle i\rangle)\subsetneq f(x_{\pi_u(i)})$ 
and $\prod_{j\in  c}(N_{\iota(t\conc\langle j\rangle)}\cap X)\subseteq I$.
By the continuity of $f$, choose 
$e(u\conc\langle i\rangle)$ 
with
$e(u)\conc\langle\pi_u(i)\rangle\subseteq e(u\conc\langle i\rangle)\subsetneq x_{\pi_u(i)}$ 
which also satisfies \ref{subs proof 1} from the previous construction.
As above, $\iota$ is an order homomorphism for $(X,I)$ with $[\iota]=f\circ [e]$. 
Since $\pi_u$ is injective for all $u\in {}^{<\kappa}\kappa$, 
$e(u\conc\langle i\rangle)\wedge e(u\conc\langle j\rangle)=e(u)$ for all $i\neq j\in c$, so $e$ is a strict $\wedge$-homomorphism.

It follows that in this situation, $[e]$ is a homeomorphism onto its image. 
Thus, if $f$ is injective or a homeomorphism onto its image, then the same holds for $f\comp [e]$. This will be relevant in  Subsection~\ref{subsection: ODD ODDI ODDH}. 
\end{remark}

\newpage 

\section{The open dihypergraph dichotomy for definable sets}
\label{section: proof of the main result}
In this section we prove the main result, Theorem~\ref{main theorem}.
The argument consists of three steps. 
In the situation of the theorem, 
$\kappa$ is a regular infinite cardinal, 
$2\leq\ddim\leq \kappa$, and
$X$ is a subset of ${}^\kappa\kappa$ and 
$H$ is a box-open $\ddim$-dihypergraph on ${}^\kappa\kappa$ 
in some $\Col(\kappa,\lle\lambda)$-generic extension $V[G]$. 
We may assume that $H\restr X$ does not admit a $\kappa$-coloring. 
In the first step, we find an ordinal $\nu<\lambda$ that reflects the set $\mathcal T$ of trees that code $H$-independent closed sets
in the sense that
${\mathcal T}\cap V[G_\nu] \in V[G_\nu]$, 
where
$G_\nu$ denotes the restriction of $G$ to $\Col(\kappa,\lle\nu)$. 
In the second step, 
we work in $V[G_\nu]$ 
to find a $\Col(\kappa,\lle\lambda)$-name 
for an element of $X$ 
that avoids all closed sets induced by trees in $\mathcal T \cap V[G_\nu]$. 
Then we 
transform this to an $\Add(\kappa,1)$-name $\nna$ to facilitate the third step, 
where $\nna$ is used 
to construct a continuous homomorphism 
from $\dhH\ddim$ to $H\restr X$. 
The values of this homomorphism come from the evaluations of $\nna$ 
by a $\kappa$-perfect set of 
generic
filters with well-behaved quotient forcings.
The main technical step in the proof, the \emph{Nice Quotient Lemma for $\Col(\kappa,\lle\lambda)$},
guarantees the existence 
in $V[G]$ 
of such a $\kappa$-perfect set.
More precisely, it constructs
a continuous function $g$ 
from ${}^\kappa\ddim$ 
to the set of $\Add(\kappa,1)$-generic filters with well-behaved quotient forcings
such that $g$ induces a
homomorphism from $\dhH\ddim$ to $H$ via~$\nna$. 
Controlling the quotient forcings is necessary to ensure the values of the homomorphism are in~$X$.
The proof of the Nice Quotient Lemma is postponed to Subsections~\ref{section: proof in the countable case} and~\ref{section: proof in the uncountable case}.

In the countable case, 
our proof \todog{of the main theorem} 
resembles 
Solovay's proof of the consistency of  the perfect set property $\PSP_\omega(\defsets\omega)$~\cite{Solovay1970} and 
Feng's proof of the consistency of the open graph dichotomy $\OGD_\omega(\defsets\omega)$ \cite{FengOCA}.%
\footnote{See Definitions~\ref{OGD def} and~\ref{def: perfect set}.}
In this case, the Nice Quotient Lemma is an easy consequence of Solovay's lemma showing that all elements of the Baire space in $V[G]$ have $\Col(\omega,\lle\lambda)$ as a quotient. 
However, it is well known that Solovay's lemma fails in the uncountable case.
The Nice Quotient Lemma overcomes this difficulty by ensuring that 
all elements of the $\kappa$-Baire space
that occur in the construction do have $\Col(\kappa,\lle\lambda)$ as a quotient in the uncountable case.
This was done in \cite{SchlichtPSPgames} 
for the $\kappa$-perfect set property using a similar idea but different technical means.

\subsection{Three important steps}

\subsubsection{Reflection to intermediate models}
\label{subsection: restricting sets to intermediate models}

We will work with the L\'evy collapse of an inaccessible or a 
Mahlo cardinal.%
\footnote{A cardinal $\lambda$ is \emph{Mahlo} if the set of inaccessible cardinals $\nu<\lambda$ is stationary in $\lambda$.}  
The next lemma describes an essential property of such extensions. 

Recall the notation $\PP_\alpha, \PP_I, \PP^\alpha, G_\alpha, G_I, G^\alpha$ from page~\pageref{notation for the Levy collapse and subforcings}. 
Note that if $\lambda$ is inaccessible, then $\PP^\alpha$ is equivalent to $\PP_\lambda$ for all $\alpha<\lambda$, by Corollary~\ref{forcing equivalent to the Levy collapse}.

\begin{lemma} 
\label{restricting sets to intermediate models} 
Suppose that $\kappa<\lambda$ are regular infinite cardinals, $G$ is $\PP_\lambda$-generic over $V$ and $X$ is a subset of ${}^\kappa\kappa$ in $V[G]$. 
Suppose that either 
\begin{enumerate-(1)} 
\item 
\label{restricting sets to intermediate models 1} 
$\lambda$ is inaccessible and $X\in \defsets\kappa^{V[G]}$, 
or 
\item 
\label{restricting sets to intermediate models 2} 
$\lambda$ is Mahlo. 
\end{enumerate-(1)} 
Then there is a stationary (in $V$) subset $S\in V$ of $\lambda$ such that for all $\nu\in S$,  we have $X \cap V[G_\nu] \in V[G_\nu]$. 
In fact, assuming \ref{restricting sets to intermediate models 1}, $S$ can be chosen as a final segment of $\lambda$. 
Assuming~\ref{restricting sets to intermediate models 2}, $S$ can be chosen so that every $\nu\in S$ is inaccessible. 
\end{lemma} 
\begin{proof} 
First assume \ref{restricting sets to intermediate models 1}. 
In $V[G]$, let 
$$X:=X_{\varphi,a}:=\{ x\in {}^\kappa\kappa: \varphi(x,a) \},$$
where $\varphi(x,a)$ is a formula with a parameter $a\in {}^\kappa\Ord$. 
Since $\PP_\lambda$ has the $\lambda$-c.c.,~$a\in V[G_\alpha]$ for some $\alpha<\lambda$.
We shall show that $X\cap V[G_\nu]\in V[G_\nu]$ for any $\nu\in S:=[\alpha,\lambda)$.
Let $x\in({}^\kappa\kappa)^{V[G_\nu]}$.
Since $a$ is also in $V[G_\nu]$ and
$V[G]$ is a $\PP^\nu$-generic extension of $V[G_\nu]$ and
$\PP^\nu$ is homogeneous,
\[
x\in X=X_{\varphi,a} 
\ \ \Longleftrightarrow \ \ 
\one_{\PP^\nu}\forces^{V[G_\nu]} \varphi(x,a).
\]
Thus,
$X\cap V[G_\nu]$ is equal to the set $\{ x\in({}^\kappa\kappa)^{V[G_\nu]}: 
\one_{\PP^\nu}\forces^{V[G_\nu]} \varphi(x,a)\}$, 
which is in $V[G_\nu]$.

Now assume \ref{restricting sets to intermediate models 2}. 
Let $\dot{X}\in V$ be a $\PP_\lambda$-name with $\dot{X}^G=X$. 
Define $f: \lambda\to \lambda$ as follows. 
Given $\alpha<\lambda$, let $\mathrm{Name}_\alpha$ denote the set of nice $\PP_\alpha$-names for subsets of $\kappa\times \kappa$. 
Note that $|\mathrm{Name}_\alpha |\leq2^{\max\{\kappa,|\alpha|\}}$. 
For each $\dot{x}\in \mathrm{Name}_\alpha$, choose a maximal antichain $A_{\dot{x}}$ in $\PP_\lambda$ of conditions deciding whether $\dot{x}\in \dot{X}$. 
By the $\lambda$-c.c.~of $\PP_\lambda$, 
let $\mu_{\dot{x}}<\lambda$ be least with $A_{\dot{x}}\subseteq \PP_{\mu_{\dot{x}}}$.
Finally, let $f(\alpha):=\sup \{ \mu_{\dot{x}}: \dot{x}\in \mathrm{Name}_\alpha\}<\lambda$. 

In $V$, there is a club of closure points of $f$, i.e. ordinals $\nu<\lambda$ with  $f(\alpha)<\nu$ for all $\alpha<\nu$. 
Since $\lambda$ is Mahlo, the set $S$ of inaccessible closure points of $f$ is stationary in $\lambda$. 

Let $\nu\in S$.
We shall show that $X\cap V[G_\nu] \in V[G_\nu]$.

\begin{claim*} 
For any nice $\PP_\nu$-name $\dot{x}$ for a subset of $\kappa\times\kappa$, exactly one of the following holds: 
\begin{enumerate-(i)} 
\item 
$\exists p\in G_\nu\ \,p\forces_{\PP_\lambda}^V \dot{x} \in \dot{X}$.
\item 
$\exists p\in G_\nu\ \,p\forces_{\PP_\lambda}^V \dot{x} \notin \dot{X}$. 
\end{enumerate-(i)} 
\end{claim*} 
\begin{proof} 
Since $\nu$ is inaccessible, $\PP_\nu$ has the $\nu$-c.c. 
Since $\dot{x}$ is a nice $\PP_\nu$-name, it is a $\PP_\gamma$-name for some $\gamma<\nu$. 
By the definition of $f$, there is
$A\subseteq \PP_{f(\gamma)}$ such that $A$ is
a maximal antichain in $\PP_\lambda$ of conditions deciding whether $\dot{x}\in \dot{X}$. 
Since $\nu$ is a closure point of $f$, we have $f(\gamma)<\nu$. 
Thus $A$ is a maximal antichain in $\PP_\nu$. 
Let $p\in G_\nu \cap A$. 
\end{proof}

\begin{claim*} 
If $\dot{x}\in \mathrm{Name}_\alpha$ and $\dot{x}^{G_\nu}=x$, then: 
\begin{enumerate-(i)} 
\item 
$x\in X \ \Longleftrightarrow\  \exists p\in G_\nu\ \,p\forces_{\PP_\lambda}^V \dot{x} \in \dot{X}$.
\item 
$x\notin X \ \Longleftrightarrow\ \exists p\in G_\nu\ \,p\forces_{\PP_\lambda}^V \dot{x} \notin \dot{X}$.
\end{enumerate-(i)} 
\end{claim*} 
\begin{proof} 
By the previous claim, it suffices to show the implications from right to left. 
Suppose $p\in G_\nu$ and  $p\forces_{\PP_\lambda}^V \dot{x} \in \dot{X}$. 
Since $\dot{x}$ is a $\PP_\nu$-name and $G_\nu\subseteq G$, we have $\dot{x}^{G_\nu} = \dot{x}^G\in \dot{X}^G$. 
The second implication is similar. 
\end{proof} 
In $V[G_\nu]$, define 
$X_\nu:=\{ {\dot x}^{G_\nu}:\: \exists p\in G_\nu\ \,p\forces_{\PP_\lambda}^V \dot{x} \in \dot{X} \}$. 
Then we have $X\cap V[G_\nu]=X_\nu \in V[G_\nu]$. 
\end{proof}

\begin{remark} 
\label{restricting functions to intermediate models} 
The previous lemma also has a version for functions:
Let $\kappa<\lambda$ be regular infinite cardinals, let $G$ be $\PP_\lambda$-generic over $V$ and let $F: {}^\kappa\kappa\to {}^\kappa\kappa$ be a function in $V[G]$. 
If either
\begin{enumerate-(1)} 
\item 
\label{restricting functions to intermediate models 1} 
$\lambda$ is inaccessible and $F\in \defsets\kappa^{V[G]}$, or
\item 
\label{restricting functions to intermediate models 2} 
$\lambda$ is Mahlo,
\end{enumerate-(1)} 
then there is a stationary (in $V$) subset $S\in V$ of $\lambda$ such that for all $\nu\in S$,  we have $F\restr V[G_\nu] \in V[G_\nu]$. 
Assuming \ref{restricting sets to intermediate models 1}, $S$ can be chosen as a final segment of $\lambda$, and for \ref{restricting sets to intermediate models 2}, $S$ can be chosen so that every $\nu\in S$ is inaccessible. 

To see this, apply the previous lemma
to the set 
$X= \{\langle \alpha,\beta \rangle\conc x : x\in {}^\kappa\kappa,\, (\alpha,\beta)\in F(x) \}$. 
If $X \cap V[G_\nu] \in V[G_\nu]$, then $F\restr V[G_\nu] \in V[G_\nu]$, 
\todog{If $x\in{}^\kappa\kappa\cap V[G_\nu]$, then $F(x)=\{(\alpha,\beta): \langle \alpha,\beta\rangle\conc x\in X\cap V[G_\nu]$}
since $F\restr V[G_\nu]$ can be calculated from $X \cap V[G_\nu]$. 
\end{remark} 

\begin{remark} 
The Mahlo cardinal is necessary in Lemma~\ref{restricting sets to intermediate models}~\ref{restricting sets to intermediate models 2} and Remark~\ref{restricting functions to intermediate models}~\ref{restricting functions to intermediate models 2}. 
To see this, assume the $\GCH$ holds and $\lambda>\kappa$ is any cardinal that is not Mahlo. 
We claim that 
there exists a function $F:{}^\kappa\kappa\to{}^\kappa\kappa$ in $V[G]$ such that 
$F\restr  V[G_\alpha]\notin V[G_\alpha]$ for all $\alpha$ with $\kappa\leq\alpha<\lambda$. 

First suppose $\lambda$ is singular in $V$. 
Then $\lambda$ remains singular in $V[G]$. 
Hence $\lambda\neq (\kappa^+)^{V[G]}$ and $\lambda$ is collapsed in $V[G]$, so there exists some $x\in ({}^\kappa\kappa)^{V[G]}$ with $x\notin V[G_\alpha]$ for all $\alpha<\lambda$.%
\todog{Then any $F$ mapping $\langle 0\rangle^\kappa$ to $x$ works}

Now suppose $\lambda$ is regular in $V$. 
Let $C$ be a club in $\lambda$ whose elements are not inaccessible. 
Additionally, if $\lambda$ is a successor cardinal then assume $C$ does not contain cardinals, while if $\lambda$ is a limit cardinal then assume $C$  contains only cardinals. 
For any $\alpha<\lambda$, let $\alpha'$ denote the least ordinal 
above $\alpha$ in $C$.
Take any function $F: {}^\kappa\kappa\to {}^\kappa\kappa$ in $V[G]$ with the following property: if $\alpha<\lambda$ is least with $x\in V[G_\alpha]$, then $F(x)\notin V[G_{\alpha'}]$. 

We show that 
$F\restr  V[G_\alpha]\notin V[G_\alpha]$ for all $\alpha$ with $\kappa\leq\alpha<\lambda$. 
To see this, fix any such $\alpha$. 
We distinguish two cases: 

\begin{case}
\label{case Mahlo not limit} 
$\alpha$ is not a limit point of $C$. 
\end{case} 
 
Then there is some $\gamma\in C$ with $\gamma< \alpha\leq \gamma'$. 
Take any $x\in V[G_\alpha]\oldsetminus V[G_\gamma]$. 
Let $\beta>\gamma$ be least with $x\in V[G_\beta]$. 
Then $F(x)\notin V[G_{\beta'}]$ by the definition of $F$. 
Since 
\todog{Since $\gamma<\beta\leq\alpha\leq\gamma'$, we have $\gamma'=\beta'=\alpha'$. But we only need $\alpha\leq\beta'$.}
$\alpha\leq\gamma'\leq\beta'$, we have $F(x)\notin V[G_\alpha]$ as required. 

\begin{case}
$\alpha$ is a limit point of $C$. 
\end{case} 

We claim that there is some $x\in {}^\kappa\kappa$ such that $\alpha$ is least with $x\in V[G_\alpha]$. 
First suppose that $\lambda$ is successor cardinal in $V$.\footnote{For successor cardinals $\lambda$, the proof does not use that $\alpha$ is a limit point of $C$, so Case \ref{case Mahlo not limit} is not needed.} 
Let $\mu:=|\alpha|^V$, so $\mu < \alpha <(\mu^+)^V$ 
since $C$ does not contain cardinals. 
Since $|\alpha|=\kappa$ in $V[G_{\mu+1}]$, $G_{(\mu+1,\alpha)}$ can be coded as a subset $x$ of $\kappa$ via a bijection $f: \kappa\to \alpha$ with $f\in V[G_{\mu+1}]$. 
Then $\alpha$ is least with $x\in V[G_\alpha]$. 
Now suppose $\lambda$ is a limit cardinal in $V$. 
Since $C$ contains only cardinals and the $\GCH$ holds, $\alpha$ is singular in $V$.
Thus $\alpha\neq (\kappa^+)^{V[G_\alpha]}$ and $\alpha$ is collapsed in $V[G_\alpha]$, so the claim follows. 

Then $F(x)\notin V[G_{\alpha'}]$ by the definition of $F$. Thus, $F(x)\notin V[G_\alpha]$ as required. 
 \end{remark}

\begin{remark} 
The proof of Lemma~\ref{restricting sets to intermediate models}~\ref{restricting sets to intermediate models 2} 
can be adapted to show that if $\lambda>\kappa$ is inaccessible, then
the following strengthening of $\Diamond_{\kappa^+}$
holds in
$\PP_\lambda$-generic extensions $V[G]$:
\begin{quotation}
\noindent
$\Diamond^+_{\kappa^+}$: 
\index{diamond!plus\idf$\Diamond^+_{\kappa^+}$}
there exists a sequence 
$\langle \mathcal A_\alpha:\alpha<\kappa^+\rangle$ of sets 
$\mathcal A_\alpha\subseteq\pwrset(\alpha)$ such that
\begin{enumerate-(a)}
\item\label{diamond+ a}
$|\mathcal A_\alpha|\leq \kappa$ for all $\alpha<\kappa^+$, and 
\item\label{diamond+ b}
for all $X\subseteq\kappa^+$, there exists a club $C$ in $\kappa^+$ such that for all $\alpha\in C$, 
$X\cap\alpha \in \mathcal A_\alpha$ and $C\cap \alpha \in \mathcal A_\alpha$. 
\end{enumerate-(a)}
\end{quotation}

We show that in $V[G]$,
$\langle\mathcal A_\alpha=\pwrset(\alpha)^{V[G_\alpha]}:\alpha<\kappa\rangle$ 
is a $\Diamond^+_{\kappa^+}$-sequence.
It is clear that~\ref{diamond+ a} holds in $V[G]$.
To show~\ref{diamond+ b},
we use an argument similar to the proof of Lemma~\ref{restricting sets to intermediate models}~\ref{restricting sets to intermediate models 2}.%
\footnote{Note that this version of the argument works assuming only that $\lambda$ is inaccessible.}
Let $X$ be a subset of $\kappa^+=\lambda$ in $V[G]$, and let $\dot{X}\in V$ be a $\PP_\lambda$-name with $\dot{X}^G=X$. 
In $V$, first define a function $f: \lambda\to \lambda$ as follows. 
For each $\alpha<\lambda$,
let $A$ be a maximal antichain in $\PP_\lambda$ consisting of conditions
deciding whether $\alpha\in\dot X$.
Since $\PP_\lambda$ has the $\lambda$-c.c., choose $f(\alpha)<\lambda$ with $A\subseteq \PP_{f(\alpha)}$.

Let $C$ be a club of limit ordinals $\nu<\lambda$ that are closure points of $f$. 
We shall show that $C$ satisfies~\ref{diamond+ b}.
\todog{$\nu$ will not need to be inaccessible. So $\lambda$ does not need to be Mahlo} 
Since $C\in V$, it suffices to show that $X 
\cap \nu \in \mathcal{A}_\nu=\mathcal{P}(\nu)^{V[G_\nu]}$ for any given $\nu\in C$. 

\begin{claim*} 
For any $\alpha<\nu$, exactly one of the following holds: 
\begin{enumerate-(i)} 
\item 
$\exists p\in G_\nu\ \,p\forces_{\PP_\lambda}^V \alpha \in \dot{X}$.
\item 
$\exists p\in G_\nu\ \,p\forces_{\PP_\lambda}^V \alpha \notin \dot{X}$.
\end{enumerate-(i)} 
\end{claim*} 
\begin{proof} 
By the definition of $f$, there is
$A\subseteq \PP_{f(\gamma)}$ such that $A$ is
a maximal antichain in $\PP_\lambda$ of conditions deciding whether $\alpha\in \dot{X}$. 
Since $\nu$ is a closure point of $f$, we have $f(\gamma)<\nu$. 
Thus $A$ is a maximal antichain in $\PP_\nu$. 
Let $p\in G_\nu \cap A$. 
\end{proof} 

As in the proof of Lemma~\ref{restricting sets to intermediate models}~\ref{restricting sets to intermediate models 2}, 
it now follows that 
for all $\alpha<\nu$,
$$\alpha\in X\cap\nu \  \Longleftrightarrow \  
\exists p\in G_\nu\ \,p\forces_{\PP_\lambda}^V \alpha \in \dot{X}.$$
Thus,
$X\cap\nu$ is an element of 
$\mathcal{A}_\nu=\mathcal{P}(\nu)^{V[G_\nu]}$.
\end{remark}

\subsubsection{Names and witness functions}
\label{subsection: finding names}

\begin{assumptions}
\label{proof notation}
The following assumptions will be used throughout the rest of this section, unless explicitly stated otherwise.
In $V$:
\begin{itemize}
\item
$\kappa$ is a regular infinite cardinal, 
$2\leq\ddim\leq\kappa$ 
and $\lambda>\kappa$ is an inaccessible cardinal.
$G$ is a $\PP_\lambda=\Col(\kappa,\lle\lambda)$-generic filter over $V$.
\item
$\kappa^{<\kappa}=\kappa$ holds in $V$. 
This may be assumed without loss of generality, 
because $\PP_2$ forces $\kappa{}^{<\kappa}=\kappa$ and $\PP^2$ is equivalent to $\PP_\lambda$. 
\end{itemize}
In $V[G]$: 
\label{proof notation 2}
\begin{itemize}
\item
$X$ is a subset of ${}^\kappa\kappa$ in $\defsetsk$. 
More specifically  
\[X=\Xphia=
\{x\in{}^{\kappa}\kappa:\varphi(x,a)\},\]
where 
$\varphi(x,a)$ 
is a first order formula 
\todog{$\kappa$ and $\ddim$ are clear from the context, so they do not need to be specified here}
with a parameter $a\in{}^\kappa\Ord$.
\item
$H$ is a  box-open $\ddim$-dihypergraph on ${}^\kappa\kappa$ 
such that $H\restr X$ does not admit a $\kappa$-coloring. 
\item 
$\lambda$ is a Mahlo cardinal in $V$, or $H\in\defsetsk$ in $V[G]$.%
\footnote{This assumption can be weakened as in Remark \ref{remark weaker assumption for main proof} below.}%
$^{,}$%
\footnote{If $\ddim<\kappa$, then $H\in\defsetsk$ follows from the assumption that $H$ is box-open by Lemma~\ref{<kappa dim hypergraphs are definable}.}
\label{proof notation end}
\end{itemize}

To prove Theorem~\ref{main theorem}, it suffices to prove
that $\ODD\kappa{H\restr X}$ holds under these assumptions.
Since we assume that $H\restr X$ does not admit a $\kappa$-coloring, 
it further suffices to show that there is 
a continuous homomorphism from $\dhH\ddim$ to $H\restr X$.
\end{assumptions} 

We work in $V[G]$ unless stated otherwise. Let $T$ be a subtree of ${}^{<\kappa}\kappa$.
Recall from Definition~\ref{def: independent tree} 
that $T$ is an $H$-independent tree if
for all sequences $\langle t_\alpha\in T:\alpha<\ddim\rangle$, 
we have
$\prod_{\alpha<\ddim}N_{t_\alpha}\not\subseteq H$.
Since $H$ is box-open on ${}^\kappa\kappa$, 
$[T]$ is an $H$-independent set
whenever $T$ is an $H$-independent tree 
and the converse holds when $T$ is pruned,
by Lemma~\ref{lemma: independent trees}. 

\begin{definition}
\label{def: trees in intermediate models}
\label{def: sets of H-independent trees}
For all $\alpha\leq\lambda$, let 
\index{tree!subtrees of ${}^{<\kappa}\kappa$ in $V[G_\alpha]$!all\idf $\mathcal T_\alpha$}%
\index{tree!subtrees of ${}^{<\kappa}\kappa$ in $V[G_\alpha]$!independent\idf $\indtrees\alpha$}%
\index{independent tree!independent trees in $V[G_\alpha]$\idf $\indtrees\alpha$}%
\todog{use the command ``indtrees'' for the notation $\indtrees\alpha$} 
\begin{align*}
\mathcal T_\alpha &:=\{T\in V[G_\alpha] : T
\text{ is a subtree of ${}^{<\kappa}\kappa$}\},
\\
\indtrees\alpha &:=\{T\in V[G_\alpha] : T
\text{ is an $H$-independent subtree of ${}^{<\kappa}\kappa$}\}.
\end{align*}
\end{definition}

Let $\alpha<\lambda$. 
Then $\mathcal T_\alpha$ has size $\kappa$ in $V[G]$ 
since $\lambda$ is inaccessible, 
and $\mathcal T_\alpha\in V[G_\alpha]$.
However, $\indtrees\alpha$ may not be in $V[G_\alpha]$. 

\begin{lemma} 
\label{intermediate model which see H-independence for their trees}
There is a stationary (in $V$) subset $S\in V$ of $\lambda$ such that 
$\indtrees\nu\in V[G_\nu]$ for all $\nu\in S$.
\end{lemma}
\begin{proof} 
In $V[G]$, 
code $\indtrees\lambda\subseteq\pwrset({}^{<\kappa}\kappa)$ as a subset $Y_\lambda$ of ${}^\kappa\kappa$, using a bijection $f: \kappa\to {}^{<\kappa}\kappa$ in $V$. 
If $H\in\defsetsk$, then $\indtrees\lambda\in\defsetsk$ and hence $Y_\lambda\in\defsetsk$. 
Otherwise, $\lambda$ is Mahlo.
In both cases, Lemma~\ref{restricting sets to intermediate models} yields stationarily many cardinals $\nu<\lambda$  such that $Y_\lambda\cap V[G_\nu]\in V[G_\nu]$, and hence $\indtrees\nu=\indtrees\lambda\cap V[G_\nu]\in V[G_\nu]$. 
\end{proof}

For the rest of this section, we fix a cardinal $\nu$ 
such that $\indtrees\nu\in V[G_\nu]$ and $a\in V[G_\nu]$, 
where $a\in({}^\kappa\Ord)^{V[G]}$ is the parameter with $X=\Xphia^{V[G]}$.
This may be done by the previous lemma and the $\lambda$-c.c.~for $\PP_\lambda$.

\begin{remark} 
\label{remark weaker assumption for main proof} 
Note that the assumption that either $H\in \defsetsk^{V[G]}$ or $\lambda$ is Mahlo in $V$ in Assumption~\ref{proof notation} 
was needed only to obtain the previous lemma.
From now on,
we will use only the fact that
$\indtrees\nu\in V[G_\nu]$, $a\in V[G_\nu]$ and the rest of Assumption~\ref{proof notation}. 
\end{remark} 

The next lemma shows that there is a $\PP_\lambda$-name $\na$ in $V[G_\nu]$ for an element 
of $X$ that 
avoids all sets $[T]$ with $T\in\indtrees\nu$.

\begin{lemma}
\label{lemma 3a} 
In $V[G_\nu]$, there exists a $\PP_\lambda$-name $\na$ 
with the following properties:
\begin{enumerate-(1)}
\item \label{na eq 1a}
$\one_{\PP_\lambda} \forces^{V[G_\nu]}{\na}\in\Xphia\oldsetminus V[G_\nu].$%
\footnote{Note that
$\one_{\PP_\lambda} \forces^{V[G_\nu]}{\na}\notin V[G_\nu]$ follows from~\ref{na eq 2} as in the proof of the lemma.}
\item \label{na eq 2}
$\one_{\PP_\lambda}\forces^{V[G_\nu]} {\na\notin[T]}$ for all ${T\in\indtrees\nu}$. 
\end{enumerate-(1)}
\end{lemma}
\begin{proof} 
Work in $V[G]$. 
$[T]$ is an $H$-independent set for all $T\in\indtrees\nu$
by Lemma~\ref{lemma: independent trees}~\ref{ind tree ind set}.
Since 
$H\restr X=H\restr\Xphia$ does not admit a $\kappa$-coloring and
$\indtrees\nu$ has size at most $\kappa$,
we have
\begin{equation}
\label{3a eq}
\Xphia-\bigcup\big\{[T]: T\in\indtrees\nu\big\}
\neq\emptyset.
\end{equation}
Note that~\eqref{3a eq}
is a first-order statement with parameters in $V[G_\nu]$, 
because $\nu$ was chosen so that
$\indtrees\nu$ and $a$ are in $V[G_\nu]$.
\todog{%
This will be the only place 
in the proof of Theorem~\ref{main theorem}
where we use
Lemma~\ref{intermediate model which see H-independence for their trees}.}
Since $\PP^\nu$ is homogeneous, 
$\one_{\PP^\nu}$ forces in 
$V[G_\nu]$ that \eqref{3a eq} holds.
Since the forcings $\PP^\nu$ and $\PP_\lambda$ are equivalent
by Corollary~\ref{forcing equivalent to the Levy collapse},
$\one_{\PP_\lambda}$ forces in 
$V[G_\nu]$ that \eqref{3a eq} holds.
By the maximal principle, there exists a $\PP_\lambda$-name $\na\in V[G_\nu]$ 
for an element of $\Xphia$ which satisfies~\ref{na eq 2}. 
To see that
$\one_{\PP_\lambda}\forces^{V[G_\nu]} {\na\notin V[G_\nu]}$, 
let $x\in({}^\kappa\kappa)^{V[G_\nu]}$.
Since $\{x\}$ is $H$-independent, 
$T(\{x\})$ 
is in $\indtrees\nu$ by Lemma~\ref{lemma: independent trees}~\ref{ind set ind tree}. 
Thus $\one_{\PP_\lambda}\forces^{V[G_\nu]} \sigma\neq x$ by \ref{na eq 2}. 
\end{proof}

Recall from Definition~\ref{T^sigma def} that 
if $\QQ$ is a forcing, $q\in\QQ$ and $\nna$ is a $\QQ$-name for an element of ${}^\kappa\kappa$,
then 
$$T^{\nna, q}=T^{\nna, q}_\QQ= \{t\in{}^{<\kappa}\kappa\::\: \exists r\leq q\ \ r\forces t\subseteq \nna\}$$
is the \emph{tree of possible values for $\nna$ below $q$}.
Recall also that the definition of $T^{\nna,q}$ is absolute between transitive models of $\ZFC$ with the same ${}^{<\kappa}\kappa$ 
by Lemma~\ref{T^sigma facts}~\ref{T^sigma abs}.
In particular, 
if $\nna\in V[G_\nu]$, then
$(T^{\nna,q})^{V[G_\nu]}
=(T^{\nna,q})^{V[G]}$.

The next lemma shows that the trees
$T^{\na,p}_{\PP_\lambda}$ are not $H$-independent for names $\na$ as in the previous lemma.

\begin{lemma}
\label{lemma 3b}
$T^{\na,p}_{\PP_\lambda}\in\mathcal T_\nu\oldsetminus\indtrees\nu$
for all names $\na$ as in Lemma~\ref{lemma 3a} and
all $p\in\PP_\lambda$.
\end{lemma}
\begin{proof} 
Let $p\in \PP_\lambda$, and let 
$T:=T^{\na,p}_{\PP_\lambda}$.
Since $\na\in V[G_\nu]$, we have
$T\in \mathcal T_\nu$ by Lemma~\ref{T^sigma facts}~\ref{T^sigma abs}. 
Moreover, $p\forces^{V[G_\nu]}_{\PP_\lambda}\na\in[T]$ by Lemma~\ref{T^sigma facts}~\ref{sigma is a branch of T^sigma}. 
Therefore
$T\notin \indtrees\nu$
since $\na$ was chosen so that Lemma~\ref{lemma 3a}~\ref{na eq 2} holds.
\end{proof}

We now prove an analogue of the two previous lemmas for $\Add(\kappa,1)$.
Recall that 
\[\Add(\kappa,1):=\{p:\alpha\to\kappa\,\mid\, \alpha<\kappa\}.%
\footnote{See Definition~\ref{definition of Add(kappa,1)}.}
\]

\begin{lemma}
\label{lemma 3c m}
For unboundedly many $\gamma<\lambda$, there is an $\Add(\kappa,1)$-name $\nna$ in $V[G_\gamma]$ 
for a new element of ${}^\kappa\kappa$
with the following properties: 
\begin{enumerate-(1)}
\item
\label{nna eq 1}
In $V[G_\gamma]$,
$\one_{\Add(\kappa,1)}$ forces that 
$\one_{\PP_\lambda}\forces\check{\nna}\in\Xphia.$
\item
\label{nna eq 3}
$T^{\,\nna,\,q}_{\Add(\kappa,1)}\in\mathcal T_\gamma\oldsetminus\indtrees\gamma$
for all $q\in\Add(\kappa,1)$.
\end{enumerate-(1)}
\end{lemma}

\begin{proof} 
Fix a $\PP_\lambda$-name $\na$ as in Lemma~\ref{lemma 3a}.
We first sketch the idea. 
First choose some $\gamma$ so that 
$\na\in V[G_\gamma]$ is a $\PP_\gamma$-name
and the forcings $\PP_\gamma$ and
$\Add(\kappa,1)$ are equivalent in $V[G_\gamma]$.
Then transform  
$\na$
to an $\Add(\kappa,1)$-name $\nna\in V[G_\gamma]$,
and finally show that 
Lemma~\ref{lemma 3a}~\ref{na eq 1a}
and Lemma~\ref{lemma 3b} yield the required properties of $\nna$.

Here is the detailed proof. 
First, note that for all sufficiently large $\gamma\in[\nu,\lambda)$, we have that
$\na\in V[G_\gamma]$ is a $\PP_\gamma$-name.
This holds
because $\PP_\lambda$ has the $\lambda$-c.c.~and $\na\in V[G_\nu]$ is a $\PP_\lambda$-name for an element of ${}^\kappa\kappa$.
Let $\gamma$ be sufficiently large with $\gamma=\xi+1$ for some ordinal 
$\xi<\lambda$ with $\xi^{<\kappa}=\xi$. 
Then $\PP_\gamma$ 
has size $\kappa$ in $V[G_\gamma]$.
Since $\PP_\gamma$ is also non-atomic and $\lle\kappa$-closed, 
it is equivalent in $V[G_\gamma]$ to $\Add(\kappa,1)$, 
and in fact there exists a dense embedding $\iota:\Addd(\kappa,1)\to\PP_\gamma$ in $V[G_\gamma]$ by Lemma~\ref{forcing equivalent to Add(kappa,1)}. 
Let 
$\nna:=\na^{\iota}$.%
\footnote{See Definition \ref{def: pulling back names}.} 
This is an $\Addd(\kappa,1)$-name and thus an $\Add(\kappa,1)$-name. 

We first show that in $V[G_\gamma]$, 
$\nna$ is a name for a new element of ${}^\kappa\kappa$ and
\ref{nna eq 1} holds.
Let $J$ be $\Add(\kappa,1)$-generic over $V[G_\gamma]$, 
and let $x:=\nna^J$. It suffices to prove the following claim.
\begin{claim*} 
$x\notin V[G_\gamma]$ and
$\one_{\PP_\lambda}\forces^{V[G_\gamma][J]}\check x\in\Xphia.$
\end{claim*} 
\begin{proof} 
Let $K$ be the upwards closure of $\iota[J\cap \Addd(\kappa,1)]$ in $\PP_\gamma$.
Then $K$ is a $\PP_\gamma$-generic filter over $V[G_\gamma]$ such that 
$V[G_\gamma][J]=V[G_\gamma][K]$
and
$\na^K=x$.  
We then have $x\notin V[G_\nu]$ 
by 
Lemma~\ref{lemma 3a}~\ref{na eq 1a}.
Since $K$ and $G_{[\nu,\gamma)}$ are mutually generic over $V[G_\nu]$,
we have $V[K]\cap V[G_{\gamma}]=V[G_\nu]$.
Therefore $x\notin V[G_\gamma]$.
We also have 
$\one_{\PP_\lambda} \forces^{V[G_\nu]}{\na}\in\Xphia$
by Lemma~\ref{lemma 3a}~\ref{na eq 1a}.
Since $\PP_\lambda\simeq \PP_\gamma\times\PP^\gamma$, 
$$\one_{\PP^\gamma} \forces^{V[G_\nu][K]} \check{x}\in\Xphia.$$ 
Since
$\PP^\gamma\simeq\PP_{[\nu,\gamma)}\times \PP_\lambda$%
\footnote{See Corollary~\ref{forcing equivalent to the Levy collapse}.} 
and $V[G_\gamma][K]$ is a $\PP_{[\nu,\gamma)}$-generic extension of $V[G_\nu][ K]$,
$$\one_{\PP_\lambda} \forces^{V[G_\gamma][K]} \check{x}\in\Xphia.$$ 
The claim holds since 
$V[G_\gamma][J]=V[G_\gamma][K]$. 
\end{proof} 

To see~\ref{nna eq 3}, let $q\in\Add(\kappa,1)$ and 
$T:=T^{\,\nna,\,q}_{\Add(\kappa,1)}$.
Since $\nna\in V[G_\gamma]$, we have
$T\in\mathcal T_\gamma$  
 by Lemma~\ref{T^sigma facts}~\ref{T^sigma abs}.

\begin{claim*} 
$T$ is not $H$-independent.\footnote{I.e., $T\notin\indtrees\gamma$.} 
\end{claim*} 
\begin{proof} 
Choose some $q^*\in \Addd(\kappa,1)$ with $q^*\leq q$,
and let
$T^*=T^{\,\na,\iota(q^*)}_{\PP_\lambda}$.
By Lemma~\ref{lemma 3b},
$T^*\in V[G_\nu]\subseteq V[G_\gamma]$ and 
$T^*$ is not $H$-independent.
It thus suffices to show that $T^*\subseteq T$.
To prove this, work in $V[G_\gamma]$ and 
let $v\in T^*$.
Choose $p\in\PP_\lambda$ with $p\leq \iota(q^*)$ and
$p\forces_{\PP_\lambda} v\subseteq\na$.
\todog{Here and in (ii) below, it's better to keep track of the forcing in $\forces$}
Let $p\restr\gamma:=p\restr((\gamma\times\kappa)\cap\dom(p))$. 

\todog{$p\restr\gamma$ is the retraction of $p$ to $\PP_\gamma$,  
i.e., $p\restr\gamma=\pi_{\id}(p)$ for the identity map $\id:\PP_\gamma\to\PP_\lambda$ and
the map 
$\pi_{\id}$ from Definition~\ref{def: retraction}.}
Then $p\restr\gamma\in\PP_\gamma$ and
$p\restr\gamma\leq\iota(q^*)$
since $\iota(q^*)\in\PP_\gamma$.
As the formula ``$x\subseteq y$'' is absolute,
$p\restr\gamma\forces_{\PP_\gamma} v\subseteq\na$.
By the density of $\ran(\iota)$ in $\PP_\gamma$,
there exists $r\in\Addd(\kappa,1)$ 
such that 
$\iota(r)\leq p\restr\gamma$. Then
\begin{enumerate-(i)}
\item\label{3c proof 3}
$\iota(r)\leq\iota(q^*)$ and
\item\label{3c proof 4}
$\iota(r)\forces_{\PP_\gamma} v\subseteq\na$.
\end{enumerate-(i)}
Since $\iota$ is a dense embedding, 
\ref{3c proof 3} implies $r\leq q^*$ as $\Addd(\kappa,1)$ is separative. 
Because $\nna=\na^{\iota}$,
\ref{3c proof 4} implies $r\forces v\subseteq \nna$.
Thus, $r$ witnesses that
$v\in T$
as required. 
\end{proof} 
This completes the proof of Lemma~\ref{lemma 3c m}.
\end{proof}

We now define the notion of an $H$-witness function for $\nna$ and $\QQ$.
The existence of such a function
 is an explicit 
and equivalent way of stating that
the trees $T_\QQ^{\nna,q}$ are not $H$-independent for all $q\in\QQ$.

\begin{definition} 
\label{definition H-witness function} 
Suppose $\QQ$ is a forcing and $\nna$ is a $\QQ$-name for an
element of ${}^\kappa\kappa$.
\index{function!witness function@$H$-witness function\idf} 
An \emph{$H$-witness function} for $\nna$ and $\QQ$
is a function
$w:\QQ\times\ddim\to{}^{<\kappa}\kappa$
such that for all $p\in\QQ$ and $\alpha<\ddim$
\begin{enumerate-(1)}
\item
$w(p,\alpha)\in T_\QQ^{\nna,p}$ and
\item
$\prod_{\beta<\ddim}N_{w(p,\beta)}\subseteq H$.
\end{enumerate-(1)}
\end{definition}

Such functions exist for the names
$\nna$ from
Lemma~\ref{lemma 3c m} and $\QQ=\Add(\kappa,1)$.%
\footnote{Note that $H$-witness functions need not exist.
For instance, if $\nna=\check x$ for some $x$ in the ground model, then 
$T_\QQ^{\nna,q}=T(\{x\})$ 
for any $q\in\QQ$. 
Therefore $T_\QQ^{\nna,q}$ is
$H$-independent and
there is no $H$-witness function for~$\check x$.}

\begin{definition}
\label{def: step function}
\index{function!step function\idf} 
A \emph{step function} for a given forcing $\QQ$ is a function $s:\QQ\times\ddim\to \QQ$
such that
$s(q,\alpha)\leq q$
for all $q\in\QQ$ and $\alpha<\ddim$.
\end{definition}

We now show that if there is an $H$-witness function, then there is an $H$-witness function with some nice properties and a corresponding step function.

\begin{lemma}
\label{lemma 3b reformulation}
Suppose $\QQ$ is a forcing and $\nna$ is a $\QQ$-name for an element of ${}^\kappa\kappa$.
\begin{enumerate-(1)}
\item\label{nice witness}
If there exists an $H$-witness function for $\nna$ and $\QQ$, then there exists an $H$-witness 
function $w$ such that
\[\nna_{[p]}\subsetneq w(p,\alpha)\,
\footnote{Recall the notation $\nna_{[p]}$ from Definition~\ref{sigma[q] def}. Note that this requirement is equivalent to: if $p\forces v\subseteq\nna$, then $v\subsetneq w(p,\alpha)$.} 
\]
for all $p\in\QQ$ and $\alpha<\ddim$.
\item\label{witness step}
If $w$ is an $H$-witness function for $\nna$ and $\QQ$, then 
there exists a step function $s$ for $\QQ$ such that
$$s(p,\alpha)\forces w(p,\alpha)\subseteq\nna$$
for all $p\in\QQ$ and $\alpha<\ddim$.
\end{enumerate-(1)}
\end{lemma}

\begin{proof}
To see~\ref{nice witness}, suppose $u:\QQ\times\ddim\to{}^{<\kappa}\kappa$
is an $H$-witness function for $\nna$ and $\QQ$.
For any $p\in \QQ$, we have $\nna_{[p]}\cup u(p,\alpha)\in T_\QQ^{\nna,p}$ for all $\alpha<\ddim$, since $\nna_{[p]}$ equals the stem of $T_\QQ^{\nna,p}$ 
by Lemma~\ref{T^sigma facts}~\ref{T^sigma nodes comp}.
It is easy to see that $T_\QQ^{\nna,p}$ has no terminal nodes. 
Hence there is a sequence
$\langle w_p^\alpha:\alpha<\ddim\rangle$ in $ T_\QQ^{\nna,p}$ 
such that $\nna_{[p]}\cup u(p,\alpha)\subsetneq w_p^\alpha$ for all $\alpha<\ddim$.
Define
$w:\QQ\times\ddim\to{}^{<\kappa}\kappa$
by letting
$w(p,\alpha):=w_p^\alpha$ for all $p\in\QQ$ and $\alpha<\ddim$. It is clear that $w$ is as required.

To see~\ref{witness step}, let $w:\QQ\times\ddim\to{}^{<\kappa}\kappa$ be an $H$-witness function for $\nna$ and $\QQ$.
By the definition of $T_\QQ^{\nna,q}$, 
there exists $s_q^\alpha$
such that $s_q^\alpha\leq p$ and  
$s_q^\alpha\forces w(p,\alpha)\subseteq\nna$
for all $q\in\QQ$ and $\alpha<\ddim$.
Define $s:\QQ\times\ddim\to\QQ$ by letting $s(q,\alpha):=s^\alpha_q$ 
for all $q\in\QQ$ and $\alpha<\ddim$.
Then $s$ is as required.
\end{proof}


\subsubsection{Construction of continuous homomorphisms}
\label{subsection: the main step}
In our proof of Theorem~\ref{main theorem}, 
we will define a continuous homomorphism $f$ from $\dhH\ddim$ to $H\restr X$. 
The values of the homomorphism come from the evaluations of the $\Add(\kappa,1)$-name $\nna$ from the previous subsection 
by a $\kappa$-perfect set of generic filters. 
The difficulty is to ensure that $\ran(f)\subseteq X$. 
The next lemma, the \emph{Nice Quotient Lemma for $\Col(\kappa,\lle\lambda)$}, does this by controlling quotient forcings:  
It will guarantee the existence 
in $V[G]$ 
of a continuous function 
whose range consists of elements with well-behaved quotient forcings.
This function will induce the required homomorphism $f$ via~$\nna$.

We will construct 
this function 
along an appropriately chosen step function.
We will need the next definition. 

\begin{definition}\ 
\label{def: s-function}
Suppose $s$ is a step function for a given forcing $\QQ$.
\index{function!s-function@$s$-function\idf} 
A strict order reversing map $h:{}^{<\kappa}\ddim\to\QQ$ is called an \emph{$s$-function} 
if for all $t\in{}^{<\kappa}\ddim$ and all $\alpha<\ddim$: 
\begin{equation}
\label{eq: s-function}
h\big(t\conc\langle\alpha\rangle\big)<
s\big(h(t), \alpha\big). 
\end{equation}
\end{definition}

Note that it is necessary to require that $h$ is strict order reversing in the previous definition
since this does not follow from~\eqref{eq: s-function}.
\todog{\eqref{eq: s-function} does not guarantee that $h$ is order reversing at limits}
However, \eqref{eq: s-function} does imply that 
$h(u)<h(t)$ when $u$ is a direct successor of $t$. 
Moreover, 
if $\QQ=\Add(\kappa,1)$
and $h$ is an $s$-function, then 
$[h]:{}^\kappa\ddim\to{}^\kappa\kappa$ is continuous and
$$[h](x)=\bigcup_{\alpha<\kappa}s\big( h(x\restr\alpha), x(\alpha)\big)$$ 
holds for all $x\in{}^\kappa\ddim$.

\begin{lemma}{\bf (Nice Quotient Lemma for $\Col(\kappa,\lle\lambda)$)}
\label{lemma: quotient lemma}
\index{Nice Quotient Lemma for!Col@$\Col(\kappa,\lle\lambda)$\idf}
The following holds in $V[G]$. 
Suppose that $\gamma<\lambda$ and
$s$ is a step-function for $\Add(\kappa,1)$.
Then there exists an $s$-function $h$
such that for all $x\in \ran([h])$: 
\begin{enumerate-(1)}
\item\label{nsf 1}
$x$ is $\Add(\kappa,1)$-generic over $V[G_\gamma]$.%
\footnote{Here, we identify $x$ with
$\kkppred{x}=\{t\in\Add(\kappa,1): t\subsetneq x\}$.
}
\item\label{nsf 2}
$V[G]$ is a 
$\PP_\lambda$-generic
extension of $V[G_\gamma][x]$.
\end{enumerate-(1)}
\end{lemma}

We shall prove two cases of the Nice Quotient Lemma separately: the countable case in 
Subsection~\ref{section: proof in the countable case} and the uncountable case in 
Subsections~\ref{subsection: main theorem for uncountable kappa} and~\ref{subsection: the forcing Q}.

In the rest of this subsection, we aim to construct a continuous homomorphism 
$f:{}^{\kappa}\ddim\to{}^\kappa\kappa$
from $\dhH\ddim$ to $H\restr X$.
To this end,
fix an ordinal $\gamma$ and an $\Add(\kappa,1)$-name $\nna$ in $V[G_\gamma]$
as in Lemma~\ref{lemma 3c m}.
Then there exists an $H$-witness function for $\Add(\kappa,1)$ and $\nna$.
Thus, we may also fix
an $H$-witness function $w$ for $\nna$ and step function $s$ for $\nna$ and $w$ as in Lemma~\ref{lemma 3b reformulation}.
Lastly, let
$h$ be an $s$-function as in the previous lemma.

We argue in $V[G]$.
We define a function $f:{}^\kappa\ddim\to{}^\kappa\kappa$ by letting
\[
f(z):= \nna^{[h](z)}
\]
for all $z\in{}^{\kappa}\ddim$. 

\begin{lemma}
\label{ran(f) subset of X}
$\ran(f)\subseteq X$.
\end{lemma}
\begin{proof}
Let $z\in({}^{\kappa}\ddim)^{V[G]}$, and let $x:=[h](z)$. 
Then $f(z)=\nna^x$.
Since $\nna$ was chosen so that 
Lemma~\ref{lemma 3c m}~\ref{nna eq 1} 
holds in $V[G_\gamma]$
and $x$ is $\Add(\kappa,1)$-generic 
over $V[G_\gamma]$ 
by Lemma~\ref{lemma: quotient lemma}~\ref{nsf 1},
\[
\qquad\qquad
\one_{\PP_\lambda}\forces f(z)\in\Xphia.\]
holds in $V[G_\gamma][x]$. 
Therefore
$f(z)\in{(\Xphia)}^{V[G]}=X$, 
by Lemma~\ref{lemma: quotient lemma}~\ref{nsf 2}.
\end{proof}

It now suffices to prove the following.

\begin{lemma}
\label{lemma: f cont hom}
$f$ is a continuous homomorphism from $\dhH\ddim$ to $H$.
\end{lemma}
\begin{proof}
We will define 
a map
$\iota:{}^{<\kappa}\ddim\to{}^{<\kappa}\kappa$ 
and show that it is an order homomorphism for 
$({}^\kappa\kappa,H)$ 
with $f=[\iota]$.
This is sufficient by Lemma~\ref{homomorphisms and order preserving maps}~\ref{hop 2}. 

Let $\iota(\emptyset):=\emptyset.$
If $\lh(v)\in\Succ$ and $v=u\conc\langle\alpha\rangle$, then let
$\iota(v):=w\big(h(u), \alpha\big)$.
If $\lh(v)\in\Lim$, 
let $\iota(v):=\bigcup_{u\subsetneq v}\iota(u)$.

\begin{claim*}
For all $v\in {}^{<\kappa}\ddim$, $h(v)\forces_{\Add(\kappa,1)}^{V[G_\gamma]} \iota(v)\subseteq \tau$. 
\end{claim*}
\begin{proof} 
This is proved by induction on $\lh(v)$. 
It is clear for $v=\emptyset$. 
For the successor case, note that  
$$s(h(u),\alpha) \forces_{\Add(\kappa,1)} w(h(u),\alpha)\subseteq \tau$$ 
holds in $V[G]$
by Lemma \ref{lemma 3b reformulation} \ref{witness step}. 
Therefore this also holds in $V[G_\gamma]$ since the formula 
``$p\forces \theta\subseteq \eta$'' is absolute between transitive models of $\ZFC$.%
\footnote{
It
can be defined by a recursion which 
uses only absolute concepts; see \cite{KunenBook2013}*{Theorem II.4.15}. 
}
Since $h$ is an $s$-function, we have $s(h(u),\alpha)\subseteq h(u\conc\langle \alpha\rangle)$. 
Moreover, $\iota(u\conc\langle\alpha\rangle)=w(h(u),\alpha)$ by the definition of $\iota$. 
Hence 
$$h(u\conc\langle\alpha\rangle) \forces_{\Add(\kappa,1)}^{V[G_\gamma]} \iota(u\conc\langle\alpha\rangle)\subseteq \tau.$$ 
If $\lh(v)\in\Lim$, then $h(v)\supseteq \bigcup_{u\subsetneq v} h(u)$ and $\iota(v)= \bigcup_{u\subsetneq v} \iota(u)$. 
Thus $h(v)\forces_{\Add(\kappa,1)}^{V[G_\gamma]} \iota(v)\subseteq\tau$ by the inductive hypothesis. 
\end{proof} 

Now $\iota$ is strict order preserving, since 
$$ \iota(u)\subseteq \tau_{[h(u)]}  \subsetneq  w(h(u),\alpha)= \iota(u\conc \langle\alpha\rangle)$$ 
by the previous claim, Lemma \ref{lemma 3b reformulation} \ref{nice witness} and the definition of $\iota$. 
To see that $\iota$ is an order homomorphism for $({}^\kappa\kappa,H)$, note that
\[
\prod_{\alpha<\ddim} N_{\iota(u\conc\langle\alpha\rangle)}=
\prod_{\alpha<\ddim} N_{w(h(u),\alpha)}
\subseteq H
\]
by the definition of $\iota$ and 
since $w$ is an $H$-witness function.
To see that 
$f=[\iota]$, recall that $f(z)=\tau^{[h](z)}$ by the definition of $f$.
Since $h(z\restr \alpha) \forces_{\Add(\kappa,1)}^{V[G_\gamma]} \iota(z\restr \alpha)\subseteq \tau$, we have $\iota(z\restr \alpha)\subseteq f(z)$ for all $\alpha<\kappa$. 
\end{proof}

This completes the proof of Theorem~\ref{main theorem} except for the proof of the 
Nice Quotient Lemma for $\Col(\kappa,\lle\lambda)$.

\subsection{The countable case}
\label{section: proof in the countable case}

In this section, we complete the proof of the open dihypergraph dichotomy for definable sets in the countable case (Theorem~\ref{main theorem} for $\kappa=\omega$). 
The missing step is the Nice Quotient Lemma for $\Col(\omega,\lle\lambda)$. 
We make use of the next lemma: 

\begin{lemma}[Solovay \cite{Solovay1970}, see \cite{MR1940513}*{Corollary 26.11}] 
\label{Solovay lemma 1}
\index{Solovay's lemma\idf}
For any $x\in{({}^\omega\Ord)}^{V[G]}$, there exists a $\PP_\lambda$-generic
filter\footnote{For $\kappa=\omega$, $\PP_\lambda$ equals $\Col(\omega,\lle\lambda)$, where $\lambda$ is inaccessible. }
$K$ over $V[x]$ such that $V[G]=V[x][K]$.
\end{lemma}

The proof idea for Solovay's lemma is that for any regular uncountable cardinal $\mu$, $\Col(\omega,\mu)$ is the unique forcing of size $\mu$ that makes $\mu$ countable \cite{MR1940513}*{Lemma 26.7}. 
\todog{The phrase ``quotient forcing of $V[x]$'' makes sense by Proposition 10.10 in Kanamori's book. See Proposition 10.21 therein for a proof of Solovay's lemma.}
Using this, one shows that the 
quotient forcing of $V[x]$ in $V[G]$ is equivalent to $\Col(\omega,\lle\lambda)$. 

We use the notation from the previous section for $\kappa=\omega$. 
It remains to prove the following.

\begin{lemma}{\bf (Nice Quotient Lemma for $\Col(\omega,\lle\lambda)$)}
\label{lemma: quotient lemma - countable case}
The following holds in $V[G]$. 
Suppose that $\gamma<\lambda$ and
$s$ is a step-function for $\Add(\omega,1)$.
Then there exists an $s$-function $h$
such that for all $x\in \ran([h])$: 
\begin{enumerate-(1)}
\item
\label{nsf 1 omega}
$x$ is $\Add(\omega,1)$-generic over $V[G_\gamma]$.
\item
\label{nsf 2 omega}
$V[G]$ is a $\PP_\lambda$-generic extension of $V[G_\gamma][x]$.
\end{enumerate-(1)}
\end{lemma}
\begin{proof} 
By Solovay's Lemma~\ref{Solovay lemma 1} for $V[G_\gamma]$, it suffices to construct an $s$-function $h$ such that~\ref{nsf 1 omega} holds.
Let $\langle D_n:n<\omega\rangle\in V[G]$ enumerate all dense subsets of $\Add(\omega,1)$ that are elements of~$V[G_\gamma]$. 
This is possible since $({2^\omega})^{V[G_\gamma]}$ is countable in $V[G]$.
We construct $h:{}^{<\omega}\ddim\to\Add(\omega,1)$ with the following properties for all $u,v\in{}^{<\omega}\ddim$ and $\alpha<\kappa$: 
\begin{enumerate-(i)}
\item\label{om 1} 
If $u\subsetneq v$, then $h(v)< h(u)$.
\vspace{2 pt}
\item\label{om 2} 
$h(u\conc\langle\alpha\rangle)<
s(h(u), \alpha)$.
\vspace{2 pt}
\item\label{om 0} $h(u)\in D_{\lh(u)}$.
\end{enumerate-(i)}
We construct $h(u)$ by recursion on $\lh(u)$.
Let $h(\emptyset)\in D_0$.
Now fix $u\in{}^{<\omega}\ddim$, and suppose that $h(u)$ has been constructed. 
For each $\alpha<\ddim$, 
choose $h(u\conc\langle\alpha\rangle)< s(h(u),\alpha)$ in $D_{\lh(u)+1}$. 
Then $h$ satisfies 
\ref{om 1}-\ref{om 0}. 
By~\ref{om 1}, $h$ is order preserving, and so by~\ref{om 2}, it is an $s$-function.
By~\ref{om 0}, each $x\in\ran([h])$ is $\Add(\omega,1)$-generic over $V[G_\gamma]$.
\end{proof} 

\begin{remark} 
While $\OGD^\omega_\omega(\analytic)$ is provable in $\ZFC$ alone \cite{CarroyMillerSoukup}*{Theorem~1.1}, $\OGD^\omega_\omega(\defsets\omega)$ has the consistency strength of an inaccessible cardinal, since $\OGD^\omega_\omega(\bf{\Pi}^1_1)$ implies that $\PSP_\omega(\bf{\Pi}^1_1)$ holds and thus there is an inaccessible cardinal in $L$ 
(see \cite{MR1940513}*{Theorem 25.38}.)
\end{remark}

\subsection{The uncountable case}
\label{section: proof in the uncountable case}
In this section, we complete the proof of the open dihypergraph dichotomy for definable sets in the uncountable case (Theorem~\ref{main theorem} for $\kappa>\omega$). 
We that assume $\kappa$ is uncountable
throughout the section.

\subsubsection{\texorpdfstring{$\kappa$}{kappa}-analytic sets}
\label{subsection: main theorem for analytic X} 

We begin with the special case of Theorem~\ref{main theorem} for 
$\kappa$-analytic 
sets. 
The proof is much more direct than in the general case, 
since it does not need the Nice Quotient Lemma for $\Col(\kappa,\lle\lambda)$.
It adapts an argument from \cite{SzThesis}*{Section~3.1}
for the open graph dichotomy $\OGD_\kappa(\analytic(\kappa))$. 
We first give a proof 
for {closed} sets.

\begin{lemma}
\label{main theorem for closed X} 
In $V[G]$, suppose $X$ is a closed subset of ${}^\kappa\kappa$. Then $\ODD\kappa {H\restr X}$ holds.
\end{lemma}
\begin{proof}
Recall our Assumptions~\ref{proof notation}.
In particular, we assume that $H\restr X$ does not admit a $\kappa$-coloring. It suffices to show that
there is a continuous homomorphism $f:{}^{\kappa}\ddim\to X$ from $\dhH\ddim$ to $H\restr X$.

In $V[G]$, let $S$ be a subtree of ${}^{<\kappa}\kappa$ such that $[S]=X=\Xphia$. 
There is some $\alpha<\lambda$ such that $S\in V[G_\alpha]$, since $\PP_\lambda$ has the $\lambda$-c.c. 
Now let $\gamma>\alpha$ and let $\nna$ be an $\Add(\kappa,1)$-name as in Lemma~\ref{lemma 3c m}.
\todog{$\one\forces\nna\notin\check V$ is not needed in this proof. (It is needed for the proof of the general case.)}
In $V[G_\gamma]$, 
$\one_{\Add(\kappa,1)}$ forces that 
$\one_{\PP_\lambda}\forces\check{\nna}\in [S]$ by 
Lemma~\ref{lemma 3c m}~\ref{nna eq 1}.
Since the formula ``$x\in [S]$'' is absolute between transitive models of $\ZFC$, we have
\begin{equation}
\label{an eq 1}
\one_{\Add(\kappa,1)} \forces^{V[G_\gamma]} \nna\in[S].
\end{equation}
In $V[G]$, 
we construct functions  
$h:{}^{<\kappa}\ddim\to\Add(\kappa,1)$ 
and 
$t:{}^{<\kappa}\ddim\to{}^{<\kappa}\kappa$
with the following properties 
for all $u,v\in{}^{<\kappa}\ddim$:
\begin{enumerate-(i)}
\item \label{an 1} 
If $u\subsetneq v$, then $t(u)\subsetneq t(v)$ and $h(v)< h(u)$.
\vspace{2 pt}
\item \label{an 2} $h(u)\forces^{V[G_\gamma]} t(u)\subseteq \nna$.
\vspace{2 pt}
\item \label{an 3}
$
\prod_{\alpha<\ddim}N_{t({u\conc\langle\alpha\rangle})}\subseteq H
.
$
\end{enumerate-(i)}
Since the tree $T^{\tau,q}_{\Add(\kappa,1)}$ is not $H$-independent for all $q\in\Add(\kappa,1)$ 
by Lemma~\ref{lemma 3c m}~\ref{nna eq 3},
there exists an $H$-witness function for $\nna$. 
Thus, we may fix an $H$-witness function $w$ and a step function $s$ for $\nna$ as in Lemma~\ref{lemma 3b reformulation}. 
%
Construct $h(v)$ and $t(v)$ by recursion on~$\lh(v)$.
Let $h(\emptyset):=\one_{\Add(\kappa,1)}$ 
and $t(\emptyset):=\emptyset$.
Now suppose that $v\in{}^{<\kappa}\ddim$ and 
$h(u),\,t(u)$ have been constructed for all $u\subsetneq v$.
If $\lh(v)\in\Lim$,
let $t(v):=\bigcup_{u\subsetneq v}t(u)$. 
Since 
$\Add(\kappa,1)$ is $\lle\kappa$-closed, there is some $h(v)< h(u)$ for all $u\subsetneq v$. 
If $\lh(v)\in\Succ$ and $v=u\conc\langle\alpha\rangle$,
let $t(v):=w(h(u),\alpha)$
 and $h(v)<s(h(u),\alpha)$.

Now suppose that 
$h:{}^{<\kappa}\ddim\to\Add(\kappa,1)$ 
and 
$t:{}^{<\kappa}\ddim\to S$
have been constructed. 
Note that $t$ is strict order preserving by \ref{an 1}. 
Hence $f=[t]:{}^\kappa\ddim\to X$ is defined.
By \ref{an 3}, $t$ is an order homomorphism for $({}^\kappa\kappa,H)$, so 
$f$ is a continuous homomorphism from $\dhH\ddim$ to~$H$. 
It thus suffices that
 $\ran(f)\subseteq [S]=X$. 
This holds since
\eqref{an eq 1} and \ref{an 2} imply that
$\ran(t)\subseteq S$.
\end{proof}

\begin{remark}
\label{alternative proof for closed X}
The following alternative version of the previous proof can be understood as a special case of our proof of Theorem~\ref{main theorem}. 
Let $S$, $\gamma$ and $\nna$ be as in the previous proof. Then
$\one_{\Add(\kappa,1)} \forces^{V[G_\gamma]} \nna\in[S]$. 
Choose an $H$-witness function $w$ and a step function $s$ as in Subsection~\ref{subsection: the main step}.
A similar construction as in 
the Nice Quotient Lemma for $\Col(\omega,\lle\lambda)$%
\footnote{See the proof of Lemma \ref{lemma: quotient lemma - countable case}.} 
shows that there exists an $s$-function 
$h:{}^{<\kappa}\ddim\to{}^{<\kappa}\kappa$ 
such that $[h](x)$ is $\Add(\kappa,1)$-generic over $V[G_\gamma]$ for all $x\in{}^\kappa\ddim$.\footnote{This is a weaker version of the Nice Quotient Lemma~\ref{lemma: quotient lemma} 
obtained by omitting~\ref{nsf 2}.} 
For $f: {}^\kappa\ddim\to {}^\kappa\kappa$ with $f(x)=\nna^{[h](x)}$, 
it follows that $\ran(f)\subseteq [S]$ by the absoluteness of ``$x\in[S]$''. 
One can show that $f$ is a homomorphism from $\dhH\ddim$ to $H$ as in 
Lemma~\ref{lemma: f cont hom}. 
Thus, $f$ is a continuous homomorphism from $\dhH\ddim$ to $H\restr[S]=H\restr X$.
\end{remark}

By the previous lemma, Theorem~\ref{main theorem} holds for all closed subsets of the $\kappa$-Baire space.
Since every 
$\kappa$-analytic 
subset of ${}^\kappa\kappa$ is the continuous image of a closed subset of ${}^\kappa\kappa$, 
we can extend this 
to 
$\kappa$-analytic 
sets by Lemma~\ref{ODD for continuous images}.

\begin{corollary}
\label{main theorem for analytic X}
Theorem~\ref{main theorem} holds for all 
$\kappa$-analytic
subsets of ${}^\kappa\kappa$.
\end{corollary}

\subsubsection{Definable sets}
\label{subsection: main theorem for uncountable kappa}
We prove the Nice Quotient Lemma for $\Col(\kappa,\lle\lambda)$ and thus Theorem~\ref{main theorem} for uncountable $\kappa$.
The main difficulty is that quotient forcings are not well behaved in the uncountable setting: 

\begin{remark} 
\label{failure of Solovay's lemma} 
Solovay's Lemma~\ref{Solovay lemma 1} does not extend to uncountable regular cardinals $\kappa$ by the following well-known counterexample. 
Let $V[J*K]$ be an $\Add(\kappa,1)*\dot{\QQ}$-generic extension of $V$, where $\Add(\kappa,1)$ adds a stationary set $S$ and $\dot{\QQ}$ shoots a club through its complement. 
Thus $\dot{\QQ}^J$ is not stationary set preserving in $V[J]$. 
Hence  $V[J*K]$ is not a generic extension of $V[J]$ by a ${<}\kappa$-closed forcing. 
This is discussed in more detail in \cite{SchlichtPSPgames}*{Subsection 1.3}. 
\end{remark} 

Instead of Solovay's lemma, we will use the following Nice Quotient Lemma for $\Add(\kappa,1)$. 
Its proof is postponed to
the next subsection.

\begin{lemma}{\bf (Nice Quotient Lemma for $\Add(\mu,1)$)}
\label{Q main lemma}
\index{Nice Quotient Lemma for!Add@$\Add(\mu,1)$\idf}
Let $M$ be a transitive class model of $\ZFC$.\footnote{I.e., $M$ can be a proper class or a set.} 
In $M$, suppose that $\mu$ is an uncountable cardinal with $\mu^{<\mu}=\mu$, 
$2\leq\ddimq\leq\mu$ and
$s:\Add(\mu,1)\times\ddimq\to\Add(\mu,1)$ is a step function for $\Add(\mu,1)$.
For any $\Add(\mu,1)$-generic filter $K$ over $M$, 
there exists in $M[K]$ an $s$-function 
$h:{}^{<\mu}\ddimq\to\Add(\mu,1)$
such that for all 
$x\in \ran([h])$: 
\begin{enumerate-(1)}
\item\label{Q main 1}
$x$ is $\Add(\mu,1)$-generic over $M$.
\item\label{Q main 2}
$M[K]$ is an $\Add(\mu,1)$-generic extension of $M[x]$.
\end{enumerate-(1)}
In fact, in~\ref{Q main 1},
for all distinct $y,z\in{}^\mu\ddimq$,
$[h](y)$ and $[h](z)$ are mutually $\Add(\mu,1)$-generic over $M$.
\end{lemma}

\begin{lemma}{\bf (Nice Quotient Lemma  for $\Col(\kappa,\lle\lambda)$)}
\label{lemma: quotient lemma - uncountable case}
The following holds in $V[G]$. 
Suppose that $\gamma<\lambda$ and
$s$ is a step-function for $\Add(\kappa,1)$.
Then there exists an $s$-function 
$h:{}^{<\kappa}\ddim\to\Add(\kappa,1)$
such that for all $x\in \ran([h])$: 
\begin{enumerate-(1)}
\item\label{nsf 1 uncountable}
$x$ is $\Add(\kappa,1)$-generic over $V[G_\gamma]$.
\item\label{nsf 2 uncountable}
$V[G]$ is a 
$\PP_\lambda$-generic
extension of $V[G_\gamma][x]$.
\end{enumerate-(1)}
In fact in~\ref{nsf 1 uncountable}, for all distinct $y,z\in{}^\kappa\ddim$,
$[h](y)$ and $[h](z)$ are mutually $\Add(\kappa,1)$-generic over $V[G_\gamma]$.
\end{lemma}
\begin{proof}
\label{proof: quotient lemma uncountable case} 
We rearrange $V[G]$ as an $\Add(\kappa,1)$-generic extension of a model $M$. 
Fix some $\alpha\in[\gamma,\lambda)$ with $s\in V[G_\alpha]$ by the $\lambda$-c.c. 
%
Since $\PP^\alpha$ is equivalent to $\Add(\kappa,1)\times\PP_\lambda$,\footnote{See Corollary~\ref{forcing equivalent to the Levy collapse}.} 
we write 
$$V[G]=V[G_\alpha\times K\times J],$$ 
where is $K\times J$ is an $\Add(\kappa,1)\times\PP_\lambda$-generic filter over $V[G_\alpha]$. 
Then $V[G]$ is an $\Add(\kappa,1)$-generic extension of $M:=V[G_\alpha\times J]$. 

Note that $s$ is a step-function in $M$, 
since this concept is absolute between transitive $\ZFC$-models with the same ${}^{<\kappa}\kappa$.
Fix an $s$-function $h$ in $V[G]$
as in the previous lemma.
%
Then each $x\in\ran([h])$ is $\Add(\kappa,1)$-generic over $V[G_\gamma]$, since $V[G_\gamma]\subseteq M$. 
It remains to see that $V[G]=M[K]$ is a 
$\PP_\lambda$-generic extension of $V[G_\gamma][x]$.
Recall that $M[K]$ is an $\Add(\kappa,1)$-generic extension of $M[x]$. 
Moreover, 
$$M[x]=V[G_\alpha\times J][x]=V[G_\gamma][x][ G_{[\gamma,\alpha)}\times J]$$ 
is a $\PP_{[\alpha,\gamma)}\times\PP_\lambda$-generic extension of  $V[G_\gamma][x]$. 
Since $\PP_{[\alpha,\gamma)}\times\PP_\lambda\times\Add(\kappa,1)$ is equivalent to $\PP_\lambda$,\footnote{See Corollary~\ref{forcing equivalent to the Levy collapse}.} 
this suffices. 
\end{proof} 

\begin{remark}
\label{the f constructed is injective remark 1}
In the uncountable case only, the proof yields the following stronger version of Theorem~\ref{main theorem}. 
For a given $\ddim$-dihypergraph $I$ on ${}^\kappa\kappa$,
let 
$\ODDI\kappa I$ denote the statement:  
\begin{quotation}
$\ODDI\kappa I$: Either $I$ has a $\kappa$-coloring, or 
there exists 
an \emph{injective} continuous homomorphism
from~$\dhH{\ddim}$ to~$I$.%
\footnote{See Definition~\ref{def: ODD variants} below.}
\end{quotation}
We claim that
if $\kappa<\lambda$ are regular uncountable cardinals and $2\leq\ddim\leq\kappa$,
then the following hold in all $\PP_\lambda$-generic extensions $V[G]$:
\begin{enumerate-(1)}
\item\label{the f constructed is injective remark 1 1}
$\ODDI\kappa\ddim(\defsetsk,\defsetsk)$ if $\lambda$ is inacessible.%
\footnote{The definitions of 
$\ODDI\kappa\ddim(\defsetsk,\defsetsk)$ and
$\ODDI\kappa\ddim(\defsetsk)$  
are analogous to the ones in Definition~\ref{def: ODD for classes}.
Note that \ref{the f constructed is injective remark 1 1} implies $\ODDI\kappa\ddim(\defsetsk)$ for all $\ddim$ with $2\leq\ddim<\kappa$ by Lemma~\ref{<kappa dim hypergraphs are definable}.}
\item\label{the f constructed is injective remark 1 2}
$\ODDI\kappa\ddim(\defsetsk)$ if $\lambda$ is Mahlo.
\end{enumerate-(1)}

To prove this, recall that after the Nice Quotient Lemma \ref{lemma: quotient lemma}, 
we defined a continuous homomorphism $f$ from $\dhH\ddim$ to $H\restr X$ by letting $f(x):=\nna^{[h](x)}$, where $\gamma$ and $\nna$ are as in Lemma~\ref{lemma 3c m}. 
In particular, note that $\nna$ is a name in $V[G_\gamma]$ for a new element of  ${}^\kappa\kappa$.
It remains to show that $f$ is injective. Let $x,y$ be distinct elements of ${}^\kappa\ddim$, 
let $x':=[h](x)$ and $y':=[h](y)$.
Then $f(x)=\nna^{x'}\in V[G_\gamma][x'] - V[G_\gamma]$
and
$f(y)=\nna^{y'}\in V[G_\gamma][y'] - V[G_\gamma]$.
Since
$x'$ and $y'$ are mutually generic over $V[G_\gamma]$,
we have
$V[G_\gamma][x']\cap V[G_\gamma][y']=V[G_\gamma]$.
Therefore
$f(x)\neq f(y)$.

Theorem~\ref{theorem: ODD ODDI}
and  Corollary~\ref{main theorem strong version} 
below provide an alternative proof using a weak variant of the $\Diamond$ principle. 
\end{remark}

\begin{remark} 
Note that the Nice Quotient Lemmas~\ref{Q main lemma} and \ref{lemma: quotient lemma - uncountable case} both fail in the countable case, 
since $\ODDI\omega\omega(\defsets\omega,\defsets\omega)$ fails by Proposition \ref{ODDI fails for D omega} below. 
While Lemma~\ref{lemma: quotient lemma - uncountable case} does hold without the last condition on mutual genericity by Lemma \ref{lemma: quotient lemma - countable case}, we do not know if this is the case for Lemma~\ref{Q main lemma}. 
\end{remark}

\subsubsection{The Nice Quotient Lemma for \texorpdfstring{$\Add(\mu,1)$}{Add(mu,1)}} 
\label{subsection: the forcing Q}

This subsection is dedicated to the proof of Lemma~\ref{Q main lemma}. 
We make the same assumptions as in this lemma
throughout the subsection. 
In particular,
$M$ is a transitive class model of $\ZFC$. 
In $M$, $\mu$ is an uncountable cardinal with $\mu^{<\mu}=\mu$, 
\todog{we need to assume that $\mu^{<\mu}=\mu$ holds in $M$ e.g. to have $|\QQ|=\mu$}
$2\leq\ddimq\leq \mu$ and 
$s:\Add(\mu,1)\times\ddimq\to\Add(\mu,1)$
is a step function for $\Add(\mu,1)$.%
\footnote{See Definition~\ref{def: step function}}

Our proof uses ideas from~\cite{SchlichtPSPgames}.
We first give a short description of the proof.
Working in $M$, we define a forcing $\QQ$ whose domain contains
\todog{I wrote ``whose domain contains'' because: the partial order on $\QQ$ is unusual; also, $\QQ$ consists of \emph{strict} partial $s$-functions only}
partial functions of size $\lle\mu$ approximating 
an $s$-function.%
\footnote{See Definition~\ref{def: s-function}.}
We first prove that
$\QQ$ is equivalent to $\Add(\mu,1)$. 
We then show that
if $h$ is the $s$-function added in the natural way by a given $\QQ$-generic filter, 
then $h$ satisfies the requirements in Lemma~\ref{Q main lemma}.
A significant portion of the arguments in this subsection are aimed at proving the second requirement 
that elements of $\ran([h])$ have well-behaved quotients.

We now present the detailed proof.
We argue either in $M$ or in a $\QQ$-generic extension of $M$ throughout this subsection. 

Recall that an $s$-function for $\Add(\mu,1)$ is a strict order preserving
map $h:{}^{<\mu}\ddimq\to{}^{<\mu}\mu$%
\footnote{Note that ${}^{<\mu}\mu$ is $\Add(\mu,1)$ with the ordering reversed.}
which is ``built along'' $s$ in 
the sense of Definition~\ref{def: s-function}.
In \ref{psf1} of the next definition, we consider partial functions 
with the same property.
Our forcing $\QQ$ will consist of 
such partial $s$-functions of size $\lle\mu$ 
with the additional technical requirement of strictness defined in \ref{psf2} 
that is used to ensure $\QQ$ is separative. 
\todog{$\supsetneq$ (instead of $\supseteq$) in 
the definition of ``partial $s$-function''
is also needed only to ensure separativity}

\begin{definition}\ 
\label{def: partial s-function}
\begin{enumerate-(a)}
\item\label{psf1}
\index{function!partial s-function@partial $s$-function\idf}%
A \emph{partial $s$-function} is a
strict order preserving \todog{the requirement in the equation below does not guarantee that $p$ is order preserving at limit lengths}
partial function 
$p:{}^{<\mu}\ddimq\partialto{}^{<\mu}\mu$ 
such that 
$\dom(p)$ is a subtree of ${}^{<\mu}\ddimq$, and
if $u,u\conc\langle \alpha\rangle\in\dom(p)$, then
\[
p\big(u\conc\langle\alpha\rangle\big)
\supsetneq 
s\big(p(u), \alpha\big).
\]
\item\label{psf2}
A partial $s$-function $p$ is \emph{strict} 
\index{function!strict partial $s$-function\idf} 
if for all $t\in\dom(p)$ with $\lh(t)\in\Lim$,\todog{``strict order preserving'' already implies this requirement for $\lh(t)\in\Succ$}  
we have
\[p(t)\supsetneq \bigcup_{u\subsetneq t} p(u).\]
\end{enumerate-(a)}
\end{definition}

We now define the forcing $\QQ$ that will be used in our proof of the Nice Quotient 
Lemma~\ref{Q main lemma} for $\Add(\mu,1)$.
It is embedded into a poset $\mpo$ that will be useful as well.

\begin{definition}\ 
\label{def of the forcing}
\begin{enumerate-(a)}
\item\label{mpo def}
\index{poset of strict partial $s$-functions!all\idf$\mpo$}%
$\mpo$ consists of all strict partial $s$-functions.
\item\label{QQ def}
\index{poset of strict partial $s$-functions!with domain of size $<\mu$\idf$\QQ$}%
\index{forcing!of strict partial $s$-functions\idf$\QQ$}%
$\QQ$ consists of all $q\in\mpo$ with $\left|\dom(q)\right|<\mu$.
\item\label{ordering def}
We equip both $\QQ$ and $\mpo$ with the partial order $\mleq$ defined by letting 
$p\mleq q$ 
if the following conditions hold:
\begin{enumerate-(i)}
\item
$\dom (q)\subseteq\dom (p)$. 
\item
$q(t)=p(t)$ for every non-terminal node $t$ of $\dom(q)$.
\item
$q(t)\subseteq p(t)$ for every terminal node $t$ of $\dom(q)$.
\end{enumerate-(i)}
\end{enumerate-(a)}
\end{definition}

Thus, a condition $q$ is extended by moving its topmost values upwards and end-extending the tree $\dom(q)$. 
Note that the definition of $\QQ$ is absolute to transitive models 
$N\supseteq M$ of $\ZFC$ with $({}^{<\mu}\mu)^M=({}^{<\mu}\mu)^N$. 
In detail, since $\QQ$ is $\lle\mu$-closed by Lemma~\ref{Q equivalent to Add(kappa,1)} below, the definition of $\QQ$ is absolute between $M$ and $\QQ$-generic extensions of $M$.
However, 
$\mpo^N$ may contain some 
total $s$-functions 
$h:{}^{<\mu}\ddimq\to{}^{<\mu}\mu$
that are not in $M$.

We now discuss some properties of $\mpo$ which will be useful in this subsection.
The next lemma 
gives an equivalent reformulation 
of the definition of strict partial $s$-functions. 
\todog{Both the next claim and lemma will be needed right away} 

\begin{lemma}
\label{def of the ordering equiv}
Let $q:{}^{<\mu}\ddimq\partialto{}^{<\mu}\mu$ be a 
\todog{the assumption that $q$ is strict order preserving is not needed}
\todog{$\supsetneq$ (instead of $\supseteq$) in \eqref{Q equiv eq} is needed for $\QQ$ to be separative; see the last paragraph of the proof of Lemma~\ref{Q separative}.\\
This is why $\supsetneq$ was needed in the def of ``(partial) $s$-function'' and also why the partial $s$-functions in $\QQ$ need to be strict}
partial function whose domain 
is a subtree of~${}^{<\mu}\ddimq$. 
Then 
$q$ is a strict partial $s$-function (i.e., $q\in\mpo$)
if and only if 
for all $t\in\dom(q)$,
\begin{equation}
\label{Q equiv eq}
q(t)\supsetneq \bigcup_{\alpha< \lh(t)} s\big(q(t\restr\alpha), t(\alpha) \big). 
\end{equation}
\end{lemma}
\begin{proof}
First suppose that 
$q$ is a strict partial $s$-function 
and $t\in\dom(q)$.
If $\lh(t)\in\Lim$, then 
\[
q(t)\supsetneq \bigcup_{\alpha< \lh(t)} q(t\restr\alpha)=
\bigcup_{\alpha< \lh(t)} q(t\restr(\alpha+1))\supseteq
\bigcup_{\alpha< \lh(t)} s\big(q(t\restr\alpha), t(\alpha)\big).
\]
The inclusion on the left 
holds since $q$ is strict, while the 
inclusion on the right holds 
since $q$ is a partial $s$-function. 
Now assume
$\lh(t)\in\Succ$ and $t=u\conc\langle\beta\rangle$.
Since $s$ is a step function
and
$q$ is a partial $s$-function,
we have 
\smallskip 
\begin{equation*}
\begin{gathered}
 q(u)\subseteq s(q(u),\beta)\subsetneq q(t),  
\\[4pt] 
 s\big(q(t\restr\alpha), t(\alpha)\big)\subsetneq
q(t\restr(\alpha+1))
\subseteq q(u)
\smallskip 
\end{gathered}
\end{equation*}
for all $\alpha<\lh(u)$.
Therefore
$\bigcup_{\alpha< \lh(t)} s\big(q(t\restr\alpha), t(\alpha)\big)=
s(q(u),\beta)\subsetneq q(t)$. 

Conversely,~\eqref{Q equiv eq} implies immediately that $q$ is a strict partial $s$-function. 
In particular, $q$ is strict order preserving since $s$ is a step-function.
\end{proof}

Recall that $K\subseteq\mpo$ is a \emph{downward directed} 
\index{downward directed subset\idf}
subset of the poset $\mpo$ if for every $p,q\in K$, there exists $r\in K$ with $r\mleq p,q$.
In particular, any $\QQ$-generic filter $K$ over $M$ is a downward directed subset
of $\mpo^{M[K]}$.
We next define $\mbigcap K\in \mpo$ for any downward directed 
subset $K$ of $\mpo$. This will turn out to equal the infimum in relevant cases.

\begin{definition}
\label{mbigcap def}
Suppose $K$ is a downward directed subset of $\mpo$.
\index{meet of!strict partial $s$-functions\idf$\mbigcap K$}
Let $\mbigcap K$ denote the partial function from ${}^{<\mu}\ddimq$ to ${}^{\leq\mu}\mu$ defined as follows:
\smallskip 
\begin{equation*}
\begin{gathered}
 \dom(\mbigcap K\big):=\bigcup\{\dom(q):q\in K\},
\\[4pt]
 (\mbigcap K) (t):=\bigcup\{q(t):{q\in K,\,t\in\dom(q)}\}.
\end{gathered}
\end{equation*}
\end{definition}

Note that the definition of $\mbigcap K$ 
is absolute between transitive models $N\subseteq N'$ of $\ZFC$ with $K\in N$. \todog{This might not be needed, but it's interesting to note that the infimum of a subset $K\in M$ of $\QQ$ will be the same in $\mpo^M$ and $\mpo^{M[K]}$}
The next lemma shows that $\mbigcap K$ is indeed the infimum of $K$ in~$\mpo$ in all relevant cases. 

\begin{lemma}
\label{mbigcap facts}
Let $K$ be a downward directed subset of $\mpo$.
\begin{enumerate-(1)}
\item \label{mbigcap fact 1}
If $q\in K$ and $t$ is a non-terminal node of $\dom(q)$, then 
$(\mbigcap K)(t)=q(t)$.
\item \label{mbigcap fact 2}
Suppose that for any $t\in \dom(\mbigcap K)$, there is some $p\in K$ such that $t$ is a non-terminal node in $\dom(p)$. 
Then $\ran(\mbigcap K) \subseteq {}^{<\mu}\mu$. 
\todog{The assumption that $|K|<\mu$ also implies that $\ran(\mbigcap K) \subseteq {}^{<\mu}\mu$. However, this will not be needed. \ref{mbigcap fact 4} will be enough.}
\item \label{mbigcap fact 3}
If $\ran(\mbigcap K) \subseteq {}^{<\mu}\mu$, then $\mbigcap K\in\mpo$, and 
$\mbigcap K$
is the infimum of $K$ in $\mpo$.
\item \label{mbigcap fact 4} If $K$ is a downward directed subset of $\QQ$ of size $\lle\mu$, then $\mbigcap K\in\QQ$ and $\mbigcap K$ is the infimum of $K$ in $\QQ$. 
\end{enumerate-(1)}
\end{lemma}

\begin{proof}
For \ref{mbigcap fact 1}, 
note that $q(t)\subseteq (\mbigcap K)(t)$.
Conversely, if $p\in K$ and $t\in\dom(p)$, then there exists $r\in K$ such that $r\mleq p,q$ and therefore
$p(t)\subseteq r(t)=q(t)$. 
Hence $(\mbigcap K)(t)\subseteq q(t).$ 

For \ref{mbigcap fact 2}, let $t\in \dom(\mbigcap K)$ and 
take $p\in K$ such that $t$ is a non-terminal node in $\dom(p)$. 
Then $(\mbigcap K) (t)=p(t)\in {}^{<\mu}\mu$ by \ref{mbigcap fact 1}. 

For \ref{mbigcap fact 3},
we first claim that $\mbigcap K\in\mpo$.
Since $\dom(\mbigcap K)$ is a subtree of ${}^{<\mu}\ddimq$,
it suffices to show~\eqref{Q equiv eq} for $\mbigcap K$
by the previous lemma.
Let $t\in\dom(\mbigcap K)$ and take $p\in K$ with $t\in\dom(p)$.
Then $t\restr\alpha$ is a non-terminal node of $\dom(p)$ for all $\alpha<\lh(t)$,  
so $(\mbigcap K)(t\restr\alpha)=p(t\restr\alpha)$ by \ref{mbigcap fact 1}. 
Since~\eqref{Q equiv eq} holds for $p$,
\[
\bigcup_{\alpha< \lh(t)} s\big((\mbigcap K)(t\restr\alpha), t(\alpha) \big)=
\bigcup_{\alpha< \lh(t)} s\big(p(t\restr\alpha), t(\alpha) \big)\subsetneq
p(t)\subseteq
(\mbigcap K)(t).
\]
We now claim that $\mbigcap K$ is the infimum of $K$ in $\mpo$. 
We first show that $\mbigcap K\mleq p$ for any $p\in K$.  
By definition, $\dom(p)\subseteq \dom(\mbigcap K)$ 
and $p(t)\subseteq (\mbigcap K)(t)$ for any $t\in\dom(p)$.
Moreover, if $t$ is a non-terminal in $\dom(p)$, then $(\mbigcap K)(t)=p(t)$ by  \ref{mbigcap fact 1}. 
Hence $\mbigcap K\mleq p$. 
Now, suppose $h\in\mpo$ and $h\mleq p$ for all $p\in K$. It suffices to show 
$h\mleq \mbigcap K$.
By definition, $\dom(\mbigcap K)\subseteq\dom(h)$ and $(\mbigcap K)(t)\subseteq h(t)$ for all $t\in\dom(\mbigcap K)$. 
Moreover, if $t$ is a non-terminal node of $\dom(\mbigcap K)$, then there exists $p\in K$ such that $t$ is a non-terminal node of $\dom(p)$.
Hence 
$(\mbigcap K)(t)=p(t)=h(t)$ by
\ref{mbigcap fact 1} and $h\mleq p$.
Therefore $h\leq\mbigcap K$.

For \ref{mbigcap fact 4}, note that $|\dom(\mbigcap K)|<\mu$ and $\ran(\mbigcap K)\subseteq {}^{<\mu}\mu$. 
Thus the claim follows from~\ref{mbigcap fact 3} since $\QQ\subseteq\mpo$.
\end{proof}

\begin{remark}
\label{remark: infima in BB and mpo}
There is no general relationship between $\mbigcap K$ and the infimum of $K$ in the Boolean completion $\BB(\QQ)$ of $\QQ$.
However, they are equal if $K$ is a downward directed subset of $\QQ$ of size ${<}\mu$ 
using Lemma~\ref{mbigcap facts}~\ref{mbigcap fact 4} 
and the fact that $\QQ$ is a dense subset of $\BB(\QQ)$.%
\footnote{$\QQ$ can be assumed to be a dense subset of $\BB(\QQ)$,
since we will prove in Lemma~\ref{Q separative} below that $\QQ$ is separative. 
}
\end{remark}

Our first goal is to show that the forcing $\QQ$ is equivalent to $\Add(\kappa,1)$.
This is done in Lemma~\ref{Q equivalent to Add(kappa,1)}.
Before proving this, 
we characterize
when two given elements of $\QQ$ are compatible, in the next lemma. 
This will be used to prove that $\QQ$ is 
non-atomic%
\footnote{See the proof Lemma~\ref{Q equivalent to Add(kappa,1)}.} 
and separative.%
\footnote{See Lemma~\ref{Q separative}.} 

\begin{lemma}
\label{perp_Q claim}
Two elements $p$ and $q$ of $\QQ$ are compatible if and only if
the following hold for all $t\in\dom(p)\cap\dom(q)$:
\begin{enumerate-(1)}
\item\label{comp 1}  
$p(t)\compat q(t)$.%
\footnote{I.e., $p(t)\subseteq q(t)$ or $q(t)\subseteq p(t)$.}
\item\label{comp 2} 
If $t$ is a non-terminal node of $\dom(p)$, then
$q(t)\subseteq p(t)$. 
\item\label{comp 3} 
If $t$ is a non-terminal node of $\dom(q)$, then
$p(t)\subseteq q(t)$. 
\end{enumerate-(1)}
In particular, if $p$ and $q$ are compatible, then
$p(t)=q(t)$ 
whenever $t$ is non-terminal in both $\dom(p)$ and $\dom(q)$.
\end{lemma}
\begin{proof}
First, assume $p$ and $q$ are compatible
and $t\in\dom(p)\cap\dom(q)$.
Let $r\in \QQ$ be a common extension of $p$ and $q$.
Then
$r(t)\supseteq p(t)$ and $r(t)\supseteq q(t)$, and hence $p(t)\compat q(t)$. 
If 
$t$ is a non-terminal node of $\dom(p)$,
then $q(t)\subseteq r(t)=p(t)$.
Similarly, 
if $t$ is a non-terminal node of $\dom(q)$,
then $p(t)\subseteq r(t)=q(t)$.

Now assume \ref{comp 1}--\ref{comp 3} hold for all $t\in\dom(p)\cap\dom(q)$.
Let $r$ be the unique function with $\dom(r)=\dom(p)\cup\dom(q)$ and $r(t)=p(t)\cup q(t)$ for all $t\in\dom(r)$. 
Then  $r(t)=\max\{p(t),q(t)\}$
for all $t\in\dom(p)\cap\dom(q)$
by~\ref{comp 1}.
Moreover,
$r(t)=p(t)$ for all 
non-terminal nodes $t$ of $\dom(p)$
by~\ref{comp 2},
and
$r(t)=q(t)$  for all 
non-terminal nodes $t$ of $\dom(q)$
by~\ref{comp 3}.
Using these observations, it is easy to check that $r\in\QQ$ and 
that $r\leq p, q$.
\end{proof}

\begin{lemma}
\label{Q equivalent to Add(kappa,1)}
$\QQ$ is equivalent to $\Add(\mu,1)$.
\end{lemma}
\begin{proof}
By Lemma~\ref{forcing equivalent to Add(kappa,1)}, it suffices to show that $\QQ$ is  a non-atomic $\lle\mu$-closed forcing of size $\mu$. 
The assumption $\mu^{<\mu}=\mu$ implies that $|\QQ|=\mu$.
To show that $\QQ$ is $\lle\mu$-closed, suppose that $\langle p_\alpha:\alpha<\xi\rangle$ is a descending sequence of elements of $\QQ$ of length $\xi<\mu$. Then 
$p=\mbigcap\{p_\alpha:\alpha<\xi\}$ is an element of $\QQ$ such that $p\leq p_\alpha$ for all $\alpha<\xi$ by Lemma~\ref{mbigcap facts}~\ref{mbigcap fact 4}.
To show that $\QQ$ is non-atomic, let $p\in \QQ$. Take a minimal element $t$ of ${{}^{<\mu}\ddimq}-\dom(p)$, let
$u=\bigcup_{\alpha< \lh(t)} s\big(q(t\restr\alpha), t(\alpha)\big),$
and choose $u_0,u_1\in{}^{<\mu}\mu$ with $u_0\perp u_1$ and $u_0,u_1\supsetneq u$.
Finally, let $q_i:=p\cup\{(t,u_i)\}$ for $i=0,1$. 
Then $q_0,q_1\in\QQ$ by Lemma~\ref{def of the ordering equiv}.
It is clear that 
$q_0,q_1\leq p$, and
$q_0\perp q_1$ holds by 
the previous lemma.
\end{proof}

Throughout the rest of this subsection, we use the following notation.
\begin{itemize}
\item
\label{h_K def}
Suppose $K$ is a $\QQ$-generic filter over $M$.
In $M[K]$, let
\index{function!s-functionz@$s$-function added by $\QQ$\idf$h_K$, $\dot h$}%
$$
h_K=\mbigcap K.
$$
\item
\label{dot h def}
In $M$, let $\dot K$ denote the canonical name for a $\QQ$-generic filter,
and let 
\index{function!s-functionz@$s$-function added by $\QQ$\idf$h_K$, $\dot h$}%
$\dot h$ be a $\QQ$-name for $\mbigcap \dot K$ in the sense that
$$
\one_\QQ\forces\:\dot h=\mbigcap \dot K.
$$
\end{itemize}
$h_K$ is well defined
since $K$ is a filter on $\QQ$ and hence an upward directed subset of 
$\mpo^{M[K]}$.
In the rest of this subsection, we will show that 
$h_K$ is an $s$-function which
satisfies the conclusion of the Nice Quotient Lemma~\ref{Q main lemma} for $\Add(\mu,1)$.
We first prove 
that $h_K$ is indeed an $s$-function.
We will need the following lemma.

\begin{lemma}
\label{h_K domain sublemma}
If $t\in{}^{<\mu}\ddimq$, then 
$D_t:=\{r\in\QQ:t\in\dom(r)\}$
is a dense subset of $\QQ$. 
\end{lemma}
\begin{proof}
We prove the lemma by induction on $\lh(t)$.
Assume 
$D_{t\restr\alpha}$ is dense in $\QQ$
for all $\alpha<\lh(t)$.
To show that $D_t$ is dense in $\QQ$, let $p\in\QQ$.
Using the fact that $\QQ$ is $\lle\mu$-closed, construct a decreasing sequence 
$\langle q_\alpha:\alpha<\lh(t)\rangle$ of conditions below $p$
such that 
$q_\alpha\in D_{t\restr\alpha}$, and let
$q:=\mbigcap\{q_\alpha:\alpha<\lh(t)\}.$
Then $q\in\QQ$
and $q\leq  p$
 by Lemma~\ref{mbigcap facts}~\ref{mbigcap fact 4}, and
 $t\restr\alpha\in\dom(q)$ for all $\alpha<\lh(t)$.
We define a condition 
$r\in D_t$ below $q$
as follows.
If $t\in\dom(q)$, then let $r:=q$.
If $t\notin\dom(q)$, fix $u\in{}^{<\mu}\mu$ with 
$
u\supsetneq \bigcup_{\alpha< \lh(t)} s\big(q(t\restr\alpha), t(\alpha) \big)
$
and let $r:=q\cup\{(t,u)\}$. 
Then $r\in \QQ$ by Lemma~\ref{def of the ordering equiv},
and 
it is clear that $t\in\dom(r)$ and $r\leq  q$. 
Thus, 
$r$ is an element of $D_t$ below~$p$.
\end{proof}

\begin{corollary}
\label{h_K is an s-function}
Suppose that $K$ is a $\QQ$-generic filter over $M$. Then in $M[K]$:
\begin{enumerate-(1)}
\item\label{dom h_K}
$\dom(h_K)={}^{<\mu}\ddimq$.
\item\label{h_K in mpo}
$h_K\in\mpo$ and $h_K$ is the infimum of $K$ in $\bar{\QQ}$. 
\item\label{h_K s-function}
$h_K$ is an $s$-function.
\end{enumerate-(1)}
\end{corollary}
\begin{proof}
\ref{dom h_K} follows from the previous lemma.
For~\ref{h_K in mpo}, it suffices to show $\ran(h_K)\subseteq {}^{<\mu}\mu$ by Lemma~\ref{mbigcap facts}~\ref{mbigcap fact 3}.
To see this, note that $\dom(h_K)={}^{<\mu}\ddimq=\bigcup_{p\in K}\dom(p)$ by~\ref{dom h_K} and the definition of~$\mbigcap$. 
Thus,  $h_K=\mbigcap K$ by Lemma~\ref{mbigcap facts}~\ref{mbigcap fact 2}, as required.
\ref{h_K s-function} follows immediately from~\ref{dom h_K} and~\ref{h_K in mpo}.
\end{proof}

The following lemma shows that any $\QQ$-generic filter
$K$ is definable from $h_K$.
Hence
$M[K]=M[h_K]$. 

\begin{lemma}
\label{K and h_K}
For all $\QQ$-generic filters $K$ over $M$ 
\todog{This lemma is used twice later.}
and all $q\in\QQ$,
\[
q\in K 
\ \Longleftrightarrow \ 
h_K\mleq q.
\]
\end{lemma}
\begin{proof}
If $q\in K$, then $h_K\mleq q$  
by the previous corollary.
Conversely, suppose $h_K\mleq q$. 
Since $K$ is $\QQ$-generic, it suffices to show that all elements $p$ of $K$ are compatible with $q$ in $\QQ$. 
Given $p\in K$, 
let $r:=h_K\restr(\dom(p)\cup \dom(q))$. 
Then $r\in\QQ$ since 
$h_K\in\mpo^{M[K]}$ and $|\dom(p)\cup\dom(q)|<\mu$. 
Moreover, $r\leq  p,\,q$ 
since $h_K\mleq p,\,q$. 
\end{proof}


In the remainder of the subsection, we will prove that $h_K$ satisfies the conclusion of the Nice Quotient Lemma~\ref{Q main lemma} for $\Add(\mu,1)$,
i.e., that
for all $x\in \ran([h_K])$, 
\begin{enumerate-(1)}
\item\label{Q main 1 again}
$x$ is $\Add(\mu,1)$-generic over $M$,
and
\item\label{Q main 2 again}
$M[K]$ is an $\Add(\mu,1)$-generic extension of $M[x]$.
\end{enumerate-(1)}

In order to show~\ref{Q main 2 again}, we will later calculate the quotient forcing for a given element of $\ran([h_K])$.
When doing so, it will be more convenient to work with certain dense subforcings $\QQ^{*}$ of $\QQ$.
These will be 
the intersections of ${<}\mu$ many subforcings
of the form $\QQ^*(\vartheta,\sigma)$ given in the next definition.
Regarding notation,
$\height(A)$, $\lh(A)$ were defined in Subsection~\ref{preliminaries: trees},
and $\vartheta_{(p)},\,\vartheta_{[p]}$ in Definition~\ref{T^sigma def}. 

\begin{definition}
\label{Q* def}
Let $\vartheta$ and $\sigma$ be $\QQ$-names such that 
$\one_\QQ\forces\vartheta\in{}^{\mu}\ddimq\wedge[\dot h](\vartheta)=\sigma$. 
\index{forcing!subforcing for names\idf$\QQ^*(\vartheta,\sigma)$} 
Then $\QQ^*(\vartheta,\sigma)$ denotes the subforcing of $\QQ$ which consists of all conditions $p\in \QQ$ such that the following hold:%
\footnote{%
Note that 
we can equivalently 
replace $\vartheta_{(p)}$ by $\vartheta_{[p]}$
and $\sigma_{(p)}$ by $\sigma_{[p]}$
everywhere in this definition, 
by Lemma~\ref{lemma: sigma_q and sigma[q]}.
}
\begin{enumerate-(a)}
\item\label{Q* def 1}
$\height(\dom p)\in \Lim$ and 
$\height(\dom p)\subseteq \dom(\vartheta_{(p)})$.
\item\label{Q* def 3}
$\lh(\ran p)\in \Lim$ and $\lh(\ran p)\subseteq \dom(\sigma_{(p)})$. 
\item\label{Q* def 2} 
$\vartheta_{(p)}\restr\gamma\in\dom(p)$ 
for all $\gamma<\height(\dom(p))$. 
\end{enumerate-(a)}
\end{definition}

Note that if $p\in \QQ^*(\vartheta,\sigma)$, then 
we get the equality $\lh(\ran p)= \dom(\sigma_{(p)})$
in \ref{Q* def 3}, 
by Lemma~\ref{dom sigma_p} below. 

We next show that $\QQ^*(\vartheta,\sigma)$ is a dense subforcing of $\QQ$ and
more generally, that
the intersection $\QQ^*$ of ${<}\mu$ many subforcings of this form is also dense in $\QQ$.
We will need the following lemma.

\begin{lemma}
\label{decreasing chains in Q}
Suppose $\langle p_\alpha:\alpha<\xi\rangle$ is a strictly decreasing sequence of elements of $\QQ$ of length $\xi<\mu$.
Let
\[
p:=\mbigcap\{p_\alpha:\alpha<\xi\}.
\]
Then 
$\height(\dom p)=\bigcup_{\alpha<\xi}\height(\dom p_\alpha)$
and
$\lh(\ran p)=\bigcup_{\alpha<\xi}\lh(\ran p_\alpha)$.
\end{lemma}
\begin{proof}
The first part of the conclusion holds since
$\dom(p)=\bigcup_{\alpha<\xi}\dom(p_\alpha)$ by definition.
Moreover,
$\lh(p(t))=\bigcup\{\lh(p_\alpha(t)):\:\alpha<\xi,\, t\in\dom(p_\alpha)\}$ for all $t\in\dom(p)$ by definition.
Therefore
$\lh(\ran p)=\bigcup\{\lh(p(t)):t\in\dom(p)\}=
\bigcup\{\lh(p_\alpha(t)):\:\alpha<\xi,\,t\in\dom(p_\alpha)\}
=\bigcup_{\alpha<\xi}\lh(\ran p_\alpha)$. 
\end{proof}

\begin{lemma}
\label{Q* dense}
Suppose that $\xi<\mu$. 
For all $i<\xi$,
let $\vartheta_i$ and $\sigma_i$
denote $\QQ$-names such that
$\one_\QQ\forces \vartheta_i\in{}^{\mu}\ddimq\wedge\sigma_i=[\dot h](\vartheta_i)$, 
and let
\begin{equation*}
\QQ^*:=\bigcap_{i<\xi}\QQ^*(\vartheta_i,\sigma_i).
\end{equation*}
Then
$\QQ^*$ 
is a dense subforcing of $\QQ$.
\end{lemma}

\begin{proof}
\todog{This proof works only when $\mu>\omega$}
Let $q\in \QQ$. Construct by recursion a strictly decreasing sequence 
$\langle p_n:n<\omega\rangle$ of conditions in $\QQ$ as follows.   
Let $p_0:=q$. 
If $n<\omega$ and $p_n$ has been defined, use Lemma~\ref{h_K domain sublemma}
and the fact that 
$\QQ$ is $\lle\mu$-closed to choose $p_{n+1}$ such that
\begin{enumerate-(i)}
\item\label{dense 1}
$p_{n+1}$ decides 
$\vartheta_i\restr\height(\dom p_n)$ and $\sigma_i\restr\lh(\ran p_n)$
for all $i<\xi$,
\item\label{dense 0}
$(\vartheta_i)_{[p_n]}\in\dom(p_{n+1})$ for all $i<\xi$,
\item \label{dense 2}
$\height(\dom p_{n+1})>\height(\dom p_n)$, and
\item \label{dense 3}
$\lh(\ran p_{n+1})>\lh(\ran p_n)$.
\end{enumerate-(i)}
Let $p:=\mbigcap\{p_n:n<\omega\}$. 
Then $p\in\QQ$ and $p\leq  q$ by Lemma~\ref{mbigcap facts}~\ref{mbigcap fact 4}.
To prove that $p\in\QQ^*$, 
let $i<\xi$.
We show that $p\in \QQ^*(\vartheta_i,\sigma_i)$.
The fact that 
$\height(\dom p)\in \Lim$ and $\lh(\ran p)\in \Lim$
follows from \ref{dense 2}, \ref{dense 3} and Lemma \ref{decreasing chains in Q}.
By \ref{dense 1},
$\height(\dom p_n)\subseteq \dom{(\vartheta_i)_{(p_{n+1})}}$  
and
$\lh(\ran p_n)\subseteq\dom{(\sigma_i)_{(p_{n+1})}}$
for all $n<\omega$. 
From this and Lemma~\ref{decreasing chains in Q}, we obtain:
\todog{we have to use $\lh(\ran p)$ instead of $\height(\ran p)$ in the def of $\QQ^{*}$ for this step to work}
\begin{equation*}
\begin{gathered}
\height(\dom p)=\bigcup_{n<\omega}\height(\dom p_n)
\subseteq
\bigcup_{n<\omega}\dom{(\vartheta_i)_{(p_{n+1})}}
\subseteq \dom{(\vartheta_i)_{(p)}}, 
\\
\lh(\ran p)=\bigcup_{n<\omega}\lh(\ran p_n)
\subseteq
\bigcup_{n<\omega}\dom{(\sigma_i)_{(p_{n+1})}}
\subseteq \dom{(\sigma_i)_{(p)}}.
\end{gathered}
\end{equation*}
Thus, 
Definition~\ref{Q* def}~\ref{Q* def 1} and~\ref{Q* def 3}
hold for
$\vartheta_i,\,\sigma_i$ and $p$.
To show that 
\ref{Q* def 2} also holds, let 
$\gamma<\height(\dom p)$ and take $n<\omega$ with
$\gamma\leq\height(\dom p_{n})$.
Then by \ref{dense 1},
$\gamma\subseteq\dom{(\vartheta_i)_{(p_{n+1})}}$
or equivalently, $\gamma\leq\lh{\left((\vartheta_i)_{[p_{n+1}]}\right)}$.
Let
\[
t
=(\vartheta_i)_{[p_{n+1}]}\restr\gamma
=(\vartheta_i)_{(p_{n+1})}\restr\gamma
.\]
Then 
$t\in\dom(p_{n+2})\subseteq\dom(p)$
by \ref{dense 0}.
We further have 
$t\subseteq
(\vartheta_i)_{(p_{n+1})}\subseteq 
(\vartheta_i)_{(p)}.$
Hence $(\vartheta_i)_{(p)}\restr\gamma=t\in\dom(p)$, as required.
\end{proof}

The next lemma
is a stronger form of the following observation: 
Let $p\in\QQ^*(\vartheta,\sigma)$, where $\vartheta$ and $\sigma$ are $\QQ$-names 
such that $\one_\QQ\forces\vartheta\in{}^{\mu}\ddimq\wedge[\dot h](\vartheta)=\sigma$. 
Let $t$ be the least initial segment of $\vartheta_{[p]}$ with $t\notin\dom(p)$.%
\footnote{$t$ exists since 
$\vartheta_{[p]}\notin\dom(p)$
if $p\in\QQ^*(\vartheta,\sigma)$
by Definition~\ref{Q* def}~\ref{Q* def 1} 
and Lemma~\ref{lemma: sigma_q and sigma[q]}.
For instance, if $\vartheta=\check x$ for some $x$ in~$({}^\mu\mu)^M$, then $\vartheta_{[p]}=x$.} 
Then 
$t=\vartheta_{(p)}\restr\height(\dom p)$.
Furthermore, let
$
v:=\bigcup\big\{p(t\restr\alpha):{\alpha<\height(\dom p)}\big\}.
$
For any choice of $u$ with $v\subsetneq u$,
$r:=p\cup\{(t,u)\}$ 
remains a strict partial $s$-function (i.e., $r\in\QQ$)
and moreover, 
$\sigma_{(r)}\supseteq u$.
Note that 
$\sigma_{(p)}=v$ by Lemma~\ref{dom sigma_p} below.
We will need the next lemma
for ${<}\mu$ many pairs of names $(\vartheta_i,\sigma_i)$.

\begin{lemma}
\label{extending elements of Q*}
Using the same assumptions as in Lemma~\ref{Q* dense}, 
suppose that $p\in\QQ^*$. 
Let $t_i :=(\vartheta_i)_{(p)}\restr\height(\dom p)$ for all $i<\xi$.
\begin{enumerate-(1)}
\item\label{eeq1}
$t_i$ is the least initial segment of $(\vartheta_i)_{[p]}$ with $t_i\notin\dom(p)$ for all $i<\xi$.
\item\label{eeq2}
If $t_i\neq t_j$, 
$u_i\in{}^{<\mu}\mu$ and
$\bigcup\big\{p(t_i\restr\alpha):{\alpha<\height(\dom p)}\big\}\subsetneq u_i$
for all $i<j<\xi$ 
and
$$r:= p\cup\{(t_i, u_i):i<\xi\},$$
then $r\in\QQ$,
$r\leq  p$ and
$r\forces (t_i\subseteq \vartheta_i\wedge u_i\subseteq \sigma_i)$ 
for all $i<\xi$.%
\footnote{In fact, $p\forces t_i\subseteq \vartheta_i$ since $t_i\subseteq(\vartheta_i)_{(p)}$.} 
\end{enumerate-(1)}
\end{lemma}
\begin{proof}
For \ref{eeq1},
note that 
$t_i\in{}^{<\mu}\ddimq$ 
since 
$p\in\QQ^*(\vartheta_i,\sigma_i)$,
and in particular, since we have 
$\height(\dom p)\subseteq \dom\big({(\vartheta_i)}_{(p)}\big)$
by definition.
The latter also implies $t_i\subseteq{(\vartheta_i)}_{[p]}$ by Lemma~\ref{lemma: sigma_q and sigma[q]}.
It thus suffices that $t_i\restr\gamma\in\dom(p)$ for all $\gamma<\height(\dom p)$, but this follows from $p\in\QQ^*(\vartheta_i,\sigma_i)$ by definition.

For \ref{eeq2}, we first show $r\in\QQ$.
Since $\xi<\mu$ and
$t_i\neq t_j$ for all $i<j<\xi$, 
$r$ is a partial function from ${}^{<\mu}\ddimq$ to ${}^{<\mu}\mu$ 
with $|\dom(r)|<\mu$.
$\dom(r)$ is a subtree of ${}^{<\mu}\ddimq$
by \ref{eeq1}.
Moreover, $r$ is a strict partial $s$-function by Lemma~\ref{def of the ordering equiv} and
our assumption about the values $u_i=r(t_i)$.
Therefore
$r\in\QQ$, as required.

It is clear that $r\leq  p$. 
If $i<\xi$, then
$r$ forces $t_i\subseteq\vartheta_i$ 
since $t_i\subseteq(\vartheta_i)_{(p)}\subseteq(\vartheta_i)_{(r)}$.
Therefore
$r\forces\dot h(t_i)\subseteq [\dot h](\vartheta_i)=\sigma_i$.
Since $r\forces \dot h\mleq r$ by Corollary~\ref{h_K is an s-function}~\ref{h_K in mpo},
$r$ also forces that $u_i\subseteq\dot h(t_i)$. Therefore 
$r\forces u_i\subseteq\sigma_i$.
\end{proof}

The next lemma states that for $p\in \QQ^*$, $\sigma_{(p)}$ can be calculated from $p$ and $\vartheta$. 

\begin{lemma}
\label{dom sigma_p}
If $p\in\QQ^*:=\QQ^*(\vartheta,\sigma)$
where
 $\vartheta$ and $\sigma$ are $\QQ$-names such that
$\one_\QQ\forces \vartheta\in{}^{\mu}\ddimq\wedge\sigma=[\dot h](\vartheta)$,
then
\todog{the second statement is used in Lemma \ref{P* is <kappa closed}}
\[
\sigma_{(p)}=\bigcup\big\{p(\vartheta_{(p)}\restr\alpha): \alpha<\height(\dom p)\big\}. 
\]
In particular, $\dom(\sigma_{(p)})=\lh(\ran p)$. 
\end{lemma}
\begin{proof}
Let $p\in\QQ^*$, 
and let 
$v:= \bigcup\big\{p(\vartheta_{(p)}\restr\alpha): \alpha<\height(\dom p)\big\}.$
We first show $v=\sigma_{(p)}$.
Since
$p\forces\dot h\mleq p$
by Lemma \ref{K and h_K} and
$p\forces\vartheta_{(p)}\subseteq \vartheta$,
\[p\forces p(\vartheta_{(p)}\restr\alpha)\subseteq\dot h(\vartheta_{(p)}\restr\alpha)\subseteq [\dot h](\vartheta)=\sigma.\]
for all $\alpha<\height(\dom p)$.
Thus
$p\forces v\subseteq\sigma$, so $v\subseteq\sigma_{(p)}$.
Suppose $v\subsetneq \sigma_{(p)}$. 
Take $u\in{}^{<\mu}\mu$ with $v\subsetneq u$
and  $u\perp \sigma_{(p)}$.
Let
$t:=(\vartheta_{(p)})\restr\height(\dom p)$ 
and $r:=p\cup\{(t,u)\}.$
Then $r\in\QQ$ and 
$r\forces u\subseteq \sigma$ 
by Lemma \ref{extending elements of Q*}. 
However, $r\leq  p$ and therefore $r\forces\sigma_{(p)}\subseteq\sigma$, a contradiction.

To show the second part of the conclusion, 
observe that 
$\sigma_{(p)}=v$  implies 
\[
\dom(\sigma_{(p)})=\lh(v)
=\bigcup\{\lh(p(\vartheta_{(p)}\restr\alpha):\:\alpha<\height(\dom p)\}
\subseteq\lh(\ran p).
\]
Conversely,
$\lh(\ran p)\subseteq\dom(\sigma_{(p)})$
follows from our assumption that $p\in\QQ^*$. 
\end{proof}

The next lemma will be used to prove the ${<}\mu$-closure of the quotient forcings for elements of $\ran([h_K])$. 

\begin{lemma}
\label{P* is <kappa closed}
Suppose that
$\QQ^*:=\QQ^*(\vartheta,\sigma),$
where
 $\vartheta$ and $\sigma$ are $\QQ$-names such that
$\one_\QQ\forces \vartheta\in{}^{\mu}\ddimq\wedge\sigma=[\dot h](\vartheta)$.
If
$\langle p_\alpha:\alpha<\xi\rangle$ is a strictly decreasing chain of elements of $\QQ^*$ of length $\xi<\mu$ and
\[
p:=\mbigcap\{p_\alpha:\alpha<\xi\},
\]
then $p\in\QQ^*$ and $\sigma_{(p)}=\bigcup_{\alpha<\xi}\sigma_{(p_\alpha)}$.
\end{lemma}
\begin{proof}
We first show that $p\in\QQ^*$.
Since $|\dom(p)|<\mu$,
we have $p\in\QQ$ by Lemma~\ref{mbigcap facts}.
The fact that $\height(\dom p)$ and $\lh(\ran p)$ are limit ordinals follows
from Lemma~\ref{decreasing chains in Q} in the case that $\xi\in\Lim$, 
and
from the fact that $p_\alpha\in\QQ^*$ in the case that $\xi=\alpha+1$.
\todog{we have to use $\lh(\ran p)$ instead of $\height(\ran p)$ in the def of $\QQ^{*}$ for this step to work}
Moreover,
\begin{equation*}
\begin{gathered}
\height(\dom p)=\bigcup_{\alpha<\xi}\height(\dom p_\alpha)
\subseteq
\bigcup_{\alpha<\xi}\dom(\vartheta_{(p_\alpha)})
\subseteq \dom(\vartheta_{(p)}), 
\\
\lh(\ran p)=\bigcup_{\alpha<\xi}\lh(\ran p_\alpha)
\subseteq
\bigcup_{\alpha<\xi}\dom(\sigma_{(p_\alpha)})
\subseteq \dom(\sigma_{(p)}) 
\end{gathered}
\end{equation*}
by
Lemma~\ref{decreasing chains in Q} 
and since
$p_\alpha\in\QQ^*$ and $p\leq  p_\alpha$ for all $\alpha<\xi$.
Thus, Definition~\ref{Q* def}~\ref{Q* def 1} and~\ref{Q* def 3} hold.
To show 
\ref{Q* def 2}, let $\gamma<\height(\dom p)$.
Then $\gamma<\height(\dom p_\alpha)$ for some $\alpha<\xi$. 
Since 
$\vartheta_{(p_\alpha)}\subseteq\vartheta_{(p)}$,
we have
$\vartheta_{(p)}\restr\gamma=\vartheta_{(p_\alpha)}\restr\gamma\in \dom(p_\alpha)\subseteq\dom(p).$

To show the second part of the conclusion, let $v:=\bigcup_{\alpha<\xi}\sigma_{(p_\alpha)}$.
Then $v\subseteq\sigma_{(p)}$. 
Moreover,
\[
\dom(\sigma_{(p)})=\lh(\ran p)=\bigcup_{\alpha<\xi}\lh(\ran p_\alpha)=
\bigcup_{\alpha<\xi}\dom(\sigma_{(p_\alpha)})=\dom v
\]
by the previous lemma
and Lemma~\ref{decreasing chains in Q}.
Hence $\sigma_{(p)}=v$.
\end{proof}

The next lemma proves the first part of the Nice Quotient Lemma~\ref{Q main lemma} for the forcing $\Add(\mu,1)$. 

\begin{lemma}
\label{branches are Add(kappa,1) generic}
Suppose $K$ is $\QQ$-generic over $M$ and
$\langle x_i:i<\xi\rangle$ is a sequence of distinct
elements of $({}^{\mu}\ddimq)^{M[K]}$ of length $\xi<\mu$.
Then 
$\prod_{i<\xi}[h_K](x_i)$
is $\Add(\mu,\xi)$-generic over $M$.\footnote{$[h_K](x_i)$ is identified with the filter of its initial segments in $\Add(\mu,1)$.}
\end{lemma}
\begin{proof}
Let $i<\xi$. 
By basic facts about forcing, 
there exist $\QQ$-names $\vartheta_i$ and $\sigma_i$  such that
$(\vartheta_i)^K=x_i$ 
and 
$\one_\QQ\forces\vartheta_i\in{}^\mu\ddimq\wedge [\dot{h}](\vartheta_i)=\sigma_i$.
In more detail, let $\vartheta'_i$ be a $\QQ$-name such that $(\vartheta'_i)^K=x_i$ and let 
\begin{align*}
\vartheta_i&:=\big\{\big((\alpha,\beta),q\big):\: 
(\alpha,\beta)\in\mu\times\ddimq
,\ \, 
q\forces\vartheta'_i\in{}^\mu\ddimq\wedge(\alpha,\beta)\in\vartheta'_i
\big\},
\\
\sigma_i&:=\big\{\big((\alpha,\beta),q\big):\: 
(\alpha,\beta)\in\mu\times\mu
,\ \, 
q\forces(\alpha,\beta)\in[\dot h](\vartheta_i)
\big\}. 
\end{align*}

Since $x_i\neq x_j$ for all $i<j<\xi$, there exists an ordinal $\alpha<\mu$ and a condition $q\in K$ such that 
$q\forces\vartheta_i\restr\alpha\neq \vartheta_j\restr\alpha$
 for all $i<j<\xi$.
We may also assume $\height(\dom q)\geq\alpha$, since
$\{p\in\QQ:\height(\dom p)\geq\alpha\}$ is dense in $\QQ$ by Lemma \ref{h_K domain sublemma}.

\begin{claim*}
\label{branches are Add(kappa,1) generic subclaim}
If $D$ is a dense subset of $\Add(\mu,\xi)$, then
$$
D':=\big\{r\in\QQ\::\: {\exists u\in D}\ \ {\forall i<\xi} \ \ {r\forces u(i)\subseteq\sigma_i}\big\}
$$
is dense below $q$.
\end{claim*}
\begin{proof}
Suppose that $q'\in\QQ$ and $q'\leq  q$. 
Take some $p\leq  q'$ with $p\in\bigcap_{i<\xi}\QQ^*(\vartheta_i,\sigma_i)$ by Lemma~\ref{Q* dense}. 
For all $i<\xi$, choose $t_i\in{}^{<\mu}\ddimq$ and $u_i\in{}^{<\mu}\mu$ as in Lemma~\ref{extending elements of Q*}. That is, let
$t_i:=(\vartheta_i)_{(p)}\restr\height(\dom p)$ and
choose $u_i$ so that
$u_i\supsetneq \bigcup\big\{p(t_i\restr\alpha)\::\:{\alpha<\height(\dom p)}\big\}.$
In addition, we may assume $\langle u_i : i<\xi\rangle\in D$. 
Since $p\leq  q$, we have $t_i\neq t_j$ for all $i<j<\xi$.
Thus $r:=p\cup\{(t_i,u_i):i<\xi\}$ is a condition in $\QQ$ below $p$ which forces $u_i\subseteq \sigma_i$ for all $i<\xi$ by Lemma \ref{extending elements of Q*}. 
Hence $r\leq  q'$ and $r\in D'$ as required. 
\end{proof}
To see that $G':=\prod_{i<\xi}[h_K](x_i)$ is $\Add(\mu,\xi)$-generic,  
let $D$ be a dense subset of $\Add(\mu,\xi)$.
Then $D'\cap K\neq\emptyset$
by the previous claim.
Hence there exists $u\in D$ 
with $u(i)\subseteq (\sigma_i)^K=[h_K](x_i)$ 
for all $i<\xi$, i.e., with
$u\in G'$.
This completes the proof of Lemma~\ref{branches are Add(kappa,1) generic}.
\end{proof}

\begin{remark}
\label{the f constructed is injective remark 2}
The previous lemma
shows that the special case of the Nice Quotient Lemma for $\Add(\mu,1)$ holds with the following stronger $\xi$-dimensional variant  
of~\ref{Q main 1}: 
\begin{quotation}\noindent
If $\langle x_i:i<\xi\rangle$ is a sequence of distinct
elements of $({}^{\mu}\ddimq)^{M[K]}$ of length $\xi<\mu$,
then 
$\prod_{i<\xi}[h_K](x_i)$
is $\Add(\mu,\xi)$-generic over $M$.
\end{quotation}
\end{remark}

In the rest of this subsection, we will prove the second part of the Nice Quotient Lemma for $\Add(\mu,1)$.
We need to show that 
the quotient forcing 
in a $\QQ$-generic extension $M[K]$ of $M$
for an element $\sigma^K$ 
of $\ran([h_K])$ 
is equivalent to $\Add(\mu,1)$. 
The precise statement is given in Lemma~\ref{quotient forcing of sigma 2} below. 

For this purpose, it will be more convenient to work 
with the Boolean completion of $\QQ$.
We therefore use the next lemma. 

\begin{lemma}
\label{Q separative}
$\QQ$ is separative.
\end{lemma}
\begin{proof}
Suppose $p,q\in\QQ$ and $p\not\leq  q$. 
We aim to find $r\in\QQ$ with $r\leq p$ and $r\perp  q$.
We may assume $p\compat q$. 
First, suppose $q(t)\not\subseteq p(t)$ 
for some $t\in\dom(p)\cap\dom(q)$. 
Then $p(t)\subsetneq q(t)$ and $t$ is a terminal node of $\dom(p)$ by Lemma~\ref{perp_Q claim}~\ref{comp 1} and \ref{comp 2}. 
Take $v\in{}^{<\mu}\mu$ with $p(t)\subsetneq v$ and $v\perp  q(t)$.
Let $r$ be the element of $\QQ$
obtained from $p$ by extending the value of $t$ to $v$.
That is,
let $\dom(r):=\dom(p)$, and for all $u\in\dom(r)$, let
\[
r(u):=
\begin{cases}
v &\text{ if $u=t$,}
\\
p(u) &\text{ if $u\neq t$}.
\end{cases}
\]
Then 
$r< p$, and $r\perp  q$ by Lemma~\ref{perp_Q claim}~\ref{comp 1}.

Now, suppose $q(t)\subseteq p(t)$ for all $t\in\dom(p)\cap\dom(q)$.
Then $q(t)=p(t)$  
for all $t\in \dom(p)\cap \dom(q)$ such that $t$ is non-terminal in $\dom(q)$ 
by Lemma~\ref{perp_Q claim}~\ref{comp 3}. 
Thus, since $p\not\leq  q$, we must have $\dom(q)\not\subseteq\dom(p)$.
Take $t\in\dom(q)$ of minimal length with $t\notin\dom(p)$, and let
\[
u:= \bigcup_{\alpha< \lh(t)} s\big(p(t\restr\alpha), t(\alpha) \big)=
 \bigcup_{\alpha< \lh(t)} s\big(q(t\restr\alpha), t(\alpha) \big). 
\]
The second equality holds since 
for all $\alpha<\lh(t)$, 
$t\restr\alpha$ is a non-terminal node of $\dom(q)$ with $t\restr\alpha\in\dom(p)$,
and hence $p(t\restr\alpha)=q(t\restr\alpha)$.
Then $u\subsetneq q(t)$ holds by Lemma~\ref{def of the ordering equiv} since $q\in \QQ$. 
Take $v\in{}^{<\mu}\mu$ with $u\subsetneq v$ and $v\perp q(t)$,
and let $r:=p\cup\{(t,v)\}$. 
Then $r\in\QQ$ by Lemma~\ref{def of the ordering equiv}, and clearly $r< p$. 
We further have  $r\perp  q$ by Lemma~\ref{perp_Q claim}~\ref{comp 1}. 
\end{proof}

Recall 
that $\BB(\QQ)$ denotes the Boolean completion of $\QQ$. We may assume that 
$\QQ$ is a dense subset of $\BB(\QQ)$, since $\QQ$ is separative.
Write $\BB:=\BB(\QQ)$. 
Let $\bigwedgepo\BB$ and $\bigveepo\BB$ denote the infimum and supremum in $\BB$.%
\footnote{Recall that $\bigwedgepo\BB K$ agrees with $\mbigcap K$ for downward directed subsets $K$ of $\QQ$ of size $\lle\mu$ by Remark~\ref{remark: infima in BB and mpo}.}
Write
$\boolval \varphi:=\boolvalpo \varphi \BB$ for all formulas $\varphi$ in the forcing language.
Assume $\vartheta$ and $\sigma$ denote $\BB$-names such that 
$\boolval{\vartheta\in{}^\mu\ddimq \wedge [\dot{h}](\vartheta)=\sigma} =\one_\BB$, and let $\QQ^*:=\QQ^*(\vartheta,\sigma)$.
Note that $\QQ^*$ is a dense subforcing of $\BB$ by Lemma \ref{Q* dense}.

\begin{lemma} 
\label{u subset sigma}
For all $u,v\in{}^{<\mu}\mu$: 
\begin{enumerate-(1)} 
\item 
\label{u subset sigma 1}
If $u\subseteq v$, then $\boolval{u\subseteq\sigma}\geq\boolval{v\subseteq\sigma}$. 
\item 
\label{u subset sigma 2}
If $\boolval{u\subseteq\sigma}>\boolval{v\subseteq\sigma}$, then $u\subsetneq v$. 
\end{enumerate-(1)} 
\end{lemma}
\begin{proof} 
\ref{u subset sigma 1} is immediate. 
For 
\ref{u subset sigma 2}, suppose $\boolval{u\subseteq\sigma}>\boolval{v\subseteq\sigma}$. 
Then $u\compat v$ since otherwise $\boolval{u\subseteq\sigma}\perp\boolval{v\subseteq\sigma}$. 
If $v\subseteq u$, then $\boolval{v\subseteq\sigma}\leq\boolval{u\subseteq\sigma}$ by \ref{u subset sigma 1}. 
Therefore $u\subsetneq v$. 
\end{proof}

Recall from Definition \ref{def: generated Boolean subalgebra} 
that
$\BB(\sigma)=\BB^\QQ(\sigma)$ denotes the complete Boolean subalgebra of $\BB$ that is completely generated by $\{\boolval{(\alpha,\beta)\in\sigma}:\:\alpha,\beta<\mu\}$. 
\label{mmap def}
Note that $\sigma_{(p)}\in{}^{<\mu}\mu$ for all $p\in\QQ^*$ 
by Lemma~\ref{dom sigma_p}. 
Define the map $\mmap:\QQ^*\to\BB(\sigma)$ by letting
$$
\mmap(p):=\boolval{\sigma_{(p)}\subseteq\sigma}%
$$
for all $p\in \QQ^*$.%
\footnote{$\mmap(p)\in\BB(\sigma)$ since $\mmap(p)=\bigwedgepo\BB\big\{\boolval{(\alpha,\beta)\in\sigma}:(\alpha,\beta)\in\sigma_{(p)}\big\}$
and $\sigma_{(p)}\in{}^{<\mu}\mu$.}
Let $\RR^*:=\ran(\mmap)=\{\mmap(q):q\in\QQ^*\}$.

\begin{lemma}
\label{mmap obs}
For all $p,q\in\QQ^*$:
\begin{enumerate-(1)}
\item\label{mmap obs 1} $q\leq \mmap(q)$.
\item\label{mmap obs 2} If $p\leq q$, then $\mmap(p)\leq\mmap(q)$.
\end{enumerate-(1)}
\end{lemma}
\begin{proof}
\ref{mmap obs 1} holds
since $q\forces \sigma_{(q)}\subseteq \sigma$.
For~\ref{mmap obs 2}, suppose $p\leq q$. Then $\sigma_{(p)}\supseteq \sigma_{(q)}$ and hence $\mmap(p)\leq\mmap(q)$ by 
the previous lemma.
\end{proof}

\begin{lemma} \
\label{R* complete subforcing} 
\begin{enumerate-(1)} 
\item\label{mmap reduction}
If $q\in\QQ^*$, then $\mmap(q)$ is a reduction\footnote{See Definition \ref{def: complete embedding}.} of $q$ to $\RR^*$ with respect to the inclusion $\id_{\RR^*}: \RR^*\to \BB$. 
\item\label{R* csf}
$\RR^*$ is a complete subforcing of $\BB$.
\end{enumerate-(1)}
\end{lemma}
\begin{proof}
For \ref{mmap reduction}, suppose $r\in\RR^*$ and $r\leq \mmap(q)$. 
We aim to show $r\compat q$. 
We may assume $r< \mmap(q)$ by the previous lemma.
Take $p\in\QQ^*$ with $r=\mmap(p)$. 
Then
\[
\boolval{\sigma_{(p)}\subseteq\sigma}=\mmap(p)<\mmap(q)=\boolval{\sigma_{(q)}\subseteq\sigma}.
\]
Therefore $\sigma_{(p)}\supsetneq \sigma_{(q)}$ by Lemma \ref{u subset sigma}.
Let $t:=\vartheta_{(q)}\restr\height(\dom q)$ and take $u\in{}^{<\mu}\mu$ with  $u\supseteq\sigma_{(p)}$.
Then 
\[u\supsetneq \sigma_{(q)}=\bigcup\big\{q(\vartheta_{(q)}\restr\alpha): \alpha<\height(\dom q)\big\} 
\] 
by Lemma~\ref{dom sigma_p}.
Thus $q$, $t$ and $u$ satisfy the requirements of Lemma~\ref{extending elements of Q*}, and  hence
$\bar q:=q\cup\{(t,u)\}\in\QQ$, $\bar q\leq q$ and 
$\bar q\forces u\subseteq \sigma$.
Since $u\supseteq \sigma_{(p)}$, we have 
\[
\bar q\leq\boolval{u\subseteq\sigma}\leq 
\boolval{\sigma_{(p)}\subseteq\sigma}=\mmap(p)=r.
\]
This shows that $q$ and $r$ are compatible in $\QQ$ and hence in $\BB$.

For \ref{R* csf}, it suffices to show that every $b\in\BB$ has a reduction to $\RR^*$.
Given $b\in \BB$, take $q\in\QQ^*$ with $q\leq b$. 
Since $\mmap(q)$ is a reduction of $q$ to $\RR^{*}$, 
$\mmap(q)$ is also a reduction of $b$ to~$\RR^{*}$.
\end{proof}

Let $\BB(\RR^*)$ be the Boolean completion of $\RR^{*}$.
Since $\RR^*$ is a complete subforcing of $\BB$
and therefore separative, 
we may assume 
that $\RR^{*}$ is dense in $\BB(\RR^{*})$ 
and that $\BB(\RR^*)$ is a complete Boolean subalgebra of $\BB$.%
\footnote{See Fact~\ref{fact}~\ref{fact: complete embeddings and Boolean completions}.}
It is easy to see 
that $\BB(\RR^*)$ is then equal to the complete Boolean subalgebra of $\BB$ that is completely generated by $\RR^*$. 

\begin{lemma}
\label{B(R*)=B(sigma)}
$\BB(\RR^*)=\BB(\sigma)$.
\end{lemma}
\begin{proof}
$\BB(\RR^*)\subseteq\BB(\sigma)$
since $\RR^*\subseteq\BB(\sigma)$ by definition.
For $\BB(\sigma)\subseteq\BB(\RR^*)$, 
it suffices to show
$\boolval{(\alpha,\beta)\in\sigma}\in\BB(\RR^*)$ for all $\alpha,\beta\in\mu$.
Suppose $\alpha,\beta\in\mu$ and let
$$a:=\bigveepo\BB\{\mmap(q):q\in\QQ^{*},\,(\alpha,\beta)\in\sigma_{(q)}\}.$$
Then $a\in\BB(\RR^{*})$.
We claim that $\boolval{(\alpha,\beta)\in\sigma}=a$.
Since $\QQ^*$ is dense in $\BB$ and $q\leq \pi^*(q)$ for all~$q\in\QQ^*$, 
\begin{align*}
\boolval{(\alpha,\beta)\in\sigma} 
&=\bigveepo\BB\{q\in\QQ^*:\: q\leq\boolval{(\alpha,\beta)\in\sigma}\} \\
&=\bigveepo\BB\{q\in\QQ^*:\: (\alpha,\beta)\in\sigma_{(q)}\}  \\
& \leq  \bigveepo\BB\{\mmap(q):\: q\in\QQ^*,\, (\alpha,\beta)\in\sigma_{(q)}\}=a.
\end{align*}
On the other hand, if $q\in\QQ^*$ and $(\alpha,\beta)\in\sigma_{(q)}$, then
$\mmap(q)=\boolval{\sigma_{(q)}\subseteq\sigma}\leq\boolval{(\alpha,\beta)\in\sigma}.$
Hence 
$a\leq \boolval{(\alpha,\beta)\in\sigma}$.
\end{proof}

Let $\pi:\BB\to \BB(\sigma)$ be the retraction associated to the inclusion map $\id_{\BB(\sigma)}:\BB(\sigma)\to\BB$. That is, let 
$
\pi(b):=\bigwedgepo{\BB}\{a\in\BB(\sigma) : a\geq b\}
$
for all $b\in\BB$.%
\footnote{See Definition~\ref{def: retraction}.}

\begin{lemma}
\label{mmap natural projection}
$\pi(p)=\mmap(p)$
for all $p\in\QQ^*$.
\todog{Equivalently: for all $p\in\QQ^*$, $\boolval{p\in\BB/\BB(\sigma)} = \boolval{\sigma_{(b)}\subseteq\sigma.}$}
\end{lemma}
\begin{proof}
Let $p\in\QQ^*$. 
Then $\pi(p)\leq\mmap(p)$
since $\mmap(p)\in\BB(\RR^*)=\BB(\sigma)$ 
and $p\leq\mmap(p)$. 
Moreover, $\mmap(p)\leq\pi(p)$, since
$\mmap(p)$ is a reduction of $p$ to $\RR^{*}$ 
\todog{I think we use that $\RR^*$ is dense in $\BB(\RR^*)$ here}
and hence also to $\BB(\RR^*)=\BB(\sigma)$
and
$\pi(p)$ is the largest reduction of $p$ to $\BB(\sigma)$ by Lemma~\ref{retraction facts}.
\end{proof}

The next lemma proves the second part of the Nice Quotient Lemma for $\Add(\mu,1)$. 

\begin{lemma}
\label{quotient forcing of sigma 2}
${\BB(\sigma)}$ forces that the quotient forcing 
$\BB/\BB(\sigma)$ is equivalent to the forcing $\Add(\mu,1)$.
\end{lemma}
\begin{proof}
Suppose $J$ is a $\BB(\sigma)$-generic filter over $M$. 
It suffices to show that in $M[J]$, 
the forcing $\BB/J$ is non-atomic and
has a dense subforcing of size ${\leq}\mu$ that is $\lle\mu$-closed,
by Lemma~\ref{forcing equivalent to Add(kappa,1)}. 
We use the standard facts about quotient forcings in Fact~\ref{fact}~\ref{fact: generic extensions wrt quotient forcings}. 

To see that $\BB/J$ is non-atomic in $M[J]$, it suffices to show that  $M[J][K]\neq M[J]$ for all $\BB/J$-generic filters $K$ over $M[J]$.
Let $K$ be $\BB/J$-generic over $M[J]$.
Then $K$ is also $\BB$-generic over $M$, $M[J][K]=M[K]$ and  $J=K\cap\BB(\sigma)$. 
Then $M\big[\sigma^K\big]=M[J]$ by
Lemma~\ref{fact: generated Boolean subalgebra}. 
Take $y,z\in ({}^\mu\ddimq)^{M[K]}$ with 
$\sigma^K=
[h_{K\cap \QQ}](y)$ and $y\neq z$.%
\footnote{Recall the notation $h_{K\cap \QQ}$ from page~\pageref{h_K def} after Lemma~\ref{Q equivalent to Add(kappa,1)}.}
Then $[h_{K\cap \QQ}](z)$ 
is $\Add(\mu,1)$-generic over $M\big[\sigma^K\big]$ 
by Lemma~\ref{branches are Add(kappa,1) generic}. 
Since $[h_{K\cap \QQ}](z)\in M[K]$, 
we have $M[K]\neq M\big[\sigma^K\big]$ as required.

In $M[J]$,
let
$\PP^*:=\QQ^{*}\cap (\BB/J).$
Note that $\PP^*$ is a dense subforcing of $\BB/J$
since $\QQ^*$ is a dense subforcing of $\BB$, 
and $|\PP^*|\leq|\QQ^*|\leq|\QQ|\leq\mu$.
To prove that
$\PP^*$ is $\lle\mu$-closed, 
let $K$ be $\BB/J$-generic over $M[J]$.
Then $K$ is $\BB$-generic over $M$
and $J=K\cap\BB(\sigma)$.
It suffices  
to show that
$$\PP^*=\{q\in\QQ^*:\sigma_{(q)}\subseteq\sigma^K\},$$
since the right hand side is $\lle\mu$-closed 
by Lemma~\ref{P* is <kappa closed}.
Let $q\in\QQ^*$.
We then have 
$q\in\PP^*
\Longleftrightarrow 
\pi(q)\in J
\Longleftrightarrow 
\pi(q)\in K$
by Lemma~\ref{retraction facts}~\ref{rf2}.
Since $\pi(q)=\mmap(q)=\boolval{\sigma_{(q)}\subseteq\sigma}$
by Lemma~\ref{mmap natural projection}, 
$\pi(q)\in K
\Longleftrightarrow 
\sigma_{(q)}\subseteq\sigma^K$, as required.
\end{proof}

If $K$ is a $\BB$-generic filter over $M$,
$x\in\ran([h_{K\cap\QQ}])$,
and $\sigma$ is a $\BB$-name for $x$ with $\boolval{\sigma\in\ran([\dot h])}=\one_\BB$%
\footnote{Recall the notation $\dot h$ from page~\pageref{dot h def}.
Note that all $x\in\ran([h_{K\cap\QQ}])$ are of this form, as in the proof of Lemma~\ref{branches are Add(kappa,1) generic}.
}
then $M[K]$ is an $\Add(\kappa,1)$-generic extension of 
$M[K\cap\BB(\sigma)]=M[x]$
by the previous lemma. 
This completes the proof of the Nice Quotient Lemma~\ref{Q main lemma} for $\Add(\mu,1)$. 

\subsection{A global version} 
\label{subsection global}

Recall that a class forcing extension is called \emph{tame} 
\index{tame class forcing extension\idf}
if it is a model of $\ZFC$. 

\begin{theorem} 
\label{theorem: global version}
Assume there exists a proper class of inaccessible cardinals. 
Then there is a tame class forcing extension of $V$ 
in which 
$\ODD\kappa\ddim(\defsetsk,\defsetsk)$
holds for all regular infinite cardinals $\kappa$ and all $\ddim$ with $2\leq\ddim\leq\kappa$.%
\footnote{Note that this implies $\ODD\kappa\ddim(\defsetsk)$ for all $\ddim$ with $2\leq\ddim<\kappa$ by Lemma~\ref{<kappa dim hypergraphs are definable}.}
\end{theorem} 
\begin{proof} 
Let $\bar{\kappa}:=\langle \kappa_{\alpha}: \alpha\in \Ord\rangle$ denote the order preserving enumeration of the closure $C$ of the class of inaccessible cardinals and $\omega$. 
We define the following Easton support iteration $\bar{\PP}=\langle \PP_{\alpha},\dot{\PP}_{\alpha}: \alpha\in \Ord\rangle$ with bounded support at regular limits and full support at singular limits. 
Let $\dot{\nu}_\alpha$ be a $\PP_\alpha$-name for the least regular cardinal $\nu\geq \kappa_\alpha$ 
that is not collapsed by $\PP_\alpha$, and let $\dot{\PP}_{\alpha}$ be a $\PP_{\alpha}$-name for $\Col(\dot{\nu}_\alpha,{{<}\kappa_{\alpha+1}})$. 
We may assume that the names $\dot{\PP}_\alpha$ are chosen in a canonical way so that the iteration is definable, for instance by working with nice names for elements of the forcing. 
Let $\PP$ denote the class forcing defined by $\bar{\PP}$. 
Moreover, let $\dot{\PP}^{(\alpha)}$ denote a $\PP_\alpha$-name for the tail forcing of the iteration at stage $\alpha$. 
It follows from the definition of the iteration that $|\PP_\alpha|<\kappa_{\alpha+1}$ and $\one_{\PP_\alpha} \forces \dot{\PP}^{(\alpha)}$ is ${<}\kappa_\alpha$-closed, for all $\alpha\in\Ord$. 

Now suppose that $G$ is $\PP$-generic over $V$. 
We write $G_\alpha:=G\cap \PP_\alpha$ and $\PP^{(\alpha)}:=(\dot{\PP}^{(\alpha)})^{G_\alpha}$ for $\alpha\in\Ord$. Moreover, let $\nu_\alpha:=\dot{\nu}^{G_\alpha}$ for $\alpha\geq 1$ and write $\bar{\nu}:=\langle \nu_\alpha: \alpha\geq 1\rangle$. 
\todon{pagebreak inserted}

\pagebreak

\begin{claim*} \ 
\begin{enumerate-(a)} 
\item 
\label{global case claim 1} 
If $\kappa_\alpha$ is inaccessible in $V$, then $\kappa_\alpha$ remains a regular cardinal in $V[G]$, $\nu_\alpha=\kappa_\alpha$ and ${(\kappa_\alpha^+)}^{V[G]}=\kappa_{\alpha+1}$. 
\item 
If $\kappa_\alpha$ is a singular limit in $V$, then 
${(\kappa_\alpha^+)}^{V[G]}=\nu_\alpha$. 
\end{enumerate-(a)} 
\end{claim*} 
\begin{proof} 
If $\kappa_\alpha$ is inaccessible in $V$, it follows from the $\Delta$-system lemma that $\PP_\alpha$ has the $\kappa_\alpha$-c.c.
\todog{Dori's note to self: See the proof of Theorem 10.17 in Kanamori's book. This also follows from \cite{MR1940513}*{Theorems 16.4 and 16.30}.}
The claims in \ref{global case claim 1} easily follow from this and the fact that $\PP^{(\alpha+1)}$ is ${<}\kappa_{\alpha+1}$-closed. 
If $\kappa_\alpha$ is a singular limit in $V$, then 
$\nu_\alpha$ is the least regular cardinal strictly above $\kappa_\alpha$ in $V[G_\alpha]$ by the definition of $\dot{\nu}_\alpha$. 
Moreover, $\nu_\alpha$ is not collapsed in $V[G]$, since $\PP^{(\alpha+1)}$ is ${<}\kappa_{\alpha+1}$-closed. 
\end{proof} 

By the previous claim, $\bar{\nu}$ enumerates the class of regular infinite cardinals in $V[G]$. 
Thus, suppose that $\alpha\geq 1$, $\kappa=\nu_\alpha$ and $2\leq \ddim\leq\kappa$.
In $V[G]$, 
let $X\in\defsetsk$ be a subset of $\kk$, and 
let $H\in\defsetsk$ be a box-open $\ddim$-dihypergraph on $\kk$.

\begin{claim*} 
$\ODD\kappa{H\restr X}$ holds in $V[G]$. 
\end{claim*} 
\begin{proof} 
Since $\one_{\PP_{\beta}}$ 
forces that 
$\dot{\PP}_{\beta}$ is homogeneous for all $\beta\in\Ord$, the tail forcing $\PP^{(\beta)}$ is homogeneous for all $\beta\in\Ord$. 

We first claim that $X,H\in V[G_{\alpha+1}]$. 
Note that $\PP^{(\alpha+1)}$ does not add new elements of ${}^{\kappa}\Ord$ since $\PP^{(\alpha+1)}$ is ${<}\kappa_{\alpha+1}$-closed. 
Suppose that $X=X_{\varphi,a}^{V[G]}$ where $\varphi$ is a formula and $a\in {}^\kappa\Ord$. 
Since  $\PP^{(\alpha+1)}$ is homogeneous, $X= X_{\psi,a}^{V[G_{\alpha+1}]}$ where $\psi(x,a)$ denotes the formula $\one_{\PP^{(\alpha+1)}}\forces^{V[G_{\alpha+1}]}\varphi(x,a)$. 
Hence $X\in V[G_{\alpha+1}]$, and similarly $H\in V[G_{\alpha+1}]$. 

Since $\kappa=\nu_\alpha$, $\dot{\PP}_\alpha$ is a name for $\Col(\kappa,\lle\kappa_{\alpha+1})$ and hence $\ODD\kappa{H\restr X}$ holds in $V[G_{\alpha+1}]$ 
by Theorem~\ref{main theorem}~\ref{mainthm ddim=kappa}. 
Since $\PP^{(\alpha+1)}$ is $\lle\kappa_{\alpha+1}$-closed, a $\kappa$-coloring of $H\restr X$ in $V[G_{\alpha+1}]$ remains a $\kappa$-coloring in $V[G]$, and a continuous homomorphism from $\dhH\ddim$ to $H\restr X$ in $V[G_{\alpha+1}]$ remains one in $V[G]$. 
Hence $\ODD\kappa{H\restr X}$ holds in $V[G]$. 
\end{proof} 

$\ODD\kappa\ddim(\defsetsk,\defsetsk)$ 
holds in $V[G]$ by the previous claim. 
\end{proof} 

It is open whether a global version of Theorem~\ref{main theorem} \ref{mainthm Mahlo} holds in analogy with the previous theorem. 
The problem is that in an iterated forcing extension as above, new box-open dihypergraphs are added by the tail forcings. 

\begin{remark} 
In the model of the previous theorem, we actually have $\ODDI\kappa\ddim(\defsets\kappa,\defsets\kappa)$ 
for all regular uncountable $\kappa$ and $d\leq\kappa$  
by Remark \ref{the f constructed is injective remark 1}.%
\footnote{See Remark~\ref{the f constructed is injective remark 1} fot the definition of $\ODDInp$. For $\ddim<\kappa$, we have $\ODDI\kappa\ddim(\defsets\kappa)$.
In fact, the continuous homomorphism can even be chosen to be a homeomorphism onto a closed set in this case by
Definition~\ref{def: ODD variants} and Corollary~\ref{main theorem strong version} below.}
\end{remark}

\newpage 

\section{Variants} 
\label{section: stronger forms of ODD}

We begin with some preliminary results 
regarding conditions under which 
the map $[\iota]$ is closed and the tree $T(\iota)$ is sufficiently closed for a given strict order preserving map~$\iota$ in Subsection~\ref{subsection: closure properties}. 
We show in Subsection~\ref{subsection: OGD and ODD} 
that the open graph dichotomy 
$\OGD_\kappa(X)$
is equivalent to $\ODD\kappa 2(X)$. 
In fact, we prove a more general version of this equivalence for an extension of $\OGD_\kappa(X)$ to $\ddim$-hypergraphs of higher dimensions $2\leq\ddim\leq\kappa$. 

In the remainder of the section,
\todol{introduction to section changed}
we show that the definition of $\ODD\kappa H$ is optimal in several ways.
It is shown to be nontrivial in Subsection~\ref{subsection: examples} in the sense that for any non-$\kappa$-colorable box-open dihypergraph $H$ on a subset of ${}^\kappa\kappa$, $\ODD\kappa {H\restr X}$ fails for some set $X$. 
It is then shown that 
$\ODD\kappa H$ fails for some product-closed\footnote{I.e., $H$ is a relatively closed subset of $\dhK\ddim {{}^\kappa\kappa}$, where ${}^\ddim {({}^\kappa\kappa)}$ has the product topology.}
hypergraph $H$ on ${}^\kappa\kappa$ 
in any dimension $\ddim\geq 2$.
$\dhH\ddim$ witnesses that for any $\ddim\geq 3$, the existence of a continuous homomorphism cannot be replaced by the existence of a large complete subhypergraph.
Stronger counterexamples are discussed 
in Subsection~\ref{subsection: examples}. 
%
One may further wonder whether for relatively box-open\footnote{I.e., $H$ is box-open on its domain $\domh H$.}
$\ddim$-dihypergraphs $H$ on ${}^\kappa\kappa$, 
the continuous homomorphism provided by $\ODD\kappa H$ can be chosen 
to have additional properties such as injectivity. 
Recall that
in the uncountable case,
 the version of $\ODD\kappa\ddim(\defsets\kappa)$ with an injective homomorphism
holds after a L\'evy collapse of a Mahlo cardinal to $\kappa^+$ and the version for definable dihypergraphs holds if the cardinal is inaccessible.\footnote{See Remark~\ref{the f constructed is injective remark 1}.}
We show in Subsection~\ref{subsection: ODD ODDI ODDH} 
that more generally, 
$\ODD\kappa H$ is equivalent to its variant with an injective homomorphisms
assuming
a weak version $\Diamondi\kappa$ of $\Diamond_\kappa$ in the uncountable case.
However, the version of $\ODD\omega H$  with an injective homomorphism fails in general in the countable case.
For $\ddim<\kappa$,
the continuous homomorphism in $\ODD\kappa H$ can 
be chosen to be a homeomorphism onto a closed subset
even without assuming $\Diamondi\kappa$. 
%
%
For $\ddim=\kappa$, one cannot in general find a homomorphism that is a homeomorphism onto its image.
However, this is always possible
assuming all hyperedges of $H$ are injective,
and the image can be chosen to be 
{a relatively closed subset 
of a set $Y\subseteq{}^\kappa\kappa$}
under additional assumptions on $H$ and $Y$.
%
The latter result will be relevant for versions of the Hurewicz dichotomy studied in Subsection~\ref{subsection: Hurewicz} below.
Subsection~\ref{subsection: diamondi} studies the principle $\Diamondi\kappa$
in preparation for the next subsection. 

\subsection{Closure properties of maps}
\label{subsection: closure properties}

Let $Y$ and $Z$ be topological spaces.
Recall that a map $f:Y\to Z$ is called \emph{closed} \index{closed map\idf}
if $f(X)$ is a closed subset of $Z$ for any closed subset $X$ of $Y$. 
It is easy to see that 
\todol{slight correction}
$f$ is closed if and only if $\closure{f(X)}\subseteq f(\closure{X})$ for all subsets $X$ of $Y$.%
\footnote{Hence a continuous map $f:Y\to Z$ is closed if and only if $\closure{f(X)}=f(\closure{X})$ for all subsets $X$ of $Y$.} 
\label{Lim^iota_t def}
For any map 
$f: {}^{\kappa}\kappa\to {}^{\kappa}\kappa$,
$X\subseteq{}^\kappa\kappa$ and $t\in{}^{<\kappa}\kappa$, 
let $\Lim_t^{f}(X)$ 
\index{limit set\idf$\Lim_t^{f}(X)$, $\Lim_t^{f}$}
denote the union of all 
limit points of sets of the form 
$\{f(x_\alpha) : \alpha<\kappa\}$ 
where 
$x_\alpha\in N_{t\conc\langle\alpha\rangle}\cap X$
for all $\alpha<\kappa$,
and let
$\Lim_t^{f}:=\Lim_t^{f}({}^\kappa\kappa)$.
For strict order preserving maps $\iota: {}^{<\kappa}\kappa\to {}^{<\kappa}\kappa$,
we also write
$\Lim_t^{\iota}(X):=\Lim_t^{[\iota]}(X)$ and
$\Lim_t^{\iota}:=\Lim_t^{[\iota]}$. 
We call these \emph{limit sets} of $f$ and $\iota$. 
\todog{Is $\Lim_t^{\iota}$ the same as the set of limit points of $\langle \iota(t\conc\langle\alpha\rangle):\alpha<\kappa\rangle$? Is something similar true for $\Lim_t^{\iota}(X)$ or even $\Lim_t^{f}(X)$? (answer: no)}

\begin{lemma} 
\label{lemma: [e] closed map}
Suppose that 
$\iota:{}^{<\kappa}\ddim\to{}^{<\kappa}\kappa$ is $\perp$- and strict order preserving. 
\begin{enumerate-(1)}
\item\label{[e] closed when ddim<kappa}
If $2\leq\ddim<\kappa$, then 
$[\iota]$ is a closed map. 
\item\label{closure of ran[e] when ddim=kappa}
If $\ddim=\kappa$, 
then 
$$\closure{[\iota](X)}=[\iota](\closure{X})\cup\bigcup_{t\in T(X)} \Lim^\iota_t(X).$$
for all subsets $X$ of ${}^\kappa\kappa$. 
In particular, 
 $[\iota]$ is a closed map if 
$\Lim^\iota_t=\emptyset$ for all $t\in{}^{<\kappa}\kappa$. 
\end{enumerate-(1)}
\end{lemma} 
\begin{proof}  
We prove~\ref{[e] closed when ddim<kappa} and~\ref{closure of ran[e] when ddim=kappa} simultaneously. 
Fix $2\leq\ddim\leq\kappa$ and a subset $X$ of ${}^\kappa\ddim$. 
Note that it suffices to show  $[\iota](\closure{X})=\closure{[\iota](X)}$ for \ref{[e] closed when ddim<kappa}. 
It is clear that $[\iota](\closure{X})\subseteq\closure{[\iota](X)}$ and that 
$\Lim^\iota_t(X)\subseteq\closure{[\iota](X)}$ for all $t\in{}^{<\kappa}\kappa$ in the case $\ddim=\kappa$. 
Now suppose that $y\in\closure{[\iota](X)}$. 
It suffices to show $y \in [\iota](\closure{X})\cup\bigcup_{t\in T(X)} \Lim^\iota_t(X)$. 
$y$ is the limit of some convergent sequence 
$\langle [\iota](x_\alpha):\alpha<\kappa\rangle$ with $x_\alpha\in X$ for each $\alpha<\kappa$.
If $\bar x:=\langle x_\alpha:\alpha<\kappa\rangle$ converges to some $x$, then $x\in \closure X$ and $y=f(x)\in[\iota](\closure X)$. 
Therefore we now assume that $\bar{x}$ does not converge. 
We will see that $\ddim=\kappa$ and $y\in\Lim^\iota_t(X)$ for some $t\in{}^{<\kappa}\kappa$. 

Let $\gamma\leq \kappa$ be least such that
$\langle x_\alpha\restr\gamma : \alpha<\kappa\rangle$ is not eventually constant. 
Then clearly $\gamma>0$. 
We first claim that $\gamma$ is a successor ordinal. 
Fix some $\delta_\beta<\kappa$ such that $\langle x_\alpha\restr\beta : \delta_\beta\leq\alpha<\kappa\rangle$ is constant for each $\beta<\gamma$ and let $\delta:=\sup_{\beta<\gamma} \delta_\beta<\kappa$. 
If $\gamma$ were a limit ordinal, then $\langle x_\alpha\restr\gamma : \delta\leq\alpha<\kappa\rangle$ would be constant. 

Let $\gamma=\beta+1$. 
We may assume that 
$\langle x_\alpha\restr\beta:\alpha<\kappa\rangle$ has a constant value $t\in{}^{<\kappa}\kappa$ by replacing $\bar x$ by a tail. 
Then $x_\alpha\restr\gamma=t\conc\langle x_\alpha(\beta)\rangle$ for all $\alpha<\kappa$.
Next, we claim that 
$\langle x_\alpha(\beta):\alpha<\kappa\rangle$ 
is unbounded in $\kappa$; in particular $\ddim=\kappa$. 
Towards a contradiction, suppose that there is a strict upper bound $\delta<\kappa$ with 
$x_\alpha(\beta)<\delta$ for all $\alpha<\kappa$.
Since
$\langle x_\alpha(\beta):\alpha<\kappa\rangle$ 
is not eventually constant, 
there exist $i<j<\kappa$ and disjoint unbounded subsets $I,\,J$ of $\kappa$ with 
$x_\alpha(\beta)=i$ for all $\alpha\in I$
and
$x_\alpha(\beta)=j$ for all $\alpha\in J$.
Then
$[\iota](x_\alpha)\supseteq \iota(t\conc\langle i\rangle)$
for all $\alpha\in I$ and
$[\iota](x_\alpha)\supseteq \iota(t\conc\langle j\rangle)$
for all $\alpha\in J$.
Since $\iota$ is $\perp$-preserving,
$\iota(t\conc\langle i\rangle)\perp\iota(t\conc\langle j\rangle)$.
But then $\langle[\iota](x_\alpha):\alpha<\kappa\rangle$ cannot be convergent. 

To show that  $y\in\Lim^\iota_t(X)$, 
we may assume that $\langle x_\alpha(\beta):\alpha<\kappa\rangle$ is injective by replacing $\bar x$ with a subsequence. 
Let $\bar z= \langle z_\alpha : \alpha<\kappa \rangle$ be a sequence with $z_\alpha\in N_{t\conc\langle \alpha\rangle}\cap X$ and
$z_\alpha=x_\delta$ if $\alpha= x_\delta(\beta)$. 
Since $\iota$ is $\perp$-preserving, $[\iota]$ is injective. 
Thus $y$ is a limit point of the set $\{[\iota](z_\alpha):\alpha<\kappa\}$ and hence $y\in\Lim^\iota_t(X)$. 
\end{proof} 

The next corollary follows from the previous lemma and Lemma~\ref{[e] continuous}.

\begin{corollary}
\label{[e] homeomorphism onto a closed set}
Suppose 
$\iota:{}^{<\kappa}\ddim\to{}^{<\kappa}\kappa$ is $\perp$- and strict order preserving. 
If $\ddim<\kappa$ or if $\ddim=\kappa$ and $\Lim^\iota_t=\emptyset$ for all $t\in{}^{<\kappa}\kappa$, then
$[\iota]$ is a homeomorphism from ${}^\kappa\ddim$ onto a closed subset of ${}^\kappa\kappa$. 
\end{corollary}

The next lemma shows that the assumption that $\iota$ is $\perp$-preserving can be omitted from the previous lemma if $\kappa$ is weakly compact. 

\begin{lemma} 
\label{lemma: f closed map kappa weakly compact}
Suppose that $\kappa=\omega$ or $\kappa$ is weakly compact.
Let $f:{}^{\kappa}\ddim\to{}^{\kappa}\kappa$ be a continuous map.
\begin{enumerate-(1)}
\item\label{f closed when ddim<kappa kappa weakly compact}
If $2\leq\ddim<\kappa$, then 
$f$ is a closed map. 
\item\label{closure of ran(f) when kappa weakly compact}
If $\ddim=\kappa$, 
then $$\closure{f(X)}=f(\closure{X})\cup\bigcup_{t\in T(X)} \Lim^f_t(X).$$ 
for all subsets $X$ of ${}^\kappa\kappa$. 
In particular, 
$f$ is a closed map if 
$\Lim^f_t=\emptyset$ for all $t\in{}^{<\kappa}\kappa$. 
\end{enumerate-(1)}
\end{lemma} 
\begin{proof}  
\ref{f closed when ddim<kappa kappa weakly compact} holds since closed subsets of ${}^
\kappa\ddim$ are $\kappa$-compact by the weak compactness of~$\kappa$, 
$\kappa$-compactness is preserved under continuous images and 
$\kappa$-compact subsets of ${}^\kappa\kappa$ are closed.

For~\ref{closure of ran(f) when kappa weakly compact},
note that clearly $f(\closure{X})\subseteq\closure{f(X)}$ and that 
$\Lim^f_t(X)\subseteq\closure{f(X)}$ for all $t\in{}^{<\kappa}\kappa$.
It suffices to show $y \in f(\closure{X})\cup\bigcup_{t\in T(X)} \Lim^f_t(X)$ for any $y\in\closure{f(X)}$. 
Fix a sequence 
$\langle f(x_\alpha):\alpha<\kappa\rangle$ converging to $y$ with $x_\alpha\in X$ for all $\alpha<\kappa$. 
We may assume $f(x_\alpha)\neq y$ for all $\alpha<\kappa$.
If $\bar x:=\langle x_\alpha:\alpha<\kappa\rangle$ is bounded, then it is contained in a $\kappa$-compact subset $C$ of $\closure X$ since $\kappa$ is weakly compact. Then $f(C)$ is also $\kappa$-compact and thus closed, so $y\in f(C)\subseteq f(\closure X)$.
Now suppose that $\bar x$ is unbounded. 
Let $\beta<\kappa$ be least such that 
$\langle x_\alpha(\beta):\alpha<\kappa\rangle$ is unbounded in $\kappa$.
We may assume that $\langle x_\alpha(\beta):\alpha<\kappa\rangle$ is injective by replacing $\bar x$ with a subsequence. 
\todol{slight correction}%
There are $\lle\kappa$ many possibilities for $x_\alpha\restr\beta$, as $\alpha$ ranges over $\kappa$,
since $\kappa$ is inaccessible 
and $\langle x_\alpha(\gamma):\alpha<\kappa\rangle$ is bounded for all $\gamma<\beta$.
We may thus assume that $\langle x_\alpha\restr\beta:\alpha<\kappa\rangle$
has a constant value $t\in{}^{<\kappa}\kappa$ by replacing $\bar x$
with a further subsequence.
Then $x_\alpha\restr{(\beta+1)}=t\conc\langle x_\alpha(\beta)\rangle$ for all $\alpha<\kappa$.
Let $\bar z= \langle z_\alpha : \alpha<\kappa \rangle$ be a sequence with $z_\alpha\in N_{t\conc\langle \alpha\rangle}\cap X$ and
$z_\alpha=x_\gamma$ if $\alpha= x_\gamma(\beta)$. 
Since $y\neq f(x_\gamma)$ for all $\gamma<\kappa$,
$y$ is a limit point of the set $\{f(z_\alpha):\alpha<\kappa\}$, and thus $y\in\Lim^f_t(X)$. 
\end{proof} 

The next remark shows that the following assumptions cannot be omitted: $\iota$ is $\perp$-preserving in Lemma \ref{lemma: [e] closed map} and $\kappa$ is weakly compact in Lemma \ref{lemma: f closed map kappa weakly compact}. 

\begin{remark} 
\label{ran[e] not closed}
Suppose that $\kappa$ is not weakly compact. 
\begin{enumerate-(1)} 
\item\label{ran[e] not closed 1}
There exists a strict order preserving map $\iota: {}^{<\kappa}2\to {}^{<\kappa} 2$ such that $\ran([\iota])$ is not closed, and thus
$[\iota]$ is not a closed map.
To see this, we first claim that there exists a \emph{barrier}
\index{barrier\idf} 
$A$ in ${}^{<\kappa}2$ of size $\kappa^{<\kappa}=\kappa$,\footnote{A similar argument also works without the assumption that $\kappa^{<\kappa}=\kappa$.} 
i.e. a maximal antichain in ${}^{<\kappa}2$ of size $\kappa^{<\kappa}=\kappa$ with the following equivalent properties: (a) every $x\in {}^{\kappa}2$ 
has an initial segment in $A$ and (b) $T(A)$ has no $\kappa$-branches. 
If $\kappa$ is inaccessible, let $A$ be the boundary of a
\index{tree!Aronszajn@$\kappa$-Aronszajn\idf}
$\kappa$-Aronszajn subtree of ${}^{<\kappa}2$.%
\footnote{Recall that a subtree $T$ of ${}^{<\kappa}\kappa$ of height $\kappa$ is a \emph{$\kappa$-Aronszajn tree} if its levels $T\cap{}^\alpha\kappa$ have size $\lle\kappa$ for all $\alpha<\kappa$ and its set $[T]$ of branches is empty.}
If $\kappa=\mu^+$, let $A:={}^\mu 2$. 

Fix an injective enumeration $\langle t_\alpha : \alpha<\kappa\rangle$ of $A$. 
Further fix a strictly increasing sequence $\langle \gamma_\alpha : \alpha<\kappa\rangle$ with $\gamma_\alpha \geq \lh(t_\alpha)$ and let  $\iota(t_\alpha):= \langle 0\rangle^{\gamma_\alpha}\conc\langle1\rangle$ for all $\alpha<\kappa$. 
We now extend $\iota$ 
to a strict order preserving map
from ${}^{<\kappa}2$ to ${}^{<\kappa}2$.
If $s\subsetneq t$ for some $t\in A$, let $\iota(s):=\langle 0\rangle^{\lh(s)}$. 
We further 
extend $\iota$ above $A$ arbitrarily 
to a strict order preserving map.
Then $\langle 0\rangle^\kappa$ witnesses that $\ran([\iota])$ is not closed. 
Note that we can also guarantee that $[\iota]$ is injective
by extending $\iota$ in a $\perp$-preserving way above $A$.%
\footnote{I.e., so that for all $u\in A$ and all
incompatible $v,w\supsetneq u$, we have $\iota(v)\perp\iota(w)$).
E.g.~let $\iota(u \conc v):=\iota(u)\conc v$ for all 
$u\in A$ and $v\in{}^{<\kappa}\kappa$.}

\item\label{ran[e] not closed 2}
There exists a strict order preserving map $\theta:{}^{<\kappa}\kappa\to{}^{<\kappa}\kappa$ such that $\Lim^{\theta}_t=\emptyset$ for all $t\in{}^{<\kappa}\kappa$ and $\ran([\theta])$ is not closed. 
To see this, define $\theta(s\conc t):=\iota(s)\conc t$ for all $s\in{}^{<\kappa}2$ and $t\in{}^{<\kappa}\kappa$ with $t(0)\geq 2$, 
where $\iota$ is as in \ref{ran[e] not closed 1}.
Note that if $[\iota]$ is injective, then so is $[\theta]$.
\end{enumerate-(1)} 
\end{remark} 

If $\mu$ is a regular cardinal, let $\cof_\mu$ 
\index{cofmu@$\cof_\mu$, $\cof^\gamma_\mu$, $\cof^\gamma_{<\mu}$,$\cof^\gamma_{>\mu}$\idf}
denote the class of limit ordinals 
$\alpha$
with $\cof(\alpha)=\mu$, 
and
for any ordinal $\gamma$, 
let $\cof^\gamma_\mu:=\cof_\mu\cap \gamma$.
We define $\cof_{{\leq}\mu}$, $\cof_{<\mu}$,
$\cof_{{>}\mu}$, $\cof^\gamma_{>\mu}$ etc.~similarly
for arbitrary ordinals $\mu$.
Given a class $C$ of limit ordinals and a subtree $T$ of ${}^{<\kappa}\kappa$, we say that
$T$ is \emph{$C$-closed} 
\index{tree!closed@$C$-closed\idf}%
\index{closed tree@$C$-closed tree\idf}%
if every strictly increasing sequence in $T$
whose supremum has length in $C$ 
has an upper bound in $T$. 
\todoo{typo corrected: we changed $A$ to $C$ in ``\ldots with length in $C$}
\todol{we also replaced ``\ldots with length in $C$'' by ``\ldots whose supremum has length in $C$'' since this definition is more general. 
This does not affect the next lemmas or remarks and the definition does not appear anywhere else.}

The next lemma shows that the relevant closure properties of $T(\iota)$ are determined by the following sets. 
For $\iota : {}^{<\kappa}\ddim \to {}^{<\kappa}\kappa$ and $u\in {}^{<\kappa}\ddim$, we write $T^{\iota,u}$ for the downward closure of 
$\{\iota(u\conc \langle \alpha \rangle): \alpha<\kappa \}$ in ${}^{<\kappa}\ddim$. 

\begin{lemma}
\label{lemma: T(e) closure properties} 
If $2 \leq \ddim\leq \kappa$, $\iota: {}^{<\kappa}\ddim \to {}^{<\kappa}\kappa$ is $\perp$- and strict order preserving and $\mu<\kappa$ is a regular infinite cardinal, then the following statements are equivalent: 
\begin{enumerate-(1)} 
\item 
\label{lemma: T(e) closure properties 1} 
$T(\iota)$ is $\cof^\kappa_\mu$-closed. 
\item 
\label{lemma: T(e) closure properties 2} 
For all $u\in {}^{<\kappa}\ddim$, $T^{\iota,u}$ is $\cof^\kappa_\mu$-closed. 
\end{enumerate-(1)} 
\end{lemma} 
\begin{proof} 
\ref{lemma: T(e) closure properties 1} $\Rightarrow$ \ref{lemma: T(e) closure properties 2}: 
Suppose that $w$ is a limit of nodes in $T^{\iota,u}$ with $w \in T(\iota) \,\setminus\, T^{\iota,u}$. 
There exists some $\eta<\kappa$ with $\iota(u\conc\langle \eta\rangle)\subsetneq w$, since otherwise $w$ properly extends $\iota(u)$ and is incompatible with $\iota(t)$ for all $t\supsetneq u$, but this would entail $w\notin T(\iota)$. 
Then $\iota(u\conc\langle \xi\rangle)\perp w$ for all $\xi\neq \eta$, so $w$ cannot be a limit of nodes in $T^{\iota,u}$. 

\ref{lemma: T(e) closure properties 2} $\Rightarrow$ \ref{lemma: T(e) closure properties 1}: 
Suppose that $w$ is a limit of nodes in $T(\iota)$ with $\lh(w) \in \cof^\kappa_\mu$. 
We shall show that $w\in T(\iota)$. 
Let 
$A := \{ t\in {}^{<\kappa}\ddim: \iota(t)\subseteq w\}$. Note that the elements of $A$ are pairwise comparable, since $\iota$ is $\perp$-preserving. 
Hence $u:= \bigcup A$ and $v:= \bigcup \{\iota(t) : t\in A\}$ are elements of ${}^{<\kappa}\ddim$. 
We may assume that $v\subsetneq w$, since otherwise $w=v\subseteq\iota(u)$.

For all $\gamma<\lh(w)$, pick some $s_\gamma$ such that $w{\restr}\gamma \subseteq \iota(s_\gamma)$ and $\lh(s_\gamma)>\lh(u)$. 
We claim that $u\subsetneq s_\gamma$ for all $\gamma$ with $\lh(v) <\gamma< \lh(w)$. 
To see this, note that for any $t\in A$, 
we have $\iota(t)\subseteq v\subseteq w{\restr} \gamma \subseteq \iota(s_\gamma)$, and hence $t\subseteq s_\gamma$ since $\iota$ is $\perp$-preserving. 
As this holds for all $t\in A$, we have $u\subseteq s_\gamma$. 
As $\lh(s_\gamma)>\lh(u)$, we have $u\subsetneq s_\gamma$. 

For each $\gamma$ as above, we can thus pick some $\eta_\gamma$ with $u\conc\langle\eta_\gamma\rangle \subseteq s_\gamma$. 
It suffices to show $w{\restr}\gamma \subseteq \iota(u\conc\langle\eta_\gamma\rangle)$, since this implies $w\in T^{\iota,u}\subseteq T(\iota)$. 
To see this, note that $w{\restr}\gamma \compat \iota(u\conc\langle \eta_\gamma\rangle)$, since both are initial segments of $\iota(s_\gamma)$. 
But $\iota(u\conc\langle \eta_\gamma\rangle)\subseteq w{\restr}\gamma\subseteq w$ would entail $u\conc\langle \eta_\gamma\rangle \subseteq u$ by the definition of $u$. 
\end{proof} 

The next corollary is immediate: 

\begin{corollary} 
\label{cor: T(e) closed}
Suppose that $\iota: {}^{<\kappa}\ddim \to {}^{<\kappa}\kappa$ is $\perp$- and strict order preserving. 
\begin{enumerate-(1)} 
\item 
If $2\leq \ddim<\kappa$, then $T(\iota)$ is $\cof^\kappa_{>\ddim}$-closed. 
\item 
If the nodes $\iota(t\conc\langle\alpha\rangle)$ 
split at the same node
for each $t\in {}^{<\kappa}\ddim$,%
\footnote{I.e., the $\iota(t\conc\langle\alpha\rangle)$'s extend pairwise different direct successors of some $s_t\in{}^{<\kappa}\kappa$, for each $t\in{}^{<\kappa}\ddim$.
Equivalently, $\iota$ is an order homomorphism for the dihypergraph $\dhhd$ defined right after Definition~\ref{def: Hwd}.} 
then $T(\iota)$ is ${<}\kappa$-closed. 
\end{enumerate-(1)} 
\end{corollary} 

The next remark shows
that 
the restriction to $\cof^\kappa_{>\ddim}$ 
is necessary in the previous corollary. 

\begin{remark} 
\label{example: T(e) not kappa-closed}
Suppose $\ddim\leq\kappa$ and 
$\eta\in\cof^\kappa_{\leq\ddim}$. 
Then there exists a $\perp$- and strict order preserving function $\iota: {}^{<\kappa}\ddim \to {}^{<\kappa}\kappa$ such that  $T(\iota)$ is not $\{\eta\}$-closed. 
To see this, choose a strictly increasing sequence $\langle \gamma_\alpha:\alpha<\cof(\eta)\rangle$ cofinal in~$\eta$,
and let
$\iota=\iota_{\ddim,\eta}$ be any $\perp$- and strict order preserving map
from ${}^{<\kappa}\ddim$ to ${}^{<\kappa}\kappa$
such that
\[
\iota\left(\langle \alpha\rangle\right)=
\begin{cases}
\langle 0\rangle^{\gamma_\alpha}\conc\langle 1\rangle 
&\text{if $\alpha<\cof(\eta)$,}
\\
\langle\alpha\rangle
&\text{if $\alpha\in[\cof(\eta),\ddim)$.}
\end{cases}
\]
Then the node
${\langle 0\rangle}^{\eta}$
witnesses that
$T(\iota)$ is not $\{\eta\}$-closed. 
For $\ddim=\kappa$, it is also easy to make sure that
$\Lim^\iota_t=\emptyset$ for all $t\in{}^{<\kappa}\kappa$.%
\footnote{Let, for example, $\iota(\langle \alpha \rangle \conc t)=\iota(\langle \alpha\rangle)\conc t$ for all $\alpha<\kappa$ and $t\in{}^{<\kappa}\kappa$.
Note that $\Lim^\iota_\emptyset=\emptyset$ is guaranteed by the choice of the $\iota(\langle \alpha\rangle)$'s.}

Suppose that $\ddim=\kappa$. 
Then there exists a $\perp$- and strict order preserving function $\psi: {}^{<\kappa}\ddim \to {}^{<\kappa}\kappa$ such that $T(\psi)$ is not $\{\eta\}$-closed for all limit ordinals $\eta<\kappa$, and at the same time 
$\Lim^\psi_t=\emptyset$ for all $t\in{}^{<\kappa}\kappa$. 
To see this, let $\langle \eta_\alpha:\alpha<\kappa\rangle$ be an enumeration of
all the limit ordinals below $\kappa$.
For all $\alpha<\kappa$, let $\iota_\alpha$ be the map 
$\iota_{\kappa,\eta_\alpha}$ 
defined in the previous paragraph.
Define
$\psi:{}^{<\kappa}\kappa\to{}^{<\kappa}\kappa$
where 
$\psi(\emptyset):=\emptyset$
and
$\psi(\langle\alpha\rangle\conc t):= 
\langle \alpha\rangle\conc\iota_\alpha(t)$
for all $\alpha<\kappa$ and
$t\in{}^{<\kappa}\kappa$.
Then $T(\psi)$ is not $\{\eta\}$-closed for any limit ordinal $\eta<\kappa$,
and we can also ensure that $\Lim^\psi_t=\emptyset$ for all $t\in{}^{<\kappa}\kappa$.
\end{remark}

\begin{lemma}
\label{perfect subsets and sop maps}
Suppose $2\leq\ddim<\omega$.
\begin{enumerate-(1)}
\item\label{sop maps -> perfect subsets}
If $\iota:{}^{<\kappa}\ddim\to{}^{<\kappa}\kappa$ 
is $\perp$- and strict order preserving, then
$\ran([\iota])$ is a $\kappa$-perfect set.
\item\label{perfect subsets -> sop maps}
If $X\subseteq{}^\kappa\kappa$ has a $\kappa$-perfect subset, then
there exists a $\perp$- and strict order preserving map 
$\iota:{}^{<\kappa}\ddim\to{}^{<\kappa}\kappa$ with $\ran([\iota])\subseteq X$.
\end{enumerate-(1)}
\end{lemma}
\begin{proof}
For~\ref{sop maps -> perfect subsets}, let $\iota:{}^{<\kappa}\ddim\to{}^{<\kappa}\kappa$ be
$\perp$- and strict order preserving.
By Lemma~\ref{lemma: [e] closed map},
$\ran([\iota])$ is a closed set, so
$\ran([\iota])=[T(\iota)]$.
It thus suffices to show that $T(\iota)$ is a $\kappa$-perfect tree.
$T(\iota)$ is $\lle\kappa$-closed
by Corollary~\ref{cor: T(e) closed}.
$T(\iota)$ is also cofinally splitting since $\iota$ preserves $\perp$.
For~\ref{perfect subsets -> sop maps}, let
$T$ be a $\kappa$-perfect tree $T$ with $[T]\subseteq X$. 
It is easy to construct a $\perp$- and strict order preserving map 
$\iota:{}^{<\kappa}\ddim\to T$ by recursion.
Then $\ran([\iota])\subseteq[T]\subseteq X$.
\end{proof}


\subsection{The open graph dichotomy} 
\label{subsection: OGD and ODD}
Throughout this subsection, we assume that $X$ is a subset of ${}^\kappa\kappa$.
We say that a graph $G$ on $X$ has a \emph{$\kappa$-perfect complete subgraph} \index{perfect@$\kappa$-perfect!complete subgraph\idf}
if there exists a $\kappa$-perfect subset $Y\subseteq X$ such that the complete graph $\gK Y$ 
is a subset of $G$. 
Recall that the open graph dichotomy $\OGD_\kappa(X)$ is the following statement:
\begin{quotation}
$\OGD_\kappa(X)$:
If $G$ is an open graph on $X$, then
either $G$ admits a $\kappa$-coloring or 
$G$ has a $\kappa$-perfect 
complete subgraph. 
\end{quotation}
In this subsection, we show that 
the open graph dichotomy 
is equivalent to $\ODD\kappa 2(X)$.%
\footnote{%
Thus, it also follows from 
$\ODD\kappa\ddim(X,\defsets\kappa)$
for any $\ddim>2$ 
by Lemmas~\ref{<kappa dim hypergraphs are definable} and~\ref{comparing ODD for different dimensions}.} 
In fact, we prove a more general version of 
this equivalence for an extension of $\OGD_\kappa(X)$ to $\ddim$-hypergraphs for higher dimensions $2\leq\ddim\leq\kappa$. 

Recall that a \emph{$\ddim$-hypergraph} is a $\ddim$-dihypergraph that is closed under permutation of hyperedges.%
\footnote{See Definition~\ref{def: dihypergraph}~\ref{def: dihypergraph hypergraph}. In this subsection only, $H$ always denotes a hypergraph and $I$ denotes a dihypergraph.}
For any $\ddim$-dihypergraph $I$,
let $\hypg I$ denote the 
\emph{symmetrization of $I$}, i.e., the smallest hypergraph containing~$I$: 
\index{symmetrization of!a dihypergraph\idf $\hypg H$}
$$\hypg I:=\bigcup
\{ 
\bar{x}^\pi:\:
{\bar{x}\in I,\ \pi\in\Sym(\ddim)}\}. 
$$
Then 
\index{symmetrization of!HH@$\dhH\ddim$\idf $\hhH\ddim$}
$\hhH\ddim=\bigcup
\left\{\prod_{\alpha\in \ddim}N_{t\conc\langle \pi(\alpha)\rangle}:\:
{t\in{}^{<\kappa}\ddim,\,\pi\in\Sym(\ddim)}\right\}
$,
and $\hhH2$ equals the complete graph $\gK{{}^\kappa 2}$.

\begin{definition}
\label{def: OHD}
$\OHD\kappa\ddim(X)$ 
\index{open hypergraph dichotomy!$d$-dimensional for a set\idf$\OHD\kappa\ddim(X)$}
states that the following holds for all box-open \emph{$\ddim$-hypergraphs} $H$ on $X$: 
\begin{quotation}
$\OHD\kappa H$:
\index{open hypergraph dichotomy!for a hypergraph\idf$\OHD\kappa H$}
Either $H$ admits a $\kappa$-coloring, or 
there exists a continuous homomorphism 
from~$\hhH{\ddim}$ to~$H$.
\end{quotation}
\end{definition}

We first show that $\OHD\kappa 2(X)$ is equivalent to the open graph dichotomy $\OGD_\kappa(X)$.
Since $\hhH 2=\gK{{}^\kappa 2}$,
it suffices to show that 
a graph $G$ on $X$ has a $\kappa$-perfect complete subgraph
if and only if there is a continuous homomorphism from $\gK{{}^\kappa 2}$ to $G$.
The next lemma is
a more general version of this for dihypergraphs.
We say a $\ddim$-dihypergraph 
$I$
has a \emph{$\kappa$-perfect complete subhypergraph} \index{perfect@$\kappa$-perfect!complete subhypergraph\idf}
if there exists a $\kappa$-perfect subset $Y\subseteq X$ such that $\dhK \ddim Y\subseteq I$. 

\begin{lemma}
\label{lemma: kappa-perfect subdh}
Suppose that $2\leq\ddim\leq\kappa$ and 
$I$ is a $\ddim$-dihypergraph on $X$. 
\begin{enumerate-(1)}
\item
\label{lemma: kappa-perfect subdh 1}
If $f$ is a continuous homomorphism from $\dhH\ddim$ to $I$,
then $\ran(f)$ has a $\kappa$-perfect subset.
\item
\label{lemma: kappa-perfect subdh 2}
$I$ has a $\kappa$-perfect complete subhypergraph if and only if 
there exists a continuous homomorphism from $\dhK\ddim{{}^\kappa\! \ddim}$ to~$I$.
\end{enumerate-(1)}
\end{lemma}
\begin{proof}
For~\ref{lemma: kappa-perfect subdh 1}, note first that
$f$ is also a homomorphism from $\dhH\ddim$ to 
$\dhK\ddim {\ran(f)}$. 
Since
$\dhK\ddim {\ran(f)}$
is relatively box-open,
there exists 
a continuous order homomorphism 
$\iota:{}^{<\kappa}\ddim\to{}^{<\kappa}\kappa$
for 
$\dhK\ddim {\ran(f)}$ 
by Lemma~\ref{homomorphisms and order preserving maps}.
For all $t\in{}^{<\kappa}\ddim$,
there exists $u,v\supsetneq t$ with $\iota(u)\perp \iota(v)$
\todog{We can not guarantee that there exists $\alpha<\beta<\ddim$ with $\iota(t\conc\langle\alpha\rangle)\perp\iota(t\conc\langle\beta\rangle)$. 
The negation of this statement (all the $\iota(t\conc\langle\alpha\rangle)$ being comparable)
 does NOT imply that \emph{some} sequence in $\prod_{\alpha<\ddim}(N_{\iota(t\conc\langle\alpha\rangle)}\cap \ran(f))$ is constant, since the union $w$ of all the  $\iota(t\conc\langle\alpha\rangle)$'s might not be in $\ran(f)\cup T(\ran(f))$.}
since otherwise,
every sequence in $\prod_{\alpha<\ddim}(N_{\iota(t\conc\langle\alpha\rangle)}\cap \ran([\iota]))$
would be constant, 
so $\iota$ would not be an order homomorphism
for $\dhK\ddim{\ran(f)}$.
Using this observation and the continuity of $\iota$, 
\todol{We needed to change this sentence in the proof, to ensure that $[\theta]\subseteq[\iota]$. (The previous version, where $\theta:{}^{<\kappa}2\to \ran(\iota)$ is an arbitrary $\perp$- and strict order preserving, may not suffice.)}
it is easy to construct a continuous 
strict order preserving map 
$e:{}^{<\kappa}2\to {}^{<\kappa}\ddim$ by recursion such that $\theta:=\iota\comp e$ is $\perp$-preserving.
Since $\theta$ is also strict order preserving,
$\ran([\theta])$ is a $\kappa$-perfect set by Lemma~\ref{perfect subsets and sop maps}.
Moreover, 
$[\theta]=[\iota]\comp[e]$ and hence $\ran([\theta])\subseteq\ran([\iota])\subseteq\ran(f)$.
%

For \ref{lemma: kappa-perfect subdh 2}, 
suppose first that $f$ is a homomorphism from 
$\dhK\ddim{{}^\kappa\ddim}$
to $I$.
Then 
$\dhK\ddim{\ran(f)}\subseteq I$.
Since $\ran(f)$ has a $\kappa$-perfect subset 
by \ref{lemma: kappa-perfect subdh 1}, 
$I$ has a $\kappa$-perfect complete subhypergraph.
Conversely,
let $Y$ be a $\kappa$-perfect subset of $X$ such that $\dhK\ddim{Y}\subseteq I$.
Take any map $f: {}^\kappa\ddim\to Y$ that is a homeomorphism onto its image. 
Since $f$ is injective, it is a homomorphism from
$\dhK\ddim{{}^\kappa\! \ddim}$ to~$\dhK\ddim Y\subseteq I$.
\end{proof}

The special case of \ref{lemma: kappa-perfect subdh 2} for the complete dihypergraph $I:=\dhK\ddim X$ shows that $X$ has a $\kappa$-perfect subset if and only if there exists a continuous injection from ${}^\kappa \ddim$ onto a closed subset of $X$.%
\footnote{This can also be proved directly as in \cite{LuckeMottoRosSchlichtHurewicz}*{Lemma~2.9} using Lemma~\ref{lemma: [e] closed map}.}
It also implies the next corollary since $\dhK\ddim X$ has a $\kappa$-coloring if and only if $|X|\leq\kappa$.

\begin{corollary}
\label{cor: PSP from ODD}
$\PSP_\kappa(X)\Longleftrightarrow \ODD\kappa{\dhK\ddim X}$ for all $\ddim$ with $2\leq\ddim\leq\kappa$.
\end{corollary}

Let $H$ be a {$\ddim$-hypergraph} on $X$. 
Then
any homomorphism from $\dhH{\ddim}$ to $H$ is also a homomorphism from $\hhH\ddim$ to $H$,
and 
hence $\OHD\kappa H$ is equivalent to $\ODD\kappa H$.
Thus
$\OHD\kappa\ddim(X)$ is equivalent to $\ODD\kappa\ddim(X,\mathcal H_\ddim)$,
where $\mathcal H_\ddim$ denotes the class of all $\ddim$-hypergraphs.
The next theorem 
shows that $\OHD\kappa\ddim(X)$ is in fact equivalent to
$\ODD\kappa\ddim(X)$.

\begin{theorem}
\label{theorem: ODD hypergraphs}
Suppose $2\leq\ddim\leq\kappa$ and $I$ is a $\ddim$-dihypergraph on $X$.
\begin{enumerate-(1)}
\item\label{hyp0} If $I$ is box-open on $X$, then
$\hypg I$ is box-open on $X$.
\item\label{hyp1} $I\equivf\hypg I$.\footnote{See Definition~\ref{def: H-full}.}
\item\label{hyp2} If $I$ is box-open on $X$, then 
$\ODD\kappa I\Longleftrightarrow\OHD\kappa{\hypg I}$.%
\footnote{Equivalently, there exists a continuous homomorphism from $\dhH\ddim$ to $I$ if and only if there exists a continuous homomorphism from $\hhH\ddim$ to $\hypg I$.}
\todoq{can the assumption that $I$ is box-open be omitted? It's used in the proof below}
\end{enumerate-(1)}
In particular, $\OHD\kappa\ddim(X)$ is equivalent to $\ODD\kappa\ddim(X)$.
\end{theorem}
\begin{proof}
\ref{hyp0} is clear.
\ref{hyp1} holds since a subset of ${}^\kappa\kappa$ is $\hypg I$-independent if and only if it is $I$-independent.
For \ref{hyp2},
note that
$\ODD\kappa I\Longleftrightarrow\ODD\kappa{\hypg I}$ by
\ref{hyp0}, \ref{hyp1} and Corollary~\ref{cor: ODD subsequences}. 
The latter is equivalent to 
$\OHD\kappa{\hypg I}$ by the paragraph preceding the theorem.
The last claim follows from~\ref{hyp0} and~\ref{hyp2}.
\end{proof}

\begin{corollary}
\label{homomorphisms and perfect homogeneous subgraphs}
Suppose $I$ is an open digraph 
 on $X$, 
\todoq{can the assumption that $I$ is open be omitted? It's used in the proof below}
and let $G:=\hypg I$ be the smallest graph containing $I$.
Then the following statements are equivalent: 
\begin{enumerate-(1)}
\item\label{hps 0} 
There is a continuous homomorphism from $\dhH 2$ to~$I$.
\item\label{hps 1} 
There is a continuous homomorphism from $\gK{{}^\kappa\! 2}$ to~$G$.
\item\label{hps 2} 
$G$ has a $\kappa$-perfect complete subgraph.
\end{enumerate-(1)}
In particular,
the open graph dichotomy $\OGD_\kappa(X)$ is equivalent to $\ODD\kappa 2(X)$.
\end{corollary}
\begin{proof}
\ref{hps 0} $\Leftrightarrow$ \ref{hps 1} follows from Theorem~\ref{theorem: ODD hypergraphs}~\ref{hyp2} 
since $\hhH 2=\gK{{}^\kappa 2}$.
\todog{$H$ needs to be open to apply Theorem~\ref{theorem: ODD hypergraphs}~\ref{hyp2}}
\ref{hps 1} $\Leftrightarrow$ \ref{hps 2} follows from Lemma~\ref{lemma: kappa-perfect subdh}~\ref{lemma: kappa-perfect subdh 2} 
since $\dhK 2 {{}^\kappa 2}=\gK {{}^\kappa 2}$.
\end{proof}

\subsection{Counterexamples}
\label{subsection: examples}

We show that the open dihypergraph dichotomy is nontrivial for any
non-$\kappa$-colorable
box-open dihypergraph $H$ on any subset $X$ of 
${}^\kappa\kappa$ in the sense that $\ODD\kappa {H\restr Y}$ fails for some $Y\subseteq X$. 
Furthermore, one can neither replace open by closed dihypergraphs in $\ODD\kappa\ddim(X)$ 
nor 
continuous homomorphisms by large complete subhypergraphs. 
The arguments for the latter two claims generalize known examples from the countable setting. 

The next lemma is used to show nontriviality of the open dihypergraph dichotomy.
\todol{the next two lemmas and corollary (about nontriviality) are new}

\begin{lemma}
\label{ODD fails for some Y new}
Suppose that 
$2\leq\ddim\leq\kappa$ and
$H$ is a box-open $\ddim$-dihypergraph on some subset $X$ of ${}^\kappa\kappa$.
If 
$|X\setminus A|=2^\kappa$
for all $A\subseteq X$ such that $H\restr A$ has a $\kappa$-coloring,
then there exists some $Y\subseteq X$ such that 
$\ODD\kappa{H\restr Y}$ fails.
\end{lemma}
\begin{proof}
We can assume that there exists at least one order homomorphism for $(X,H)$, since otherwise, $\ODD\kappa H$ fails by Lemma~\ref{homomorphisms and order preserving maps}~\ref{hop 1}, and then the conclusion holds for $Y:=X$. 
Let 
$\langle A_\alpha:\alpha<2^\kappa\rangle$ enumerate 
all subsets of ${}^\kappa\kappa$ which can be written as the
union of $\kappa$ many closed 
$H$-independent sets,
and let
$\langle \iota_\alpha:\alpha<2^\kappa\rangle$ enumerate all order homomorphisms for $(X,H)$, possibly with repetitions. 
Since for each $\alpha<2^\kappa$,
we have $|X\setminus A_\alpha|=2^\kappa$ by assumption
and
$|\ran([\iota_\alpha])|=2^\kappa$ 
by Lemma~\ref{lemma: kappa-perfect subdh}~\ref{lemma: kappa-perfect subdh 1}, it is easy to 
construct sequences
$\bar y:=\langle y_\alpha:\alpha<2^\kappa\rangle$
and
$\bar z:=\langle z_\alpha:\alpha<2^\kappa\rangle$
with
$y_\alpha\in X\setminus A_\alpha$,
$z_\alpha\in\ran([\iota_\alpha])$
and
$y_\alpha\neq z_\beta$ 
for all $\alpha,\beta<2^\kappa$.
Then $Y:=\ran(\bar y)$ and $Z:=\ran(\bar z)$ are disjoint subsets of $X$ with 
$Y\not\subseteq A_\alpha$
and 
$Z\cap \ran([\iota_\alpha])\neq\emptyset$
for all $\alpha<2^\kappa$.
To see that $\ODD\kappa{H\restr Y}$ fails, assume first that $H\restr Y$ has a $\kappa$-coloring and take $H$-independent sets $Y_i$ for $i<\kappa$ with $Y=\bigcup_{i<\kappa} Y_i$. 
By Corollary~\ref{independence and closure},
$\closure Y_i$ is $H$-independent as well for each $i<\kappa$ and hence
$\bigcup_{i<\kappa}\closure Y_i=A_\alpha$ for some $\alpha<2^\kappa$. But then $Y\subseteq A_\alpha$, a contradiction.
Now, assume that there exists
a continuous homomorphism from $\dhH\ddim$ to $H\restr Y$ 
and take
an order homomorphism $\iota$ for $(X,H)$
with $\ran([\iota])\subseteq Y$
by Lemma~\ref{homomorphisms and order preserving maps}~\ref{hop 1}.
Then $\iota=\iota_\alpha$ for some $\alpha<2^\kappa$ and hence
$Y\cap \ran([\iota])\neq\emptyset$. This contradicts the fact that $Y$ and $Z$ are disjoint.
\end{proof}

We next show that the assumption in the previous lemma holds if there is a continuous homomorphism from $\dhH\ddim$ to $H$.

\begin{lemma}
\label{lemma: kappa-perfect subdh new}
Suppose that 
$2\leq\ddim\leq\kappa$ and
$H$ is a box-open $\ddim$-dihypergraph on a subset $X$ of ${}^\kappa\kappa$.
If $f$ is a continuous homomorphism from $\dhH\ddim$ to $H$,
then $\ran(f)\setminus A$ has a $\kappa$-perfect subset for all $A\subseteq X$ such that $H\restr A$ has a $\kappa$-coloring.\footnote{This lemma strengthens both Lemmas~\ref{two options in ODD are mutually exclusive}~\ref{me 1} and~\ref{lemma: kappa-perfect subdh}~\ref{lemma: kappa-perfect subdh 1}.}
\end{lemma}
\begin{proof}
By Lemma~\ref{homomorphisms and order preserving maps}~\ref{hop 1}, we may assume that $f=[\iota]$ for some 
order homomorphism $\iota$ for $(X,H)$. 
Take some $H$-independent sets $A_i$ for $i<\kappa$ with $A=\bigcup_{i<\kappa} A_i$.
We construct a continuous strict order preserving map 
$e:{}^{<\kappa} 2\to{}^{<\kappa}\ddim$ 
such that 
\begin{enumerate-(i)}
\item \label{lemma: kappa-perfect subdh new i}
$\theta:=\iota\comp e$ is $\perp$-preserving, and
\item \label{lemma: kappa-perfect subdh new ii}
$N_{\theta(t\conc\langle i\rangle)}\cap A_{\lh(t)}=\emptyset$ for all $t\in{}^{<\kappa}2$ and $i=0,1$.
\end{enumerate-(i)}
We define $e(t)$ by recursion on $\lh(t)$. 
Let $e(\emptyset)=\emptyset$. 
If $\lh(t)\in\Lim$, let $e(t)=\bigcup_{s\subsetneq t}e(s)$.
Suppose that $e(t)$ has been defined.
To construct 
$e(t\conc\langle 0\rangle)$ and $e(t\conc\langle 1\rangle)$,
first take an immediate successor $u$ of $e(t)$ with 
$N_{\iota(u)}\cap A_{\lh(t)}=\emptyset$.
Such $u$ exists since 
otherwise, one could obtain a hyperedge of $H\restr A_{\lh(t)}$ by choosing an arbitrary element of $N_{\iota(e(t)\conc\langle i\rangle)}\cap A_{\lh(t)}$ for each $i<\ddim$, 
as $\iota$ is an order homomorphism for $(X,H)$.
This contradicts the $H$-independence of $A_{\lh(t)}$.
Let $e(t\conc\langle 0\rangle)$ and $e(t\conc\langle 1\rangle)$
be extensions of $u$ in ${}^{<\kappa} \ddim$
such that 
$\theta(t\conc\langle 0\rangle):=\iota(e(t\conc\langle 0\rangle))$
and
$\theta(t\conc\langle 1\rangle):=\iota(e(t\conc\langle 1\rangle))$
are incomparable.
Such extensions exist
since otherwise, every sequence in 
$\prod_{i<\ddim}N_{\iota(u\conc\langle i\rangle)}\cap \ran([\iota])$
would be constant, so $\iota$ would not be an order homomorphism for $(X,H)$.
Once $\theta$ has been constructed, 
note that $\ran([\theta])$ is a $\kappa$-perfect subset of $\ran(f)$
by Lemma~\ref{perfect subsets and sop maps} and 
\ref{lemma: kappa-perfect subdh new i}.
Moreover, $\ran([\theta])\cap A=\emptyset$ by
\ref{lemma: kappa-perfect subdh new ii}.
\end{proof}

\begin{corollary}
\label{cor: ODD fails for some Y new}
Suppose  
$2\leq\ddim\leq\kappa$
and 
$H$ is a non-$\kappa$-colorable box-open $\ddim$-dihypergraph on a subset $X$ of ${}^\kappa\kappa$.
Then there exists some $Y\subseteq X$ such that 
$\ODD\kappa{H\restr Y}$ fails.
\end{corollary}
\begin{proof}
If $\ODD\kappa H$ fails, then the conclusion holds for $Y:=X$. 
Otherwise, there exists a continuous homomorphism from $\dhH\ddim$ to $H$ by our assumption.
Hence the conclusion follows from the previous two lemmas.
\end{proof}

\begin{remark}
\label{remark: ODD fails simultaneously for some Y new}
It is also possible to get a simultaneous failure
for all dihypergraphs
which are as in Lemma~\ref{ODD fails for some Y new} with arbitrary dimensions ${\leq}\kappa$. 
Suppose $X$ is a subset of ${}^\kappa\kappa$. Let $\mathcal H_X$ denote the class of all box-open 
${\leq}\kappa$ dimensional dihypergraphs $H$ on $X$
such that  
$|X\setminus A|=2^\kappa$ holds
for any $A\subseteq X$ for which $H\restr A$ admits a $\kappa$-coloring.%
\footnote{In particular, $\mathcal H_X$ contains all box-open ${\leq}\kappa$ dimensional dihypergraphs $H$ on $X$ such that 
there exists a continuous homomorphism from $\dhH{(\ddim_{H})}$ to $H$ 
by Lemma~\ref{lemma: kappa-perfect subdh new}. Here, $\ddim_H$ denotes the dimension of $H$.}
Then there exists some $Y\subseteq X$ such that 
$\ODD\kappa{H\restr Y}$ fails
for each $H\in\mathcal H_X$.
This can be seen by modifying the proof of Lemma~\ref{ODD fails for some Y new} as follows. 
We can assume that $X$ has a $\kappa$-perfect subset,
since otherwise, $\ODD\kappa{H}$ fails for all $H\in\mathcal H_X$ by Lemma~\ref{lemma: kappa-perfect subdh}~\ref{lemma: kappa-perfect subdh 1}, and then the conclusion holds for $Y:=X$. 
Let $\langle P_\alpha:\alpha<2^\kappa\rangle$ enumerate all $\kappa$-perfect subsets of $X$, possibly with repititions,
and let $\langle A_\alpha:\alpha<2^\kappa\rangle$ enumerate all subsets of ${}^\kappa\kappa$  
of the form $\bigcup_{i<\kappa}C_i$ where each $C_i$ is closed and there exists $H\in\mathcal H_X$ such that each $C_i$ is $H$-independent.
As in the proof of Lemma~\ref{ODD fails for some Y new},
construct disjoint sets $Y$ and $Z$
such that $Y\not\subseteq A_\alpha$ and $Z\cap P_\alpha\neq\emptyset$ for all $\alpha<\kappa$. 
Then $\ODD\kappa{H}$ fails for each $H\in\mathcal H_X$, since
the first property guarantees that $H\restr Y$ is not $\kappa$-colorable,\todog{We use that $H$ is box-open on $X$ here.} 
and the second property guarantees that 
$Y$ does not have a $\kappa$-perfect subset. 
Hence there is no continuous homomorphism from $\dhH{(\ddim_H)}$ to $H\restr Y$ by Lemma~\ref{lemma: kappa-perfect subdh}~\ref{lemma: kappa-perfect subdh 1}, where $\ddim_H$ denotes the dimension of $H$.
\end{remark}

The next example shows that 
$\ODD\kappa H$ may fail for some closed hypergraph $H$
in any dimension $\ddim\geq 2$.

\begin{example} 
\label{OGD fails for closed graphs} 
For any ordinal $\ddim\geq 2$, 
there exists a product-closed%
\footnote{I.e., $H$ is a relatively closed subset 
of $\dhK\ddim {{}^\kappa\kappa}$, where ${}^\ddim {({}^\kappa\kappa)}$ has the product topology.
Equivalently, $H\cup \dhC\ddim {{}^\kappa\kappa}$ is a closed subset of ${}^\ddim {({}^\kappa\kappa)}$ with the product topology.} 
\index{dihypergraph!product-closed\idf}
$\ddim$-hypergraph $H\in\defsetsk$ 
on ${}^\kappa\kappa$ with the properties: 
\begin{enumerate-(a)} 
\item\label{cg no coloring} 
$H$ does not admit a $\kappa$-coloring.  
\item\label{cd all}
\begin{enumerate-(i)} 
\item\label{cd no complete subgraph}
$H$ does not have a complete subhypergraph of size $\kappa^+$.  
\item\label{cd no continuous homomorphism}
There is no continuous homomorphism from $\dhH\ddim$ to $H$. 
\end{enumerate-(i)} 
\end{enumerate-(a)} 

We first provide an example for $\ddim=2$. 
An example for $\kappa=\omega$ can be found in \cite{TodorcevicFarah}*{Proposition~10.1}, so suppose $\kappa$ is uncountable. 
We follow \cite{MR1940513}*{Exercise 29.9} and \cite{SzThesis}*{Example 3.3}. 
The example is based on the notion of a \emph{strongly $\kappa$-dense linear order}. 
\index{strongly dense linear order@strongly $\kappa$-dense linear order\idf}%
This is a linear order with no $({<}\kappa,{<}\kappa)$-gaps. 
\index{gap@$({<}\kappa,{<}\kappa)$-gap\idf}%
A $({<}\kappa,{<}\kappa)$-gap in a linear order $(\mathbb{L},\leq)$ is a pair $(A,B)$ of subsets of $\mathbb{L}$ with the properties: 
\begin{enumerate-(1)}
  \item 
    $a<b$ for all $a\in A$ and $b\in B$. 
  \item 
    There is no $x\in \mathbb{L}$ with $a<x<b$ for all $a\in A$ and $b\in B$. 
  \item 
  $|A|, |B|<\kappa$. 
\todog{Strongly $\kappa$ dense linear orders have no endpoints: consider the $A=\emptyset$ case, and the $B=\emptyset$ case} 
\end{enumerate-(1)} 
It is easy to see that a strongly $\kappa$-dense linear order $\mathbb{L}$ of size $\kappa$ exists for any $\kappa$ with $\kappa^{<\kappa}=\kappa$.%
\footnote{%
It is also easy to show using a back-and-forth argument that any two strongly $\kappa$-dense linear orders of size $\kappa$ are isomorphic.}
\todog{Both existence and uniqueness up to isomorphism also follow from the fact that 
a linear order is strongly $\kappa$-dense if and only if it is a $\kappa$-saturated model of the theory $T$ of dense linear orders with no endpoints.}
For instance, let $\QQ_\kappa$ denote the set of all $x\in {}^\kappa\{-1,0,1 \}$ that take values in $\{-1,1\}$ up to some $\alpha<\kappa$ and are constant with value $0$ from $\alpha$. 
$\QQ_\kappa$ with the lexicographical order is strongly $\kappa$-dense.

Let $(\mathbb{L},<)$ be a strongly dense linear order of size $\kappa$, and let 
\[
  X:=\{y\in \pwrset(\mathbb{L})\,:\, {\leq}{\restriction}y
  \text{ is a well-order}\}.
\] 
Since $|\mathbb{L}|=\kappa$, $\pwrset (\mathbb{L}) $ can be identified with ${}^\kappa 2$. 
Since $\kappa$ is uncountable, $X$ is a closed subset.\footnote{For $\kappa=\omega$, $X$ would be a $\coanalytic$ subset.} 
Therefore it suffices to define a closed graph $H$ on $X$ with the required properties.
Write $x\sqsubset y$ if $x$ is a strict initial segment of $y$ with respect to $\mathbb{L}$. 
For $x,y\in X$, let 
\[ (x,y) \in H  \ :\Longleftrightarrow\   x\sqsubset y  \text{ or }  y\sqsubset x. \] 
It is easy to see that $H$ is a closed graph on $X$ with no complete subgraph of size $\kappa^+$.
Thus there is no continuous homomorphism from $\dhH 2$ to $H$ by Corollary~\ref{homomorphisms and perfect homogeneous subgraphs}.

We now show that $H$ does not admit a $\kappa$-coloring.
\todoo{This paragraph was streamlined and made more explicit}
Towards a contradiction, suppose that $X=\bigcup_{\alpha<\kappa} X_\alpha$ where $X_\alpha$ is $H$-independent for each $\alpha<\kappa$. 
Construct a strictly increasing chain
$\langle x_\alpha:\alpha<\kappa\rangle$ in $X$ with respect to end extension 
\todoo{We need the $x_\alpha$'s to be increasing with respect to end extension (instead of $\subseteq$)! Otherwise the construction might fail in step $\omega$.}
and a strictly 
$\leq$-decreasing chain 
$\langle r_\alpha :\alpha<\kappa\rangle$ in $\mathbb{L}$ 
such that for all $\alpha<\kappa$, 
$x_\alpha$ is bounded below $r_\alpha$.\footnote{There exists some $q<r_\alpha$ such that $p<q$ for all $p\in x_\alpha$.} 
In each step, we first choose $r_\alpha$ such that $r_\alpha<r_\beta$ for all $\beta<\alpha$ and $\bigcup_{\beta<\alpha}x_\beta$ is bounded below $r_\alpha$. 
This is possible since $\QQ_\kappa$ is strongly $\kappa$-dense. 
We then choose an end extension $x_\alpha$ of $\bigcup_{\beta<\alpha}x_\beta$ that is bounded below $r_\alpha$. 
Moreover, we ensure that $x_\alpha\in X_\alpha$ holds if this is possible in the current step. 
Finally, let  $x:=\bigcup_{\alpha<\kappa}x_\alpha$ and suppose that $x\in X_\beta$ for some $\beta<\kappa$. 
Then $x_\beta\in X_\beta$ by the construction and $(x_\beta,x)\in H$, so $X_\beta$ is not $H$-independent. 

Given the example for $\ddim=2$,
the case $\ddim>2$ follows from the next claim.

\begin{claim*}
Suppose $\ddim>2$ and
\todoq{Does the claim work for digraphs instead of graphs? (We used that $G$ is a graph in the proof.)}
$G$ is a graph on ${}^\kappa\kappa$. 
Let $H_G$ be the $\ddim$-hypergraph consisting of those sequences $\bar x$ such that $\ran(\bar x)=\{x,y\}$ for some $(x,y)\in G$.
\begin{enumerate-(1)}
\item\label{H_G is a closed hypergraph}
\begin{enumerate-(i)}
\item
If $G$ is closed on ${}^\kappa\kappa$, then $H_G$ is product-closed on ${}^\kappa\kappa$.
\item
If $G\in\defsetsk$, then $H_G\in\defsetsk$.
\end{enumerate-(i)}
\item\label{H_G coloring}
A function $c:{}^\kappa\kappa\to\kappa$ is a $\kappa$-coloring of $H_G$
if and only if
it is a $\kappa$-coloring of $G$.
\item\label{H_G large}
\begin{enumerate-(i)}
\item\label{H_G large compl}
$H_G$ does not have a complete subhypergraph of size $3$.
\item\label{H_G large hom}
If there is a continuous homomorphism from $\dhH\ddim$ to $H_G$,
then there is a continuous homomorphism from $\dhH 2$ to $G$.
\todog{This does NOT follow from Lemma~\ref{lemma: ODD subsequences} since $G$ may not be box-open}
\end{enumerate-(i)}
\end{enumerate-(1)}
\end{claim*}
\begin{proof}
\ref{H_G is a closed hypergraph}
is clear. 
\ref{H_G coloring} holds since a subset of ${}^\kappa\kappa$ is $H_G$-independent if and only if it is $G$-independent.

In~\ref{H_G large},
\ref{H_G large compl} is immediate.
For
\ref{H_G large hom}, suppose that $f$ is a continuous homomorphism from $\dhH\ddim$ to $H_G$.
We construct a continuous $\perp$- and strict order preserving map
$\iota:{}^{<\kappa} 2\to{}^{<\kappa}\ddim$ such that 
$$
f(N_{\iota(t\conc\langle 0\rangle)})
\times
f(N_{\iota(t\conc\langle 1\rangle)})
\subseteq G
$$
holds for all $t\in{}^{<\kappa} \ddim$.
We construct $\iota(t)$ by recursion on $\lh(t)$.
Let $\iota(\emptyset):=\emptyset$.
If $\lh(t)\in\Lim$,
and $\iota(s)$ has been defined for all $s\subsetneq t$, then
let $\iota(t):=\bigcup_{s\subsetneq t}\iota(s)$.
We now assume $\iota(t)$ has been defined
and construct 
$\iota(t\conc\langle 0\rangle)$
and $\iota(t\conc\langle 1\rangle)$.
For each $i<\ddim$, take $x_i$ extending $\iota(t)\conc\langle i\rangle$.
Since 
$\bar x:=\langle x_i:i<\ddim\rangle$ is a hyperedge of $\dhH\ddim$
and
$f$ is a homomorphism from $\dhH\ddim$ to $\dhK\ddim{{}^\kappa\kappa}$,
there exist $j<k<\ddim$
with $f(x_{j})\neq f(x_{k})$.
Using the continuity of $f$,
take $\alpha<\kappa$
such that 
$\alpha>\lh(\iota(t))$ and
$f(N_{x_j\restr\alpha})\cap f(N_{x_k\restr\alpha})=\emptyset$.
Since $f$ is a homomorphism from $\dhH\ddim$ to $H_G$
and $x_i\restr\alpha$ extends $\iota(t)\conc\langle i\rangle$ for each $i<\ddim$,
we have
$\prod_{i<\ddim}f(N_{x_i\restr\alpha})\subseteq H_G$.
Since
$f(N_{x_j\restr\alpha})$ and $f(N_{x_k\restr\alpha})$ are disjoint,
we must have
\todog{We use that $G$ is symmetric in this step.}
$$f(N_{x_j\restr\alpha})\times f(N_{x_k\restr\alpha})\subseteq G.$$
Let
$\iota(t\conc\langle 0\rangle)=x_j\restr\alpha$
and
$\iota(t\conc\langle 0\rangle)=x_k\restr\alpha$.

Now, suppose that $\iota$ has been constructed and let
$g:=f\comp[\iota]$.
Then $g$ is continuous.
To see that $g$ is a homomorphism from $\dhH 2$ to $G$,
suppose $(x,y)\in\dhH 2$ and let $t\in{}^\kappa 2$ 
be the node where $x$ and $y$ split.
Then
$(g(x),g(y))\in g(N_{t\conc\langle 0\rangle})\times
g(N_{t\conc\langle 1\rangle})
=
f(N_{\iota(t\conc\langle 0\rangle)})\times
f(N_{\iota(t\conc\langle 1\rangle)})\subseteq G$, 
as required.
\end{proof}
\end{example}

\begin{remark} 
There is a simpler way to get a dihypergraph with the same properties as in the previous claim instead of a hypergraph. 
Given a graph $G$ on ${}^\kappa\kappa$ and $\ddim>2$, let 
$I_G$ denote the $\ddim$-dihypergraph that consists of all $\langle x\rangle\conc \langle y : 1\leq i<\ddim\rangle$ with $(x,y)\in G$. 
It is again clear that
\ref{H_G is a closed hypergraph},
\ref{H_G coloring}
and~\ref{H_G large}\ref{H_G large compl} 
hold for $I_G$.
\ref{H_G large}\ref{H_G large hom} 
holds as well, since
for any homomorphism $f$ from $\dhH\ddim$ to $I_G$, $f\restr {}^\kappa2$ is a homomorphism from $\dhH2$ to $G$. 
\end{remark} 

\begin{remark} 
Note that the closed set $X$ in the above example is not homeomorphic to ${}^\kappa\kappa$. 
In fact, one can show 
as in \cite{MR3430247}*{Section 2} 
that $X$ is not a continuous image of ${}^\kappa\kappa$. 
To see this, one only has to replace the relation $R_x$ coded by a subset $x$ of $\kappa$ by the relation ${<_\mathbb{L}}\restr x$ throughout the proof of \cite{MR3430247}*{Lemma 2.4}. 
\end{remark}


A \emph{complete subhypergraph} 
of a $\ddim$-dihypergraph $H$ on $X$ is a subhypergraph of the form $\dhK \ddim Y\subseteq H$ for some subset $Y\subseteq X$. 
In the definition of $\ODD\kappa H$, one cannot replace 
the existence of a continuous homomorphism with the existence of a large complete subhypergraph if $|d|\geq 3$, 
since $\dhH\ddim$ does not admit a $\kappa$-coloring by Lemma \ref{two options in ODD are mutually exclusive} 
and it does not have a complete subhypergraph of size $2$. 
Instead of complete subhypergraphs,
it is thus more interesting to consider subhypergraphs of the form 
$\dhII\ddim Y$ 
consisting of all \emph{injective} $\ddim$-sequences in some subset $Y$ of $X$. 
However, $\dhII\ddim Y\not\subseteq\dhH\ddim$ for any $Y$ of size $|\ddim|$. 
Similarly, the hypergraph $\hhH\ddim$ does not admit a $\kappa$-coloring 
and if $d$ is infinite, then 
$\dhII\ddim Y\not\subseteq\hhH\ddim$ for any $Y$ of size $|\ddim|$.
If $d$ is finite, then
$\dhII\ddim Y\not\subseteq\hhH\ddim$ for any $Y$ of size $|\ddim|+1$.
We now describe some more interesting counterexamples with stronger properties.


\begin{example} 
\label{3 dim OGA fails}
For any ordinal $\ddim\geq 3$, there exists 
a product-open%
\footnote{I.e., $H$ is an open subset of ${}^\ddim X$ in the product topology.} 
$\ddim$-dihypergraph 
$H\in\defsetsk$ on ${}^\kappa 2$ 
with the properties:
\todog{Let $[X]^3=\{(x,y,z)\in {}^3 X:x,y,z$ are pairwise distinct$\}$. For $\ddim=3$, $H\cap [X]^3$ will be a clopen subset of $[X]^3$, but $H$ will not be a closed subset of $\dhK 3 X$. Even if it were, it would not guarantee that $H'$ (as in Lemma 3.3) is a closed subset of $\dhK\ddim X$ for $\ddim>3$.}
\begin{enumerate-(a)}
\item\label{3d OGA no coloring}
$H$ does not admit a $\kappa$-coloring. 
\item\label{3d OGA no complete subgraph}
$\dhII\ddim Y\not\subseteq H$ for any $Y$ of size at least $\max\{5,|\ddim|\}$.
\end{enumerate-(a)}
For $\ddim=3$, $H$ can be chosen to be a hypergraph.

It suffices to provide an example for $\ddim=3$, since one can then apply a projection 
(see Lemma~\ref{sublemma: ODD for different dimensions}). 
We follow  \cite{TodorcevicFarah}*{Section 10} and \cite{SzThesis}*{Example~3.4}. 
Let $X$ denote the set of those $x\in {}^\kappa 2$ with $x(\alpha)=0$ for $\alpha=0$ and all limits $\alpha<\kappa$.
Since $X$ is homeomorphic to ${}^\kappa 2$, 
it suffices to define an open dihypergraph on $X$ with the required properties. 
Recall that for $x\neq y$ in ${}^\kappa2$, $\Delta(x,y)$ denotes the least $\alpha<\kappa$ with $x(\alpha)\neq y(\alpha)$.%
\footnote{Note that $\Delta(x,y)$ equals the length of the node $x\wedge y$ where $x$ and $y$ split.} 

For any 
$\bar x\in\dhII 3 X$, 
let
$\bar x_0,\, \bar x_1,\, \bar x_2$
list $\ran(\bar x)$ such that 
$\Delta(\bar x_0,\bar x_1)<\Delta(\bar x_1,\bar x_2)$.
As can be seen in Figure \ref{figure: a triple},  
$\bar{x}_0$ is unique while the role of $\bar{x}_1$ and $\bar{x}_2$ can be switched. 
Let 
$\Delta_0^{\bar x}:=\Delta(\bar{x}_0,\bar{x}_1)=\Delta(\bar{x}_0,\bar{x}_2)$ 
and 
$\Delta_1^{\bar x}:=\Delta(\bar{x}_1,\bar{x}_2)$. 
Then $\Delta_0^{\bar x},\,\Delta_1^{\bar x}$ are successor ordinals with $\Delta_0^{\bar x}<\Delta_1^{\bar x}$. 
For any 
$\bar x\in\dhII 3 X$, let
$$
\bar x\in H\ :\Longleftrightarrow\ 
\bar{x}_2{\left(\Delta_0^{\bar x}-1\right)}
\neq
\bar{x}_2{\left(\Delta_1^{\bar x}-1\right)}.
$$ 
It is easy to see that $H\in\defsets\kappa$ is an open $3$-hypergraph. 

{
\setlength{\abovecaptionskip}{-10pt} 
\setlength{\intextsep}{10pt} 
\newcommand{\x}{0.5} 
\newcommand{\y}{1} 
\begin{figure}[H]
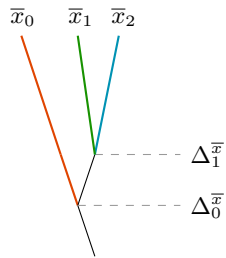
 
\tikz[scale=0.9,font=\small]{ 

\draw
(0.5*\x,0.75*\y) edge[-] (0*\x,1.5*\y) 
(0*\x,1.5*\y) edge[-, thick, RedOrangeArrow] (-0.5*\x,2.25*\y) 
(0*\x,1.5*\y) edge[-] (0.5*\x,2.25*\y); 

\draw
(0*\x,1.5*\y) edge[-, dashed, LightGrey] (3*\x,1.5*\y) 
(0.5*\x,2.25*\y) edge[-, dashed, LightGrey] (3*\x,2.25*\y); 

\draw 
(0.5*\x,2.25*\y) edge[thick, -, TealArrow] (1.2*\x,4*\y); 

\draw
(-0.5*\x,2.25*\y) edge[thick, -, RedOrangeArrow] (-1.65*\x,4*\y) 
(0.5*\x,2.25*\y) edge[thick, -, DarkGreenArrow] (0*\x,4*\y); 

\draw
    (-1.6*\x,4*\y) node[above] {\footnotesize 
    $\bar{x}_0$} 
    (0.1*\x,4*\y) node[above] {\footnotesize 
    $\bar{x}_1$} 
    (1.35*\x,4*\y) node[above] {\footnotesize 
    $\bar{x}_2$}; 
    
\draw
    (3*\x,1.5*\y) node[right] 
{\footnotesize $\Delta_0^{\bar x}$}
    (3*\x,2.25*\y) node[right] 
{\footnotesize $\Delta_1^{\bar x}$};
    
\draw 
    (0*\x,0*\y) node[above] {} 
    (4*\x,4*\y) node[above] {}; 
}
\caption{A triple in $H$} 
\label{figure: a triple}
\end{figure}}

For~\ref{3d OGA no coloring}, it suffices to show that any $H$-independent subset $Y$ of $X$ is nowhere dense in $X$, since $X$ is homeomorphic to~${}^\kappa 2$ and thus not the union of $\kappa$ many nowhere dense subsets. 
To this end, take any $s\in T(X)$.
There exist  $s_0,s_1,s_2\in T(X)$ 
with $s\subseteq s_i$ for all $i\leq 2$ and
\todon{In the next 2 figures: The gap between the figure and the caption was decreased by LOCALLY setting the value of ``abovecaptionskip'', and so was the gap between the text and the figure (``intextsep''). Remove before submitting?}
$N_{s_0}\times N_{s_1}\times N_{s_2}\subseteq H$ 
by the definition of $H$. 
Since $Y$ is $H$-independent, there is some $i\leq 2$ with $N_{s_i}\cap Y=\emptyset$. 

\vspace{5pt} 
{
\setlength{\abovecaptionskip}{8pt} 
\setlength{\intextsep}{10pt} 
\newcommand{\x}{0.5} 
\newcommand{\y}{1} 

\begin{figure}[H]
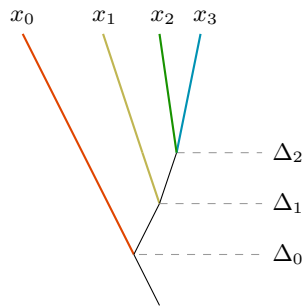
 
\centering
\tikz[scale=0.9,font=\small]{ 

\draw 
(0*\x,0) edge[-] (-0.75*\x,0.75*\y); 

\draw
(-0.75*\x,0.75*\y) edge[-, thick, RedOrangeArrow] (-4*\x,4*\y) 
(-0.75*\x,0.75*\y) edge[-] (0*\x,1.5*\y) 
(0*\x,1.5*\y) edge[-, thick, YellowArrow] (-0.5*\x,2.25*\y) 
(0*\x,1.5*\y) edge[-] (0.5*\x,2.25*\y); 

\draw
(-0.6*\x,0.75*\y) edge[-, dashed, LightGrey] (3*\x,0.75*\y) 
(0*\x,1.5*\y) edge[-, dashed, LightGrey] (3*\x,1.5*\y) 
(0.5*\x,2.25*\y) edge[-, dashed, LightGrey] (3*\x,2.25*\y); 

\draw 
(0.5*\x,2.25*\y) edge[thick, -, TealArrow] (1.2*\x,4*\y); 

\draw
(-0.5*\x,2.25*\y) edge[thick, -, YellowArrow] (-1.65*\x,4*\y) 
(0.5*\x,2.25*\y) edge[thick, -, DarkGreenArrow] (0*\x,4*\y); 

\draw
    (-4*\x,4*\y) node[above] {\footnotesize 
    $x_0$} 
    (-1.6*\x,4*\y) node[above] {\footnotesize 
    $x_1$} 
    (0.1*\x,4*\y) node[above] {\footnotesize 
    $x_2$} 
    (1.35*\x,4*\y) node[above] {\footnotesize 
    $x_3$}; 
    
\draw
    (3*\x,0.75*\y) node[right] {\footnotesize $\Delta_0$}
    (3*\x,1.5*\y) node[right] {\footnotesize $\Delta_1$}
    (3*\x,2.25*\y) node[right] {\footnotesize $\Delta_2$};
    
\draw 
    (0*\x,0*\y) node[above] {} 
    (4*\x,4*\y) node[above] {}; 
} 
\caption{A splitting pattern witnessing $\dhII 3 Y\not\subseteq H$} 
\label{figure: splitting pattern}
\end{figure}
}
To show~\ref{3d OGA no complete subgraph}, 
suppose that $Y\subseteq X$ and $|Y|=5$. 
By considering the splitting nodes of pairs in $Y$ from bottom to top, it is easy to check that there exist $x_0,\, x_1,\, x_2,\, x_3\in Y$ with 
$\Delta_0<\Delta_1<\Delta_2$, where 
$\Delta_i:=\Delta(x_i, x_{i+1})$ 
for $i\leq 2$ (see Figure~\ref{figure: splitting pattern}).
Let $i,j$ be such that $i<j\leq 2$ and
$x_3{\left(\Delta_i-1\right)}=x_3{\left(\Delta_j-1\right)}.$
Since $(x_i,x_j,x_3)\not\in H$, we have $\dhII 3 Y\not\subseteq H$.
\end{example}

We have not studied whether there exists a $\ddim$-hypergraph with the previous properties for any $\ddim\geq 4$. 
In the previous example, 
the set of all $x\in {}^\kappa2$ with $x(\alpha)=0$ for all even $\alpha<\kappa$ is a $\kappa$-perfect $H$-independent set for all $\ddim\geq 3$.
The next example shows that
one can strengthen the non-existence of a $\kappa$-coloring to the non-existence of a large independent set. 

\begin{example}
\label{3 dim OGA fails He}
For any ordinal $\ddim\geq 3$, there exists 
a product-open 
$\ddim$-dihypergraph $H\in\defsetsk$ on ${}^\kappa \kappa$ 
with the properties:
\todog{$H_3$ will be a clopen subset of $\dhK\ddim{}$, but $H_\ddim$ may not even be a closed subset of $\dhK\ddim{}$ (in the product or in the box topology)}
\begin{enumerate-(a)}
\item\label{He 3d OGA no independent set}
$\dhII\ddim Y\cap H\neq\emptyset$ for any $Y$ of size $\kappa^+$.
\item\label{He 3d OGA no complete subgraph}
$\dhII\ddim Y\not\subseteq H$ for any $Y$ of size $\kappa^+$.
\end{enumerate-(a)}
For $\ddim=3$, $H$ can be chosen to be a hypergraph.
Note that~\ref{He 3d OGA no independent set} implies that 
there is no $H$-independent set of size $\kappa^+$ and hence
$H$ does not admit a $\kappa$-coloring. 

It suffices to provide an example for $\ddim=3$, since one can then apply a projection 
(see Lemma~\ref{sublemma: ODD for different dimensions}). 
We follow \cite{HeHigherDimOCA}*{Example~1.3}.\footnote{This is also a special case of an argument of Blass \cite{blass1981partition}*{page~273} with a stronger conclusion.} 
For any 
$\bar x\in\dhII 3{{}^\kappa\kappa}$, 
let $\bar x^{0}\lexleq \bar x^{1}\lexleq \bar x^{2}$ list $\ran(\bar{x})$ in lexicographic order 
and define
$$ \bar x\in H\ :\Longleftrightarrow\   
\Delta{\left(\bar x^{0},\bar x^{1}\right)}
\leq
\Delta{\left(\bar x^{1},\bar x^{2}\right)}.$$ 
Note that $H\in \defsetsk$ is an open hypergraph. 

We first show \ref{He 3d OGA no complete subgraph}. 
Suppose that $\dhII 3 Y\subseteq H$. 
We claim that the lexicographic order on $Y$ is a well-order with order type at most $\kappa+1$. 
All splitting nodes of $T(Y)$ are compatible and thus lie on a branch $z$. 
Since $\dhII 3 Y\subseteq H$, we have $y\lexleq z$ for all $y\in Y$. 
Therefore for any splitting node $s\subsetneq z$ of $T(Y)$,  
the order type of the set of all $y\in Y\setminus \{z\}$ with $y\wedge z=s$ 
is an ordinal ${<}\kappa$. 
Since $z$ may be in $Y$, the order type of $Y$ is at most $\kappa+1$. 
Hence $|Y|\leq\kappa$.

For \ref{He 3d OGA no independent set},
suppose that $\dhII\ddim Y\cap H=\emptyset$.
Then $\Delta{\left(\bar x^{0},\bar x^{1}\right)}>
\Delta{\left(\bar x^{1},\bar x^{2}\right)}$
for all
$\bar x\in\dhII 3Y$. 
We claim that the reverse lexicographic order on $Y$ is a well-order with order type at most $\kappa+1$. 
As before, all splitting nodes of $T(Y)$ lie on a branch $z$. 
If $s\subsetneq z$ is a splitting node, then exactly one element of $Y$ splits off at $s$. 
Therefore the order type is at most $\kappa+1$, and hence $|Y|\leq\kappa$. 
\end{example}
 
We have not studied whether there exists a $\ddim$-hypergraph with the properties of the previous example for any $\ddim\geq 4$. 
He \cite{HeHigherDimOCA}*{Theorem 4.3}, Di Prisco and Todor\v cevi\'c  \cite{di1998perfect}*{Sections 3 \& 4} 
showed for 
dimensions $3\leq n<\omega$ 
that open $n$-hypergraphs with additional structural properties 
either admit a countable coloring or contain a perfect complete subhypergraph. 


\subsection{\texorpdfstring{$\Diamondi\kappa$}{Diamond\^{}i\_kappa}}
\label{subsection: diamondi}
In this subsection, we define the combinatorial principle $\Diamondi\kappa$ 
that is used for some of our results, for instance Theorem~\ref{theorem: ODD ODDI}, 
and discuss its relationship with some other versions of $\Diamond_\kappa$.
It is a weak version of $\Diamond_\kappa$ that holds at all inaccessible cardinals. 

\begin{definition}
\label{diamondi def}
Let $\kappa$ be an infinite cardinal and let $\gamma\geq 2$.
A \emph{$\Diamondi{\kappa,\gamma}(\cof)$-sequence} (respectively \emph{$\Diamondi{\kappa,\gamma}(\stat)$-sequence})  is a sequence
$\langle A_\alpha:\alpha<\kappa\rangle$ 
of sets $A_\alpha\subseteq{}^\alpha\gamma$ such that 
\begin{enumerate-(a)}
\item\label{diamondi def 1}
$|A_\alpha|<\kappa$ for all $\alpha<\kappa$, and 
\item\label{diamondi def 2}
for all $y\in{}^{\kappa}\gamma$, the set
$\{\alpha<\kappa:\:y\restr\alpha\in A_\alpha\}$
is cofinal (stationary) in $\kappa$. 
\end{enumerate-(a)}
Let 
\index{diamond!weak cofinal version\idf$\Diamondi{\kappa,\gamma}(\cof)$}%
\index{diamond!weak stationary version\idf$\Diamondi{\kappa,\gamma}(\stat)$}
$\Diamondi{\kappa,\gamma}(\cof)$ (respectively, $\Diamondi{\kappa,\gamma}(\stat)$) 
denote the statement that a $\Diamondi{\kappa,\gamma}(\cof)$-sequence (respectively, $\Diamondi{\kappa,\gamma}(\stat)$-sequence) exists. 
Let $\Diamondi\kappa$ 
\index{diamond!weak a@weak version\idf$\Diamondi\kappa$}
denote the statement $\Diamondi{\kappa,\kappa}(\cof)$.
Similarly, 
a $\Diamondi\kappa$-sequence 
\index{sequences!Diamondi sequence@$\Diamondi\kappa$-sequence\idf}
is a $\Diamondi{\kappa,\kappa}(\cof)$-sequence.
\end{definition} 

We shall see in the next lemma
that all these principles are equivalent as long as 
$\kappa$ is an uncountable regular cardinal 
and 
$2\leq \gamma\leq\kappa$.
Note that $\Diamondi{\kappa,2}(\stat)$ is exactly the 
principle $\Dell\kappa$  introduced in \cite{ShelahModelsIV}*{Definition~2.10}. 

\begin{lemma} 
\label{versions of diamondi are equivalent} 
Suppose that $\kappa$ is an uncountable regular cardinal. 
Then $\Diamondi{\kappa,\beta}(\varphi)$ and $\Diamondi{\kappa,\gamma}(\psi)$ are equivalent 
for all $\beta,\gamma$ with 
$2\leq \beta,\gamma\leq \kappa$ 
and all $\varphi, \psi\in \{\cof, \stat\}$. 
\end{lemma} 
\begin{proof} 
It is clear that 
$\Diamondi{\kappa,\gamma}(\stat)\Longrightarrow\Diamondi{\kappa,\gamma}(\cof)$
holds for all ordinals $\gamma\geq 2$.
\begin{claim*}
If $2\leq\beta<\gamma\leq\kappa$, then 
$\Diamondi{\kappa,\gamma}(\stat)\Longrightarrow\Diamondi{\kappa,\beta}(\stat)$.
\end{claim*}
\begin{proof}
Suppose that 
$\langle A_\alpha:\alpha<\kappa\rangle$ 
is a $\Diamondi{\kappa,\gamma}(\stat)$-sequence.
Then $\langle A_\alpha\cap{}^\alpha\beta:\alpha<\kappa\rangle$ 
is a $\Diamondi{\kappa,\beta}(\stat)$-sequence.
\end{proof}

\begin{claim*}
$\Diamondi{\kappa,2}(\stat)\Longrightarrow\Diamondi{\kappa,\kappa}(\stat)$.
\end{claim*}
\begin{proof}
This is similar to 
the argument 
sketched in 
\cite{KunenBook2013}*{Exercise III.7.9}. 
Suppose $\langle A_\alpha:\alpha<\kappa\rangle$ is a 
$\Diamondi{\kappa,2}(\stat)$-sequence.
Let $f$ be any bijection between $\kappa\times\kappa$ and $\kappa$. 
For all
$\alpha<\kappa$, let 
$\mathcal A_\alpha:=\big\{s^{-1}(\{1\}):
s\in A_\alpha
\big\}$,%
\footnote{Thus, $\mathcal A_\alpha$ is a subset of $\pwrset(\alpha)$, and
 $A_\alpha$ consists of the characteristic functions of elements of $\mathcal A_\alpha$.}
and let
\[\mathcal B_\alpha=
\left\{f^{-1}{\left(S \right)}: S\in \mathcal A_\alpha\right\}
.\]
Using the fact that $C=\{\alpha<\kappa:f(\alpha\times\alpha)=\alpha\}$ is a club subset of $\kappa$,
it is easy to see that 
for all $Y\subseteq\kappa\times\kappa$, 
the set of ordinals $\alpha<\kappa$ with 
$Y\cap(\alpha\times\alpha)\in \mathcal B_\alpha$ is stationary in $\kappa$.
For all $\alpha<\kappa$, 
let $B_\alpha:=\mathcal B_\alpha\cap {}^\alpha\alpha$.
Then 
$\langle B_\alpha:\alpha<\kappa\rangle$ is a $\Diamondi\kappa(\stat)$-sequence.
Indeed, if $x\in{}^\kappa\kappa$, then 
for a club set of ordinals $\alpha<\kappa$ 
we have 
$\ran(x\restr\alpha)\subseteq\alpha$,
or equivalently,
$x\restr\alpha=x\cap(\alpha\times\alpha)\in{}^\alpha\alpha$.\footnote{This step uses that $\kappa$ is regular.} 
Thus, the set
$\{\alpha<\kappa: x\restr\alpha\in B_\alpha\}$ is a stationary subset of~$\kappa$.
\end{proof}

It now suffices show that 
$\Diamondi{\kappa,2}(\cof)\Longrightarrow\Diamondi{\kappa,2}(\stat)$.
First, note that 
$\Diamondi{\kappa,2}(\stat)$ 
always holds 
if $\kappa$ is inaccessible, since
$\langle {}^\alpha 2:\alpha<\kappa\rangle$ is a
$\Diamondi{\kappa,2}(\stat)$-sequence.
In the case that $\kappa$ is not inaccessible, the implication 
$\Diamondi{\kappa,2}(\cof)\Longrightarrow\Diamondi{\kappa,2}(\stat)$
follows from the next claim. 
It is a 
variant 
of the Proposition in \cite{MatetCorrigendumDiamond} whose proof works verbatim in this setting. 

\begin{claim*}
Suppose that $\lambda<\kappa$ is an infinite cardinal such that $2^\lambda\geq\kappa$.
Suppose that there exists a sequence
$\langle A_\alpha:\alpha<\kappa\rangle$ 
of sets $A_\alpha\subseteq{}^\alpha 2$ such that 
\begin{enumerate-(a)}
\item
\label{diamondi weakest a}
$|A_\alpha|<\kappa$ for all $\alpha<\kappa$, and 
\item
\label{diamondi weakest b}
for all $y\in{}^{\kappa}2$, there exists $\alpha\geq\lambda$ 
such that $y\restr\alpha\in A_\alpha$.
\end{enumerate-(a)}
Then $\Diamondi{\kappa,2}(\stat)$ holds.
\end{claim*}
Lemma~\ref{versions of diamondi are equivalent} follows immediately from the previous claims.
\end{proof} 

\begin{lemma} \ 
\label{diamondi claim}
\begin{enumerate-(1)}
\item\label{diamondi diamond inacc}
If $\Diamond_\kappa$ holds or $\kappa$ is inaccessible, then $\Diamondi\kappa$ holds.
\item\label{diamondi kappa^<kappa}
$\Diamondi\kappa$ implies that $\kappa^{<\kappa}=\kappa$.
\end{enumerate-(1)}
\end{lemma}
\begin{proof}
If 
$\langle a_\alpha\in{}^\alpha 2:\alpha<\kappa\rangle$ 
is a $\Diamond_\kappa$-sequence, then
$\big\langle \{a_\alpha\}:\alpha<\kappa\big\rangle$ 
is a $\Diamondi{\kappa,2}(\stat)$-sequence.
If $\kappa$ is inaccessible, then 
$\langle {}^\alpha 2:\alpha<\kappa\rangle$ 
is a $\Diamondi{\kappa,2}(\stat)$-sequence.
Under both of these assumptions, $\kappa$ is an uncountable regular cardinal
and therefore
$\Diamondi{\kappa,2}(\stat)$ implies
$\Diamondi\kappa=\Diamondi{\kappa,\kappa}(\cof)$
by Lemma~\ref{versions of diamondi are equivalent}.

To prove~\ref{diamondi kappa^<kappa}, suppose that 
$\langle A_\alpha:\alpha<\kappa\rangle$ is a $\Diamondi\kappa$-sequence. 
It is sufficient to show that
${}^{<\kappa}\kappa\subseteq
\bigcup_{\alpha<\kappa}T(A_\alpha)
$ 
since the latter set 
has size $\kappa$. 
Suppose $s\in{}^{<\kappa}\kappa$.
Let $x$ be an arbitrary element of ${}^\kappa\kappa$ extending $s$,
and let 
$\alpha\geq\lh(s)$ be such that $x\restr\alpha\in A_\alpha$.
Then $s\subseteq x\restr\alpha$ and hence 
$s\in{}T(A_\alpha)$. 
\end{proof}

\begin{remark}
\label{diamondi omega}
Suppose that either $\kappa=\omega$ or that $\kappa$ is a singular strong limit cardinal. 
Then
$\Diamondi{\kappa,2}(\stat)$ holds 
since $\langle {}^\alpha 2:\alpha<\kappa\rangle$ 
is a $\Diamondi{\kappa,2}(\stat)$-sequence,
but $\Diamondi\kappa=\Diamondi{\kappa,\kappa}(\cof)$ is false.
To show the latter statement, first note that
that $\Diamondi\kappa$ implies $\kappa$ is regular by
Claim~\ref{diamondi claim}~\ref{diamondi kappa^<kappa}.
Seeking a contradiction, suppose that $\Diamondi\omega$ holds and let
$\langle A_n:n<\omega\rangle$ be a $\Diamondi\omega$-sequence.
Define $x\in{}^\omega\omega$
by letting
$x(n):=\sup\{\ran(a):a\in A_{n+1}\}+1$
for each $n<\omega$.
%
Take $m<\omega$ with
$x\restr m\in A_m$.
Then 
\[x(m-1)>\sup\{\ran(a):a\in A_{m}\}\geq\sup\left(\ran(x\restr m)\right)\geq x(m-1),\]
a contradiction.
\end{remark}

\begin{remark} 
\label{diamondi remark} \ 
\begin{enumerate-(1)}
\item\label{diamondi remark 1} For $\kappa=\lambda^+\geq\aleph_1$,
Definition~\ref{diamondi def}~\ref{diamondi def 1} 
is equivalent to 
$|A_\alpha|\leq\lambda$ for all $\alpha<\kappa$.
Therefore $\Diamondi\kappa$ is equivalent to $\Diamond_\kappa$
by a result of Kunen's 
\cite{JensenKunenDiamond}.%
\footnote{See also \cite{KunenBook2013}*{Theorem~III.7.8 and Exercise~III.7.16}.}%
\item\label{diamondi remark 2} For $\kappa=\lambda^+\geq\aleph_2$,
$\kappa^{<\kappa}=\kappa$ is equivalent to $\Diamond_\kappa$ 
\cite{ShelahDiamonds} 
and therefore also to $\Diamondi\kappa$. 
\item\label{diamondi remark 3} 
For inaccessible cardinals $\kappa$, 
$\Diamond_\kappa$ and $\Diamondi\kappa$ are not in general equivalent, 
since the failure of $\Diamond_\kappa$ at the least inaccessible cardinal $\kappa$ is consistent \cite{GolshaniDiamondInaccessible}. 
\end{enumerate-(1)}
\end{remark}

To the best knowledge of the authors, it is not known whether 
any of the assumptions
$\kappa^{<\kappa}=\kappa$,
$\Diamondi\kappa$
and $\Diamond_\kappa$ are equivalent 
when $\kappa$ is weakly inaccessible but not inaccessible.

The next lemma describes a
$\delta$-dimensional equivalent version 
\index{diamond!weak z@weak $\ddim$-dimensional version\idf}
of $\Diamondi\kappa$ for ordinals $2\leq\delta<\kappa$. 
Recall that 
$\dhC\delta X$ denotes the set of constant sequences $\langle c\rangle ^\ddim \in{}^\ddim X$
and 
$\dhK\delta X={}^\delta X-\dhC\delta X$.

\begin{lemma}
\label{diamondi equiv}
Suppose $\kappa$ is a regular uncountable cardinal and $2\leq\delta<\kappa$.
Then the following statements are equivalent:
\begin{enumerate-(1)}
\item\label{diamondi orig}
$\Diamondi\kappa$.
\item\label{diamondi delta dim}
There exists a sequence 
$\langle B_\alpha:\alpha<\kappa\rangle$ of sets 
$B_\alpha\subseteq\dhK\delta{{}^\alpha\kappa}$ such that
\begin{enumerate-(a)}
\item\label{di1}
$|B_\alpha|<\kappa$ for all $\alpha<\kappa$, and
\item\label{di2}
for all $\bar x\in\dhK\delta{{}^\kappa\kappa}$,
there exists $\alpha<\kappa$ 
with 
$\bar x\cwrestr\alpha\in B_\alpha$.%
\footnote{Recall that
$\bar x\cwrestr\alpha:=\langle x_i\restr\alpha:i<\delta\rangle$ 
for any sequence 
$\bar x=\langle x_i:i\in\delta \rangle\in{}^\delta{({}^\kappa\kappa)}$ and $\alpha<\kappa$.}
\end{enumerate-(a)}
\end{enumerate-(1)}
\end{lemma}
\begin{proof}
\ref{diamondi delta dim} $\Rightarrow$ \ref{diamondi orig}:
Suppose that $\langle B_\alpha:\alpha<\kappa\rangle$ is a sequence as in~\ref{diamondi delta dim}. 
For all $\alpha<\kappa$, let 
$
A_\alpha:=\{s_0: \langle s_\beta:\beta<\delta\rangle\in B_\alpha\}.
$
We claim that $\langle A_\alpha:\alpha<\kappa\rangle$ is a $\Diamondi\kappa$-sequence. 
To see this, suppose $x\in{}^\kappa\kappa$ and $\gamma<\kappa$.  
It suffices to find $\alpha<\kappa$ such that $\alpha>\gamma$
and $x\restr\alpha\in A_\alpha$.
Take $\bar x=\langle x_\beta:\beta<\delta\rangle$ in 
$\dhK\delta{{}^\kappa\kappa}$ 
with $x_0=x$ and
$\bar x\cwrestr\gamma\in\dhC\delta{{}^\gamma\kappa}$.
By~\ref{diamondi delta dim}\ref{di2}, there exists $\alpha<\kappa$ such that 
$\bar x\cwrestr\alpha\in B_\alpha$.
Then $x\restr\alpha=x_0\restr\alpha\in A_\alpha$, 
and $\alpha>\gamma$ since 
$\bar x\cwrestr\gamma\in\dhC\delta{{}^\gamma\kappa}$ 
but
$\bar x\cwrestr\alpha\in B_\alpha\subseteq\dhK\delta{{}^\alpha\kappa}$. 

\ref{diamondi orig} $\Rightarrow$ \ref{diamondi delta dim}:
Let $\langle A_\alpha:\alpha<\kappa\rangle$ be a $\Diamondi\kappa$-sequence.
For every
$\gamma\leq\kappa$, let
$g_\gamma:
{}^{\delta\cdot\gamma}\kappa
\to
{}^\delta({}^\gamma\kappa)$
be the bijection 
which maps each
$s\in{}^{\delta\cdot\gamma}\kappa$
to the unique element
$g_\gamma(s)=
\langle t_\beta:\beta<\delta\rangle$ of 
${}^\delta({}^\gamma\kappa)$
with
$t_\beta(\alpha)=s(\delta\cdot\alpha+\beta)$
for all $\alpha<\gamma$ and $\beta<\delta$.
For each $\gamma<\kappa$, let
\[
B_\gamma=
\bigcup_{\beta<\delta}
\left\{
g_\gamma\big(s\restr(\delta\cdot\gamma)\big):
s\in A_{\delta\cdot\gamma+\beta}
\right\}.
\]
We claim that $\langle B_\gamma:\gamma<\kappa\rangle$ satisfies~\ref{diamondi delta dim}.
Since $\kappa$ is regular,
$|B_\gamma|<\kappa$ for all $\gamma<\kappa$.
To see~\ref{diamondi delta dim}\ref{di2},
take $\bar x\in\dhK\delta{{}^\kappa\kappa}$ and let $z:=g_\kappa^{-1}(\bar x)$.
Since $z\in{}^\kappa\kappa$ and $\langle A_\alpha:\alpha<\kappa\rangle$ is a $\Diamondi\kappa$-sequence,
there exist ordinals
$\gamma<\kappa$  
and $\beta<\delta$ such that
$z\restr(\delta\cdot\gamma+\beta)\in A_{\delta\cdot\gamma+\beta}$.
Then 
$
\bar x\cwrestr\gamma=
g_\kappa(z)\cwrestr\gamma=
g_\gamma\big(z\restr(\delta\cdot\gamma)\big)\in
B_\gamma.
$
\end{proof}

\subsection{Strong variants} 
\label{subsection: ODD ODDI ODDH}

In this subsection, we study the relationships between $\ODD\kappa H$ 
and its strengthenings
given in the next definition,
for relatively box-open dihypergraphs $H$ on ${}^\kappa\kappa$.
\todol{These introductory sentences are slightly more detailed.}
These variants,
ordered by increasing strength, 
are obtained by replacing the continuous homomorphism with a map with stronger properties.
Recall that a dihypergraph $H$ is defined to be \emph{relatively box-open} if it is box-open on its domain $\domh H$.\footnote{See Definition~\ref{def: relatively box-open}.}

\begin{definition} 
Suppose that $2\leq\ddim\leq\kappa$ and 
$H$ is a $\ddim$-dihypergraph on  ${}^\kappa\kappa$.
\index{open dihypergraph dichotomy!variant where homomorphisms are!A@injective\idf $\ODDI\kappa H$}%
\index{open dihypergraph dichotomy!variant where homomorphisms are!B@homeomorphisms\idf $\ODDH\kappa H$}%
\index{open dihypergraph dichotomy!variant where homomorphisms are!C@homeomorphisms onto closed\newline\hspace{32 pt}sets\idf $\ODDC\kappa H$}%
\label{def: ODD variants}
\begin{equation*}
\parbox{\dimexpr\linewidth-60pt}{%
\hypertarget{ODDI def}{}%
\strut
{$\ODDI\kappa H$:}
Either $H$ admits a $\kappa$-coloring, or 
there exists 
an \emph{injective} continuous homomorphism
from~$\dhH{\ddim}$ to~$H$.
\strut
}
\end{equation*}
\begin{equation*}
\parbox{\dimexpr\linewidth-60pt}{%
\hypertarget{ODDH def}{}%
\strut
{$\ODDH\kappa H$:}
Either $H$ admits a $\kappa$-coloring, or 
there exists 
a homomorphism 
from~$\dhH{\ddim}$ to~$H$
which is
\emph{a homeomorphism} from ${}^\kappa\ddim$ onto its image.
\strut
}
\end{equation*}
\begin{equation*}
\parbox{\dimexpr\linewidth-60pt}{%
\hypertarget{ODDC def}{}%
\strut
{$\ODDC\kappa H$:}
Either $H$ admits a $\kappa$-coloring, or 
there exists 
a homomorphism 
from~$\dhH{\ddim}$ to~$H$
which is
\emph{a homeomorphism} from ${}^\kappa\ddim$ onto 
a \emph{closed} subset of~${}^\kappa\kappa$.
\strut
}
\end{equation*}
\end{definition}

Note that if $H$ and $I$ are box-open $\ddim$- and $c$- dihypergraphs with $H\equivif I$ on a subset $X$ of ${}^\kappa\kappa$, respectively,
then $\ODDI\kappa H\Longleftrightarrow\ODDI\kappa I$ and 
$\ODDH\kappa H\Longleftrightarrow\ODDH\kappa I$
by Remark~\ref{remark: ODD injective subsequences}.%
\footnote{If $c=d$ and $I\subseteq H$, then the assumption that $H$ is box-open may be omitted, since in this case any homomorphism from $\dhH\ddim$ to $I$ is a homomorphism to $H$ as well.}

We first show that all of the dichotomies in the previous definition 
are in fact equivalent to $\ODD\kappa H$ for $\ddim<\kappa$ dimensional relatively box-open dihypergraphs $H$,
in the next theorem and corollary.
\todol{New introductory paragraph describing the results in this subsection. These sentences were moved or copied over from later paragraphs in the middle of the subsection (and then edited).}
We study the $\ddim=\kappa$ case
in the rest of the subsection. 
These results are summarized in Figure~\ref{figure: ODD variants}.
In particular, we show that $\ODDI\kappa H$ 
is equivalent to
$\ODD\kappa H$ assuming $\Diamondi\kappa$.%
\footnote{Recall that $\Diamondi\kappa$ implies $\kappa>\omega$ by Remark~\ref{diamondi omega}.}
However, 
$\ODDI\omega H$ and $\ODDH\kappa{I}$ fail for some box-open definable dihypergraphs $H$ on ${}^\omega\omega$ and $I$ on ${}^\kappa\kappa$.
At the end of the subsection,
we show that some of the dichotomies
are equivalent for relatively box-open $\kappa$-dihypergraphs with additional properties. 
These results are summarized in 
Figure~\ref{figure: ODD variants special cases}.

The next lemma characterizes the existence of a homomorphism as in $\ODDH\kappa H$ via order homomorphisms.

\begin{lemma}
\label{lemma: perp-preserving order homomorphisms and homeomorphisms}
Suppose
that $2\leq\ddim\leq\kappa$,
$X\subseteq{}^\kappa\kappa$
and $H$ is a $\ddim$-dihypergraph on  ${}^\kappa\kappa$.
\begin{enumerate-(1)}
\item\label{ppohh homeo}
If $\iota$
is a $\perp$-preserving order homomorphism for $(X,H)$,
then $[\iota]$ is a homomorphism from $\dhH\ddim$ to $H\restr X$ which is a homeomorphism from ${}^\kappa\ddim$ onto its image.%
\item\label{homeo ppohh}
If $H\restr X$ is box-open on $X$ and there exists a homomorphism $f$ from $\dhH\ddim$ to $H\restr X$ which is a homeomorphism from ${}^\kappa\ddim$ onto its image,
then there exists a continuous $\perp$-preserving
order homomorphism $\iota$ for $(X,H)$ 
and a continuous strict $\wedge$-homomorphism%
\footnote{See Definition~\ref{def: wedge-homomorphism}. Note that $e:{}^{<\kappa}\ddim\to{}^{<\kappa}\kappa$  is a strict $\wedge$-homomorphism if and only if it is strict order preserving and $e(t\conc\langle \alpha\rangle) \wedge e(t\conc\langle \beta\rangle) =e(t)$ for all $t\in{}^{<\kappa}\ddim$ and $\alpha<\beta<\ddim$. It is easy to see that such a map $e$ is $\perp$-preserving, and hence $[e]$ is a homeomorphism onto its image by Lemma~\ref{[e] continuous}~\ref{[e] homeomorphism}.} 
$e:{}^{<\kappa}\ddim\to{}^{<\kappa}\ddim$
with $[\iota]=f\comp [e]$.
\todog{This lemma is only used in this subsection. \ref{homeo ppohh} used only in the proof of Prop~\ref{ODDH fails for D kappa}.}
\end{enumerate-(1)}
\end{lemma}
\begin{proof}
For \ref{ppohh homeo}, 
$[\iota]$ is a homeomorphism from ${}^{\kappa}\ddim$ onto its image
by Lemma~\ref{[e] continuous}~\ref{[e] homeomorphism}.
It is a homomorphism from $\dhH\ddim$ to $H\restr X$
by Lemma~\ref{homomorphisms and order preserving maps}~\ref{hop 2}. 

The argument for~\ref{homeo ppohh} is similar to the proof of Lemma~\ref{homomorphisms and order preserving maps}~\ref{hop 1}. 
We construct continuous strict order preserving maps
$\iota:{}^{<\kappa}\ddim\to{}^{<\kappa}\kappa$ 
and 
$e:{}^{<\kappa}\ddim\to{}^{<\kappa}\ddim$
with the following properties for all $t\in{}^{<\kappa}\ddim$ and 
$\alpha<\beta<\ddim$:
\begin{enumerate-(i)}
\item\label{homeo ppoh proof 0}
$e(t)\conc\langle \alpha\rangle\subseteq e(t\conc\langle \alpha\rangle)$. 
\item\label{homeo ppoh proof 1}
$f(N_{e(t)})\subseteq N_{\iota(t)}$.
\item\label{homeo ppoh proof perp}
$\iota(t\conc\langle \alpha\rangle)\perp\iota(t\conc\langle \beta\rangle)$.
\item\label{homeo ppoh proof 2} 
$\prod_{\gamma<\ddim}(N_{\iota(t\conc\langle \gamma \rangle)}\cap X)\subseteq H$.
\end{enumerate-(i)}
Construct $\iota(t)$ and $e(t)$ by recursion on $\lh(t)$.
Let $\iota(\emptyset):=e(\emptyset)=\emptyset$.
Let $t\in{}^{<\kappa}\ddim\,\setminus\{\emptyset\}$, and suppose that $\iota(u)$ and $e(u)$ have been defined for all $u\subsetneq t$.
If $\lh(t)\in\Lim$,
let $\iota(t):=\bigcup_{u\subsetneq t}\iota(u)$ and
let $e(t):=\bigcup_{u\subsetneq t}e(u)$. 
Since 
$N_{e(t)}=\bigcap_{u\subsetneq t}N_{e(u)}$
and \ref{hop proof 1} holds for all $u\subsetneq t$, 
we have
\[
f(N_{e(t)})\subseteq
\bigcap_{u\subsetneq t}f(N_{e(u)})\subseteq
\bigcap_{u\subsetneq t}N_{\iota(u)}
=N_{\iota(t)}.
\] 
Suppose $\lh(t)\in\Succ$, and let $u$ be the direct predecessor of $t$.
Construct 
$\iota(u\conc\langle\alpha\rangle)$ 
and
$e(u\conc\langle\alpha\rangle)$ 
for each $\alpha<\ddim$ simultaneously as follows. 
Let 
$x_\alpha\in{}^{\kappa}\ddim$ extend $e(u)\conc\langle \alpha\rangle$ 
for each $\alpha<\ddim$.
Choose
$\iota(u\conc\langle\alpha\rangle)\subsetneq f(x_\alpha)$ extending $\iota(u)$
with
$\ran(f)\cap N_{\iota(u\conc\langle\alpha\rangle)}\subseteq f(N_{e(u)\conc\langle\alpha\rangle})$
for each $\alpha<\ddim$
such that
$\prod_{\alpha<\ddim}(N_{\iota(u\conc\langle\alpha\rangle)}\cap X\subseteq H$.
The former can be achieved since 
$f$ is a homeomorphism onto its image and hence 
$f(N_{e(u)\conc\langle\alpha\rangle})$ is an open subset of $\ran(f)$. 
The latter can be achieved since 
$f$ is a homomorphism from $\dhH\ddim$ to $H\restr X$ and $H\restr X$ is box-open on $X$.
The nodes $\iota(u\conc\langle\alpha\rangle)$ are then pairwise incompatible, since the sets 
$\ran(f)\cap N_{\iota(u\conc\langle\alpha\rangle)}\subseteq f(N_{e(u)\conc\langle\alpha\rangle})$
are pairwise disjoint. 
By continuity of $f$, choose nodes 
$e(u\conc\langle\alpha\rangle)\subsetneq x_\alpha$ 
extending $e(u)\conc\langle\alpha\rangle$
with
$f(N_{e(u\conc\langle\alpha\rangle})\subseteq
N_{\iota(u\conc\langle\alpha\rangle)}$.

Suppose the
maps $\iota$ and $e$ have been constructed.
$e$~is a strict $\wedge$-homomorphism by~\ref{homeo ppoh proof 0}
and $\iota$ is 
$\perp$-preserving continuous map by~\ref{homeo ppoh proof perp}. 
$[\iota]=f\comp[e]$ by~\ref{homeo ppoh proof 1} 
as in the proof of 
Lemma~\ref{homomorphisms and order preserving maps}~\ref{hop 1}.
Hence
$\ran([\iota])\subseteq X$, 
so $\iota$ is an order homomorphism for $H$
by~\ref{homeo ppoh proof 2}.
\end{proof}


\begin{theorem}
\label{theorem: ODD ODDC when ddim<kappa}
Suppose that $2\leq\ddim<\kappa$ and $H$ is a relatively box-open $\ddim$-dihypergraph on ${}^\kappa\kappa$.
If there exists a continuous homomorphism $f$ from $\dhH\ddim$ to $H$, 
then there exists a homomorphism $h$ from $\dhH\ddim$ to $H$ which is a homeomorphism from ${}^\kappa\ddim$ onto a closed subset 
of ${}^\kappa\kappa$ 
and a continuous strict $\wedge$-homomorphism 
$e:{}^{<\kappa}\ddim\to{}^{<\kappa}\ddim$
with $h=f\comp [e]$.
\end{theorem}
\begin{proof}
The theorem will follow from the next claim.
\begin{claim*}
There exists a continuous
$\perp$-preserving
order homomorphism $\theta$ for $H$ 
and a continuous strict $\wedge$-homomorphism 
$e:{}^{<\kappa}\ddim\to{}^{<\kappa}\ddim$
with $[\theta]=f\comp [e]$.
\end{claim*}
\begin{proof}
There exists a continuous order homomorphism $\iota$ for $H$ 
and a continuous strict $\wedge$-homomorphism 
$e_0:{}^{<\kappa}\ddim\to{}^{<\kappa}\ddim$
with $[\iota]=f\comp[e_0]$
by Lemma~\ref{homomorphisms and order preserving maps}~\ref{hop 1}. 
\todog{continuity of $\iota$ is only needed for $\theta$ to be continuous. But the rest of the proof works even if $\theta$ is not continuous}
We define continuous strict order preserving maps
$\theta:{}^{<\kappa}\ddim\to{}^{<\kappa}\kappa$
and $e_1:{}^{<\kappa}\ddim\to{}^{<\kappa}\ddim$
such that 
the following hold
for all $s,t\in{}^{<\kappa}\ddim$ 
and $\alpha<\ddim$:
\begin{enumerate-(i)}
\item\label{hop2 proof 3}
If $s\perp t$, then $\theta(s)\perp \theta(t)$. 
\item\label{hop2 proof 2}
$e_1(t)\conc\langle\alpha\rangle\subseteq e_1(t\conc\langle\alpha\rangle)$.
\item\label{hop2 proof 1}
$\theta(t)= \iota(e_1(t))$.
\end{enumerate-(i)}
Construct $\theta(t)$ and $e_1(t)$ by recursion on $\lh(t)$.
Let $\theta(\emptyset):=e_1(\emptyset)=\emptyset$.
Fix $t\in{}^{<\kappa}\ddim\,\setminus\{\emptyset\}$ and assume that $\theta(u)$ and $e_1(u)$ have been constructed for all $u\subsetneq t$. 
If $\lh(t)\in\Lim$,
let $e_1(t):=\bigcup_{u\subsetneq t}e_1(u)$ 
and 
$\theta(t):= \iota(e_1(t))$.
Suppose that $\lh(t)\in\Succ$ and let $u$ be the direct predecessor of $t$.
Define 
$e_1(u\conc\langle\alpha\rangle)$
and
$\theta(u\conc\langle\alpha\rangle)$
simultaneously for all $\alpha<\ddim$ as follows.
Since
$|[\iota](N_{e(u)\conc\langle\alpha\rangle})|>\kappa$ for all $\alpha<\ddim$ 
by Lemma~\ref{two options in ODD are mutually exclusive}~\ref{me 2},
construct by recursion on $\alpha$ 
a sequence $\langle x_\alpha\in{}^\kappa\ddim:\alpha<\ddim\rangle$ 
such that 
$e_1(u)\conc\langle\alpha\rangle\subseteq x_\alpha$
and
$[\iota](x_\alpha)\neq [\iota](x_\beta)$
for all $\alpha,\,\beta<\ddim$ with $\alpha\neq\beta$.
Since $\ddim<\kappa$,
there exists $\gamma<\kappa$ such that
$\iota(x_\alpha\restr\gamma)\perp \iota(x_\beta\restr\gamma)$
and
$e_1(u)\conc\langle\alpha\rangle\subseteq x_\alpha\restr\gamma$
for all $\alpha,\,\beta<\ddim$ with $\alpha\neq\beta$.
Let 
$e_1(u\conc\langle\alpha\rangle):=x_\alpha\restr\gamma$
and 
$\theta(u\conc\langle\alpha\rangle):=\iota(x_\alpha\restr\gamma)$
for all $\alpha<\ddim$.

Suppose the maps $\theta$ and $e_1$ have been constructed.
$\theta$ is $\perp$-preserving by~\ref{hop2 proof 3}.
$e_1$ is a strict $\wedge$-homomorphism by~\ref{hop2 proof 2}
and hence so is 
$e:=e_0\comp e_1$.
Moreover,
$[\theta]=[\iota]\comp[e_1]=
f\comp [e]$
since 
$\theta=\iota\comp e_1$ by~\ref{hop2 proof 1}
and
$[\iota]=f\comp[e_0]$.
Lastly,
$\theta$ is an order homomorphism for $H$
since
$\ran([\theta])\subseteq\ran(f)\subseteq \domh H$
and
$\prod_{\alpha<\ddim}(N_{\theta(t\conc\langle\alpha\rangle)}\cap \domh H)
\subseteq
\prod_{\alpha<\ddim}(N_{\iota(e_1(t)\conc\langle\alpha\rangle)}\cap \domh H)
\subseteq H$
for all $t\in{}^{<\kappa}\ddim$.
The last two inclusions hold since
$\iota(e_1(t)\conc\langle\alpha\rangle)\subseteq\theta(t\conc\langle\alpha\rangle)$
for all $\alpha<\ddim$
by~\ref{hop2 proof 2} and~\ref{hop2 proof 1}
and since $\iota$ is an order homomorphism for $H$.
\end{proof}
Let $\theta$ be as in the claim.
Then $h:=[\theta]$ is a homomorphism from $\dhH\kappa$ to $H$ which is a homeomorphism from ${}^\kappa\ddim$ onto $\ran(h)$
by Lemma~\ref{lemma: perp-preserving order homomorphisms and homeomorphisms}~\ref{ppohh homeo}.
Since $\ddim<\kappa$,
$\ran(h)$ 
is a closed subset of ${}^\kappa\kappa$
by Lemma~\ref{lemma: [e] closed map}~\ref{[e] closed when ddim<kappa}.
\end{proof}

\begin{corollary}
\label{cor: ODD ODDC when ddim<kappa}
Suppose that $2\leq\ddim<\kappa$ and $H$ is a relatively box-open $\ddim$-dihypergraph on~${}^\kappa\kappa$. Then
$\ODDC\kappa H$
$\Longleftrightarrow$
$\ODDH\kappa H$
$\Longleftrightarrow$
$\ODDI\kappa H$
$\Longleftrightarrow$
$\ODD\kappa H$.
\end{corollary}

In the rest of this subsection, we study the relationships between $\ODD\kappa H$ 
and its strengthenings
in Definition~\ref{def: ODD variants} 
for $\ddim=\kappa$
and relatively box-open $\kappa$-dihypergraphs $H$.
Our results are summarized in Figure~\ref{figure: ODD variants}.%
\footnote{Blue arrows are clickable and lead to the results where the implications are proved.} 
Note that all implications from left to right in 
Figure~\ref{figure: ODD variants} 
hold for all dihypergraphs by 
definition. 
Recall that $\Diamondi\kappa$ implies $\kappa^{<\kappa}=\kappa>\omega$ 
and that it always holds when $\kappa\geq \aleph_2$ is a successor cardinal or inaccessible.\footnote{See Lemma~\ref{diamondi claim} and Remarks~\ref{diamondi omega} and \ref{diamondi remark}~\ref{diamondi remark 2}.}
\todon{In the next figure, ``abovecaptionskip'' (distance between figure and caption) and ``intextsep'' (distance between text and figure) were set LOCALLY. Remove before submitting to journal?}%

{
\hypersetup{linkcolor=VeryDarkBlueNode}
\setlength{\abovecaptionskip}{0pt}
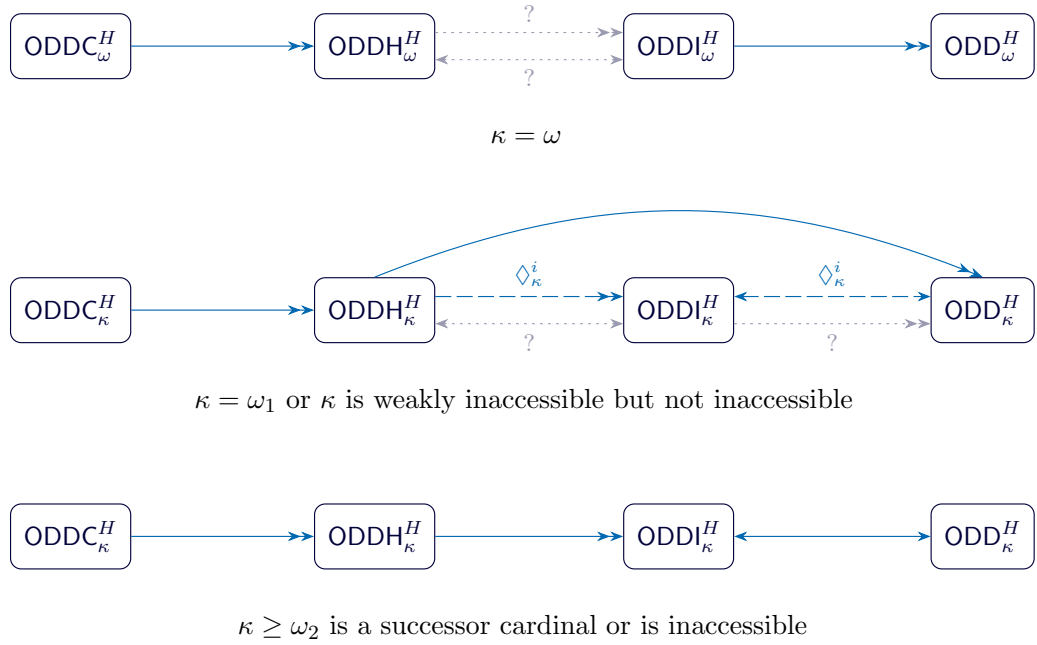
\begin{figure}[t] 
\begin{tikzpicture}[scale=1]
\tikzstyle{theory1}=[draw,rounded corners,scale=.90, minimum width=8.0mm, minimum height=8.0mm,line width=\widthline,VeryDarkBlueNode]
\tikzstyle{info}=[rounded corners, inner sep= 2pt, align=left,scale=0.82]
\tikzstyle{label}=[rounded corners, inner sep= 2pt, align=left,scale=0.95]

\tikzstyle{medarrow}=[>=Stealth, line width=\widtharrow, scale=1.8] 
\tikzstyle{consarrow}=[>=Stealth, line width=\widtharrow, scale=1.8, dash pattern=on 6pt off  2pt]
\tikzstyle{questionarrow}=[>=Stealth, line width=\widtharrow, scale=1.8, dash pattern=on 1pt off  2pt]

\tikzstyle{infoarrow}=[>=Stealth, scale=1] 



%


 \node (om1) at (0,10.2) [theory1] {\hyperlink{ODDC def}{\parbox[c][.7cm]{1.5cm}{$\,\ODDC\omega H$}}};
 \node (om2) at (4,10.2) [theory1] {\hyperlink{ODDH def}{\parbox[c][.7cm]{1.5cm}{$\,\ODDH\omega H$}}};
 \node (om3) at (8,10.2) [theory1] {\hyperlink{ODDI def}{\parbox[c][.7cm]{1.35cm}{$\,\ODDI\omega H$}}};
 \node (om4) at (12,10.2) [theory1] {\hyperref[ODD def]{\parbox[c][.7cm]{1.25cm}{$\,\ODD\omega H$}}}; 
\node (l1) at (6,9.0) [label, text width=1.0 cm] {$\kappa=\omega$};

 \node (unc1) at (0,6.7) [theory1] {\hyperlink{ODDC def}{\parbox[c][.7cm]{1.5cm}{$\,\ODDC\kappa H$}}};
 \node (unc2) at (4,6.7) [theory1] {\hyperlink{ODDH def}{\parbox[c][.7cm]{1.5cm}{$\,\ODDH\kappa H$}}};
 \node (unc3) at (8,6.7) [theory1] {\hyperlink{ODDI def}{\parbox[c][.7cm]{1.35cm}{$\,\ODDI\kappa H$}}};
 \node (unc4) at (12,6.7) [theory1] {\hyperref[ODD def]{\parbox[c][.7cm]{1.25cm}{$\,\ODD\kappa H$}}}; 
\node (l2) at (6,5.5) [label, text width=9.2 cm] {$\kappa=\omega_1$ 
or $\kappa$ is weakly inaccessible but not inaccessible
};

 \node (dia1) at (0,3.7) [theory1] {\hyperlink{ODDC def}{\parbox[c][.7cm]{1.5cm}{$\,\ODDC\kappa H$}}};
 \node (dia2) at (4,3.7) [theory1] {\hyperlink{ODDH def}{\parbox[c][.7cm]{1.5cm}{$\,\ODDH\kappa H$}}};
 \node (dia3) at (8,3.7) [theory1] {\hyperlink{ODDI def}{\parbox[c][.7cm]{1.35cm}{$\,\ODDI\kappa H$}}};
 \node (dia4) at (12,3.7) [theory1] {\hyperref[ODD def]{\parbox[c][.7cm]{1.25cm}{$\,\ODD\kappa H$}}}; 
\node (l3) at (6,2.5) [label, text width=8.0 cm] {$\kappa\geq\omega_2$ is a successor cardinal or is inaccessible
};
\node (i1) at (2.75,1.4) 
[info, text width=10.0 cm] {
Implications for relatively box-open $H$ and $\ddim=\kappa$:
};
\node(a1) at (-1.2, 0.9)[info,BlueArrow]{A};
\node(b1) at (3.30, 0.9)[info, text width=9.0 cm]{\textcolor{BlueArrow}{B:}
$A$ and $B$ are equivalent for all such $H$};
\draw[infoarrow,BlueArrow] (a1) edge[<->] (b1);
\node(a2) at (-1.20, 0.45)[info,BlueArrow]{A};
\node(b2a) at (3.30, 0.45)[info, text width=9.0 cm]{\textcolor{BlueArrow}{B:}
$A$ implies $B$ for all such $H$
but the reverse%
};
\node(b2b) at (2.60,0.00)
[info, text width=9.6 cm] {
implication fails for some
such $H$%
};
\draw[infoarrow,BlueArrow] (a2) edge[->>] (b2a);
\node (i3) at (10.0,0.70)
[info, text width=9.0 cm] {
{\bf \color{BlueArrow} solid arrow:} 
provable for all $\kappa$ with $\kappa^{<\kappa}=\kappa$%
\newline
{\bf \color{BlueArrow} dashed arrow:} 
consistent and follows from the \\assumption in the superscript%
\newline
{\bf \color{GreyBlueNode} dotted arrow:} 
its consistency is an open question%
}; 

 \draw[medarrow, BlueArrow]
([yshift=0mm]unc1.east) edge [->>] 
node {\hyperref[lemma:ODDH and not ODDC]{\phantom{xxx}}} 
([yshift=0mm]unc2.west)
([yshift=0mm]unc2.north) edge [->>,bend left=22] 
node {\hyperlink{ODD and not ODDH}{\phantom{xxx}}} 
([yshift=0mm]unc4.north);
\draw[consarrow, BlueArrow]
([yshift=1mm]unc3.east) edge[<->] 
node[midway,above] {{\footnotesize $\Diamondi\kappa$}} 
node {\hyperlink{ODD and ODDI Diamondi}{\phantom{xxx}}} 
([yshift=1mm]unc4.west)
([yshift=1mm]unc2.east) edge[->>] 
node[midway,above] {{\footnotesize $\Diamondi\kappa$}} 
node {\hyperlink{ODDI and not ODDH Diamondi}{\phantom{xxx}}} 
([yshift=1mm]unc3.west);
 \draw[questionarrow, GreyBlueNode]
([yshift=-1mm]unc3.west) edge[<->] 
([yshift=-1mm]unc2.east)
([yshift=-1mm]unc3.east) edge[->>] 
([yshift=-1mm]unc4.west);
%
 \draw[medarrow, BlueArrow]
([yshift=0mm]dia1.east) edge [->>] 
node {\hyperref[lemma:ODDH and not ODDC]{\phantom{xxx}}} 
([yshift=0mm]dia2.west)
([yshift=0mm]dia3.east) edge [<->] 
node {\hyperlink{ODD and ODDI Diamondi provable}{\phantom{xxx}}} 
([yshift=0mm]dia4.west)
([yshift=0mm]dia2.east) edge[->>] 
node {\hyperlink{ODD and ODDI Diamondi provable}{\phantom{xxx}}} 
([yshift=0mm]dia3.west);
%
\draw[medarrow, BlueArrow]
([yshift=0mm]om1.east) edge [->>] 
node {\hyperref[lemma:ODDH and not ODDC]{\phantom{xxx}}} 
([yshift=0mm]om2.west)
([yshift=0mm]om3.east) edge[->>] 
node {\hyperlink{ODD and not ODDI}{\phantom{xxx}}} 
([yshift=0mm]om4.west);
\draw[questionarrow, GreyBlueNode]
([yshift=-1mm]om2.east) edge[<->] 
([yshift=-1mm]om3.west)
([yshift=1mm]om2.east) edge[->>] 
([yshift=1mm]om3.west);
\end{tikzpicture}
\caption{Variants of the open dihypergraph dichotomy}
\label{figure: ODD variants}
\end{figure}}

The next lemma shows that
in general, $\ODDH\kappa H$ 
does not imply
$\ODDC\kappa H$.

\begin{lemma} 
\label{lemma:ODDH and not ODDC}
There exists a box-open $\kappa$-dihypergraph $H$ on ${}^\kappa\kappa$ such that 
\begin{enumerate-(1)}
\item\label{lemma:ODDH and not ODDC 1}
there is a homomorphism from $\dhH\kappa$ to $H$ which is a homeomorphism from ${}^\kappa\kappa$ onto its image, but
\item\label{lemma:ODDH and not ODDC 2}
there is no continuous homomorphism from $\dhH{\kappa}$ to $H$ with closed image.
\end{enumerate-(1)}
\end{lemma} 
\begin{proof} 
Fix a strict order preserving map $\iota: {}^{<\kappa}\kappa\to {}^{<\kappa}2$ 
\todoo{We replaced $\pi$ by $\iota$ everywhere (since it's not exactly the same map as in Subsection~\ref{subsection: original KLW}).}
with $\iota(t\conc\langle\alpha\rangle)=\iota(t)\conc \langle0\rangle^{\alpha}{}\conc \langle1\rangle$ for all $t\in {}^{<\kappa}\kappa$ and $\alpha<\kappa$,
\todog{The value of $\iota$ is unspecified at $\emptyset$ and at limit lengths. If we choose $\iota$ to be continuous with $\iota(\emptyset)=\emptyset$, we obtain the map defined at the beginning of Subsection~\ref{subsection: original KLW}.} 
and consider 
\index{dihypergraph!induced by a strict order preserving function\idf$\HH_{\iota,X},\,\HH_\iota$}
$\HH_{\iota}:=
\bigcup_{t\in{}^{<\kappa}\ddim}\prod_{\alpha<\ddim} N_{\iota(t\conc\langle\alpha\rangle)}
.$
Since $\iota$ is $\perp$-preserving, $\HH_\iota$ is a box-open dihypergraph on ${}^\kappa\kappa$ and $\iota$ is an order homomorphism for  $({}^\kappa\kappa,\HH_\iota)$.\footnote{In fact, $\HH_\iota$ is the smallest dihypergraph $H$ on ${}^\kappa\kappa$ such that $\iota$ is an order homomorphism for $({}^\kappa\kappa,H)$. 
$\HH_{\iota}$ is the same as the dihypergraph $\HH_{\iota, {}^\kappa\kappa}$ defined in  Remark~\ref{remark: HH_iota}.}
Thus, 
$[\iota]: {}^\kappa\kappa\to {}^\kappa\kappa$ is a 
homomorphism from  $\dhH{\kappa}$ to $\HH_\iota$ which is a homeomorphism onto its image by Lemma~\ref{lemma: perp-preserving order homomorphisms and homeomorphisms}~\ref{ppohh homeo}.
\ref{lemma:ODDH and not ODDC 2} will follow from the next claim.

\begin{claim*} 
If $f: {}^\kappa\kappa\to {}^\kappa\kappa$ is a continuous homomorphism from $\dhH{\kappa}$ to $\HH_\iota$, then 
every element of $\ran(f)$ takes the value $1$ unboundedly often. 
\end{claim*} 
\begin{proof} 
Fix $x\in {}^\kappa\kappa$ and $\xi<\kappa$ and 
let $u:=f(x)\restr\xi$. 
By continuity, pick some $\gamma<\kappa$ such that 
$f(N_{x\restr\gamma})\subseteq N_u$. 
Take $\beta$ with $(x\restr\gamma)\conc\langle\beta\rangle \subseteq x$, 
and choose a hyperedge $\langle x_\alpha:\alpha<\kappa\rangle$ of $\dhH\kappa$
such that 
$x_\beta=x$ and
$(x\restr\gamma)\conc\langle\alpha\rangle\subseteq x_\alpha$ for all $\alpha<\kappa$.
%
Since $\langle f(x_\alpha): \alpha<\kappa \rangle \in \HH_\iota$, 
there exists some $v\in{}^{<\kappa}\kappa$ such that 
$v\conc \langle0\rangle^{\alpha}\conc \langle1\rangle 
\subseteq f(x_\alpha)$ 
for all $\alpha<\kappa$. 
To see that $u\subseteq v$, note that $u\subseteq f(x_\alpha)$ for all $\alpha<\kappa$. 
For $\alpha=0,1$, this implies
$u\compat v\conc \langle1\rangle$ and $u\compat v\conc \langle0\rangle\conc \langle1\rangle$.
Since these split at $v$, we have $u\subseteq v$. 
Since $v\conc \langle0\rangle^{\beta}{}\conc \langle1\rangle \subseteq f(x)$,  we have $f(x)(\eta)=1$ for some $\eta>\lh(u)=\xi$. 
\end{proof} 

We claim that there is no continuous homomorphism $f: {}^\kappa\kappa\to {}^\kappa\kappa$ from $\dhH{\kappa}$ to $\HH_\iota$ with closed range. 
Towards a contradiction, suppose that $f$ is such a homomorphism. 
Let $\langle x_\alpha:\alpha<\kappa\rangle\in\dhH\kappa$ be an arbitrary hyperedge of $\dhH\kappa$.
Since $\langle f(x_\alpha): \alpha<\kappa \rangle \in \HH_\iota$, 
there is some $v\in {}^{<\kappa}\kappa$ such that $v\conc \langle0\rangle^{\alpha}{}\conc \langle1\rangle \subseteq f(x_\alpha)$ for all $\alpha<\kappa$. 
If $\ran(f)$ is closed, then $v\conc \langle0\rangle^{\kappa} \in \ran(f)$. 
But this contradicts the previous claim. 
\end{proof} 

\hypertarget{ODD and not ODDH}{}%
\hypertarget{ODD and not ODDI}{}%
We next show that $\ODD\kappa H$ does not in general imply
$\ODDH\kappa H$
and for $\kappa=\omega$, $\ODD\omega H$ may not even imply $\ODDI\omega H$. 
The following hypergraph\footnote{Recall that a $\ddim$-hypergraph is a $\ddim$-dihypergraph that is closed under permutations of hyperedges.} 
will provide a counterexample to both implications.

\begin{definition}
\label{def: dhD}
Let 
$\dhD\kappa$ denote
\index{hypergraph!of sequences!DD@somewhere dense\idf$\dhD\kappa$}
the $\kappa$-hypergraph
on ${}^\kappa\kappa$
consisting of all somewhere dense sequences 
$\langle x_\alpha:\alpha<\kappa\rangle$.\footnote{I.e., 
$\{x_\alpha:\alpha<\kappa\}\cap N_t$ 
is a dense subset of $N_t$ for some $t\in{}^{<\kappa}\kappa$.
}
\end{definition}

Observe that $\dhD\kappa$ is in $\defsets\kappa$: 
It is in fact a $\kappa$-Borel
subset of the space ${}^\kappa({}^\kappa\kappa)$ with the $\lle\kappa$-box topology.
\todog{By the conventions in this paper,  ${}^\kappa({}^\kappa\kappa)$ is equipped with the box-topology unless otherwise stated!!} 
We first show that $\dhD\kappa$ is box-open and $\ODD\kappa{\dhD\kappa}$ holds.

\begin{proposition}\ 
\label{lemma: dhD}
\begin{enumerate-(1)}
\item\label{dhD box-open}
$\dhD\kappa$ is  box-open on ${}^\kappa\kappa$.
\item\label{ODD dhD}
There is a continuous homomorphism from $\dhH\kappa$ to $\dhD\kappa$.
\end{enumerate-(1)}
\end{proposition}
\begin{proof}
To show \ref{dhD box-open},
let $\bar x=\langle x_\alpha:\alpha<\kappa\rangle\in\dhD\kappa$.
We claim that the box-open neighborhood $U:=\prod_{\alpha<\kappa}N_{x_\alpha\restr\alpha}$
of $\bar x$ is a subset of 
$\dhD\kappa$.
Let $\bar y=\langle y_\alpha:\alpha<\kappa\rangle\in U$.
Take $t\in{}^{<\kappa}\kappa$ such that
$\bar x$ is dense in $N_t$.
Then for any $u\in{}^{<\kappa}\kappa$ with $u\supseteq t$,
there is some $\alpha>\lh(u)$ with $x_\alpha \in N_u$. 
Since $x_\alpha\restr \alpha=y_\alpha\restr \alpha$, we have $y_\alpha\in N_u$. 
Thus, $\bar y$ is also dense in $N_t$.

To show~\ref{ODD dhD}, 
let $\iota:{}^{<\kappa}\kappa\to{}^{<\kappa}\kappa$ 
be any strict order preserving map such that
for all $t\in{}^{<\kappa}\kappa$, 
$\langle \iota(t\conc\langle\alpha\rangle) : \alpha<\kappa\rangle$
enumerates
$\succ{\iota(t)}=\{u\in{}^{<\kappa}\kappa: \iota(t)\subsetneq u\}$.
Then $\iota$ is an order homomorphism for $\dhD\kappa$,
because
any sequence
$\bar y\in
\prod_{\alpha<\kappa}N_{\iota(t\conc\langle\alpha\rangle)}$
is dense in $N_{\iota(t)}$.
By Lemma~\ref{homomorphisms and order preserving maps} \ref{hop 2}, 
$[\iota]:{}^\kappa\kappa\to {}^\kappa\kappa$ is a continuous homomorphism from  $\dhH{\kappa}$ to $\dhD\kappa$.
\end{proof}

The next proposition 
shows that $\ODDH{\kappa}{\dhD\kappa}$ fails, 
and thus $\ODDH{\kappa}{\kappa}({}^\kappa\kappa,\defsets\kappa)$ fails.%
\footnote{The definition of $\ODDH\kappa\kappa({}^\kappa\kappa,\defsets\kappa)$ is analogous to Definition~\ref{def: ODD for classes}.}

\begin{proposition}
\label{ODDH fails for D kappa}
There is no homomorphism from $\dhH\kappa$ to $\dhD\kappa$ which is a homeomorphism from ${}^\kappa\kappa$ onto its image.
\end{proposition}
\begin{proof}
By Lemma~\ref{lemma: perp-preserving order homomorphisms and homeomorphisms}~\ref{homeo ppohh}, it suffices to show that no order homomorphism $\iota$ for $\dhD\kappa$ preserves $\perp$. 
Suppose that $\iota$ is an order homomorphism for $\dhD\kappa$. 
\begin{claim*}
For any  $u\in{}^{<\kappa}\kappa$, 
there exists $\alpha<\beta<\kappa$ 
such that
$\iota(u\conc\langle\alpha\rangle)\compat
\iota(u\conc\langle\beta\rangle)$.
\end{claim*}
\begin{proof}
Suppose that
$\iota(u\conc\langle\alpha\rangle)\perp
\iota(u\conc\langle\beta\rangle)$
for all $\alpha<\beta<\kappa$.
Take a hyperedge $\langle x_\alpha:\alpha<\kappa\rangle$ of $\dhH\kappa$ with $u\conc\langle\alpha\rangle\subseteq x_\alpha$ for all $\alpha<\kappa$ 
and let
$Y=\{[\iota](x_\alpha):\alpha<\kappa\}$.
Then $Y\cap N_{\iota(u\conc\langle\alpha\rangle)}=\{[\iota](x_\alpha)\}$ for all $\alpha<\kappa$.
On the other hand $Y$ is dense in $N_t$ for some $t\in{}^{<\kappa}\kappa$, since $[\iota]$ is a homomorphism from $\dhH\kappa$ to $\dhD\kappa$ by Lemma~\ref{homomorphisms and order preserving maps}~\ref{hop 2}. 
In particular,
there exist $\alpha<\beta<\kappa$ such that
$[\iota](x_\alpha), [\iota](x_\beta)\in N_t$. 
Hence $t$ is compatible with both
$\iota(u\conc\langle\alpha\rangle)$
and
$\iota(u\conc\langle\beta\rangle)$.
Since these are incompatible,  $t\subsetneq\iota(u\conc\langle\alpha\rangle)$.
Therefore $Y$ is also dense in $N_{\iota(u\conc\langle\alpha\rangle)}$,
contradicting the fact that 
$|Y\cap N_{\iota(u\conc\langle\alpha\rangle)}|=1$.
\end{proof}
This completes the proof of the proposition. 
\end{proof}

The next proposition shows that $\ODDI{\omega}{\dhD\omega}$ fails, 
and thus $\ODDI{\omega}{\omega}({}^\omega\omega,\defsets\omega)$ fails. 

\begin{proposition}
\label{ODDI fails for D omega}
There is no injective continuous homomorphism from $\dhHomega\omega$ to $\dhD\omega$.
\end{proposition}
\begin{proof}
Suppose that $f: {}^\omega\omega\to {}^\omega\omega$ is a continuous homomorphism from $\dhHomega\omega$ to $\dhD\omega$. 
We shall construct $x\neq y$ with $f(x)=f(y)$ using the next claim. 
\begin{claim*} 
Suppose that $u,t\in {}^{<\omega}\omega$ and $f(N_u)\cap N_t \neq\emptyset$. 
Then there exist $u',t'\in {}^{<\omega}\omega$ such that $u'\supsetneq u$, $t'\supsetneq t$ and $f(N_{u'})$ is dense in $N_{t'}$. 
\end{claim*} 
\begin{proof} 
Since $f(N_u)\cap N_t \neq\emptyset$ and $f$ is continuous, there exists some $u'\supsetneq u$ with $f(N_{u'})\subseteq N_t$. 
The set $f(N_{u'})=\bigcup_{ i<\omega}f(N_{u'\conc\langle i\rangle})$ is dense in $N_w$ for some $w\in {}^{<\omega}\omega$ since $f$ is a homomorphism from $\dhHomega\omega$ to $\dhD\omega$. 
Since $N_w\cap N_t\neq\emptyset$,
we have $t\compat w$. 
Let $t'\supsetneq t\cup w$. Then $f(N_{u'})$ is dense in $N_{t'}$. 
\end{proof} 

Using the previous claim, we now construct strictly increasing sequences $\langle u_n : n<\omega\rangle$, $\langle v_n : n<\omega\rangle$ and $\langle t_n : n<\omega\rangle$ 
of elements of ${}^{<\omega}\omega$ such that for all $n<\omega$: 
\begin{enumerate-(i)} 
\item 
$u_n \perp v_n$. 
\item 
\label{ODDI fails for D omega claim ii}
$f(N_{u_n})$ and $f(N_{v_n})$ are both dense in $N_{t_n}$. 
\end{enumerate-(i)} 
We first define $u_0$, $v_0$ and $t_0$. 
Fix a hyperedge $\langle x_i:i<\omega\rangle$ of $\dhHomega\omega$ with $\langle i\rangle\subseteq x_i$ for all $i<\omega$.
Since $f$ is a homomorphism from $\dhHomega\omega$ to $\dhD\omega$,
$\langle f(x_i):i<\omega\rangle$ 
is dense in $N_{t}$ for some $t\in {}^{<\omega}\omega$. 
Hence there is some $i<\omega$ with $f(N_{\langle i\rangle})\cap N_t\neq\emptyset$. 
By the previous claim for $\langle  i\rangle$ and $t$, there is some $u_0\supsetneq \langle  i\rangle$ and some $t'_0\supsetneq t$ such that $f(N_{u_0})$ is dense in $N_{t_0}$. 
Since 
$\langle f(x_j):j<\omega\rangle$ 
is dense in $N_{t}$, there is some $j<\omega$ with $j\neq i$ and $f(N_{\langle j\rangle})\cap N_{t'_0}\neq \emptyset$. 
By the previous claim for $\langle j\rangle$ and $t'_0$, there is some $v_0\supsetneq \langle j\rangle$ and some $t_0\supsetneq t'_0$ such that $f(N_{v_0})$ is dense in $N_{t_0}$. 
We have $u_0\perp v_0$ since $u_0\supseteq \langle i\rangle$ and $v_0\supseteq \langle j\rangle$. 

Suppose that $u_n$, $v_n$ and $t_n$ have been constructed. 
By the previous claim for $u_n$ and $t_n$, there is some $u_{n+1}\supsetneq u_n$ and some $t'_{n+1}\supsetneq t_n$ such that $f(N_{u_{n+1}})$ is dense in $N_{t'_{n+1}}$. 
By the previous claim for $v_n$ and $t'_{n+1}$, there is some $v_{n+1}\supsetneq v_n$ and some $t_{n+1}\supsetneq t'_{n+1}$ such that $f(N_{v_{n+1}})$ is dense in $N_{t_{n+1}}$. 
This completes the construction. 

Finally, let $x:=\bigcup_{n<\omega}u_n$, $y:=\bigcup_{n<\omega}v_n$ and $z:=\bigcup_{n<\omega}t_n$. 
We claim that $f(x)=z$. 
Otherwise $f(x)\notin N_{t_n}$ for some $n<\omega$. Since $f$ is continuous, there is some $k\geq n$ with $f(N_{u_k})\cap N_{t_n}=\emptyset$ and thus $f(N_{u_k})\cap N_{t_k}=\emptyset$. But this contradicts \ref{ODDI fails for D omega claim ii}.
Similarly, $f(y)=z$. 
Thus $x\neq y$ and $f(x)=f(y)=z$. 
\end{proof} 

\begin{remark} 
\label{remark: dhD-}
We will see in the next theorem that the previous proposition
fails for uncountable $\kappa$. 
Note that the construction in its proof 
cannot necessarily be continued at limit stages: the inductive hypothesis need not be preserved, since the intersection of less than $\kappa$ many dense sets may be empty. 
\end{remark} 

\begin{remark}
\label{remark Dminus} 
Let $\dhD\kappa^-$ 
\index{hypergraph!of sequences!DDm@dense inside a basic open set\idf$\dhD\kappa^-$}
denote the hypergraph on ${}^\kappa\kappa$  of all non-constant sequences $\bar{x}$ such that for some $t\in {}^{<\kappa}\kappa$, $\ran(\bar{x})$ is contained in $N_t$ and is dense in $N_t$. 
Observe that $\dhD\kappa^-$ is 
box-open on ${}^\kappa\kappa$
and there is a continuous homomorphism from $\dhH\kappa$ to $\dhD\kappa^-$.
The latter follows from
Proposition~\ref{lemma: dhD} and
Lemma~\ref{lemma: ODD subsequences},
since 
any hyperedge of $\dhD\kappa$ has a subsequence in $\dhD\kappa^-$ 
and hence
any $\dhD\kappa^-$-independent set is also $\dhD\kappa$-independent,
so $\dhD\kappa\leqf\dhD\kappa^-$.%
\footnote{See Definition~\ref{def: H-full}.}
Thus, the facts about the existence of homomorphisms in
Propositions~\ref{lemma: dhD}--\ref{ODDI fails for D omega}
hold for any 
dihypergraph $H$ on ${}^\kappa\kappa$
with $\dhD\kappa^-\subseteq H \subseteq \dhD\kappa$. 
\end{remark}

\hypertarget{ODD and ODDI Diamondi}{}%
\hypertarget{ODDI and not ODDH Diamondi}{}%
The next theorem shows that $\Diamondi\kappa$ 
implies $\ODD\kappa H\Longleftrightarrow\ODDI\kappa H$
for relatively box-open $\kappa$-dihypergraphs $H$.

\begin{theorem}
\label{theorem: ODD ODDI}
Suppose $\Diamondi\kappa$ holds%
\footnote{See Definition~\ref{diamondi def}. Recall that $\Diamondi\kappa$ implies $\kappa^{<\kappa}=\kappa>\omega$ by Remark~\ref{diamondi omega}.}
and $H$ is a relatively box-open $\kappa$-dihypergraph on~${}^\kappa\kappa$.
If there exists a continuous homomorphism from $\dhH\kappa$ to $H$, then there exists an injective continuous homomorphism $g$ from $\dhH\kappa$ to $H$
and a continuous strict $\wedge$-homomorphism $e:{}^{<\kappa}\kappa\to{}^{<\kappa}\kappa$ with $g=f\comp [e]$.
\end{theorem}
\begin{proof}
We use the following equivalent 
$2$-dimensional version of $\Diamondi\kappa$
obtained in
Lemma~\ref{diamondi equiv}:

\begin{quotation}
There exists a sequence 
$\langle B_\alpha:\alpha<\kappa\rangle$ of sets 
$B_\alpha\subseteq\gK{{}^\alpha\kappa}$
of size
$|B_\alpha|<\kappa$ 
such that
for all $(x,y)\in\gK{{}^\kappa\kappa}$,
there exists $\alpha<\kappa$ 
with $(x\restr\alpha,y\restr\alpha)\in B_\alpha$. 
\end{quotation}

\begin{claim*}
There exists an order homomorphism $\theta$ for $H$ \todog{This $\theta$ may not be continuous}
and a continuous strict $\wedge$-ho\-mo\-morphism $e:{}^{<\kappa}\kappa\to{}^{<\kappa}\kappa$
such that
$[\theta]=f\comp [e]$ and
for all $\alpha<\kappa$, we have
$\theta(s)\perp \theta(t)$
for all $(s,t)\in B_\alpha$. 
\end{claim*}
\begin{proof}
There exists an order homomorphism $\iota$ for $H$
and 
a strict $\wedge$-homomorphism
$e_0:{}^{<\kappa}\kappa\to{}^{<\kappa}\kappa$
with $f=[\iota]\comp [e_0]$
by Lemma~\ref{homomorphisms and order preserving maps}~\ref{hop 1}. 
We construct strict order preserving maps
$\theta:{}^{<\kappa}\kappa\to{}^{<\kappa}\kappa$
and $e_1:{}^{<\kappa}\kappa\to{}^{<\kappa}\kappa$
such that the following hold
for all $s,t\in{}^{<\kappa}\kappa$ and $\alpha<\kappa$:
\begin{enumerate-(i)}
\item\label{hop3 proof 3}
If $(s,t)\in B_\alpha$, then $\theta(s)\perp \theta(t)$. 
\item\label{hop3 proof 2}
$e_1(t)\conc\langle\alpha\rangle\subseteq e_1(t\conc\langle\alpha\rangle)$.
\item\label{hop3 proof 1}
$\theta(t)= \iota\big(e_1(t)\big)$.
\end{enumerate-(i)}
Construct $\theta(t)$ and $e_1(t)$ by recursion on $\lh(t)$.
Fix $\alpha<\kappa$ and suppose that $\theta(u)$ and $e(u)$ have been constructed for all $u\in{}^{<\alpha}\kappa$. 
For each $t\in{}^\alpha\kappa$, let 
\[
\epsilon(t):=
\begin{cases}
\emptyset &\text{ if $t=\emptyset$,}
\\
e_1(u)\conc\langle\beta\rangle &\text{ if $\alpha\in\Succ$ and $t=u\conc\langle\beta\rangle$,}
\\
\displaystyle\bigcup_{u\subsetneq t}e_1(u) &\text{ if $\alpha\in\Lim$.}
\end{cases}
\]
Let 
$\langle x_t:t\in{}^\alpha\kappa\rangle$ 
be a sequence in ${}^\kappa\kappa$ such that  
$\epsilon(t)\subsetneq x_t$ 
 for all $t\in{}^\alpha\kappa$
and $[\iota](x_s)\neq [\iota](x_t)$
 for all $s,t\in{}^\alpha\kappa$ with $s\neq t$. 
Such a sequence can be defined 
using recursion with respect to some well-order of ${}^\alpha\kappa$
since 
$|[\iota]{(N_{\epsilon(t)})}|>\kappa$ for all $t\in{}^\alpha\kappa$
by Lemma~\ref{two options in ODD are mutually exclusive}~\ref{me 2}. 
Since 
\todog{this is the only place we need to use that $|B_\alpha|<\kappa$}
$|B_\alpha|<\kappa$,
there exists $\gamma<\kappa$ such that 
$\iota(x_s\restr\gamma)\perp\iota(x_t\restr\gamma)$ for all $(s,t)\in B_\alpha$.
For each $t\in{}^\alpha\kappa$, take
$e_1(t)\subsetneq x_t$ 
extending
$\epsilon(t)$
with
$\lh(e_1(t))\geq\gamma$
and let
$\theta(t):=\iota(e_1(t))$.
This choice of
$e_1(t)$
and
$\theta(t)$
clearly ensures~\ref{hop3 proof 2} and~\ref{hop3 proof 1}.
\ref{hop3 proof 3} 
holds since
$\iota(x_t\restr\gamma)
\subseteq
\iota(e_1(t))=\theta(t)$
for all $t\in{}^\alpha\kappa$.  

Suppose the maps $\theta$ and $e_1$ have been constructed.
Then $e_1$ is a strict $\wedge$-homomorphism by~\ref{hop3 proof 2}, and hence so is $e:=e_0\comp e_1$.
The same argument as in the proof of 
Theorem~\ref{theorem: ODD ODDC when ddim<kappa}
shows that
$\theta$ is an order homomorphism for $H$ 
with $[\theta]=f\comp[e_0\comp e_1]$.
By~\ref{hop3 proof 3}, $\theta$ is as required.
\end{proof}

Let $\theta$ be as in the claim.
$g:=[\theta]$ is a continuous homomorphism from $\dhH\kappa$ to $H$
by Lemma~\ref{homomorphisms and order preserving maps}~\ref{hop 2}.
To see that $g$ is injective, take $x,y\in{}^\kappa \kappa$ with $x\neq y$. 
Choose $\alpha<\kappa$ such that
$(x\restr\alpha,y\restr\alpha)\in B_\alpha$. 
Then $\theta(x\restr\alpha)\perp \theta(y\restr\alpha)$, and therefore $g(x)\neq g(y)$.
\end{proof} 

\hypertarget{ODD and ODDI Diamondi provable}{}%
\hypertarget{ODDI and not ODDH Diamondi provable}{}%
We do not know whether the assumption of $\Diamondi\kappa$ in the previous theorem can be removed. 
The assumption is very weak: recall that $\Diamondi\kappa$ holds if $\kappa$ is inaccessible or $\kappa\geq\omega_2$ is a successor cardinal with $\kappa^{<\kappa}=\kappa$.\footnote{See Lemma~\ref{diamondi claim} and Remark~\ref{diamondi remark}.} 
It is open whether 
the previous theorem
holds for $\omega_1$ and for all weakly inaccessible cardinals.

Note that $\Diamondi\kappa$ holds in all $\Col(\kappa,\lle\lambda)$-generic extensions $V[G]$, where $\lambda>\kappa$ is inaccessible. 
We thus immediately obtain a stronger version of Theorem~\ref{main theorem}.
The definitions of 
the principles in the next corollary
are analogous to Definition~\ref{def: ODD for classes},
and they can be reformulated 
in terms of relatively box-open dihypergraphs as in
\todog{keep this in the text: it is needed to obtain the next corollary from the main theorem, since the results in this subsection talk about relatively box-open dihypergraphs only}
Lemmas~\ref{lemma: two versions of ODD for definable sets} and~\ref{lemma: two versions of definable ODD}.

\begin{corollary}
\label{main theorem strong version}
Suppose $\kappa$ is a regular infinite cardinal and $2\leq\ddim\leq\kappa$. 
If $\lambda>\kappa$ is inaccessible and $G$ is a $\Col(\kappa,\lle\lambda)$-generic filter over $V$, then the following statements hold in $V[G]$:%
\begin{enumerate-(1)}
\vspace{2 pt}
\item \label{mtvs 2}
$\ODDI\kappa\ddim(\defsets\kappa,\defsets\kappa)$ if $\kappa$ is uncountable. 

\item \label{mtvs 3} 
$\ODDI\kappa\ddim(\defsets\kappa)$ if $\kappa$ is uncountable and $\lambda$ is Mahlo. 
\end{enumerate-(1)}
If $\ddim<\kappa$, then in fact $\ODDC\kappa\ddim(\defsets\kappa)$ holds in $V[G]$. 
\end{corollary}
\begin{proof}
\ref{mtvs 2} and \ref{mtvs 3} for $\ddim=\kappa$ follow from Theorems~\ref{main theorem} and~\ref{theorem: ODD ODDI}.
The last statement
follows from Lemma~\ref{<kappa dim hypergraphs are definable} and Corollary~\ref{cor: ODD ODDC when ddim<kappa}. 
\ref{mtvs 2} and \ref{mtvs 3} for $\ddim<\kappa$ follow. 
\end{proof}

In the rest of this section, we show that some of the dichotomies in Figure \ref{figure: ODD variants} are equivalent for relatively box-open $\kappa$-dihypergraphs $H$ on ${}^\kappa\kappa$ with additional properties.
These properties are 
defined in terms of the hypergraphs in the next definition.

\begin{definition} 
\label{def: dhI}
\label{def: dhN}
\label{def: dhth}
Suppose that $Y$ is a subset of ${}^\kappa\kappa$.
\begin{enumerate-(a)}
\item
For $2\leq \ddim\leq \kappa$, let
$\dhI\ddim:=\dhII\ddim{{}^\kappa\kappa}$ 
\index{hypergraph!of sequences!IIk@injective in ${}^\kappa\kappa$\idf$\dhI\ddim,\,\dhIkappa$} 
denote the set of
all injective 
$\ddim$-sequences in  ${}^\kappa\kappa$.
Let $\dhIkappa:=\dhI\kappa$. 
\item
Let $\dhN Y:=\dhN Y_{{}^\kappa\kappa}$ 
\index{hypergraph!of sequences!NNY@with no limits in $Y$ and the sequence\idf$\dhN Y$}
denote the set of
all non-constant $\kappa$-sequences $\bar x$ in ${}^\kappa\kappa$
such that $\ran(\bar x)$
has no limit points in $Y\cup\ran(\bar x)$.%
\footnote{It is equivalent that no injective subsequence of $\bar x$ converges to an element of $Y\cup\ran(\bar x)$.
For sequences $\bar x$ in $Y$, it is equivalent to write $Y$ instead of $Y\cup\ran(\bar x)$.}
\todog{Let $H$ denote the original version of $\dhth Y$ (i.e. the one where $Y\cup\ran(\bar x)$ is replaced by $Y$ in the definition). 
Then $H$ is not be box-open, as witnessed by any injective sequence which converges to one of its elements that is not in $Y$.
However, $\dhth Y$ is an $H$-full subdihypergraph of $H$, since any injective convergent sequence contains a subsequence $\bar x$ with $\lim(\bar x)\notin\ran(\bar x)$.
Hence $\ODD\kappa H\Leftrightarrow\ODD\kappa{\dhth Y}$.}
Let $\dhNkk:=\dhN{\,{}^\kappa\kappa}.$
\index{hypergraph!of sequences!NN@with no injective convergent subsequence\idf$\dhNkk$}
\item 
\index{hypergraph!of sequences!HHtopY@injective with no limits in~$Y$~\idf$\dhth Y$}%
\index{hypergraph!of sequences!HHtop@injective with no convergent subsequences\idf$\dhthkk$}%
Let $\dhth Y:=\dhIkappa\cap \dhN Y$ and
$\dhthkk:=\dhth{\,{}^\kappa\kappa}$.%
\footnote{The subscript in $\dhth Y$ indicates that this dihypergraph is important for the topological Hurewicz dichotomy.}
\end{enumerate-(a)}
\end{definition}

Note that $\dhNkk$ consists of all $\kappa$-sequences in ${}^\kappa\kappa$ 
with no injective convergent subsequences,
and
$\dhthkk$ consists of all injective $\kappa$-sequences in ${}^\kappa\kappa$ without convergent subsequences.
$\dhthkk$~was defined in \cite{CarroyMillerSoukup}*{Section 2} for $\kappa=\omega$.   

Note that $\dhI\ddim$ is in $\defsetsk$ for all $\ddim\leq\kappa$, as are $\dhN Y$ and $\dhth Y$ when $Y\in\defsetsk$.
$\dhI\ddim$ is box-open on ${}^\kappa\kappa$ for $\ddim<\kappa$, while $\dhIkappa$ is not box-open.
To see the latter statement, take any injective sequence which converges to one of its elements.
Moreover, 
$\dhN Y$ is not box-open on ${}^\kappa\kappa$,
as witnessed by
any non-constant sequence which is eventually constant. 
However, $\dhth Y=\dhIkappa\cap \dhN Y$ is box-open.
To see this, suppose $\bar{x}=\langle x_\alpha : \alpha<\kappa\rangle\in \dhth Y$.
For each $\alpha<\kappa$,
\todog{We need to use that $\ran(\bar x)$ is a discrete set here! (Equivalently, that it has no limit points in $\ran(\bar x)$)}
pick some $\gamma_\alpha$ with $\alpha\leq \gamma_\alpha<\kappa$ such that $x_\alpha$ is the unique element of $\bar x$ in 
$N_{x_\alpha\restr\gamma_\alpha}$. 
We have $\prod_{\alpha<\kappa} N_{x_\alpha\restr\gamma_\alpha} \subseteq \dhIkappa$ since the sets $N_{x_\alpha\restr\gamma_\alpha}$ 
\todog{\tiny{Suppose $Y:=N_{x_\alpha\restr\gamma_\alpha}\cap N_{x_\beta\restr\gamma_\beta}\neq\emptyset$ 
and $\gamma_\alpha\leq\gamma_\beta$. 
\\
If $y\in Y$, then\\ $x_\alpha\restr\gamma_\alpha=y\restr\gamma_\alpha$ \\$=x_\beta\restr\gamma_\alpha$, 
so\\ $x_\beta\in N_{x_\alpha\restr\gamma_\alpha}$,\\ contradiction.}} 
are pairwise disjoint. 
Moreover, we have $\prod_{\alpha<\kappa} N_{x_\alpha\restr\gamma_\alpha} \subseteq \dhN Y$ 
since $\bar x\in\dhN Y$ and $\gamma_\alpha\geq\alpha$ for all $\alpha<\kappa$. 
\todog{\tiny{Claim:
$\prod_{\alpha<\kappa} N_{x_\alpha\restr\gamma_\alpha}$\\$\subseteq\dhN Y$.
Proof: If the $\gamma_\alpha$'s are unboun-\\ded in $\kappa$, then this follows from\\ $\bar x\in\dhN Y$.
If the $\gamma_\alpha$'s are bounded\\ by $\gamma<\kappa$, then 
this holds 
since the sets $N_{x_\alpha\restr\gamma_\alpha}$ are pairwise disjoint.}}

We shall prove the equivalences displayed in Figure~\ref{figure: ODD variants special cases}
for relatively box-open $\kappa$-di\-hyper\-graphs~$H$ on ${}^\kappa\kappa$.
The meaning of the arrows is as in Figure~\ref{figure: ODD variants}.%
\footnote{Blue arrows are clickable and lead to the corresponding results.} 
\todon{In the next figure, ``abovecaptionskip'' (distance between figure and caption) and ``intextsep'' (distance between text and figure) were set LOCALLY. Remove before submitting to journal?}

\vspace{5 pt}

{
\setlength{\abovecaptionskip}{8pt}
\setlength{\intextsep}{8 pt}
\hypersetup{linkcolor=VeryDarkBlueNode}
\begin{figure}[H]
\begin{tikzpicture}[scale=1]
\tikzstyle{theory1}=[draw,rounded corners,scale=.90, minimum width=8mm, minimum height=8.0mm,line width=\widthline,VeryDarkBlueNode]

\tikzstyle{medarrow}=[>=Stealth, line width=\widtharrow, scale=1.8] 
\tikzstyle{infoarrow}=[>=Stealth, scale=1] 
\tikzstyle{consarrow}=[>=Stealth, line width=\widtharrow, scale=1.8, dash pattern=on 6pt off  2pt]
\tikzstyle{questionarrow}=[>=Stealth, line width=\widtharrow, scale=1.8, dash pattern=on 1pt off  2pt]

\node (1) at (0,0) [theory1] {\hyperlink{ODDC def}{\parbox[c][.7cm]{1.78cm}{$\,\ODDC\kappa{H}$}}};
 \node (2) at (4,0) [theory1] {\hyperlink{ODDH def}{\parbox[c][.7cm]{1.5cm}{$\,\ODDH\kappa H$}}};
 \node (3) at (8,0) [theory1] {\hyperlink{ODDI def}{\parbox[c][.7cm]{1.35cm}{$\,\ODDI\kappa H$}}};
 \node (4) at (12,0) [theory1] {\hyperref[ODD def]{\parbox[c][.7cm]{1.25cm}{$\,\ODD\kappa H$}}}; 


%
\draw[medarrow,VeryDarkBlueNode] 
(1.east) edge [->,VeryDarkBlueNode](2.west);

\draw[medarrow,BlueArrow] 
(3.west) edge [<->] 
node[midway,above]{{\footnotesize $H\cap\dhIkappa$ is rela-}} 
node {\hyperlink{hyp: ODDI ODDH when H cap dhI is box-open}{\phantom{xxx}}} 
node[midway,below]{{\footnotesize tively box-open}} 
(2.east) 
([xshift=-0.7mm]4.north) edge [<->,bend right=22] 
node[midway,above]{{\footnotesize $H\subseteq\dhIkappa$}} 
node {\hyperlink{hyp: ODD iff ODDH for dhI}{\phantom{xxx}}} 
(2.north) 
([xshift=0.7mm]4.north) edge [<->,bend right=30] 
node[midway,above]{{\footnotesize $H\subseteq\dhthkk$}} 
node {\hyperlink{hyp: ODD iff ODDC for dhth}{\phantom{xxx}}} 
([xshift=-0.7mm]1.north)
([xshift=0mm]3.south) edge [<->,bend left=22] 
node[midway,below]{{\footnotesize $H\subseteq\dhNkk$}} 
node {\hyperlink{hyp: ODDI ODDC dhN}{\phantom{xxx}}} 
([xshift=0.7mm]1.south); 

\draw[consarrow, BlueArrow]
(3.east) edge[<->] 
node[midway,above] {{\footnotesize $\Diamondi\kappa$}} 
node {\hyperlink{ODD and ODDI Diamondi}{\phantom{xxx}}} 
(4.west);
\end{tikzpicture}

\caption{Equivalences
in some special cases}

\label{figure: ODD variants special cases}
\end{figure}
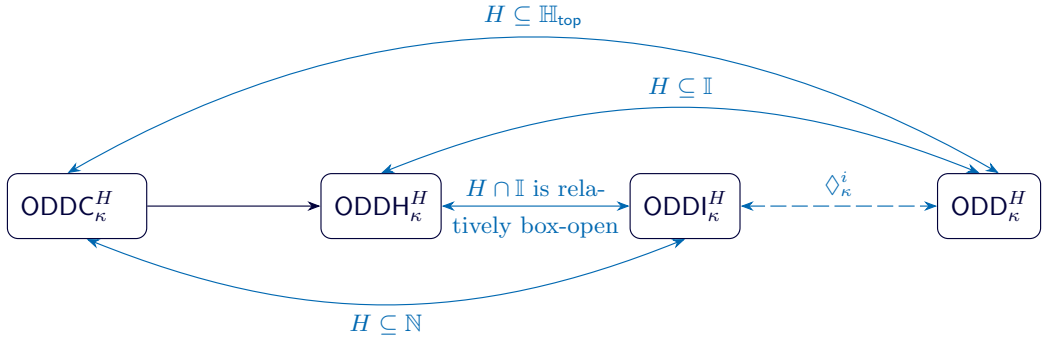}

\begin{lemma}
\label{order homomorphisms and injective H}
Suppose that $2\leq\ddim\leq \kappa$.
\begin{enumerate-(1)}
\item\label{dhI injectivity}
A map $f:{}^\kappa\ddim\to {}^\kappa\kappa$ is injective if and only if it is a homomorphism from $\dhH\ddim$ to~$\dhI\ddim$.
\item\label{dhI perp}
Suppose 
$\iota:{}^{<\kappa}\ddim\to{}^{<\kappa}\kappa$ is
strict order preserving with
$\ran([\iota])\subseteq X\subseteq{}^\kappa\kappa$.
Then $\iota$
is an order homomorphism for 
$(X,\dhI\ddim)$ 
if and only if
$\iota$ preserves~$\perp$.
\end{enumerate-(1)}
\end{lemma}
\begin{proof}
For \ref{dhI injectivity}, suppose first that $f$ is injective. 
Since any sequence in $\dhH\ddim$
is injective,
its $f$-image
is also injective, i.e., it is in $\dhI\ddim$.
Conversely, suppose that $f$ is a homomorphism from $\dhH\ddim$ to $\dhI\ddim$.
Take distinct elements $x,y$ of ${}^\kappa d$ and a hyperedge of $\dhH\ddim$ that contains $x$ and $y$. Since the $f$-image of this hyperedge is in $\dhI\ddim$, we have $f(x)\neq f(y)$.

For \ref{dhI perp},
it is clear that if $\iota$ preserves $\perp$, then is an order homomorphism for $(X,\dhI\ddim)$.
Conversely, suppose that $\iota$ is an order homomorphism for $(X,\dhI\ddim)$.
To see that $\iota$ preserves $\perp$,
it suffices to show that
$
\iota(u\conc\langle\alpha\rangle)
\perp
\iota(u\conc\langle\beta\rangle)
$
holds for all $u\in{}^\kappa\ddim$ and all $\alpha,\beta<\kappa$ with $\alpha\neq\beta$.
Seeking a contradiction, suppose that 
$
\iota(u\conc\langle\alpha\rangle)
\subseteq
\iota(u\conc\langle\beta\rangle). 
$
Since 
$\ran(\iota)\subseteq T(X)$,
choose an
an arbitrary element
$x_\gamma$ 
of $N_{\iota(u\conc\langle\gamma\rangle)}\cap X$ 
for all $\gamma\in\kappa-\{\alpha\}$ 
and let
$x_\alpha:=x_\beta$.
Then $\langle x_\gamma:\gamma<\kappa\rangle$ is not injective.
However, 
since
$x_\alpha\in 
N_{\iota(u\conc\langle\beta\rangle)}\subseteq 
N_{\iota(u\conc\langle\alpha\rangle)}$,
we have
$\langle x_\gamma:\gamma<\kappa\rangle \in\prod_{\gamma<\ddim}(N_{\iota(u\conc\langle\gamma\rangle)}\cap X)\subseteq \dhI\ddim.$
\end{proof}

\begin{lemma} 
\label{lemma: dhN dhth}
Suppose that $Y\subseteq{}^\kappa\kappa$,
and $f:{}^\kappa\kappa\to Y$. 
\begin{enumerate-(1)}
\item\label{dhN closed maps} 
If $f$ is a closed map,\footnote{$f$ is meant to be closed as a map with codomain $Y$.} 
then $f$ is a homomorphism from $\dhH\kappa$ to $\dhN Y$.%
\footnote{The converse also holds when
when $\kappa=\omega$ or $\kappa$ is weakly compact, but it may fail for non-weakly compact $\kappa$; see Remark~\ref{weakly compact dhN}.}
\item\label{dhth closed injective maps} 
If $f$ is an injective closed map, then $f$ is a homomorphism from $\dhH\kappa$ to $\dhth Y$.
\item\label{dhth order homomorphism}
If 
$X\subseteq Y$ and $\iota:{}^{<\kappa}\kappa\to{}^{<\kappa}\kappa$
is an order homomorphism for $(X,\dhth Y)$,
then $[\iota]$ is a closed map from ${}^\kappa\kappa$ to~$Y$.
\end{enumerate-(1)}
\end{lemma}
\begin{proof}
For \ref{dhN closed maps}, 
let $\langle x_\alpha:\alpha<\kappa\rangle\in\dhH\kappa$. 
Then $\{x_\alpha:\alpha<\kappa\}$ is a closed discrete set. 
Hence 
$\{f(x_\alpha):\alpha<\kappa \}$ is a relatively closed discrete subspace of $Y$ and therefore has no limit points in $Y\cup\ran(\bar x)$.
\ref{dhth closed injective maps} follows from~\ref{dhN closed maps} and Lemma~\ref{order homomorphisms and injective H}~\ref{dhI injectivity}.
For~\ref{dhth order homomorphism}, 
note that 
$\ran([\iota])\subseteq X\subseteq Y$
and that 
$\iota$ is $\perp$-preserving by Lemma~\ref{order homomorphisms and injective H}~\ref{dhI perp}.
It thus suffices by  Lemma~\ref{lemma: [e] closed map}
\todog{$\iota$ needs to be $\perp$-preserving in order to apply Lemma~\ref{lemma: [e] closed map}}
to show that $\Lim^\iota_t\cap Y=\emptyset$
for all $t\in{}^{<\kappa}\kappa$, 
i.e., for all $t\in{}^{<\kappa}\kappa$ and all sequences 
$\bar x=\langle x_\alpha:\alpha<\kappa\rangle$ 
with ${t\conc\langle \alpha\rangle}\subseteq x_\alpha$ for all $\alpha<\kappa$,
$\{[\iota](x_\alpha):\alpha<\kappa\}$ has no limit points in $Y$.%
\footnote{See the paragraph before Lemma~\ref{lemma: [e] closed map} for the definition of $\Lim^\iota_t$.}  
This holds since 
$\langle [\iota](x_\alpha):\alpha<\kappa\rangle\in\dhN Y$. 
\todog{$\langle [\iota](x_\alpha)\rangle_\alpha\subseteq \prod(N_{\iota(t\conc\langle \alpha\rangle)}\cap X)\subseteq\dhN Y$.} 
\end{proof}

Suppose $H$ is a relatively box-open $\kappa$-dihypergraph on ${}^\kappa\kappa$.
The next theorem shows that if $H\subseteq\dhIkappa$,  
then 
$\ODD\kappa H\Longleftrightarrow\ODDH\kappa H$  
and if
$H\subseteq\dhthkk$, then
$\ODD\kappa H\Longleftrightarrow\ODDC\kappa{H}$.
Thus if $H\subseteq\dhthkk$, then all of the dichotomies in  Figure~\ref{figure: ODD variants special cases} are equivalent to $\ODD\kappa H$.
In fact, we obtain a similar equivalence for the next variant of $\ODDC\kappa{H}$. 
Here the range of the homomorphism is required to be closed relative to a subset $Y$ of ${}^\kappa\kappa$. 

\begin{definition}
\label{def: ODDC HY}
Suppose that $Y\subseteq{}^\kappa\kappa$ and $H$ is 
a $\kappa$-dihypergraph%
\footnote{$\ODDC\kappa{H,Y}$ can be defined for $\ddim$-dihypergraphs $H$ on ${}^\kappa\kappa$ for $\ddim<\kappa$ as well. However, this is equivalent to $\ODDC\kappa{H\restr Y}$ for any relatively box-open $\ddim$-dihypergraph $H$ on ${}^\kappa\kappa$ by Theorem~\ref{theorem: ODD ODDC when ddim<kappa}.}
on ${}^\kappa\kappa$.
\todog{For the footnote: if $H$ is relatively box-open, then $H\restr Y$ is also relatively box-open. Apply  Theorem~\ref{theorem: ODD ODDC when ddim<kappa} to $H\restr Y$.}
\index{open dihypergraph dichotomy!variant where homomorphisms are!homeomorphisms onto relatively closed sets\idf $\ODDC\kappa{H,Y}$}
\begin{quotation}
{$\ODDC\kappa{H,Y}$:}
Either $H\restr Y$ admits a $\kappa$-coloring, or 
there exists 
a homomorphism 
from~$\dhH{\kappa}$ to~$H\restr Y$
which is
\emph{a homeomorphism} from ${}^\kappa\kappa$ onto 
a \emph{relatively closed} subset of~$Y$.%
\end{quotation}
\end{definition}

Note that $\ODDC\kappa H$ is the special case of $\ODDC\kappa{H,Y}$ for $Y:={}^\kappa\kappa$, and for all $Y\subseteq Z\subseteq{}^\kappa\kappa$ with $\domh H\subseteq Y$, we have $\ODDC\kappa{H,Z} \Longrightarrow \ODDC\kappa{H,Y}$. 
The next theorem will be used to prove Theorem~\ref{theorem: THD and ODD dhth}.

\begin{theorem}
\label{theorem: strong variants for dhI dhth}
\hypertarget{hyp: ODD iff ODDH for dhI}{}%
\hypertarget{hyp: ODD iff ODDC for dhth}{}%
Suppose that $Y\subseteq{}^\kappa\kappa$ and $H$ is a relatively box-open $\kappa$-dihypergraph on~$Y$.
Suppose there exists  a continuous 
homomorphism from $\dhH\kappa$ to $H$.
\begin{enumerate-(1)}
\item
\label{ODD iff ODDH for dhI}
If $H\subseteq\dhIkappa$, then there exists a 
homomorphism from $\dhH\kappa$ to $H$ which is a homeomorphism from ${}^\kappa\kappa$ onto its image.
\todog{This works for $\ddim<\kappa$ as well. But in that case, we already know that $\ODD\kappa H\Leftrightarrow\ODDC\kappa H$}
\item
\label{ODD iff ODDC for dhth}
If $H\subseteq\dhth Y$, then
there exists a 
homomorphism from $\dhH\kappa$ to $H$ which is a homeomorphism 
from ${}^\kappa\kappa$ onto a relatively closed subset of $Y$.
\end{enumerate-(1)}
\end{theorem}
\begin{proof}
For \ref{ODD iff ODDH for dhI}, note that there exists an order homomorphism  $\iota$ for $H$
by Lemma~\ref{homomorphisms and order preserving maps}.
If $H\subseteq\dhIkappa$, then
$\iota$ also preserves $\perp$ by Lemma~\ref{order homomorphisms and injective H}~\ref{dhI perp}.
Thus 
$[\iota]$ is as required by Lemma~\ref{lemma: perp-preserving order homomorphisms and homeomorphisms}. 
\ref{ODD iff ODDC for dhth} follows from the proof of~\ref{ODD iff ODDH for dhI} and 
Lemma~\ref{lemma: dhN dhth}~\ref{dhth order homomorphism}.
\todog{We need to use that $\domh H\subseteq Y$ either here or in the previous lemma, to have $\ran([\iota])\subseteq Y$.}
\end{proof}

\begin{corollary}
\label{cor: strong variant for dhN} 
\label{cor: strong variant for H cap dhI box-open}
\hypertarget{hyp: ODDI ODDC dhN}{}%
\hypertarget{hyp: ODDI ODDH when H cap dhI is box-open}{}%
Suppose that $Y\subseteq{}^\kappa\kappa$ and $H$ is a $\kappa$-dihypergraph on $Y$.
Suppose there exists an injective continuous homomorphism from $\dhH\kappa$ to~$H$.
\begin{enumerate-(1)}
\item
\label{ODDI iff ODDH, H cap dhI}
If $H\cap\dhIkappa$ is relatively box-open,
then there exists a homomorphism from $\dhH\kappa$ to $H$ which is a homeomorphism from ${}^\kappa\kappa$ onto its image.
\item
\label{ODD iff ODDC for dhN}
If $H\subseteq\dhN Y$ and $H$ is relatively box-open, then
there exists a homomorphism from $\dhH\kappa$ to $H$ which is a homeomorphism 
from ${}^\kappa\kappa$ onto a relatively closed subset of $Y$.
\end{enumerate-(1)}
\end{corollary}
\begin{proof}
If $f$ is an injective continuous homomorphism from $\dhH\kappa$ to~$H$, then
it is a continuous homomorphism from $\dhH\kappa$ to $H\cap\dhIkappa$.
If $H\cap\dhIkappa$ is relatively box-open, then 
by Theorem~\ref{theorem: strong variants for dhI dhth}~\ref{ODD iff ODDH for dhI},
there exists a homeomorphism $h$ from ${}^\kappa\kappa$ onto $\ran(h)$ that is a homomorphism from $\dhH\kappa$ to $H\cap\dhIkappa$ and thus to $H$. 
If $H\subseteq\dhN Y$, 
then $f$ is a continuous homomorphism from $\dhH\kappa$ to
$I:=H\cap\dhIkappa=H\cap\dhth Y$. 
Moreover, if $H$ also is relatively box-open, then 
$I$ is relatively box-open 
since $\dhth Y$ is relatively box-open.
Then there exists a homeomorphism $h$ from ${}^\kappa\kappa$ onto a closed set that is a homomorphism from $\dhH\kappa$ to $I$ and hence to $H$, 
by Theorem~\ref{theorem: strong variants for dhI dhth}~\ref{ODD iff ODDC for dhth}, 
\end{proof}

Suppose $H$ is a $\kappa$-dihypergraph on $Y$. 
If $H\cap\dhIkappa$ is relatively box-open, then
$\ODDI\kappa H\Longleftrightarrow\ODDH\kappa H$,
If $H\subseteq\dhN Y$ and $H$ is relatively box-open, then 
$\ODDI\kappa H\Longleftrightarrow\ODDC\kappa{H,Y}$
by the previous corollary.
By Theorem~\ref{theorem: ODD ODDI},
these are further equivalent to $\ODD\kappa H$ 
if we assume $\Diamondi\kappa$. 
In particular, this holds 
for all inaccessible cardinals and all successor cardinals $\kappa\geq\omega_2$.

\begin{remark}
\label{weakly compact dhN}
If $\kappa=\omega$ or $\kappa$ is weakly compact, then any continuous homomorphism $f$ from $\dhH\kappa$ to $\dhN Y$ is a closed map from ${}^\kappa\kappa$ to $Y$. 
To see this,
it suffices 
 by Lemma~\ref{lemma: f closed map kappa weakly compact}~\ref{closure of ran(f) when kappa weakly compact}
that the sets $\Lim^f_t\cap Y$ are all empty. 
But the latter follows immediately from the statement that $f$ is a homomorphism from $\dhH\kappa$ to $\dhN Y$. 
This argument provides an alternative proof of Corollary~\ref{cor: strong variant for dhN}~\ref{ODD iff ODDC for dhN}
for weakly compact cardinals $\kappa$ and $\kappa=\omega$.
However, the claim fails for $Y:={}^\kappa\kappa$ at non-weakly compact cardinals $\kappa>\omega$, even if we assume that $f$ is injective.%
\footnote{I.e., a continuous homomorphism from $\dhH\kappa$ to $\dhthkk$ need not be a closed map.}
To see this, 
let $f:=[\theta]$ for the map
$\theta$ defined in Remark~\ref{ran[e] not closed}~\ref{ran[e] not closed 2}.
\end{remark}

\newpage 

\section{Applications}
\label{section: applications}

In this section, we prove a number of applications of $\ODD\kappa\kappa(X)$.
\todog{The new def of $\ODD\kappa\kappa(X,\defsetsk)$ is needed for the applications to follow from it.}
We obtain analogues to classical results in descriptive set theory such as the Hurewicz 
\todon{Keep this todo-note: added \$\{\}\$ for right size of footnote number} 
dichotomy \cites{Hurewicz1928,SaintRaymond1975,Kechris1977}, 
its extension by Kechris, Louveau and Woodin \cite{KechrisLouveauWoodin1987}*{Theorem 4} and the Jayne--Rogers theorem \cite{JayneRogers1982}*{Theorem 5}.
These results lift work of Carroy, Miller and Soukup \cite{CarroyMillerSoukup} to the uncountable setting.
We further obtain a strong version of the Kechris--Louveau--Woodin dichotomy relative to box-open $\ddim$-dihypergraphs. 
In fact, this dichotomy implies the $\ddim$-dimensional open dihypergraph dichotomy, and they are equivalent for dimension $\kappa$. 
This dichotomy is then used
to study
the open dihypergraph dichotomy restricted to dihypergraphs on ${}^\kappa\kappa$. 
The latter principle has surprising strength. 
For instance, it implies the open dihypergraph dichotomy for $\kappa$-analytic sets and therefore it has at least the consistency strength of an inaccessible cardinal. 
We further prove that it implies the determinacy of V\"a\"an\"anen's perfect set game \cite{VaananenCantorBendixson}*{Section 2} for all subsets of ${}^\kappa\kappa$ and thereby extend a result of V\"a\"an\"anen. 
We derive the asymmetric $\kappa$-Baire property \cite{SchlichtPSPgames}*{Section 3} from the open dihypergraph dichotomy. 
This implication is new even in the countable setting, 
where 
the asymmetric Baire property and the Baire property for $\defsets\omega$ are equivalent. 
Lastly, we prove implications and separations
between some of these applications. 

The strong variant of the Kechris--Louveau--Woodin dichotomy for definable sets 
and our results on the determinacy of V\"a\"an\"anen's game follow from $\ODD\kappa\kappa(\defsetsk)$. 
Therefore a Mahlo cardinal suffices for these applications. 
The remaining results for definable sets follow from $\ODD\kappa\kappa(\defsetsk,\defsetsk)$ and thus 
an inaccessible cardinal suffices for them. 
We do not know the precise consistency strength of many of the applications.
The only lower bounds come from the fact that the $\kappa$-perfect set property for closed subsets of ${}^\kappa\kappa$ for uncountable $\kappa$ has the same consistency strength as an inaccessible cardinal.

\subsection{Variants of the Hurewicz dichotomy}
\label{subsection: Hurewicz}

We obtain several generalizations of the Hurewicz dichotomy 
for subsets of the space ${}^\kappa\kappa$. 
The first statement in the next definition is the \emph{topological Hurewicz dichotomy}%
\footnote{This is called the \emph{Hurewicz dichotomy} in \cite{LuckeMottoRosSchlichtHurewicz}.} 
for a subset $X$ of ${}^\kappa\kappa$ \cite{LuckeMottoRosSchlichtHurewicz}. 
The second one describes a generalization to two subsets $X$ and $Y$ of ${}^\kappa\kappa$ with $X\subseteq Y$. 

\begin{definition}
\label{def: top Hwd}
Suppose that $X$ and $Y$ are subsets of ${}^\kappa\kappa$ with $X\subseteq Y$. 
\begin{quotation}
{$\THD_\kappa(X)$:}
\index{Hurewicz dichotomy for a!topological Hurewicz dichotomy @\mygobble|see{topological Hurewicz dichotomy}}%
\index{topological Hurewicz dichotomy!for a set\idf$\THD_\kappa(X)$}%
Either 
$X$ is contained in a $K_\kappa$ subset of ${}^\kappa\kappa$, 
or
$X$ contains a closed subset of ${}^\kappa\kappa$ which is homeomorphic to ${}^\kappa\kappa$.
\end{quotation}
\begin{quotation}
{$\THD_\kappa(X,Y)$:}
\index{topological Hurewicz dichotomy!variant with two parameters for!sets\idf$\THD_\kappa(X,Y)$}
Either 
$X$ is contained in a $K_\kappa$ subset of $Y$, 
or
$X$ contains a relatively closed subset of $Y$ which is homeomorphic to ${}^\kappa\kappa$.
\end{quotation}
For classes $\mathcal C$ and $\mathcal D$, 
let $\THD_\kappa(\mathcal C, \mathcal D)$ 
\index{topological Hurewicz dichotomy!variant with two parameters for!classes\idf$\THD_\kappa(\mathcal C,\mathcal D)$}
state that $\THD_\kappa(X,Y)$ holds for all subsets $X\in\mathcal C$ and $Y\in\mathcal D$ of ${}^\kappa\kappa$ with $X\subseteq Y$. 
Let $\THD_\kappa(\mathcal C)$ 
\index{topological Hurewicz dichotomy!for a class\idf$\THD_\kappa(\mathcal C)$}
state that $\THD_\kappa(X)$ holds for all subsets $X\in \mathcal C$ of ${}^\kappa\kappa$. 
\end{definition}

In each case the alternatives are mutually exclusive,\footnote{This also follows from Theorem \ref{theorem: THD and ODD dhth} below.} since any $\kappa$-compact subset of ${}^\kappa\kappa$ is bounded by an element of ${}^\kappa\kappa$ and thus all $K_\kappa$ subsets are eventually bounded.
Note that $\THD_\kappa(X)$ equals $\THD_\kappa(X,{}^\kappa\kappa)$. 

$\THD_\kappa(X,Y)$ is motivated by the following dichotomies from \cite{zdomskyy}*{pages 6-7}. 
We call the first one the 
\emph{exact topological Hurewicz dichotomy}.\footnote{The exact topological Hurewicz dichotomy for analytic subsets of Polish spaces follows from \cite{Hurewicz1928}*{Section~6} by passing to a compactification. 
Note that Zdomskyy called this statement the \emph{precise Hurewicz dichotomy}. }
\index{topological Hurewicz dichotomy!exact\idf$\THD_\kappa(X,X)$}

\begin{quotation}
$\THD_\kappa(X,X)$: 
Either $X$ is $K_\kappa$ or $X$ has a relatively closed subset homeomorphic to ${}^\kappa\kappa$. 
\end{quotation}
\begin{quotation}
$\THD_\kappa(X,\Gdelta(\kappa))$: 
For every 
$\Gdelta(\kappa)$ superset $G$ of $X$, either $X$ is contained in a $K_\kappa$ subset of $G$, or $X$ contains a relatively closed subset of $G$ which is homeomorphic to ${}^\kappa\kappa$. 
\end{quotation}

Zdomskyy derived variants of Menger's conjecture%
\footnote{A topological space $X$ has the \emph{Menger property} if for any sequence $\langle \mathcal{U}_n : n\in\omega \rangle$ of open covers, there exists a sequence $\langle \mathcal{V}_n : n\in\omega \rangle$ of finite sets $\mathcal{V}_n\subseteq \mathcal{U}_n$ such that $\bigcup_{n\in\omega} \mathcal{V}_n$ is a cover of $X$ \cite{hurewicz1926verallgemeinerung}*{Section 1}. 
\emph{Menger's conjecture} for a subset $X$ of a metric space \cite{hurewicz1926verallgemeinerung}*{Section 1} states that $X$ is $K_\sigma$ if it has the Menger property. 
Menger originally phrased the conjecture as a problem for all separable metric spaces for the following equivalent form of the Menger property \cite{menger1924}*{Section 2}: 
For every cover $\mathcal U$ of $X$ such that each element of $X$ is contained in arbitrarily small elements of $\mathcal U$, there exists a sequence of sets in $\mathcal U$ covering $X$ whose diameters converge to $0$. 
The equivalence was proved by Hurewicz \cite{hurewicz1926verallgemeinerung}*{Section~1, Satz~V \& Satz~VI}. 
Note that the claim is stated for precompact subsets of Polish spaces, called compact at that time, but this assumption can be removed  easily by passing to a compactification. 
Menger's conjecture for analytic sets is immediate from a result of Hurewicz \cite{hurewicz1926verallgemeinerung}*{Section~4, Satz~XX} by passing to a compactification. 
More recently, Tall, Todor\v{c}evi\'c and Tokg\" oz proved that Menger's conjecture for all sets of reals in $L(\RR)$ 
is equiconsistent with the existence of an inaccessible cardinal \cite{tall2021strength}*{Theorem 1.3}.}
in the uncountable setting
from these principles. 
He asked whether they are consistent for $\kappa$-analytic subsets of ${}^\kappa\kappa$ \cite{zdomskyy}*{pages 6-7}. 
We solve this for all subsets in $\defsets\kappa$ in Corollary~\ref{cor: Hurewicz cons} below. 
Note that in the countable setting, 
$\THD_\omega(\analytic,\Gdelta)$ is equivalent to the Hurewicz dichotomy for all analytic subsets 
of $0$-dimensional Polish spaces. 

\begin{remark}
\label{THD strong version non weakly compact}
\todol{New remark. Copy over to JEMS version}
Suppose that $\kappa>\omega$ is not weakly compact and $X\subseteq Y\subseteq{}^\kappa\kappa$. Then $\THD_\kappa(X,Y)$ is equivalent to the following seemingly stronger statement, where the second option is replaced by the second option in $\THD_\kappa(X)$:
\begin{quotation}
Either 
$X$ is contained in a $K_\kappa$ subset of $Y$, 
or
$X$ contains a closed subset of ${}^\kappa\kappa$ which is homeomorphic to ${}^\kappa\kappa$.
\end{quotation}
In particular, $\THD_\kappa(X,Y)$ implies
$\THD_\kappa(X)$ 
and also $\THD_\kappa(X,Z)$ for all $Z\subseteq{}^\kappa\kappa$ with $Y\subseteq Z$.

To prove the equivalence, it suffices to show that if $X$ contains a subset which is homeomorphic to ${}^\kappa\kappa$, then it also contains such a closed subset.
This follows from the observation that any homeomorphism $f$ from ${}^\kappa\kappa$ onto a subset of $X$ is a continuous homomorphism from $\dhH\kappa$ to the hypergraph $\dhIkappa\restr X$ of all injective $\kappa$-sequences in $X$,\footnote{See Definition~\ref{def: dhI}.
Moreover, recall that a $\ddim$-hypergraph is a $\ddim$-dihypergraph that is closed under permutations of hyperedges.} 
and hence its contains a $\kappa$-perfect subset by Lemma~\ref{lemma: kappa-perfect subdh}. 
Thus $X$ also contains a  closed homeomorphic image of ${}^\kappa 2$.
The latter is homeomorphic to ${}^\kappa\kappa$ by \cite{hung1973spaces}*{Theorem 1}. 
\end{remark}

We first obtain $\THD_\kappa(X,Y)$ as a special case of $\ODD\kappa\kappa(X)$.
In fact, $\ODD\kappa\kappa(X,\defsetsk)$ suffices for all $Y\in\defsetsk$. 
We use the hypergraph $\dhth Y=\dhIkappa\cap \dhN Y$ 
from Definition~\ref{def: dhth}. 
Note that $\dhth Y$ is box-open on ${}^\kappa\kappa$ 
and $\dhth Y\in\defsetsk$ for all $Y\in\defsetsk$.
Moreover, 
a subset $A$ of ${}^\kappa\kappa$ is $\dhth Y$-independent if and only if every injective $\kappa$-sequence in $A$ has a subsequence converging to an element of $Y$.%
\footnote{To see the direction from left to right, suppose $\bar a$ is an injective sequence in $A$. 
Since $\bar a\notin\dhth Y$, it has a convergent subsequence $\bar{b}$ with limit $x\in \ran(\bar{b})\cup Y$. We can assume $x\notin \ran(\bar{b})$ by passing to a further subsequence. Since $\bar{b} \notin \dhth Y$, we have $x\in Y$.} 
Recall that $A'$ denotes the set of {limit points} of $A$.
\todoo{We corrected the typo: wrote $A$ instead of $X$}

\begin{lemma}
\label{dhth and compactness} 
Suppose $A$ and $Y$ are subsets of ${}^\kappa\kappa$.
$A$ is $\dhth Y$-independent if and only if 
it is $\kappa$-precompact and $A'\subseteq Y$.
\end{lemma}
\begin{proof} 
First, suppose that $A$ is $\dhth Y$-independent. 
Then $A'\subseteq Y$,
since otherwise there exists
an injective $\kappa$-sequence $\bar a$ in $A$ converging to an element $x$ of ${}^\kappa\kappa\,\setminus\, Y$ with $x\notin\ran(\bar a)$ and hence 
$A$ is not $\dhth Y$-independent.
We claim that $A$ is $\kappa$-precompact.
Otherwise, take a sequence 
$\langle U_\alpha:\alpha<\kappa\rangle$ of 
pairwise disjoint%
\footnote{We may assume that the basic open sets $U_\alpha$'s are disjoint by replacing
each $U_\alpha$ with a family of
${\lleq}\kappa$ many pairwise disjoint basic open sets
whose union is
$U_\alpha\oldsetminus\bigcup_{\beta<\alpha} U_\beta$
and then removing those sets whose intersection with $\bar A$ is empty.}
basic open sets with $\bar A\subseteq \bigcup_{\alpha<\kappa} U_\alpha$
and $\bar A\cap U_\alpha\neq\emptyset$ for all $\alpha<\kappa$.
Let $a_\alpha\in A\cap U_\alpha$ for each $\alpha<\kappa$.
Since $\bar a:=\langle a_\alpha:\alpha<\kappa\rangle$ is not in 
$\dhth Y$, it has a subsequence $\bar{b}$ of length $\kappa$ converging to some $y\in Y$.\todog{$y\in {}^\kappa\kappa$ suffices here instead of $y\in Y$. This part of the argument shows that if $\closure A$ is not $\kappa$-compact, then $A$ is not $\dhthkk$-independent} 
Let $\beta<\kappa$ with $y\in U_\beta$. 
Since at most one element of $\bar{b}$ is in $U_\beta$, $\bar{b}$ cannot converge to $y$. 

Conversely, suppose that 
$A$ is $\kappa$-precompact and $A'\subseteq Y$.
To show that $A$ is $\dhth Y$-independent,
let $\bar a$ be an injective sequence in $A$.
Then $\bar a$ has a convergent subsequence $\bar b$, since otherwise $\ran(\bar a)$ is a closed discrete set of size $\kappa$, and hence $A$ cannot be $\kappa$-precompact.
Moreover, the limit of $\bar b$ is in $Y$ since $A'\subseteq Y$.
Hence $\bar a\notin\dhth Y$.
\end{proof}

In particular, if $A\subseteq Y\subseteq{}^\kappa\kappa$, then $A$ is $\dhth Y$-independent if and only if $\bar A$ is a $\kappa$-compact subset of $Y$.

\begin{theorem}
\label{theorem: THD and ODD dhth}
Suppose $X$ and $Y$ are subsets of ${}^\kappa\kappa$ with $X\subseteq Y$. 
\begin{enumerate-(1)}
\item\label{THD ODD 1}
$\dhth Y\restr X$ has a $\kappa$-coloring if and only if $X$ is 
contained in 
a $K_\kappa$ subset of $Y$.
\item\label{THD ODD 2} 
There exists a continuous homomorphism from $\dhH\kappa$ to $\dhth Y\restr X$ 
if and only if $X$ contains a relatively closed subset of $Y$ that is homeomorphic to ${}^\kappa\kappa$.
\end{enumerate-(1)}
Thus, $\THD_\kappa(X,Y)$ is equivalent to $\ODD\kappa{\dhth Y\restr X}$
and
in particular, $\THD_\kappa(X)$ is equivalent to $\ODD\kappa{\dhthkk\restr X}$.
\end{theorem}

\begin{proof}
For \ref{THD ODD 1}, 
suppose first that $\dhth Y\restr X $ has a $\kappa$-coloring. 
Take $\dhth Y$-independent sets $X_\alpha$ for $\alpha<\kappa$ with 
$X=\bigcup_{\alpha<\kappa} X_\alpha$.
Since $X_\alpha$ is a $\dhth Y$-independent subset of $Y$, its closure
$\closure X_\alpha$ is a $\kappa$-compact subset of $Y$ by the previous lemma.
Thus, 
$K:=\bigcup_{\alpha<\kappa} \closure X_\alpha $
is a $K_\kappa$ set with $X\subseteq K\subseteq Y$.
Conversely, suppose that $X\subseteq \bigcup_{\alpha<\kappa}X_\alpha\subseteq Y$, where each $X_\alpha$ is $\kappa$-compact. 
Since $X_\alpha$ is $\dhth Y$-independent 
by the previous lemma, we obtain a $\kappa$-coloring of $X$.
\ref{THD ODD 2} follows from
Lemma~\ref{lemma: dhN dhth}~\ref{dhth closed injective maps}
and
Theorem~\ref{theorem: strong variants for dhI dhth}~\ref{ODD iff ODDC for dhth}
for $H:=\dhth Y\restr X$.
\end{proof}

\begin{remark}
\label{remark: THD alternative proof}
We sketch an alternative proof of 
 $\ODD\kappa{\dhthkk\restr X}
\Longrightarrow \THD_\kappa(X)$
based on previous work.
If $\kappa$ is weakly compact, then 
the proof in
\cite{CarroyMillerSoukup}*{Proposition~2.1}  for the countable case can be lifted without major changes. 
\todog{Also, another step of the argument only works when $\kappa$ is inaccessible.}
Weak compactness is used to ensure that all bounded subsets of ${}^\kappa\kappa$ are $\kappa$-precompact. 
If $\kappa$ is not weakly compact, then 
${}^\kappa 2$ is homeomorphic to ${}^\kappa\kappa$ \cite{hung1973spaces}*{Theorem 1}. 
In this case,  
$\THD_\kappa(X)$ already follows from the $\kappa$-perfect set property $\PSP_\kappa(X)$
by \cite{LuckeMottoRosSchlichtHurewicz}*{Proposition~2.7}
or Proposition~\ref{remark: PSP THD non weakly compact} below.
\todol{Remark~\ref{remark: PSP THD non weakly compact}  was moved to Subsection~\ref{subsection: comparing applications} and improved}
\end{remark}

We next consider a different version of the Hurewicz dichotomy for ${}^\kappa\kappa$. 
A node $t$ in a tree $T$ is called \emph{$\lle\kappa$-splitting} (resp.~\emph{$\kappa$-splitting})
\index{node!splitting B@$\lle\kappa$-splitting\idf}%
\index{node!splitting C@$\kappa$-splitting\idf}%
if it has ${<}\kappa$ many (resp.~precisely $\kappa$ many) direct successors in $T$.

\begin{definition}
\label{def: superperfect}
\label{def: <kappa-splitting tree}
Suppose $T$ is a subtree of ${}^{<\kappa}\kappa$.
\begin{enumerate-(a)}
\item 
$T$ is called \emph{$\lle\kappa$-splitting} 
\index{tree!splitting tree@$\lle\kappa$-splitting tree\idf}%
if every node in $T$ is $\lle\kappa$-splitting. 
\item 
$T$ is \emph{$\kappa$-superperfect} 
\index{tree!superperfect@$\kappa$-superperfect\idf}%
\index{superperfect@$\kappa$-superperfect!tree\idf}%
if it is $\lle\kappa$-closed and 
every $s\in T$ extends to a 
$\kappa$-splitting node $t\in T$. 
\item
A subset $X$ of ${}^\kappa\kappa$ is \emph{$\kappa$-superperfect} 
\index{superperfect@$\kappa$-superperfect!set\idf}
if $X=[T]$ for some $\kappa$-superperfect subtree $T$ of ${}^{<\kappa}\kappa$. 
\end{enumerate-(a)}
\end{definition}

\begin{definition}[\cite{Hurewiczdef}]
\label{def: Hwd}
The \emph{Hurewicz dichotomy}  
\index{Hurewicz dichotomy for a!set\idf$\HD_\kappa(X)$}
for a subset $X$ of ${}^\kappa\kappa$ is the statement:
\begin{quotation}
{$\HD_\kappa(X)$:}
Either
there exists a sequence $\langle T_\alpha:\alpha<\kappa\rangle$ of $\lle\kappa$-splitting subtrees of ${}^{<\kappa}\kappa$ with $X\subseteq\bigcup_{\alpha<\kappa}[T_\alpha]$,
or
$X$ contains a $\kappa$-superperfect subset.
\end{quotation}
For any class $\mathcal C$, let 
\index{Hurewicz dichotomy for a!class\idf$\HD_\kappa(\mathcal C)$}
$\HD_\kappa(\mathcal C)$ state that $\HD_\kappa(X)$ holds for all subsets $X\in \mathcal C$ of ${}^\kappa\kappa$. 
\end{definition}

We next obtain $\HD_\kappa(X)$ as a special case of $\ODD\kappa\kappa(X,\defsetsk)$. 
\label{def: dhhd}
Let $\dhhd$ 
\index{hypergraph!of sequences!HHsuper@$\kappa$-splitting\idf$\dhhd$}
denote the $\kappa$-hypergraph on ${}^\kappa\kappa$ which consists of all sequences $\langle x_\alpha:\alpha<\kappa\rangle$ such that all the $x_\alpha$'s split at the same node.%
\footnote{I.e., the $x_\alpha$'s extend pairwise different direct successors of some $t\in{}^{<\kappa}\kappa$.}
$\dhhd$ was defined in \cite{CarroyMillerSoukup} for $\kappa=\omega$.%
\footnote{In \cite{CarroyMillerSoukup}, $\dhhdom$ is denoted by ${\mathbb{H}}_{{}^\omega\omega}'$.}
Note that $\dhhd$ is definable and box-open on ${}^\kappa\kappa$.

\begin{lemma}
\label{lemma: dhhd independence}
A subset $Y$ of ${}^\kappa\kappa$ is $\dhhd $-independent if and only if $T(Y)$ is $\lle\kappa$-splitting.
\end{lemma}
\begin{proof} 
If $T(Y)$ has a $\kappa$-splitting node, then $Y$ contains a sequence in $\dhhd$. 
The converse is similar. 
\end{proof}

\begin{theorem}
\label{theorem: dhhd from ODD}
Suppose $X$ is a subset of ${}^\kappa\kappa$.
\begin{enumerate-(1)}
\item\label{dhhd from ODD 1}
$\dhhd\restr X$ admits a $\kappa$-coloring if and only if there exists a sequence $\langle T_\alpha : \alpha<\kappa \rangle$ of $\lle\kappa$-splitting trees with 
$X\subseteq\bigcup_{\alpha<\kappa}[T_\alpha]$. 
\item\label{dhhd from ODD 2}
There exists a continuous homomorphism from $\dhH\kappa$ to $\dhhd\restr X$ 
if and only if $X$ contains a $\kappa$-superperfect subset.
\end{enumerate-(1)}
Thus, $\HD_\kappa(X)$ is equivalent to 
$\ODD\kappa{\dhhd{}\restr X}$.
\end{theorem}
\begin{proof}
For \ref{THD ODD 1}, 
we first suppose that $\dhhd\restr X$ has a $\kappa$-coloring. 
Take $\dhhd$-independent sets $X_\alpha$ for $\alpha<\kappa$ with 
$X=\bigcup_{\alpha<\kappa} X_\alpha$.
Then $X\subseteq\bigcup_{\alpha<\kappa} \bar{X}_\alpha=\bigcup_{\alpha<\kappa} [T(X_\alpha)]$.
Since $X_\alpha$ is $\dhhd$-independent, $T(X_\alpha)$ is ${<}\kappa$-splitting by the previous lemma.
Conversely, suppose that $X\subseteq \bigcup_{\alpha<\kappa} [T_\alpha]$, where each $T_\alpha$ is ${<}\kappa$-splitting. 
\todog{We may assume this by replacing $T_\alpha$ with $T([T_\alpha])$ if necessary.
We need to assume that the $T_\alpha$'s to be pruned to apply the previous lemma.}
We may assume that each $T_\alpha$ is pruned. 
Then each $[T_\alpha]$ is $\dhhd$-independent 
by the previous lemma, yielding a $\kappa$-coloring of~$X$. 

For \ref{dhhd from ODD 2}, suppose first
that there exists a continuous homomorphism from $\dhH\kappa$ to $\dhhd\restr X$.
Then there is a $\perp$-preserving continuous
order homomorphism $\iota$ for $(X,\dhhd)$
by Lemmas~\ref{homomorphisms and order preserving maps} and~\ref{order homomorphisms and injective H}~\ref{dhI perp}.
Then 
$T(\iota)$ is a $\kappa$-superperfect tree, 
since it is $\lle\kappa$-closed by Corollary~\ref{cor: T(e) closed}
and each $\iota(t)$ is a $\kappa$-splitting node of $T(\iota)$.
Since $\dhhd \subseteq\dhthkk$,
$[\iota]$ is a closed map from ${}^\kappa\kappa$ to ${}^\kappa\kappa$ by Lemma~\ref{lemma: dhN dhth}~\ref{dhth order homomorphism}.
In particular, $\ran([\iota])$ is closed. 
Thus, $\ran([\iota])=[T(\iota)]$ is a $\kappa$-superperfect subset of $X$.

Conversely, suppose 
that $T$ is a $\kappa$-superperfect tree with $[T]\subseteq X$.
Let $\iota:{}^{<\kappa}\kappa\to T$ 
be any order preserving map
such that for all $u\in{}^{<\kappa}\kappa$,
$\langle \iota(u)\conc \langle \alpha\rangle:\alpha<\kappa\rangle$ is an injective enumeration of the direct successors of some $\kappa$-splitting node $t$ of $T$ with $\iota(u)\subseteq t$. 
Then $\iota$ is an order homomorphism for $\dhhd$ with $\ran([\iota])\subseteq X$,
so $[\iota]$ is a continuous homomorphism from $\dhH\kappa$ to $\dhhd\restr X$
by Lemma~\ref{homomorphisms and order preserving maps}.
\end{proof}

The next result follows from  Theorems~\ref{main theorem}, \ref{theorem: THD and ODD dhth} and~\ref{theorem: dhhd from ODD}.
The claim for $\HD_\kappa(\defsetsk)$ and $\THD_\kappa(\defsetsk)$  
was shown in \cite{Hurewiczdef}. 

\begin{corollary}
\label{cor: Hurewicz cons}
The following hold in all $\Col(\kappa,\lle\lambda)$-generic extensions of $V$: 
\begin{enumerate-(1)}
\item\label{cor: Hurewicz cons 1}
$\HD_\kappa(\defsetsk)$,
$\THD_\kappa(\defsetsk)$ and
$\THD_\kappa(\defsetsk,\defsetsk)$ 
if $\lambda>\kappa$ is inaccessible in $V$.
\item\label{cor: Hurewicz cons 2}
$\THD_\kappa(\defsetsk,\pwrset({}^\kappa\kappa))$ 
if $\lambda>\kappa$ is Mahlo in~$V$.
\end{enumerate-(1)}
\end{corollary}

If $\kappa$ is not weakly compact, then it suffices in \ref{cor: Hurewicz cons 2} to assume that $\lambda$ is inaccessible, since 
$\THD_\kappa(\defsetsk,\pwrset({}^\kappa\kappa))$ 
follows from $\THD_\kappa(\defsetsk,\defsetsk)$ by Remark~\ref{THD strong version non weakly compact}.
Alternatively, it also follows from
$\PSP_\kappa(\defsetsk)$ by \cite{LuckeMottoRosSchlichtHurewicz}*{Proposition~2.7} or Proposition~\ref{remark: PSP THD non weakly compact} below.

Figure~\ref{figure: Hwd options simple}
summarizes the
 implications between the two options in $\HD_\kappa(X)$ and $\THD_\kappa(X)$
under varying assumptions on $\kappa$. 
A double arrowhead denotes strict implication.
\vspace{10 pt}

{
\setlength{\intextsep}{10 pt}
\setlength{\abovecaptionskip}{7 pt}
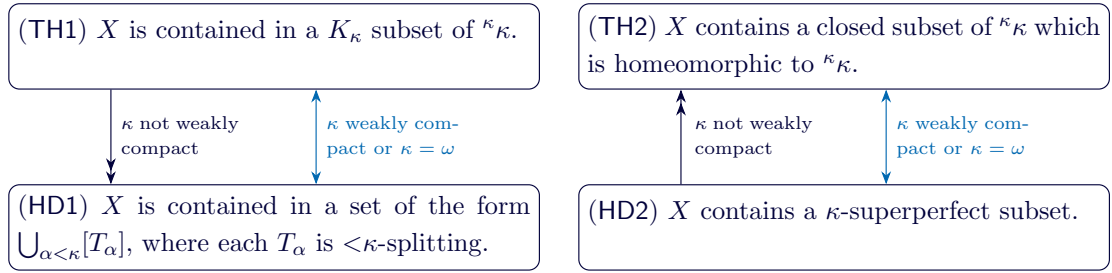
\begin{figure}[H]
\begin{tikzpicture}[scale=.3]
\tikzstyle{theory1}=[draw,rounded corners,scale=.5, minimum width=10mm, minimum height=6.5mm,scale=1.8,line width=\widthline,VeryDarkBlueNode]
\tikzstyle{info}=[rounded corners, inner sep= 2pt, align=left,scale=0.82]

\tikzstyle{medarrow}=[>=Stealth, line width=\widtharrow, scale=1.8] 
%
%
     \node (1) at (25,-3) [theory1] {\parbox[c][1cm]{7.5cm}{\thl 
$X$ contains a closed subset of ${}^\kappa\kappa$ which is homeomorphic to ${}^\kappa\kappa$.
}}; 
     \node (2) at (25,-11) [theory1] {\parbox[c][1cm]{7.5cm}{\hdl 
$X$ contains a $\kappa$-superperfect subset. 
\\ \ 
}}; 
     \node (4) at (0,-3) [theory1] {\parbox[c][1cm]{7.5cm}{\ths 
$X$ is contained in a $K_\kappa$ subset of ${}^\kappa\kappa$.
\\ \ 
}}; 
     \node (5) at (0,-11) [theory1] {\parbox[c][1cm]{7.5cm}{\hds 
$X$ is contained in a set of the form $\bigcup_{\alpha<\kappa}[T_\alpha]$, where each $T_\alpha$ is $\lle\kappa$-splitting.
}}; 

\draw[medarrow, BlueArrow]
([xshift=75mm]1.south west) edge[<->] 
node[midway,right] {\parbox[c]{1.75 cm}{\tiny $\kappa$ weakly compact or $\kappa=\omega$}} 
node {\hyperref[proposition: HD and THD options]{\phantom{xxx}}} 
([xshift=75mm]2.north west) 
([xshift=75mm]4.south west) edge[<->] 
node[midway,right] {\parbox[c]{1.75 cm}{\tiny $\kappa$ weakly compact or $\kappa=\omega$}} 
node {\hyperref[proposition: HD and THD options]{\phantom{xxx}}} 
([xshift=75mm]5.north west); 

\draw[medarrow,VeryDarkBlueNode]
([xshift=25mm]2.north west) edge[->>] 
node[midway,right,inner sep=5 pt] {\parbox[c]{1.55 cm}{\tiny $\kappa$~not weakly compact }} 
([xshift=25mm]1.south west) 
([xshift=25mm]4.south west) edge[->>] 
node[midway,right] {\parbox[c]{1.55 cm}{\tiny $\kappa$~not weakly compact}} 
([xshift=25mm]5.north west); 
\end{tikzpicture}
\caption{The options in $\THD_\kappa(X)$ and $\HD_\kappa(X)$}
\label{figure: Hwd options simple}
\end{figure}
}

It is clear that the implications 
\ths $\Longrightarrow$ \hds 
and 
\hdl $\Longrightarrow$ \thl
hold for all $X\subseteq{}^\kappa\kappa$
at all infinite cardinals $\kappa=\kappa^{<\kappa}$.
The converse implications hold at weakly compact cardinals $\kappa$ and $\kappa=\omega$ by the next proposition.
However, they fail for $X={}^\kappa 2$ at all uncountable non-weakly compact cardinals $\kappa$:
\hds holds since ${}^{<\kappa} 2$ is ${<}\kappa$-splitting
and \thl holds since ${}^\kappa 2$ is homeomorphic to ${}^\kappa\kappa$ 
\cite{hung1973spaces}*{Theorem~1}. 

\begin{proposition}
\label{proposition: HD and THD options}
If $\kappa$ is weakly compact or $\kappa=\omega$, then the corresponding options in Figure~\ref{figure: Hwd options simple} are equivalent for all $X\subseteq{}^\kappa\kappa$.
In particular, $\THD_\kappa(X)\Longleftrightarrow\HD_\kappa(X)$.
\end{proposition}
\begin{proof}
We show that $\dhthkk\equivf\dhhd$.%
\footnote{See Definition~\ref{def: H-full}.} 
Then the corresponding options in $\ODD\kappa{\dhthkk\restr X}$ and $\ODD\kappa{\dhhd\restr X}$ 
are equivalent for all $X\subseteq{}^\kappa\kappa$
by Corollary~\ref{cor: ODD subsequences}, and hence 
the claim then follows from 
Theorems~\ref{theorem: THD and ODD dhth} and~\ref{theorem: dhhd from ODD}.
Since $\dhhd\subseteq\dhthkk$,
it suffices to show that any hyperedge $\bar x$ of $\dhthkk$ has a subsequence in $\dhhd$.
It follows from the partition property of weakly compact cardinals that
$\bar x$ has a subsequence $\bar y$ such that all triples in $\bar y$ have the same splitting type. 
Then $\bar y\in\dhhd$ since otherwise it would converge, 
contradicting $\bar x\in\dhthkk$.
\end{proof}

The previous proposition also follows from 
\cite{LuckeMottoRosSchlichtHurewicz}*{Lemma 2.6 \& Proposition~2.11}
by the next observation.

\begin{remark}
\label{remark: HD small option eventually bounded}
It is easy to see that if $X$ is eventually bounded, then the first option in $\HD_\kappa(X)$ holds.
The converse holds for inaccessible cardinals $\kappa$ and $\kappa=\omega$ but fails for uncountable cardinals $\kappa$ that are not inaccessible by the following example from \cite{Hurewiczdef}. 
Suppose that $\mu<\kappa$ is chosen such that $2^\mu=\kappa$. 
Let $\langle s_\alpha : \alpha<\kappa \rangle$ 
enumerate ${}^\mu2$ without repetitions. 
Let $\iota:{}^{<\kappa}\kappa\to{}^{<\kappa}\kappa$ be a continuous strict order preserving map such that
$\iota(\emptyset)=\emptyset$
and 
$\iota(t\conc\langle\alpha\rangle)=\iota(t)\conc s_\alpha\conc\langle\alpha\rangle$
for all $t\in{}^{<\kappa}\kappa$ and $\alpha<\kappa$.
Then $X:=\ran([\iota])$ is not eventually bounded,
since for any $y\in {}^\kappa\kappa$, there exists some $x \in X$ with $y(\mu\cdot (\alpha+1))=x(\alpha)$ for all $\alpha<\kappa$. 
However,
$T(X)=T(\iota)$ 
is $\lleq 2$-splitting.%
\footnote{I.e., every node in $T(X)$ has at most $2$ direct successors.}
Note that $X$ is closed by Lemma~\ref{lemma: [e] closed map}~\ref{closure of ran[e] when ddim=kappa}
and
$T(X)$ is $\lle\kappa$-closed
by Lemma~\ref{lemma: T(e) closure properties}.
\end{remark}

\subsection{Extensions of the Kechris--Louveau--Woodin dichotomy}
\label{subsection: KLW}

We show that the open dihypergraph dichotomy implies 
an extension 
of the Kechris--Louveau--Woodin dichotomy \cite{KechrisLouveauWoodin1987}*{Theorem 4} 
to subsets $X,\,Y$ of  ${}^\kappa\kappa$  which are not necessarily disjoint.
We will denote this extension by $\KLW_\kappa(X,Y)$.
We further derive a more general dichotomy
$\KLW^\ddim_\kappa(X,Y)$
that is defined relative to arbitrary box-open $\ddim$-dihypergraphs. 
$\KLW^\ddim_\kappa(X)$ then states that this dichotomy holds for all sets $Y$. 
We will see that 
$\KLW^\ddim_\kappa(X)$ implies $\ODD\kappa\ddim(X)$ for all dimensions $2\leq\ddim\leq\kappa$.
In fact, the two are equivalent for $\ddim=\kappa$.
%
We further show that
$\ODD\kappa\kappa({}^\kappa\kappa)$ implies $\ODD\kappa\kappa(\analytic(\kappa))$. 
\todog{This needs: (1) $\KLW^{d}_\kappa(X,Y)$ implies $\ODD\kappa\ddim(X\cap Y)$ for all $X$ and all closed $Y$, (2) $\ODD\kappa\kappa(X)$ implies $\KLW^{d}(X,Y)$ for all $d$ and $Y$.}
Note that the latter statement is clear for 
$\kappa$-analytic
sets which are continuous images of ${}^\kappa\kappa$ 
by Lemma~\ref{ODD for continuous images}, but not every $\kappa$-analytic subset of ${}^\kappa\kappa$ is of this form if $\kappa$ is uncountable \cite{MR3430247}*{Theorem~1.5}.

\subsubsection{A version for arbitrary sets}
\label{subsection: original KLW}

\index{embedding!of kj@of ${}^{<\kappa}\kappa$ into ${}^{<\kappa} 2$\idf$\pi$}%
\index{embedding!of kk@of ${}^\kappa\kappa$ into ${}^\kappa 2$\idf$[\pi]$}%
Define $\pi:{}^{<\kappa}\kappa\to{}^{<\kappa} 2$ by letting 
$$\pi(t):=
\bigoplus_{\alpha<\lh(t)} ( \langle0\rangle^{t(\alpha)}\conc\langle 1\rangle ) 
$$
for all $t\in {}^{<\kappa}\kappa$.%
\footnote{A map satisfying the recursive definition of $\pi$ at successor levels already appeared in the proof of Lemma~\ref{lemma:ODDH and not ODDC}.} 
\todoo{We added the footnote (instead of a similar footnote which we deleted from the proof of Lemma~\ref{lemma:ODDH and not ODDC}).}
Then 
\index{sequences!with cofinally many nonzero values in!c@${}^{<\kappa} 2$\idf$\SS_\kappa$} 
$$
\SS_\kappa := \SS^2_\kappa := \ran(\pi) 
$$
consists of all nodes $u$ such that $u(\alpha)=1$ for cofinally many $\alpha<\lh(u)$. In particular, it contains all nodes of successor length ending with $1$. 
Let further 
\index{rationals@$\kappa$-rationals in!c@${}^\kappa 2$\idf$\QQ_\kappa$} 
\index{irrationals@$\kappa$-irrationals in!c@${}^\kappa 2$\idf$\RR_\kappa$} 
\todoo{Typo corrected in the definition of $\QQ_\kappa$ in the equation: we replaced ${}^\kappa\kappa$ by ${}^\kappa 2$}
\begin{equation*}
\begin{gathered}
\RR_\kappa:=\RR^2_\kappa:=\ran([\pi]), 
\\
\QQ_\kappa:=\QQ^2_\kappa:={}^\kappa 2\, \setminus\, \RR_\kappa.
\end{gathered}
\end{equation*}
Then $\QQ_\kappa$ consists of those $x\in{}^\kappa 2$ with 
$x(\alpha)=0$ for a final segment of $\alpha<\kappa$ and $\RR_\kappa$ 
of those with 
$x(\alpha)=1$ for cofinally many $\alpha<\kappa$. 
Since $\pi$ is $\perp$- and strict order preserving, 
$[\pi]$ 
is a homeomorphism from ${}^\kappa\kappa$ onto $\RR_\kappa$.

We call a map 
$f:{}^\kappa \ddim\to{}^{\kappa}\kappa$ 
a \emph{reduction of}
\index{reduction of!subsets of ${}^\kappa\kappa$\idf}
$(A,\,{{}^\kappa\ddim\,\setminus A})$ to $(X,Y)$ 
if $f(A)\subseteq X$ and $f({}^\kappa\ddim\,\setminus A)\subseteq Y$.
Recall that for any subset $X$ of ${}^\kappa\kappa$, 
$X'$ denotes the set of {limit points} of~$X$.

\begin{definition}
\label{def: KLW dichotomy}
\index{Kechris--Louveau--Woodin dichotomy!A@for arbitrary sets!A@standard\idf$\KLW_\kappa(X,Y)$} 
For subsets $X,Y$ of ${}^\kappa\kappa$, let
$\KLW_\kappa(X,Y)$ denote the statement: 
\begin{quotation}
$\KLW_\kappa(X,Y)$: 
Either $X=\bigcup_{\alpha<\kappa} X_\alpha$ for some sequence $\langle X_\alpha : \alpha<\kappa\rangle$ of sets with $X'_\alpha\cap Y=\emptyset$ for all $\alpha<\kappa$,
or there is a homeomorphism $f$
from ${}^\kappa 2$ onto a closed subset of ${}^\kappa\kappa$
that reduces
$(\RR_\kappa,\QQ_\kappa)$ to $(X,Y)$.
\end{quotation}
\end{definition}
If the first option holds, then $|X\cap Y|\leq\kappa$, since $X_\alpha\cap Y$ is a discrete space for each $\alpha<\kappa$ and hence it has size at most $\kappa$.
If the second option holds, then 
$f(\QQ_\kappa)$ and $f(\RR_\kappa)$ are 
disjoint dense subsets of $\ran(f)$ such that $|f(\QQ_\kappa)|=\kappa$ and $f(\RR_\kappa)$ is homeomorphic to~${}^\kappa\kappa$. 

We further say that a set $A$ \emph{separates} 
$X$ from $Y$ if $X\subseteq A$ and $A\cap Y=\emptyset$. 
For disjoint subsets $X$ and $Y$ of ${}^\kappa\kappa$, it is easy to see that 
$\KLW_\kappa(X,Y)$ is equivalent to the statement: 
\begin{quotation}
Either 
there is a $\Fsigma(\kappa)$ set $A$ separating $X$ from $Y$, 
or there is a homeomorphism 
from ${}^\kappa 2$ onto a closed subset of ${}^\kappa\kappa$
that reduces
$(\RR_\kappa,\QQ_\kappa)$ to $(X,Y)$.
\end{quotation}
For disjoint subsets $X,\, Y$ of ${}^\omega\omega$ with $X$ analytic, this is precisely the 
Kechris--Louveau--Woodin dichotomy \cite{KechrisLouveauWoodin1987}*{Theorem 4}.\footnote{See also \cite{KechrisBook}*{Theorem 21.22}.}
Thus $\KLW_\omega(X,Y)$ is an extension of their statement to arbitrary subsets $X$ and $Y$.
The case when $X$ and $Y$ are not disjoint will be useful for later applications.

\begin{definition}
\label{def: nice reduction} 
\index{nice reduction\idf}%
\index{reduction of!Z@\mygobble|see {nice reduction}}%
Let $X$ and $Y$ be subsets of ${}^\kappa\kappa$.
A \emph{nice reduction} for $(X,Y)$ is a
$\perp$- and strict order preserving map 
$\Phi:{}^{<\kappa}2\to{}^{<\kappa}\kappa$
such that $[\Phi]$ reduces $(\RR_\kappa,\QQ_\kappa)$ to $(X,Y)$.
\end{definition}

Note that if $\Phi$ is such a map, then
$[\Phi]$ is a homeomorphism from ${}^\kappa 2$ onto a closed 
subset of~${}^\kappa\kappa$ by Lemma~\ref{perfect subsets and sop maps}~\ref{sop maps -> perfect subsets}.

\begin{definition}
\label{def: strong KLW dichotomy}
\index{Kechris--Louveau--Woodin dichotomy!A@for arbitrary sets!A@variant\idf$\sKLW_\kappa(X,Y)$} 
For subsets $X,Y$ of ${}^\kappa\kappa$, let
$\sKLW_\kappa(X,Y)$ denote the statement: 
\begin{quotation}
$\sKLW_\kappa(X,Y)$: 
Either 
$X=\bigcup_{\alpha<\kappa} X_\alpha$ for some sequence $\langle X_\alpha : \alpha<\kappa\rangle$ of sets with $X'_\alpha\cap Y=\emptyset$ for all $\alpha<\kappa$,
or there is a nice reduction 
for $(X,Y)$.
\end{quotation}
\end{definition} 

$\sKLW_\kappa(X,Y)$ clearly implies $\KLW_\kappa(X,Y)$ 
and moreover, they are
equivalent by the results in this subsection. 
In fact, we will show that each one  
is equivalent to 
the same special case of $\ODD\kappa\kappa(X)$ 
in Theorem~\ref{theorem: KLW from ODD}.
Moreover, we will see that 
it suffices that the map $f$ 
in $\KLW_\kappa(X,Y)$
is continuous 
with $f(\RR_\kappa)\cap f(\QQ_\kappa)=\emptyset$
instead of a homeomorphism onto a closed set.%
\footnote{See Theorem~\ref{theorem: KLW from ODD}~\ref{KLW from ODD hom}.}
The proofs of these equivalences are similar to that of \cite{CarroyMiller}*{Theorem~3 \& Proposition~7} but avoid the use of 
compactness
and are perhaps more direct.

We first define the relevant hypergraphs.\footnote{Recall that a $\kappa$-hypergraph is a $\kappa$-dihypergraph that is closed under permutations of hyperedges.} 
\label{def of CC}
Given a subset $Y$ of ${}^\kappa\kappa$, 
\index{hypergraph!of sequences!CCa@convergent with limit in $Y$ but not in the sequence\idf$\CC^Y$} 
we write $\CC^Y$ for the $\kappa$-hypergraph consisting of all 
convergent sequences $\bar x$ in ${}^\kappa\kappa$ 
with $\lim_{\alpha<\kappa}(\bar x)\in Y\setminus \ran(\bar x)$.%
\footnote{I.e., $\lim_{\alpha<\kappa}(\bar x)\neq x_\alpha$ for all $\alpha<\kappa$.}
$\CC^Y$ was defined in the proof of \cite{CarroyMiller}*{Theorem~3} for $\kappa=\omega$.
Note that $\CC^Y$ is box-open on ${}^\kappa\kappa$, 
and $\CC^Y\in\defsetsk$ if $Y\in\defsetsk$. 

The following lemmas characterize when $\CC^Y$ admits a $\kappa$-coloring. 
The next lemma is immediate. 

\begin{lemma} 
\label{lemma: KLW independence}
For subsets $Y$ and $Z$ of ${}^\kappa\kappa$, 
$Z$ is $\CC^Y$-independent if and only if $Y\cap Z'=\emptyset$. 
\end{lemma} 

In the special case that $Y$ is disjoint from $Z$, $Z$ is $\CC^Y$-independent if and only if $Y\cap\closure{Z}=\emptyset$, 
since $\closure{Z}=Z\cup Z'$. 

\begin{lemma}
\label{lemma: CC coloring} 
The following statements are equivalent for all 
subsets $X$ and $Y$ of ${}^\kappa\kappa$.
\begin{enumerate-(1)}
\item\label{CC coloring 1}
$\CC^Y\restr X$ admits a $\kappa$-coloring.
\item\label{CC coloring 2}
$X$ is covered by sets $X_\alpha$ for $\alpha<\kappa$ with $X_\alpha'\cap Y=\emptyset$.
\item\label{CC coloring 3}
$|X\cap Y|\leq\kappa$ and $X\setminus Y$ can be separated from $Y$ by a $\Fsigma(\kappa)$ set.
\end{enumerate-(1)}
\end{lemma}
\begin{proof}
\ref{CC coloring 1} $\Rightarrow$ \ref{CC coloring 2}:
Suppose that $\CC^Y\restr X$ admits a $\kappa$-coloring 
and take $\CC^Y\restr X$-independent sets $X_\alpha$ with
$X=\bigcup_{\alpha<\kappa} X_\alpha$.
Then $X_\alpha' \cap Y=\emptyset$ for all $\alpha<\kappa$ by Lemma \ref{lemma: KLW independence}. 

\ref{CC coloring 2} $\Rightarrow$ \ref{CC coloring 3}:
For each $\alpha<\kappa$, let $Y_\alpha:=Y\cap X_\alpha$
and $Z_\alpha:=X_\alpha\setminus Y$.
Since $Y_\alpha$ is a subset of $Y$ with no limit points in $Y$,
it is discrete, so
$|Y_\alpha|\leq\kappa$.
Since $X\cap Y = \bigcup_{\alpha<\kappa}Y_\alpha$,
we have $|X\cap Y|\leq\kappa$.
Since $Z_\alpha$ is disjoint from $Y$ and has no limit points in $Y$,
its closure is also disjoint from $Y$.
Hence $Z:=\bigcup_{\alpha<\kappa} \closure Z_\alpha$ is a $\Fsigma(\kappa)$ set separating $X\setminus Y$ from $Y$.

\ref{CC coloring 3} $\Rightarrow$ \ref{CC coloring 1}:
Since $|X\cap Y|\leq\kappa$, it suffices to show that $\CC^Y\restr(X\setminus Y)$ admits a $\kappa$-coloring. 
To see this,
take closed subsets $Z_\alpha$ of ${}^\kappa\kappa$ such that
$Z:=\bigcup_{\alpha<\kappa} Z_\alpha$
separates
$X\setminus Y$ from $Y$ and note that 
each $Z_\alpha$ is $\CC^Y$-independent by Lemma \ref{lemma: KLW independence}. 
\end{proof}

In the special case $X:={}^\kappa\kappa$, the equivalence 
\ref{CC coloring 1} $\Leftrightarrow$ \ref{CC coloring 3}
yields the following statement. 

\begin{corollary}
\label{cor: CC coloring}
For any subset $Y$ of ${}^\kappa\kappa$,
$\CC^Y$ admits a $\kappa$-coloring if and only if $Y$ is a $\Gdelta(\kappa)$ set of size~$\kappa$.
\end{corollary}

We now prove that the existence of a continuous homomorphism from $\dhH\kappa$ to $\CC^Y\restr X$ 
is equivalent to the existence of
a nice reduction for $(X,Y)$.

\begin{lemma} 
\label{lemma: continuous reduction}
Suppose that 
$X$ and $Y$ are subsets of ${}^\kappa\kappa$ and
$f$ is a continuous reduction of $(\RR_\kappa,\QQ_\kappa)$ to $(X,Y)$
with $f(\RR_\kappa) \cap f(\QQ_\kappa)=\emptyset$.
Then $f\comp [\pi]$ is a continuous homomorphism from $\dhH\kappa$ to $\CC^Y\restr X$.
\end{lemma}
\begin{proof}
Suppose that $\bar x=\langle x_\alpha:\alpha<\kappa\rangle$ is a hyperedge of $\dhH\kappa$.
Take $t\in{}^{<\kappa}\kappa$ such that $x_\alpha$ extends $t\conc\langle \alpha\rangle$ for each $\alpha<\kappa$.
Since
$$
\pi(t)\conc\langle0\rangle^\alpha \conc\langle1\rangle
=
\pi(t\conc\langle\alpha\rangle)
\subseteq [\pi](x_\alpha)
$$
for all $\alpha<\kappa$,
$\langle [\pi](x_\alpha):\alpha<\kappa\rangle$ converges to 
$y:=\pi(t)\conc \langle0\rangle^\kappa \in \QQ_\kappa$.
Since $f$ is continuous, 
$\big\langle (f\comp[\pi])(x_\alpha):\alpha<\kappa\big\rangle$ 
converges to $f(y)$.
Since $f$ is a reduction of $(\RR_\kappa,\QQ_\kappa)$ to $(X,Y)$,
this sequence is in $X$ and $f(y)\in Y$. 
Using that $f(\RR_\kappa)\cap f(\QQ_\kappa)=\emptyset$, we have 
$(f\comp[\pi])(x_\alpha)\neq f(y)$ for all $\alpha<\kappa$. 
Thus   
$\big\langle (f\comp[\pi])(x_\alpha):\alpha<\kappa\big\rangle\in \CC^Y\restr X$ as required.
\end{proof}

Suppose that $x\in{}^\kappa\kappa$ and 
$\bar{t}=\langle t_\alpha : \alpha<\kappa \rangle$ is a sequence of elements of ${}^{<\kappa}\kappa$.
\label{def: convergence of nodes}
We say that $\bar t$ \emph{converges to} $x$ 
\index{convergent sequence of!nodes of ${}^{<\kappa}\kappa$\idf}
if for all $\gamma<\kappa$, we have 
$x\restr\gamma\subseteq t_\alpha$ for all sufficiently large ordinals $\alpha<\kappa$.
It is equivalent to state that the lengths $\lh(t_\alpha)$ converge to $\kappa$ and 
some (equivalently, all) sequences $\bar{x}\in\prod_{\alpha<\kappa} N_{t_\alpha}$
converges to~$x$.

\begin{lemma}
\label{lemma: CC convergence}
Suppose that $Y$ is a subset of ${}^\kappa\kappa$ and
$\bar{t}=\langle t_\alpha : \alpha<\kappa \rangle$ is a sequence of elements of~${}^{<\kappa}\kappa$. 
The following conditions are equivalent: 
\begin{enumerate-(1)} 
\item 
\label{lemma: CC convergence 1}
$\prod_{\alpha<\kappa}N_{t_\alpha}\subseteq \CC^Y$. 
\item 
\label{lemma: CC convergence 2}
$\bar t$ converges to some $y\in Y$ with $y\perp t_\alpha$ for all $\alpha<\kappa$. 
\end{enumerate-(1)} 
\end{lemma}
\begin{proof}
\ref{lemma: CC convergence 1} $\Rightarrow$ \ref{lemma: CC convergence 2}:
Take any $\bar{x}\in \prod_{\alpha<\kappa}N_{t_\alpha}$. 
By assumption, $\bar{x}$ converges to some $y\in Y$. 
We claim that $\bar{t}$ converges to $y$. 
This is clear if the lengths $\lh(t_\alpha)$ converge to $\kappa$. 
Otherwise, there exists some $\gamma<\kappa$ and an unbounded subset $I$ of $\kappa$ such that $y\restr \gamma \not\subseteq t_\alpha$ for all $\alpha\in I$. 
We may assume that $\kappa\,\setminus\, I$ is also unbounded by shrinking $I$.
For each $\alpha\in I$,
pick some $y_\alpha \in N_{t_\alpha}$ with $y_\alpha\perp y\restr\gamma$.
Define $\bar{z}=\langle z_\alpha : \alpha<\kappa\rangle$ by letting $z_\alpha:= x_\alpha$ for $\alpha\notin I$ and $z_\alpha:= y_\alpha$ for $\alpha \in I$. 
Then $\bar{z}$ does not converge, since 
$y\restr\gamma \perp z_\alpha$ for all $\alpha\in I$ 
and
$y\restr\gamma\subseteq z_\alpha$ for unboundedly many $\alpha\in\kappa\,\setminus\, I$.
As $\bar{z}\in \prod_{\alpha<\kappa}N_{t_\alpha}$, this contradicts the assumption. 
Therefore $\bar t$ converges to $y$.
We now claim that 
$t_\alpha \perp y$ for all $\alpha<\kappa$.
Otherwise, 
there exists a hyperedge $\langle z_\beta : \beta<\kappa\rangle$ in $\CC^Y$ with 
$z_\alpha=y$ for some $\alpha<\kappa$ 
and 
$t_\beta\subseteq z_\beta$ for all $\beta<\kappa$,
but no hyperedge of $\CC^Y$ can converge to one of its elements. 

\ref{lemma: CC convergence 2} $\Rightarrow$ \ref{lemma: CC convergence 1}:
If \ref{lemma: CC convergence 2} holds, 
then any $\langle x_\alpha:\alpha<\kappa\rangle\in \prod_{\alpha<\kappa}N_{t_\alpha}$
converges to $y$ and $y\neq x_\alpha$ for all $\alpha<\kappa$.
So $\prod_{\alpha<\kappa}N_{\iota(t\conc\langle\alpha\rangle)}\subseteq \CC^Y$.
\end{proof}

\begin{remark}
\label{remark: CC convergence}
The previous proof also shows the following stronger version of 
\ref{lemma: CC convergence 1} $\Rightarrow$ \ref{lemma: CC convergence 2}:
If $X$ and $Y$ are subsets of ${}^\kappa\kappa$
and $\bar{t}=\langle t_\alpha:\alpha<\kappa\rangle$ is a sequence in $T(X)$ with $\prod_{\alpha<\kappa} (N_{t_\alpha}\cap X)\subseteq \CC^Y$, then 
$\bar{t}$ converges to some $y\in Y$.
Moreover, if $y\in X$ then $y\perp t_\alpha$ for all $\alpha<\kappa$.
Note that the last statement does not necessarily hold if $y\notin X$. 
To see this, let $X:=\RR_\kappa$, $Y:=\QQ_\kappa$
and 
$t_\alpha:=\langle 0\rangle ^\alpha$ for all $\alpha<\kappa$.
\todog{Then $\prod_{\alpha<\kappa} (N_{t_\alpha}\cap X)\subseteq \CC^Y$ and
the $t_\alpha$'s converge to $y:=\langle 0\rangle^\kappa$, but $y\not\perp t_\alpha$ for all $\alpha>0$.} 
\end{remark}

\begin{remark}
\label{remark: nice reduction}
Lemma~\ref{lemma: continuous reduction} has the following version for nice reductions.
If $\Phi$ is a nice reduction for $(X,Y)$, then $\iota:=\Phi\comp\pi$ is an order homomorphism for 
$({}^\kappa\kappa,\CC^Y)$ and $\ran([\iota])\subseteq X$.
To prove this, 
suppose $t\in{}^{<\kappa}\kappa$ and
let
$u_\alpha:=\pi(t\conc\langle\alpha\rangle)=\pi(t)\conc\langle0\rangle^\alpha \conc\langle1\rangle$ for all $\alpha<\kappa$.
Since $[\Phi]$ is continuous and $\langle u_\alpha:\alpha<\kappa\rangle$ converges to 
$\pi(t)\conc \langle0\rangle^\kappa$,
the sequence $\langle \iota(t\conc\langle\alpha\rangle):\alpha<\kappa\rangle=
\langle \Phi(u_\alpha):\alpha<\kappa\rangle$
converges to $y_t:=[\Phi](\pi(t)\conc \langle0\rangle^\kappa)\in Y$. 
As $\Phi$ is $\perp$-preserving, 
$y_t\perp\iota(t\conc\langle\alpha\rangle)$ for all $\alpha<\kappa$.
Therefore $\prod_{\alpha<\kappa}N_{\iota(t\conc\langle\alpha\rangle)}\subseteq \CC^Y$
by  Lemma~\ref{lemma: CC convergence}.
Moreover, $\ran([\iota])=\ran([\Phi]\comp[\pi])=[\Phi](\RR_\kappa)$ is a subset of $X$.
\end{remark}

\begin{remark} 
\label{ran(iota) subset T(X)}
If $\iota$ is an order homomorphism for $({}^\kappa\kappa,\CC^Y)$, then $\ran(\iota)\subseteq T(Y)$ and hence
$\ran([\iota])\subseteq\closure Y$. 
To see this, let $t\in{}^{<\kappa}\kappa$. 
Then $\langle \iota(t\conc\langle\alpha\rangle):\alpha<\kappa\rangle$ converges to some $y_t\in Y$
by Lemma \ref{lemma: CC convergence}.
Since $\iota(t)\subseteq\iota(t\conc\langle\alpha\rangle)$ for all $\alpha<\kappa$,
we have $\iota(t)\subseteq y_t$ and thus $\iota(t)\in T(Y)$.
\end{remark} 

We regard $\SS_\kappa$ and
${}^{<\kappa}\kappa$ 
as $\wedge$-semilattices with respect to inclusion. 
Note that a map $e:\SS_\kappa\to{}^{<\kappa}\kappa$  is a strict $\wedge$-homomorphism%
\footnote{See Definition~\ref{def: wedge-homomorphism}.} 
if and only if
it is strict order preserving with $e(s\conc\langle 0\rangle^\alpha\conc\langle 1\rangle) \wedge e(s\conc\langle 0\rangle^\beta\conc\langle 1\rangle) =e(s)$
for all $s\in\SS_\kappa$ and $\alpha<\beta<\kappa$.
It is easy to see that such a map is $\perp$-preserving.

\begin{lemma}
\label{lemma: nice reduction}
Suppose that
 $\iota$ is an
order homomorphism 
for $({}^\kappa\kappa,\CC^Y)$ and $\ran([\iota])\subseteq X$,
where $X,Y\subseteq{}^\kappa\kappa$.
Then there exists a nice reduction $\Phi:{}^{<\kappa}2\to{}^{<\kappa}\kappa$ for 
$(X,Y)$
and a continuous 
strict $\wedge$-homomorphism
$e:\SS_\kappa\to{}^{<\kappa}\kappa$
with 
$\iota\comp e=\Phi\restr\SS_\kappa$.
\end{lemma}
\begin{proof}
For each $t\in{}^{<\kappa}\kappa$,
let $y_t$ denote the unique element of $Y$ such that 
the sequence $\langle \iota(t\conc\langle\alpha\rangle):\alpha<\kappa\rangle$
converges to $y_t$ 
and $y_t\perp\iota(t\conc\langle\alpha\rangle)$ holds for all $\alpha<\kappa$.
$y_t$ exists
by Lemma~\ref{lemma: CC convergence}.
%
%
We shall define strict order preserving functions $\Phi: {}^{<\kappa}2\to {}^{<\kappa}\kappa$ and $e: \SS_\kappa\to {}^{<\kappa}\kappa$ such that the following hold: 
\begin{enumerate-(i)} 
\item \label{prop KLW 1} 
$\Phi{\restr}\SS_\kappa = \iota \circ e$. 
\item\label{prop KLW 3}
$e$ is a $\wedge$-homomorphism and $\Phi$ is $\perp$-preserving.
\item \label{prop KLW 2} 
$\Phi(s\conc\langle0\rangle^\alpha)\subseteq y_{e(s)}$ for all $s\in \SS_\kappa$ and $\alpha<\kappa$. 
\end{enumerate-(i)} 

We define $\Phi(t)$ and $e(t)$ by recursion. 
Regarding the order of construction,  
let $e(\emptyset):=\emptyset$. 
If $s\in \SS_\kappa$ and $\lh(s)\in\Lim$, let 
$e(s):=\bigcup_{t\subsetneq s}e(t)$.
If $e(s)$ is defined for some $s\in \SS_\kappa$, 
we define $\Phi(s\conc \langle0\rangle^{\alpha})$, $\Phi(s\conc\langle0\rangle^{\alpha}{}\conc\langle1\rangle)$ 
and
$e(s\conc\langle0\rangle^{\alpha}{}\conc\langle1\rangle)$ 
for all $\alpha<\kappa$
in the next step, by recursion on $\alpha$.
The idea for their construction 
is as follows 
(see Figure~\ref{figure: KLW proof}).
Since $\langle \iota(e(s)\conc\langle\eta\rangle):\eta<\kappa\rangle$
converges to $y_{e(s)}$,
we will be able to take a subsequence
$\langle \iota(u_\alpha):\alpha<\kappa\rangle$
such that
$\iota(u_\alpha)\wedge y_{e(s)}\subsetneq \iota(u_\beta)\wedge y_{e(s)}$ for all $\alpha,\beta<\kappa$ with $\alpha<\beta$.
Then let 
$e(s\conc\langle0\rangle^{\alpha}{}\conc\langle1\rangle):=u_\alpha$ and 
$\Phi(s\conc\langle0\rangle^{\alpha}{}\conc\langle1\rangle):=\iota(u_\alpha)$. 
Note that $u_\alpha\wedge u_\beta=e(s)$ when $\alpha<\beta<\kappa$, so $e$ will be a $\wedge$-homomorphism.
Since $y_{e(s)}\perp \iota(u_\alpha)$, 
we can define
$\Phi(s\conc\langle0\rangle^{\alpha+1})$ as some $v_\alpha\subseteq y_{e(s)}$ that extends $\iota(u_\alpha)\wedge e(s)$. 

\vspace{5pt} 
{
\newcommand{\x}{0.6} 
\newcommand{\y}{1} 
\newcommand{\z}{8} 

\begin{figure}[H]
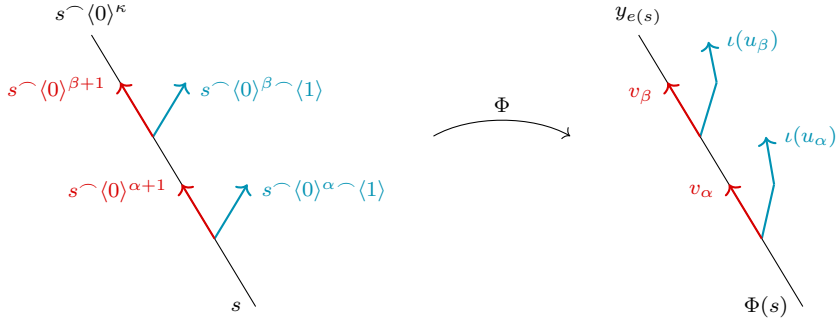
 
\centering
\tikz[scale=0.9,font=\scriptsize]{ 
\draw 
(\x-130pt,0); 

\draw 
(\x+2,2.5\y) edge[->,bend left, looseness=0.8] (\x+4,2.5\y); 

\draw 
    (\x+3,2.7\y) node[above] {$\Phi$}; 

\draw 
(0*\x,0) edge[-] (-4*\x,4*\y); 

\draw[thick,color={TealArrow}] 
(-1*\x,1*\y) edge[->] (-0.2*\x,1.8*\y) 
(-2.5*\x,2.5*\y) edge[->] (-1.7*\x,3.3*\y); 

\draw[thick,color={RedArrow}] 
(-1*\x,1*\y) edge[->] (-1.8*\x,1.8*\y) 
(-2.5*\x,2.5*\y) edge[->] (-3.3*\x,3.3*\y)
; 

\draw 
    (0*\x,0*\y) node[above] {} 
    (4*\x,4*\y) node[above] {}; 

\draw
    (-4*\x,4*\y) node[above] {$s\conc\langle 0\rangle^\kappa$};

\draw
    (-0.1*\x,0*\y) node[left] {$s$} 
    (-0.1*\x,1.7*\y) node[right] {\color{TealArrow}$s\conc\langle 0\rangle^{\alpha}\conc\langle 1\rangle$} 
    (-1.6*\x,3.2*\y) node[right] {\color{TealArrow}$s\conc\langle 0\rangle^{\beta}\conc\langle 1\rangle$}
    (-1.9*\x,1.7*\y) node[left] {\color{RedArrow}$s\conc\langle0\rangle^{\alpha+1}$} 
    (-3.4*\x,3.2*\y) node[left] {\color{RedArrow}$s\conc\langle0\rangle^{\beta+1}$};

\draw 
(0*\x+\z,0) edge[-] (-4*\x+\z,4*\y); 

\draw[thick,color={TealArrow}] 
(-1*\x+\z,1*\y) edge[-] (-0.7*\x+\z,1.8*\y) 
(-0.7*\x+\z,1.8*\y) edge[->] (-0.9*\x+\z,2.5*\y) 

(-2.5*\x+\z,2.5*\y) edge[-] (-2.1*\x+\z,3.3*\y) 
(-2.1*\x+\z,3.3*\y) edge[->] (-2.3*\x+\z,3.9*\y); 

\draw[thick,color={RedArrow}] 
(-1*\x+\z,1*\y) edge[->] (-1.8*\x+\z,1.8*\y) 
(-2.5*\x+\z,2.5*\y) edge[->] (-3.3*\x+\z,3.3*\y)
;

\draw 
    (0*\x+\z,0*\y) node[above] {} 
    (4*\x+\z,4*\y) node[above] {}; 

\draw
    (-4*\x+\z,4*\y) node[above] {$y_{e(s)}$}; 
  
\draw
    (-0.1*\x+\z,0*\y) node[left] {$\Phi(s)$} 
    (-0.7*\x+\z,2.5*\y) node[right] {\color{TealArrow} $\iota(u_\alpha)$} 
    (-2.1*\x+\z,3.9*\y) node[right] {\color{TealArrow} $\iota(u_\beta)$}
    (-1.9*\x+\z,1.7*\y) node[left] {\color{RedArrow} $v_\alpha$}
    (-3.4*\x+\z,3.1*\y) node[left] {\color{RedArrow} $v_\beta$};  
} 
\caption{The construction of $\Phi$.
This configuration appears for all $\alpha,\beta<\kappa$ with $\alpha<\beta$. 
} 
\label{figure: KLW proof}
\end{figure} 
}

Here is the detailed construction.
Let $\Phi(s):=\iota(e(s))$.
If $\alpha\in\Lim$, let
$\Phi(s\conc \langle0\rangle^{\alpha}):=
\bigcup_{\beta<\alpha}
\Phi(s\conc \langle0\rangle^{\beta})$. 
Now, assume that $\Phi(s\conc\langle 0\rangle^{\alpha})$ has been defined for some $\alpha<\kappa$.
Since 
$\Phi(s\conc \langle0\rangle^{\alpha})\subseteq y_{e(s)}$ and
$\langle \iota(e(s)\conc\langle\beta\rangle):\beta<\kappa\rangle$
converges to $y_{e(s)}$, there exists a node
$u_\alpha=e(s)\conc\langle\gamma\rangle$ with 
$\Phi(s\conc \langle0\rangle^{\alpha})\subsetneq \iota(u_\alpha)$.
Since $y_{e(s)}\perp \iota(u_\alpha)$, we have
$y_{e(s)}\wedge \iota(u_\alpha)\subsetneq \iota(u_\alpha)$. 
Let
$v_\alpha$ be the unique direct successor of
$y_{e(s)}\wedge \iota(u_\alpha)$ 
with $v_\alpha\subseteq y_{e(s)}$
(see Figure~\ref{figure: KLW proof new}).
Then $\iota(u_\alpha)$ and $v_\alpha$ are incompatible extensions of 
$\Phi(s\conc \langle0\rangle^{\alpha})$.
Let 
$e(s\conc\langle0\rangle^{\alpha}{}\conc\langle1\rangle):=u_\alpha$, 
$\Phi(s\conc\langle0\rangle^{\alpha}{}\conc\langle1\rangle):=\iota(u_\alpha)$
and 
$\Phi(s\conc \langle0\rangle^{\alpha+1}):=v_\alpha$.
This completes the construction of $\Phi$.

\vspace{5pt} 
{
\newcommand{\x}{0.6} 
\newcommand{\y}{1} 
\newcommand{\z}{8} 

\begin{figure}[H]
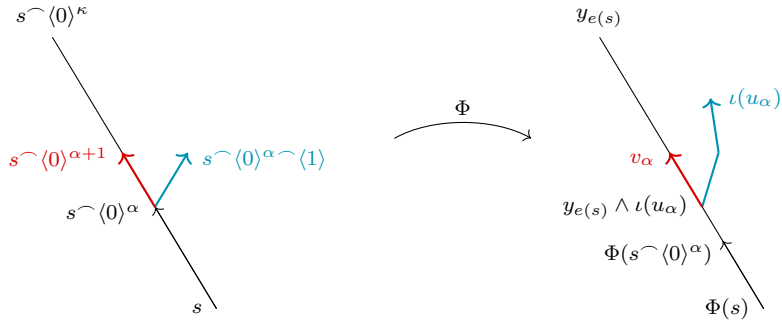
 
\centering
\tikz[scale=0.9,font=\scriptsize]{ 
\draw 
(\x-140pt,0); 

\draw 
(\x+2,2.5\y) edge[->,bend left, looseness=0.8] (\x+4,2.5\y); 

\draw 
    (\x+3,2.7\y) node[above] {$\Phi$}; 

\draw 
(0*\x,0) edge[-] (-4*\x,4*\y) 
(0*\x,0) edge[->] (-1.5*\x,1.5*\y); 

\draw[thick,color={TealArrow}]  
(-1.5*\x,1.5*\y) edge[->] (-0.7*\x,2.3*\y); 
\draw[thick,color={RedArrow}] 
(-1.5*\x,1.5*\y) edge[->] (-2.3*\x,2.3*\y); 

\draw 
    (0*\x,0*\y) node[above] {} 
    (4*\x,4*\y) node[above] {}; 
\draw
    (-4*\x,4*\y) node[above] {$s\conc\langle 0\rangle^\kappa$}; 

\draw
    (-0.1*\x,0*\y) node[left] {$s$} 
    (-0.6*\x,2.2*\y) node[right] {\color{TealArrow}$s\conc\langle 0\rangle^{\alpha}\conc\langle 1\rangle$} 
    (-1.6*\x,1.4*\y) node[left] {$s\conc\langle 0\rangle^{\alpha}$} 
    (-2.4*\x,2.2*\y) node[left] {\color{RedArrow}$s\conc\langle 0\rangle^{\alpha+1}$};

    \draw 
(0*\x+\z,0) edge[-] (-4*\x+\z,4*\y) 
(0*\x+\z,0) edge[->] (-1.0*\x+\z,1.0*\y); 

\draw[thick,color={TealArrow}]  
(-1.5*\x+\z,1.5*\y) edge[-] (-1.1*\x+\z,2.3*\y) 
(-1.1*\x+\z,2.3*\y) edge[->] (-1.3*\x+\z,3.1*\y); 

\draw[thick,color={RedArrow}] 
(-1.5*\x+\z,1.5*\y) edge[->] (-2.3*\x+\z,2.3*\y); 

\draw 
    (0*\x+\z,0*\y) node[above] {} 
    (4*\x+\z,4*\y) node[above] {}; 
\draw
    (-4*\x+\z,4*\y) node[above] {$y_{e(s)}$}; 
  
\draw
    (-0.1*\x+\z,0*\y) node[left] {$\Phi(s)$} 
    (-1.1*\x+\z,3.1*\y) node[right] {\color{TealArrow}$\iota(u_\alpha)$} 
    (-1.0*\x+\z,0.8*\y) node[left] {$\Phi(s\conc\langle 0\rangle^{\alpha})$} 
    (-1.6*\x+\z,1.5*\y) node[left] { $y_{e(s)} \wedge \iota(u_\alpha)$} 
    (-2.4*\x+\z,2.2*\y) node[left] {\color{RedArrow}$v_\alpha$}; 
    
   } 
\caption{Defining $\Phi$ for the successor step} 
\label{figure: KLW proof new}
\end{figure} 
}

We claim that $\Phi$ is a nice reduction for $(X,Y)$. 
%
To see that $[\Phi](\RR_\kappa)\subseteq X$, take any $x\in \RR_\kappa$. 
Then $x=\bigcup_{\alpha<\kappa} s_\alpha$ with $s_\alpha \in \SS_\kappa$ strictly increasing, so $[\Phi](x)=\bigcup_{\alpha<\kappa} \Phi(s_\alpha) = \bigcup_{\alpha<\kappa} (\iota\circ e) (s_\alpha) = [\iota] ( [e](x))\in X$ by~\ref{prop KLW 1}. 
To see that $[\Phi](\QQ^\ddim_\kappa)\subseteq Y$, let $x\in\QQ_\kappa$.
Then $x=s\conc\langle 0\rangle^\kappa$ for some $s\in\SS_\kappa$,
so $[\Phi](x)=y_{e(s)}\in Y$ 
by~\ref{prop KLW 2}.
\end{proof}

\begin{theorem}
\label{theorem: KLW from ODD}
Suppose that $X$ and $Y$ are subsets of ${}^\kappa\kappa$.
\begin{enumerate-(1)}
\item\label{general KLW from ODD coloring} 
The following statements are equivalent: 
\begin{enumerate-(a)}
\item\label{general KLW from ODD coloring 1} 
$\CC^Y\restr X$ admits a $\kappa$-coloring.
\item\label{general KLW from ODD coloring 2} 
$X$ is covered by sets $X_\alpha$ for $\alpha<\kappa$ with $X_\alpha'\cap Y=\emptyset$.
\item\label{general KLW from ODD coloring 3} 
$|X\cap Y|\leq\kappa$ and $X\setminus Y$ can be separated from $Y$ by a $\Fsigma(\kappa)$ set.
\end{enumerate-(a)}
\item\label{KLW from ODD hom}
The following statements are equivalent:
\begin{enumerate-(a)}
\item\label{KLW hom a} There exists a nice reduction for $(X,Y)$.
\item\label{KLW hom b} There exists a homeomorphism from ${}^\kappa 2$ onto a closed subset of ${}^\kappa\kappa$ which reduces
$(\RR_\kappa,\QQ_\kappa)$ to $(X,Y)$.
\item\label{KLW hom c} There exists a continuous reduction of $(\RR_\kappa,\QQ_\kappa)$ to $(X,Y)$
such that $f(\RR_\kappa) \cap f(\QQ_\kappa)=\emptyset$.
\item\label{KLW hom d} There exists a continuous homomorphism from $\dhH\kappa$ to $\CC^Y\restr X$.
\end{enumerate-(a)}
\end{enumerate-(1)}
Therefore $\KLW_\kappa(X,Y)$, $\sKLW_\kappa(X,Y)$ and $\ODD\kappa{\CC^Y\restr X}$ are all equivalent.
\end{theorem}
\begin{proof}
\ref{general KLW from ODD coloring} was proved in Lemma \ref{lemma: CC coloring}. 
In~\ref{KLW from ODD hom}, 
it is clear that 
\ref{KLW hom a}
$\Rightarrow$
\ref{KLW hom b}
$\Rightarrow$
\ref{KLW hom c}.
Moreover,
\ref{KLW hom c}
$\Rightarrow$
\ref{KLW hom d}
follows from Lemma~\ref{lemma: continuous reduction}.
To show
\ref{KLW hom d}
$\Rightarrow$
\ref{KLW hom a},
suppose that $f$ is a continuous homomorphism from $\dhH\kappa$ to $\CC^Y\restr X$ and hence to $\CC^Y$ as well. 
Then there exists an order homomorphism $\iota$ for $({}^\kappa\kappa,\CC^Y)$ 
with $\ran([\iota])\subseteq\ran(f)\subseteq X$ by
Lemma~\ref{homomorphisms and order preserving maps}.
Then \ref{KLW hom a} follows from Lemma~\ref{lemma: nice reduction}.
\end{proof}

\begin{remark}
\label{remark: CC and injectivity}
In the previous theorem,
$\CC^Y$ can be equivalently replaced by the $\kappa$-hypergraph
$\CC^Y_i:=\CC^Y\cap \dhIkappa$ consisting of all 
\index{hypergraph!of sequences!CCi@injective convergent with limit in $Y$ but not in the sequence \idf$\CC^Y_i$} 
\emph{injective} elements of $\CC^Y$. 
This follows from 
Corollary~\ref{cor: ODD subsequences} 
using the facts that
$\CC^Y_i$ is box-open on ${}^\kappa\kappa$,
$\CC^Y_i\subseteq \CC^Y$ and 
any hyperedge of $\CC^Y$ has a subsequence in $\CC^Y_i$ and hence 
$\CC^Y_i\equivf\CC^Y$.%
\footnote{See Definition~\ref{def: H-full}} 
\todog{If $X\subseteq Y$, then $\ODD\kappa{\CC_i^Y\restr X}$ is equivalent to $\ODDC\kappa{\CC^Y_i\restr X,Y}$ and hence to $\ODDC\kappa{\CC^Y\restr X,Y}$. This follows from Theorem~\ref{theorem: strong variants for dhI dhth} since $\CC^Y_i\subseteq\dhth Y$. But it also follows from Theorem~\ref{theorem: KLW from ODD} \ref{KLW from ODD hom}~\ref{KLW hom d}$\Rightarrow$\ref{KLW hom b} and Lemma \ref{lemma: continuous reduction}.} 
\end{remark}

The next result follows from Theorem~\ref{main theorem} and the previous theorem.
Given classes $\mathcal C$ and $\mathcal D$, let
$\KLW_\kappa(\mathcal C,\mathcal D)$ denote
\index{Kechris--Louveau--Woodin dichotomy!A@for arbitrary sets!B@for classes\idf$\KLW_\kappa(\mathcal C,\mathcal D)$} 
the statement that $\KLW_\kappa(X,Y)$ holds for all subsets $X\in\mathcal C$ and $Y\in\mathcal D$ of ${}^\kappa\kappa$. 
Moreover, let $\KLW_\kappa(\mathcal C)$ 
\index{Kechris--Louveau--Woodin dichotomy!A@for arbitrary sets!C@for a class\idf$\KLW_\kappa(\mathcal C)$} 
denote $\KLW_\kappa(\mathcal C,\powerset({}^\kappa\kappa))$.
\index{Kechris--Louveau--Woodin dichotomy!A@for arbitrary sets!C@for a set\idf$\KLW_\kappa(X)$} 
\index{Kechris--Louveau--Woodin dichotomy!A@for arbitrary sets!C@for a set and class\idf$\KLW_\kappa(X,\mathcal C)$} 
If $\mathcal C$ has one single element $X$, then 
we write $X$ instead of $\mathcal C$, and similarly for $\mathcal D$.

\begin{corollary}
The following hold in all $\Col(\kappa,\lle\lambda)$-generic extensions of $V$: 
\begin{enumerate-(1)}
\item
$\KLW_\kappa(\defsetsk,\defsetsk)$
if $\lambda>\kappa$ is inaccessible in $V$.
\item
$\KLW_\kappa(\defsetsk)$
if $\lambda>\kappa$ is Mahlo in~$V$.
\end{enumerate-(1)}
\end{corollary} 

In particular, this answers a question of Bergfalk, Chodounsk\'y, Guzm\'an and Hru\v s\'ak asking whether any proper $\Fsigma(\kappa)$ set is $\Fsigma(\kappa)$-complete. 

\begin{remark} 
\label{remark: Sigma02 complete}
It is consistent that that any proper $\Fsigma(\kappa)$ set is $\Fsigma(\kappa)$-complete.%
\footnote{A subset $X$ of ${}^\kappa\kappa$ is called \emph{$\Fsigma(\kappa)$-complete} if every other $\Fsigma(\kappa)$ subset of ${}^\kappa\kappa$ is the preimage of $X$ under a continuous function on ${}^\kappa\kappa$.} 
In fact, this statement is equivalent to 
\todol{This remark was strengthened slightly.}
the statement that
$\KLW_\kappa({}^\kappa\kappa\setminus X,X)$ holds for all $\Fsigma(\kappa)$ subsets $X$ of ${}^\kappa\kappa$.
To see this, 
note that a subset $X$ of ${}^\kappa\kappa$ is $\Gdelta(\kappa)$ if and only if
the first option in $\KLW_\kappa({}^\kappa\kappa\setminus X,X)$ holds,
by Theorem~\ref{theorem: KLW from ODD}~\ref{general KLW from ODD coloring}.
Hence
it suffices 
by Theorem~\ref{theorem: KLW from ODD}~\ref{KLW from ODD hom}~\ref{KLW hom b}$\Leftrightarrow$\ref{KLW hom c}
to show 
that a subset $X$ of ${}^\kappa\kappa$
is $\Fsigma(\kappa)$-complete if and only if
there exists a continuous function $f:{}^\kappa 2\to{}^\kappa\kappa$ with $f^{-1}(\QQ_\kappa)=X$.
But this follows from the fact that
$\QQ_\kappa$ is $\Fsigma(\kappa)$-complete \cite{beatrice}. 
\end{remark}

The next corollaries show that the $\kappa$-perfect set property is a special case of $\KLW_\kappa(X,Y)$. 

\begin{corollary} 
\label{cor: PSP from ODD for the kappa-Baire space}
Suppose that $X$ and $Y$ are subsets of ${}^\kappa\kappa$. 
\begin{enumerate-(1)} 
\item 
\label{cor: PSP from ODD for the kappa-Baire space a}
$\KLW_\kappa(X,Y)\Longleftrightarrow\PSP_\kappa(X\cap Y)$
if $X\setminus Y$ can be separated from $Y$ by a $\Fsigma(\kappa)$ set.%
\item 
\label{cor: PSP from ODD for the kappa-Baire space b}
$\KLW_\kappa(X \setminus Y,Y)\land\PSP_\kappa(X\cap Y)\Longrightarrow\KLW_\kappa(X,Y)$. 
\end{enumerate-(1)} 
\end{corollary}
\begin{proof}
For \ref{cor: PSP from ODD for the kappa-Baire space a}, it suffices 
to show that $\ODD\kappa{\CC^Y\restr X}\Longleftrightarrow\ODD\kappa{\gK{X\cap Y}}$ by Theorem~\ref{theorem: KLW from ODD} and Corollary~\ref{cor: PSP from ODD}.
By Corollary~\ref{cor: ODD subsequences}, it further suffices that $\CC^Y\restr X\equivf\gK{X\cap Y}$,
or equivalently, that $\CC^Y\restr A$ admits a $\kappa$-coloring if and only if $|A\cap Y|\leq\kappa$ for all subsets $A$ of $X$. But this follows from Lemma~\ref{lemma: CC coloring}. 
Moreover, \ref{cor: PSP from ODD for the kappa-Baire space b} follows from \ref{cor: PSP from ODD for the kappa-Baire space a}. 
\end{proof}

\begin{corollary}\ 
\label{cor: PSP from ODD for the kappa-Baire space 2}
\begin{enumerate-(1)}
\item\label{cor: PSP from ODD for the kappa-Baire space 2 a}
$\PSP_\kappa(X)\Longleftrightarrow\KLW_\kappa(X,{}^\kappa\kappa)$
for all subsets $X$ of ${}^\kappa\kappa$.
\item\label{cor: PSP from ODD for the kappa-Baire space 2 b}
$\PSP_\kappa(Y)\Longleftrightarrow\KLW_\kappa({}^\kappa\kappa,Y)$
for all $\Gdelta(\kappa)$ subsets $Y$ of ${}^\kappa\kappa$.
\end{enumerate-(1)}
\end{corollary}
\begin{proof}
\ref{cor: PSP from ODD for the kappa-Baire space 2 a}
and
\ref{cor: PSP from ODD for the kappa-Baire space 2 b}
are the special cases of the Corollary \ref{cor: PSP from ODD for the kappa-Baire space}
 for $Y:={}^\kappa\kappa$
and $X:={}^\kappa\kappa$, respectively.
\end{proof}

Since $\PSP_\kappa(\closedsets(\kappa))$ has the consistency strength of an inaccessible cardinal, 
it follows from 
\ref{cor: PSP from ODD for the kappa-Baire space 2 b}
and Theorem~\ref{theorem: KLW from ODD} 
that
$\ODD\kappa\kappa({}^\kappa\kappa,\defsetsk)$ has the same consistency strength.

\subsubsection{A version relative to dihypergraphs}
\label{subsection: KLW^H}
We assume throughout this subsection that $X,\,Y$ are subsets of ${}^\kappa\kappa$ and that $H$ is a $\ddim$-dihypergraph on ${}^\kappa\kappa$, where $2\leq\ddim\leq\kappa$, unless otherwise stated.
We apply $\ODD\kappa\kappa(X)$ to derive a version $\KLW^H_\kappa(X,Y)$ of the Kechris--Louveau--Woodin dichotomy relative to $H$, assuming $H$ is box-open on ${}^\kappa\kappa$.

\begin{definition}\ 
\label{def: H-limit point}
\begin{enumerate-(a)}
\item Let $H^y$ denote the set 
\index{dihypergraph!slice\idf$H^y$} 
of those sequences $\langle x_i: 1\leq i<\ddim\rangle$ such that 
$\langle y\rangle\conc\langle x_i:1\leq i<\ddim\rangle\in H$.
\item 
$y\in{}^\kappa\kappa$ is an \emph{$H$-limit point} of $X$ 
\index{limit pointsz@$H$-limit points\idf$X^H$}
if 
$H^y\restr(X\cap N_{y\restr\alpha})\neq\emptyset$
for all $\alpha<\kappa$. 
$X^H$ denotes the set of $H$-limit points of $X$.
\item 
$X$ is called \emph{$H$-closed} if 
\index{closed set@$H$-closed set, $H$-closure\idf}
$X^H\subseteq X$.
\end{enumerate-(a)}
\end{definition}
In particular, any $H$-limit point of $X$ is also a limit point of $X$. 
The converse holds if $H$ is the complete $\ddim$-hypergraph $\dhK\ddim{{}^\kappa\kappa}$ and, 
more generally, if $\dhI\ddim\subseteq H$. 

\begin{lemma} 
\label{H-closure} 
Suppose that $H$ is box-open on ${}^\kappa\kappa$. 
\begin{enumerate-(1)} 
\item 
\label{H-closure 1} 
Any $H$-limit point of $\closure{X}$ is an $H$-limit point of $X$. 
\item 
\label{H-closure 2} 
$X\cup X^H$ is $H$-closed.
\end{enumerate-(1)} 
\end{lemma} 
\begin{proof} 
For \ref{H-closure 1}, suppose that $y$ is an $H$-limit point of $\overline{X}$ and $\alpha<\kappa$. 
Pick a sequence  $\langle x_i: 1\leq i<\ddim\rangle\in H^y$ with $x_i\in \overline{X}\cap N_{y\restr\alpha}$ for all $1\leq i<\ddim$.  
Since $H$ is box-open, there exist $\alpha_i\geq\alpha$ for $1\leq i<\ddim$ such that $\prod_{1\leq i<\ddim} N_{x_i\restr \alpha_i} \subseteq H^y\restr N_{y\restr\alpha}$. 
It follows that $y$ is an $H$-limit point of $X$. 
\ref{H-closure 2} follows from \ref{H-closure 1} since $X\cup X^H\subseteq\closure X$. 
\end{proof} 

$X\cup X^H$ is called the \emph{$H$-closure} of $X$. 
\index{closed set@$H$-closed set, $H$-closure\idf}
This is the smallest $H$-closed set containing $X$ if $H$ is box-open on ${}^\kappa\kappa$.
Note that $\closure X$ is $H$-closed as well.

We now define variants of $\KLW_\kappa(X,Y)$ relative to box-open dihypergraphs. 
Let $\QQ^\ddim_\kappa$ 
\index{rationals@$\kappa$-rationals in!d@${}^\kappa\ddim$\idf$\QQ^\ddim_\kappa$}
denote the set of those $x\in{}^\kappa\ddim$ such that $x=t\conc\langle 0\rangle^\kappa$ for some $t\in{}^{<\kappa}\ddim$.
\index{irrationals@$\kappa$-irrationals in!d@${}^\kappa\ddim$\idf$\RR^\ddim_\kappa$}
Let $\RR^\ddim_\kappa:={}^\kappa\ddim\,\setminus\,\QQ^\ddim_\kappa$.

\begin{definition}
\label{def: KLW for hypergraphs}
Let $\KLW_\kappa^H(X,Y)$ 
and ${\sKLW_\kappa}^H(X,Y)$ 
denote the following statements:
\begin{quotation}
\index{Kechris--Louveau--Woodin dichotomy!C@for $\ddim$-dihypergraphs!A@standard\idf$\KLW_\kappa^H(X,Y)$} 
$\KLW_\kappa^H(X,Y)$: Either 
$X=\bigcup_{\alpha<\kappa} X_\alpha$
for some sequence $\langle X_\alpha:\alpha<\kappa\rangle$ of sets 
$X_\alpha$ with $X_\alpha^H\cap Y=\emptyset$, 
or there exists a continuous
homomorphism $f$ from $\dhH\ddim$ to $H$
which reduces $(\RR^\ddim_\kappa,\QQ^\ddim_\kappa)$ to $(X,Y)$.%
\footnote{I.e., $f(\RR^\ddim_\kappa)\subseteq X$ and $f(\QQ^\ddim_\kappa)\subseteq Y$.}
\end{quotation}
\vspace{0 pt}
\begin{quotation}
\index{Kechris--Louveau--Woodin dichotomy!C@for $\ddim$-dihypergraphs!A@variant\idf${\sKLW_\kappa}^H(X,Y)$} 
${\sKLW_\kappa}^H(X,Y)$: Either 
$X=\bigcup_{\alpha<\kappa} X_\alpha$
for some sequence $\langle X_\alpha:\alpha<\kappa\rangle$ of sets 
$X_\alpha$ with $X_\alpha^H\cap Y=\emptyset$, 
or there exists an order homomorphism $\Phi$ for $({}^\kappa\kappa,H)$
such that $[\Phi]$ reduces $(\RR^\ddim_\kappa,\QQ^\ddim_\kappa)$ to $(X,Y)$.%
\end{quotation}
\end{definition}

We will show that these two dichotomies are in fact equivalent 
if $H$ is box-open on ${}^\kappa\kappa$
in Theorem~\ref{theorem: ODD and KLW^H}.
Note that if $X$ and $Y$ are disjoint, then the first option of each dichotomy
holds if and only if
$X$ can be separated from $Y$ by a $\kappa$-union of $H$-closed sets. 

The above dichotomies generalize the Kechris--Louveau--Woodin dichotomy. 
\todog{It is easy to show that the converse directions of both \ref{remark: KLW^G for graphs G 2} and \ref{remark: KLW^G for graphs G 3} also hold for $G:=\KK$.}%
\todog{It is also easy to see that for $H:=\dhI\ddim$ where $2\leq\ddim\leq\kappa$,
${\sKLW_\kappa}^H(X,Y)$ is equivalent to the following \emph{$\ddim$-dimensional version of $\sKLW_\kappa(X,Y)$}: either $X=\bigcup_{\alpha<\kappa} X_\alpha$ where $X_\alpha'\cap Y=\emptyset$ or there is a $\perp$- and strict order preserving map $\Phi$ such that $[\Phi]$ reduces $(\RR^\ddim_\kappa,\QQ^\ddim_\kappa)$ to $(X,Y)$.
\\[1 pt]
Note that $[\Phi]$ is a homeomorphism from ${}^\ddim\kappa$ onto its image. If $\ddim<\kappa$, then its image is closed. This could be used to define a $\ddim$-dimensional version of $\KLW_\kappa(X,Y)$ as well.}%
\index{graph!complete graph on!${}^{\kappa}\kappa$\idf$\KK$}%
For the complete graph~$\KK:=\gK{{}^\kappa\kappa}=\dhI 2$, 
$\KLW_\kappa^{\KK}(X,Y)$ is equivalent to
$\KLW_\kappa(X,Y)$ 
by the next remark and 
Theorem \ref{theorem: KLW from ODD}~\ref{KLW from ODD hom}~\ref{KLW hom b}$~\Leftrightarrow$~\ref{KLW hom c}.
The next remark also gives a direct proof of the equivalence of
${\sKLW_\kappa}^{\KK}(X,Y)$ and
$\sKLW_\kappa(X,Y)$.
For higher dimensions $2<\ddim\leq\kappa$,
$\KLW_\kappa(X,Y)$ is equivalent to $\KLW^I_\kappa(X,Y)$ for 
$I:=\proj_{\ddim, 2}^{-1}(\KK){\restr}{}^\kappa\kappa$%
\footnote{Recall that $\proj_{\ddim,c}$ denotes projection onto coordinates in $c$. See the paragraph before Lemma~\ref{sublemma: ODD for different dimensions}.}
by Remark~\ref{KLW^H for different dimensions} below.

\begin{remark}\ 
\label{remark: KLW^G for graphs G}
For digraphs $G$, the two options in $\KLW_\kappa^G(X,Y)$ and ${\sKLW_\kappa}^G(X,Y)$
can be characterized as follows. 
\begin{enumerate-(1)}
\item 
\label{remark: KLW^G for graphs G 1}
Regarding the first option, note that
$y$ is a $G$-limit point of a set $A$ if and only if there exists a sequence $\langle x_\alpha : \alpha<\kappa\rangle$ in $A$ converging to $y$ such that $(y,x_\alpha)\in G$ for all $\alpha<\kappa$. 
\item
\label{remark: KLW^G for graphs G 2}
For the second option 
in $\KLW^G_\kappa(X,Y)$,
observe that
any homomorphism $f$ from $\dhH 2$ to $G$ is injective
since $G\subseteq\KK$.
In particular, if $f$ is as 
in $\KLW^G_\kappa(X,Y)$,
then
$f(\RR_\kappa)\cap f(\QQ_\kappa)=\emptyset$.
\item
\label{remark: KLW^G for graphs G 3}
The second option 
in ${\sKLW_\kappa}^G(X,Y)$ is equivalent to the existence of an order homomorphism for $({}^\kappa\kappa,G)$ which is a nice reduction for $(X,Y)$. 
To see this, note that any order homomorphism for $({}^\kappa\kappa,G)$ is $\perp$-preserving
by Lemma~\ref{order homomorphisms and injective H}~\ref{dhI perp},
since $G\subseteq\KK$.
In particular, if $G$ is a graph, then $C:=\ran([\Phi])$ is
a $\kappa$-perfect $G$-clique%
\footnote{See Definition~\ref{def: clique}.}
such that both $X$ and $Y$ are dense in $C$.
\end{enumerate-(1)}
\end{remark}

We will show that 
both $\KLW^H_\kappa(X,Y)$ and ${\sKLW_\kappa}^H(X,Y)$
are equivalent to the same special case of $\ODD\kappa\kappa(X)$ assuming $H$ is box-open on ${}^\kappa\kappa$ 
in Theorem~\ref{theorem: ODD and KLW^H}.
We first define the relevant dihypergraphs.
Suppose that $\langle X_\alpha : \alpha<\kappa \rangle$ is a sequence of subsets of ${}^\kappa\kappa$. 
We say that $\langle X_\alpha : \alpha<\kappa \rangle$ \emph{converges}
\index{convergent sequence of!subsets of ${}^{\kappa}\kappa$\idf}
to $y\in {}^\kappa\kappa$ if for every $\gamma<\kappa$, there exists some $\alpha<\kappa$ with $\bigcup_{\alpha<\beta<\kappa} X_\beta \subseteq N_{y\restr \gamma}$.%
\footnote{For example, if ${t_\alpha}\in{}^{<\kappa}\kappa$ for each $\alpha<\kappa$, then 
$\langle N_{t_\alpha}: \alpha<\kappa \rangle$
converges to $y$ if and only if 
$\langle t_\alpha:\alpha<\kappa\rangle$ converges to $y$ 
(in the sense defined on page~\pageref{def: convergence of nodes}).}
Let 
$$D:=\kappa\times(\ddim\,\setminus\{0\}),$$%
\label{def of CC^XH}%
and let $\CC^{Y,H}$ denote the $D$-dihypergraph 
\index{dihypergraph!z@\mygobble|see{hypergraph}}%
\index{dihypergraph!with convergent slices\idf$\CC^{Y,H}$}%
on ${}^\kappa\kappa$ which consists
of all sequences 
$\langle x_{(\alpha,i)} : (\alpha,i)\in D\rangle$ 
with $x_{(\alpha,i)} \in {}^\kappa\kappa$ 
such that the sets 
$X_\alpha:= \{x_{(\alpha,i)} : 1\leq i<{\ddim}\}$ 
converge to an element $y$ of $Y$ with
$\bar{x}_\alpha:=\langle x_{(\alpha,i)} : 1\leq i<{\ddim} \rangle \in H^y$ 
for all $\alpha<\kappa$ (see Figure~\ref{figure: CC^XH}).
Note that $\CC^{Y,H}\restr X$ consists of those $D$-sequences which witness that 
$X^H\cap Y\neq\emptyset$.
%
%
Note that $\CC^{Y,H}$ can be identified  with a $\kappa$-dihypergraph  $\bar\CC^{Y,H}$ on  ${}^\kappa\kappa$ via a bijection $D\to\kappa$, for instance with $\CC^Y$ if $H=\gK{{}^\kappa\kappa}$. 
If $H$ is box-open on ${}^\kappa\kappa$,
then $\CC^{Y,H}$ and $\bar\CC^{Y,H}$ are 
box-open on ${}^\kappa\kappa$.%
\footnote{In fact, it suffices that $H^y$ is box-open on ${}^\kappa\kappa$ for each $y\in Y$.}
Moreover, both are in $\defsetsk$ whenever $Y,H\in\defsetsk$.

\vspace{10pt} 
{%
\newcommand{\x}{0.6}
\newcommand{\y}{1}
\begin{figure}[H] 
\centering
\tikz[scale=1,font=\scriptsize]{ 
\draw 
(0*\x,0) edge[-] (-4*\x,4*\y); 

\draw[-, thick, dashed, color={DarkGreenArrow}] 
(-2.25*\x,2.25*\y) edge[-] (-2.0*\x,4*\y) 
(-2.25*\x,2.25*\y) edge[-] (0.1*\x,4*\y); 

\draw[thick,color={TealArrow}] 
(-2.25*\x,2.25*\y) edge[-] (-1.7*\x,3.12*\y) 
(-1.7*\x,3.12*\y) edge[-] (-1.3*\x,4*\y); 

\draw
    (-1.8*\x,4.2*\y) node[right] {\tiny \color{DarkGreenArrow}$\bar{x}_\beta \in H^y$} 
    (-1.4*\x,3.8*\y) node[right] {\tiny \color{TealArrow}$x_{(\beta,i)}$};

\draw[-, thick, dashed, color={DarkGreenArrow}] 
(-0.75*\x,0.75*\y) edge[-] (0.5*\x,4*\y) 
(-0.75*\x,0.75*\y) edge[-] (2.8*\x,4*\y); 

\draw[thick,color={TealArrow}] 
(-0.75*\x,0.75*\y) edge[-] 
(0.3*\x,2.25*\y) 
(0.3*\x,2.25*\y) edge[-] (1.2*\x,4.0*\y); 

\draw
    (0.65*\x,4.2*\y) node[right] {\tiny \color{DarkGreenArrow}$\bar{x}_\alpha \in H^y$} 
    (1.1*\x,3.8*\y) node[right] {\tiny \color{TealArrow}$x_{(\alpha,i)}$};

\draw
    (-4*\x,4*\y) node[above] {$y$}; 

    } 
\caption{$\CC^{Y,H}$} 
\label{figure: CC^XH}
\end{figure}
}

The next two lemmas characterize the existence of $\kappa$-colorings of $\CC^{Y,H}\restr X$.

\begin{lemma}
\label{lemma: coloring C^XH}
A subset $A$ of ${}^\kappa\kappa$ is $\CC^{Y,H}$-independent if and only if 
$A^H\cap Y=\emptyset$.
\end{lemma}
\begin{proof}
Suppose $y\in A^H\cap Y$.
For each $\alpha<\kappa$,
choose an arbitrary sequence
$\langle a_{(\alpha,i)}:1\leq i<{\ddim}\rangle$ in $H^y\restr (A\cap N_{y\restr\alpha})$.
Then 
$\langle a_{(\alpha,i)}: (\alpha,i)\in D\rangle$ is a hyperedge of 
$\CC^{Y,H}\restr A$.
Conversely, suppose
$\langle a_{(\alpha,i)}: (\alpha,i)\in D\rangle\in\CC^{Y,H}\restr A$,
and let
$\bar a_\alpha:=\langle a_{(\alpha,i)}: 1\leq i<\ddim\rangle$ for all $\alpha<\kappa$.
Then the sets
$\ran(\bar a_\alpha)$ 
converge to some $y\in Y$.
Since
$\bar a_\alpha\in H^y$ for all $\alpha<\kappa$,
we have $y\in A^H$.
\end{proof}

\begin{lemma}
\label{coloring H vs C^XH} 
Suppose that $H$ is box-open on ${}^\kappa\kappa$. 
The following statements are equivalent.
\begin{enumerate-(1)} 
\item
\label{coloring H vs C^XH 1} 
$\CC^{Y,H}\restr X$ admits a $\kappa$-coloring.
\item
$X=\bigcup_{\alpha<\kappa} X_\alpha$
for some sequence $\langle X_\alpha:\alpha<\kappa\rangle$ of sets 
$X_\alpha$ with $X_\alpha^H\cap Y=\emptyset$.
\label{coloring H vs C^XH 2} 
\item
\label{coloring H vs C^XH 3} 
$H\restr (X\cap Y)$ admits a $\kappa$-coloring and
$X\setminus Y$ can be separated from $Y$ by a $\kappa$-union of $H$-closed sets. 
\end{enumerate-(1)} 
\end{lemma} 
\begin{proof}
\ref{coloring H vs C^XH 1} $\Rightarrow$ \ref{coloring H vs C^XH 2}:
Suppose that $\CC^Y\restr X$ admits a $\kappa$-coloring 
and take $\CC^{Y,H}\restr X$-independent sets $X_\alpha$ with
$X=\bigcup_{\alpha<\kappa} X_\alpha$.
Then $X_\alpha^H\cap Y=\emptyset$
for all $\alpha<\kappa$ by Lemma~\ref{lemma: coloring C^XH}.

\ref{coloring H vs C^XH 2} $\Rightarrow$ \ref{coloring H vs C^XH 3}:
Let $Y_\alpha:=Y\cap X_\alpha$ for each $\alpha<\kappa$. 
To shows that $H\restr(X\cap Y)$ admits a $\kappa$-coloring, it suffices to show that $H\restr Y_\alpha$ admits a $\kappa$-coloring for each $\alpha<\kappa$. 
Since $Y_\alpha^H\cap Y=\emptyset$, there exists for all $y\in Y$ some $c(y)\subsetneq y$ such that $H^y\restr (Y_\alpha\cap N_{c(y)})= \emptyset$.
We claim that for each $t\in {}^{<\kappa}\kappa$, the set
$B_t:=\{y\in Y_\alpha : c(y)=t\}$
is $H$-independent. 
\todog{We use $Y_\alpha\subseteq Y$ here, since we need $y\in Y$}
Otherwise, pick a sequence $\langle y_i : i<\ddim\rangle \in H\restr  B_t$. 
Then  
$\langle y_i : 1\leq i<\ddim\rangle \in H^{y_0}\restr B_t$.
Since $B_t\subseteq Y_\alpha\cap N_t$ and
$c(y_0)=t$, this contradicts the definition of $c$. 
To show that $X\setminus Y$ can be separated from $Y$ by a $\kappa$-union of $H$-closed sets, 
let $Z_\alpha$ denote the $H$-closure of $X_\alpha\setminus Y$.%
\footnote{I.e., $Z_\alpha:=(X_\alpha\setminus Y)\cup(X_\alpha\setminus Y)^H$.}
This is $H$-closed by Lemma \ref{H-closure} \ref{H-closure 2}. 
It is disjoint from $Y$ since $(X_\alpha\setminus Y)^H\cap Y=\emptyset$.
\todog{$\closure{X_\alpha\setminus Y}$ is also $H$-closed, but it may not be disjoint from $Y$}

\ref{coloring H vs C^XH 3} $\Rightarrow$ \ref{coloring H vs C^XH 1}:
It suffices to show that $\CC^{Y,H}\restr(X\setminus Y)$ admits a $\kappa$-coloring. 
To see this, take $H$-closed subsets $Z_\alpha$ of ${}^\kappa\kappa$ for $\alpha<\kappa$ such that
their union separates
$X\setminus Y$ from $Y$. 
Note that each $Z_\alpha$ is $\CC^{Y,H}$-independent by Lemma \ref{lemma: coloring C^XH}. 
\end{proof}

The next lemma will be used to characterize the existence of continuous homomorphisms from $\dhH D$ to $\CC^{Y,H}\restr X$.

\begin{lemma}
\label{lemma: oh for C^XH}
Let 
$\iota:{}^{<\kappa}D\to{}^{<\kappa}\kappa$ 
be a strict order preserving map and define
$U^t_\alpha:=\bigcup_{1\leq i<\ddim} N_{\iota(t\conc\langle(\alpha,i)\rangle)}$
for all $t\in{}^{<\kappa}D$ and $\alpha<\kappa$.
The following statements are equivalent: 
\begin{enumerate-(1)} 
\item 
\label{lemma: oh for C^XH direction one}
 $\iota$ is an order homomorphism for $({}^\kappa\kappa,\CC^{Y,H})$. 
 \item 
 \label{lemma: oh for C^XH direction two}
 For each $t\in{}^{<\kappa}D$, the sequence
$\langle U^t_\alpha:\alpha<\kappa\rangle$ converges to some $y_t\in Y$
such that 
$\prod_{1\leq i<\ddim} N_{\iota(t\conc\langle(\alpha,i)\rangle)}
\subseteq H^{y_t}$
for all $\alpha<\kappa$.
\end{enumerate-(1)} 
\end{lemma}
\begin{proof}
The proof is similar to that of Lemma~\ref{lemma: CC convergence}. 

\ref{lemma: oh for C^XH direction two} $\Rightarrow$ \ref{lemma: oh for C^XH direction one}: Let
$t\in{}^{<\kappa}D$, and
suppose that 
$\langle U^t_\alpha:\alpha<\kappa\rangle$ converges to $y_t\in Y$.
Take an arbitrary sequence
$\bar{a}=\langle a_{(\alpha,i)} : (\alpha,i)\in D\rangle$ 
with
$a_{(\alpha,i)}\in N_{\iota(t\conc\langle (\alpha,i)\rangle)}$ 
for all $(\alpha,i)\in D$.
Then $A_\alpha:=\{a_{(a,i)}:1\leq i<\ddim\}$ is a subset of $U^t_\alpha$ for all $\alpha<\kappa$, 
so the sets $A_\alpha$
converge to $y_t$.
Since $\prod_{i<\kappa} N_{\iota(t\conc\langle(\alpha,i)\rangle)}\subseteq H^{y_t}$
for each $\alpha<\kappa$,
we have $\langle a_{(\alpha,i)} :  1\leq i<\ddim \rangle\in H^{y_t}$ for each $\alpha<\kappa$,
so $\bar a\in \CC^{Y,H}$.
This shows that $\iota$ is an order homomorphism for $({}^\kappa\kappa,\CC^{Y,H})$.

\ref{lemma: oh for C^XH direction one} $\Rightarrow$ \ref{lemma: oh for C^XH direction two}: Fix $t\in{}^{<\kappa}D$ 
and 
$\bar{a}=\langle a_{(\alpha,i)} : (\alpha,i)\in D\rangle$ 
with
$a_{(\alpha,i)}\in N_{\iota(t\conc\langle (\alpha,i)\rangle)}$ 
for all $(\alpha,i)\in D$. 
Since $\bar a\in \CC^{Y,H}$, the sets
$A_\alpha:=\{a_{(a,i)}:1\leq i<\ddim\}$ converge to some $y_t\in Y$.
We first show that $\langle U^t_\alpha:\alpha<\kappa\rangle$ converges to $y_t$. 
Otherwise, there exists some $\gamma<\kappa$
and a sequence $\langle x_\alpha : \alpha<\kappa\rangle$ with
$x_\alpha\in U^t_\alpha$ 
and $x_\alpha\notin N_{y_t\restr\gamma}$ for unboundedly many $\alpha<\kappa$.
Take any sequence
$\bar{b}=\langle b_{(\alpha,i)} : (\alpha,i)\in D\rangle$ 
with $b_{(\alpha,i)}\in N_{\iota(t\conc\langle (\alpha,i)\rangle)}$
and $x_\alpha \in B_\alpha:= \{b_{(\alpha,i)} : 1\leq i<\ddim\}$. 
Since $\bar b\in \CC^{Y,H}$, the sets
$B_\alpha$
converge to some $z_t\in Y$.
Now $y_t\neq z_t$
since $B_\alpha\not\subseteq N_{y_t\restr\gamma}$ for unboundedly many $\alpha<\kappa$.
For all $(\alpha,i)\in D$, let 
\[
c_{(\alpha,i)}:=
\begin{cases}
a_{(\alpha,i)} &\text{ if $\alpha$ is even,}
\\
b_{(\alpha,i)} &\text{ if $\alpha$ is odd.}
\end{cases}
\]
Since $\iota$ is an order homomorphism for $({}^\kappa\kappa,\CC^{Y,H})$, the sets 
$C_\alpha:= \{ c_{(\alpha,i)} : 1\leq i<\ddim\}$ converge. 
But they cannot converge to both $y_t$ and $z_t$ at the same time. 
It now suffices to show for any $\alpha<\kappa$ that any 
$\bar x=\langle x_{(\alpha,i)}: 1\leq i<\ddim\rangle$ in 
$\prod_{1\leq i<\ddim} N_{\iota(t\conc\langle(\alpha,i)\rangle)}$
is also in $H^{y_t}$.
This follows from the fact that we may extend $\bar x$ to a hyperedge of $\CC^{Y,H}$ by choosing
$x_{(\beta, i)}\in N_{\iota(t\conc\langle(\beta,i)\rangle)}$ arbitrarily for all $\beta\neq\alpha$ and $1\leq i<\ddim$. 
This fact holds since $\iota$ is an order homomorphism for $({}^\kappa\kappa,\CC^{Y,H})$. 
\end{proof}

\begin{remark} 
For any order homomorphism 
$\iota:{}^{<\kappa}D\to{}^{<\kappa}\kappa$ for $({}^\kappa\kappa,\CC^{Y,H})$, we have $\ran(\iota)\subseteq T(Y)$ and $\ran([\iota])\subseteq \closure Y$. 
To see this, note that 
the sets $U^t_\alpha$ converge to some $y_t\in Y$ by the previous lemma. 
Since $U^t_\alpha \subseteq N_{\iota(t)}$ for each $t\in{}^{<\kappa}D$,
we have $y_t\in N_{\iota(t)}$, so $\iota(t)\in T(Y)$. 
\end{remark} 

\begin{theorem}
\label{theorem: ODD and KLW^H}
Suppose $H$ is box-open on ${}^\kappa\kappa$. 
\begin{enumerate-(1)}
\item
\label{theorem: ODD and KLW^H 1}
The following statements are equivalent.
\begin{enumerate-(a)} 
\item
$\CC^{Y,H}\restr X$ admits a $\kappa$-coloring.
\item
$X=\bigcup_{\alpha<\kappa} X_\alpha$
for some sequence $\langle X_\alpha:\alpha<\kappa\rangle$ of sets 
$X_\alpha$ with $X_\alpha^H\cap Y=\emptyset$.
\item
$H\restr (X\cap Y)$ admits a $\kappa$-coloring and
$X\setminus Y$ can be separated from $Y$ by a $\kappa$-union of $H$-closed sets. 
\end{enumerate-(a)} 
\item
\label{theorem: ODD and KLW^H 2}
The following statements are equivalent.
\begin{enumerate-(a)}
\item
\label{theorem: ODD and KLW^H: oh H}
There exists an order homomorphism $\Phi$ for $({}^\kappa\kappa,H)$ 
such that $[\Phi]$ reduces $(\RR^\ddim_\kappa,\QQ^\ddim_\kappa)$ to $(X,Y)$.
\item 
\label{theorem: ODD and KLW^H: ch H}
There exists a continuous homomorphism from $\dhH\ddim$ to $H$ which reduces $(\RR^\ddim_\kappa,\QQ^\ddim_\kappa)$ to $(X,Y)$.
\item
\label{theorem: ODD and KLW^H: ch CC}
There exists a continuous homomorphism from $\dhH D$ to $\CC^{Y,H}\restr X$.
\end{enumerate-(a)}
\end{enumerate-(1)}
Therefore $\KLW^H_\kappa(X,Y)$,
${\sKLW_\kappa}^H(X,Y)$
and $\ODD D{\CC^{Y,H}\restr X}$ are all equivalent. 
\end{theorem}
\begin{proof}
\ref{theorem: ODD and KLW^H 1} was proved in Lemma~\ref{coloring H vs C^XH}.
For~\ref{theorem: ODD and KLW^H 2},
define the map $\pi_{\ddim}:{}^{<\kappa}D\to{}^{<\kappa}\ddim$ 
\index{embedding!of zkj@of ${}^{<\kappa}(\kappa\times\ddim\setminus\{0\})$ into ${}^{<\kappa}\ddim$\idf$\pi_\ddim$}%
\index{embedding!of zkk@of ${}^{\kappa}(\kappa\times\ddim\setminus\{0\})$ into ${}^{\kappa}\ddim$\idf$[\pi_\ddim]$}%
by letting 
$$\pi_{\ddim}(t)=
\bigoplus_{\alpha<\lh(t)}\big(\langle0\rangle^{{t(\alpha)}_0}
\conc\langle {t(\alpha)}_1\rangle\big),
$$
where 
${t(\alpha)}_0$ and ${t(\alpha)}_1$ denote the unique elements of $\kappa$ and $\ddim\,\setminus\{0\}$, respectively, 
with $t(\alpha)=\big({t(\alpha)}_0,{t(\alpha)}_1\big)$.
Note that $\RR^\ddim_\kappa=\ran([\pi_{\ddim}])$ and 
\index{sequences!with cofinally many nonzero values in!d@${}^{<\kappa}\ddim$\idf$\SS^\ddim_\kappa$}%
$$\SS^\ddim_\kappa := \ran(\pi_{\ddim})$$ 
consists of those nodes $u\in{}^{<\kappa}\ddim$ such that $u(\alpha)\neq 0$ for cofinally many $\alpha<\lh(u)$. 

\ref{theorem: ODD and KLW^H: oh H} $\Rightarrow$ \ref{theorem: ODD and KLW^H: ch H} is clear.
For
\ref{theorem: ODD and KLW^H: ch H} $\Rightarrow$ \ref{theorem: ODD and KLW^H: ch CC},
let $f$ be a continuous homomorphism from $\dhH\ddim$ to $H$ which reduces $(\RR^\ddim_\kappa,\QQ^\ddim_\kappa)$ to $(X,Y)$.
We show that $f\comp[\pi_\ddim]$ is a  continuous homomorphism from $\dhH D$ to $\CC^{Y,H}\restr X$.
To see this, 
suppose that 
$\bar a:=\langle a_{(\alpha,i)}:(\alpha,i)\in D\rangle$
is a hyperedge of $\dhH D$.
Take $t\in{}^{<\kappa}D$ 
with $t\conc\langle(\alpha,i)\rangle\subsetneq a_{(\alpha,i)}$ for each $(\alpha,i)\in D$.
Since
$[\pi_\ddim](a_{(\alpha,i)})$
extends
$\pi_\ddim(t\conc \langle(\alpha,i)\rangle)=
\pi_\ddim(t)\conc\langle 0\rangle^{\alpha}\conc\langle i\rangle$
for each $(\alpha,i)\in D$,
the sets 
$$X_\alpha:= \big\{[\pi_\ddim](a_{(\alpha,i)}): 1\leq i<\ddim\big\}$$
converge to 
$y_t:=\pi_{\ddim}(t)\conc\langle 0\rangle^\kappa$.
Since $f$ is continuous, the sets
$f(X_\alpha)$ converge to $f(y_t)$.
Since $f$ is a reduction of $(\RR^\ddim_\kappa,\QQ^\ddim_\kappa)$ to $(X,Y)$,
$f(X_\alpha)\subseteq X$ for each $\alpha<\kappa$ and $f(y_t)\in Y$.
Since $\prod_{i<\ddim} N_{\pi_\ddim(t)\conc\langle 0\rangle^{\alpha}\conc\langle i\rangle} \subseteq \dhH\ddim$ for each $\alpha<\kappa$ and $f$ is a homomorphism from $\dhH\ddim$ to~$H$, $(f\comp[\pi_\ddim])(\bar a)$ is a hyperedge of $\CC^{Y,H}\restr X$, as required.

To show~\ref{theorem: ODD and KLW^H: ch CC} $\Rightarrow$ \ref{theorem: ODD and KLW^H: oh H},
suppose $f$ is a continuous homomorphism from $\dhH D$ to $\CC^{Y,H}\restr X$
and therefore to $\CC^{Y,H}$ as well.
Then there exists an
order homomorphism 
$\iota: {}^{<\kappa} D\to{}^{<\kappa}\kappa$ 
for $({}^\kappa\kappa,\CC^{Y,H})$
with $\ran([\iota])\subseteq \ran(f)\subseteq X$
by Lemma~\ref{homomorphisms and order preserving maps}.
For each $t\in{}^{<\kappa} D$,
the sets $U^t_\alpha:=\bigcup_{1\leq i<\ddim} N_{\iota(t\conc\langle(\alpha,i)\rangle)}$
converge to some $y_t\in Y$ 
by Lemma~\ref{lemma: oh for C^XH}.
We use a variant of the construction in the proof of Lemma~\ref{lemma: nice reduction}.
We shall define continuous strict order preserving functions 
$\Phi: {}^{<\kappa}\ddim\to {}^{<\kappa}\kappa$ 
and 
$e: \SS^\ddim_\kappa\to {}^{<\kappa}D$ with the following properties: 
\begin{enumerate-(i)} 
\item\label{prop ODD for C^HX 1}
$\Phi{\restr}\SS^\ddim_\kappa = \iota \circ e$. 
\item\label{prop ODD for C^HX 2}
$\Phi(s\conc\langle0\rangle^\alpha)\subseteq y_{e(s)}$ for all $s\in \SS^\ddim_\kappa$ and $\alpha<\kappa$. 
\item\label{prop ODD for C^HX 3}
$\prod_{\alpha<\kappa}N_{\Phi(t\conc\langle\alpha\rangle)}\subseteq H$ for all $t\in{}^{<\kappa}\kappa$.
\end{enumerate-(i)}

We define $\Phi(t)$ and $e(t)$ by recursion. 
Regarding the order of construction,  
let $e(\emptyset):=\emptyset$. 
If $s\in \SS^\ddim_\kappa$ and $\lh(s)\in\Lim$, let 
$e(s):=\bigcup_{t\subsetneq s}e(t)$.
If $e(s)$ is defined for some $s\in \SS^\ddim_\kappa$, 
we define $\Phi(s\conc \langle0\rangle^{\alpha})$, $\Phi(s\conc\langle0\rangle^{\alpha}{}\conc\langle i\rangle)$ 
and
$e(s\conc\langle0\rangle^{\alpha}{}\conc\langle i\rangle)$ 
for all $\alpha<\kappa$ and all 
$i\in\ddim\,\setminus\{0\}$
in the next step by recursion on $\alpha$, as follows
(see Figure~\ref{figure: ODD(kappa^kappa) main theorem}).

\vspace{5pt} 
{%
\newcommand{\x}{0.6}%
\newcommand{\y}{1}%
\newcommand{\z}{8}%
\begin{figure}[H]
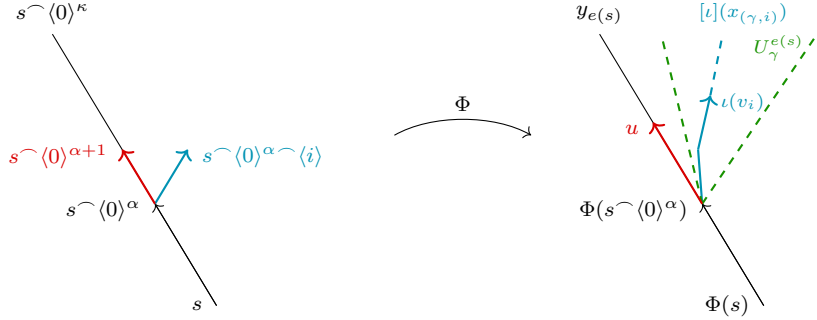
 
\centering
\tikz[scale=0.9,font=\scriptsize]{ 
\draw 
(\x-130pt,0); 

\draw 
(\x+2,2.5\y) edge[->,bend left, looseness=0.8] (\x+4,2.5\y); 

\draw 
    (\x+3,2.7\y) node[above] {$\Phi$}; 

\draw 
(0*\x,0) edge[-] (-4*\x,4*\y) 
(0*\x,0) edge[->] (-1.5*\x,1.5*\y); 

\draw[thick,color={TealArrow}]  
(-1.5*\x,1.5*\y) edge[->] (-0.7*\x,2.3*\y); 

\draw[thick,color={RedArrow}] 
(-1.5*\x,1.5*\y) edge[->] (-2.3*\x,2.3*\y); 

\draw 
    (0*\x,0*\y) node[above] {} 
    (4*\x,4*\y) node[above] {}; 

\draw
    (-4*\x,4*\y) node[above] {$s\conc\langle 0\rangle^\kappa$}; 

\draw
    (-0.1*\x,0*\y) node[left] {$s$} 
    (-0.6*\x,2.2*\y) node[right] {\color{TealArrow}$s\conc\langle 0\rangle^{\alpha}\conc\langle i\rangle$} 
    (-1.6*\x,1.4*\y) node[left] {$s\conc\langle 0\rangle^{\alpha}$} 
    (-2.4*\x,2.2*\y) node[left] {\color{RedArrow}$s\conc\langle 0\rangle^{\alpha+1}$};

    \draw 
(0*\x+\z,0) edge[-] (-4*\x+\z,4*\y) 
(0*\x+\z,0) edge[->] (-1.5*\x+\z,1.5*\y); 

\draw[thick,color={TealArrow}]  
(-1.5*\x+\z,1.5*\y) edge[-] (-1.6*\x+\z,2.3*\y) 
(-1.6*\x+\z,2.3*\y) edge[->] (-1.3*\x+\z,3.1*\y); 

\draw[-, thick, dashed, color={TealArrow}]  
(-1.3*\x+\z,3.1*\y) edge[-] (-1*\x+\z,4*\y); 

\draw[-, thick, dashed, color={DarkGreenArrow}]  
(-1.5*\x+\z,1.5*\y) edge[-] (-2.5*\x+\z,4*\y) 
(-1.5*\x+\z,1.5*\y) edge[-] (1.3*\x+\z,4*\y); 

\draw[thick,color={RedArrow}] 
(-1.5*\x+\z,1.5*\y) edge[->] (-2.7*\x+\z,2.7*\y);

\draw 
    (0*\x+\z,0*\y) node[above] {} 
    (4*\x+\z,4*\y) node[above] {}; 

\draw
    (-4*\x+\z,4*\y) node[above] {$y_{e(s)}$}
    (-0.5*\x+\z,4*\y) node[above] {\color{TealArrow} \tiny $[\iota](x_{(\gamma,i)})$}
    (-0.5*\x+\z,3.8*\y) node[right] {\color{DarkGreenArrow}  \tiny $U^{e(s)}_\gamma$}; 
     
\draw
    (-0.1*\x+\z,0*\y) node[left] { $\Phi(s)$} 
    (-1.3*\x+\z,3.0*\y) node[right] {\color{TealArrow} \tiny $\iota(v_i)$} 
    (-1.6*\x+\z,1.4*\y) node[left] {$\Phi(s\conc\langle 0\rangle^{\alpha})$} 
    (-2.8*\x+\z,2.6*\y) node[left] {\color{RedArrow}$u$}; 

   } 
\caption{Defining $\Phi$ in the successor step} 
\label{figure: ODD(kappa^kappa) main theorem}
\end{figure}%
}%
For each $(\beta,i)\in D$, fix some 
$x_{(\beta,i)}\in {}^\kappa D$ with
$e(s)\conc\langle(\beta,i)\rangle\subseteq x_{(\beta,i)}$.
Let $\Phi(s):=\iota(e(s))$.
If $\alpha\in\Lim$, let
$\Phi(s\conc \langle0\rangle^{\alpha}):=
\bigcup_{\beta<\alpha}
\Phi(s\conc \langle0\rangle^{\beta})$. 
Now, assume that $\Phi(s\conc\langle 0\rangle^{\alpha})$ has been defined for some $\alpha<\kappa$.
Since $\Phi(s\conc\langle 0\rangle^{\alpha})\subsetneq y_{e(s)}$ and the sets $U^{e(s)}_\beta$ converge to $y_{e(s)}$, 
there exists $\gamma<\kappa$ with 
$U^{e(s)}_\gamma\subseteq N_{\Phi(s\conc\langle 0\rangle^{\alpha})}$.
\todog{We can also get $e$ to be $\wedge$-preserving, if
we make sure that the $\gamma=\gamma_\alpha$ chosen for this $\alpha$ differs from the $\gamma_\beta$'s for all $\beta<\alpha$.
So if $\iota$ is $\perp$-preserving, 
then $\Phi\restr \SS^\ddim_\kappa=\iota\comp e$ is also $\perp$-preserving.}
In particular, we have
$\Phi(s\conc\langle 0\rangle^{\alpha})\subseteq [\iota](x_{(\gamma,i)})$
for every 
$i\in\ddim\,\setminus\{0\}$.
Since 
$\langle [\iota](x_{(\beta,i)}): (\beta,i)\in D\rangle$ 
is a hyperedge of $\CC^{Y,H}$, 
\todog{This is why we needed to fix all the $x_{(\beta,i)}$'s first, instead of choosing them one by one for each $\beta$}
we have 
$$\langle y_{e(s)}\rangle\conc\langle [\iota](x_{(\gamma,i)}): 1\leq i<\ddim\rangle\in H.$$
Since $H$ is box-open on ${}^\kappa\kappa$, 
there exist
$u\subsetneq y_{e(s)}$ and 
$v_i\subsetneq x_{(\gamma,i)}$
with
$\Phi(s\conc\langle 0\rangle^{\alpha})\subsetneq u$ 
and
$e(s)\conc\langle(\gamma,i)\rangle\subseteq v_i$
such that
$N_{u}\times\prod_{1\leq i<\ddim} N_{\iota(v_i)}\subseteq H$.
Now let $\Phi(s\conc \langle0\rangle^{\alpha+1}):=u$, 
$e(s\conc\langle0\rangle^{\alpha}{}\conc\langle {i}\rangle):=v_i$ 
and 
$\Phi(s\conc\langle0\rangle^{\alpha}{}\conc\langle{i}\rangle):=\iota(v_i)$ 
for all $i\in\ddim\,\setminus\{0\}$.
This completes the construction of $\Phi$.

We have $[\Phi](\RR^\ddim_\kappa)\subseteq X$ 
since $[\Phi]\restr\RR^\ddim_\kappa=[\iota]\comp[e]$
by~\ref{prop ODD for C^HX 1} and $\ran([\iota])\subseteq X$.
Moreover, $[\Phi](\QQ^\ddim_\kappa)\subseteq \{y_{e(s)}: s\in\SS^\ddim_\kappa\}\subseteq Y$
by~\ref{prop ODD for C^HX 2}.
Lastly, $\Phi$ is an order homomorphism for $({}^\kappa\kappa,H)$ by~\ref{prop ODD for C^HX 3}.
\end{proof}

The next result follows from Theorem~\ref{main theorem} and the previous theorem. 
We will use the following 
notation for classes $\mathcal C$, $\mathcal D$ and $\mathcal H$. 
\index{Kechris--Louveau--Woodin dichotomy!C@for $\ddim$-dihypergraphs!AA@for classes\idf$\KLW_\kappa^{\ddim,\mathcal H}(\mathcal C,\mathcal D)$}%
$\KLW_\kappa^{\ddim,\mathcal H}(\mathcal C,\mathcal D)$ denotes 
the statement that $\KLW^H_\kappa(X,Y)$ holds for all subsets $X\in\mathcal C$ and $Y\in\mathcal D$ of ${}^\kappa\kappa$
and all box-open $\ddim$-dihypergraphs $H\in\mathcal H$ on ${}^\kappa\kappa$.
\index{Kechris--Louveau--Woodin dichotomy!C@for $\ddim$-dihypergraphs!AB@for classes of sets\idf$\KLW_\kappa^{\ddim}(\mathcal C,\mathcal D)$}%
We omit 
$\mathcal D$ from the notation if it is $\powerset({}^\kappa\kappa)$, and we omit
$\mathcal H$ if it is the class of all $\ddim$-dihypergraphs on ${}^\kappa\kappa$. 
\index{Kechris--Louveau--Woodin dichotomy!C@for $\ddim$-dihypergraphs!B@for a class of sets\idf$\KLW^{\ddim}_\kappa(\mathcal C)$}%
For example, $\KLW^{\ddim}_\kappa(\mathcal C)$ states
that 
$\KLW^H_\kappa(X,Y)$ holds for all $X,Y\subseteq{}^\kappa\kappa$ with $X\in\mathcal C$ and all box-open $\ddim$-dihypergraphs $H$ on~${}^\kappa\kappa$.
\index{Kechris--Louveau--Woodin dichotomy!C@for $\ddim$-dihypergraphs!B@for a set\idf$\KLW^{\ddim}_\kappa(X)$}%
\index{Kechris--Louveau--Woodin dichotomy!C@for $\ddim$-dihypergraphs!B@for a set and class\idf$\KLW^{\ddim}_\kappa(X,\mathcal C)$}%
If $\mathcal C$ has one single element $X$, then 
we write $X$ instead of $\mathcal C$, and similarly for $\mathcal D$ and~$\mathcal H$. 

\begin{corollary}
\label{corollary: KLW^H cons}
For any $2\leq\ddim\leq\kappa$,
the following hold in all $\Col(\kappa,\lle\lambda)$-generic extensions of $V$:
\begin{enumerate-(1)}
\item
$\KLW^{\ddim,\defsetsk}_\kappa(\defsetsk,\defsetsk)$
if $\lambda>\kappa$ is inaccessible in $V$.
\item
$\KLW^{\ddim}_\kappa(\defsetsk)$
if $\lambda>\kappa$ is Mahlo in~$V$.
\end{enumerate-(1)}
\end{corollary}


The next corollaries show that the $\ddim$-dimensional open dihypergraph dichotomy is a special case 
of $\KLW^\ddim_\kappa(X)$ for all dimensions $2\leq\ddim\leq\kappa$.
Moreover, the two dichotomies are equivalent in the $\kappa$-dimensional case.

\begin{corollary} 
\label{theorem: KLW H and ODD H}
Suppose that $H$ is box-open on ${}^\kappa\kappa$.%
\footnote{This corollary extends Corollary~\ref{cor: PSP from ODD for the kappa-Baire space} from the complete graph to arbitrary box-open dihypergraphs.}
\begin{enumerate-(1)} 
\item 
\label{theorem: KLW H and ODD H a}
$\KLW_\kappa^H(X,Y)\Longleftrightarrow\ODD\kappa {H\restr (X\cap Y)}$ 
if $X\setminus Y$  can be separated from $Y$ by a union of $\kappa$ many $H$-closed sets. 
\item 
\label{theorem: KLW H and ODD H b}
$\KLW_\kappa^H(X\setminus Y,Y)\land\ODD\kappa {H\restr (X\cap Y)}\Longrightarrow\KLW^H_\kappa(X,Y)$. 
\end{enumerate-(1)} 
\end{corollary}
\begin{proof}
For \ref{theorem: KLW H and ODD H a}, it suffices 
to show that $\ODD\kappa{\CC^{Y,H}\restr X}\Longleftrightarrow\ODD\kappa{H\restr{(X\cap Y)}}$ by Theorem \ref{theorem: ODD and KLW^H}.
By Corollary~\ref{cor: ODD subsequences}, it further suffices that $\CC^{Y,H}\restr X\equivf H\restr (X\cap Y)$, 
or equivalently, that $\CC^{Y,H}\restr A$ admits a $\kappa$-coloring if and only if $H\restr(A\cap Y)$ admits a $\kappa$-coloring for all subsets $A$ of $X$. 
But this follows from
Theorem \ref{theorem: ODD and KLW^H}~\ref{general KLW from ODD coloring}. 
Moreover, \ref{theorem: KLW H and ODD H b} follows from \ref{theorem: KLW H and ODD H a}. 
\end{proof}

\begin{corollary}\ 
\label{cor: from KLW to ODD}
Suppose $2\leq\ddim\leq\kappa$.%
\footnote{\ref{cor: from KLW to ODD 2} and \ref{cor: from KLW to ODD 3} are analogues of Corollary~\ref{cor: PSP from ODD for the kappa-Baire space 2}~\ref{cor: PSP from ODD for the kappa-Baire space 2 a} and \ref{cor: PSP from ODD for the kappa-Baire space 2 b}.} 
\begin{enumerate-(1)}
\item\label{cor: from KLW to ODD 1}
If $X\setminus Y$ can be separated from $Y$ by a $\Fsigma(\kappa)$ set, then
\begin{enumerate-(a)}
\item 
$\KLW_\kappa^\ddim(X,Y)\Longleftrightarrow\ODD\kappa\ddim(X\cap Y)$, 
\item 
$\KLW_\kappa^{\ddim,\defsetsk}(X,Y)\Longleftrightarrow\ODD\kappa\ddim(X\cap Y,\defsetsk)$.
\end{enumerate-(a)}
\vspace{2 pt}
\item\label{cor: from KLW to ODD 2}
\begin{enumerate-(a)}
\item\label{cor: from KLW to ODD 2 a}
$
\KLW_\kappa^\ddim(X,{}^\kappa\kappa)
\Longleftrightarrow
\ODD\kappa\ddim(X)
$,
\item\label{cor: from KLW to ODD 2 b}
$
\KLW_\kappa^{\ddim,\defsetsk}(X,{}^\kappa\kappa)
\Longleftrightarrow
\ODD\kappa\ddim(X,\defsetsk)
$.
\end{enumerate-(a)}
\item\label{cor: from KLW to ODD 3}
If $Y$ is a $\Gdelta(\kappa)$ set, then
\begin{enumerate-(a)}
\item\label{cor: from KLW to ODD 3 a}
$
\KLW_\kappa^\ddim({}^\kappa\kappa,Y)
\Longleftrightarrow
\ODD\kappa\ddim(Y)
$,
\item\label{cor: from KLW to ODD 3 b}
$
\KLW_\kappa^{\ddim,\defsetsk}({}^\kappa\kappa,Y)
\Longleftrightarrow
\ODD\kappa\ddim(Y,\defsetsk)
$.
\end{enumerate-(a)}
\end{enumerate-(1)}
\end{corollary}
\begin{proof}
\ref{cor: from KLW to ODD 1} follows from the previous corollary, since any closed set is also $H$-closed for any box-open dihypergraph $H$ on ${}^\kappa\kappa$.
\ref{cor: from KLW to ODD 2} and \ref{cor: from KLW to ODD 3}
 follow from \ref{cor: from KLW to ODD 1} for $Y:={}^\kappa\kappa$ and $X:={}^\kappa\kappa$, respectively.
\end{proof}


The next corollary strengthens 
this significantly in the $\kappa$-dimensional case.

\begin{corollary}\ 
\label{cor: from KLW to ODD kappa}
\begin{enumerate-(1)}
\item\label{cor: from KLW to ODD kappa 1}
$ \ODD\kappa\kappa(X)\Longleftrightarrow
\KLW_\kappa^\kappa(X)
$.
\item\label{cor: from KLW to ODD kappa 2}
$\ODD\kappa\kappa(X,\defsetsk)
\Longleftrightarrow
\KLW_\kappa^{\kappa,\defsetsk}(X)
$.
\end{enumerate-(1)}
\end{corollary}
\begin{proof}
For~\ref{cor: from KLW to ODD kappa 1}, 
$\ODD\kappa\kappa(X)$ implies $\KLW_\kappa^\kappa(X)$ 
by Theorem~\ref{theorem: ODD and KLW^H}.
The converse direction follows from \ref{cor: from KLW to ODD 2}~\ref{cor: from KLW to ODD 2 a}
in the previous corollary since
$\KLW^\kappa_\kappa(X)$ implies 
$\KLW^\kappa_\kappa(X,{}^\kappa\kappa)$.
The proof of~\ref{cor: from KLW to ODD kappa 2} is analogous.
\end{proof}

The next corollary is used to prove Theorem~\ref{thm: ODD for kappa^kappa and analytic sets}.  

\begin{corollary}
\label{cor: ODD(X) implies ODD(closed subsets of X)}
If $X\setminus Y$ can be separated from $Y$ by a $\Fsigma(\kappa)$ set, then 
\begin{enumerate-(1)}
\item\label{cor: ODD(X) implies ODD(closed subsets of X) 1}
$\ODD\kappa\kappa(X)\Longrightarrow\ODD\kappa\kappa(X\cap Y)$,
\item\label{cor: ODD(X) implies ODD(closed subsets of X) 2}
$\ODD\kappa\kappa(X,\defsetsk)\Longrightarrow\ODD\kappa\kappa(X\cap Y,\defsetsk)$ if $Y\in\defsetsk$. 
\end{enumerate-(1)}
\end{corollary}
\begin{proof}
For~\ref{cor: ODD(X) implies ODD(closed subsets of X) 1},
$\ODD\kappa\kappa(X)$ implies $\KLW^\kappa_\kappa(X,Y)$
by Theorem~\ref{theorem: ODD and KLW^H},
and $\KLW^\kappa_\kappa(X,Y)$ implies $\ODD\kappa\kappa(X\cap Y)$ 
by Corollary~\ref{theorem: KLW H and ODD H}~\ref{cor: from KLW to ODD 1}.
The proof of~\ref{cor: ODD(X) implies ODD(closed subsets of X) 2} is analogous.
\end{proof}

In particular, these implications hold if $Y$ is a $\Gdelta(\kappa)$ set. 
Now Theorem~\ref{thm: ODD for kappa^kappa and analytic sets} follows by letting $X:={}^\kappa\kappa$, since 
$\ODD\kappa\kappa(Y)$ for closed sets $Y$ implies the same statement for $\kappa$-analytic sets by Lemma~\ref{ODD for continuous images} and the same holds for $\ODD\kappa\kappa(Y,\defsetsk)$.

\begin{remark}
\label{KLW^H for different dimensions}
Suppose $2\leq c<\ddim$. 
We show that $\KLW^\ddim_\kappa(X,Y)$ implies $\KLW^c_\kappa(X,Y)$
for all subsets $X,\,Y$ of ${}^\kappa\kappa$. 
To see this, let $H$ be a $c$-dihypergraph on ${}^\kappa\kappa$
and let
$I:=\proj_{\ddim, c}^{-1}(H){\restr}{}^\kappa\kappa$.%
\footnote{Recall the notation $\proj_{\ddim,c}$ from the paragraph before Lemma~\ref{sublemma: ODD for different dimensions}.}
We claim that $\KLW^I_\kappa(X,Y)$ is equivalent to $\KLW_\kappa^H(X,Y)$. 
\begin{enumerate-(1)}
\item 
To see that the first options of the dichotomies are equivalent, note that 
$y$ is an $H$-limit point of a set $A$ if and only if 
$y$ is an $I$-limit point of $A$. 
\item 
If $g$ is as in the second option of $\KLW_\kappa^I(X,Y)$, 
then $f:=g\restr{{}^\kappa c}$ is 
a continuous homomorphism from $\dhH c$ to $H$ as in the proof of Lemma~\ref{sublemma: ODD for different dimensions}~\ref{dif dims homomorphisms}. 
Moreover, $f$ reduces $(\RR^c_\kappa,\QQ^c_\kappa)$ to $(X,Y)$ since 
$\RR^c_\kappa=\RR^d_\kappa\cap{{}^\kappa c}$
and
$\QQ^c_\kappa=\QQ^d_\kappa\cap{{}^\kappa c}$.
Now suppose that $f$ is as in the second option of $\KLW_\kappa^H(X,Y)$. 
Fix a map $k: d\to c$ with $k\restr c = \id_c$ and $k(i)\neq 0$ for all $i\in d\,\setminus\, c$. 
Define $g:{}^\kappa\ddim\to X$ by letting $g(x):=f(k\comp x)$. 
Then $g$ is a continuous homomorphism from $\dhH d$ to $H$ as in the proof of Lemma~\ref{sublemma: ODD for different dimensions}~\ref{dif dims homomorphisms}. 
\todog{This works for any $c\subsetneq \ddim$ with $0\in c$. But $0\in c$ is needed for both directions in both (1) and (2).}
Since $x\in \RR^d_\kappa$ if and only if $k\circ x\in \RR^c_\kappa$ for all $x\in {}^\kappa\kappa$, $g$ reduces $(\RR^d_\kappa,\QQ^d_\kappa)$ to $(X,Y)$. 
\end{enumerate-(1)}
\end{remark}


\subsection{The determinacy of V\"a\"an\"anen's perfect set game}
\label{subsection: Vaananens game}

In this subsection, 
we obtain the determinacy of the games studied in \cite{VaananenCantorBendixson}*{Section 2}
for all subsets of ${}^{\kappa}\kappa$ 
as an application of $\ODD\kappa\kappa({}^\kappa\kappa)$.
In fact, the result yields additional properties of the underlying set in the case that player $\pone$ has a winning strategy. 
Similarly, $\ODD\kappa\kappa({}^\kappa\kappa,\defsetsk)$ implies the determinacy of  V\"a\"an\"anen's game for all definable subsets of ${}^\kappa\kappa$.

By Cantor--Bendixson analysis, any topological space can be decomposed into a scattered and a crowded\footnote{Recall that a topological space $X$ is \emph{scattered} if each nonempty subspace contains an isolated point.
$X$ is \emph{crowded} if it has no isolated points.} 
subset by iteratively removing isolated points. If the space has a countable base, then the scattered part is countable.  
V\"a\"an\"anen adapted this to subsets of generalized Baire spaces via 
a game of length $\xi\leq\kappa$.
The following is a variant of V\"a\"an\"anen's game.
\todog{$\cbii(X)$ implies the determinacy of V\"a\"an\"anen's (original) perfect set game $\GV_\kappa(X,x)$ for all $x\in X$
$\cbi(X)$ implies the determinacy of $\GV_\kappa(X)$}%

\begin{definition}[\cite{VaananenCantorBendixson}*{Section 2}] 
\label{def: Vaananen's game} 
\index{V\"a\"an\"anen's perfect set game!full game\idf$\GV_\xi(X)$}%
\emph{V\"a\"an\"anen's perfect set game}%
\footnote{In fact, V\"a\"an\"anen did not consider the full game $\GV_\xi(X)$, but only defined the local versions $\GV_\xi(X,x)$ for all $x\in {}^\kappa\kappa$ with the requirement $x_\alpha\in X$ for all $\alpha\geq 1$ \cite{VaananenCantorBendixson}. These versions are given at the end of this definition. Since we work with the games $\GV_\xi(X,x)$ for $x\in X$ only, we can amalgamate them into a single game.} 
$\GV_\xi(X)$ 
of length $\xi\leq\kappa$ for a subset 
$X$ of ${}^\kappa\kappa$
is played by two players $\pone$ and $\ptwo$ as follows. 
%
$\pone$ plays a strictly increasing continuous sequence $\langle \gamma_\alpha : \alpha< \xi \rangle$ of ordinals below $\kappa$ with $\gamma_0=0$.
$\ptwo$ plays 
an injective sequence
$\langle x_\alpha : \alpha< \xi \rangle$ of elements of $X$ 
with $x_\alpha \supseteq (x_\beta\restr \gamma_{\beta+1})$
for all $\beta<\alpha$. 
{
\begin{table}[H]
\centering
\begin{tabular}{c cccccccccccccc}
  $\pone$  &  $\gamma_0$    &      & $\gamma_1$ &                & \ldots & 
  &            &$\gamma_{\alpha}$ &   &\ldots \\[5 pt] 

  $\ptwo$ &      & $x_0$ &           & $x_1$ &        &           & \ldots & 
  &  $x_\alpha$&           &  & \ldots 
\end{tabular}
\end{table}
}

\noindent 
Each player loses immediately if they do not follow these requirements. 
If both follow the rules in all rounds, then $\ptwo$ wins. 

\index{V\"a\"an\"anen's perfect set game!with fixed first move\idf$\GV_\xi(X,x)$}%
Moreover, the game $\GV_\xi(X,x)$ is defined just like $\GV_\xi(X)$ for any $x$ in $X$, except that $\ptwo$ must play $x_0=x$ in the first round. 
\end{definition}

We say that a player \emph{wins} a game if they have a winning strategy. 

\begin{definition}[\cite{VaananenCantorBendixson}*{Sections 2 \& 3}]  
\label{def: kernel}
\label{def: scattered part}
Suppose that $\xi\leq\kappa$ and that $X$ is a subset of ${}^\kappa\kappa$. 
\begin{enumerate-(1)} 
\item\label{Sc def}
\index{scattered part@$\xi$-scattered part\idf$\Sc_\xi(X)$}
The \emph{$\xi$-scattered part} of $X$ is
$\Sc_\xi(X)= \{ x\in X :  \text{ \pone\ wins  $\GV_\xi(X,x)$} \}$. 
\item\label{Ker def}
\index{kernel@$\xi$-kernel\idf$\Ker_\xi(X)$}
The \emph{$\xi$-kernel} of $X$ is
$\Ker_\xi(X)= \{ x\in X :  \text{ \ptwo\ wins  $\GV_\xi(X,x)$} \}$.%
\footnote{V\"a\"an\"anen originally defined the $\xi$-kernel of $X$ as the set $\Ker^V_\xi(X)$ of all $x\in {}^\kappa\kappa$ such that $\ptwo$ wins $\GV_\xi(X,x)$. 
Note that $\Ker_\xi^V(X)\cap X=\Ker_\xi(X)$ and that
$\bar{\Ker_\xi(X)}=\Ker^V_\xi(X)$ for all $\xi\geq\omega$.}  
\end{enumerate-(1)} 
\index{scattered set@$\xi$-scattered set\idf} 
$X$ is \emph{$\xi$-scattered} if $X=\Sc_\xi(X)$ and 
\index{crowded set@$\xi$-crowded set\idf} 
\emph{$\xi$-crowded} if $X=\Ker_\xi(X)\neq\emptyset$. 
\end{definition} 

These definitions extend well-know notions from topology: 
$\omega$-scattered is equivalent to scattered and $\omega$-crowded is equivalent to crowded. 
Moreover, any subset of ${}^\kappa\kappa$ whose closure is $\kappa$-perfect is a $\kappa$-crowded set.
In fact, a subset $X$ of ${}^\kappa\kappa$ is $\kappa$-crowded if and only if every $x\in X$ is contained in some subset $Y$ of $X$ whose closure is $\kappa$-perfect by Lemma~\ref{theorem: GV II} below. 

Note that $\Sc_\xi(X)$ is a $\xi$-scattered relatively open subset 
and $\Ker_\xi(X)$ is a relatively closed subset of $X$. 
$\Ker_\xi(X)$ is $\xi$-crowded if  
\todog{this may fail when $\xi=\gamma+1$ where $\gamma$ is indecomposable}
$\xi$ is an indecomposable ordinal.\footnote{I.e., $\alpha+\beta<\xi$ for all $\alpha,\beta<\xi$.} 
Moreover: 
\begin{enumerate-(a)} 
\item 
$X$ is $\xi$-scattered $\ \Longleftrightarrow\ $ $\pone$ wins $\GV_\xi(X)$.  
\item 
$\Ker_\xi(X)\neq\emptyset$ $\ \Longleftrightarrow\ $ $\ptwo$ wins $\GV_\xi(X)$. 
\end{enumerate-(a)} 

By the Gale-Stewart theorem, $\GV_\xi(X)$ is determined for all $\xi\leq\omega$. 
This is no longer true for $\xi>\omega$. 
For instance, there exists a closed subset $X$ of ${}^{\omega_1}\omega_1$ such that $\GV_{\omega+1}(X)$ is not determined,
\todog{This does not necessarily imply the non-determinacy of $\GV_\kappa(X)$}
and in fact, it is consistent that there exists such a set $X$ of size $\omega_2$ 
\cite{VaananenCantorBendixson}.%
\footnote{See \cite{GalgonThesis}*{Section 1.5} for variants of these results.}
The next property can be understood both as the determinacy of V\"a\"an\"anen's game of length $\kappa$ with an extra condition and as a variant of the Cantor--Bendixson analysis.

\begin{definition}
\label{def: cbi}
For any class $\mathcal C$, 
$\cbi(\mathcal C)$ states that the following holds for all subsets $X\in\mathcal C$ of ${}^\kappa\kappa$:
\index{Cantor--Bendixson dichotomy!first variant!for sets\idf$\cbi(X)$}%
\index{Cantor--Bendixson dichotomy!first variant!for classes\idf$\cbi(\mathcal C)$} 
\begin{quotation}
$\cbi(X)$: Either 
$|X|\leq\kappa$ and $\pone$ wins $\GV_\kappa(X)$, 
or 
$\ptwo$ wins $\GV_\kappa(X)$.
\end{quotation}
\end{definition}
Equivalently, $\cbi(X)$ states that either $X$ is $\kappa$-scattered and of size ${\leq}\kappa$, or it contains a $\kappa$-crowded subset.%
\footnote{If $\ptwo$ wins $\GV_\kappa(X)$, then $\Ker_\kappa(X)\neq\emptyset$ is a $\kappa$-crowded subset of $X$.}
This principle may fail even for closed sets. 
In fact, it is consistent that
there exists an $\omega_1$-scattered closed subset of ${}^{\omega_1}\omega_1$ of size $\omega_2$ \cite{VaananenCantorBendixson}*{Theorem~3}.\footnote{See \cite{GalgonThesis}*{Section 1.5} for variants of this result.} 

We first show that
$\ODD\kappa\kappa({}^\kappa\kappa) \Longrightarrow \cbi(\pwrset({}^\kappa\kappa))$
and  
$\ODD\kappa\kappa({}^\kappa\kappa,\defsets\kappa) \Longrightarrow \cbi(\defsetsk)$. 
We
use the hypergraph $\CC^X$
from Subsection~\ref{subsection: original KLW}. 
This consists of all 
convergent 
$\kappa$-sequences $\bar x$ in ${}^\kappa \kappa$ 
with $\lim_{\alpha<\kappa}(\bar x)\in X \setminus \ran(\bar x)$. 
\todoq{What can we say about the complexity of $\CC^X$ if we know the complexity of $X$?}

\begin{lemma}
\label{theorem: GV I}
Let $X\subseteq{}^\kappa\kappa$.
If $\CC^X$ has a $\kappa$-coloring, then $|X|\leq\kappa$ and
$\pone$ wins $\GV_\kappa(X)$.
\end{lemma}
\begin{proof}
We have already shown $|X|\leq\kappa$ in 
Corollary~\ref{cor: CC coloring}.
For the second claim, recall that a subset $A$ of ${}^\kappa\kappa$ is $\CC^X$-independent if and only if $A'\cap X=\emptyset$,
where $A'$ denotes the set of limit points of $A$.
Take $\CC^X$-independent sets $A_\alpha$ with ${}^\kappa\kappa=\bigcup_{\alpha<\kappa} A_\alpha$.
We define the following strategy for player $\pone$. 
In rounds of the form $\alpha+1$, $\pone$ chooses $\gamma_{\alpha+1}$ so that 
\begin{enumerate-(i)} 
\item 
\label{theorem: GV I a}
$x_{\alpha}\restr\gamma_{\alpha+1}\neq
x_{\beta}\restr\gamma_{\alpha+1}$ for all $\beta<\alpha$, and 
\item 
\label{theorem: GV I b}
$A_\alpha\cap N_{x_\alpha\restr\gamma_{\alpha+1}}\subseteq\{x_\alpha\}$. 
\end{enumerate-(i)} 
$\pone$ can achieve the first property because $\ptwo$ plays an injective sequence and the second property because $x_\alpha\in X$ cannot be a limit point of $A_\alpha$. 
Suppose $\ptwo$ wins a run of the game where $\pone$ has used this strategy. 
Then $y:=\bigcup_{\alpha<\kappa} x_\alpha\restr\gamma_{\alpha+1}$ 
\todog{this is where we use that $x_\alpha\supseteq x_\beta\restr\gamma_{\beta+1}$ for all $\beta<\alpha$}
is an element of ${}^\kappa\kappa$, 
so $y\in A_\beta$ for some $\beta<\kappa$.
Since $y\in N_{x_\beta\restr\gamma_{\beta+1}}$, 
we must have $y=x_\beta$ by \ref{theorem: GV I b}. 
This contradicts \ref{theorem: GV I a} for any $\alpha>\beta$. 
\end{proof}

The equivalence of \ref{II wins GV} and \ref{DIS subset} in the next lemma was shown in \cite{SzThesis}*{Proposition~2.69}. 
Our proof is essentially the same, but we include it due to the different notation.%
\footnote{In \cite{SzThesis}, a subset $Y$ of ${}^\kappa\kappa$ is called $\kappa$-strongly dense in itself
if $\closure Y$ is $\kappa$-perfect and $\Ker_\xi(X)$ denotes V\"a\"an\"anen's definition of the kernel.} 

\begin{lemma}
\label{theorem: GV II}%
The following statements are equivalent for any $X\subseteq {}^{\kappa}\kappa$
and $x\in X$:
\begin{enumerate-(1)}
\item\label{II wins GV} 
$x\in\Ker_\kappa(X)$.
\item\label{DIS subset}
There exists $Y\subseteq X$ such that $x\in Y$ and $\closure Y$ is $\kappa$-perfect.
\item\label{GV rect oh}
There is a nice reduction
$\Phi$ for $({}^\kappa\kappa, X)$ with $x\in[\Phi]\big(\QQ_\kappa)$.%
\todog{The proof shows that $\Phi$ can equivalently be chosen to be continuous.}
\end{enumerate-(1)}
\end{lemma}
\begin{proof}
\ref{GV rect oh} $\Rightarrow$ \ref{DIS subset}: Suppose $\Phi$ is a nice reduction for $({}^\kappa\kappa, X)$. 
Let $Y:=[\Phi]\big(\QQ_\kappa)$.
Then $x\in Y\subseteq X$. 
Moreover,
$\closure{Y}=\ran([\Phi])$ is $\kappa$-perfect 
by Lemma~\ref{perfect subsets and sop maps}.

\ref{DIS subset} $\Rightarrow$ \ref{II wins GV}: Let
$Y$ be a subset of $X$ with $x\in Y$ such that $\bar Y$ is $\kappa$-perfect. 
Then $T(Y)$ is a $\kappa$-perfect tree.
It is straightforward to construct a winning strategy $\tau$ for  $\ptwo$ in $\GV_\kappa(Y,x)$ 
using the fact that $T(Y)$ is cofinally splitting in successor rounds of the game and
the fact that $T(Y)$ is $\lle\kappa$-closed in limit rounds.
Since $Y\subseteq X$, 
$\tau$ is also a winning strategy for  $\ptwo$ in $\GV_\kappa(X,x)$.

\ref{II wins GV} $\Rightarrow$ \ref{GV rect oh}: 
Suppose that $\tau$ is a winning strategy for  $\ptwo$ in $\GV_\kappa(X,x)$.
We construct a nice reduction for $({}^\kappa\kappa,X)$
by having  $\ptwo$ use $\tau$ repeatedly in response to
different partial plays of  $\pone$.
In detail, we construct
a $\perp$- and strict order preserving 
function $\Phi: {}^{<\kappa}2\to T(X)$, a continuous 
strict order preserving function   
$\langle \gamma_s : s\in \SS_\kappa \rangle$%
\footnote{Recall the definition of $\SS_\kappa$ from Subsection~\ref{subsection: original KLW}.} 
of ordinals below $\kappa$ 
and a function $\langle x_s : s\in {}^{<\kappa}\kappa \rangle$ of elements of $X$ such that the following hold for all $s\in{}^{<\kappa}2$: 
\begin{enumerate-(i)}
\item\label{GV two 0}
$\gamma_{\emptyset}=0$ and $x_{\emptyset}=x$.
\item\label{GV two i}
$\Phi(s)=x_s\restr\gamma_{s\conc\langle 1\rangle}$.
\item\label{GV two ii}
$x_s=\tau\big(\langle\gamma_{t}:\, t\subseteq s,\, t\in\SS_\kappa\rangle\big).$%
\end{enumerate-(i)}
For all $s\in{}^{<\kappa}\kappa$, let
$p_s:= \langle \gamma_t,\, x_t:\,t\subseteq s,\, t\in\SS_\kappa\rangle.$
\ref{GV two ii} states that $p_s$ is a partial run of $\GV_\kappa(X)$ where $\ptwo$ uses the strategy $\tau$.
Note that $x_s$ is constant in intervals disjoint from $\SS_\kappa$, but takes a new value at each $s\in \SS_\kappa$: 
If $s\in\SS_\kappa$, then
$\bigcup \{p_u:\, u\subsetneq s,\,u\in\SS_\kappa\}$
is extended to $p_s$
by $\pone$ playing $\gamma_s$ as defined below 
and $\ptwo$ choosing $x_s$ using $\tau$. 
Otherwise,
$s=t\conc \langle 0\rangle^{\alpha}$ for some $\alpha<\kappa$. Then $p_s=p_t$, so $x_s=x_t$.

We shall construct $x_s$, $\gamma_s$ and $\Phi(s)$ 
by recursion.
Regarding the order of the construction, let $\gamma_\emptyset=0$ and $x_\emptyset=x$. 
If $s\in\SS_\kappa$ and $\lh(s)\in\Lim$, let 
$\gamma_s:=\bigcup\{\gamma_{t}:\, t\subsetneq s,\,t\in\SS_\kappa\}$.
If $s\in{}^{<\kappa} 2$
and 
$\gamma_{t}$ has been constructed for all $t\subseteq s$ with $t\in\SS_\kappa$, then
\todog{The recursion is ordered this way because if $s=t\conc\langle 1\rangle$, then $\gamma_s$ also has to be constructed at this point.}
construct $x_s$, $\gamma_{s\conc\langle 1\rangle}$ and $\Phi(s)$ in the next step as follows.
Let $x_s$ be as in~\ref{GV two ii}.
We aim to choose 
$\gamma_{s\conc\langle 1\rangle}$ so that 
$\Phi$ remains $\perp$- and strict order preserving for 
$\Phi(s):=x_s\restr\gamma_{s\conc\langle 1\rangle}$.

The next claim shows that 
$\Phi(s)$
extends $\Phi(t)$ for all $t\subsetneq s$ if $\gamma_{s\conc\langle 1\rangle}$ is chosen sufficiently large. 
This will ensure that $\Phi$ is strict order preserving.

\begin{claim*}
$\Phi(t)\subsetneq x_s$ for all $t\subsetneq s$.
\end{claim*}
\begin{proof}
If $s=t\conc\langle 0\rangle^\alpha$, then
$\Phi(t)\subseteq x_t=x_s$.
Otherwise, take the least $\alpha\geq \lh(t)$ with $s(\alpha)=1$.
Then $s\supseteq t\conc\langle 0\rangle^\alpha\conc\langle 1\rangle$.
By \ref{GV two i}
applied to $t\conc\langle 0\rangle^\alpha$
and \ref{GV two ii}, 
we have 
$
\Phi(t)
\subseteq 
\Phi(t\conc\langle 0\rangle^\alpha\conc\langle 1\rangle)
=
(x_{t\conc\langle 0\rangle^\alpha}\restr
\gamma_{t\conc\langle 0\rangle^\alpha\conc\langle 1\rangle})
\subseteq 
x_s.
$
\end{proof}

It remains to show that $\Phi$ is $\bot$-preserving if $\gamma_{u^\smallfrown \langle 1\rangle}$ is chosen sufficiently large for all $u$ of successor length. 
To see this, suppose that $s=t\conc\langle i\rangle$. 
Since $\ptwo$ has to play an injective sequence in $\GV_\kappa(X,x)$,
we have
$x_{t\conc\langle 0\rangle}=x_t
\neq
x_{t\conc\langle 1\rangle}$
by \ref{GV two ii}.
Therefore $\Phi(t\conc\langle 0\rangle):= x_{t\conc\langle 0\rangle}\restr\gamma_{t\conc\langle 0,1\rangle}$
and
$\Phi(t\conc\langle 1\rangle):= x_{t\conc\langle 1\rangle}\restr\gamma_{t\conc\langle 1,1\rangle}$
are incompatible extensions of $\Phi(t)$, if 
$\gamma_{t\conc\langle 0,1\rangle}$ 
and
$\gamma_{t\conc\langle 1,1\rangle}$ 
are sufficiently large. 
This completes the construction. 

Note that we can ensure that $\Phi$ is continuous by letting $\gamma_{s\conc\langle 1\rangle}:=
\bigcup_{t\subsetneq s}\gamma_{t\conc\langle 1\rangle}$ for $\lh(s)\in\Lim$. 
Then $\Phi(s)=\bigcup_{t\subsetneq s}\Phi(t)$ by the previous claim and~\ref{GV two i}.

\begin{claim*} 
$
[\Phi](\pi(u)\conc\langle0\rangle^\kappa)
=
x_{\pi(u)}
$
for all $u\in {}^{<\kappa}\kappa$. 
\end{claim*} 
\begin{proof} 
We have 
$[\Phi](\pi(u)\conc\langle0\rangle^\kappa) = \bigcup_{\alpha<\kappa} \Phi(\pi(u)\conc\langle0\rangle^\alpha)= x_{\pi(u)}$.  
The first equality holds by the definition of $[\Phi]$. 
The last one holds since $\langle \Phi(\pi(u)\conc\langle0\rangle^\alpha) : \alpha<\kappa \rangle$ is a strictly increasing chain of initial segments of $x_{\pi(u)}$ by \ref{GV two i} and since $x_s=x_t$ if $s=t\conc \langle 0\rangle^{\alpha}$. 
\end{proof} 

By the previous claim, $[\Phi](\pi(\emptyset)\conc\langle0\rangle^\kappa)=x_{\emptyset}=x$ and
$[\Phi](\QQ_\kappa)
\subseteq \{x_s:s\in{}^{<\kappa}\kappa\}\subseteq X$. 
Since $\Phi$ is also $\perp$- and strict order preserving, it is as 
required in \ref{GV rect oh}.
\end{proof}

The next theorem follows from the two previous lemmas and Theorem~\ref{theorem: KLW from ODD}. 
It also holds for $\CC^X_i:=\CC^X\cap\dhIkappa$ instead of $\CC^X$ by Remark~\ref{remark: CC and injectivity}.
\todoq{Question/idea: is it possible to get a set $X$ as in (1) whose closure is as in  \cite{VaananenCantorBendixson}*{Theorem~3}?}

\begin{theorem} \
\label{theorem: ODD implies CB1} 
Suppose $X$ is a subset of ${}^\kappa\kappa$.
\begin{enumerate-(1)} 
\item\label{ODD CB1 coloring}
If $\CC^X$ has a $\kappa$-coloring, then $|X|\leq\kappa$ and
 $\pone$ wins $\GV_\kappa(X)$. 
\item\label{ODD CB1 homomorphism}
There exists a continuous homomorphism from $\dhH\kappa$ to $\CC^X$
if and only if 
 $\ptwo$ wins $\GV_\kappa(X)$. 
\end{enumerate-(1)} 
In particular, 
each of 
$\ODD\kappa{\CC^X}$ and $\KLW_\kappa({}^\kappa\kappa,X)$%
\footnote{Note that $\ODD\kappa{\CC^X}$ and $\KLW_\kappa({}^\kappa\kappa,X)$ are equivalent by Theorem~\ref{theorem: KLW from ODD}.}
implies $\cbi(X)$. 
\end{theorem} 

\todog{``NOT $\cbi(X)$ and $\Det(\GV_\kappa(X))$ for some closed set $X$'' is consistent for $\kappa=\omega_1$ by \cite{VaananenCantorBendixson}*{Theorem~3} (which says that
it is consistent that 
there exists an $\omega_1$-scattered closed subset $X$ of ${}^{\omega_1}\omega_1$ of size $\omega_2$).}%
%

\begin{remark}
\label{PSP iff CB1 for Gdelta sets}
Suppose $X$ is a $\Gdelta(\kappa)$ subset of ${}^\kappa\kappa$. 
Then 
$\cbi(X)\Longleftrightarrow\KLW_\kappa({}^\kappa\kappa,X)$
since
the converse of~\ref{ODD CB1 coloring} holds by Corollary~\ref{cor: CC coloring}.
Hence $\cbi(X)\Longleftrightarrow\PSP_\kappa(X)$
by Corollary~\ref{cor: PSP from ODD for the kappa-Baire space 2}~\ref{cor: PSP from ODD for the kappa-Baire space 2 b}.
\end{remark}

It is open whether 
$\cbi(X)\Longrightarrow\KLW_\kappa({}^\kappa\kappa,X)$
holds 
for all subsets $X$ of ${}^\kappa\kappa$,
or equivalently, whether
the converse of~\ref{ODD CB1 coloring} holds.
It is further equivalent to ask whether all $\kappa$-scattered sets of size $\kappa$ are $\Gdelta(\kappa)$ by Corollary~\ref{cor: CC coloring}.
\todol{I added this sentence in this version.}
\todog{``NOT $\cbi(X)$ and $\Det(\GV_\kappa(X))$ for some closed set $X$'' is consistent for $\kappa=\omega_1$ by \cite{VaananenCantorBendixson}*{Theorem~3} (which says that
it is consistent that 
there exists an $\omega_1$-scattered closed subset $X$ of ${}^{\omega_1}\omega_1$ of size $\omega_2$).}%

\begin{remark}
\label{remark: PSP and CB1 for closed sets}
Lemmas \ref{theorem: GV I} and \ref{theorem: GV II}  
generalize the following observations 
from closed subsets to arbitrary subsets of ${}^\kappa\kappa$:
\begin{enumerate-(1)}
\item 
If $X$ is closed and 
$|X|\leq\kappa$, then $\pone$ wins $\GV_\kappa(X)$ \cite{VaananenCantorBendixson}*{Proposition 3}.%
\footnote{Note that if $X$ is closed, 
or more generally $\Gdelta(\kappa)$, 
then $|X|\leq\kappa$ if and only if $\CC^X$ has a $\kappa$-coloring by Corollary~\ref{cor: CC coloring}.}
\item 
If $X$ is closed, then 
$\Ker_\kappa(X)$ is the union of all $\kappa$-perfect subsets of $X$ \cite{VaananenCantorBendixson}*{Lemma~1 \& Proposition~1}.
\end{enumerate-(1)} 
These observations already show $\cbi(X)\Longleftrightarrow\PSP_\kappa(X)$ for closed sets $X$. 
\end{remark}

The next property is an analogue of Cantor--Bendixson analysis:
it ensures that a subset $X$ of~${}^\kappa\kappa$ can be turned into
a $\kappa$-crowded set
by first removing 
a $\kappa$-scattered subset consisting of 
up to $\kappa$ many points.
\begin{definition}
\label{def: cbii}
For any class $\mathcal C$, 
$\cbii(\mathcal C)$ states that the following holds for all subsets $X\in\mathcal C$ of ${}^\kappa\kappa$: 
\index{Cantor--Bendixson dichotomy!second variant!for sets\idf$\cbii(X)$}%
\index{Cantor--Bendixson dichotomy!second variant!for classes\idf$\cbii(\mathcal C)$} 
\begin{equation*}
\cbii(X):\:
X=\Ker_\kappa(X)\cup\Sc_\kappa(X) 
\text{ and } 
|\Sc_\kappa(X)|\leq\kappa. 
\end{equation*}
\end{definition}
Note that
$\cbii(X)$ implies that $\GV_\kappa(X,x)$ is determined for every $x\in X$,%
\footnote{It is easy to see that this implies the determinacy of V\"a\"an\"anen's original games $\GV_\kappa(X,x)$ for all $x\in {}^\kappa\kappa$.}
and hence that $\GV_\kappa(X)$ is determined. 
The next theorem shows that 
$\ODD\kappa\kappa({}^\kappa\kappa) \Longrightarrow \cbii(\pwrset({}^\kappa\kappa))$ and
$\ODD\kappa\kappa({}^\kappa\kappa,\defsetsk) \Longrightarrow \cbii(\defsetsk)$.

\begin{theorem}
\label{prop: cbi iff cbii}
Suppose that $\mathcal C$ is 
a class.
\begin{enumerate-(1)}
\item \label{cbi cbii 1}
$\cbii({\mathcal C}) \Longrightarrow \cbi({\mathcal C})$.
\item\label{cbi cbii 2}
If $X\cap N_t\in{\mathcal C}$ 
for all subsets $X\in{\mathcal C}$ of ${}^\kappa\kappa$ and all $t\in{}^{<\kappa}\kappa$, 
then 
$$\cbi({\mathcal C}) \Longleftrightarrow \cbii({\mathcal C}).$$
In particular,
$\KLW_\kappa({}^\kappa\kappa,\mathcal C)$%
\footnote{Equivalently, $\ODD\kappa{\CC^X}$ for all $X\in\mathcal C$, by Theorem~\ref{theorem: KLW from ODD}.}
implies $\cbii(\mathcal C)$. 
\end{enumerate-(1)}
\end{theorem}
\begin{proof}
For \ref{cbi cbii 1}, let $X\subseteq{}^\kappa\kappa$ and assume
$\cbii(X)$.
To see $\cbi(X)$, suppose that
$\ptwo$ does not win $\GV_\kappa(X)$.
Then
$\Ker_\kappa(X)=\emptyset$. 
By $\cbii(X)$,
$X$ is a $\kappa$-scattered set of size~$\lleq\kappa$,
as required.

Now suppose that $\cbi({\mathcal C})$ holds for some class ${\mathcal C}$ as in \ref{cbi cbii 2}, and take $X\in{\mathcal C}$. To show $\cbii(X)$, let
$S$ denote the set of those $s\in{}^{<\kappa}\kappa$ such that 
$X\cap N_s$ is a $\kappa$-scattered set of size $\leq\kappa$.
Since 
$|\bigcup_{s\in S} (X\cap N_s)|\leq\kappa$
and
$\Sc_\kappa(X)\cap\Ker_\kappa(X)=\emptyset$,
it suffices to show that
\[
X\oldsetminus\Ker_\kappa(X)\subseteq
\bigcup_{s\in S} (X\cap N_s)
\subseteq\Sc_\kappa(X).
\]
To see the first inclusion, suppose that $y\in X\oldsetminus\Ker_\kappa(X)$.
\todoo{typo corrected}
Then $\ptwo$ does not win $\GV_\kappa(X,y)$, so 
there is some $\gamma$ such that
$\ptwo$ does not win $\GV_\kappa(X\cap N_{y\restr\gamma})$.
Since $X\cap N_{y\restr\gamma}\in{\mathcal C}$, we have
$\cbi(X\cap N_{y\restr\gamma})$.
Hence $|X\cap N_{y\restr\gamma}|\leq\kappa$ and
$\pone$ wins $X\cap N_{y\restr\gamma}$, so
$y\restr\gamma\in S$. 
For the second inclusion,
let $y\in\bigcup_{s\in S} (X\cap N_s)$.
Then $y\restr\gamma\in S$ for some $\gamma<\kappa$.
$\pone$ has the following winning strategy in $\GV_\kappa(X,y)$:
$\pone$ first plays $\gamma_0=0$ and $\gamma_1=\gamma$. 
In subsequent rounds, $\pone$ uses their
winning strategy in $\GV_\kappa(X\cap N_{y\restr\gamma})$
to determine their moves in $\GV_\kappa(X,y)$.
Thus, $y\in\Sc_\kappa(X)$.
The last assertion in \ref{cbi cbii 2} now follows immediately from Theorem~\ref{theorem: ODD implies CB1}. 
\end{proof}

The next corollary follows from the results in this subsection and Theorem~\ref{main theorem}.

\begin{corollary}
\label{cor: CB consistency}
The following hold in all $\Col(\kappa,\lle\lambda)$-generic extensions of $V$: 
\begin{enumerate-(1)}
\item\label{cor: CB consistency 1}
$\cbii(\defsetsk)$ 
if $\lambda>\kappa$ is inaccessible in $V$.
\item\label{cor: CB consistency 2}
$\cbii(\pwrset({}^\kappa\kappa))$ 
if $\lambda>\kappa$ is Mahlo in $V$.
\end{enumerate-(1)}
\end{corollary}

We can equivalently replace $\cbii(\mathcal C)$ by $\cbi(\mathcal C)$, where $\mathcal C$ denotes $\defsetsk$ and $\pwrset({}^\kappa\kappa)$, respectively.
In the model in \ref{cor: CB consistency 2}, V\"a\"an\"anen's game is therefore determined for all subsets of ${}^\kappa\kappa$. 
The previous corollary strengthens two results of V\"a\"an\"anen 
\cite{VaananenCantorBendixson}*{Theorem~1 \& Theorem~4}.  
He showed that both
$\cbiiomegaone(\closedsets(\omega_1))$ and 
the statement $\cboomegaone$ defined right below%
\footnote{$\cboomegaone$ follows from  $\cbiomegaone(\powerset({}^{\omega_1}\omega_1))$.}
hold in $\Col(\omega_1,\lle\lambda)$-generic extensions if $\lambda$ is measurable.%
\footnote{In fact, V\"a\"an\"anen obtained these two statements from a 
principle $I(\omega)$ 
which holds after L\'evy-collapsing a measurable cardinal to $\omega_2$ \cite{GalvinJechMagidor}.
$I(\omega)$ is equiconsistent with a measurable cardinal since it implies the existence of a precipitous ideal on $\omega_2$.} 

\begin{remark}
\label{cbii remark}
Galgon showed that 
$\cbii(\closedsets(\kappa))$
can be obtained by the L\'evy collapse of an inaccessible cardinal to $\kappa^+$  
\cite{GalgonThesis}*{Theorem~1.4.5}. 
$\cbii(\closedsets(\kappa))$
is in fact equivalent to 
$\PSP_\kappa(\closedsets(\kappa))$ 
\cite{SzThesis}*{Proposition~2.16}. 
Our proof of Theorem~\ref{prop: cbi iff cbii} generalizes the proof of the latter result. 
\end{remark}

Using the previous corollary and \cite{SzThesis}*{Corollary~4.35}, 
we can lower the consistency strength of 
the dichotomy on the size of complete subhypergraphs (cliques) 
of finite dimensional $\Gdelta(\kappa)$ dihypergraphs on $\analytic(\kappa)$ sets studied in \cite{SzVaananen} 
and its generalization to families of dihypergraphs. 
Given a family $\mathcal H$ of dihypergraphs on a set $X$,
we say that $Y\subseteq X$ is an \emph{$\mathcal H$-clique} if  
$\dhK{\ddim_H}Y\subseteq H$ for each $H\in\mathcal H$, where $\ddim_H$ denotes the arity of $H$.
\index{Dolezal Kubis dichotomy@Dole\v{z}al--Kubi\'s dichotomy\idf $\DK_\kappa(X)$}
\begin{quotation}
$\DK_\kappa(X)$:
If $\mathcal H$ is a set of $\kappa$ many finite dimensional $\Gdelta(\kappa)$ %
dihypergraphs on $X$
and there is an $\mathcal H$-clique of size $\kappa^+$,
then there is a $\kappa$-perfect $\mathcal H$-clique.%
\footnote{%
$\Gdelta(\kappa)$ can equivalently be replaced by open. 
Moreover, $\DK_\kappa(X)$ is a statement about 
certain product-open $\omega$-dihypergraphs, 
as $Y$ is an $\mathcal H$-clique if and only if 
$\dhK{\omega}Y\subseteq \bigcap_{H\in\mathcal H}\proj_{\omega, \ddim_H}^{-1}(H){\restr}X$. }
\end{quotation}
$\DK_\kappa(X)$ extends a dichotomy of Dole\v{z}al and Kubi\'s \cite{DolezalKubis} to uncountable cardinals.${}$\footnote{They proved $\DK_\omega(\analytic)$, extending the special cases for a single $\Gdelta$ graph \cite{ShelahBorelSq} and a single finite dimensional $\Gdelta$ hypergraph \cite{Kubis}.} 
By \cite{SzThesis}*{Corollary~4.35},${}$%
\footnote{In \cite{SzThesis},
$\DK_\kappa(X)$, $\Diamondi\kappa$ and $\cbo$ are denoted by
$\mathrm{PIF}_\kappa(X)$, $\mathscr{DI}_\kappa$
and
$\mathrm{DISP}_\kappa$,
respectively. 
Moreover, 
$\DK_\kappa(X)$ is formulated as a dichotomy about the size of independent sets with respect to $\kappa$ many finite dimensional  $\Fsigma(\kappa)$ dihypergraphs.}
\todon{Footnote number size problem solved (twice) by adding \$\{\}\$.} 
$\DK_\kappa(\analytic(\kappa))$ follows from
$\Diamondi\kappa$ and 
\index{Cantor--Bendixson dichotomy!weak variant\idf$\cbo$} 
\begin{quotation}
$\cbo$:
$\ptwo$ wins $\GV_\kappa(X)$ for all subsets $X$ of ${}^\kappa\kappa$ of size $\kappa^+$.
\end{quotation}
Since $\cbo$ follows from $\cbi(\powerset({}^\kappa\kappa))$, 
the previous corollary yields:
\begin{corollary}
\label{corollary: DK}
$\DK_\kappa(\analytic(\kappa))$
holds in all $\Col(\kappa,\lle\lambda)$-generic extensions of $V$ 
if $\lambda>\kappa$ is a Mahlo cardinal in $V$. 
\end{corollary} 
This strengthens a joint result of V\"a\"an\"anen and the second-listed author \cite{SzVaananen} showing that $\DK_\kappa(\analytic(\kappa))$ holds in $\Col(\kappa,\lle\lambda)$-generic extensions for any measurable cardinal $\lambda>\kappa$.\footnote{They showed that the combination of $\Diamond_\kappa$ and a principle $I^-(\kappa)$ implies $\DK_\kappa(\analytic(\kappa))$. Note that $I^-(\omega_1)$ is the same as $I(\omega)$. 
The consistency of $I^-(\kappa)$ can be proved similarly to that of $I(\omega)$ in \cite{GalvinJechMagidor}.} 
It is open whether one can obtain $\DK_\kappa(\analytic)$ directly from $\ODD\kappa\kappa({}^\kappa\kappa)$ by constructing a suitable dihypergraph. 
Moreover, we have not considered the case of $\lle\kappa$-dimensional dihypergraphs.

\begin{remark}
\label{remark: GV for graphs with modified winning condition}
Suppose $X$ and $Y$ are subsets of ${}^\kappa\kappa$ and $G$ is a digraph on ${}^\kappa\kappa$. 
We consider  
a more general game $\GV^G_\kappa(X,Y)$ 
\index{V\"a\"an\"anen's perfect set game!Z@variant for a graph and!a payoff set\idf$\GV^G_\xi(X,Y)$}%
that is defined by modifying the winning condition in $\GV_\kappa(X)$. 
$\ptwo$ wins a run if and only if 
$(x_\alpha,x_{\alpha+1})\in G$ for all $\alpha<\kappa$%
\footnote{If $G$ is an open digraph, then it is easy to see that we get an equivalent game if we require instead that $(x_\beta,x_\alpha)\in G$ for all $\beta<\alpha<\kappa$.} 
and
$\bigcup_{\alpha<\kappa} x_\alpha\restr\gamma_{\alpha+1} \in Y$. 
We claim that for any open digraph $G$ on ${}^\kappa\kappa$,
$\KLW^G_\kappa(Y,X)$%
\footnote{Equivalently, $\ODD\kappa{\CC^{X,G}\restr Y}$, by Theorem~\ref{theorem: ODD and KLW^H}.}
implies the following.%
\index{Cantor--Bendixson dichotomy!Z@relative to a graph\idf$\cbig G(X,Y)$} 
\begin{quotation}
$\cbig G(X,Y)$: 
Either $G\restr (X\cap Y)$ admits a $\kappa$-coloring and $\pone$ wins $\GV^G_\kappa(X,Y)$,
or $\ptwo$ wins~$\GV^G_\kappa(X,Y)$.
\end{quotation}
This follows from~\ref{remark: GV for graphs with modified winning condition 1} and~\ref{remark: GV for graphs with modified winning condition 2} below.
%
For $Y={}^\kappa\kappa$ and $G=\gK {{}^\kappa\kappa}$, 
$\cbig G(X,Y)$
is equivalent to $\cbi (X)$. 
Moreover, $\cbig G(X,Y)$ relates to $\ODD\kappa{G\restr(X\cap Y)}$ just like $\cbi (X)$ relates to $\PSP_\kappa(X)$. 
In particular, 
they are equivalent 
if $X$ is a $\Gdelta(\kappa)$ subset of ${}^\kappa\kappa$
by~\ref{remark: GV for graphs with modified winning condition 3} below.%
\footnote{The game $\GV^G_\kappa(X):=\GV^G_\kappa(X,{}^\kappa\kappa)$ for open graphs $G$
was studied in \cite{SzThesis}*{Section~3}
with the aim of obtaining games that relate to $\OGD_\kappa(X)$ just like $\GV_\kappa(X)$ relates to $\PSP_\kappa(X)$ for closed sets $X$.
It is denoted by $\GV^1_\kappa(X,R)$ for $R=(X\times X)\oldsetminus G$ there.} 

\begin{enumerate-(1)}
\item\label{remark: GV for graphs with modified winning condition 1}
%
Suppose the first option in $\KLW^G_\kappa(Y,X)$ holds, i.e., $Y$ is the $\kappa$-union of sets $A_\alpha$ 
with no $G$-limit points\footnote{See Definition~\ref{def: H-limit point}.} in $Y$.
We claim that 
$G\restr(X\cap Y)$ admits a $\kappa$-coloring and $\pone$ wins $\GV^G_\kappa(X,Y)$.
The former follows from 
Lemma~\ref{coloring H vs C^XH}~\ref{coloring H vs C^XH 2}~$\Rightarrow$~\ref{coloring H vs C^XH 3}.
The proof of the latter is similar to Lemma~\ref{theorem: GV I}.
Consider the following strategy for $\pone$ in $\GV^G_\kappa(X,Y)$.
In rounds of the form $\alpha+1$, $\pone$ chooses $\gamma_{\alpha+1}$ so that
\begin{enumerate-(i)}
\item\label{GV G one a}
$(x_\beta,y)\in G$
for all $y\in N_{x_\alpha\restr\gamma_{\alpha+1}}$
if $\alpha=\beta+1$ 
and
\todog{Equivalently:
$N_{x_\alpha\restr\gamma_{\alpha+1}}\subseteq G^{x_{\alpha-1}}$}
\item\label{GV G one b}
$(x_\alpha,y)\notin G$ 
for all $y\in A_\alpha\cap N_{x_\alpha\restr\gamma_{\alpha+1}}$.
\todog{Equivalently: $G^{x_\alpha}\restr N_{x_\alpha\restr\gamma_{\alpha+1}}=\emptyset$}
\end{enumerate-(i)}
$\pone$ can achieve~\ref{GV G one a} since $G$ is open and $(x_\beta,x_\alpha)\in G$ if $\alpha=\beta+1$.
$\pone$ can achieve~\ref{GV G one b} since 
$x_\alpha\in X$ is not a $G$-limit point of $A_\alpha$.
Suppose $\ptwo$ wins a run of the game where $\pone$ has used this strategy and let
$y:=\bigcup_{\alpha<\kappa} x_\alpha\restr\gamma_{\alpha+1}$.
Then $y\in Y$, so
$y\in A_\alpha$ for some $\alpha<\kappa$.
Hence
$(x_\alpha,y)\notin G$ by~\ref{GV G one b} since 
$y\in N_{x_\alpha\restr\gamma_{\alpha+1}}$.
However, 
$(x_\alpha,y)\in G$ by~\ref{GV G one a} 
since $y\in N_{x_{\alpha+1}\restr\gamma_{\alpha+2}}$.

\item\label{remark: GV for graphs with modified winning condition 2}
We claim that 
the second option in $\KLW_\kappa^G(Y,X)$ holds if and only if
$\ptwo$ wins $\GV_\kappa^G(X,Y)$.
Recall that the former is equivalent to the existence of
an order homomorphism $\Phi$ for $({}^\kappa\kappa,G)$
such that $[\Phi]$ reduces $(\RR_\kappa,\QQ_\kappa)$ to $(Y,X)$
by Theorem~\ref{theorem: ODD and KLW^H} \ref{theorem: ODD and KLW^H 2}~\ref{theorem: ODD and KLW^H: oh H}$\Leftrightarrow$\ref{theorem: ODD and KLW^H: ch H}.
Suppose first that $\Phi$ is such an order homomorphism.
We define the following strategy for $\ptwo$ in $\GV_\kappa^G(X,Y)$. 
As part of their strategy, $\ptwo$ also constructs a continuous strictly increasing sequence
$\langle s_\alpha : \alpha<\kappa \rangle$ in $\SS_\kappa$ 
with
$\Phi(s_\alpha\conc\langle 0\rangle^\kappa)=x_\alpha$
for all $\alpha<\kappa$,
and a sequence $\langle\delta_{\alpha}:\alpha<\kappa\rangle$  of ordinals below $\kappa$.
Suppose $\pone$ has played $\langle \gamma_\beta : \beta\leq\alpha \rangle$ and $\ptwo$ has played $\langle x_\beta : \beta<\alpha\rangle$ and has constructed
$s_\beta$ for all $\beta<\alpha$ and 
$\delta_{\beta}$ for all $\beta$ with $\beta+1<\alpha$.
If $\alpha\in\Lim$, 
$\ptwo$ lets $s_\alpha:=\bigcup_{\beta<\alpha}s_\beta$. 
If $\alpha=\beta+1$,
$\ptwo$ picks $\delta_\beta<\kappa$ so that
$x_\beta\restr\gamma_\alpha\subseteq\Phi(s_\beta\conc\langle 0\rangle^{\delta_\beta})$
and lets $s_\alpha:=s_\beta\conc\langle 0\rangle^{\delta_\beta}\conc\langle 1\rangle$.
In both cases, $\ptwo$ plays $x_\alpha=[\Phi](s_\alpha\conc\langle 0\rangle^\kappa)$.
Note that $x_\alpha\in X$ since $[\Phi](\QQ_\kappa)\subseteq X$.

Suppose $\langle \gamma_\alpha, x_\alpha : \alpha<\kappa \rangle$ is a run according to this strategy and $\ptwo$ has constructed $s_\alpha$ and $\delta_\alpha$ for all $\alpha<\kappa$. 
To see that $\ptwo$ wins,
it suffices to show that
\begin{enumerate-(i)}
\item\label{GV G two i}
$x_\beta\restr\gamma_{\beta+1}\subseteq x_\alpha$ for all $\beta<\alpha<\kappa$,
\item\label{GV G two ii}
$(x_\beta,x_\alpha)\in G$ for all $\beta<\alpha<\kappa$,
\item\label{GV G two iii}
$y:=\bigcup_{\alpha<\kappa} x_\alpha\restr\gamma_{\alpha+1}$ is in $Y$. 
\end{enumerate-(i)}
In particular, note that \ref{GV G two ii} implies that 
$\langle x_\alpha:\alpha<\kappa\rangle$ is an injective sequence and $(x_\alpha,x_{\alpha+1})\in G$ for all $\alpha<\kappa$.
Moreover, \ref{GV G two i} holds since
$x_\beta\restr\gamma_{\beta+1}\subseteq \Phi(s_\alpha)\subseteq 
[\Phi](s_\alpha\conc\langle 0\rangle^\kappa)= x_\alpha$ for all $\beta<\alpha<\kappa$
by the construction.
For \ref{GV G two ii}, suppose that $\beta<\alpha<\kappa$ and let  
$u:=s_\beta\conc\langle 0\rangle^{\delta_\beta}$.
Then 
$\Phi(u\conc\langle 0\rangle)\subseteq x_\beta$ 
and
since $u\conc\langle 1\rangle = s_{\beta+1}\subseteq s_\alpha$, we also have
$\Phi(u\conc\langle 1\rangle)\subseteq \Phi(s_\alpha)\subseteq x_\alpha$.
Since $\Phi$ is an order homomorphism for $({}^\kappa\kappa,G)$, $(x_\beta,x_\alpha)\in N_{\Phi(u\conc\langle 0\rangle)}\times N_{\Phi(u\conc\langle 0\rangle)}\subseteq G$. 
For~\ref{GV G two iii},
let $z=\bigcup_{\alpha<\kappa} s_\alpha$. Then $z\in\RR_\kappa$ and
$[\Phi](z)=\bigcup_{\alpha<\kappa} \Phi(s_\alpha)=y$.
Hence $y\in [\Phi](\RR_\kappa)\subseteq Y$.

For the converse direction,
suppose that $\tau$ is a winning strategy for $\ptwo$ in $\GV^G_\kappa(X,Y)$.
Construct functions
$\Phi: {}^{<\kappa}2\to T(X)$,
$\langle \gamma_s : s\in \SS_\kappa \rangle$
and $\langle x_s : s\in {}^{<\kappa}\kappa \rangle$ as in the proof of 
Lemma~\ref{theorem: GV II}~\ref{II wins GV}~$\Rightarrow$~\ref{GV rect oh} 
with the additional property that
$N_{\Phi(s\conc \langle 0\rangle)}\times N_{\Phi(s\conc\langle 1\rangle)}\subseteq G$ 
for all $s\in{}^{<\kappa}\kappa$.
This can be guaranteed 
since $G$ is open and
$(x_{s\conc\langle 0\rangle},x_{s\conc\langle 1\rangle})=
(x_s,x_{s\conc\langle 1\rangle})\in G$ for all $s\in{}^{<\kappa}\kappa$.
Then $\Phi$ is a nice reduction for $({}^\kappa\kappa,X)$ 
and moreover, it is an order homomorphism for $({}^\kappa\kappa,G)$ by the additional requirement. 
It thus suffices to show that $[\Phi](\RR_\kappa)\subseteq Y$.
Let $z\in\RR_\kappa$. 
It suffices to show that $[\Phi](z)$ is produced during a run of $\GV^G_\kappa(X,Y)$ where $\ptwo$ uses $\tau$.
Since 
$s\in \SS_\kappa$ for cofinally many initial segments $s$ of $z$ and
$\Phi(s)=x_s\restr\gamma_{s\conc\langle 1\rangle}$ for all $s\in{}^{<\kappa}\kappa$ by the definition of $\Phi$, we have
$$
[\Phi](z)=\bigcup\{\Phi(s):s\subseteq z, s\in \SS_\kappa\}=
\bigcup\{x_s\restr\gamma_{s\conc\langle 1\rangle}
:s\subseteq z, s\in \SS_\kappa\}.
$$
Thus, $[\Phi](z)$ is produced during the run
$\langle \gamma_s, x_s :s\subseteq z, s\in \SS_\kappa\rangle$
of $\GV^Y_\kappa(X)$ where $\ptwo$ uses $\tau$, as required.

\item\label{remark: GV for graphs with modified winning condition 3}
For any $\Gdelta(\kappa)$ subset $X$ of ${}^\kappa\kappa$,
$\cbig G(X,Y)\Longleftrightarrow\KLW_\kappa^G(Y,X)\Longleftrightarrow\ODD\kappa {G\restr X\cap Y}$.
In fact, it suffices that 
$Y\setminus X$ can be separated from $X$ by a union of $\kappa$ many $G$-closed sets.
The first equivalence follows from 
\ref{remark: GV for graphs with modified winning condition 1},
\ref{remark: GV for graphs with modified winning condition 2}
and
Lemma~\ref{coloring H vs C^XH}~\ref{coloring H vs C^XH 3}~$\Rightarrow$~\ref{coloring H vs C^XH 2}.
The second equivalence holds by 
Corollary~\ref{theorem: KLW H and ODD H}.
\end{enumerate-(1)}

One can aim to generalize the preceding arguments to higher dimensions to show that $\KLW_\kappa^H(Y,X)$ implies an analogue of $\cbig G(X,Y)$
for any box-open dihypergraph $H$. 
\todoq{The special case for ${}^\kappa\kappa\times G$ is intended to show that  $\KLW_\kappa^{{}^\kappa\kappa\times G}(Y,X)$ implies an analogue of $\cbi$ with respect to a different variant 
of V\"a\"an\"anen's game for open graphs $G$ that was studied in \cite{SzThesis}*{Section~3} for the special case of closed sets $X$ and $Y={}^\kappa\kappa$. It is denoted by $\GV^2_\kappa(X,R)$ for $R=(X\times X)\oldsetminus G$ there.}
It is further likely that similar ideas as in the proof of Theorem~\ref{prop: cbi iff cbii} show that $\cbig G(X,Y)$ is equivalent to an 
analogue%
\footnote{For any $x\in X$, define the game $\GV^G_\kappa(X,Y,x)$ just like $\GV^G_\kappa(X,Y)$ except that $\ptwo$ must play $x_0=x$ in the first round.
Let $\Sc^G_\kappa(X,Y)$, respectively $\Ker^G_\kappa(X,Y)$, denote the set of those $x\in X$ such that $\pone$, respectively $\ptwo$, wins $\GV^G_\kappa(X,Y,x)$.
The analogue of $\cbii$ would state that
$X$ can be decomposed as the union of $\Sc^G_\kappa(X,Y)$ and $\Ker^G_\kappa(X,Y)$ 
and $G\restr\Sc^G_\kappa(X,Y)$ has a $\kappa$-coloring.}
of $\cbii$ for $\GV^G_\kappa(X,Y)$ for any open graph $G$. 
This analogue would then follow from $\OGD_\kappa(X\cap Y)$ if $X$ is a $\Gdelta(\kappa)$ set,
answering \cite{SzThesis}*{Question~3.57}.
\end{remark}

\subsection{The asymmetric Baire property}
\label{subsection: almost Baire property}

\hypertarget{def: kappa-meager}{}%
\hypertarget{def: Baire property}{}%
A subset $X$
\todol{This paragraph changed: previously, we'd defined ``$\kappa$-meager in $Y$ (where $Y\subseteq{}^\kappa\kappa$) to mean that $X\cap Y$ is $\kappa$-meager (in ${}^\kappa\kappa$). We've replaced this with the usual definition. The two are only equivalent in the case that $Y$ is open, but not for arbitrary subsets in general.}
of a topological space $Y$ is called \emph{$\kappa$-meager in $Y$} 
\index{meager@$\kappa$-meager\idf}
if it is the union of $\kappa$ many nowhere dense subsets of $Y$ and \emph{$\kappa$-comeager} 
\index{comeager@$\kappa$-comeager\idf}
if its complement is $\kappa$-meager. 
$Y$ usually stands for the space ${}^\kappa\kappa$ and is then omitted from the notation. 
Note that for open subsets $Y$ of ${}^\kappa\kappa$, a subset $X$ of $Y$ is $\kappa$-meager in $Y$ if and only if $X\cap Y$ is $\kappa$-meager, and $Y$ is $\kappa$-comeager in $Y$ if and only if $Y\setminus X$ is $\kappa$-meager. 

A subset $A$ of ${}^\kappa\kappa$ has the \emph{$\kappa$-Baire property} 
\index{Baire property@$\kappa$-Baire property\idf}%
if there exists an open subset $U$ of ${}^\kappa\kappa$ such that 
$X\triangle U$ is $\kappa$-meager. 
If $\kappa$ is uncountable, then the $\kappa$-Baire property always fails for some $\kappa$-analytic subset of ${}^\kappa\kappa$:
Halko and Shelah \cite{MR1880900} observed that the \emph{club filter}
\index{club filter \idf $\mathrm{Club}_\kappa$}
$$\mathrm{Club}_\kappa = \{x \in {}^\kappa\kappa: \{\alpha<\kappa: 0<x(\alpha)\} \text{ contains a club}\}$$ 
does not have the $\kappa$-Baire property.
To see this,
suppose first that $ \mathrm{Club}_\kappa$ is $\kappa$-comeager in some basic open set $N_t$. Then there is a sequence $\langle U_\alpha : \alpha<\kappa\rangle$ of open dense subsets of $N_t$ 
such that 
$\bigcap_{\alpha<\kappa}U_\alpha\subseteq \mathrm{Club}_\kappa$. 
We construct a strictly increasing sequence $\langle t_\alpha : \alpha<\kappa \rangle$ in ${}^{<\kappa}\kappa$ with $t_0=t$ as follows. 
If $\alpha=\beta+1$, choose $t_\alpha \supsetneq t_\beta$ with $N_{t_{\alpha}} \subseteq U_\beta$.  
If $\alpha$ is a limit, let $t_{\alpha}:=\bigcup_{\beta<\alpha} t_\beta \cup \{ \langle \alpha, 0 \rangle \}$. 
Let $x:=\bigcup_{\alpha<\kappa} t_\alpha$.
Then $x\in \bigcap_{\alpha<\kappa}U_\alpha$ and $x\notin \mathrm{Club}_\kappa$, a contradiction. 
One can similarly see that $ \mathrm{Club}_\kappa $ is not $\kappa$-meager. 

We consider a variant of the $\kappa$-Baire property from \cite{SchlichtPSPgames}.
Its definition uses the following types of strict order preserving maps.

\begin{definition}[\cite{SchlichtPSPgames}*{Definition 3.1}]
\label{def: dense strict order preserving map}
Let $S\subseteq{}^{<\kappa}\kappa$, and  let $t\in{}^{<\kappa}\kappa$.
\begin{enumerate-(a)}
\item $S$ is \emph{dense above $t$} 
\index{dense!subset of ${}^{<\kappa}\kappa$ above a node\idf}
if for any $u\supseteq t$ 
there exists $s\in S$ with $s\supseteq u$. 
$S$ is \emph{nowhere dense} 
\index{nowhere dense subset of ${}^{<\kappa}\kappa$\idf}
if it is not dense above any $t\in{}^{<\kappa}\kappa$.%
\todog{If $S$ is upwards closed, then $S$ is dense above $t$ if and only if $\bigcup_{s\in S} N_s$ is dense in $N_t$. Thus, $S$ is nowhere dense if and only if $\bigcup_{s\in S} N_s$ is nowhere dense.}
\item A strict order preserving map $\iota:{}^{<\kappa}\kappa\to{}^{<\kappa}\kappa$ is
\emph{dense} 
\index{dense!strict order preserving map\idf}
if for all $t\in{}^{<\kappa}\kappa$, the set
$\{\iota(t\conc\langle \alpha\rangle):\alpha<\kappa\}$ 
is dense above $\iota(t)$.
\end{enumerate-(a)}
\end{definition}

The property of being $\kappa$-comeager
can be characterized via 
such maps:
a subset $X$ of ${}^\kappa\kappa$ is $\kappa$-comeager in a basic open set $N_t$
if and only if there exists a continuous dense strict order preserving map 
$\iota:{}^{<\kappa}\kappa\to{}^{<\kappa}\kappa$
with $\iota(\emptyset)=t$ and
$\ran([\iota])\subseteq X$
\cite{SchlichtPSPgames}*{Lemma~3.2}.

\begin{definition}[\cite{SchlichtPSPgames}*{Definition 3.3}]
\label{def: ABP} 
A subset $X$ of  ${}^\kappa\kappa$ has the 
\index{asymmetric $\kappa$-Baire property for a!set\idf$\ABP_\kappa(X)$} 
\emph{asymmetric $\kappa$-Baire property} $\ABP_\kappa(X)$%
\footnote{$\ABP_\kappa(X)$ was called the \emph{almost Baire property} in \cite{SchlichtPSPgames}. Note that $X$ is $\kappa$-meager if and only if there exists a continuous dense strict order preserving map $\iota:{}^{<\kappa}\kappa\to{}^{<\kappa}\kappa$ with $\iota(\emptyset)=\emptyset$ and $\ran([\iota])\subseteq {}^\kappa\kappa\setminus X$, by \cite{SchlichtPSPgames}*{Lemma~3.2}.  Therefore our definition of $\ABP_\kappa(X)$ is equivalent to the one given in \cite{SchlichtPSPgames}*{Definition~3.3}.}
if
either $X$ is $\kappa$-meager
or there exists a dense strict order preserving map 
$\iota:{}^{<\kappa}\kappa\to{}^{<\kappa}\kappa$ 
with $\ran([\iota])\subseteq X$.
For any class $\mathcal C$, 
\index{asymmetric $\kappa$-Baire property for a!class\idf$\ABP_\kappa(\mathcal C)$} 
let $\ABP_\kappa(\mathcal C)$ state that $\ABP_\kappa(X)$ holds for all subsets $X\in\mathcal C$ of ${}^\kappa\kappa$. 
\end{definition} 

This 
property is equivalent to the determinacy of the Banach-Mazur game of length~$\kappa$ \cite{SchlichtPSPgames}*{Lemma~3.6}. 
Note that $\iota$ is not required to be continuous in the previous definition, so $\ABP_\kappa(\mathcal C)$
is weaker than the $\kappa$-Baire property
for subsets of ${}^\kappa\kappa$ in $\mathcal C$.
The version of $\ABP_\kappa(\mathcal C)$ with a continuous map $\iota$ is equivalent to the statement that each subset in $\mathcal C$ is either $\kappa$-comeager in some basic open set or $\kappa$-meager. 
If $\mathcal C$
is closed under preimages of shift maps $\sigma_t: {}^\kappa\kappa\to {}^\kappa\kappa$ for each $t\in {}^{<\kappa}\kappa$, where $\sigma_t(x)=t^\smallfrown x$, 
then the latter statement is equivalent to the Baire property for all sets in $\mathcal C$.\footnote{To see this, let $U$ be the union of all basic open sets $N_t$ such that $N_t\setminus X$ is $\kappa$-meager. 
Then $U\triangle X$ is $\kappa$-meager.} 
Since continuity is trivial for $\omega$, 
$\ABP_\omega(\mathcal C)$ is equivalent to the Baire property for subsets of ${}^\omega\omega$ in $\mathcal C$. 

the asymmetric $\kappa$-Baire property $\ABP_\kappa(X)$ 
is a special case of $\ODD\kappa\kappa(X,\defsetsk)$.
To see this, we use the hypergraph $\dhD\kappa$ from Definition \ref{def: dhD}. 
Note that $\dhD\kappa\in\defsetsk$ is a box-open hypergraph on its domain ${}^\kappa\kappa$. 

\begin{lemma}
\label{lemma: dhD independence}
A subset $Y$ of ${}^\kappa\kappa$ is $\dhD\kappa$-independent if and only if 
it is nowhere dense. 
\end{lemma}
\begin{proof}
Suppose $Y$ is somewhere dense, i.e., $Y$ is dense in $N_t$ for some $t\in{}^{<\kappa}\kappa$.
Fix $y_u\in N_u\cap Y$ for all $u\supseteq t$.
Then any injective enumeration of the values $y_u$ is a 
hyperedge of $\dhD\kappa\restr Y$.
Conversely, if $\dhD\kappa\restr Y$ is nonempty,
then $Y$ is dense is some basic open set $N_t$.
\end{proof}

The next lemma shows that the property of being an order homomorphism for $(X,\dhD\kappa)$ in inherited to supersets of $X$.

\begin{lemma} 
\label{lemma: dhD and density}
Let $X\subseteq{}^\kappa\kappa$.
\begin{enumerate-(1)} 
\item 
\label{lemma: dhD and density 1}
Let $\langle t_\alpha:\alpha<\kappa\rangle$ be a sequence of elements of  $T(X)$. 
Then $\prod_{\alpha<\kappa} N_{t_\alpha}\cap X \subseteq \dhD \kappa$\footnote{Note that $\prod_{\alpha<\kappa} N_{t_\alpha}\cap X \neq\emptyset$, as $t_\alpha\in T(X)$ for all $\alpha<\kappa$.}  
if and only if 
$\{t_\alpha:\alpha<\kappa\}$ is dense above some $s\in{}^{<\kappa}\kappa$.
\todog{Thus, a subtree $T$ of ${}^{<\kappa}\kappa$ is $\dhD\kappa$-independent if and only if it's nowhere dense (i.e. not dense above any $s\in{}^{<\kappa}\kappa$)}
\item 
\label{lemma: dhD and density 2}
Any order homomorphism $\iota$ for $(X,\dhD\kappa)$ is an order homomorphism for $\dhD\kappa$. 
\todog{Note that $\domh{\dhD\kappa}={}^\kappa\kappa$}
\end{enumerate-(1)} 
\end{lemma}
\begin{proof}
We first show \ref{lemma: dhD and density 1}. 
If $\{t_\alpha:\alpha<\kappa\}$ is dense above $s$, 
then every element of
$\prod_{\alpha<\kappa} N_{t_\alpha}\cap X$ is dense in~$N_s$. 
Conversely, suppose that $\{t_\alpha:\alpha<\kappa\}$ is nowhere dense. 
It suffices to construct a sequence 
$\bar x\in\prod_{\alpha<\kappa} N_{t_\alpha}\cap X$ which is nowhere dense in ${}^\kappa\kappa$. 
Let $I$ denote the set of $\alpha<\kappa$ such that $T(X)$ is dense above some $t\supseteq t_\alpha$, or equivalently, $\succ{t}\subseteq T(X)$ as $T(X)$ is downward closed.%
\footnote{Recall that $\succ{t}= \{ u\in {}^{<\kappa}\kappa : t\subsetneq u \}$.} 
Then $\kappa\,\setminus\, I$ is precisely the set of those $\alpha<\kappa$ such that $({}^{<\kappa}\kappa)\setminus T(X)$ is dense above $t_\alpha$. 

We first claim that any $\bar x\in\prod_{\alpha\in \kappa \setminus I} N_{t_\alpha}\cap X$ is nowhere dense in ${}^\kappa\kappa$. 
Otherwise $\bar{x}$ is dense in $N_t$ for some $t\in {}^{<\kappa}\kappa$. 
Pick some $\alpha\in \kappa\,\setminus\, I$ with $x_\alpha \in N_t$. 
Since $\ran(x)\subseteq X$, $T(X)$ is dense above $t\cup t_\alpha$. 
But then $\alpha\in I$. 

It therefore suffices to find a sequence 
$\bar y\in\prod_{\alpha\in I} N_{t_\alpha}\cap X$ that is nowhere dense in ${}^\kappa\kappa$. 
To this end, we construct a sequence
$\langle u_\alpha:\alpha\in I\rangle$
such that the following hold for all $\alpha\in I$:
\begin{enumerate-(i)}
\item\label{dhD density 1} 
$u_\alpha\in T(X)$ and $t_\alpha\subsetneq u_\alpha$.
\item\label{dhD density 2} 
$u_\alpha\not\subseteq t_\gamma$ for all $\gamma<\kappa$.
\item\label{dhD density 3} 
For all $\beta\in I\cap \alpha$, either
$u_\alpha\perp u_\beta$ or $u_\alpha=u_\beta$.
\end{enumerate-(i)} 
Suppose that $\alpha\in I$ and $u_\beta$ has been constructed for all $\beta\in I\cap \alpha$.
If there exists
$\beta\in I\cap \alpha$ with
$u_\beta\supsetneq t_\alpha$, then let $u_\alpha:=u_\beta$.
Otherwise, we have $t_\alpha\perp u_\beta$ for all $\beta<\alpha$ by \ref{dhD density 2}.
As $\alpha\in I$, there exists $s\supseteq t_\alpha$ with $\succ{s}\subseteq T(X) $. 
Since $\{t_\gamma:\gamma<\kappa\}$ is nowhere dense, there exists 
$u_\alpha\supsetneq s$ in $T(X)$ 
such that $u_\alpha\not\subseteq t_\gamma$ for all $\gamma<\kappa$.
Then $u_\alpha\supsetneq t_\alpha$, so
$u_\alpha\perp u_\beta$ for all $\beta\in I\cap \alpha$.

Since $u_\alpha\in T(X)$ for all $\alpha<\kappa$, we may choose a sequence
 $\bar y\in\prod_{\alpha\in I} N_{u_\alpha}\cap X \subseteq \prod_{\alpha\in I} N_{t_\alpha}\cap X$ 
such that $y_\alpha=y_\beta$ whenever $u_\alpha=u_\beta$.
Since the $u_\alpha$'s are pairwise incomparable, 
$\bar y$ is nowhere dense in ${}^\kappa\kappa$. 

\ref{lemma: dhD and density 1} shows that the property  $\prod_{\alpha<\kappa} N_{t_\alpha}\cap X \subseteq \dhD \kappa$ is inherited to supersets of $X$. 
\ref{lemma: dhD and density 2} follows. 
\end{proof}

\begin{lemma}\ 
\label{dense sop and order homomorphisms}
\begin{enumerate-(1)}
\item\label{dense sop -> oh}
Any dense strict order preserving map $\iota:{}^{<\kappa}\kappa\to{}^{<\kappa}\kappa$ 
is an order homomorphism for $\dhD\kappa$. 
\item\label{oh -> dense sop}
If there exists an order homomorphism $\iota$ for $\dhD\kappa$, then there exists a dense strict order preserving map $\theta:{}^{<\kappa}\kappa\to{}^{<\kappa}\kappa$
with $\ran([\theta])\subseteq \ran([\iota])$.
\todog{Note that $\theta$ may not be continuous, even if $\iota$ was.
Uncomment for the example and proof
}
\end{enumerate-(1)}
\end{lemma}
\begin{proof}
\ref{dense sop -> oh} is clear.${}$\footnote{This is the easy direction of Lemma \ref{lemma: dhD and density}.} 
\todon{The symbols \$\{\}\$ were added just before the footnote to make the number appear in the right size.}
For \ref{oh -> dense sop},
suppose there is an order homomorphism $\iota$ for $\dhD \kappa$.
By Lemma~\ref{lemma: dhD and density} \ref{lemma: dhD and density 1} there exists a function $s: {}^{<\kappa}\kappa \to {}^{<\kappa}\kappa$ such that for all $t\in{}^{<\kappa}\kappa$, 
$\langle\iota(t\conc\langle\alpha\rangle):\alpha<\kappa\rangle$
is dense above $s(t)$.
Since $\iota$ is strict order preserving, $s(t)\supseteq\iota(t)$.

We define a 
continuous strict order preserving map $e:{}^{<\kappa}\kappa\to{}^{<\kappa}\kappa$ such that 
$\theta:=s\comp e$ is dense and strict order preserving. 
We define $e(t)$ by recursion on $\lh(t)$.
Let $e(\emptyset):=\emptyset$. If $\lh(t)\in \Lim$, then let $e(t)=\bigcup_{u\subsetneq t} e(u)$.
Suppose $e(t)$ has been defined.
Let $\theta(t):=s(e(t))$.
Let $\langle e(t\conc \langle\beta\rangle):\beta<\kappa\rangle$ be an injective enumeration of 
the set of those $e(t)\conc\langle\alpha\rangle$ such that
$\iota(e(t)\conc\langle\alpha\rangle)\supsetneq \theta(t)$.
It is clear from the construction that $\theta$ is dense and strict order preserving. 

Since
$\iota(e(t))\subseteq s(e(t))=\theta(t)\subseteq \iota\big(e(t\conc\langle\alpha\rangle)\big)$
for all $t\in{}^{<\kappa}\kappa$ and all $\alpha<\kappa$,
we have $[\theta]=[\iota]\comp[e]$. 
Therefore $\ran([\theta])\subseteq\ran([\iota])$.
\end{proof}

\begin{theorem} 
\label{theorem: ABP from ODD}
Suppose $X$ is a subset of ${}^\kappa\kappa$.
\begin{enumerate-(1)}
\item\label{ABP from ODD 1}
$\dhD \kappa{\restr}X$ admits a $\kappa$-coloring if and only if $X$ is $\kappa$-meager.
\todoq{If $\dhD\kappa\restr X$ admits a $\kappa$-coloring, then there is a continuous homomorphism from $\dhH\kappa$ to $\dhD\kappa\restr({}^\kappa\kappa\setminus X)$. Is this true for any other dihypergraphs $H$? For any other ones from our applications?
\\
Proof: Since $X$ is $\kappa$-meager, then there is a (continuous) dense strict order preserving map $\iota$ with $\ran([\iota])\subseteq {}^\kappa\kappa\setminus X$
(and $\iota(\emptyset)=\emptyset$). 
Then $\iota$ is also an order homomorphism for $\dhD\kappa$ by the previous lemma and hence for $({}^\kappa\kappa\setminus X,\dhD\kappa)$.}
\item\label{ABP from ODD 2}
There exists a continuous homomorphism from $\dhH\kappa$ to $\dhD \kappa{\restr}X$
if and only if there exists a dense strict order preserving map 
$\iota:{}^{<\kappa}\kappa\to{}^{<\kappa}\kappa$ with $\ran([\iota])\subseteq X$.
\end{enumerate-(1)}
Thus, $\ABP_\kappa(X)$ is equivalent to $\ODD\kappa{\dhD\kappa{\restr}X}$.
\end{theorem}
\begin{proof}
\ref{ABP from ODD 1} holds since a subset of ${}^\kappa\kappa$ is  $\dhD \kappa$-independent if and only if it is nowhere dense by Lemma~\ref{lemma: dhD independence}.
 
For \ref{ABP from ODD 2}, suppose there exists a continuous homomorphism from $\dhH\kappa$ to $\dhD \kappa{\restr}X$. 
We obtain an order homomorphism $\iota$ for $(X,\dhD\kappa)$ by Lemma~\ref{homomorphisms and order preserving maps}. 
$\iota$ is also an order homomorphism for $\dhD\kappa$ by Lemma \ref{lemma: dhD and density}. 
We then obtain the required map from Lemma~\ref{dense sop and order homomorphisms} \ref{oh -> dense sop}. 
The converse holds by Lemma~\ref{dense sop and order homomorphisms} \ref{dense sop -> oh}. 
\end{proof}

\begin{remark} 
In Theorem \ref{theorem: ABP from ODD},  $\dhD\kappa$ can be equivalently replaced by the hypergraph
$\dhD\kappa^-$ studied in Remark \ref{remark Dminus}. 
This follows from Corollary~\ref{cor: ODD subsequences}, 
since 
$\dhD \kappa^-\equivf \dhD \kappa$ and both are box-open on ${}^\kappa\kappa$.
\end{remark} 

Theorem~\ref{main theorem} and the previous theorem
yield a new proof of the following result from \cite{SchlichtPSPgames}. 

\begin{corollary}
$\ABP_\kappa(\defsetsk)$
holds in all $\Col(\kappa,\lle\lambda)$-generic extensions of $V$ if $\lambda>\kappa$ is inaccessible in $V$.
\end{corollary}

\subsection{An analogue of the Jayne--Rogers theorem}
\label{subsection: Jayne Rogers}

We will derive an analogue of Jayne and Rogers' characterization of 
${\bf \Delta}^0_2$-measurable functions \cite{JayneRogers1982}*{Theorem~5} 
from the open dihypergraph dichotomy. 
They proved the following principle for functions with analytic domain 
in the countable setting. 

Suppose throughout this subsection that $X$ is a subset of ${}^\kappa\kappa$
and $f:X\to {}^\kappa\kappa$ is a function. 
We say that $f$ is
\emph{$\kappa$-continuous with closed pieces} 
\index{function!continuous with closed pieces@$\kappa$-continuous with closed pieces\idf}%
if $X$ is the union of $\kappa$ many 
relatively closed subsets $C$
such that $f\restr C$ is continuous.

\begin{definition}
\label{def: JR} 
\index{Jayne--Rogers dichotomy for a!function\idf$\JR_\kappa^f$}%
\index{Jayne--Rogers dichotomy for a!set\idf$\JR_\kappa(X)$}%
\index{Jayne--Rogers dichotomy for a!class\idf$\JR_\kappa(\mathcal C)$}%
Let $\JR_\kappa^f$ denote the statement that the following conditions are equivalent: 
\begin{enumerate-(1)}
\item 
$f$ is ${\bf \Delta}^0_2(\kappa)$-measurable.\footnote{See the end of Subsection~\ref{preliminaries: notation} for the definition. Equivalently, $f$ is ${\bf \Pi}^0_2(\kappa)$-measurable.}
\todog{uncomment for the proof of the claim in the footnote 
}
\item 
$f$ is $\kappa$-continuous with closed pieces. 
\end{enumerate-(1)}
$\JR_\kappa(X)$ states that $\JR_\kappa^g$ holds for all $g:X\to{}^\kappa\kappa$.
$\JR_\kappa(\mathcal C)$ states that $\JR_\kappa(X)$ holds for all subsets $X\in\mathcal C$ of ${}^\kappa\kappa$, where $\mathcal C$ is any class. 
\end{definition}

We will obtain $\JR_\kappa^f$ as a special case of 
$\ODD\kappa\kappa(\graph(f))$, where
$\graph(f)$ denotes the \emph{graph} of $f$.%
\footnote{I.e., $\graph(f)$ consists of those pairs $(x,y)$ in $X\times{}^\kappa\kappa$ such that $f(x)=y$.}
\index{function!graph of\idf$G(f)$}%
\index{graph!of a function\idf$G(f)$}%
If $X=\dom(f)$ is in $\defsetsk$, 
then $\ODD\kappa\kappa(\graph(f),\defsetsk)$ suffices.
Here, we mean the variant of  Definition~\ref{def: ODD for classes} 
for the underlying space ${}^\kappa\kappa\times{}^\kappa\kappa$ instead of ${}^\kappa\kappa$.
Since there is a homeomorphism $k$ from
${}^\kappa\kappa\times{}^\kappa\kappa$
onto
${}^\kappa\kappa$, 
$\ODD\kappa\kappa(\graph(f),\defsetsk)$
is equivalent to $\ODD\kappa\kappa(k(\graph(f)),\defsetsk)$.

Our proof is similar to that of \cite{CarroyMiller}*{Theorem 13 \& Proposition 18}.
Recall from Subsection~\ref{subsection: original KLW} that $\CC^X$ denotes the set of all 
convergent 
$\kappa$-sequences $\bar x$ in ${}^\kappa\kappa$ 
with $\lim_{\alpha<\kappa}(\bar x)\in X\setminus \ran(\bar x)$. 
We will work with the $\kappa$-hypergraph 
$\JJ^f$ on
\index{hypergraph!Jayne--Rogers!standard\idf$\JJ^f$}
${}^\kappa\kappa\times{}^\kappa\kappa$ 
consisting of those sequences
$\langle (x_\alpha,y_\alpha):\alpha<\kappa\rangle$ such that 
$\bar x=\langle x_\alpha:\alpha<\kappa\rangle \in \CC^{X}$ and $f(\lim_{\alpha<\kappa}(\bar x))$ is not in the closure of  ${\{y_\alpha:\alpha<\kappa\}}$. 
We call $\JJ^f$ a \emph{Jayne--Rogers hypergraph} due to its introduction in \cite{CarroyMiller}*{Theorem~13} for $\kappa=\omega$ for a proof of the Jayne--Rogers theorem \cite{JayneRogers1982}*{Theorem~5}. 
Note that $\JJ^f\restr\graph(f)$ 
consists of those sequences
that project to $\CC^X$ and witness discontinuity of $f$.
%
%
Clearly $\JJ^f$ is a box-open dihypergraph on ${}^\kappa\kappa\times{}^\kappa\kappa$,
and $\JJ^f\in\defsetsk$ when $f\in\defsetsk$.

Let $\proj_k$ 
\index{projection of $\ddim$-sequences onto!the $i^{\mathrm{th}}$ coordinate \idf$\proj_i$}%
denote the projection onto the $k^{\mathrm{th}}$ coordinate for $k\in\{0,1\}$. 

\begin{lemma}
\label{lemma: JR independence}
Suppose $A\subseteq\graph(f)$ and $C$ is the closure of $\proj_0(A)$ in $X$.
Then $A$ is $\JJ^f$-independent if and only if 
$f\restr C$ is continuous.
\end{lemma}
\begin{proof}
First, suppose that  $A$ is $\JJ^f$-independent. 
Assuming $f\restr C$ is not continuous at $y\in C$, 
take an injective sequence
$\bar x=\langle x_\alpha:\alpha<\kappa\rangle$ in $\proj_0(A)$ converging to $y$
such that $\langle f(x_\alpha):\alpha<\kappa\rangle$ does not converge to $f(y)$.
Then there exists a subsequence
$\langle x_{\alpha_i}:i<\kappa\rangle$ of $\bar x$ and some $\gamma<\kappa$ 
with 
$f(x_{\alpha_i}) \notin N_{f(y)\restr\gamma}$
for all $i<\kappa$.
Since 
$f(y)\notin\closure{\{f(x_{\alpha_i}):i<\kappa\}}$,
the sequence
$\langle (x_{\alpha_i}, f(x_{\alpha_i})):i<\kappa\rangle$ 
contradicts the fact that 
$A$ is $\JJ^f$-independent.

Conversely, suppose that $f\restr C$ is continuous.
It suffices to show for any sequence 
$\bar x=\langle x_\alpha:\alpha<\kappa\rangle$ in $C$ with $\bar x \in \CC^X$ 
that 
$\langle (x_\alpha,f(x_\alpha)):\alpha<\kappa\rangle$ is not in $\JJ^f$.
But this follows from the fact that
$f(\lim_{\alpha<\kappa} \bar x)=\lim_{\alpha<\kappa} (\langle f(x_\alpha):\alpha<\kappa\rangle)$
by the continuity of $f\restr C$.
\end{proof}

The following lemmas 
will be used to characterize the existence of continuous homomorphisms from $\dhH \kappa$ to $\JJ^f\restr\graph(f)$.
Recall from Subsection~\ref{subsection: closure properties} that for a strict order preserving map $\iota: {}^{<\kappa}\kappa\to {}^{<\kappa}\kappa$, $\Lim_t^{\iota}$ denotes the union of all 
limit points of sets of the form 
$\{[\iota](x_\alpha) : \alpha<\kappa\}$ 
where 
$x_\alpha\in N_{t\conc\langle\alpha\rangle}\cap X$
for all $\alpha<\kappa$. 
The next lemma shows that $\iota$ can be ``thinned out''
to a map $\iota \circ e$ such that $\Lim_t^{\iota\circ e}$ contains at most one element for any $t\in {}^{<\kappa}\kappa$. 

\begin{lemma} 
\label{lemma: small Lim} 
If $\iota:{}^{<\kappa}\kappa\to{}^{<\kappa}\kappa$ 
is strict order preserving, then there exists a 
continuous strict $\wedge$-homomorphism%
\footnote{See Definition~\ref{def: wedge-homomorphism}. Note that a map $e: {}^{<\kappa}\kappa\to {}^{<\kappa}\kappa$ is a strict $\wedge$-homomorphism if and only if it is strict order preserving with $e(t\conc\langle\alpha\rangle) \wedge e(t\conc\langle\beta\rangle)= e(t)$ for all $t\in{}^{<\kappa}\kappa$ and all $\alpha<\beta<\kappa$.}
$e:{}^{<\kappa}\kappa\to{}^{<\kappa}\kappa$ 
with 
$|\Lim_t^{\iota\comp e}|\leq 1$ for all $t\in {}^{<\kappa}\kappa$. 
\end{lemma} 
\begin{proof}
We construct $e(t)$ by recursion on $t$. 
Let $e(\emptyset)=\emptyset$ and if $t$ is a limit, let $e(t):=\bigcup_{s\subsetneq t} e(s)$. 
Now suppose that $e(t)$ has been constructed. 
If there exists no $\bar{x}=\langle x_\alpha : \alpha<\kappa\rangle$ in $\prod_{\alpha<\kappa} N_{ e(t)\conc\langle \alpha\rangle}$ such that $\langle [\iota](x_\alpha) : \alpha<\kappa\rangle$ has a convergent subsequence,
then let $e(t\conc \langle \alpha \rangle):=e(t)\conc \langle \alpha \rangle$ for all $\alpha<\kappa$. 
Then $\Lim^{\iota\comp e}_t=\Lim^\iota_{e(t)}=\emptyset$.
Otherwise fix a strictly increasing sequence $\langle \alpha_i : i<\kappa\rangle$ of ordinals below $\kappa$ and $\langle x_i : i<\kappa \rangle$ in $\prod_{i<\kappa} N_{e(t)\conc\langle \alpha_i\rangle}$ such that $\langle [\iota](x_i) : i<\kappa\rangle$ converges to some $y$.
For each $i<\kappa$, take $e(t\conc \langle i \rangle)$ with 
$e(t)\conc \langle \alpha_i \rangle \subsetneq e(t\conc \langle i \rangle) \subsetneq x_i$ 
and $\lh\big(e(t\conc \langle i \rangle)\big)\geq i$.
Then $\big\langle \iota\big(e(t\conc \langle i \rangle)\big) : i<\kappa \big\rangle$ converges to~$y$,%
\footnote{Equivalently, all sequences in  
$\prod_{i<\kappa} N_{\iota(e(t)\conc\langle \alpha_i\rangle)}$ converge to $y$; 
see the paragraph before Lemma~\ref{lemma: CC convergence}.}
so $\Lim^{\iota\comp e}_t=\{y\}$.
\end{proof}

We next extend the notion of order homomorphisms 
to dihypergraphs on ${}^\kappa\kappa\times{}^\kappa\kappa$.
Let
${}^{<\kappa}\kappa\otimes{}^{<\kappa}\kappa:=
\bigcup_{\alpha<\kappa}({}^\alpha\kappa\times{}^\alpha\kappa)$.
Given a strict order preserving map
$\theta:{}^{<\kappa}\ddim\to({}^{<\kappa}\kappa\otimes{}^{<\kappa}\kappa)$
where $\ddim\leq\kappa$,
let $\theta_i:=\proj_i\comp\theta$ 
for $i\in\{0,1\}$ and
define $[\theta]: {}^\kappa\ddim\to {}^\kappa\kappa\times{}^\kappa\kappa$ by letting $[\theta](x):=([\theta_0](x),[\theta_1](x))$. 
The definition of \emph{continuity} for $\theta$ is analogous to Definition~\ref{def: order preserving}~\ref{cont def}.%
\footnote{I.e., 
$\theta(t)=(\bigcup_{s\subsetneq t}\theta_0(s),$ $\bigcup_{s\subsetneq t}\theta_1(s))$ for all $t\in{}^\kappa\kappa$ with $\lh(t)\in\Lim$.}

In the following definition and lemma, 
we assume that 
$\ddim\geq 2$ is an ordinal,
$Y$ is a subset of 
${}^\kappa\kappa\times{}^\kappa\kappa$
and $H$ is a 
$\ddim$-dihypergraph on ${}^\kappa\kappa\times{}^\kappa\kappa$.

\begin{definition}
\label{def: order hom for kappa^kappa times kappa^kappa}
We say that a strict order preserving map 
$\theta:{}^{<\kappa}\ddim\to({}^{<\kappa}\kappa\otimes{}^{<\kappa}\kappa)$
is an \emph{order homomorphism} for $(Y,H)$
\index{order homomorphism for!dihypergraphs on ${}^\kappa\kappa\times{}^\kappa\kappa$\idf} 
if $\ran([\theta])\subseteq Y$
and 
for all $t\in {}^{<\kappa}\ddim$, 
\[
\prod_{\alpha<\ddim} 
(N_{\theta_0(t\conc\langle \alpha\rangle)}
\times
N_{\theta_1(t\conc\langle \alpha\rangle)}) \cap Y 
\subseteq H.
\]
\end{definition}

The next lemma is an analogoue of Lemma~\ref{homomorphisms and order preserving maps} for dihypergraphs on ${}^\kappa\kappa\times {}^\kappa\kappa$. 

\begin{lemma} \ 
\label{hop - version 2}
\begin{enumerate-(1)}
\item\label{hop 2 - version 2} 
If $\theta$ is an order homomorphism for $(Y,H)$, then 
$[\theta]$ is a continuous homomorphism from $\dhH\ddim$ to $H\restr Y$. 
\todog{This is a verbatim analogue of  Lemma~\ref{homomorphisms and order preserving maps}. Weaker versions of both (1) and (2) would suffice, but it might be easier to read the verbatim analogue}
\item\label{hop 1 - version 2} 
If $H\restr Y$ is box-open on $Y$ and there exists
a continuous homomorphism $h$
from $\dhH\ddim$ to~$H\restr Y$, then
there exists a continuous
order homomorphism $\theta$ for $(Y,H)$
and a continuous strict $\wedge$-homomorphism
$e:{}^{<\kappa}\ddim\to{}^{<\kappa}\ddim$
with $[\theta]=h\comp [e]$.
\end{enumerate-(1)}
\end{lemma}
\begin{proof} 
Let $k$ denote the natural coding of elements $(u,v)$ of
${}^{\leq\kappa}\kappa\otimes{}^{\leq\kappa}\kappa$%
\footnote{We write
${}^{\leq\kappa}\kappa\otimes{}^{\leq\kappa}\kappa:=
({}^{<\kappa}\kappa\otimes{}^{<\kappa}\kappa)
\cup
({}^{\kappa}\kappa\times{}^{\kappa}\kappa)$.} 
as elements 
$k(u,v)$ of ${}^{\leq\kappa}\kappa$
such that
$\lh(k(u,v))=2\cdot\lh(u)=2\cdot\lh(v)$.%
\footnote{That is, for all $\alpha<\lh(u)$, let
$k(u,v)(2\cdot\alpha)=u(\alpha)$  
and let
$k(u,v)(2\cdot\alpha+1)=v(\alpha)$.}
$k$ induces the $\ddim$-dihypergraph $\bar H:=k^\ddim(H)$ on 
${}^\kappa\kappa$. 
Moreover, $\theta: {}^{<\kappa}\ddim\to {}^{<\kappa}\kappa\times {}^{<\kappa}\kappa$ is an order homomorphism for $(Y,H)$ 
if and only if $k\comp \theta$ is an 
order homomorphism for $(k(Y),\bar H)$. 
\todog{Use that $(x,y)$ is in
$N_{\theta_0(t\conc\langle \alpha\rangle)}
\times
N_{\theta_1(t\conc\langle \alpha\rangle)}$
if and only if 
$k((x,y))$ is in $N_{(k\comp\theta)(t\conc\langle\alpha\rangle)}$.}

For \ref{hop 2 - version 2}, suppose $\theta$ is an order homomorphism for $(Y,H)$. 
Then $[k\comp\theta]=k\comp[\theta]$ 
\todog{$\big[k\restr{}^{<\kappa}\kappa\big]=[k]\restr{}^\kappa\kappa$}
is a continuous homomorphism from $\dhH\ddim$ to 
$\bar H\restr k(Y)$ 
by Lemma~\ref{homomorphisms and order preserving maps}. 
Hence $[\theta]$ is a continuous homomorphism from $\dhH\ddim$ to 
$H\restr Y$.

For \ref{hop 1 - version 2}, suppose $h$ is a continuous homomorphism from $\dhH\ddim$ to $H\restr Y$.
Then $k\comp h$ is a continuous homomorphism from $\dhH\ddim$ to $\bar H\restr k(Y)$.
Since $\bar H\restr k(Y)$ is a box-open dihypergraph on $k(Y)$,
there exists 
a continuous order homomorphism $\iota$ for $(k(Y),\bar H)$ 
and a continuous strict $\wedge$-homomorphism
$e:{}^{<\kappa}\ddim\to{}^{<\kappa}\ddim$
with $[\iota]=k\comp h\comp [e]$
by Lemma~\ref{homomorphisms and order preserving maps}. 
We can assume that $\lh(\iota(t))$ is even for all $t\in{}^\kappa\kappa$ by Remark~\ref{remark: order homomorphisms even length}. 
Then $\theta:=k^{-1}\comp \iota$ is an order homomorphism for $(Y,H)$
with $[\theta]=k^{-1}\comp [\iota]=h\comp [e]$.
$\theta$ is continuous since both $k^{-1}$ and $\iota$ are continuous.
\end{proof}

\begin{theorem}
\label{theorem: JR main theorem} \ 
\begin{enumerate-(1)}
\item\label{JR coloring}
$\JJ^f\restr\graph(f)$  admits a $\kappa$-coloring if and only if $f$ is $\kappa$-continuous with closed pieces. 
\item\label{JR homomorphism} 
There exists a continuous homomorphism from $\dhH\kappa$ to 
$\JJ^f\restr\graph(f)$ 
if and only if 
there exists a continuous map $g:{}^\kappa 2\to X$ with the following properties:
\begin{enumerate-(a)}
\item
\label{JR homomorphism 1}
$f\comp g$ is a reduction of $\RR_\kappa$ to a closed subset of ${}^\kappa\kappa$.\footnote{A reduction $h: {}^\kappa2\to {}^\kappa\kappa$ of a set $A$ to a set $B$ is a function with $A=h^{-1}(B)$. 
$\RR_\kappa,\,\QQ_\kappa$ and $\SS_\kappa$ were defined in Subsection~\ref{subsection: original KLW}.}

\item
\label{JR homomorphism 2}
$(f\circ g){\restr}\RR_\kappa$ is continuous. 
\end{enumerate-(a)}
\end{enumerate-(1)}
In \ref{JR homomorphism}, we can equivalently take $g$ to be injective, or even a homeomorphism onto a closed subset of ${}^\kappa\kappa$. 
\end{theorem}
\begin{proof}
For \ref{JR coloring}, suppose first that $\graph(f)=\bigcup_{\alpha<\kappa} A_\alpha$, where each $A_\alpha$ is a 
$\JJ^f$-independent subset of $\graph(f)$. 
For each $\alpha<\kappa$, let $C_\alpha$ be the closure of $\proj_0(A_\alpha)$ in $X$.
Then $f\restr C_\alpha$ is continuous by Lemma~\ref{lemma: JR independence}, and $X=\bigcup_{\alpha<\kappa} C_\alpha$,
so $f$ is $\kappa$-continuous with closed pieces.
Conversely, suppose that the sequence $\langle C_\alpha:\alpha<\kappa\rangle$ witnesses that $X$ is $\kappa$-continuous with closed pieces.
Since each $f\restr C_\alpha$ is continuous,
each of the sets 
$A_\alpha:=\{(x,f(x)): x\in C_\alpha\}$ is $\JJ^f$-independent by  Lemma~\ref{lemma: JR independence}.
Since 
$X=\bigcup_{\alpha<\kappa} C_\alpha$,
we have $\graph(f)=\bigcup_{\alpha<\kappa} A_\alpha$.

For \ref{JR homomorphism}, suppose that $g:{}^\kappa 2\to X$ is a continuous map with \ref{JR homomorphism 1} and \ref{JR homomorphism 2}. 
Recall the definition of $\pi$ in 
Subsection~\ref{subsection: original KLW}. 
Let $h:=(g\circ [\pi], f\circ g \circ [\pi]): {}^\kappa\kappa\to X\times {}^\kappa\kappa$. 
Then $\ran(h)\subseteq\graph(f)$,
and $h$
is continuous  
since $f\comp g$ is continuous on $\RR_\kappa=\ran([\pi])$.

It remains to show that $h$ is a homomorphism from $\dhH\kappa$ to $\JJ^f$. 
To see this, suppose that $\langle y_\alpha : \alpha<\kappa\rangle$ is a hyperedge in $\dhH\kappa$ with $t^\smallfrown\langle\alpha\rangle\subseteq y_\alpha$ for all $\alpha<\kappa$. 
Then $t^\smallfrown \langle 0\rangle^\alpha {}^\smallfrown\langle1\rangle \subseteq [\pi](y_\alpha)$ for all $\alpha<\kappa$, 
so $\lim_{\alpha<\kappa} [\pi](y_\alpha) = \pi(t)^\smallfrown \langle 0\rangle^\kappa$. 
Since $g$ is continuous, $\lim_{\alpha<\kappa} (g\circ[\pi])(y_\alpha)=x_t:=g(\pi(t)^\smallfrown\langle0\rangle^\kappa)$. 
Since $f\circ g$ is a reduction of $\RR_\kappa$, $g(\RR_\kappa)$ and $g(\QQ_\kappa
)$ are disjoint. 
Hence $(g\circ[\pi])(y_\alpha)\neq x_t$ for all $\alpha<\kappa$ and $\langle (g\circ[\pi])(y_\alpha): \alpha<\kappa\rangle \in \CC^X$. 
Suppose that $f\circ g$ is a reduction of $\RR_\kappa$ to the closed set $C$. 
Then $f(x_t)\notin\closure{\{z_\alpha : \alpha<\kappa \}}\subseteq C$, where $z_\alpha := (f\circ g \circ [\pi])(y_\alpha)$. 
Thus $\langle h(y_\alpha) : \alpha<\kappa \rangle \in \JJ^f$ as required. 

For the converse,
suppose that there is a continuous homomorphism from $\dhH\kappa$ to 
$\JJ^f\restr\graph(f)$.
Since $\JJ^f$ is box-open on ${}^\kappa\kappa\times{}^\kappa\kappa$, 
there is an order homomorphism 
$\theta=(\theta_0,\theta_1)$ for 
$({{}^\kappa\kappa\times{}^\kappa\kappa},\,\JJ^f)$ 
with $\ran([\theta])\subseteq \graph(f)$%
\footnote{Equivalently, 
$\ran([\theta_0])\subseteq X$ and
$f\comp[\theta_0]=[\theta_1]$.}
by Lemma~\ref{hop - version 2}.
Then $\theta_0$ is an 
order homomorphism for $({}^\kappa\kappa,\CC^X)$.
Hence 
$\langle \theta_0(t\conc\langle \alpha\rangle):\alpha<\kappa\rangle$
converges to some
$x^{\theta_0}_t\in X$ 
for each $t\in{}^{<\kappa}\kappa$
by Lemma~\ref{lemma: CC convergence}.

We now show that $\theta$ can be assumed to have some additional properties.

\begin{claim*}  
There exists an 
order homomorphism $\psi$ 
for 
$({{}^\kappa\kappa\times{}^\kappa\kappa},\,\JJ^f)$ 
with $\ran([\psi])\subseteq \graph(f)$ 
such that for all $t\in{}^{<\kappa}\kappa$,
$f(x^{\psi_0}_t)$ is not in the closure of $\ran([\psi_1])$.
\end{claim*} 
\begin{proof}
Let $\theta$ be any order homomorphism 
for 
$({{}^\kappa\kappa\times{}^\kappa\kappa},\,\JJ^f)$ 
with $\ran([\theta])\subseteq \graph(f)$. 
We will construct a 
strict $\wedge$-homomorphism
$e:{}^{<\kappa}\kappa\to{}^{<\kappa}\kappa$ such that for $\psi:=\theta\comp e$ we have 
$$\forall t\in{}^{<\kappa}\kappa\ \ f(x^{\psi_0}_{t})\notin\closure{\ran([\psi_1])}.$$ 
This suffices, since
$\psi$ will also be an order homomorphism for
$({{}^\kappa\kappa\times{}^\kappa\kappa},\,\JJ^f)$ 
because
$e$ is a $\wedge$-homomorphism and 
$\JJ^f$ is closed under subsequences and permutations of hyperedges.

\setcounter{case}{0} 

\begin{case}
\label{case 1 jayne rogers} 
$[\theta_1]{\restr}N_{t}$ is constant for some $t\in {}^{<\kappa}\kappa$. 
\end{case} 
In this case, let $y$ denote the unique element of $[\theta_1](N_{t})$. 
Let $e(u):=t\conc u$ for all $u\in {}^{<\kappa}\kappa$. 
For $\psi:= \theta \circ e$, $\ran([\psi_1])$ has the unique element $y$. 
It thus suffices to show $f(x_{u}^{\psi_0})\neq y$ for all $u\in {}^{<\kappa}\kappa$. 
To see this, pick $y_\alpha$ with $u\conc\langle\alpha \rangle \subseteq y_\alpha$ for $\alpha<\kappa$. 
Then 
$f(x^{\psi_0}_u)=f(x^{\theta_0}_{t\conc u})$
\todog{$f(x^{\psi_0}_u)=f(x^{\theta_0}_{t\conc u})$ and $[\psi_1](y_\alpha)=[\theta_1](t\conc y_\alpha)$ extends 
$\theta_1(t\conc u\conc\langle\alpha\rangle)$ for each $\alpha$}
is not in the closure $C$ of 
$\{ [\psi_1](y_\alpha) : \alpha<\kappa \}=
\{[\theta_1](t\conc y_\alpha)\:\alpha<\kappa\}$, since $\theta$ is an order homomorphism 
for $({{}^\kappa\kappa\times{}^\kappa\kappa},\,\JJ^f)$ 
But $C$ has the unique element $y$. 
 
\begin{case}
\label{case 2 jayne rogers}
$|\{f(x_t^{\theta_0}) : s\subseteq t\in {}^{<\kappa}\kappa \} |<\kappa$ for some $s\in {}^{<\kappa}\kappa$ 
and Case \ref{case 1 jayne rogers} does not occur. 
\end{case} 
We first claim that there exists some $u\in {}^{<\kappa}\kappa$ such that 
$A_u:=\{ t\in{}^{<\kappa}\kappa : f(x_{u}^{\theta_0})=f(x_t^{\theta_0})\}$ 
is dense above some $v\in{}^{<\kappa}\kappa$.%
\footnote{See Definition~\ref{def: dense strict order preserving map}.}
Let $\bar{u}=\langle u_\alpha : \alpha<\gamma \rangle$ be a sequence in ${}^{<\kappa}\kappa$ such that $\gamma<\kappa$ and for any $t\in {}^{<\kappa}\kappa$ extending $s$, there exists some $\alpha<\gamma$ with $f(x_t^{\theta_0})=f(x_{u_\alpha}^{\theta_0})$. 
Supposing each $A_{u_\alpha}$ is nowhere dense, we can construct a strictly increasing sequence $\langle s_\alpha : \alpha<\gamma \rangle$ in ${}^{<\kappa}\kappa$ with $s\subseteq s_0$ such that for each $\alpha<\gamma$, there exists no $t\supseteq s_\alpha$ with $f(x_t^{\theta_0})=f(x_{u_\alpha}^{\theta_0})$. 
Then for $t= \bigcup_{\alpha<\gamma} s_\alpha$, we have $f(x_t^{\theta_0})\neq f(x_{u_\alpha}^{\theta_0})$ for all $\alpha<\gamma$. 
But this contradicts the choice of~$\bar{u}$. 

Since $[\theta_1]\restr N_{v}$ is not constant,
there exists some $y\in N_v$ with $[\theta_1](y)\neq f(x^{\theta_0}_u)$.
Take $w\subsetneq y$ with $v\subseteq w$ and $\theta_1(w)\perp f(x^{\theta_0}_u)$.
Since $A_u$ is dense above $w$, we can
construct a strict $\wedge$-homomorphism 
$e: {}^{<\kappa}\kappa \to {}^{<\kappa}\kappa$
with $\ran(e)\subseteq A_u$ 
and $w\subseteq e(\emptyset)$.
Let $\psi:=\theta\circ e$. 
Then 
$x^{\psi_0}_t=x^{\theta_0}_{e(t_0)}$ 
for any $t\in {}^{<\kappa}\kappa$
since $e$ is a $\wedge$-homomorphism, so
$f(x_t^{\psi_0})=f(x^{\theta_0}_u)$ 
since $\ran(e)\subseteq A_u$.
Hence
$f(x^{\psi_0}_t)\notin\closure{\ran([\psi_1])}$
since $\ran([\psi_1])$ is a subset of $N_{\theta_1(w)}$.

\begin{case}
\label{case 3 jayne rogers}
Neither Case \ref{case 1 jayne rogers} nor Case \ref{case 2 jayne rogers} occurs. 
\end{case} 

By Lemma~\ref{lemma: small Lim} for $\theta_1$, 
there exists a strict $\wedge$-homomorphism
$\epsilon:{}^{<\kappa}\kappa\to{}^{<\kappa}\kappa$ 
such that $|\Lim^{\theta_1\comp \epsilon}_t|\leq 1$ for all $t\in{}^{<\kappa}\kappa$.
Since $\JJ^f$ is closed under subsequences and permutations of hyperedges,
$\theta\comp \epsilon$
is an order homomorphism for 
$({{}^\kappa\kappa\times{}^\kappa\kappa},\,\JJ^f)$ 
Thus, 
we may assume that $|\Lim^{\theta_1}_t|\leq 1$ for all $t\in{}^{<\kappa}\kappa$,
by replacing $\theta$ with $\theta\comp \epsilon$ if necessary.
Fix an enumeration $\langle t_\alpha : \alpha<\kappa \rangle$ of ${}^{<\kappa}\kappa$ such that
$t_\alpha\subsetneq t_\beta$ implies $\alpha<\beta$.

\begin{subclaim*} 
\label{claim avoid previous limits} 
There exists a strict $\wedge$-homomorphism
$e: {}^{<\kappa}\kappa \to {}^{<\kappa}\kappa$ such that $\psi:=\theta\comp e$ has the following properties for all $\alpha,\beta<\kappa$: 
\begin{enumerate-(i)} 
\item 
\label{claim avoid previous limits 1} 
$f(x_{t_\beta}^{\psi_0}) \notin \Lim_{t_\alpha}^{\psi_1}$ if $\alpha\leq\beta$. 
\item 
\label{claim avoid previous limits 2} 
$f(x_{t_\alpha}^{\psi_0}) \notin N_{\psi_1(t_\beta)}$ if $\alpha<\beta$. 
\end{enumerate-(i)} 
\end{subclaim*} 
\begin{proof} 
Note that \ref{claim avoid previous limits 1} already holds for $\theta_1$ and $\alpha=\beta$. 
To see this, observe that $f(x^{\theta_0}_{t}) \notin \Lim_{t}^{\theta_1}$ for all $t\in {}^{<\kappa}\kappa$, 
since $\langle\theta_0(t\conc\langle\alpha\rangle):\alpha<\kappa\rangle$ converges to 
$x^{\theta_0}_{t}$
and $\theta$ is an order homomorphism for 
$({{}^\kappa\kappa\times{}^\kappa\kappa},\,\JJ^f)$.
Since this property remains true for $\psi_1=\theta_1\comp e$ for any choice of $e$, 
we only need to ensure \ref{claim avoid previous limits 1} in the case $\alpha<\beta$. 

We first show how to arrange \ref{claim avoid previous limits 1}. 
We define $e(t_\beta)$ by recursion on $\beta$. 
Let $B_\beta:=\bigcup_{\alpha<\beta} \Lim_{e(t_\alpha)}^{\theta_1}$. 
Since $|\Lim_{t_\alpha}^{\theta_1}|\leq 1$ for all $\alpha<\kappa$, we have $|B_\beta|<\kappa$. 
Let 
$u:=e(t_\gamma)\conc\langle\eta\rangle$ if $t_\beta=t_\gamma\conc\langle\eta\rangle$ for some $\gamma<\beta$ and $\eta<\kappa$. 
If $t_\beta$ has limit length, then let 
$u:=\bigcup_{s\subsetneq t_\beta} e(s)$.%
\footnote{This is defined since each $s\subsetneq t_\beta$ is equal to $t_\alpha$ for some $\alpha<\beta$.}
Since $|\{f(x^{\theta_0}_t) : u\subseteq t\in{}^{<\kappa}\kappa \} |\geq \kappa$ by the case assumption, 
we can find some $t$ extending 
$u$
with $f(x^{\theta_0}_t)\notin B_\beta$. 
Then $e(t_\beta):=t$ satisfies \ref{claim avoid previous limits 1}.
\todog{Proof: $f(x^{\psi_0}_{t_\beta})=f(x^{\theta_0}_t)\notin B_\beta\supseteq\bigcup_{\alpha<\kappa}\Lim_{t_\alpha}^{\psi_1}$}
Note that $e$ is a strict $\wedge$-homomorphism.

Assume $\theta$ satisfies \ref{claim avoid previous limits 1}. 
We next show how to arrange \ref{claim avoid previous limits 2}. 
Again, we define $e(t_\beta)$ by recursion on $\beta$. 
We let $e$ be continuous at nodes $t_\beta$ of limit length.
Suppose that $t_\beta$ is a direct successor of $t_\gamma$ with $t_\beta=t_\gamma\conc\langle \eta \rangle$. 
Since $[\theta_1](N_{e(t_\gamma)\conc\langle \eta \rangle})$ has at least $2$ elements by the case assumption, we can pick some proper extension $u_0$ of $e(t_\gamma)\conc\langle \eta \rangle$ such that $\theta_1(u_0)$ is incompatible with $f(x^{\theta_0}_{e(t_0)})$. 
Continuing in the same fashion, we construct a strictly increasing sequence $\langle u_\alpha : \alpha<\beta\rangle$ such that $\theta_1(u_\alpha)$ is incompatible with $f(x^{\theta_0}_{e(t_\alpha)})$ for all $\alpha<\beta$. 
Let  $e(t_\beta) := \bigcup_{\alpha<\beta}u_\alpha$. 
Then \ref{claim avoid previous limits 2} holds for $t_\beta$. 
\todog{Proof: $\psi_1(t_\beta)=\theta_1(e(t_\beta))$ is incomatible with $f(x^{\psi_0}_{t_\alpha})=f(x^{\theta_0}_{e(t_\alpha)})$ for all $\alpha<\beta$.}
Again, $e$ is a strict $\wedge$-homomorphism. 
\end{proof}

Let $\psi:=\theta\comp e$, where $e$ is as in the previous subclaim. 
We show that 
$f(x_{t}^{\psi_0})\notin\closure{\ran([\psi_1])}$ for all $t\in{}^{<\kappa}\kappa$.
Suppose that this fails for some $t\in{}^{<\kappa}\kappa$.
Suppose that $t=t_\alpha$. 
We first show that $f(x_{t}^{\psi_0})\notin\ran([\psi_1])$. Otherwise, 
suppose that  $y\in{}^\kappa\kappa$ with 
$f(x_{t}^{\psi_0})=[\psi_1](y)=\bigcup_{u\subsetneq y}\psi_1(u)$.
Take any $\beta$ with $\alpha<\beta<\kappa$ and $t_\beta\subsetneq y$.
Then 
$f(x_{t}^{\psi_0})\in N_{\psi_1(t_\beta)}$, contradicting \ref{claim avoid previous limits 2}. 
We next show that $f(x_{t}^{\psi_0})\notin\closure{\ran([\psi_1])}$. 
Suppose otherwise. 
Since $\closure{\ran([\psi_1])}=
\ran([\psi_1])\cup\bigcup_{\beta<\kappa}\Lim^{\psi_1}_{t_\beta}$ by Lemma \ref{lemma: [e] closed map} \ref{closure of ran[e] when ddim=kappa}, 
we have $f(x_{t}^{\psi_0})\in \Lim^{\psi_1}_{t_\beta}$ 
for some $\beta<\kappa$. 
Then $\alpha<\beta$ by \ref{claim avoid previous limits 1}. 
Since
$\Lim^{\psi_1}_{t_\beta}\subseteq
\closure{\bigcup_{\eta<\kappa} N_{\psi_1(t_\beta\conc\langle\eta\rangle)}}
\subseteq N_{\psi_1(t_\beta)}$,
we have 
$f(x_{t}^{\psi_0})\in N_{\psi_1(t_\beta)}$. 
But this contradicts~\ref{claim avoid previous limits 2}. 
\end{proof} 

Let $\psi$ be an order homomorphism for 
$({{}^\kappa\kappa\times{}^\kappa\kappa},\,\JJ^f)$ 
as in the previous claim.
Since $\psi_0$ is an 
order homomorphism for 
$({}^\kappa\kappa,\CC^X)$,
there exists 
a nice reduction $\Phi:{}^{<\kappa}2\to{}^{<\kappa}\kappa$ for 
$(\ran([\psi_0]),X)$
and a strict $\wedge$-homomorphism%
\footnote{Note that a strict order preserving map $\epsilon:\SS_\kappa\to{}^{<\kappa}\kappa$  is a strict $\wedge$-homomorphism if and only if $\epsilon(s\conc\langle 0\rangle^\alpha\conc\langle 1\rangle) \wedge \epsilon(s\conc\langle 0\rangle^\beta\conc\langle 1\rangle)=\epsilon(s)$ for all $s\in\SS_\kappa$ and $\alpha<\beta<\kappa$.}
$\epsilon:\SS_\kappa\to{}^{<\kappa}\kappa$
such that 
$\psi_0\comp \epsilon=\Phi\restr\SS_\kappa$
by Lemma~\ref{lemma: nice reduction}.
We will show $g:=[\Phi]:{}^\kappa 2\to {}^\kappa \kappa$ is as required in 
\ref{JR homomorphism}.

Since $\Phi$ is $\perp$- and strict order preserving,
$g$ is a homeomorphism onto a closed subset of ${}^\kappa\kappa$ 
\todog{this is needed because of the last sentence in the theorem}
by Lemma~\ref{perfect subsets and sop maps}~\ref{sop maps -> perfect subsets}.
We next show that $\ran(g)\subseteq X$ and that $(f\circ g){\restr}\RR_\kappa$ is continuous.
It is clear that $g(\QQ_\kappa)\subseteq X$
and $g(\RR_\kappa)\subseteq \ran([\psi_0])$.
Since $\ran([\psi])\subseteq\graph(f)$,
we have
$\ran([\psi_0])\subseteq X$
and hence $\ran(g)\subseteq X$. 
Moreover, note that
$f\comp[\psi_0]=[\psi_1]$ since 
$\ran([\psi])\subseteq\graph(f)$.
Thus $(f\comp g)\restr\RR_\kappa = [\psi_1 \comp \epsilon]$
and hence it is continuous.

We now show that $f\comp g$ is a reduction of $\RR_\kappa$ to a closed subset of ${}^\kappa\kappa$. 
Since $(f\comp g)(\RR_\kappa)\subseteq\ran([\psi_1])$,
it suffices to show that
$(f\comp g)(\QQ_\kappa)$ and $\closure{\ran([\psi_1])}$ are disjoint.
Suppose that $z\in\QQ_\kappa$. Then $z=s\conc\langle 0\rangle^\kappa$ for some $s\in\SS_\kappa$.
For each $\alpha<\kappa$, let $t_\alpha:=s\conc\langle 0\rangle^\alpha\conc \langle 1\rangle$.
We claim that 
$g(z)=x^{\psi_0}_{\epsilon(s)}$.
This suffices, since
$f(x^{\psi_0}_{\epsilon(s)})\notin\closure{\ran([\psi_1])}$ by the choice of $\psi$.
Since $\epsilon$ is a $\wedge$-homomorphism from $\SS_\kappa$ to ${}^{<\kappa}\kappa$,
there exists an injective sequence 
$\langle \eta_\alpha:\alpha<\kappa\rangle$ of ordinals below $\kappa$ such that
$\epsilon(s)\conc\langle \eta_\alpha\rangle\subseteq \epsilon(t_\alpha)$
and
hence $\psi_0(\epsilon(s)\conc\langle \eta_\alpha\rangle)\subseteq \Phi(t_\alpha)$.
On the one hand,
$\langle\Phi(t_\alpha):\alpha<\kappa\rangle$ converges to 
$g(z)=[\Phi](s\conc\langle 0\rangle^\kappa)$.
On the other hand,
$\psi_0(\epsilon(s)\conc\langle \alpha\rangle)$ 
converges to $x^{\psi_0}_{\epsilon(s)}$
and therefore so does
$\psi_0(\epsilon(s)\conc\langle \eta_\alpha\rangle)$.
Since we have 
$\psi_0(\epsilon(s)\conc\langle \eta_\alpha\rangle)\subseteq \Phi(t_\alpha)$,
$g(z)=x^\psi_{\epsilon(s)}$ must hold. 
\end{proof}

We next obtain an analogue of Jayne and Rogers' characterization of ${\bf \Delta}^0_2$-measurable functions \cite{JayneRogers1982}*{Theorem 5}
from the open dihypergraph dichotomy. 

\begin{corollary}
\label{cor: JR}\ 
$\ODD{\kappa}{\JJ^f\restr\graph(f)}$ 
implies $\JR_\kappa^f$. 
\end{corollary}
\begin{proof} 
Note that any $\kappa$-continuous function with closed pieces is ${\bf\Delta}^0_2(\kappa)$-measurable. 
Thus, it suffices to show that the condition in Theorem \ref{theorem: JR main theorem} \ref{JR homomorphism} fails for any ${\bf\Delta}^0_2(\kappa)$-measurable function $f: X\to {}^\kappa\kappa$. 
To see this,
suppose that there exists a continuous map $g:{}^\kappa 2\to X$ such that $f\comp g$ is a reduction of $\RR_\kappa$ to a closed subset of ${}^\kappa\kappa$. 
Then $\RR_\kappa$ would be a ${\bf \Sigma}^0_2(\kappa)$ subset of ${}^\kappa\kappa$. 
\end{proof}

\begin{remark}
\label{remark: injective JJ}
Let $\JJ^f_i$ 
\index{hypergraph!Jayne--Rogers!variant\idf$\JJ^f_i$}
denote the variant of $\JJ^f$ where $\CC^X$ is replaced with the $\kappa$-hypergraph $\CC_i^X:=\CC^X\cap \dhIkappa$ of all injective sequences in $\CC^X$. 
In the previous theorem and corollary,
$\JJ^f$ can be equivalently replaced by $\JJ_i^f$. 
This follows from Corollary~\ref{cor: ODD subsequences}
since $\JJ_i^f$ 
is box-open on ${}^\kappa\kappa$,
$\JJ^f_i\subseteq\JJ^f$
and any hyperedge of $\JJ^f$ has a subsequence in $\JJ^f_i$ and hence $\JJ^f_i\equivf\JJ^f$.
\end{remark}

\begin{corollary} \ 
\label{cor: Jayne Rogers} 
\begin{enumerate-(1)} 
\item 
\label{cor: Jayne Rogers 1} 
$\ODD\kappa\kappa(\graph(f),\defsetsk)\Longrightarrow
\JR_\kappa^f$ 
for any $X\in\defsetsk$ and any $f:X\to{}^\kappa\kappa$. 
\item 
\label{cor: Jayne Rogers 2} 
$\ODD\kappa\kappa(\defsetsk,\defsetsk)\Longrightarrow\JR_\kappa(\defsetsk)$.
\end{enumerate-(1)} 
\end{corollary} 
\begin{proof} 
We prove~\ref{cor: Jayne Rogers 1} and~\ref{cor: Jayne Rogers 2} simultaneously.
Suppose that $X\in\defsetsk$ and $f:X\to{}^\kappa\kappa$.
It suffices to show that $\JR^f_\kappa$ follows from each of
$\ODD\kappa\kappa(\graph(f),\defsetsk)$ 
and $\ODD\kappa\kappa(\defsetsk,\defsetsk)$. 
First suppose that $f\in\defsetsk$.
Then $\JJ^f$ is in $\defsetsk$ 
and we thus have 
$\ODD\kappa\kappa(\graph(f),\defsetsk)\Longrightarrow\ODD\kappa{\JJ^f\restr\graph(f)}\Longrightarrow\JR^f_\kappa$
by the previous corollary.
Since $\graph(f)$ is then in $\defsetsk$, we further have 
$\ODD\kappa\kappa(\defsetsk,\defsetsk)\Longrightarrow\ODD\kappa\kappa(\graph(f),\defsetsk)\Longrightarrow\JR^f_\kappa$.
If on the other hand $f\notin\defsetsk$, then $\JR^f_\kappa$ holds trivially since
all ${\bf\Delta}^0_2(\kappa)$-measurable functions on $X$ and all $\kappa$-continuous functions on $X$ with closed pieces are in $\defsetsk$. 
\end{proof} 

Theorem \ref{main theorem} and the previous corollary yield the following result.

\begin{corollary}
\label{cor: JR 2} 
$\JR_\kappa(\defsetsk)$ holds 
in all $\Col(\kappa,\lle\lambda)$-generic extensions 
if $\lambda>\kappa$ is inaccessible in $V$.
\end{corollary}

\subsection{Some implications}
\label{subsection: comparing applications}

In this subsection, 
we collect implications and separations
between some of the applications of the open dihypergraph dichotomy from the previous subsections such as 
the Kechris--Louveau--Woodin dichotomy, its generalization to arbitrary sets, 
the $\kappa$-perfect set property, versions of the Hurewicz dichotomy 
and the analogue to the Jayne--Rogers theorem.
The known information on implications between the above principles is summarized in the following Figures \ref{figure: various implications weakly compact}--\ref{figure: various implications not pspclosed}.%
\footnote{Blue arrows are clickable and lead to the corresponding results. Recall that $\KLW_\kappa(\mathcal C,\mathcal D)$ states that $\KLW_\kappa(X,Y)$ holds for all \emph{arbitrary} subsets $X\in\mathcal C,\,Y\in\mathcal D$ of~${}^\kappa\kappa$ and $\KLW_\kappa(\mathcal C):=\KLW(\mathcal C,\pwrset({}^\kappa\kappa))$. Moreover, recall that $\THD_\kappa(\mathcal C,\mathcal D)$ states that $\THD_\kappa(X,Y)$ holds for all sets $X\subseteq Y$ with $X\in\mathcal C$ and $Y\in\mathcal D$ and $\THD_\kappa(\mathcal C):=\THD_\kappa(\mathcal C,\{{}^\kappa\kappa\})$.}
Throughout this subsection, 
$\mathcal C$ and $\mathcal D$ denote classes of subsets of ${}^\kappa\kappa$. 
\index{Kechris--Louveau--Woodin dichotomy!B@for disjoint sets!A@for classes\idf$\KLWdis_\kappa(\mathcal C,\mathcal D)$} 
Define $\KLWdis_\kappa(\mathcal C,\mathcal D)$ as the statement that $\KLW_\kappa(X,Y)$ holds for all \emph{disjoint} 
subsets $X\in\mathcal C$ and $Y\in\mathcal D$ of ${}^\kappa\kappa$
and let 
\index{Kechris--Louveau--Woodin dichotomy!B@for disjoint sets!B@for a class\idf$\KLWdis_\kappa(\mathcal C)$} 
$\KLWdis_\kappa(\mathcal C):=\KLWdis_\kappa(\mathcal C,\pwrset({}^\kappa\kappa))$.
$\KLWdis_\kappa(\mathcal C)$
is a direct analogue of the Kechris--Louveau--Woodin dichotomy
for classes $\mathcal C$ of subsets of the $\kappa$-Baire space.

\ \vspace{0 pt}
{
\setlength{\abovecaptionskip}{5 pt}
\begin{figure}[H] 
\begin{tikzpicture}[scale=1.0]
\tikzstyle{theory1}=[draw,rounded corners,scale=.900, minimum width=8.0mm, minimum height=8.0mm,line width=\widthline,VeryDarkBlueNode]
\tikzstyle{theory2}=[draw,rounded corners,scale=.900, minimum width=8.0mm, minimum height=8.0mm,line width=\widthline,GreyBlueNode]
\tikzstyle{info}=[rounded corners, inner sep= 2pt, align=left,scale=0.82]

\tikzstyle{cons}=[dash pattern=on 6pt off  2pt]
\tikzstyle{dotted}=[dash pattern=on 1pt off  2pt]
\tikzstyle{medarrow}=[>=Stealth, line width=\widtharrow, scale=1.8, VeryDarkBlueNode] 
\tikzstyle{infoarrow}=[>=Stealth, scale=1] 

 \node (11) at (0,3.6) [theory1] {{\parbox[c][.6cm]{2.1cm}{$\,\KLW_\kappa(\mathcal{C},\mathcal D)$}}};
 \node (13) at (8,3.6) [theory1] {{\parbox[c][.6cm]{1.5cm}{$\,\PSP_\kappa(\mathcal{C})$}}};
 \node (21) at (0,1.8) [theory1] {{\parbox[c][.6cm]{2.3cm}{$\,\KLWdis_\kappa(\mathcal{C},\mathcal D)$
}}};
 \node (22) at (4,1.8) [theory1] {{\parbox[c][.6cm]{2.15cm}{
$\,\THD_\kappa(\mathcal C,\mathcal D)$
}}};
 \node (23) at (8,1.8) [theory1] {{\parbox[c][.6cm]{1.6cm}{$\,\THD_\kappa(\mathcal C)$}}};
 \node (24) at (11.6,1.8) [theory1] {{\parbox[c][.6cm]{1.38cm}{$\,\HD_\kappa(\mathcal C)$}}};
 \node (33) at (8,0.0) [theory1] {{\parbox[c][.6cm]{1.25cm}{$\,\JR_\kappa(\mathcal C)$}}};

\node (i1) at (2.25,7.20) 
[info, text width=10.0 cm] {
Meaning of the arrows in Figures~\ref{figure: various implications weakly compact}--\ref{figure: various implications not pspclosed}:
};
\node(a1) at (-1.7, 6.67)[info,BlueArrow]{A};
\node(b1) at (2.80, 6.67)[info, text width=9.0 cm]{\textcolor{BlueArrow}{B:}
the implication holds for all classes $\mathcal C,\mathcal D$ 
};
\node(b1b) at (2.10,6.22)
[info, text width=9.6 cm] {
satisfying the assumptions on the arrows
};
\draw[infoarrow,BlueArrow] (a1) edge[->] (b1);
\node(a2) at (-1.70, 5.77)[info,BlueArrow]{A};
\node(b2a) at (2.80, 5.77)[info, text width=9.0 cm]{\textcolor{BlueArrow}{B:}
the implication fails for the class $\mathcal C$ on the
};
\node(b2b) at (2.10,5.32)
[info, text width=9.6 cm] {
arrow and any $\mathcal D\supseteq\Fsigma(\kappa)\cup\Gdelta(\kappa)$
};
\draw[infoarrow,BlueArrow] (a2) edge[->] (b2a);
\draw[infoarrow,BlueArrow] (-1.40,5.680) edge[-] (-1.20,5.870);
%
\node (i3) at (10.0,6.00)
[info, text width=9.0 cm] {
{\bf \color{BlueArrow} solid:} 
the (non-)implication is provable from the corresponding assumptions%
\newline
{\bf \color{BlueArrow} dashed:} 
the (non-)implication is consistent with\\ the corresponding assumptions%
}; 

\node (pointA1) at (6.15,4.10){};
\node (pointA2) at (6.15,3.00){};
\node (pointAA1) at (6.15,2.95){};
\node (pointAA2) at (6.15,2.78){};
\node (pointBA1) at (6.15,2.76){};
\node (pointBA2) at (6.15,2.57){};
\node (pointB1) at (6.15,2.53){};
\node (pointB2) at (6.15,2.05){};
\node (pointC1) at (6.15,1.97){};
\node (pointC2) at (6.15,1.22){};
\node (pointD1) at (6.15,1.15){};
\node (pointD2) at (6.15,0.75){};
\node (pointE1) at (6.15,0.55){};
\node (pointE2) at (6.15,-0.45){};

 \draw[medarrow,dotted, DottedLine]
([yshift=0mm]pointA1) edge ([yshift=0mm]pointA2)
([yshift=0mm]pointAA1) edge ([yshift=0mm]pointAA2)
([yshift=0mm]pointBA1) edge ([yshift=0mm]pointBA2)
([yshift=0mm]pointB1) edge ([yshift=0mm]pointB2)
([yshift=0mm]pointC1) edge ([yshift=0mm]pointC2) 
([yshift=0mm]pointD1) edge ([yshift=0mm]pointD2) 
([yshift=0mm]pointE1) edge ([yshift=0mm]pointE2);

\draw[medarrow,  VeryDarkBlueNode!80] 
%
([yshift=0mm]22.east) edge [->] 
node[pos=0.57,above,inner sep=2 pt]{\tiny{${}^\kappa\kappa\in \mathcal D$}}
([yshift=0mm]23.west)
%
([xshift=0mm]11.south) edge [->] 
node {\hyperref[sep: KLW and KLWdis]{\phantom{xxx}}} 
([xshift=0mm]21.north) 
;

\draw[medarrow, BlueArrow] 
([yshift=0mm]11.east) edge [->]
node {\hyperref[connection btw KLW, KLWdis and PSP]{\phantom{xxx}}} 
node[midway,above,inner sep=2 pt]{\tiny{${}^\kappa\kappa\in \mathcal D$}}
([yshift=0mm]13.west)
%
%
([yshift=0.4mm]21.east) edge [->] 
node {\hyperref[cor: KLW implies THD weakly compact]{\phantom{xxx}}} 
node[midway, above, inner sep=2 pt]{\tiny{$\mathcal C, \mathcal D$ closed}}
([yshift=0.4mm]22.west)
%
([yshift=-0.4mm]21.east) edge [<->]
node {\hyperref[KLW and THD for precompact X]{\phantom{xxx}}} 
node[below, inner sep=4 pt]{\tiny $\mathcal C\, {\subseteq}\,$eventually}
node[below, inner sep=12 pt]{\tiny \phantom{x} bounded}
([yshift=-0.4mm]22.west)
%
([yshift=0mm]23.east) edge [<->]
node {\hyperref[proposition: HD and THD options]{\phantom{xxx}}} 
([yshift=0mm]24.west)
%
([yshift=-1.85mm,xshift=-0.165mm]22.east) edge [->]
node {\hyperref[cor: THD implies JR weakly compact]{\phantom{xxx}}} 
node[midway,below,sloped] {\tiny {$\Gdelta(\kappa){\subseteq}\mathcal C{\subseteq}\mathcal D$}}
node[midway,above,sloped]{\tiny {$\mathcal C$ closed}}
([yshift=1.85mm,xshift=0.145mm]33.west)
; 


\draw[medarrow, BlueArrow] 
([yshift=-1.3mm,xshift=0.03mm]13.west) 
edge [->] 
node[pos=0.50,above,sloped] {\hyperref[remark: separating KLWdis and PSP]{\tiny $\kappa=\omega$}} 
node[pos=0.50,below,sloped,inner sep=3 pt] {\hyperref[remark: separating KLWdis and PSP]{\tiny $\mathcal C=L(\RR)[U]$}} 
([yshift=0.2mm,xshift=4mm]21.north); 
\draw[medarrow, BlueArrow](1.92,1.560) edge (2.04,1.470); 

\draw[medarrow, BlueArrow] 
([yshift=-1.85mm,xshift=0.15 mm]13.west) edge [->] 
node {\hyperref[sep: not PSP to THD omega]{\phantom{xxx}}} 
node[pos=0.47,above, sloped]{\tiny{$\kappa=\omega$}}
node[pos=0.37,below, sloped,inner sep=2 pt]{\tiny{$\mathcal C=L(\RR)[U]$}}
([yshift=1.85mm,xshift=-0.15 mm]22.east); 
\draw[medarrow, BlueArrow](3.37,1.545) edge (3.49,1.465); 
%
%
%
\draw[medarrow, BlueArrow] 
([yshift=-3.0mm,xshift=2.0 mm]11.east) edge [-,line width=8pt,draw=white] ([yshift=3.0mm,xshift=-2.0 mm]22.west) 
([yshift=-1.85mm,xshift=-0.15 mm]11.east) edge [->] 
node {\hyperref[cor: KLW implies THD weakly compact]{\phantom{xxx}}} 
node[pos=0.4,above,sloped,inner sep=2 pt]{\tiny 
$\mathcal C, \mathcal D$ closed
}
([yshift=1.85mm,xshift=0.15 mm]22.west) 
;
\end{tikzpicture}
\caption{Implications in the case that $\kappa$ is weakly compact or $\omega$}
\label{figure: various implications weakly compact}
\end{figure}
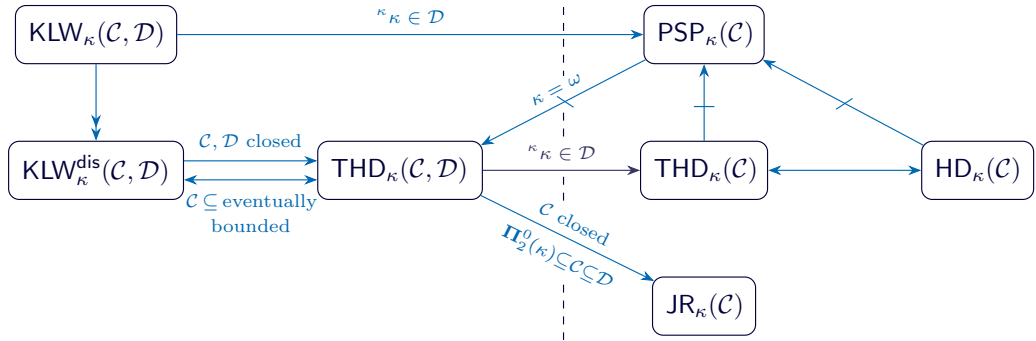
}

{
\setlength{\abovecaptionskip}{5 pt}
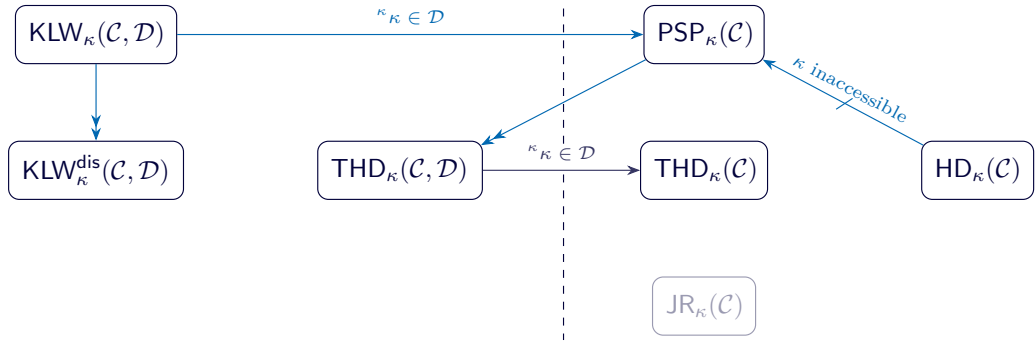
\begin{figure}[H] 
\begin{tikzpicture}[scale=1.0]
\tikzstyle{theory1}=[draw,rounded corners,scale=.900, minimum width=8.0mm, minimum height=8.0mm,line width=\widthline,VeryDarkBlueNode]
\tikzstyle{theory2}=[draw,rounded corners,scale=.900, minimum width=8.0mm, minimum height=8.0mm,line width=\widthline,GreyBlueNode]
\tikzstyle{cons}=[dash pattern=on 6pt off  2pt]
\tikzstyle{dotted}=[dash pattern=on 1pt off  2pt]
\tikzstyle{medarrow}=[>=Stealth, line width=\widtharrow, scale=1.8, VeryDarkBlueNode] 


 \node (11) at (0,3.6) [theory1] {{\parbox[c][.6cm]{2.1cm}{$\,\KLW_\kappa(\mathcal{C},\mathcal D)$}}};
 \node (13) at (8,3.6) [theory1] {{\parbox[c][.6cm]{1.5cm}{$\,\PSP_\kappa(\mathcal{C})$}}};
 \node (21) at (0,1.8) [theory1] {{\parbox[c][.6cm]{2.3cm}{$\,\KLWdis_\kappa(\mathcal{C},\mathcal D)$
}}};
 \node (22) at (4,1.8) [theory1] {{\parbox[c][.6cm]{2.15cm}{
$\,\THD_\kappa(\mathcal C,\mathcal D)$
}}};
 \node (23) at (8,1.8) [theory1] {{\parbox[c][.6cm]{1.6cm}{$\,\THD_\kappa(\mathcal C)$}}};
 \node (24) at (11.6,1.8) [theory1] {{\parbox[c][.6cm]{1.38cm}{$\,\HD_\kappa(\mathcal C)$}}};
 \node (33) at (8,0.0) [theory1] {{\parbox[c][.6cm]{1.25cm}{$\,\JR_\kappa(\mathcal C)$}}};

\node (pointA1) at (6.15,4.08){};
\node (pointA2) at (6.15,-0.45){};
\draw[medarrow,dotted,DottedLine]
([yshift=0mm]pointA1) edge ([yshift=0mm]pointA2);


\draw[medarrow, VeryDarkBlueNode!80] 
%
%
([xshift=0mm]11.south) edge [->] 
node {\hyperref[sep: KLW and KLWdis]{\phantom{xxx}}} 
([xshift=0mm]21.north) 
;

\draw[medarrow, BlueArrow] 
([yshift=0.4mm]11.east) edge [->]
node[midway,above,inner sep=2 pt]{\hyperref[connection btw KLW, KLWdis and PSP]{\phantom{x}{\tiny ${}^\kappa\kappa\in \mathcal D$}\phantom{x}}}
([yshift=0.4mm]13.west)
%
([yshift=-0.4mm]13.west) edge [->]
node[midway,below,inner sep=2 pt]{\hyperref[prop: PSP pre-Kkappa non weakly compact]{\phantom{x}{\tiny$\mathcal C\, {\subseteq}\,$pre-$K_\kappa$}\phantom{x}}}
([yshift=-0.4mm]11.east)
%
([yshift=-1.30mm,xshift=0.03mm]13.east) edge [->]
node{\hyperref[prop: PSP pre-Kkappa non weakly compact]{\phantom{X}}}
node[midway,above,sloped,inner sep=2 pt]{\tiny$\mathcal C\, {\subseteq}\,$pre-$K_\kappa$}
([yshift=1.85mm,xshift=0.15mm]24.west)
%
([yshift=-1.85mm,xshift=-0.15mm]13.east) edge [->,bend left=30]
node{\hyperref[prop: PSP pre-Kkappa non weakly compact]{\phantom{X}}}
node[pos=0.5,right,inner sep=2 pt]{\tiny$\mathcal C\,{\subseteq}$pre-$K_\kappa$}
([yshift=1.85mm,xshift=-0.15mm]33.east)
%
([yshift=-1.85mm,xshift=0.165 mm]13.west) edge [->] 
node {\hyperref[remark: PSP THD non weakly compact]{\phantom{xxx}}} 
([yshift=1.85mm,xshift=-0.145 mm]22.east) 
%
([yshift=0mm]21.east) edge [<->]
node {\hyperref[KLW and THD for precompact X]{\phantom{xxxx}}} 
node[above, inner sep=2 pt]{\tiny $\mathcal C\, {\subseteq}\,$pre-$K_\kappa$}
([yshift=0mm]22.west)
%
([yshift=-1.85mm,xshift=-0.15 mm]11.east) edge [->] 
node {\hyperref[cor: KLW to THD new]{\phantom{xxx}}} 
node[pos=0.5,above,sloped,inner sep=2 pt]{\tiny ${}^\kappa\kappa\in\mathcal D$}
([yshift=1.85mm,xshift=0.15 mm]22.west) 
%
([yshift=0mm]22.east) edge [->]
node {\hyperref[THD strong version non weakly compact]{\phantom{xxx}}} 
([yshift=0mm]23.west)
; 
\end{tikzpicture}
\caption{Implications in the case that $\kappa>\omega$ is not weakly compact
}
\label{figure: various implications not weakly compact}
\end{figure}
}

In Figures~\ref{figure: various implications weakly compact}--\ref{figure: various implications not pspclosed}, we assume that 
$\mathcal C$ and $\mathcal D$ are closed under finite intersections, 
and in addition, $\mathcal D$ is closed under complements.%
\footnote{However, note that these two assumptions together are only relevant 
for the implications $\THD_\kappa(\mathcal C,\mathcal D)\Longrightarrow\JR_\kappa(\mathcal C)$ 
in Figure~\ref{figure: various implications weakly compact}
and $\PSP_\kappa(\mathcal C)\Longrightarrow\KLW_\kappa(\mathcal C,\mathcal D)$ in Figure~\ref{figure: various implications not weakly compact}
and the implications between
$\KLW_\kappa(\mathcal C,\mathcal D)$, $\KLWdis_\kappa(\mathcal C,\mathcal D)$, $\THD_\kappa(\mathcal C, \mathcal D)$ in Figure~\ref{figure: various implications pspclosed}; the latter assumption is also needed for 
the equivalence between $\KLWdis_\kappa(\mathcal C,\mathcal D)$ and  $\THD_\kappa(\mathcal C, \mathcal D)$
in Figures~\ref{figure: various implications weakly compact}--\ref{figure: various implications not weakly compact}.
} 
Some of the implications in the figures use the following additional assumptions that appear on the arrows in abbreviated form. 
Firstly, we call a class $\mathcal C$ of subsets of ${}^\kappa\kappa$ \emph{closed} if it is closed under preimages of continuous functions $f:{}^\kappa\kappa\to{}^\kappa\kappa$ and under images of homeomorphisms from ${}^\kappa\kappa$  onto their image. 
We write \emph{pre-$K_\kappa$} 
\index{preKkappa@pre-$K_\kappa$\idf}
as an abbreviation for the class of subsets of $K_\kappa$ sets; 
recall that for weakly compact $\kappa$ or $\omega$, these are precisely the eventually bounded subsets of ${}^\kappa\kappa$.\footnote{See \cite{LuckeMottoRosSchlichtHurewicz}*{Lemma~2.6}.} 
The notation $\mathcal C \downset \mathcal D$ 
\index{downward closed subset\idf$\downset$} 
will mean that $\mathcal C$ is a downward closed subset of the partial order $(\mathcal D,\subseteq)$.%
\footnote{I.e., for all $X\in\mathcal C$ and $Y\in\mathcal D$ with $Y\subseteq X$, we have $Y\in\mathcal C$.
$\mathcal C \downset \mathcal D$ holds for example if $\mathcal C=\mathcal D\cap\pwrset(X)$ for some $X\subseteq{}^\kappa\kappa$, or more generally, if $\mathcal C=\mathcal D\cap \mathcal I$ for some ideal $\mathcal I\subseteq\pwrset({}^\kappa\kappa)$.}
Finally, writing $L(\RR)[U]$ implies that $L(\RR)$ is a Solovay model and $U$ is a selective ultrafilter generic over $L(\RR)$ for $\pwrset(\omega)/\mathsf{fin}$ as in \cite{di1998perfect}. 
Thus, the non-implications in the next Figure~\ref{figure: various implications weakly compact} use these additional assumptions.

The next Figure~\ref{figure: various implications pspclosed}
\todol{The next two figures 13 and 14 are new. There are also new arrows in the previous two figures 11 and 12. Some of the non-implications in Fig 11-12 were moved to Fig 14}
summarizes additional implications which hold 
under the assumption that the $\kappa$-perfect set property holds for closed sets.%
\footnote{In particular, these implications hold for $\kappa=\omega$.} 
Note that it suffices to assume the $\kappa$-perfect set property for $\kappa$-compact sets%
\footnote{See Propositions~\ref{prop: PSP THD new} and~\ref{remark: PSP THD non weakly compact} below.}
for all implications between $\PSP_\kappa(\mathcal C)$, $\THD_\kappa(\mathcal C)$ and $\THD_\kappa(\mathcal C,\mathcal D)$ 
in Figure~\ref{figure: various implications pspclosed}. 
The idea is that if $\kappa>\omega$ is not weakly compact,  
this assumption implies that 
all $\kappa$-compact sets have size $\kappa$.\footnote{See Proposition~\ref{prop: PSP THD new}.}
Note that the assumption is strictly weaker in the sense that it is consistent with 
the failure of the $\omega_1$-perfect set property for closed sets for $\kappa=\omega_1$ 
relative to an inaccessible cardinal by \cite{LambieStejPSP}*{Theorem~A}
and~\cite{LuckeMottoRosSchlichtHurewicz}*{Lemma~2.2}.%
\footnote{See Remark~\ref{sep: KLW, KLWdis and THD}~\ref{sep: KLWdis and THD}.}

\vspace{5 pt}
{
\setlength{\abovecaptionskip}{5 pt}
\begin{figure}[H] 
\begin{tikzpicture}[scale=1.0]
\tikzstyle{theory1}=[draw,rounded corners,scale=.900, minimum width=8.0mm, minimum height=8.0mm,line width=\widthline,VeryDarkBlueNode]
\tikzstyle{theory2}=[draw,rounded corners,scale=.900, minimum width=8.0mm, minimum height=8.0mm,line width=\widthline,GreyBlueNode]
\tikzstyle{cons}=[dash pattern=on 6pt off  2pt]
\tikzstyle{dotted}=[dash pattern=on 1pt off  2pt]
\tikzstyle{medarrow}=[>=Stealth, line width=\widtharrow, scale=1.8, VeryDarkBlueNode] 


 \node (11) at (0,3.6) [theory1] {{\parbox[c][.6cm]{2.1cm}{$\,\KLW_\kappa(\mathcal{C},\mathcal D)$}}};
 \node (13) at (8,3.6) [theory1] {{\parbox[c][.6cm]{1.5cm}{$\,\PSP_\kappa(\mathcal{C})$}}};
 \node (21) at (0,1.8) [theory1] {{\parbox[c][.6cm]{2.3cm}{$\,\KLWdis_\kappa(\mathcal{C},\mathcal D)$
}}};
 \node (22) at (4,1.8) [theory1] {{\parbox[c][.6cm]{2.15cm}{
$\,\THD_\kappa(\mathcal C,\mathcal D)$
}}};
 \node (23) at (8,1.8) [theory1] {{\parbox[c][.6cm]{1.6cm}{$\,\THD_\kappa(\mathcal C)$}}};
 \node (24) at (11.6,1.8) [theory2] {{\parbox[c][.6cm]{1.38cm}{$\,\HD_\kappa(\mathcal C)$}}};
 \node (33) at (8,0.0) [theory2] {{\parbox[c][.6cm]{1.25cm}{$\,\JR_\kappa(\mathcal C)$}}};

\node (pointA1) at (6.15,4.10){};
\node (pointA2) at (6.15,2.28){};
\node (pointB1) at (6.15,2.25){};
\node (pointB2) at (6.15,1.70){};
\node (pointC1) at (6.15,1.68){};
\node (pointC2) at (6.15,1.58){};
\node (pointD1) at (6.15,1.45){};
\node (pointD2) at (6.15,-0.45){};

\draw[medarrow,dotted,DottedLine]
([yshift=0mm]pointA1) edge ([yshift=0mm]pointA2) 
([yshift=0mm]pointB1) edge ([yshift=0mm]pointB2) 
([yshift=0mm]pointC1) edge ([yshift=0mm]pointC2) 
([yshift=0mm]pointD1) edge ([yshift=0mm]pointD2);

\draw[medarrow, BlueArrow] 
([yshift=0mm]11.south) edge [<->] 
node[midway,left]{\tiny $\mathcal C\downset\mathcal D$}
node {\hyperref[prop: PSP KLW new 3]{\phantom{xxx}}} 
([yshift=0mm]21.north) 
%
([yshift=-0.3mm,xshift=6.2mm]21.north) edge [->] 
node[pos=0.5,sloped] {\hyperref[prop: PSP KLW new]{\phantom{yyyyyyyyy}}} 
node[pos=0.5,sloped,above,inner sep=2 pt] {\hyperref[prop: PSP KLW new]{\tiny $\mathcal C\subseteq\mathcal D$}} 
([yshift=0.0mm,xshift=0.0mm]13.west) 
%
%
%
%
%
([yshift=-2.35mm,xshift=-0.15 mm]11.east) edge [-,line width=17pt,draw=white] ([yshift=1.35mm,xshift=-0.2 mm]22.west) 
([yshift=-1.85mm,xshift=-0.15 mm]11.east) edge [<->] 
node {\hyperref[prop: PSP THD new 3]{\phantom{xxxx}}} 
node[pos=0.45,above,sloped]{\tiny $\mathcal C\,{\subseteq}\,$pre-$K_\kappa$}
node[pos=0.45,below,sloped]{\tiny $\mathcal C\downset\mathcal D$}
([yshift=1.85mm,xshift=0.15 mm]22.west) 
;
%
\draw[medarrow, BlueArrow] 
([yshift=-1.85mm,xshift=0.15 mm]13.west) edge [<->] 
node[pos=0.3,below,sloped,inner sep=0 pt]{\hyperref[remark: PSP THD non weakly compact]{\phantom{xxx}}} 
node[midway,below,sloped,inner sep=1 pt]{\hyperref[remark: PSP THD non weakly compact]{\tiny $\kappa$ not weakly}}
node[midway,below,sloped,inner sep=7 pt]{\tiny compact}
([yshift=1.00mm,xshift=-0.03 mm]22.east) 
%
([yshift=0mm]22.east) edge [<->] 
node{\hyperref[remark: PSP THD non weakly compact]{\phantom{xxx}}} 
node[midway,below,sloped,inner sep=1 pt]{\hyperref[remark: PSP THD non weakly compact]{\tiny $\kappa$ not weakly}}
node[midway,below,sloped,inner sep=7 pt]{\tiny compact}
([yshift=0mm]23.west)
%
([yshift=0mm]23.north) edge [<->] 
node {\hyperref[remark: PSP THD non weakly compact]{\phantom{xxx}}} 
node[right,pos=0.60]{\tiny $\kappa$ not weakly}
node[right,pos=0.40]{\tiny compact}
([yshift=0mm]13.south) 
%
%
([yshift=0mm]21.east) edge [->]
node {\hyperref[cor: THD KLW non-weakly compact new]{\phantom{xxxx}}} 
node[pos=0.45,above, inner sep=2 pt]{\tiny $\mathcal C\subseteq\mathcal D$}
node[pos=0.45,below,sloped,inner sep=1 pt]{{\tiny $\kappa\,$not weakly}}
node[pos=0.45,below,sloped,inner sep=7 pt]{\tiny compact}
([yshift=0mm]22.west)
%
%
%
([yshift=1.85mm,xshift=-0.165 mm]22.east) edge [->] 
node[pos=0.5,above,sloped,inner sep=2 pt]{\hyperref[prop: PSP THD new]{{\tiny $\mathcal C\subseteq\mathcal D$}}}
([yshift=-1.0mm,xshift=0.03 mm]13.west) 
;
\end{tikzpicture}
\caption{Additional implications assuming $\PSP_\kappa(\closedsets(\kappa))$ 
}
\label{figure: various implications pspclosed}
\end{figure}
}

The next Figure~\ref{figure: various implications not pspclosed} shows that several of the implications from the previous Figure~\ref{figure: various implications pspclosed} may fail if the failure of the $\kappa$-perfect set property for closed sets is assumed.
The meaning of the arrows is as in Figure~\ref{figure: various implications weakly compact}, except that it suffices to assume ${}^\kappa\kappa\in\mathcal D$ 
for all non-implications. 
Note that the non-implications from the two versions of the topological Hurewicz dichotomy to $\PSP_\kappa(\mathcal C)$ are provable for $\mathcal C:=K_\kappa$ from the stronger assumption that the $\kappa$-perfect set property for $\kappa$-compact sets fails.%
\footnote{See Remark~\ref{sep: KLW, KLWdis and THD}~\ref{sep: PSP and THD Kkappa}.}

\vspace{5 pt}
{
\setlength{\abovecaptionskip}{5 pt}
\begin{figure}[H] 
\begin{tikzpicture}[scale=1.0]
\tikzstyle{theory1}=[draw,rounded corners,scale=.900, minimum width=8.0mm, minimum height=8.0mm,line width=\widthline,VeryDarkBlueNode]
\tikzstyle{theory2}=[draw,rounded corners,scale=.900, minimum width=8.0mm, minimum height=8.0mm,line width=\widthline,GreyBlueNode]
\tikzstyle{cons}=[dash pattern=on 6pt off  2pt]
\tikzstyle{dotted}=[dash pattern=on 1pt off  2pt]
\tikzstyle{medarrow}=[>=Stealth, line width=\widtharrow, scale=1.8, VeryDarkBlueNode] 


 \node (11) at (0,3.6) [theory1] {{\parbox[c][.6cm]{2.1cm}{$\,\KLW_\kappa(\mathcal{C},\mathcal D)$}}};
 \node (13) at (8,3.6) [theory1] {{\parbox[c][.6cm]{1.5cm}{$\,\PSP_\kappa(\mathcal{C})$}}};
 \node (21) at (0,1.8) [theory1] {{\parbox[c][.6cm]{2.3cm}{$\,\KLWdis_\kappa(\mathcal{C},\mathcal D)$
}}};
 \node (22) at (4,1.8) [theory1] {{\parbox[c][.6cm]{2.15cm}{
$\,\THD_\kappa(\mathcal C,\mathcal D)$
}}};
 \node (23) at (8,1.8) [theory1] {{\parbox[c][.6cm]{1.6cm}{$\,\THD_\kappa(\mathcal C)$}}};
 \node (24) at (11.6,1.8) [theory1] {{\parbox[c][.6cm]{1.38cm}{$\,\HD_\kappa(\mathcal C)$}}};
 \node (33) at (8,0.0) [theory2] {{\parbox[c][.6cm]{1.25cm}{$\,\JR_\kappa(\mathcal C)$}}};

\node (pointA1) at (6.15,4.10){};
\node (pointA2) at (6.15,2.91){};
\node (pointB1) at (6.15,2.75){};
\node (pointB2) at (6.15,2.28){};
\node (pointC1) at (6.15,2.25){};
\node (pointC2) at (6.15,-0.45){};

\draw[medarrow,dotted,DottedLine]
([yshift=0mm]pointA1) edge ([yshift=0mm]pointA2) 
([yshift=0mm]pointB1) edge ([yshift=0mm]pointB2)
([yshift=0mm]pointC1) edge ([yshift=0mm]pointC2);



\draw[medarrow, BlueArrow] 
([xshift=0mm]21.north) edge [->] 
node[pos=0.55,right,inner sep=3 pt]{\tiny{$\mathcal C=\Fsigma(\kappa)$}}
node {\hyperref[sep: KLW and KLWdis]{\phantom{xxx}}} 
([xshift=0mm]11.south);
\draw[medarrow, BlueArrow](-0.060,1.53) edge (0.060,1.46); 
%
\draw[medarrow, BlueArrow] 
([yshift=-0.3mm,xshift=6.2mm]21.north) edge [->] 
node[pos=0.55,sloped,above,inner sep=2 pt]{\tiny{$\mathcal C=\Fsigma(\kappa)$}}
node[pos=0.5,sloped] {\hyperref[sep: KLW and KLWdis]{\phantom{yyyyyyyyy}}} 
([yshift=-0.55mm,xshift=0.0mm]13.west) 
;
\draw[medarrow, BlueArrow](2.02,1.555) edge (2.14,1.475); 
%


\draw[medarrow,cons, BlueArrow] 
(21.east) edge [->] 
node[pos=0.45,below]{\tiny{$\kappa=\omega_1$}}
node[pos=0.5,above,inner sep=2 pt]{\tiny{$\mathcal C=\Fsigma(\kappa)$}}
node {\hyperref[sep: KLWdis and THD]{\phantom{xxx}}} 
(22.west); 
\draw[medarrow, BlueArrow](1.04,1.04) edge (1.16,0.96); 

\draw[medarrow,cons, BlueArrow] 
([yshift=-1.85mm,xshift=-0.15 mm]21.east) edge [->,bend right=22] 
node[pos=0.40,below,sloped,inner sep=2 pt]{\tiny{$\mathcal C=\Fsigma(\kappa)$,}}
node[pos=0.60,below,sloped,inner sep=5 pt]{\tiny{$\kappa=\omega_1$}}
node {\hyperref[sep: KLWdis and THD]{\phantom{xxx}}} 
([yshift=-1.85mm,xshift=0.15 mm]23.west); 
\draw[medarrow, BlueArrow](2.22,0.485) edge (2.34,0.405); 

\draw[medarrow,cons, BlueArrow] 
([yshift=1.85mm,xshift=0.145 mm]24.west) edge [->]  
node[midway,below, sloped,inner sep=3 pt]{\tiny{$\kappa$ inaccessible}}
node[midway,above,sloped,inner sep=2 pt]{\tiny{$\mathcal C=\analytic(\kappa)$}}
node {\hyperref[sep: PSP and HD]{\phantom{xxx}}} 
([yshift=-1.85mm,xshift=-0.15 mm]13.east);
\draw[medarrow, BlueArrow](5.41,1.46) edge (5.53,1.52); 

\draw[medarrow,cons, BlueArrow] 
([yshift=1.85mm,xshift=-0.15 mm]22.east) edge [->]  
%
node[pos=0.45,above,sloped,inner sep=2 pt]{\tiny{$\mathcal C=\analytic(\kappa)$}}
node[midway,below,sloped,inner sep=3 pt]{\hyperref[remark: PSP THD non weakly compact]{\tiny $\kappa$ not weakly}}
node[midway,below,sloped,inner sep=9 pt]{\tiny compact}
node {\hyperref[impl and sep: PSP and THD non weakly compact]{\phantom{xxx}}} 
([yshift=-1.85mm,xshift=0.15 mm]13.west);
\draw[medarrow, BlueArrow](3.37,1.545) edge (3.49,1.465); 

\draw[medarrow,cons, BlueArrow] 
([yshift=0mm]23.north) edge [->] 
node[pos=0.25,right,inner sep=1 pt]{\tiny{$\mathcal C=\analytic(\kappa)$}}
node {\hyperref[sep: PSP and THD]{\phantom{xxx}}} 
([yshift=0mm]13.south); 
\draw[medarrow, BlueArrow](4.38,1.54) edge (4.51,1.46); 
%
\end{tikzpicture}
\caption{Additional non-implications assuming 
$\neg\PSP_\kappa(\closedsets(\kappa))$} 
\label{figure: various implications not pspclosed}
\end{figure}
}

We now proceed with proving the implications in 
the figures, starting with the those in 
Figure~\ref{figure: various implications weakly compact}
and those in Figures~\ref{figure: various implications not weakly compact}--\ref{figure: various implications pspclosed}
which hold without the assumption that $\kappa$ is uncountable and not weakly compact. 
We then prove the remaining implications using this additional assumption. 
The separations in Figure~\ref{figure: various implications not pspclosed}
will be discussed at the end of the section. 

We begin with the implications
between versions of 
the Kechris--Louveau--Woodin dichotomy and
the $\kappa$-perfect set property.  
First, note that $\KLW_\kappa(\mathcal C,\mathcal D)\Longrightarrow\KLWdis_\kappa(\mathcal C,\mathcal D)$ holds by definition.
Recall that $\mathcal C\land\mathcal D$ denotes the class of all sets of the form $X\cap Y$ where $X\in\mathcal C$ and $Y\in\mathcal D$, and $\check{\mathcal C}$ denotes the class of complements of elements of $\mathcal C$.
\todon{pagebreak inserted}

\pagebreak

\begin{proposition}
\label{connection btw KLW, KLWdis and PSP}
Suppose that for all $X\in\mathcal C$, there exists $Y\in\mathcal D$ with $X\subseteq Y$.%
\footnote{For example, this holds if ${}^\kappa\kappa\in\mathcal D$.
Moreover, it would suffice to assume in the proposition that
all elements of $\mathcal C$ can be obtained as the intersection  
of sets $X\in\mathcal C$ and $Y\in\mathcal D$ such that $X\setminus Y$ can be separated from $Y$ by a $\Fsigma(\kappa)$ set.}
\begin{enumerate-(1)}
\item \label{connection btw KLW, KLWdis and PSP 1}
$\KLW_\kappa(\mathcal C,\mathcal D)\Longrightarrow \PSP_\kappa(\mathcal C)$.
\item \label{connection btw KLW, KLWdis and PSP 2}
$\KLW_\kappa(\mathcal C,\mathcal D) \Longleftrightarrow\KLWdis_\kappa(\mathcal C,\mathcal D)\land\PSP_\kappa(\mathcal C)$
if we have $\mathcal C\wedge(\mathcal D\cup\check{\mathcal D})\subseteq\mathcal C$.%
\footnote{This holds for example if $\mathcal C\downset \mathcal D$ and $\mathcal D$ is closed under finite intersections and complements.}
\todog{the assumption in \ref{connection btw KLW, KLWdis and PSP 2}  does not imply the one in the footnote: let $\mathcal D:=\kappa$-Borel and $\mathcal C:=\analytic(\kappa)$} 
\end{enumerate-(1)}
\end{proposition}
\begin{proof}
\ref{connection btw KLW, KLWdis and PSP 1} 
follows immediately from Corollary~\ref{cor: PSP from ODD for the kappa-Baire space}~\ref{cor: PSP from ODD for the kappa-Baire space a}.
For~\ref{connection btw KLW, KLWdis and PSP 2},
the direction from left to right follows from \ref{connection btw KLW, KLWdis and PSP 1}.
The converse follows from Corollary~\ref{cor: PSP from ODD for the kappa-Baire space}~\ref{cor: PSP from ODD for the kappa-Baire space b}.
\end{proof}

In particular, 
 $\KLW_\kappa(\mathcal C,\mathcal C)\Longleftrightarrow\KLWdis_\kappa(\mathcal C,\mathcal C)$ 
if
$\mathcal C$ is the class of definable or $\kappa$-Borel subsets of ${}^\kappa\kappa$
and
$\PSP_\kappa(\mathcal C)$ holds.
The latter can be weakened to  $\PSP_\kappa(\closedsets(\kappa))$ 
by the next proposition.

\begin{proposition}
\label{prop: PSP KLW new}
Assume $\PSP_\kappa(\closedsets(\kappa))$.
\todol{New proposition}
\begin{enumerate-(1)}
\item\label{prop: PSP KLW new 1} 
$\KLW_\kappa(X,{}^\kappa\kappa\setminus X)\Longrightarrow\PSP_\kappa(X)$ for all $X\subseteq{}^\kappa\kappa$.
\item\label{prop: PSP KLW new 2} 
$\KLWdis_\kappa(\mathcal C,\mathcal D)\Longrightarrow\PSP_\kappa(\mathcal C)$ if we have $\check{\mathcal C}\subseteq \mathcal D$. 
\item\label{prop: PSP KLW new 3} 
$\KLWdis_\kappa(\mathcal C,\mathcal D)\Longleftrightarrow\KLW_\kappa(\mathcal C,\mathcal D)$
if we have $\mathcal C\wedge\mathcal D \subseteq \mathcal C\subseteq\check{\mathcal D}$.%
\footnote{If $\mathcal D$ is closed under finite intersections and complements, then it is equivalent to assume $\mathcal C \downset \mathcal D$.} 
\end{enumerate-(1)}
\end{proposition}
\begin{proof}
For~\ref{prop: PSP KLW new 1},
suppose first that $X$ is $\Fsigma(\kappa)$. 
Then $\PSP_\kappa(X)$ holds, 
since
$\PSP_\kappa(\closedsets(\kappa))$ implies $\PSP_\kappa(\Fsigma(\kappa))$.
If $X$ is not $\Fsigma(\kappa)$, then
there exists a continuous homomorphism $f$
from $\dhH\kappa$ to $\CC:=\CC^{{}^\kappa\kappa\setminus X}\restr X$
since $\KLW_\kappa(X,{}^\kappa\kappa\setminus X)$ holds.%
\footnote{Recall that $\KLW_\kappa(X,{}^\kappa\kappa\setminus X)$ is equivalent to $\ODD\kappa \CC$ by Theorem~\ref{theorem: KLW from ODD}.}
Then $\ran(f)\subseteq X$ has a $\kappa$-perfect subset
by Lemma~\ref{lemma: kappa-perfect subdh}~\ref{lemma: kappa-perfect subdh 1}.
\ref{prop: PSP KLW new 2}
follows immediately from~\ref{prop: PSP KLW new 1}.
For~\ref{prop: PSP KLW new 3},
observe that
$\KLW_\kappa(X\setminus Y,Y)\land
\KLW_\kappa(X\cap Y, {}^\kappa\kappa\setminus(X\cap Y))\Longrightarrow \KLW_\kappa(X,Y)$
for all $X,Y\subseteq{}^\kappa\kappa$.
\todog{If $X\in\mathcal C$ and $Y\in\mathcal D$, then $X\cap Y\in\mathcal C$, $Z:={}^\kappa\kappa\setminus(X\cap Y)$ is in $\mathcal D$ and $X\setminus Y=X\cap Z$ is in $\mathcal C$}
To see this, apply~\ref{prop: PSP KLW new 1} for $X\cap Y$
and use
Corollary~\ref{cor: PSP from ODD for the kappa-Baire space}~\ref{cor: PSP from ODD for the kappa-Baire space b}.
\end{proof}

Turning to 
versions of the Hurewicz dichotomy, 
$\THD_\kappa(\mathcal C):=\THD_\kappa(\mathcal C,\{{}^\kappa\kappa\})$
is equivalent to $\HD_\kappa(\mathcal C)$ at weakly compact cardinals $\kappa$ and $\kappa=\omega$ 
by Proposition~\ref{proposition: HD and THD options}.
The equivalence between $\THD_\kappa(\mathcal C,\mathcal D)$ and 
$\KLWdis_\kappa(\mathcal C,\mathcal D)$ 
in Figures~\ref{figure: various implications weakly compact} and~\ref{figure: various implications not weakly compact}
will follow from the implications in the next
Figure~\ref{figure: KLW and THD options}.
A double arrowhead denotes strict implication.

\vspace{10 pt}

{
\setlength{\intextsep}{10 pt}
\setlength{\abovecaptionskip}{7 pt}
\begin{figure}[H]
\begin{tikzpicture}[scale=.3]
\tikzstyle{theory1}=[draw,rounded corners,scale=.5, minimum width=10mm, minimum height=6.5mm,scale=1.8,line width=\widthline,VeryDarkBlueNode]
\tikzstyle{medarrow}=[>=Stealth, line width=\widtharrow, scale=1.8] 
%

     \node (1) at (25,-2) [theory1] {\parbox[c][1cm]{7.5cm}{\thl 
$X$ contains a relatively closed subset of $Y$ which is homeomorphic to ${}^\kappa\kappa$.
}}; 
     \node (2) at (25,-11) [theory1] {\parbox[c][1.5cm]{7.5cm}{\klwl 
There is a homeomorphism 
from ${}^\kappa 2$ onto a closed subset of ${}^\kappa\kappa$
that reduces
$(\RR_\kappa,\QQ_\kappa)$ to $(X,{}^\kappa \kappa\setminus Y)$.
}}; 
     \node (4) at (0,-2) [theory1] {\parbox[c][1cm]{7.5cm}{\ths 
$X$ is contained in a $K_\kappa$ subset of $Y$.
\\ \ 
}}; 
     \node (5) at (0,-11) [theory1] {\parbox[c][1.5cm]{7.5cm}{\klws 
$X$ is contained in a $\Fsigma(\kappa)$ subset $A$ of ${}^\kappa\kappa$ with $A\subseteq Y$.
\\ \ 
}}; 

 \draw[medarrow, BlueArrow]
([xshift=85mm]1.south west) edge[<->] 
node {\hyperref[KLW and THD for precompact X]{\phantom{xxx}}} 
node[midway,right] {\parbox[c]{1.50 cm}{\tiny $X$~contained in a $K_\kappa$ set}} 
([xshift=85mm]2.north west) 
([xshift=85mm]4.south west) edge[<->] 
node {\hyperref[KLW and THD for precompact X]{\phantom{xxx}}} 
node[midway,right] {\parbox[c]{1.50 cm}{\tiny $X$~contained in a $K_\kappa$ set}} 
([xshift=85mm]5.north west);

 \draw[medarrow, VeryDarkBlueNode] 
([xshift=45mm]2.north west) edge[->>] ([xshift=45mm]1.south west) 
([xshift=45mm]4.south west) edge[->>] ([xshift=45mm]5.north west); 
\end{tikzpicture}
\caption{The options in $\THD_\kappa(X,Y)$ and $\KLW_\kappa(X,{}^\kappa\kappa\setminus Y)$ for $X\subseteq Y$}
\label{figure: KLW and THD options}
\end{figure}
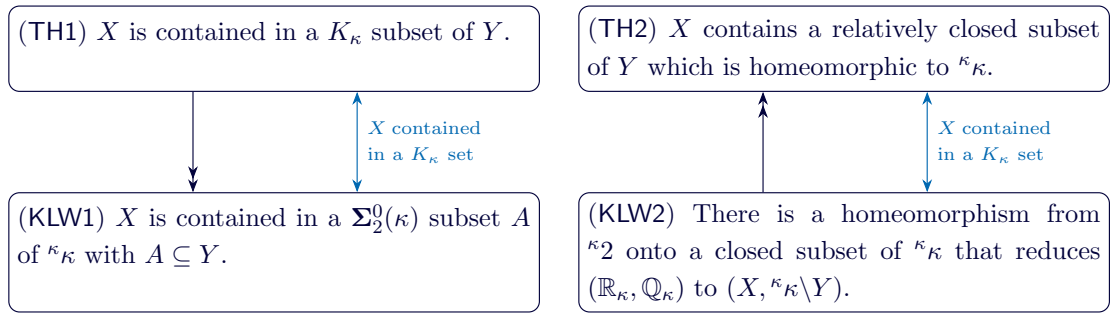
}

It is clear that the implications $\ths$ $\Longrightarrow$ $\klws$
and $\klwl$ $\Longrightarrow$ $\thl$ hold 
for all subsets $X\subseteq Y$ of ${}^\kappa\kappa$.
The converse implications hold  for subsets $X$ of $K_\kappa$ sets by the next proposition. 
However, they clearly fail 
for any closed homeomorphic copy $X$ of~${}^\kappa\kappa$,
and in particular, for $X:={}^\kappa 2$ if $\kappa>\omega$ is not weakly compact by~\cite{hung1973spaces}*{Theorem~1}.

\begin{proposition}
\label{KLW and THD for precompact X}
If $X\subseteq Y\subseteq{}^\kappa\kappa$ and $X$ 
is contained in a $K_\kappa$ subset of ${}^\kappa\kappa$,
then the corresponding options in Figure~\ref{figure: KLW and THD options} are equivalent.
In particular, $\THD_\kappa(X,Y)\Longleftrightarrow\KLW_\kappa(X,{}^\kappa\kappa\setminus Y)$. 
\end{proposition}
\begin{proof}
Let $\HH:=\dhth Y\restr X$ and $\CC:=\CC^{{}^\kappa\kappa\setminus Y}_i\restr X=(\CC^{{}^\kappa\kappa\setminus Y}\cap \dhIkappa)\restr X$.%
\footnote{Recall the definitions of $\dhIkappa$, $\dhth Y$ and $\CC^Y$ from
Definition~\ref{def: dhth} and Subsection~\ref{subsection: original KLW}.}
Write $X$ as the $\kappa$-union of $\kappa$-precompact sets $X_\alpha$.
We show that $\CC\restr X_\alpha\equivf\HH\restr X_\alpha$%
\footnote{See Definition~\ref{def: H-full}.} 
for each $\alpha<\kappa$.
This implies $\HH\equivf\CC$ and hence 
the corresponding options in $\ODD\kappa\HH$ and $\ODD\kappa\CC$ are equivalent by Corollary~\ref{cor: ODD subsequences}. 
The claim then follows from 
Theorems~\ref{theorem: THD and ODD dhth}, \ref{theorem: KLW from ODD} and Remark~\ref{remark: CC and injectivity}. 
Since $\CC\subseteq \HH$,
it suffices to see that
any hyperedge $\bar x$ of $\HH\restr X_\alpha$ contains a subsequence in $\CC$. 
Note that $\bar x$ is injective.
Since $\bar x$ is 
contained in a $\kappa$-precompact set, 
it must contain a convergent subsequence $\bar y$. 
As $\bar x\in \HH$, the limit of $\bar y$ cannot be an element of $Y\cup\ran(\bar y)$ 
and hence $\bar y\in \CC$, as required.
\end{proof}

In particular, 
for weakly compact cardinals $\kappa$ and $\kappa=\omega$,
$\KLWdis_\kappa(\mathcal C,\mathcal D)\Longleftrightarrow
\THD_\kappa(\mathcal C,\mathcal D)$ holds 
if every element of $\mathcal C$ is eventually bounded
and $\mathcal D$ is closed under complements.%
\footnote{Recall that every eventually bounded subset of ${}^\kappa\kappa$ is contained in a $K_\kappa$ set in this case by  \cite{LuckeMottoRosSchlichtHurewicz}*{Lemma~2.6}.}

\begin{corollary}
\label{cor: KLW implies THD weakly compact}
Suppose $\kappa$ is weakly compact or $\kappa=\omega$,
$\mathcal C$ and $\mathcal D$ are closed under homeomorphic images 
and $\mathcal D$ is closed under complements.
Then 
$$\KLWdis_\kappa(\mathcal C,\mathcal D)
\Longrightarrow
\THD_\kappa(\mathcal C,\mathcal D).$$ 
\end{corollary}
\begin{proof}
Let $h$ be a homeomorphism from ${}^\kappa\kappa$ onto a subset of ${}^\kappa 2$.%
\footnote{E.g., let $h:=[\pi]$ where $\pi$ is the map defined in Subsection~\ref{subsection: original KLW}.}
Take subsets $X\in\mathcal C$ and $Y\in\mathcal D$ of ${}^\kappa\kappa$. 
Then $\KLW_\kappa(h(X),{{}^\kappa\kappa\setminus h(Y)})$ is equivalent to $\THD_\kappa(h(X),h(Y))$ by the previous proposition.
The latter is equivalent to  $\THD_\kappa(X,Y)$. 
\end{proof}


Note that the previous corollary also holds 
for non-weakly compact cardinals $\kappa>\omega$
if 
$\PSP_\kappa(\closedsets(\kappa))$
is assumed and $\check{\mathcal C}\subseteq \mathcal D$ by Corollary~\ref{cor: THD KLW non-weakly compact new} below.
However, its failure is consistent with the failure of 
 $\PSP_\kappa(\closedsets(\kappa))$
by Remark~\ref{sep: KLW, KLWdis and THD} below.
\todol{It is no longer open whether the previous corollary holds without the assumption of weak compactness.}

The following two results are mostly interesting for weakly compact $\kappa$ or $\omega$, since we will mention stronger versions for non-weakly compact uncountable $\kappa$ afterwards. 
The next corollary is an analogoue of Proposition~\ref{connection btw KLW, KLWdis and PSP}~\ref{connection btw KLW, KLWdis and PSP 2} for the topological Hurewicz dichotomy for subsets of $K_\kappa$ sets. 

\begin{corollary}\ 
\label{cor: KLW, KLWdis and THD}
Suppose that every element of $\mathcal C$ is contained in a $K_\kappa$ set, $\mathcal D$ is closed under complements and $\mathcal C\wedge\mathcal D\subseteq\mathcal C$.%
\footnote{This holds for example if $\mathcal C\downset\mathcal D$ and $\mathcal D$ is closed under finite intersections and complements.} 
Then
\todol{This corollary changed: I stated the two directions seperately since \ref{cor: KLW, KLWdis and THD 1} will be used later with the assumption in \ref{cor: KLW, KLWdis and THD 2} omitted.I removed the sentence preceding the corollary.}
$\THD_\kappa(\mathcal C,\mathcal D)\land \PSP_\kappa(\mathcal C) \Longrightarrow \KLW_\kappa(\mathcal C,\mathcal D)$. 
\end{corollary}
\begin{proof}
It suffices to show that
$\THD_\kappa(X\cap Y,Y)\land \PSP_\kappa(X\setminus Y)\Longrightarrow \KLW_\kappa(X,{}^\kappa\kappa\setminus Y)$ holds for all $X\in\mathcal C$ and $Y\in\mathcal D$.
To see this, apply Proposition~\ref{KLW and THD for precompact X} for $X\cap Y$ and $Y$ 
and Corollary~\ref{cor: PSP from ODD for the kappa-Baire space}~\ref{cor: PSP from ODD for the kappa-Baire space b} 
for $X\cap Y$ and ${}^\kappa\kappa\setminus Y$.
\end{proof}

Note that if $\kappa$ is uncountable and not weakly compact, then a stronger version of 
the previous corollary holds, 
since Proposition~\ref{prop: PSP pre-Kkappa non weakly compact} below shows that
$\PSP_\kappa(\mathcal C)$ and $\mathcal C\subseteq\text{pre-}K_\kappa$ imply that every element of $\mathcal C$ has size $\kappa$.

Note that the converse of the previous corollary holds assuming ${}^\kappa\kappa\in\mathcal D$ by 
Propositions~\ref{connection btw KLW, KLWdis and PSP}~\ref{connection btw KLW, KLWdis and PSP 1} and~\ref{KLW and THD for precompact X}.
Moreover, the assumption 
the assumption $\mathcal C\subseteq\text{pre-}K_\kappa$ can be omitted for this direction if 
$\kappa$ is not weakly compact,
since $\PSP_\kappa(X)$ implies $\THD_\kappa(X,Y)$ for any superset $Y$ of $X$ by Proposition~\ref{remark: PSP THD non weakly compact equiv}. 
If $\kappa$ is weakly compact or $\omega$,
$\mathcal C$ and $\mathcal D$ are closed under homeomorphic images and $\mathcal D$ is closed under complements, then this assumption can also be omitted 
by Corollary~\ref{cor: KLW implies THD weakly compact}.

Thus, the previous corollary implies that 
$\KLW_\kappa(\mathcal C,\mathcal C)\Longleftrightarrow\THD_\kappa(\mathcal C,\mathcal C)$ 
if $\kappa$ is weakly compact or $\omega$, $\mathcal C$ is the class of definable or $\kappa$-Borel subsets of ${}^\kappa 2$ and $\PSP_\kappa(\mathcal C)$ holds. 
The latter can be weakened to the $\kappa$-perfect set property for the class of $\kappa$-compact sets by the next proposition. 
It is anologous to Proposition~\ref{prop: PSP KLW new}.

\begin{proposition}
\label{prop: PSP THD new}
Assume that the $\kappa$-perfect set property holds for all $\kappa$-compact sets.
\todol{New proposition}
\begin{enumerate-(1)}
\item\label{prop: PSP THD new 1} 
$\THD_\kappa(X,X)\Longrightarrow\PSP_\kappa(X)$ for all $X\subseteq{}^\kappa\kappa$.
\item\label{prop: PSP THD new 2} 
$\THD_\kappa(\mathcal C,\mathcal C)\Longrightarrow\PSP_\kappa(\mathcal C)$.
\item\label{prop: PSP THD new 3} 
$\THD_\kappa(\mathcal C,\mathcal D)\Longleftrightarrow\KLW_\kappa(\mathcal C,\mathcal D)$
if $\mathcal C\wedge\mathcal D\subseteq\mathcal C\subseteq\mathcal D=\check{\mathcal D}$\footnote{If $\mathcal D$ is closed under finite intersections, then it is equivalent to assume $\mathcal C \downset \mathcal D$.} 
and every element of $\mathcal C$ is contained in a $K_\kappa$ set. 
\end{enumerate-(1)}
\end{proposition}
\begin{proof}
For~\ref{prop: PSP KLW new 1}, suppose first that 
$X$ is a $K_\kappa$ set, and write $X$ as a $\kappa$-union of $\kappa$-compact sets $X_\alpha$.
Then $\PSP_\kappa(X)$ holds since $\PSP_\kappa(X_\alpha)$ holds for each $\alpha<\kappa$ by assumption.
If $X$ is not $K_\kappa$, then
there exists a continuous homomorphism $f$
from $\dhH\kappa$ to $\HH:=\dhth{X}\restr X$
since $\THD_\kappa(X,X)$ holds.%
\footnote{Recall that $\THD_\kappa(X,X)$ is equivalent to $\ODD\kappa\HH$ by Theorem~\ref{theorem: THD and ODD dhth}.}
Then $\ran(f)\subseteq X$ has a $\kappa$-perfect subset
by Lemma~\ref{lemma: kappa-perfect subdh}~\ref{lemma: kappa-perfect subdh 1}.
\ref{prop: PSP KLW new 2}
follows immediately from~\ref{prop: PSP KLW new 1}.
For~\ref{prop: PSP KLW new 3},
The direction from left to right
follows from Proposition~\ref{prop: PSP KLW new}~\ref{prop: PSP KLW new 2} and 
Corollary~\ref{cor: KLW, KLWdis and THD}.
The converse follows from Proposition~\ref{KLW and THD for precompact X}.
\end{proof}

If $\kappa$ is uncountable and not weakly compact, then there is a stronger version of \ref{prop: PSP THD new 1}, since Proposition~\ref{remark: PSP THD non weakly compact equiv} shows that $\THD_\kappa(X,Y)$ implies $\PSP_\kappa(X)$ for any $Y\supseteq X$ if the $\kappa$-perfect set property holds for $\kappa$-compact sets.

We now turn to the implications between the Jayne--Rogers theorem and other applications.
The next proposition shows that
for functions $f$ with $\kappa$-precompact range,
$\JR_\kappa^f$ follows from
the 
Kechris--Louveau--Woodin dichotomy.
Moreover, it follows
from the 
two-parameter topological Hurewicz dichotomy 
if the domain of $f$ is also $\kappa$-precompact. 
Recall that $\JR_\kappa^f$ follows from $\ODD\kappa\JJ$ for the $\kappa$-dihypergraph $\JJ$ in the next proposition
by 
Corollary~\ref{cor: JR}
and Remark~\ref{remark: injective JJ}.

\begin{proposition}
\label{JR from KLW or THD}
Suppose $X\subseteq{}^\kappa\kappa$ and $f:X\to{}^\kappa\kappa$. 
Let $\JJ:=\JJ^f_i\restr \graph(f)$,%
\footnote{Recall that $\graph(f)$ denotes the graph of $f$. The dihypergraph $\JJ^f$ and its variant $\JJ_i^f$ were defined at the beginning of Subsection~\ref{subsection: Jayne Rogers} and in Remark~\ref{remark: injective JJ}, respectively.}
and let
$A:=k(\graph(f))$ and 
$B:=k\big((X\times {}^\kappa\kappa)\setminus\! \graph(f)\big)=
k\big(X\times{}^\kappa\kappa\big)\setminus A$, 
where
$k$  is a homeomorphism from ${}^\kappa\kappa\times{}^\kappa\kappa$ onto ${}^\kappa\kappa$.
%
\begin{enumerate-(1)}
\item\label{JR from KLW or THD 1}
$\ODD\kappa\JJ\Longleftrightarrow\KLW_\kappa(A,B)$ if $\ran(f)$ 
is contained in a $K_\kappa$ subset of ${}^\kappa\kappa$.
\todog{(1) also follows from (2) in the case that $X$ is also contained in a $K_\kappa$ set by Proposition~\ref{KLW and THD for precompact X}}
\item\label{JR from KLW or THD 2}
$\ODD\kappa\JJ\Longleftrightarrow\THD_\kappa(A,{}^\kappa\kappa\setminus  B)$ if both $X$ and $\ran(f)$ are 
contained in a $K_\kappa$ subsets of ${}^\kappa\kappa$.
\end{enumerate-(1)}
\end{proposition}
\begin{proof}
Let $h:=k^{-1}$, $\CC:=
{h}^\ddim(\CC_i^{B}\restr{A})=
{h}^\ddim(\CC_i^{B})\restr\graph(f)
$
and
$\HH:=(h)^{\ddim}(\dhth{{}^\kappa\kappa\setminus B}\restr A)=(h)^{\ddim}(\dhth{{}^\kappa\kappa\setminus B})\restr \graph(f)$.
The next claim shows that 
$\JJ\equivf\CC$ if $\ran(f)$ is $\kappa$-precompact 
and $\JJ\equivf\HH$ if both $X$ and $\ran(f)$ are $\kappa$-precompact.

\begin{claim*}\ 
\begin{enumerate-(i)}
\item\label{JR from KLW subclaim 1}
$\CC\leqf\JJ$.
\item\label{JR from KLW subclaim 2}
$\JJ\leqf\CC$ if $\ran(f)$ is $\kappa$-precompact.
\item\label{JR from KLW subclaim 3}
$\JJ\subseteq\HH$.
\item\label{JR from KLW subclaim 4}
$\HH\leqf\CC$ if $X$ and $\ran(f)$ are both $\kappa$-precompact.
\end{enumerate-(i)}
\end{claim*}
\begin{proof}
For~\ref{JR from KLW subclaim 1},
note that $\CC$ consists of those injective convergent sequences $\bar a$ in $\graph(f)$ with 
$(x^{\bar a},y^{\bar a}):=\lim_{\alpha<\kappa}(\bar a)$ in $(X\times{}^\kappa\kappa)\setminus\! \graph(f)$.
It suffices to show that any such sequence $\bar a=\langle (x_\alpha,f(x_\alpha)):\alpha<\kappa\rangle$ has a subsequence in $\JJ$.
Note that 
$x^{\bar a}$ is in $X$ and
$f(x^{\bar a})\neq y^{\bar a}=\lim_{\alpha<\kappa}f(x_\alpha)$.
By passing to a subsequence of $\bar a$, we may assume that $x\neq x_\alpha$ and $f(x)\neq f(x_\alpha)$ for all $\alpha<\kappa$. 
Then $\langle x_\alpha:\alpha<\kappa\rangle$ is in $\CC^X_i$ 
\todog{If $\langle x_\alpha:\alpha<\kappa\rangle$ weren't injective, then $\bar a$ would not be either since $f$ is well defined.}
and $f(x)\notin\closure{\{f(x_\alpha):\alpha<\kappa\}}$.
This sequence is in $\JJ$, as required.

We prove~\ref{JR from KLW subclaim 2}
and~\ref{JR from KLW subclaim 3} simultaneously.
Note that $\HH$ consists of those injective $\kappa$-sequences 
in $\graph(f)$ whose limit points are all contained in $(X\times{}^\kappa\kappa)\setminus\!\graph(f)$.
Suppose $\bar a=\langle (x_\alpha,f(x_\alpha):\alpha<\kappa\rangle$ is a hyperedge of $\JJ$.
Since $\bar x:=\langle x_\alpha:\alpha<\kappa\rangle$ is in $\CC_i^X$, it must converge to some $x\in X$.
Hence all limit points of $\bar a$ are of the form $(x,y)$ for some $y\in{}^\kappa\kappa$.
Moreover, $(x,f(x))$ is not a limit point of $\bar a$ 
since $f(x)$ is not in $\closure{\{f(x_\alpha):\alpha<\kappa\}}$. 
Since $\bar a$ is injective, it is a hyperedge of $\HH$. 
Now, assume that $\ran(f)$ is $\kappa$-precompact. Then $\langle f(x_\alpha):\alpha<\kappa\rangle$ and therefore $\bar a$ has a convergent subsequence,
This subsequence is in $\CC$, 
since its limit is in $(X\times{}^\kappa\kappa)\setminus\!\graph(f)$ and it is injective.

For~\ref{JR from KLW subclaim 4}, note that if $X$ and $\ran(f)$ are both $\kappa$-precompact,
then so is $\graph(f)$.
As in the proof of Proposition~\ref{KLW and THD for precompact X}, it suffices to see that
any hyperedge $\bar a$ 
of $\HH$ contains a subsequence in $\CC$. 
Since $\bar a$ is 
contained in the $\kappa$-precompact set $\graph(f)$, 
it must contain a convergent subsequence $\bar b$. 
As $\bar a\in \HH$, the limit of $\bar b$ 
is an element of $(X\times{}^\kappa\kappa)\setminus\!\graph(f)$.
Hence $\bar b\in \CC$, as required.
\end{proof}

For~\ref{JR from KLW or THD 1},
suppose that $\ran(f)$ is contained in a $K_\kappa$ set $\bigcup_{\alpha<\kappa} C_\alpha$, where each $C_\alpha$ is $\kappa$-compact.
Let $G_\alpha:=\graph(f)\cap(X\times C_\alpha)$ for each $\alpha<\kappa$.
Then
$\graph(f)=\bigcup_{\alpha<\kappa}G_\alpha$, and
$\JJ\restr G_\alpha\equivf\CC\restr G_\alpha$ for each $\alpha<\kappa$ by the previous claim.
Hence $\JJ\equivf\CC$,
and thus 
$\ODD\kappa\JJ$ and $\ODD\kappa\CC$ are equivalent
by Corollary~\ref{cor: ODD subsequences}.
The latter is further equivalent to $\KLW_\kappa(A,B)$ by 
Theorem~\ref{theorem: KLW from ODD} and Remark~\ref{remark: CC and injectivity}. 

For~\ref{JR from KLW or THD 2},
note that if both $X$ and $\ran(f)$ are contained in $K_\kappa$ sets, then so is $\graph(f)$.
Write $\graph(f)$ as the $\kappa$-union of $\kappa$-precompact sets $G_\alpha$.
Then $\JJ\restr G_\alpha\equivf\HH\restr G_\alpha$ for each $\alpha<\kappa$ by the previous claim.
Hence $\JJ\equivf\HH$, and therefore
$\ODD\kappa\JJ$ is equivalent to $\ODD\kappa\HH$
by Corollary~\ref{cor: ODD subsequences},
which is further equivalent to $\THD_\kappa(A,{}^\kappa\kappa\setminus  B)$
by Theorem~\ref{theorem: THD and ODD dhth}.
\end{proof}

\begin{corollary}
\label{cor: THD implies JR weakly compact}\ 
Suppose that $\kappa$ is weakly compact or $\kappa=\omega$. 
\todoq{the assumption cound be shortened, maybe}
\todo{open question: does JR for functions with pre $K_\kappa$ range imply JR? \\ 
Note that the proof of this prop. might show: for any $\kappa$, same conclusion but for the weaker version of JR instead of JR  } 
If the classes $\mathcal C\subseteq \mathcal D$ contain all $\Gdelta(\kappa)$ sets,  
$\mathcal C$ is closed under continuous preimages, 
$\mathcal D$ under complements and 
both are closed under finite intersections, 
 then 
$\THD_\kappa(\mathcal C,\mathcal D)\Longrightarrow\JR_\kappa(\mathcal C)$. 
\end{corollary}
\begin{proof}
Suppose that $\THD_\kappa(\mathcal C,\mathcal D)$ holds.
It suffices to show $\JR_\kappa^g$ for all
subsets $Y\in\mathcal C$ of ${}^\kappa\kappa$
and all $\DeltaZeroTwo(\kappa)$-measurable functions $g:Y\to{}^\kappa\kappa$.
Let $h$ be a homeomorphism from ${}^\kappa\kappa$ onto a subset of ${}^\kappa 2$ 
and
 $f:=h\comp g\comp h^{-1}$.
Since $X:=\dom(f)=h(Y)$ and $\ran(f)=h(\ran(g))$ are both 
$\kappa$-precompact,
$\THD_\kappa(A,{}^\kappa\kappa\setminus B)$
is equivalent to $\ODD\kappa\JJ$,
where
$\JJ$, $k$, $A:=k(\graph(f))$ and 
$B:=k(X\times{}^\kappa\kappa)\setminus A$ 
are defined as in the previous proposition. 
We will show that $A\in\mathcal C$ and ${}^\kappa\kappa\setminus B\in\mathcal D$.
This suffices, since 
$\JR_\kappa^f$ and thus $\JR_\kappa^g$ follow from
$\THD_\kappa(A,{}^\kappa\kappa\setminus B)$
by Corollary~\ref{cor: JR} and Remark~\ref{remark: injective JJ}

Since $f$ is $\DeltaZeroTwo(\kappa)$-measurable, $\graph(f)$ is a $\Gdelta(\kappa)$ subset of ${}^\kappa\kappa \times{}^\kappa\kappa$. 
To see this, note that 
$$\graph(f)=\bigcap_{t\in{}^{<\kappa}\kappa}\big[
(X \times ({}^\kappa\kappa\setminus N_t))\cup
(f^{-1}(N_t)\times{}^\kappa\kappa)
\big],$$ 
since for all $(x,y)\in X\times{}^\kappa\kappa$ we have $f(x)=y$ $\Longleftrightarrow$ $\forall t\in{}^{<\kappa}\kappa
\big(y\in N_t\Rightarrow f(x)\in N_t\big)$.

Note that 
$X\in \mathcal C$. 
Since $X\times{}^\kappa\kappa$ is the preimage of $X$ under a projection and $k$ is a homeomorphism, we have $k(X\times{}^\kappa\kappa)\in \mathcal C$. 
Since $A:=k(\graph(f))$ is a $\Gdelta(\kappa)$ subset of $k(X\times {}^\kappa\kappa)$, we have $A\in \mathcal C$. 
Therefore ${}^\kappa\kappa\setminus B= A\cup k(({}^\kappa\kappa\setminus X)\times{}^\kappa\kappa)\in \mathcal D$. 
\end{proof}

We now turn to the implications in Figures~\ref{figure: various implications not weakly compact} and~\ref{figure: various implications pspclosed} which use that $\kappa$ is uncountable and not weakly compact.
The next proposition shows that for subsets $X$ of $K_\kappa$ sets, 
$\PSP_\kappa(X)$ can only hold if $|X|\leq\kappa$. Hence $\PSP_\kappa(X)$ implies all of the other dichotomies in Figure~\ref{figure: various implications not weakly compact} in this case.

\begin{proposition}[See \cite{LuckeMottoRosSchlichtHurewicz}*{Proposition~2.7}] 
\label{prop: PSP pre-Kkappa non weakly compact}
Suppose that $\kappa>\omega$ is not weakly compact and $X$ is a subset of a $K_\kappa$ set.
\begin{enumerate-(1)}
\item\label{prop: PSP pre-Kkappa non weakly compact 1}
$X$ does not contain a $\kappa$-perfect subset.
\item\label{prop: PSP pre-Kkappa non weakly compact 2}
$\PSP_\kappa(X)$ implies that $|X|\leq\kappa$.
\end{enumerate-(1)}
\end{proposition}
\begin{proof}
For \ref{prop: PSP pre-Kkappa non weakly compact 1}, note that otherwise $X$ would contain a closed subset homeomorphic to ${}^\kappa 2$ and hence to ${}^\kappa\kappa$ by \cite{hung1973spaces}*{Theorem 1}, 
contradicting the fact that $X$ is contained in a $K_\kappa$ set. 
\ref{prop: PSP pre-Kkappa non weakly compact 2} follows from~\ref{prop: PSP pre-Kkappa non weakly compact 1}.
\end{proof}

The next proposition shows that if $\kappa>\omega$ is not weakly compact, then 
$\PSP_\kappa(\mathcal C)\Longrightarrow\THD_\kappa(\mathcal C,\pwrset({}^\kappa\kappa))$
and that if in addition $\PSP_\kappa(\closedsets(\kappa))$ is assumed,
then 
$\PSP_\kappa(\mathcal C)$
$\Longleftrightarrow$
$\THD_\kappa(\mathcal C,\mathcal D)$
$\Longleftrightarrow$
$\THD_\kappa(\mathcal C)$
for any class $\mathcal D$. 
\index{topological Hurewicz dichotomy!strong\idf$\mathsf{sTHD}_\kappa(X)$}%
The \emph{strong topological Hurewicz di\-cho\-tomy} $\mathsf{sTHD}_\kappa(X)$ 
for a subset $X$ of ${}^\kappa\kappa$ is the statement that $X$ is either itself a $K_\kappa$ set or else $X$ contains a closed subset of ${}^\kappa\kappa$ which is homeomorphic to ${}^\kappa\kappa$.
Note that for $\kappa=\omega$ and weakly compact cardinals $\kappa$, there are $\Gdelta(\kappa)$ counterexamples to the strong topological Hurewicz dichotomy.%
\footnote{The set $\RR_\kappa$ defined at the beginning of Subsection~\ref{subsection: original KLW} is such a counterexample.}

\begin{proposition}[See \cite{LuckeMottoRosSchlichtHurewicz}*{Proposition~2.7}] 
\label{remark: PSP THD non weakly compact}
\label{remark: PSP THD non weakly compact equiv}
Suppose $\kappa$ is not weakly compact and $X\subseteq Y\subseteq{}^\kappa\kappa$. 
\todol{New proposition; part of it used to be a remark in Section 6.1} 
Then each of the following statements implies the ones following it:
\begin{enumerate-(1)}
\item\label{remark: PSP THD non weakly compact a}
$\PSP_\kappa(X)$,
\item\label{remark: PSP THD non weakly compact b}
$\mathsf{sTHD}_\kappa(X)$, 
\item\label{remark: PSP THD non weakly compact c}
$\THD_\kappa(X,Y)$.
\end{enumerate-(1)}
If the $\kappa$-perfect set property holds for all $\kappa$-compact sets, 
then all three statements are equivalent.
\end{proposition}
\begin{proof}
\ref{remark: PSP THD non weakly compact a} $\Rightarrow$ \ref{remark: PSP THD non weakly compact b}
is immediate from the fact that ${}^\kappa 2$ is homeomorphic to ${}^\kappa\kappa$ in the non-weakly compact case by \cite{hung1973spaces}*{Theorem 1}. 
\ref{remark: PSP THD non weakly compact b} $\Rightarrow$ \ref{remark: PSP THD non weakly compact c} is clear.
For the last part of the proposition,
suppose that the $\kappa$-perfect set property holds for all $\kappa$-compact sets and 
$\THD_\kappa(X,Y)$ holds.
If $X$ is contained in a $K_\kappa$ set, then $X$ has size at most $\kappa$ by the previous proposition.
Otherwise, 
there exists a continuous homomorphism $f$
from $\dhH\kappa$ to $\HH:=\dhth{X}\restr Y$
since $\THD_\kappa(X,Y)$ is equivalent to $\ODD\kappa\HH$ by Theorem~\ref{theorem: THD and ODD dhth}.
Hence $\ran(f)\subseteq X$ has a $\kappa$-perfect subset
by Lemma~\ref{lemma: kappa-perfect subdh}~\ref{lemma: kappa-perfect subdh 1}.
\end{proof}

Furthermore, recall that if $\kappa>\omega$ is not weakly compact, then
$\THD_\kappa(X,Y)$ implies $\THD_\kappa(X,Z)$ for all supersets $Z$ of $Y$ 
by Remark~\ref{THD strong version non weakly compact}.
In particular, $\THD_\kappa(X,Y)$ implies $\THD_\kappa(X)$.

\begin{corollary}
\label{cor: KLW to THD new}
$\KLW_\kappa(\mathcal C,D)\Longrightarrow\THD_\kappa(\mathcal C,\mathcal D)$ 
if ${}^\kappa\kappa\in\mathcal D$, $\mathcal C$ and $\mathcal D$ are closed under homeomorphic images and $\mathcal D$ is closed under complements.%
\footnote{It suffices to assume  ${}^\kappa\kappa\in\mathcal D$ if $\kappa$ is not weakly compact and the above closure properties for $\mathcal C$ and $\mathcal D$ suffice if $\kappa$ is weakly compact or $\kappa=\omega$.} 
\end{corollary}
\begin{proof}
First suppose $\kappa>\omega$ is not weakly compact. Then
$\KLW_\kappa(\mathcal C,\mathcal D)\Longrightarrow \PSP_\kappa(\mathcal C)$ by Proposition~\ref{connection btw KLW, KLWdis and PSP}~\ref{connection btw KLW, KLWdis and PSP 1} and
$\PSP_\kappa(\mathcal C)\Longrightarrow \THD_\kappa(\mathcal C,\mathcal D)$ by the previous proposition.
If $\kappa$ is weakly compact or $\kappa=\omega$, then the claim is immediate by Corollary~\ref{cor: KLW implies THD weakly compact}.
\end{proof}

Recall that
$\KLW_\kappa(\mathcal C,D)$
can be weakened to
$\KLWdis_\kappa(\mathcal C,D)$
in the previous statement
at weakly compact cardinals $\kappa$ and $\kappa=\omega$
if $\mathcal C$ and $\mathcal D$ are sufficiently closed
by Corollary~\ref{cor: KLW implies THD weakly compact}.
The next corollary shows that this is also true at non-weakly compact cardinals $\kappa>\omega$
if the $\kappa$-perfect set property for closed sets is assumed.

\begin{corollary}
\label{cor: THD KLW non-weakly compact new}
Suppose $\kappa>\omega$ is not weakly compact and 
 $\PSP_\kappa(\closedsets(\kappa))$ holds.
Then $\KLWdis_\kappa(\mathcal C,\mathcal D)\Longrightarrow \THD_\kappa(\mathcal C,\mathcal D)$
if $\check{\mathcal C}\subseteq \mathcal D$. 
\end{corollary}
\begin{proof}
This follows Proposition~\ref{prop: PSP KLW new}~\ref{prop: PSP KLW new 2}
since $\PSP_\kappa(\mathcal C)\Longrightarrow\THD_\kappa(\mathcal C,\mathcal D)$ by the previous proposition.
\end{proof}


We now turn to the non-implications in Figures~\ref{figure: various implications weakly compact} and~\ref{figure: various implications not pspclosed}, starting with the former.

\begin{remark}
\label{remark: separating KLWdis and PSP}
\label{sep: not PSP to THD omega}{}%
In the countable setting,
neither 
$\PSP_\omega(\mathcal C)$
nor the open graph dichotomy $\OGD_\omega(\mathcal C)$
imply 
$\KLWdis_\omega(\mathcal C,\Fsigma)$ or $\THD_\omega(\mathcal C,\Gdelta)$
for the class
$\mathcal C:=L(\RR)[U]$, if $L(\RR)$ is a Solovay model and $U$ is a selective ultrafilter on $\omega$ that is
 generic over $L(\RR)$ for $\pwrset(\omega)/\mathsf{fin}$.
This follows from results of Di Prisco and Todor\v{c}evi\'c 
\cite{di1998perfect};%
\footnote{$\OGD_\omega(X)$ is denoted by $\OCA(X)$ and $\KLW_\omega(X,Y)$ is called the \emph{Hurewicz separation property} in \cite{di1998perfect}.}
in particular, 
\cite{di1998perfect}*{Theorem~5.1} shows that $\OGD_\omega(\mathcal{C})$ holds, while
\cite{di1998perfect}*{Lemma~6.1} shows that $\KLW_\omega(X,Y)$ fails for the subset $X$ of the Cantor space induced by $U$ and $Y=\QQ_\omega$. 
\todog{D's note to self (uncomment later)
}%
Hence, $\THD_\omega(\mathcal C,\Gdelta)$ also fails
by Proposition~\ref{KLW and THD for precompact X}.
\end{remark}

\begin{remark}\ 
\label{sep: KLW, KLWdis and THD}
In the uncountable setting, the following non-implications in Figure~\ref{figure: various implications not pspclosed} follow easily from previous results in this subsection.
\begin{enumerate-(1)}
\item\label{sep: KLW and KLWdis}
If the $\kappa$-perfect set property for closed sets fails, then
$\KLWdis_\kappa(\mathcal C,\mathcal D)$ 
implies neither $\PSP_\kappa(\mathcal C)$ nor $\KLW_\kappa(\mathcal C,\mathcal D)$
for $\mathcal C:=\Fsigma(\kappa)$
and any $\mathcal D$ with ${}^\kappa\kappa\in\mathcal D$, since
$\KLWdis_\kappa(\Fsigma(\kappa))$ is trivially true
while $\PSP_\kappa(X)$ fails for some closed set $X$ and hence $\KLW_\kappa(X,{}^\kappa\kappa)$ fails by Corollary~\ref{cor: PSP from ODD for the kappa-Baire space 2}~\ref{cor: PSP from ODD for the kappa-Baire space 2 a}.

\item\label{sep: PSP and THD Kkappa}
Similarly to \ref{sep: KLW and KLWdis}, if 
the $\kappa$-perfect set property for $\kappa$-compact sets fails,
then neither $\THD_\kappa(\mathcal C)$ nor $\THD_\kappa(\mathcal C,\mathcal D)$ imply $\PSP_\kappa(\mathcal C)$
for $\mathcal C:=K_\kappa$ and any nonempty class $\mathcal D\subseteq\pwrset({}^\kappa\kappa)$,
since the former  
are trivially true for $K_\kappa$ sets while the latter fails for some $\kappa$-compact set.

\item\label{sep: KLWdis and THD}
Now, assume that $\kappa$ is uncountable and not weakly compact and
the $\kappa$-perfect set property holds for all $\kappa$-compact sets but fails for some closed subset of ${}^\kappa\kappa$.
Note that this is consistent for $\kappa=\omega_1$ 
relative to an inaccessible cardinal by \cite{LambieStejPSP}*{Theorem~A}
and~\cite{LuckeMottoRosSchlichtHurewicz}*{Lemma~2.2}.
In more detail,
the $\omega_1$-perfect set property for $\omega_1$-compact sets
holds if and only if 
\index{tree!Kurepa@$\kappa$-Kurepa\idf}
every $\omega_1$-Kurepa tree%
\footnote{Recall that a subtree $T$ of ${}^{<\kappa}\kappa$ of height $\kappa$ is a \emph{$\kappa$-Kurepa tree} if its levels $T\cap{}^\alpha\kappa$ have size $\lle\kappa$ for all $\alpha<\kappa$ and it has at least $\kappa^+$ many branches.}
contains an $\omega_1$-Aronszajn subtree
by \cite{LuckeMottoRosSchlichtHurewicz}*{Lemma~2.2}.
Moreover, the nonexistence of $\omega_1$-Kurepa trees is consistent with 
the failure of the $\omega_1$-perfect set property for closed sets
relative to an inaccessible cardinal by \cite{LambieStejPSP}*{Theorem~A}.

In this case, $\KLWdis_\kappa(\mathcal C,\mathcal D)$ can be separated from
$\THD_\kappa(\mathcal C)$ and more generally, from
$\THD_\kappa(\mathcal C,\mathcal E)$ for $\mathcal C:=\Fsigma(\kappa)$ and any nonempty classes $\mathcal{D,\,E}\subseteq\pwrset({}^\kappa\kappa)$,
since
$\THD_\kappa(X)$ and 
$\THD_\kappa(X,Y)$ also fail for some closed subset $X$ of ${}^\kappa\kappa$ and all subsets $Y$ of ${}^\kappa\kappa$, by Proposition~\ref{remark: PSP THD non weakly compact equiv}.
However, $\KLWdis_\kappa(\Fsigma(\kappa))$ is true by definition.
\end{enumerate-(1)}
\end{remark}

\begin{remark}\ 
\label{remark: connections between PSP, THD and HD}
In the uncountable setting,
results in~\cite{LuckeMottoRosSchlichtHurewicz} imply that
several versions of the Hurewicz dichotomy for $\kappa$-analytic sets do not imply the $\kappa$-perfect set property for closed sets.
Hence 
separations similar to  \ref{sep: PSP and THD Kkappa} of the previous remark are consistent for the larger class of $\kappa$-analytic sets  with the failure of the $\kappa$-perfect set property for closed sets.
In particular:
\begin{enumerate-(1)}
\item 
\label{sep: PSP and THD}{}%
It is consistent that $\THD_\kappa(\analytic(\kappa))$
holds and the $\kappa$-perfect set property for closed sets 
fails simultaneously at all regular uncountable cardinals $\kappa$ \cite{LuckeMottoRosSchlichtHurewicz}*{Corollary~5.4}. 
Note that these statements are also consistent with the existence of weakly compact cardinals, since the iteration $\PP_{<\kappa}$ used in the proof of \cite{LuckeMottoRosSchlichtHurewicz}*{Corollary~5.4} preserves the $(\kappa+2)$-strong unfoldability of any cardinal~$\kappa$ by \cite{LuckeMottoRosSchlichtHurewicz}*{Lemmas~4.3 and 4.7} and thus these cardinals, which may exist in the ground model $L$, remain weakly compact in the extension. 

\item 
\label{impl and sep: PSP and THD non weakly compact}{}%
It is consistent that 
$\THD_\kappa(\analytic(\kappa),\pwrset({}^\kappa\kappa))$ holds but
the $\kappa$-perfect set property for closed sets fails 
simultaneously 
at all uncountable non-weakly compact cardinals, 
since the model in \cite{LuckeMottoRosSchlichtHurewicz}*{Corollary~5.4} satisfies the 
strong topological Hurewicz dichotomy for $\kappa$-analytic sets%
at all such cardinals.
To see this, note that the argument in the proof of \cite{LuckeMottoRosSchlichtHurewicz}*{Theorem~1.7} works for the construction in \cite{LuckeMottoRosSchlichtHurewicz}*{Theorem~1.9} too.  

\item
\label{sep: PSP and HD}{}%
A similar non-implication from $\HD_\kappa(\analytic(\kappa))$ to the $\kappa$-perfect set property for closed sets is possible for all inaccessible cardinals $\kappa$ by arguments in \cite{LuckeMottoRosSchlichtHurewicz}*{Section 3.2}. 
This follows from
\cite{LuckeMottoRosSchlichtHurewicz}*{Proposition~3.19 and Lemma~3.23 (4) $\Rightarrow$ (1)}
and the observation that for inaccessible cardinals $\kappa$,
the first option in $\HD_\kappa(X)$ holds if and only if $X$ is eventually bounded. 
\end{enumerate-(1)}
\end{remark}

\newpage 

\section{Outlook} 
\label{section: open questions}
\numberwithin{theorem}{section} 

As above, $\kappa$ always denotes an infinite cardinal with $\kappa^{<\kappa}=\kappa$.
With regards to the main results, we ask for a {\bf global version} of Theorem~\ref{main theorem} \ref{mainthm Mahlo} for all regular cardinals and arbitrary box-open dihypergraphs. 
The difficulty with proving such a result is that in an iterated forcing extension as in the proof of Theorem \ref{theorem: global version}, new box-open dihypergraphs are added by the tail forcings and the box-open dihypergraph dichotomy does not automatically hold for them. 

\begin{problem} 
\label{q global Mahlo}
Assume there is a proper class of Mahlo cardinals. 
Is there a tame class forcing extension of $V$ where $\ODD\kappa\kappa(\defsetsk)$ holds for all 
$\kappa$? 
\end{problem} 


The {\bf consistency strength} of $\ODD\kappa\kappa(\defsets\kappa)$ is open. 
In particular, we do not know if it implies that there is an inner model with a Mahlo cardinal. 
We further ask about the consistency strength of each of its variants and applications. 

\begin{problem}
\label{q cons Mahlo} 
Do the following need a Mahlo cardinal for some uncountable $\kappa$? 
\begin{enumerate-(1)}
\item 
\label{q cons Mahlo A}
\begin{enumerate-(a)}
\item
\label{q cons Mahlo 1} 
$\ODD\kappa\kappa(\defsets\kappa)$.
\item
\label{q cons Mahlo 2}
$\ODD\kappa\kappa({}^\kappa\kappa)$. 
\end{enumerate-(a)}
\item 
\label{q cons Mahlo 3} 
\begin{enumerate-(a)}
\item\label{q cons Mahlo 3a} $\KLW^\ddim_\kappa(\defsetsk)$, where $2\leq\ddim<\kappa$.%
\footnote{See Subsection~\ref{subsection: KLW^H}.
Note that $\KLW^\kappa_\kappa(\defsetsk)$
is equivalent to $\ODD\kappa\kappa(\defsetsk)$ by Corollary~\ref{cor: from KLW to ODD kappa}.}
\todog{I.e., $\KLW^\ddim_\kappa(X,Y)$ for all $X\in\defsetsk$ and ALL $Y\subseteq{}^\kappa\kappa$. If $\ddim<\kappa$, this is stronger than $\ODD\kappa{\ddim}(\defsetsk)$,
which is equivalent to: $\KLW_\kappa^\ddim(X,Y)$ for all $\Gdelta(\kappa)$ sets $Y$ and all $X\in\defsetsk$. The latter is equiconsistent with an inaccessible.}
\item\label{q cons Mahlo 3b} $\KLW_\kappa(\defsetsk)$. 
\item\label{q cons Mahlo 3c} $\KLW_\kappa({}^\kappa\kappa)$. 
\end{enumerate-(a)}
\item\label{q cons Mahlo 3dis} $\KLWdis_\kappa(\defsetsk)$.%
\footnote{Recall that $\KLWdis_\kappa(\defsetsk)$ denotes the statement that $\KLW_\kappa(X,Y)$ holds for all \emph{disjoint} subsets $X\in\defsetsk$ and $Y$ of ${}^\kappa\kappa$, while $\KLW_\kappa(\defsetsk)$ is the version where $X$ and $Y$ are not necessarily disjoint.}
\item 
\label{q cons Mahlo B} 
\begin{enumerate-(a)}
\item 
\label{q cons Mahlo 4} 
$\cbii(\powerset({}^\kappa\kappa))$. 
\item\label{q cons Mahlo cbo}
$\cbo$.\footnote{$\cbo$ and $\DK_\kappa(X)$ were defined in the paragraphs preceding Corollary~\ref{corollary: DK}.}
\item 
\label{q cons Mahlo 5} 
The determinacy of $\GV_\kappa(X)$ for all subsets $X$ of ${}^\kappa\kappa$.
\end{enumerate-(a)}
\item\label{q cons Mahlo DK}
$\DK_\kappa(\analytic(\kappa))$.
\end{enumerate-(1)}
\end{problem}

In \ref{q cons Mahlo B}\ref{q cons Mahlo 4}, one can equivalently consider 
$\cbi(\powerset({}^\kappa\kappa))$. 
Note that
\ref{q cons Mahlo A}\ref{q cons Mahlo 1}
$\Rightarrow$
\ref{q cons Mahlo 3}\ref{q cons Mahlo 3a}
$\Rightarrow$
\ref{q cons Mahlo 3}\ref{q cons Mahlo 3b}
$\Rightarrow$
\ref{q cons Mahlo 3dis},
while
\ref{q cons Mahlo A}\ref{q cons Mahlo 2}
$\Rightarrow$
\ref{q cons Mahlo 3}\ref{q cons Mahlo 3c}
$\Rightarrow$
\ref{q cons Mahlo B}\ref{q cons Mahlo 4}
$\Rightarrow$ 
\ref{q cons Mahlo B}\ref{q cons Mahlo cbo}$+$\ref{q cons Mahlo 5}, 
and
\ref{q cons Mahlo B}\ref{q cons Mahlo cbo}~$+$~$\Diamondi\kappa$
$\Rightarrow$
\ref{q cons Mahlo DK}.%
\footnote{These implications follow from the results in Subsections~\ref{subsection: KLW} and \ref{subsection: Vaananens game}.}

The statements 
$\ODD\kappa\kappa({}^\kappa\kappa,\defsetsk)$, 
$\KLW_\kappa({}^\kappa\kappa,\defsetsk)$, 
$\cbii(\defsetsk)$
and
$\OGD_\kappa(\defsetsk)$
are each equi\-consistent with the existence of an inaccessible cardinal, since they each imply $\PSP_\kappa(X)$ for all closed sets $X$. 
For non-weakly compact $\kappa$, 
$\THD_{\kappa}(\defsets{\kappa},\powerset({}^{\kappa}\kappa))$ is consistent relative to an inaccessible cardinal, 
since it follows from $\PSP_{\kappa}(\defsets{\kappa})$
\cite{LuckeMottoRosSchlichtHurewicz}*{Proposition~2.7}, 
and hence so is $\THD_{\kappa}(\defsets{\kappa},\defsets{\kappa})$.
Moreover, for weakly compact $\kappa$, $\THD_{\kappa}(\defsets{\kappa},\powerset({}^{\kappa}\kappa))$ follows from $\KLWdis_\kappa(\defsets{\kappa})$ by Corollary \ref{cor: KLW implies THD weakly compact}. 

\begin{problem}
Do the following need an inaccessible cardinal for some uncountable $\kappa$? 
\begin{enumerate-(1)}
\item 
\begin{enumerate-(a)}
\item
$\HD_\kappa(\defsets\kappa)$.
\item 
$\THD_\kappa(\defsets\kappa)$.%
\footnote{Recall that this is the same statement as 
$\THD_\kappa(\defsets\kappa,\{{}^\kappa\kappa\})$.}
\item
$\THD_\kappa(\defsets\kappa,\defsets\kappa)$.
\end{enumerate-(a)}
\item 
$\KLWdis_\kappa(\defsetsk,\defsetsk)$. 
\todoq{can we use the failure in $L$ of ``every proper $\Fsigma(\kappa)$ set is $\Fsigma(\kappa)$-complete to show that $\omega_1$ is inaccessible in $L$?} 
\item 
$\ABP_\kappa(\defsets\kappa)$.
\item 
$\JR_\kappa(\defsetsk)$. \todog{this is equivalent to: ``$\JR_\kappa^f$ for all $f\in\defsetsk$''. Also, this is the version which we've stated in Corollary~\ref{cor: JR 2}}
\item 
$\OGD_\kappa({}^\kappa\kappa)$.
\end{enumerate-(1)}
\end{problem}

Note that it is consistent relative to $\ZFC$ that
$\THD_\kappa(\analytic(\kappa))$ holds
for all regular uncountable cardinals $\kappa$ \cite{LuckeMottoRosSchlichtHurewicz}*{Theorem~1.13},
$\THD_\kappa(\analytic(\kappa),\pwrset({}^\kappa\kappa))$ holds for all uncountable non-weakly compact cardinals $\kappa$,%
\footnote{To see this, note that the model constructed in \cite{LuckeMottoRosSchlichtHurewicz}*{Theorem~1.13} 
satisfies the strong topological Hurewicz dichotomy for $\kappa$-analytic sets at all non-weakly compact cardinals by the proof of  \cite{LuckeMottoRosSchlichtHurewicz}*{Theorem~1.7}.} 
and $\HD_\kappa(\analytic(\kappa))$ holds for all inaccessible cardinals $\kappa$
\cite{LuckeMottoRosSchlichtHurewicz}*{Section~3.2}.%
\footnote{See Remark~\ref{remark: connections between PSP, THD and HD}~\ref{sep: PSP and HD} for a detailed explanation.}
It is open whether the latter holds for uncountable successor cardinals $\kappa$.
In the countable setting, 
$\ABP_\omega(\defsets\omega)$ is equivalent to the Baire property for all sets of reals in $\defsets\omega$.
Hence it is consistent relative to $\ZFC$ by a well-known result of Shelah \cite{MR768264}*{Theorem 7.16}.
Moreover, 
$\OGD_\kappa({}^\kappa\kappa)$ is not provable in $\ZFC$ for $\kappa=\omega_2$. 
To see this, note that in $L$, 
there exists a $\kappa$-Kurepa tree $T$ 
that is a continuous image of ${}^\kappa\kappa$ 
for any successor cardinal $\kappa=\mu^+$, where $\mu$ has uncountable cofinality
\cite{lucke2020descriptive}*{Theorem 1.7}. 
Then $\PSP_\kappa([T])$ fails in $L$, so $\OGD_\kappa({}^\kappa\kappa)$ fails as well.\footnote{Note that $\OGD_\kappa(X)$ is preserved under taking continuous images by Lemma~\ref{ODD for continuous images} and Corollary~\ref{homomorphisms and perfect homogeneous subgraphs}.}
In the countable setting, it is open whether an inaccessible cardinal is required to obtain $\THD_\omega(\defsets\omega)$, 
and while  \cite{stern1985regularity}*{Theorem 2} claimed a solution, J\"org Brendle pointed out to us that this argument does not work.

We expect the open dihypergraph dichotomy to have further {\bf applications}. 
In current work, we have shown that $\ODD\omega\omega(X)$ is equivalent to the \emph{closed-sets covering property} $\mathsf{CCP}(X)$ 
introduced by Louveau \cite{louveau1980}*{Th\'eor\`eme 2.2} and further studied by Petruska \cite{petruska1992borel}*{Theorem 1}, Solecki \cite{solecki1994covering}*{Theorem 1}, 
Di Prisco and Todor\v{c}evi\'c \cite{di1998perfect}*{Section~3}. 
We have also obtained a variant of this result in the uncountable case. 
It is likely that some further applications of $\mathsf{CCP}(X)$
in \cite{di1998perfect}*{Section~3} lift to the uncountable setting. 
Moreover, we expect that
the level-two results of Lecomte-Zeleny \cite{lecomte2014baire},
and more generally,
the higher dimensional version of the Kechris--Louveau--Woodin dichotomy 
in \cite{CarroyMillerSoukup}*{Section 4}, can be extended to the uncountable setting.  
Carroy and Miller showed in unpublished work that their basis results for the classes of non-Baire class $1$ functions and functions that are not $\sigma$-continuous with closed pieces \cite{CarroyMiller2020} follow from the Kechris--Louveau--Woodin dichotomy. 
Does this hold in the uncountable setting? 

Since the Baire property for subsets of ${}^\omega\omega$ follows from instances of $\ODD\omega\omega(X)$, we ask whether a similar result holds for Lebesgue measurability. 
Note that there is no analogous proof, since the ideal of Lebesgue null sets is not generated by closed sets and it thus cannot be generated by the independent sets with respect to a box-open dihypergraph on the Baire space. 


Several 
\todol{Changed due to Section 6.6}
{\bf implications} and {\bf separations} 
between applications of the open dihypergraph dichotomy are open, 
for instance all missing implications in Figures~\ref{figure: various implications weakly compact}--\ref{figure: various implications not pspclosed}. 
In particular, implications from the $\kappa$-perfect set property $\PSP_\kappa(\mathcal C)$ to  
most of the other applications of the open dihypergraph dichotomy discussed in Section~\ref{section: applications} are open. 

\begin{problem}\label{problem app of PSP} 
Suppose that $\kappa$ is uncountable and $\mathcal C$ is closed under continuous preimages. 
\todol{Changed due to Section 6.6: this problem and the next 3 paragraphs. Since we now know that the implications from KLWdis to PSP holds if PSP holds for closed sets (and fails for $\Fsigma$ sets if the latter fails), we thought it might make more sense to ask only about the implications from PSP to other applications}
Does
$\PSP_\kappa(\mathcal{C})$ imply 
$\KLWdis_\kappa(\mathcal{C})$,
$\ABP_\kappa(\mathcal{C})$ or $\JR_\kappa(\mathcal C)$?
\end{problem}

As for implications from
the $\kappa$-perfect set property and versions of the Hurewicz dichotomy,
it is open 
for weakly compact cardinals and $\omega$ 
whether
$\PSP_\kappa(\defsets\kappa)$ implies $\THD_\kappa(\defsets\kappa)$ 
and for weakly compact cardinals whether it implies $\THD_\kappa(\defsets\kappa,\pwrset({}^\kappa\kappa))$.\footnote{For uncountable but not weakly compact $\kappa$, $\PSP_\kappa(X)$ is equivalent to $\THD_\kappa(X,Y)$ for all supersets $Y$ of $X$, and in particular to $\THD_\kappa(X)$ by Proposition~\ref{remark: PSP THD non weakly compact}.} 
Note that $\THD_\kappa(\defsets\kappa)$ and $\HD_\kappa(\defsets\kappa)$ are equivalent for weakly compact cardinals by Proposition~\ref{proposition: HD and THD options}, but we do not know whether $\THD_\kappa(\defsets\kappa,\pwrset({}^\kappa\kappa))$ implies $\HD_\kappa(\defsets\kappa)$ for non-weakly compact cardinals. 

$\PSP_\omega(\defsetsk)$ implies neither $\THD_\omega(\defsets\omega,\Gdelta)$ nor $\KLWdis_\omega(\defsets\omega,\Fsigma)$ by results of Di Prisco and Todorcevic~\cite{di1998perfect} described in
Remark~\ref{remark: separating KLWdis and PSP}.
Their results further show that
neither the perfect set property 
nor the open graph dichotomy
implies the Baire property,
since the subset of the Cantor space induced by
a nonprincipal ultrafilter on $\omega$ does not have the Baire property.
Moreover, the Baire property does not imply the perfect set property for $\defsets\omega$ in the countable setting \cite{MR768264}*{Theorem~7.16}. 
Is this separation valid in the uncountable setting too? 

Note that
$\KLW_\kappa({}^\kappa\kappa,\defsetsk)$ implies $\cbi(\defsetsk)$.%
\footnote{See Theorem~\ref{theorem: ODD implies CB1}. Note that
$\cbi(\defsetsk)$ is equivalent to $\cbii(\defsetsk)$
by Theorem~\ref{prop: cbi iff cbii}.}
\todon{Footnote size problem solved in the usual way}
One can ask similarly to Problem \ref{problem app of PSP}
whether $\PSP_\kappa(\defsetsk)$ implies $\cbi(\defsetsk)$. 
This implication is equivalent to the statement that
$\GV_\kappa(X)$ is determined for every
$X\subseteq{}^\kappa\kappa$ of size $\kappa$.
We also ask whether the reverse implication holds.
This
is equivalent to the statement that every $\kappa$-crowded set 
in $\mathcal C$
of size at least $\kappa^+$
\todog{I wrote ``at least $\kappa^+$'' insted of just ``$\kappa^+$'' since I'm not sure every $\kappa$-croweded set of size at least $\kappa^+$ has a $\kappa$-crowded set of size exactly $\kappa^+$.}
has a $\kappa$-perfect subset.%
\footnote{Note that $\PSP_\kappa(\Gdelta(\kappa))$, $\cbi(\Gdelta(\kappa))$ and $\KLW_\kappa({}^\kappa\kappa,\Gdelta(\kappa))$ 
are equivalent by Corollary~\ref{cor: PSP from ODD for the kappa-Baire space 2}~\ref{cor: PSP from ODD for the kappa-Baire space 2 b} and 
Remark~\ref{PSP iff CB1 for Gdelta sets}.}

Recall that $\ODD\kappa\kappa(\mathcal C)$ and $\KLW^\kappa_\kappa(\mathcal C)$ are {\bf equivalent}.%
\footnote{See Corollary~\ref{cor: from KLW to ODD kappa}.} 
Are some other applications of $\ODD\kappa\kappa(\defsetsk)$ equivalent to it? 
For lower dimensions, we ask whether 
$\KLW^\ddim_\kappa(\mathcal C)$ implies 
$\ODD\kappa\ddim(\mathcal C)$. 
Note that
$\ODD\kappa\ddim(\mathcal C)$ is equivalent to $\KLW^\ddim_\kappa(\mathcal C,\Gdelta(\kappa))$ 
if $\mathcal C$ is closed under intersections with $\Gdelta(\kappa)$ sets.%
\footnote{See Corollary~\ref{cor: from KLW to ODD}~\ref{cor: from KLW to ODD 1}.}

Note that one can prove the consistency of $\ODD\kappa\ddim(\defsetsk)$ for $\ddim<\kappa$ more easily than in this paper along the lines of the proof of the consistency of $\PSP_\kappa(\defsets\kappa)$ in \cite{SchlichtPSPgames}*{Section~2.2}. 
This is how we proved it originally. 
The fact that the proofs are very similar suggests that 
$\ODD\kappa\ddim(\defsetsk)$ and $\PSP_\kappa(\defsetsk)$ may be equivalent if $\ddim<\kappa$.

\begin{problem}
\label{q: PSP implies ODD2?}
Does $\PSP_\kappa(\mathcal C)$ imply $\OGD_\kappa(\mathcal C)$ if $\mathcal C$ is closed under continuous preimages?%
\footnote{Recall that $\OGD_\kappa(X)$ and $\ODD\kappa 2(X)$ are equivalent for all subsets $X$ of ${}^\kappa\kappa$ by Corollary~\ref{homomorphisms and perfect homogeneous subgraphs}.}
\end{problem}

We further ask which of the open dihypergraph dichotomies for different {\bf dimensions} $2\leq\ddim\leq\kappa$ can be separated, and in particular: 

\begin{problem}
\label{q: PSP implies ODD?}
Do the following implications hold if $3\leq\ddim<\kappa$ and $\mathcal C$ is closed under continuous preimages? 
\begin{enumerate-(1)}
\item 
\label{q: PSP implies ODD? 1}
$\OGD_\kappa(\mathcal C)\Longrightarrow\ODD\kappa\ddim(\mathcal C)$.
\item 
\label{q: PSP implies ODD? 2}
$\ODD\kappa\ddim(\mathcal C)\Longrightarrow\ODD\kappa\kappa(\mathcal C)$.
\end{enumerate-(1)}
\end{problem}

The version of the implication in \ref{q: PSP implies ODD? 2} for definable $\kappa$-dihypergraphs is also open.  
Note that 
$\OGD_\omega(\mathcal{C})$ does not imply $\ODD\omega\omega(\mathcal{C},\defsets\omega)$ by 
the results of Di Prisco and Todorcevic~\cite{di1998perfect} described in
Remark~\ref{remark: separating KLWdis and PSP}. 
We ask if this implication fails for regular uncountable cardinals $\kappa$, too.

The {\bf restriction} $\ODD\kappa\kappa({}^\kappa\kappa)$ to the underlying set ${}^\kappa\kappa$ has interesting consequences such as 
$\ODD\kappa\kappa({\bf \Sigma}^1_1(\kappa))$.%
\footnote{See Theorem~\ref{thm: ODD for kappa^kappa and analytic sets}.}
This motivates the next problem and its restriction to definable dihypergraphs. 
Similar questions are of interest for  $\KLW^\ddim_\kappa({}^\kappa\kappa)$ and other applications. 

\begin{problem}
\label{q: ODDkk}
Does $\ODD\kappa\kappa({}^\kappa\kappa)$ imply $\ODD\kappa\kappa(\defsetsk)$?
\end{problem}

Summarizing and expanding some of the above questions, 
we ask which of the implications in 
Figure~\ref{figure: ODD OGD PSP} 
can be reversed.%
\footnote{The second and third implication in each row follows from Lemmas~\ref{<kappa dim hypergraphs are definable} and~\ref{comparing ODD for different dimensions}, and Corollary~\ref{cor: PSP from ODD}, respectively.
The equivalences in the figure follow from Lemma~\ref{ODD for continuous images}. 
Blue arrows are clickable and lead to the corresponding results.} 
Note that $\OGD_\kappa(\defsets\kappa)$ is equivalent to $\ODD\kappa 2(\defsets\kappa)$ in the figure\footnote{See Corollary~\ref{homomorphisms and perfect homogeneous subgraphs}.} 
and the same holds for other classes. 
The first two items in the third row are equivalent to $\ODD\kappa\kappa({}^\kappa\kappa)$ and $\ODD\kappa\kappa({}^\kappa\kappa,\defsetsk)$, respectively. 
\todon{``intextsep'' and ``abovecaptionskip'' were set locally in the figure}

{
\setlength{\abovecaptionskip}{5 pt}
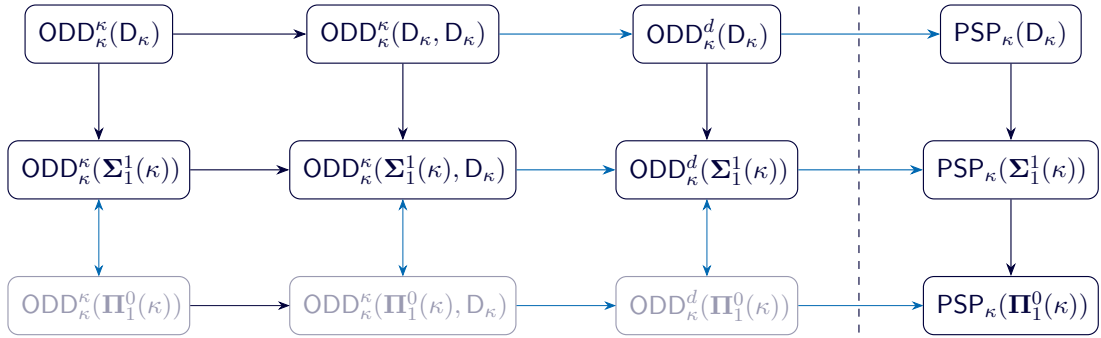
\begin{figure}[H] 
\begin{tikzpicture}[scale=1.0]
\tikzstyle{theory1}=[draw,rounded corners,scale=.900, minimum width=8.0mm, minimum height=8.0mm,line width=\widthline,VeryDarkBlueNode]
\tikzstyle{theory2}=[draw,rounded corners,scale=.900, minimum width=8.0mm, minimum height=8.0mm,line width=\widthline,GreyBlueNode]

\tikzstyle{dotted}=[dash pattern=on 1pt off  2pt]
\tikzstyle{medarrow}=[>=Stealth, line width=\widtharrow, scale=1.8, VeryDarkBlueNode] 
\tikzstyle{smallarrow}=[>=Stealth, line width=\widthlinethin, scale=1.8, VeryDarkBlueNode] 

 \node (oddD) at (0,3.6) [theory1] {{\parbox[c][.6cm]{1.9cm}{$\,\ODD\kappa\kappa(\defsetsk)$}}};
 \node (oddDD) at (4,3.6) [theory1] {{\parbox[c][.6cm]{2.55cm}{$\,\ODD\kappa\kappa(\defsetsk,\defsetsk)$}}};
 \node (ogdD) at (8,3.6) [theory1] {{\parbox[c][.6cm]{1.9cm}{$\,\ODD\kappa\ddim(\defsetsk)$}}};
 \node (pspD) at (12,3.6) [theory1] {{\parbox[c][.6cm]{1.80cm}{$\,\PSP_\kappa(\defsetsk)$}}}; 

 \node (oddA) at (0,1.8) [theory1] {{\parbox[c][.6cm]{2.4cm}{
$\,\ODD\kappa\kappa(\analytic(\kappa))$
}}};
 \node (oddAD) at (4,1.8) [theory1] {{\parbox[c][.6cm]{3.05cm}{
$\,\ODD\kappa\kappa(\analytic(\kappa),\defsetsk)$
}}};
 \node (ogdA) at (8,1.8) [theory1] {{\parbox[c][.6cm]{2.4cm}{$\,\ODD\kappa\ddim(\analytic(\kappa))$}}};
 \node (pspA) at (12,1.8) [theory1] {{\parbox[c][.6cm]{2.30cm}{$\,\PSP_\kappa(\analytic(\kappa))$}}};
 
 \node (oddC) at (0,0.0) [theory2] {{\parbox[c][.6cm]{2.4cm}{
$\,\ODD\kappa\kappa(\closedsets(\kappa))$
}}};
 \node (oddCD) at (4,0.0) [theory2] {{\parbox[c][.6cm]{3.05cm}{
$\,\ODD\kappa\kappa(\closedsets(\kappa),\defsetsk)$
}}};
 \node (ogdC) at (8,0.0) [theory2] {{\parbox[c][.6cm]{2.4cm}{$\,\ODD\kappa\ddim(\closedsets(\kappa))$}}};
 \node (pspC) at (12,0.0) [theory1] {{\parbox[c][.6cm]{2.30cm}{$\,\PSP_\kappa(\closedsets(\kappa))$}}};

 \node (pointA) at (10,4.10){};
 \node (pointB) at (10,-0.45){};
 \draw[medarrow, DottedLine]
(pointA) edge [dotted](pointB); 

\draw[medarrow, VeryDarkBlueNode] 
([yshift=0mm]oddD.east) edge [->]
([yshift=0mm]oddDD.west)
([yshift=0mm]oddA.east) edge [->] 
([yshift=0mm]oddAD.west)
([yshift=0mm]oddD.south) edge [->]
([yshift=0mm]oddA.north)
([yshift=0mm]oddDD.south) edge [->] 
([yshift=0mm]oddAD.north)
([yshift=0mm]ogdD.south) edge [->] 
([yshift=0mm]ogdA.north)
([yshift=0mm]pspA.south) edge [->] 
([yshift=0mm]pspC.north)
([yshift=0mm]pspD.south) edge [->] 
([yshift=0mm]pspA.north)
;

\draw[medarrow, BlueArrow] 
([yshift=0mm]oddDD.east) edge [->] 
node {\hyperlink{ODD^kappa implies definable ODD^ddim}{\phantom{xxx}}} 
([yshift=0mm]ogdD.west) 
([yshift=0mm]oddAD.east) edge [->]
node {\hyperlink{ODD^kappa implies definable ODD^ddim}{\phantom{xxx}}} 
([yshift=0mm]ogdA.west)
([yshift=0mm]ogdD.east) edge [->] 
node {\hyperref[cor: PSP from ODD]{\phantom{xxx}}} 
([yshift=0mm]pspD.west)
([yshift=0mm]ogdA.east) edge [->] 
node {\hyperref[cor: PSP from ODD]{\phantom{xxx}}} 
([yshift=0mm]pspA.west)
;

\draw[smallarrow, VeryDarkBlueNode] 
([yshift=0mm]oddC.east) edge [->]
([yshift=0mm]oddCD.west)
;

\draw[smallarrow, BlueArrow] 
([yshift=0mm]oddA.south) edge [<->] 
node {\hyperref[ODD for continuous images]{\phantom{xxx}}} 
([yshift=0mm]oddC.north)
([yshift=0mm]oddAD.south) edge [<->] 
node {\hyperref[ODD for continuous images]{\phantom{xxx}}} 
([yshift=0mm]oddCD.north) 
([yshift=0mm]ogdA.south) edge [<->] 
node {\hyperref[ODD for continuous images]{\phantom{xxx}}} 
([yshift=0mm]ogdC.north) 
([yshift=0mm]oddCD.east) edge [->] 
node {\hyperlink{ODD^kappa implies definable ODD^ddim}{\phantom{xxx}}} 
([yshift=0mm]ogdC.west) 
([yshift=0mm]ogdC.east) edge [->] 
node {\hyperref[cor: PSP from ODD]{\phantom{xxx}}} 
([yshift=0mm]pspC.west) 
;

\end{tikzpicture}
\caption{Implications between $\mathsf{ODD}$ for dimensions $2\leq\ddim<\kappa$ and $\PSP$}
\label{figure: ODD OGD PSP}
\end{figure}
}

Separating the various dichotomies and applications discussed above is difficult, since the models 
where they are known to hold are in most cases obtained by L\'evy collapses. 
One can try to force over a model of $\ODD\kappa\kappa(\defsets\kappa)$ to preserve some applications but break others as in \cite{di1998perfect} in the countable setting. 

Figure~\ref{figure: ODD variants} in Subsection~\ref{subsection: ODD ODDI ODDH} summarizes 
implications and separations 
between several {\bf variants} of the open dihypergraph dichotomy with additional requirements on the homomorphisms. 
Which of these implications can be reversed? 
A striking difference between the countable and uncountable settings is that one can choose the continuous homomorphisms to be injective in the latter case, assuming $\Diamondi\kappa$ holds.\footnote{See Propositions~\ref{lemma: dhD} and~\ref{ODDI fails for D omega} and Theorem~\ref{theorem: ODD ODDI}.} 
We ask whether this assumption can be eliminated. 

\begin{problem}\label{question ODDI and ODD} 
Does $\ODD{\kappa}H \Longrightarrow \ODDI{\kappa}H$ hold 
for all uncountable $\kappa$ and all relatively box-open $\kappa$-dihypergraphs $H$?%
$\,$\footnote{It is equivalent to ask whether 
the existence of a continuous homomorphism from $\dhH\kappa$ to a relatively box-open $\kappa$-dihypergraph $H$ implies the existence of an injective map with the same properties.}
\todon{footnote size problem solved in the usual way}
\end{problem} 

This implication holds for $\ddim$-dihypergraphs for $\ddim<\kappa$.%
\footnote{In fact, the homomorphism can be chosen to be a homeomorphism onto a closed set; see Theorem~\ref{theorem: ODD ODDC when ddim<kappa} and Corollary~\ref{cor: ODD ODDC when ddim<kappa}.}
Since $\Diamondi\kappa$ holds if $\kappa$ is inaccessible or a successor cardinal ${\geq}\omega_2$ \cite{ShelahDiamonds},\footnote{See Lemma~\ref{diamondi claim}.}
the remaining cases are $\omega_1$ 
and weakly inaccessible, but not inaccessible, cardinals. 
To separate $\ODD{\omega_1}H$ from $\ODDI{\omega_1}H$ for some relatively box-open $H$, one would need to consider models where $\CH$ holds but $\Diamond_{\omega_1}$ fails. 
To our knowledge, the first such model was constructed by Jensen \cite{devlin1974souslin}. 
Shelah found an alternative proof in \cite{shelah2017proper}*{Theorem~6.1}. 
A simpler construction is possible based on recent work of   and Mota \cite{aspero2017few}. 
Do these models separate the above two principles for some $H$? 
If one can get that $\ODDI\kappa\kappa(\defsets\kappa)$\footnote{The definition of $\ODDI\kappa\kappa(\defsetsk)$ is analogous to Definition~\ref{def: ODD for classes}.}
implies $\Diamondi\kappa$ for all $\kappa$, then any model of $\CH + \neg\Diamond_{\omega_1} + \ODD{\omega_1}{\omega_1}(\defsets\kappa)$ 
is as required. 
A related question is whether it is consistent that
$\ODDI\kappa H\Longrightarrow\ODDH\kappa H$ holds for all relatively box-open $\kappa$-dihypergraphs $H$. 

The questions studied in this paper make sense for 
other spaces 
such as the {\bf strong $\kappa$-Choquet spaces} studied in \cite{coskey2013generalized}
and \cite{agostini2021generalized}. 
Is it consistent that 
the open dihypergraph dichotomy holds for all such spaces and does it have similar applications? 
Are there analogues to the results in this paper for 
{\bf singular cardinals} $\kappa$ with $2^{<\kappa}=\kappa$? 
In setting of $\cof(\kappa)=\omega$, the space ${}^\omega\kappa$ is called the $\kappa$-Baire space \cites{DimonteMottoRosSingular,DimonteMottoRosShi}. 
Variants of the perfect set property and the Baire property for subsets of the $\kappa$-Baire space in $L(V_{\kappa+1})$ have been studied \cites{CramerPSPandI0,DimonteSolovay,DimontePovedaThei}.


Ben-Neria, Fuchino, Juh\'asz, Moore and Yorioka 
asked about the following problem on higher analogues to the {\bf open graph axiom}. 
Let $\OGA_\kappa$ denote the statement that $\OGA_\kappa(X)$%
\footnote{See the paragraph before Definition~\ref{OGD def}.}
holds for all subsets $X$ of ${}^\kappa\kappa$. 

\begin{problem}
\label{open q OGA} 
Is it consistent that $\OGA_\kappa$ holds for some regular uncountable cardinal $\kappa$ with $\kappa^{<\kappa}=\kappa$? 
\end{problem} 

Our results provide a choiceless model of this statement, since 
Theorem \ref{main theorem} shows that the model $L({}^\kappa\Ord)$ in a generic extension obtained by the L\'evy collapse of an inaccessible cardinal to $\kappa^+$ satisfies $\OGA_\kappa$.  
This suggests to approach Problem~\ref{open q OGA} by forcing over a choiceless model of $\OGA_\kappa$. 
\todoq{It may be useful to know whether $\OGA_\kappa$ implies 
$\OGD_\kappa(\analytic(\kappa))$ or the same for $\defsets\kappa$. This would motivate the strategy to force over  a choiceless model of $\OGD$ to obtain $\OGA$. } 



\newpage 

\bibliography{references_Philipp,references_Dori}{}

\newcommand{\noop}[1]{}
\begin{bibdiv}
\begin{biblist}

\bib{andretta2022souslin}{article}{
      author={Andretta, Alessandro},
      author={{Motto Ros}, Luca},
       title={Souslin quasi-orders and bi-embeddability of uncountable
  structures},
        date={2022},
     journal={Mem. Amer. Math. Soc.},
      volume={277},
      number={1365},
}

\bib{aspero2017few}{article}{
      author={Asper{\'o}, David},
      author={Mota, Miguel~Angel},
       title={Few new reals},
        date={2023},
     journal={J. Math. Log.},
        note={To appear},
}

\bib{agostini2021generalized}{article}{
      author={Agostini, Claudio},
      author={Motto~Ros, Luca},
      author={Schlicht, Philipp},
       title={Generalized polish spaces at regular uncountable cardinals},
        date={2023},
     journal={J. Lond. Math. Soc. (2)},
      volume={108},
      number={5},
       pages={1886\ndash 1929},
}

\bib{AbrahamRubinShelah}{article}{
      author={Abraham, Uri},
      author={Rubin, Metatyahu},
      author={Shelah, Saharon},
       title={On the consistency of some partition theorems for continuous
  colorings, and the structure of $\aleph_1$-dense real order types},
        date={1985},
     journal={Ann. Pure Appl. Logic},
      volume={29},
      number={2},
       pages={123\ndash 206},
}

\bib{avitod2010}{article}{
      author={Avil\'es, Antonio},
      author={Todor\v{c}evi\'c, Stevo},
       title={Mutiple gaps},
        date={2011},
     journal={Fund. Math.},
      volume={213},
      number={1},
       pages={15\ndash 42},
}

\bib{bekkali1991topics}{article}{
      author={Bekkali, Mohamed},
       title={Topics in set theory. {L}ebesgue measurability, large cardinals,
  forcing axioms, rho-functions. {N}otes on lectures by {S}tevo
  {T}odor\v{c}evi\'c},
        date={1991},
     journal={Lecture Notes in Math.},
      volume={1476},
       pages={109\ndash 129},
}

\bib{blass1981partition}{article}{
      author={Blass, Andreas},
       title={A partition theorem for perfect sets},
        date={1981},
     journal={Proc. Amer. Math. Soc.},
      volume={82},
      number={2},
       pages={271\ndash 277},
}

\bib{CarroyMiller}{unpublished}{
      author={Carroy, Rapha{\"e}l},
      author={Miller, Benjamin~D.},
       title={{Open graphs, Hurewicz-style dichotomies, and the Jayne-Rogers
  theorem}},
        date={2015},
        note={Unpublished note},
}

\bib{CarroyMiller2020}{article}{
      author={Carroy, Rapha{\"e}l},
      author={Miller, Benjamin~D.},
       title={Bases for functions beyond the first {B}aire class},
        date={2020},
     journal={J. Symb. Log.},
      volume={85},
      number={3},
       pages={1289\ndash 1303},
}

\bib{CarroyMillerSoukup}{incollection}{
      author={Carroy, Rapha{\"e}l},
      author={Miller, Benjamin~D.},
      author={Soukup, D{\'a}niel~T.},
       title={{T}he open dihypergraph dichotomy and the second level of the
  {B}orel hierarchy},
        date={2020},
   booktitle={Trends in set theory},
      series={Contemp. Math.},
      volume={752},
   publisher={American Mathematical Soc.},
       pages={1\ndash 20},
}

\bib{CramerPSPandI0}{article}{
      author={Cramer, Scott~S.},
       title={Inverse limit reflection and the structure of
  {$L(V_{\lambda+1})$}},
        date={2015},
     journal={J. Math. Log.},
      volume={15},
}

\bib{coskey2013generalized}{article}{
      author={Coskey, Samuel},
      author={Schlicht, Philipp},
       title={Generalized {C}hoquet spaces},
        date={2016},
     journal={Fund. Math.},
      volume={232},
       pages={227\ndash 248},
}

\bib{de2018generalization}{article}{
      author={de~Brecht, Matthew},
       title={A generalization of a theorem of {H}urewicz for quasi-{P}olish
  spaces},
        date={2018},
     journal={Log. Methods Comput. Sci.},
      volume={14},
}

\bib{DimonteSolovay}{article}{
      author={Dimonte, Vincenzo},
       title={A {S}olovay-like model for singular generalized descriptive set
  theory},
        date={2023},
     journal={Topology Appl.},
      volume={323},
}

\bib{devlin1974souslin}{book}{
      author={Devlin, Keith},
      author={Johnsbr{\aa}ten, H{\aa}vard},
       title={The {S}ouslin problem},
   publisher={Springer},
        date={1974},
}

\bib{DolezalKubis}{article}{
      author={Dole\v{z}al, Martin},
      author={Kubi\'s, Wies{\l}aw},
       title={Perfect independent sets with respect to infinitely many
  relations},
        date={2016},
     journal={Arch. Math. Logic},
      volume={55},
       pages={847\ndash 856},
}

\bib{DimonteMottoRosSingular}{unpublished}{
      author={Dimonte, Vincenzo},
      author={{Motto Ros}, Luca},
       title={Generalized descriptive set theory at singular cardinals of
  countable cofinality},
        date={2025},
        note={Preprint. Available at \url{https://arxiv.org/abs/2511.16188}},
}

\bib{DimonteMottoRosShi}{unpublished}{
      author={Dimonte, Vincenzo},
      author={{Motto Ros}, Luca},
      author={Shi, Xianghui},
       title={Generalized descriptive set theory at singular cardinals of
  countable cofinality~{II}},
        date={2025},
        note={In preparation},
}

\bib{DimontePovedaThei}{unpublished}{
      author={Dimonte, Vincenzo},
      author={Poveda, Alejandro},
      author={Thei, Sebastiano},
       title={The {B}aire and perfect set properties at singular cardinals},
        date={2023},
        note={Preprint. Available at \url{https://arxiv.org/abs/2408.05973}},
}

\bib{di1998perfect}{article}{
      author={{Di Prisco}, Carlos~Augusto},
      author={Todor\v{c}evi\'c, Stevo},
       title={Perfect-set properties in {$L(\mathbb{R})[U]$}},
        date={1998},
     journal={Adv. Math.},
      volume={139},
      number={2},
       pages={240\ndash 259},
}

\bib{FengOCA}{article}{
      author={Feng, Qi},
       title={Homogeneity for open partitions of pairs of reals},
        date={1993},
     journal={Trans. Amer. Math. Soc.},
      volume={339},
      number={2},
       pages={659\ndash 684},
}

\bib{MR3235820}{article}{
      author={Friedman, Sy-David},
      author={Hyttinen, Tapani},
      author={Kulikov, Vadim},
       title={Generalized descriptive set theory and classification theory},
        date={2014},
        ISSN={0065-9266},
     journal={Mem. Amer. Math. Soc.},
      volume={230},
      number={1081},
       pages={vi+80},
      review={\MR{3235820}},
}

\bib{MR2387944}{article}{
      author={Fuchs, Gunter},
       title={Closed maximality principles: implications, separations and
  combinations},
        date={2008},
        ISSN={0022-4812},
     journal={J. Symb. Log.},
      volume={73},
      number={1},
       pages={276\ndash 308},
      review={\MR{2387944}},
}

\bib{GalgonThesis}{thesis}{
      author={Galgon, Geoff},
       title={Trees, refining, and combinatorial characteristics},
        type={Ph.D. Thesis},
        date={2016},
}

\bib{galvin1968partition}{article}{
      author={Galvin, Fred},
       title={Partition theorems for the real line},
        date={1968},
     journal={Notices Amer. Math. Soc.},
      volume={15},
      number={660},
       pages={6},
}

\bib{GalvinJechMagidor}{article}{
      author={Galvin, Fred},
      author={Jech, Thomas},
      author={Magidor, Menachem},
       title={An ideal game},
        date={1978},
     journal={J. Symb. Log.},
      volume={43},
      number={2},
       pages={284\ndash 292},
}

\bib{gregoriades2021turing}{article}{
      author={Gregoriades, Vassilios},
      author={Kihara, Takayuki},
      author={Ng, Keng~Meng},
       title={Turing degrees in {P}olish spaces and decomposability of {B}orel
  functions},
        date={2021},
     journal={J. Math. Log.},
      volume={21},
      number={1},
       pages={2050021},
}

\bib{GolshaniDiamondInaccessible}{article}{
      author={Golshani, Mohammad},
       title={({W}eak) diamond can fail at the least inaccessible cardinal},
        date={2022},
     journal={Fund. Math.},
      volume={256},
       pages={113\ndash 129},
}

\bib{HeHigherDimOCA}{article}{
      author={He, Bin},
       title={A high dimensional {O}pen {C}oloring {A}xiom},
        date={2005},
     journal={Math. Log. Q.},
      volume={51},
      number={5},
       pages={462\ndash 469},
}

\bib{hung1973spaces}{article}{
      author={Hung, Henry~H.},
      author={Negrepontis, Stelios},
       title={Spaces homeomorphic to $(2^\alpha)_\alpha$},
        date={1973},
     journal={Bull. Amer. Math. Soc.},
      volume={79},
      number={1},
       pages={143\ndash 146},
}

\bib{MR1880900}{article}{
      author={Halko, Aapo},
      author={Shelah, Saharon},
       title={On strong measure zero subsets of {$^\kappa2$}},
        date={2001},
        ISSN={0016-2736},
     journal={Fund. Math.},
      volume={170},
      number={3},
       pages={219\ndash 229},
      review={\MR{1880900}},
}

\bib{hurewicz1926verallgemeinerung}{article}{
      author={Hurewicz, Witold},
       title={{\"U}ber eine {V}erallgemeinerung des {B}orelschen {T}heorems},
        date={1926},
     journal={Mathematische Zeitschrift},
      volume={24},
      number={1},
       pages={401\ndash 421},
}

\bib{Hurewicz1928}{article}{
      author={Hurewicz, Witold},
       title={Relativ perfekte {T}eile von {P}unktmengen und {M}engen ({A})},
        date={1928},
     journal={Fund. Math.},
      volume={12},
       pages={79\ndash 109},
}

\bib{MR1940513}{book}{
      author={Jech, Thomas},
       title={Set theory. {T}he third millennium edition, revised and
  expanded},
      series={Springer Monographs in Mathematics},
   publisher={Springer-Verlag, Berlin},
        date={2003},
        ISBN={3-540-44085-2},
      review={\MR{1940513}},
}

\bib{JechTrees}{article}{
      author={Jech, Thomas},
       title={Trees},
        date={1971},
     journal={J. Symb. Log.},
      volume={36},
      number={1},
       pages={1\ndash 14},
}

\bib{JensenKunenDiamond}{unpublished}{
      author={Jensen, Ronald~B.},
      author={Kunen, Kenneth},
       title={Some combinatorial properties of {$L$ and $V$}},
        date={1969},
        note={Unpublished manuscript},
}

\bib{JayneRogers1982}{article}{
      author={Jayne, John~E.},
      author={Rogers, C.~Ambrose},
       title={First level {B}orel functions and isomorphisms},
        date={1982},
     journal={J. Math. Pures Appl.},
      volume={61},
      number={2},
       pages={177\ndash 205},
}

\bib{Kechris1977}{article}{
      author={Kechris, Alexander~S.},
       title={On a notion of smallness for subsets of the {B}aire space},
        date={1977},
     journal={Trans. Amer. Math. Soc.},
      volume={229},
       pages={191\ndash 207},
}

\bib{KechrisBook}{book}{
      author={Kechris, Alexander~S.},
       title={Classical descriptive set theory},
      series={Graduate Texts in Mathematics},
   publisher={Springer-Verlag},
     address={New York},
        date={1995},
      volume={165},
}

\bib{KechrisLouveauWoodin1987}{article}{
      author={Kechris, Alexander~S.},
      author={Louveau, Alain},
      author={Woodin, W.~Hugh},
       title={The structure of $\sigma$-ideals of compact sets},
        date={1987},
     journal={Trans. Amer. Math. Soc.},
      volume={301},
      number={1},
       pages={263\ndash 288},
}

\bib{kacenamottorossemmes}{article}{
      author={Ka\v{c}ena, Miroslav},
      author={{Motto Ros}, Luca},
      author={Semmes, Brian},
       title={Some observations on `{A} new proof of a theorem of {J}ayne and
  {R}ogers'},
        date={2013},
     journal={Real Anal. Exchange},
      volume={38},
      number={1},
       pages={121 \ndash  132},
}

\bib{Kubis}{article}{
      author={Kubi\'s, Wies{\l}aw},
       title={Perfect cliques and {$G_\delta$} colorings of {Polish} spaces},
        date={2003},
     journal={Proc. Amer. Math. Soc.},
      volume={131},
      number={2},
       pages={619\ndash 623},
}

\bib{KunenBook2013}{book}{
      author={Kunen, Kenneth},
       title={Set theory: an introduction to independence proofs},
      series={Studies in Logic},
   publisher={College Publications},
     address={London},
        date={2013},
      volume={34},
}

\bib{KunenBook1980}{book}{
      author={Kunen, Kenneth},
       title={Set theory: an introduction to independence proofs},
   publisher={North-Holland Publishing Company},
     address={Amsterdam},
        date={1980},
}

\bib{LambieStejPSP}{article}{
      author={Lambie-Hanson, Chris},
      author={Stejskalov{\'a}, {\v{S}}{\'a}rka},
       title={Kurepa trees, continuous images, and perfect set properties},
        date={2026},
     journal={Ann. Pure Appl. Logic},
      volume={177},
      number={3},
       pages={103663},
}

\bib{LuckeMottoRosSchlichtHurewicz}{article}{
      author={L{\"u}cke, Philipp},
      author={{Motto Ros}, Luca},
      author={Schlicht, Philipp},
       title={The {Hurewicz} dichotomy for generalized {Baire} spaces},
        date={2016},
     journal={Israel J. Math.},
      volume={216},
      number={2},
       pages={973\ndash 1022},
}

\bib{Hurewiczdef}{unpublished}{
      author={L\"ucke, Philipp},
      author={{Motto Ros}, Luca},
      author={Schlicht, Philipp},
       title={The {H}urewicz dichotomy for definable subsets of generalized
  {Baire} spaces},
        date={2018},
        note={Manuscript},
}

\bib{louveau1980}{article}{
      author={Louveau, Alain},
       title={$\sigma$-id{\'e}aux engendr{\'e}s par des ensembles ferm{\'e}s et
  th{\'e}or{\`e}mes d’approximation},
        date={1980},
     journal={Trans. Amer. Math. Soc.},
      volume={257},
      number={1},
       pages={143\ndash 169},
}

\bib{MR3430247}{article}{
      author={L{\"u}cke, Philipp},
      author={Schlicht, Philipp},
       title={Continuous images of closed sets in generalized {B}aire spaces},
        date={2015},
        ISSN={0021-2172},
     journal={Israel J. Math.},
      volume={209},
      number={1},
       pages={421\ndash 461},
      review={\MR{3430247}},
}

\bib{lucke2020descriptive}{incollection}{
      author={L{\"u}cke, Philipp},
      author={Schlicht, Philipp},
       title={Descriptive properties of higher {K}urepa trees},
   booktitle={Research trends in contemporary logic},
   publisher={College Publications},
        note={To appear},
}

\bib{lucke2012definability}{article}{
      author={L{\"u}cke, Philipp},
       title={{$\Sigma^1_1$}-definability at uncountable regular cardinals},
        date={2012},
     journal={J. Symb. Log.},
      volume={77},
      number={3},
       pages={1011\ndash 1046},
}

\bib{lecomte2014baire}{article}{
      author={Lecomte, Dominique},
      author={Zeleny, Miroslav},
       title={Baire-class $\xi$ colorings: the first three levels},
        date={2014},
     journal={Trans. Am. Math. Soc.},
      volume={366},
      number={5},
       pages={2345\ndash 2373},
}

\bib{MatetCorrigendumDiamond}{article}{
      author={Matet, Pierre},
       title={Corrigendum to ``{P}artitions and diamond''},
        date={1987},
     journal={Proc. Amer. Math. Soc.},
      volume={101},
      number={3},
       pages={519\ndash 520},
}

\bib{menger1924}{article}{
      author={Menger, Karl},
       title={Einige {{\"U}}berdeckungss{\"a}tze der {P}unktmengenlehre},
        date={1924},
     journal={Sitzungsber. Wien. Akad.},
      volume={133},
       pages={421\ndash 444},
}

\bib{moreno2022isomorphism}{article}{
      author={Moreno, Miguel},
       title={The isomorphism relation of theories with {S-DOP} in the
  generalised {B}aire spaces},
        date={2022},
     journal={Ann. Pure Appl. Logic},
      volume={173},
      number={2},
       pages={103044},
}

\bib{moreno2022unsuperstable}{article}{
      author={Moreno, Miguel},
       title={On unsuperstable theories in {GDST}},
        date={2023},
     journal={J. Symb. Log.},
        note={To appear},
}

\bib{moreno2023shelah}{unpublished}{
      author={Moreno, Miguel},
       title={Shelah's {M}ain {G}ap and the generalized {B}orel-reducibility},
        date={2023},
        note={Preprint. Available at \url{https://arxiv.org/abs/2308.07510}},
}

\bib{ros2013structure}{article}{
      author={{Motto Ros}, Luca},
       title={On the structure of finite level and $\omega$-decomposable
  {B}orel functions},
        date={2013},
     journal={J. Symb. Log.},
      volume={78},
      number={4},
       pages={1257\ndash 1287},
}

\bib{beatrice}{unpublished}{
      author={Motto~Ros, Luca},
      author={{Pitton}, Beatrice},
      author={Schlicht, Philipp},
       title={The {W}adge hierarchy in generalized descriptive set theory},
        date={2026},
        note={In preparation},
}

\bib{ros2010new}{article}{
      author={{Motto Ros}, Luca},
      author={Semmes, Brian},
       title={A new proof of a theorem of {J}ayne and {R}ogers},
        date={2010},
     journal={Real Anal. Exchange},
      volume={35},
      number={1},
       pages={195\ndash 204},
}

\bib{mohammadpour2021guessing}{article}{
      author={Mohammadpour, Rahman},
      author={Veli{\v{c}}kovi{\'c}, Boban},
       title={Guessing models and the approachability ideal},
        date={2021},
     journal={J. Math. Log.},
      volume={21},
      number={2},
       pages={2150003},
}

\bib{vaan}{article}{
      author={Mekler, Alan},
      author={V{\"a}{\"a}n{\"a}nen, Jouko},
       title={Trees and {$\Pi\sp 1\sb 1$}-subsets of {$\sp {\omega\sb
  1}\omega\sb 1$}},
        date={1993},
        ISSN={0022-4812},
     journal={J. Symb. Log.},
      volume={58},
      number={3},
       pages={1052\ndash 1070},
      review={\MR{MR1242054 (94k:03063)}},
}

\bib{petruska1992borel}{article}{
      author={Petruska, Gy\"orgy},
       title={On {B}orel sets with small cover: a problem of {M}.
  {L}aczkovich},
        date={1992},
     journal={Real Anal. Exchange},
      volume={18},
      number={2},
       pages={330\ndash 338},
}

\bib{Ramsey1930}{article}{
      author={Ramsey, Frank~P.},
       title={On a problem of formal logic},
        date={1930},
     journal={Proc. London Math. Soc.},
      volume={30},
       pages={264\ndash 286},
}

\bib{SaintRaymond1975}{article}{
      author={{Saint-Raymond}, Jean},
       title={Approximation des sous-ensembles analytiques par l'int\'erieur},
        date={1975},
     journal={C. R. Acad. Sci. Paris S\'er. A--B},
      volume={281},
      number={2--3},
       pages={Aii, A85\ndash A87},
}

\bib{SchlichtPSPgames}{article}{
      author={Schlicht, Philipp},
       title={Perfect subsets of generalized {Baire} spaces and long games},
        date={2017},
     journal={J. Symb. Log.},
      volume={82},
      number={4},
       pages={1317\ndash 1355},
}

\bib{ShelahDiamonds}{article}{
      author={Shelah, Saharon},
       title={Diamonds},
        date={2010},
     journal={Proc. Amer. Math. Soc.},
      volume={138},
      number={6},
       pages={2151\ndash 2161},
}

\bib{shelah2017proper}{book}{
      author={Shelah, Saharon},
       title={Proper and improper forcing},
   publisher={Cambridge University Press},
        date={2017},
}

\bib{ShelahModelsIV}{article}{
      author={Shelah, Saharon},
       title={Models with second order properties {IV}. {A} general method and
  eliminating diamonds},
        date={1983},
     journal={Ann.~Pure Appl.~Logic},
      volume={25},
       pages={183\ndash 212},
}

\bib{MR768264}{article}{
      author={Shelah, Saharon},
       title={Can you take {S}olovay's inaccessible away?},
        date={1984},
        ISSN={0021-2172},
     journal={Israel J. Math.},
      volume={48},
      number={1},
       pages={1\ndash 47},
      review={\MR{768264}},
}

\bib{ShelahBorelSq}{article}{
      author={Shelah, Saharon},
       title={Borel sets with large squares},
        date={1999},
     journal={Fund. Math.},
      volume={159},
      number={1},
       pages={1\ndash 50},
}

\bib{sierpinski1933probleme}{article}{
      author={Sierpi{\'n}ski, Wac{\l}aw},
       title={Sur un probleme de la th{\'e}orie des relations},
        date={1933},
     journal={Ann. Sc. Norm. Super. Pisa Cl. Sci.},
      volume={2},
      number={3},
       pages={285\ndash 287},
}

\bib{Solovay1970}{article}{
      author={Solovay, Robert~M.},
       title={A model of set theory in which every set of reals is {L}ebesgue
  measurable},
        date={1970},
     journal={Ann.~of Math.},
      volume={92},
      number={2},
       pages={1\ndash 56},
}

\bib{solecki1994covering}{article}{
      author={Solecki, S{\l}awomir},
       title={Covering analytic sets by families of closed set},
        date={1994},
     journal={J. Symb. Log.},
      volume={59},
      number={3},
       pages={1022\ndash 1031},
}

\bib{stern1985regularity}{article}{
      author={Stern, Jacques},
       title={Regularity properties of definable sets of reals},
        date={1985},
     journal={Ann. Pure Appl. Logic},
      volume={29},
      number={3},
       pages={289\ndash 324},
}

\bib{SzVaananen}{article}{
      author={{Sz}ir\'aki, D.},
      author={V\"a\"an\"anen, J.},
       title={A dichotomy theorem for the generalized {Baire} space and
  elementary embeddability at uncountable cardinals},
        date={2017},
     journal={Fund. Math.},
      volume={238},
       pages={53\ndash 78},
}

\bib{SzThesis}{thesis}{
      author={{Sz}ir\'aki, D.},
       title={Colorings, perfect sets and games on generalized {B}aire spaces},
        type={Ph.D. Thesis},
        date={2018},
}

\bib{TodorcevicFarah}{book}{
      author={Todor\v{c}evi\'c, Stevo},
      author={Farah, Ilijas},
       title={Some applications of the method of forcing},
   publisher={Yenisei},
     address={Moscow},
        date={1995},
}

\bib{TodorcevicPartitionProblems}{book}{
      author={Todor\v{c}evi\'c, Stevo},
       title={Some applications of the method of forcing},
      series={Contemp. Math.},
   publisher={Amer. Math. Soc.},
     address={Providence, R.I.},
        date={1989},
      volume={84},
}

\bib{tall2021strength}{article}{
      author={Tall, Franklin~D.},
      author={Todorcevic, Stevo},
      author={Tokg{\"o}z, Se{\c{c}}il},
       title={The strength of {M}enger's conjecture},
        date={2021},
     journal={Topology Appl.},
      volume={301},
       pages={107536},
}

\bib{VaananenCantorBendixson}{article}{
      author={V{\"a}{\"a}n{\"a}nen, Jouko},
       title={A {Cantor--Bendixson} theorem for the space
  $\omega_1^{\omega_1}$},
        date={1991},
     journal={Fund. Math.},
      volume={137},
      number={3},
       pages={187\ndash 199},
}

\bib{VaananenGamesTrees}{incollection}{
      author={V{\"a}{\"a}n{\"a}nen, Jouko},
       title={Games and trees in infinitary logic: A survey},
        date={1995},
   booktitle={Quantifiers: Logics, models and computation},
      editor={Krynicki, M.},
      editor={Mostowski, M.},
      editor={Szczerba, L.},
   publisher={Kluwer Academic Publishers},
       pages={105\ndash 138},
}

\bib{vaught1974invariant}{article}{
      author={Vaught, Robert},
       title={Invariant sets in topology and logic},
        date={1974},
     journal={Fund. Math},
      volume={82},
      number={269-294},
       pages={75},
}

\bib{VelickovicOCA}{incollection}{
      author={Veli\v{c}kovi\'c, Boban},
       title={Applications of the open coloring axiom},
        date={1992},
   booktitle={Set theory of the continuum},
      editor={H.~Judah, H.~Woodin, W.~Just},
   publisher={Springer-Verlag},
       pages={137\ndash 154},
}

\bib{zdomskyy}{unpublished}{
      author={Zdomskyy, Lyubomir},
       title={Combinatorial covering properties in the context of generalized
  descriptive set theory},
        date={2023},
        note={Slides, Sixth workshop on generalised Baire spaces, Technische
  Universit\"at Wien, 2023. Available at
  \url{https://sites.google.com/view/gbs23/program}},
}

\end{biblist}
\end{bibdiv}
\bibliographystyle{alpha}


\newpage 
\printindex

\todoig{To do notes color codes for the corrections:
\\
{\bf orange}: changes made as a result of the referee report
\\
{\bf blue-green}: changes we made independently of the referee report. (This includes a small corrections in one proof, maybe some typos, and some small improvements to a few results.)
\\
{\bf red, very light yellow}: changes we either still need to make, or to discuss/check/improve/implement in other versions of the file
\\
{\bf green, pink, dark yellow}: these are NOT about corrections (green: useful observations, pink: open questions, yellow: todonotes about layout issues)
}

\end{document}